\documentclass[10pt, oneside]{amsart}
\usepackage[lmargin=1in,rmargin=1in, bmargin=1in, tmargin=1in]{geometry}
\usepackage{amsmath,amsthm,amssymb,tikz}
\usepackage{tikz-cd}
\usepackage{tikz}
\usepackage{todonotes}
\usepackage{mathrsfs}
\usepackage{extarrows}
\usepackage{graphicx}

\usepackage[pagebackref, hidelinks]{hyperref}
\usepackage{IEEEtrantools}
\usepackage{mathrsfs}
\usepackage[utf8]{inputenc}
\usepackage[english]{babel}
 
\usepackage{comment}
\usepackage{csquotes}

\usepackage[shortlabels]{enumitem}

\usepackage{bbm}

\makeatletter
\def\paragraph{\@startsection{paragraph}{4}%
  \z@\z@{-\fontdimen2\font}%
  {\normalfont\bfseries}}
\makeatother

\usepackage{colonequals}

\newcommand{\QQ}{\mathbb{Q}}

\newcommand{\defeq}{\colonequals}

\renewcommand{\tilde}{\widetilde}

\DeclareMathOperator{\Gal}{Gal}    

\numberwithin{equation}{subsection}

\newtheorem{theorem}[equation]{Theorem}
\newtheorem{proposition}[equation]{Proposition}
\newtheorem{lemma}[equation]{Lemma}
\newtheorem{corollary}[equation]{Corollary}
\newtheorem*{corollary*}{Corollary}

\newtheorem{assumption}[equation]{Assumption}

\newcounter{alphalabels}

\newtheorem{theoremx}[alphalabels]{Theorem}

\theoremstyle{definition}

\newtheorem{definition}[equation]{Definition}
\newtheorem{notation}[equation]{Notation}

\newtheorem{convention}[equation]{Convention}

\theoremstyle{remark}
\newtheorem{remark}[equation]{Remark}
\newtheorem{example}[equation]{Example}

\usepackage{tikz-cd}
\usepackage{mathrsfs}

\DeclareFontFamily{U}{wncy}{}
\DeclareFontShape{U}{wncy}{m}{n}{<->wncyr10}{}
\DeclareSymbolFont{mcy}{U}{wncy}{m}{n}
\DeclareMathSymbol{\sha}{\mathord}{mcy}{"58}

\newcommand{\ide}[1]{\mathfrak{#1}}
\newcommand{\mbb}[1]{\mathbb{#1}}
\newcommand{\ordd}{\mathcal{O}}
\newcommand{\opn}[1]{\operatorname{#1}}
\newcommand{\hatot}{\hat{\otimes}}
\newcommand{\mbf}[1]{\mathbf{#1}}

\newcommand{\smat}[1]{\left(\begin{smallmatrix} #1 \end{smallmatrix}\right)}

\newcommand{\tbyt}[4]{\left( \begin{array}{cc} #1 & #2 \\ #3 & #4 \end{array} \right)}

\newcommand{\tbyts}[4]{\left( \begin{smallmatrix} #1 & #2 \\ #3 & #4 \end{smallmatrix} \right)}

\newcommand{\Addresses}{{% additional braces for segregating \footnotesize
  \bigskip
  \footnotesize

  \textsc{Mathematical Institute, University of Oxford, Woodstock Road, Oxford OX2 6GG, United Kingdom}\par\nopagebreak
  \textit{E-mail address}: \texttt{andrew.graham@maths.ox.ac.uk}

}}

\usepackage{stmaryrd}

\newcommand{\et}{\mathrm{\acute{e}t}}
\newcommand{\Qpb}{\overline{\QQ}_p} 

\newcommand{\bincoeff}[2]{\genfrac(){0pt}{0}{#1}{#2}}

\title[ ]{Unitary Friedberg--Jacquet periods and anticyclotomic p-adic L-functions}
\author{Andrew Graham}
\date{}

\begin{document}
\begin{abstract}
    We extend the construction of the $p$-adic $L$-function interpolating unitary Friedberg--Jacquet periods in previous work of the author to include the $p$-adic variation of Maass--Shimura differential operators. In particular, we develop a theory of nearly overconvergent automorphic forms in higher degrees of coherent cohomology for unitary Shimura varieties generalising previous work for modular curves. The construction of this $p$-adic $L$-function can be viewed as a higher-dimensional generalisation of the work of Bertolini--Darmon--Prasanna and Castella--Hsieh, and the inclusion of this extra variable arising from the $p$-adic iteration of differential operators will play a key role in relating values of this $p$-adic $L$-function to $p$-adic regulators of special cycles on unitary Shimura varieties. 
\end{abstract}

\maketitle

\tableofcontents

\section{Introduction}

Let $E/\mbb{Q}$ be an imaginary quadratic number field and $p$ an odd prime which splits in this extension. Let $f$ be a cuspidal newform of level $\Gamma_0(N)$ and (even) weight $k \geq 2$, where $N$ is coprime to $p$ and divisible only by primes which split in $E/\mbb{Q}$. In \cite{bertolini2013, CastellaHsieh}, the authors construct an anticyclotomic $p$-adic $L$-function which $p$-adically interpolates the square-root central critical $L$-values of the Rankin $L$-series $L(f, \chi, s)$ as $\chi$ runs through a certain range of anticyclotomic Hecke characters of $E$. To be more precise, let $E_{\infty}/E$ denote the anticyclotomic $\mbb{Z}_p$-extension. Then the authors construct a $p$-adic measure $\mathscr{L}_p(f, -) \in \overline{\mbb{Z}}_p[\![ \Gal(E_{\infty}/E) ]\!]$ such that
\[
\mathscr{L}_p(f, \hat{\chi})^2 = (\star) \cdot L(f, \chi, k/2)
\]
for any anticyclotomic character $\chi \colon E^{\times} \backslash \mbb{A}_E^{\times} \to \mbb{C}^{\times}$ of $p$-power conductor and infinity-type $(j, -j)$ with $j \geq k/2$, where $\hat{\chi} \colon \Gal(E_{\infty}/E) \to \Qpb^{\times}$ denotes the associated $p$-adic character via class field theory.\footnote{Here $(\star)$ denotes a suitable product of Euler factors and periods. In \cite{bertolini2013}, a $p$-adic $L$-function for odd weight modular forms is also constructed, however one can no longer take $\chi$ to be an anticyclotomic character of infinity-type $(j, -j)$ (the central value $L(f, \chi, k/2)$ is not Deligne-critical when $k$ is odd). For simplicity, we therefore stick to the even weight case in this introduction.} This $p$-adic $L$-function plays an important role in the study of the Bloch--Kato conjecture for anticyclotomic twists of the $p$-adic Galois representation associated with $f$, due to the striking relation between $\mathscr{L}_p(f, -)$ and generalised Heegner cycles \emph{outside} the above region of interpolation. For example, one can exploit this property to establish results towards the Bloch--Kato conjecture in analytic rank zero (see \cite[Theorem A]{CastellaHsieh}).

The key input in the construction of $\mathscr{L}_p(f, -)$ is an integral formula (Waldspurger's formula) for the central critical $L$-values in terms of a toric period involving the character $\chi$ and the nearly holomorphic modular form $\delta^{j-k/2} f$, where $\delta$ denotes the Maass--Shimura differential operator. This toric period can be interpreted algebraically by taking a $\chi$-weighted sum of values of the coherent cohomology class 
\[
\delta^{j-k/2} \eta_{f} \in \opn{H}^0\left( X_1(N), \omega^{j+k/2} \otimes \opn{Sym}^{j-k/2} \mathcal{H} \right)
\]
associated with $\delta^{j-k/2} f$ at a certain collection of CM points $\opn{CM}_E$ inside the modular curve $X_1(N)$ (here $\mathcal{H}$ denotes the first relative de Rham cohomology of the universal elliptic curve over $X_1(N)$). The idea for constructing the $p$-adic $L$-function is to then $p$-adically interpolate this weighted sum, and in particular, $p$-adically interpolate powers of the Maass--Shimura operator $\delta$. Since $p$ splits in $E/\mbb{Q}$, the collection of CM points $\opn{CM}_E$ lie in the ordinary locus $X_1(N)^{\opn{ord}} \subset X_1(N)$, and since $\delta^{j-k/2} \eta_{f}$ is a coherent cohomology class in degree zero, one can first restrict this class to a section over the ordinary locus (irrespective of any $p$-adic slope condition on $f$) before evaluating at CM points. Therefore it suffices to $p$-adically interpolate $\delta$ on the space of $p$-adic modular forms, and using the unit root splitting, this amounts to studying the $p$-adic properties of the (simpler) Atkin--Serre operator $\Theta = q \frac{d}{dq}$. Furthermore, one of the key properties used in establishing the interpolation property is that one has a canonical splitting of the Hodge filtration over $\opn{CM}_E$ which coincides with both the unit root splitting and the real-analytic splitting. 

This method has been generalised to allow the variation of the modular form in Hida/Coleman families \cite{Castella20, JLZHeegner}, and to the setting of Hilbert modular forms \cite{Hsieh14}. In this paper, we consider a different kind of generalisation. To describe this, we first introduce some notation. Let $F^+$ be a totally real number field of degree $>1$, and consider the compositum $F = E F^+$ (so $F/F^+$ is a degree two, totally imaginary extension). Let $n \geq 2$ be an integer, and let $W$ denote a $2n$-dimensional Hermitian space over $F$ with signatures at the infinite places of the form 
\[
\{ (1, 2n-1), (0, 2n), \dots, (0, 2n) \} .
\]
Suppose that $W_1 \subset W$ is an $n$-dimensional Hermitian subspace, with signatures at the infinite places given by $\{(1, n-1),(0, n), \dots, (0, n)\}$. Assume that $p$ splits completely in $F/\mbb{Q}$, and let $\mbf{G}_0$ denote the unitary group associated with the Hermitian space $W$. We let $\mbf{H}_0 \subset \mbf{G}_0$ denote the subgroup preserving the decomposition $W = W_1 \oplus W_2 \defeq W_1 \oplus W_1^{\perp}$. Let $\pi$ be a cuspidal automorphic representation of $\mbf{G}_0(\mbb{A})$, and for an anticyclotomic character of $F$, let $L(\pi, \chi, s)$ denote the $L$-function attached to the twist of the (standard) Galois representation associated with $\pi$ by $\hat{\chi}$. We are interested in constructing a $p$-adic $L$-function interpolating the (square-roots of the) central critical $L$-values of $L(\pi, \chi, s)$ as $\chi$ runs through a certain range of anticyclotomic characters.

The strategy for producing such a $p$-adic $L$-function is to generalise the method above to the setting of unitary groups. In this case, the toric periods are replaced with \emph{unitary Friedberg--Jacquet periods} (see \eqref{IntroEqn:UFJperiod} below) whose inputs involve higher dimensional analogues of the Maass--Shimura operator. These automorphic periods are a variant of the linear periods studied by Friedberg--Jacquet \cite{FJ93}.

Using work of Harris \cite{HarrisPartial} and Su \cite{Su19}, it turns out these automorphic periods can be interpreted algebraically via the coherent cohomology of the pair of unitary Shimura varieties associated with (the similitude versions of) the groups $\mbf{H}_0 \subset \mbf{G}_0$ (which is a higher dimensional analogue of $\opn{CM}_E \subset X_1(N)$). One can then hope to $p$-adically interpolate this algebraic reinterpretation, following a similar strategy as in \cite{LPSZ, LZBK21}. However, there are several key differences between this setting and the case of modular forms when $n \geq 2$:
\begin{itemize}
    \item The Shimura variety associated with $\mbf{H}_0$ no longer lies inside the ordinary locus of the Shimura variety associated with $\mbf{G}_0$ and we are therefore forced to consider $p$-adic iterations of Maass--Shimura operators on a suitable space of nearly overconvergent automorphic forms.
    \item The coherent cohomology class, which is the higher-dimensional analogue of $\delta^{j-k/2} \eta_{f}$, is a coherent cohomology class in degree $n-1$, so we also need a version of nearly overconvergent automorphic forms in higher degrees of coherent cohomology.
    \item One no longer has a canonical splitting of the Hodge filtration over the Shimura variety associated with $\mbf{H}_0$ which coincides with a real-analytic or $p$-adic splitting, so an extra argument is required to show that the pullback of the cohomology class is overconvergent.
\end{itemize}

The construction of this $p$-adic $L$-function was initiated in \cite{UFJ} by establishing functoriality of Boxer--Pilloni's higher Coleman theory \cite{BoxerPilloni} for unitary groups, however there was a restriction on the weight of $\pi$ and infinity-type of $\chi$ due to the absence of Maass--Shimura differential operators. In this article, we extend the construction in \cite{UFJ} to include the $p$-adic variation of these differential operators using a generalisation of the results in \cite{DiffOps}. In particular, we extend \cite{DiffOps} in two different ways: we construct spaces of nearly overconvergent automorphic forms in higher degrees of coherent cohomology (\emph{op.cit.}\ is only for $\opn{H}^0$); and we $p$-adically interpolate Maass--Shimura differential operators on higher dimensional Shimura varieties (\emph{op.cit.}\ is only for modular curves). Both of these aspects enable us to overcome the higher-dimensional issues highlighted above.

In addition to this, this extra variable allows one to consider values of this $p$-adic $L$-function in a certain region of twists (disjoint from the region of interpolation) where Euler system classes exist for the associated Galois representation (see \cite{ACES} for the construction of this Euler system when $F$ is imaginary quadratic). It is expected that one can prove an \emph{explicit reciprocity law} relating the image of these Euler system classes under a $p$-adic Abel--Jacobi map and values of this $p$-adic $L$-function -- generalising the results in \cite{bertolini2013} -- which would lead to new cases of the Bloch--Kato conjecture for anticyclotomic twists of $\pi$ (see \S \ref{ExpRelationWithESIntro}). 

We note that the $p$-adic $L$-function in this article interpolates unitary Friedberg--Jacquet periods, and the precise connection between these periods and values of the $L$-function needed for this $p$-adic $L$-function is still conditional on forthcoming work of Leslie--Xiao--Zhang \cite{LXZufjIII}. However, there is an overwhelming amount of evidence towards this (see \cite{ChenGan, PWZ19, LXZufjI, LXZufjII}). Furthermore, there are also alternative constructions of $p$-adic $L$-functions for automorphic representations of unitary groups, such as in \cite{EHLS, DLpadicheights} where one obtains both cyclotomic and anticyclotomic variables, however the regions of interpolation for these $p$-adic $L$-functions differ from the region of interpolation in this article (in analogy with the triple product setting, one could refer to the $p$-adic $L$-function in this paper as ``unbalanced'', whereas the $p$-adic $L$-functions in \cite{EHLS, DLpadicheights} are ``balanced'').

\begin{remark}
    We note that the construction in \cite{EHLS} similarly uses the $p$-adic interpolation of Maass--Shimura differential operators on unitary Shimura varieties, however there are two key differences between \emph{op.cit.}\ and this article. Firstly, the authors work with $p$-adic automorphic forms (not nearly overconvergent forms), which simplifies the $p$-adic interpolation, but requires an ordinarity assumption at $p$ on the automorphic representations. Secondly, the authors only need to work with coherent cohomology in $\opn{H}^0$ (and implicitly, top-degree cohomology via Serre duality); hence they do not need to $p$-adically interpolate Maass--Shimura operators in higher degrees of coherent cohomology. It seems likely that one can relax the ordinarity assumption in \emph{op.cit.}\ by extending \cite{DiffOps} to unitary Shimura varieties of signatures $(n, n)$ at all places (following a similar strategy as in this article, which extends \cite{DiffOps} to unitary Shimura varieties of signatures $(1, 2n-1)$, $(0, 2n)$, $\dots$, $(0, 2n)$).
\end{remark}

\subsection{Statement of the main results} \label{StatementOfMainResultsSSecIntro}

We now give a more detailed description of the main results of this article. For simplicity, we explain this in the setting of unitary groups without similitude (whereas in the main body of this article we work with unitary similitude groups and PEL Shimura varieties). With notation as above, fix a CM type $\Psi$ of $F$ and let $\tau_0 \in \Psi$ denote the place where $W$ has signature $(1, 2n-1)$. Let $p$ be an odd prime which splits completely in $F/\mbb{Q}$ and fix an identification $\iota_p \colon \mbb{C} \cong \Qpb$. For any $\tau \in \Psi$, let $\ide{p}_{\tau}$ denote the prime of $F$ lying above $p$ determined by embedding $\iota_p \circ \tau$, and let $\ide{p}_{\bar{\tau}}$ denote its complex conjugate. Since $p$ splits completely, we have an identification $\mbf{G}_{0, \mbb{Q}_p} = \prod_{\tau \in \Psi} \opn{GL}_{2n}$.

Let $\pi$ be a cuspidal automorphic representation of $\mbf{G}_0(\mbb{A})$ such that its component $\pi_{\infty}$ at the infinite place lies in the discrete series $L$-packet parameterised by a self-dual dominant character $\lambda \in X^*(T)^+$. More precisely, if $T \subset \mbf{G}_{0, \mbb{C}} = \prod_{\tau \in \Psi} \opn{GL}_{2n}$ denotes the standard diagonal torus, then $\lambda$ can be described as a tuple of integers $(\lambda_{1, \tau}, \dots, \lambda_{2n, \tau})$ with $\lambda_{1, \tau} \geq \lambda_{2, \tau} \geq \cdots \geq \lambda_{2n, \tau}$ and $\lambda_{i, \tau} = -\lambda_{2n+1-i, \tau}$ for all $i=1, \dots, 2n$ and $\tau \in \Psi$. Consider the Levi subgroup $M = \left(\opn{GL}_1 \times \opn{GL}_{2n-1}\right) \times \prod_{\tau \neq \tau_0} \opn{GL}_{2n} \subset \mbf{G}_{0, \mbb{C}}$ and let ${^MW}$ denote the set of minimal length representatives of the quotient $W_{M}\backslash W_{\mbf{G}_0}$ of Weyl groups. We assume that the Harish-Chandra parameter of $\pi_{\infty}$ is of the form $w_n \cdot (\lambda + \rho)$, where $\rho$ is the half-sum of the positive roots (with respect to the upper-triangular Borel) in $\mbf{G}_{0, \mbb{C}}$, and $w_n$ is the unique element in ${^MW}$ of length $n$. Finally, we assume that $\pi$ is unramified at $p$ and we fix a set $S$ containing $\infty$ and all primes where $\pi$ is ramified (the primes where there does not exist a maximal special subgroup with non-trivial fixed points on the corresponding local component of $\pi$).

The main objects of study in this article are the following unitary Friedberg--Jacquet periods, namely, for any anticyclotomic algebraic Hecke character $\chi \colon F^{\times} \mbb{A}_{F^+}^{\times} \backslash \mbb{A}_F^{\times} \to \mbb{C}^{\times}$ we set
\begin{equation} \label{IntroEqn:UFJperiod}
\mathscr{P}_{\pi, \chi}(\phi) \defeq \int_{[\mbf{H}_0]} \phi(h) \chi'\left( \frac{\opn{det}h_2}{\opn{det}h_1} \right) dh, \quad \quad h = (h_1, h_2) \in [\mbf{H}_0] \defeq \mbf{H}_0(\mbb{Q}) \backslash \mbf{H}_0(\mbb{A}), \quad \phi \in \pi,
\end{equation}
where $\chi' \colon \opn{Res}_{F^+/\mbb{Q}}\opn{U}(1)(\mbb{A}) \to \mbb{C}^{\times}$ is the unique character satisfying $\chi(z) = \chi'(\bar{z}/z)$ for any $z \in \mbb{A}_F^{\times}$. Here $h_1$ (resp. $h_2$) denotes the component of $h$ lying in the unitary group for $W_1$ (resp. $W_2=W_1^{\perp}$), and $dh$ denotes the Tamagawa measure. These periods are conjectured to be related to central values of the (standard) $L$-function $L(\Pi \otimes \chi, s)$, where $\Pi$ denotes the base-change of $\pi$ to an automorphic representation of $\opn{GL}_{2n}(\mbb{A}_F)$. More precisely, one expects that $\mathscr{P}_{\pi, \chi} \neq 0$ if and only if $\Pi$ is of symplectic type, $\opn{Hom}_{\mbf{H}_0(\mbb{A}_f)}(\pi_f, \chi' \boxtimes (\chi')^{-1}) \neq 0$, and $L(\Pi \otimes \chi, 1/2) \neq 0$ (see \cite[Conjecture 7.4]{ChenGan}). Given this conjecture, it is therefore natural to study the periods $\mathscr{P}_{\pi, \chi}(\phi)$ (for suitable choices of test data $\phi \in \pi$) in lieu of the central $L$-values $L(\Pi \otimes \chi, 1/2)$. 

Firstly, we introduce the set of anticyclotomic characters over which we intend $p$-adically interpolate the unitary Friedberg--Jacquet periods. Let $\Sigma_{\pi}$ denote the set of anticyclotomic algebraic Hecke characters $\chi \colon F^{\times} \mbb{A}_{F^+}^{\times} \backslash \mbb{A}_F^{\times} \to \mbb{C}^{\times}$ such that:
\begin{itemize}
    \item The $\infty$-type of $\chi$ is equal to $(\lambda_{n, \tau_0} + 1 + j_{\tau_0}, -(\lambda_{n, \tau_0} + 1 + j_{\tau_0}))$ at the place $\tau_0$, and equal to $(j_{\tau}, -j_{\tau})$ at any place $\tau \neq \tau_0$, for some tuple of integers $j = (j_{\tau}) \in \prod_{\tau \in \Psi} \mbb{Z}$ satisfying
    \[
    0 \leq j_{\tau_0} \leq \lambda_{n-1, \tau_0} - \lambda_{n, \tau_0}, \quad \quad 0 \leq j_{\tau} \leq \lambda_{n, \tau} \; \; (\tau \neq \tau_0).
    \]
    \item The conductor of $\chi$ is of the form $\prod_{\tau \in \Psi} (\ide{p}_{\tau} \ide{p}_{\bar{\tau}})^{c_{\tau}}$ with $c_{\tau_0} \geq 1$. We let $e = (e_{\tau}) \in \prod_{\tau \in \Psi} \mbb{Z}_{\geq 1}$ denote the tuple of integers given by $e_{\tau} = \opn{max}(1, c_{\tau})$.
\end{itemize}
Let $F_{p^{\infty}} / F$ denote the maximal anticyclotomic abelian extension which is unramified away from $p$. For any $\chi \in \Sigma_{\pi}$, we let $\hat{\chi} \colon \Gal(F_{p^{\infty}}/F) \to \Qpb^{\times}$ denote the corresponding continuous character via class field theory.

The test data we consider in this article is of the following form:
\begin{itemize}
    \item Let $K_{\infty} \subset \mbf{G}_0(\mbb{A})$ denote the maximal compact subgroup whose complexification equals $M(\mbb{C})$. We fix $\phi_{\infty} \in \pi_{\infty}$ a (non-zero) element in the minimal $K_{\infty}$-type of $\pi_{\infty}$ which is an eigenvector under the action of $K_{\infty} \cap \mbf{H}_0(\mbb{A})$.
    \item Let $K \subset \mbf{G}_0(\mbb{A}_f)$ be a compact open subgroup of the form $K = K_S \cdot \prod_{\ell \not\in S} K_\ell$, where
    \[
    K_S \subset \prod_{\substack{\ell \in S \\ \text{ finite }}} \mbf{G}_0(\mbb{Q}_{\ell}), \quad \quad K_{\ell} \subset \mbf{G}_0(\mbb{Q}_{\ell}) \text{ maximal special, }
    \]
    such that $\pi_f^K \neq 0$. We fix non-zero vectors $\phi_S \in \pi_{f, S}^{K_S}$, and $\phi_{\ell} \in \pi_{\ell}^{K_{\ell}}$ for $\ell \notin S \cup \{ p \}$, where $\pi_{f,S} = \bigotimes_{\substack{\ell \in S \\ \text{ finite }}} \pi_{\ell}$.
    \item For any $\tau \in \Psi$ and $i =1, \dots, 2n$, let $t_{p, i, \tau} \in \mbf{G}_0(\mbb{Q}_p)$ denote the diagonal matrix which is the identity outside the $\tau$-component, and in the $\tau$-component is given by $\opn{diag}(p, \dots, p, 1, \dots, 1)$ with $i$ lots of $p$. We fix a $p$-stabilisation $\phi_p \in \pi_p$, i.e., an Iwahori-fixed vector which is an eigenvector for the action of the $U_p$-Hecke operators associated with $t_{p, i, \tau}$. Explicitly, $\phi_p$ is fixed by the standard upper-triangular Iwahori subgroup $\opn{Iw} \subset \mbf{G}_0(\mbb{Q}_p)$, and there are (necessarily non-zero) complex numbers $\alpha_{i, \tau} \in \mbb{C}^{\times}$ such that
    \[
    [ \opn{Iw} \cdot t_{p, i, \tau} \cdot \opn{Iw}] \cdot \phi_p = \alpha_{i, \tau} \phi_p
    \]
    where $[ \opn{Iw} \cdot t_{p, i, \tau} \cdot \opn{Iw}]$ is the Hecke operator associated with $t_{p, i, \tau}$. We assume that this eigensystem is \emph{small slope}, i.e., for any $\tau \in \Psi$ and $i = 1, \dots, 2n-1$ one has
    \[
    v_p( \iota_p(\alpha_{i, \tau}^{\circ}) ) < \lambda_{i, \tau} - \lambda_{i+1, \tau} + 1, \quad \quad \quad \alpha_{i, \tau}^{\circ} \defeq \lambda(t_{p, i, \tau}) \alpha_{i, \tau},
    \]
    where $v_p$ denotes the $p$-adic valuation normalised such that $v_p(p) = 1$.
\end{itemize} 
For any $\chi \in \Sigma_{\pi}$, we let $\phi^{[j]}_e \in \pi$ denote the automorphic form given by
\[
\phi^{[j]}_e = (\Delta^{[j]}_{\kappa} \cdot \phi_{\infty}) \otimes \phi_S \otimes \bigotimes_{\ell \not\in S \cup \{p\}} \phi_{\ell} \otimes (u_{\opn{sph}} t_p^e \cdot \phi_p)
\]
where $\Delta^{[j]}_{\kappa} \in \mathcal{U}(\ide{g}_{\mbb{C}})$ is a certain differential operator depending on the tuple $j$ and the weight $\kappa = -w_M^{\opn{max}} [ w_n (\lambda + \rho) - \rho ]$ (see Definition \ref{AppendixDefOfDELTAkappaJDiffOp}), $u_{\opn{sph}} \in \mbf{G}_0(\mbb{Q}_p)$ is a certain representative of the open orbit of the lower-triangular Borel subgroup of $\mbf{G}_0(\mbb{Q}_p)$ acting on $\mbf{H}_0(\mbb{Q}_p) \backslash \mbf{G}_0(\mbb{Q}_p)$, and $t_p^e \in \mbf{G}_0(\mbb{Q}_p)$ denotes the diagonal matrix which in the $\tau$-component is given by $\opn{diag}(p^{e_{\tau}(2n-1)}, p^{e_{\tau}(2n-2)}, \dots, p^{e_{\tau}}, 1)$. Here $w_M^{\opn{max}} \in W_M$ denotes the longest Weyl element. One should view the operator $\Delta_{\kappa}^{[j]}$ as a higher-dimensional analogue of a certain power of the Maass--Shimura differential operator -- indeed, if we identify $\ide{g}_{\mbb{C}} = \bigoplus_{\tau \in \Psi} \ide{gl}_{2n}$, then up to a non-zero rational number the operator $\Delta_{\kappa}^{[j]}$ in the $\tau$-component is given by the $j_{\tau}$-th power of the determinant operator
\[
\opn{det}_{\tau} = \sum_{\sigma \in S_n} \opn{sgn}(\sigma) \prod_{i=1}^n E_{i, n+\sigma(i)} \; \in \; \mathcal{U}(\ide{gl}_{2n})
\]
where $E_{a, b} \in \ide{gl}_{2n}$ denotes the elementary matrix with $1$ in the $(a, b)$-th place. Furthermore, the local component $u_{\opn{sph}} t_p^e \cdot \phi_p$ is precisely the same element which appears in the study of twisted local zeta integrals associated with Shalika models (see \cite{5author, ClassicalLocus}). 

We now state the first main result of this article.

\begin{theoremx} \label{FirstTheoremIntro}
    There exists a locally analytic distribution $\mathscr{L}_{p, \phi}(\pi, -) \in \mathscr{D}^{\opn{la}}(\Gal(F_{p^{\infty}}/F), L)$ such that for any $\chi \in \Sigma_{\pi}$
    \[
    \mathscr{L}_{p, \phi}(\pi, \hat{\chi}) = (\star) \cdot (2\pi i)^{-(n-1)} \cdot \mathscr{E}_p(\pi, \chi) \cdot \mathscr{P}_{\pi, \chi}(\phi_e^{[j]})
    \]
    where:
    \begin{itemize}
        \item $L/\mbb{Q}_p$ is a sufficiently large finite extension and $(\star)$ is a non-zero rational number independent of $\pi$ and $\chi$;
        \item the factor $\mathscr{E}_p(\pi, \chi)$ is given by
        \[
         \mathscr{E}_p(\pi, \chi) = p^{-e_{\tau_0}} \left( \frac{\alpha_{n, \tau_0}}{\alpha_{n-1, \tau_0}} \right)^{e_{\tau_0}} \chi_{\ide{p}_{\bar{\tau}_0}}(-1) \chi_{\ide{p}_{\bar{\tau}_0}}(p)^{-e_{\tau_0}} \mathscr{G}(\chi_{\ide{p}_{\bar{\tau}_0}}) \left( \prod_{\tau \in \Psi} \chi_{\ide{p}_{\bar{\tau}}}(-1)^n \right) (\alpha_p^e \delta_B(t_p^e))^{-1} 
        \]
        where $\alpha_p^e$ is the eigenvalue corresponding to the action of $[\opn{Iw} \cdot t_p^e \cdot \opn{Iw}]$ on $\phi_p$, $\delta_B$ denotes the modulus function associated with the upper-triangular Borel subgroup of $\mbf{G}_0(\mbb{Q}_p)$, $\chi_{\ide{p}_{\bar{\tau}}}$ denotes the restriction of $\chi$ to $F_{\ide{p}_{\bar{\tau}}}^{\times} \cong \mbb{Q}_p^{\times}$, and $\mathscr{G}(\chi_{\ide{p}_{\bar{\tau}_0}})$ is the Gauss sum associated with the character $\chi_{\ide{p}_{\bar{\tau}_0}}$.
    \end{itemize}
\end{theoremx}

Before describing the ingredients that go into the construction of $\mathscr{L}_{p, \phi}(\pi, -)$, let us first make a few remarks.

\begin{remark}
    One can extend the result in Theorem \ref{FirstTheoremIntro} to include characters with additional tame ramification depending on the compact open subgroup $K_S$.
\end{remark}

\begin{remark}
    The appearance of the factor $\mathscr{E}_p(\pi, \chi)$ is due to the fact that, in order to $p$-adically interpolate the differential operators $\Delta^{[j]}_{\kappa}$, one must first perform a certain ``$p$-depletion'' to the class $\phi$. The analysis of this $p$-depletion operator seems to be significantly easier if one assumes the character $\chi_{\ide{p}_{\bar{\tau_0}}}$ is ramified, and is the reason why the condition $c_{\tau_0} \geq 1$ appears in the definition of $\Sigma_{\pi}$. However, we expect that with more work that one can prove an interpolation property for unramified characters too.
\end{remark}

\begin{remark}
    The restrictions on the infinity-types of characters in $\Sigma_{\pi}$ are a generalisation of the conditions appearing in \cite{bertolini2013}. However, one key difference (when $n \geq 2$) is that this interpolation set is not Zariski dense in the weight space of continuous characters of $\Gal(F_{p^{\infty}}/F)$ (so the locally analytic distribution $\mathscr{L}_{p, \phi}(\pi, -)$ is not uniquely determined by this interpolation property).
\end{remark}

\begin{remark} \label{CPRRemarkIntro}
    For the locally analytic distribution $\mathscr{L}_{p, \phi}(\pi, -)$ to have a (potentially) non-trivial interpolation property, one should at the very least impose the additional conditions that $\Pi$ is of symplectic type, $\opn{Hom}_{\mbf{H}_0(\mbb{A}_f)}(\pi_f, \mbb{C}) \neq 0$, the Satake parameters for $\pi_p$ are distinct in each component indexed by $\Psi$, and the choice of $p$-stablisation $\phi_p$ is spin (see \cite[\S 6]{5author}). Assuming these conditions hold, write
    \[
    \pi_p = \bigotimes_{\tau \in \Psi} \opn{Ind}(\theta_{\tau})
    \]
    where $\opn{Ind}(\theta_{\tau})$ denotes the smooth irreducible representation of $\opn{GL}_{2n}(\mbb{Q}_p)$ obtained as the normalised induction of a smooth unramified character $\theta_{\tau} \colon (\mbb{Q}_p^{\times})^{\oplus 2n} \to \mbb{C}^{\times}$. For any $i=1, \dots, 2n$, let $\theta_{i, \tau} \colon \mbb{Q}_p^{\times} \to \mbb{C}^{\times}$ denote the $i$-th component of $\theta_{\tau}$. We may assume that $\theta_{i, \tau} = \theta_{2n+1-i, \tau}^{-1}$ and that
    \[
    \alpha_{i, \tau} = \prod_{j=1}^i p^{n-j+\frac{1}{2}} \theta_{j, \tau}(p) .
    \]
    Similar to Waldspurger's formula, it is expected that $|\mathscr{P}_{\pi, \chi}(\psi)|^2$ should decompose as a product of local zeta integrals for the automorphic representation $\Pi$ depending on the input vector $\psi$. More precisely, one expects a decomposition of the following shape:
    \[
    \mathscr{L}_{p, \phi}(\pi, \hat{\chi})^2 = (\star) \cdot \mathscr{E}_p(\pi, \chi)^2 \cdot \left( \prod_{\tau \in \Psi} \zeta(\psi_{\tau}, \chi_{\ide{p}_{\bar{\tau}}}^{-1}, 1/2) \cdot \zeta(\psi_{\tau}, \chi_{\ide{p}_{\bar{\tau}}}, 1/2) \right) \cdot C^{(p)} \cdot L^{(p)}(\Pi \otimes \chi, 1/2)
    \]
    where $\psi = \phi_{e}^{[j]}$, $\psi_{\tau} \in \opn{Ind}(\theta_{\tau})$ denotes the component at $p$ indexed by $\tau$, $\zeta(\cdots)$ is the twisted local zeta integral associated with a Shalika model for a smooth representation of $\opn{GL}_{2n}(\mbb{Q}_p)$, $C^{(p)}$ is a product of factors depending on $\Pi$ away from $p$, and $L^{(p)}(\Pi \otimes \chi, 1/2)$ denotes the value of the $L$-function with Euler factors at $p$ removed. Suppose that $c_{\tau} \geq 1$ for all $\tau \in \Psi$ for simplicity. Then the calculations in \cite[\S 5]{5author} show that 
    \[
    \mathscr{E}_p(\pi, \chi)^2 \cdot \left( \prod_{\tau \in \Psi} \zeta(\psi_{\tau}, \chi_{\ide{p}_{\bar{\tau}}}^{-1}, 1/2) \cdot \zeta(\psi_{\tau}, \chi_{\ide{p}_{\bar{\tau}}}, 1/2) \right)
    \]
    is (up to a constant only depending on the parity of $\chi_{\ide{p}_{\bar{\tau}_0}}$) equal to
    \begin{align}
     \varepsilon(\theta_{n, \tau_0}\chi_{\ide{p}_{\bar{\tau}_0}}^{-1}, 1/2) \varepsilon(\theta_{n+1, \tau_0}\chi_{\ide{p}_{\bar{\tau}_0}}, 1/2)^{-1} \prod_{\tau \in \Psi} \prod_{i=1}^n \varepsilon(\theta_{i, \tau}\chi_{\ide{p}_{\bar{\tau}}}^{-1}, 1/2)^{-1} \varepsilon(\theta_{i, \tau}\chi_{\ide{p}_{\bar{\tau}}}, 1/2)^{-1}  \label{Eqn:ProductOfEpfactors}
    \end{align} 
    where $\varepsilon(\eta, s) = \mathscr{G}(\eta^{-1}) \cdot \eta(-p^c) p^{-cs}$ denotes the Langlands--Deligne local factor associated with any smooth character $\eta \colon \mbb{Q}_p^{\times} \to \mbb{C}^{\times}$ of conductor $p^c$ (see, e.g., \cite[p.\ 572]{CastellaHsieh}). Here we have used the calculation
    \begin{align*}
    \mathscr{E}_p(\pi, \chi) &= \varepsilon(\theta_{n, \tau_0} \chi_{\ide{p}_{\bar{\tau}_0}}^{-1}, 1/2) \cdot \left( \prod_{\tau \in \Psi} \chi_{\ide{p}_{\bar{\tau}}}(-1)^n \right) (\alpha_p^e \delta_B(t_p^e))^{-1} \\
    &= \chi_{\ide{p}_{\bar{\tau}_0}}(-1) \cdot \varepsilon(\theta_{n+1, \tau_0} \chi_{\ide{p}_{\bar{\tau}_0}}, 1/2)^{-1} \cdot \left( \prod_{\tau \in \Psi} \chi_{\ide{p}_{\bar{\tau}}}(-1)^n \right) (\alpha_p^e \delta_B(t_p^e))^{-1}. 
    \end{align*}
    
    The expression \eqref{Eqn:ProductOfEpfactors} is precisely of the shape predicted by Coates--Perrin-Riou \cite{coatesperrinriou89, coates89} (note that $L(\eta, s) = 1$ when $\eta$ is ramified). In particular, we expect that the locally analytic distribution $\mathscr{L}_{p, \phi}(\pi, -)$ should be tempered, and a $p$-adic measure if $\phi_p$ is Borel-ordinary (i.e., $v_p(\alpha_{i, \tau}^{\circ}) = 0$ for all $\tau \in \Psi$ and $i=1, \dots, 2n-1$).\footnote{In fact, a weaker ordinary condition ensuring the existence of a ``Panchishkin submodule'' should suffice for showing $\mathscr{L}_{p, \phi}(\pi, -)$ is a $p$-adic measure.}
\end{remark}

\subsubsection{Results in families} \label{ResultsInFamiliesSSecIntro}

We also construct a $p$-adic $L$-function (with the maximal amount of variation) as $\pi$ varies in a Coleman family. For this, we make the additional assumptions that the finite primes in $S$ split in $E/\mbb{Q}$, and that $\opn{dim}_{\mbb{C}}\pi_{f,S}^{K_S} = 1$. We also assume $p$-regularity, namely that the generalised eigenspace in $\pi_p^{\opn{Iw}}$ with eigensystem $\{ \alpha_{i, \tau} \}$ is one-dimensional. In other words, $\phi_p$ is a $p$-regular $p$-stabilisation which is new away from $p$. When $\pi$ satisfies these assumptions, we simply write $\mathscr{L}_{p}(\pi, -)$ for the locally analytic distribution in Theorem \ref{FirstTheoremIntro} (associated with this $p$-regular $p$-stabilisation which is new away from $p$). Under these assumptions, we prove the following theorem.

\begin{theoremx} \label{SecondTheoremIntro}
    Let $\pi$ be as in \S \ref{StatementOfMainResultsSSecIntro} which additionally satisfies the assumptions at the start of \S \ref{ResultsInFamiliesSSecIntro}. Let $\mathcal{W}$ denote the $n[F^+:\mbb{Q}]$-dimensional weight space over $\mbb{Q}_p$ parameterising self-dual continuous characters of $T(\mbb{Z}_p)$. Then there exists a sufficiently large finite extension $L/\mbb{Q}_p$ and an open affinoid neighbourhood $\Omega \defeq \opn{Spa}(\mathscr{O}_{\Omega}) \subset \mathcal{W}_L$ containing $\lambda$ such that:
    \begin{enumerate}
        \item There exists a family $\underline{\pi}$ of automorphic representations over $\Omega$ passing through $\pi$, i.e., there exists an $\mathscr{O}_{\Omega}$-valued Hecke eigensystem (for the $U_p$-Hecke operators and the Hecke operators away from $S \cup \{p\}$) and a Zariski dense subset $\Upsilon \subset \Omega(\mbb{C}_p)$ of classical weights containing $\lambda$ such that:
        \begin{itemize}
            \item the specialisation of this Hecke eigensystem at any point $x \in \Upsilon$ is the Hecke eigensystem associated with a (unique) automorphic representation $\underline{\pi}_x$ of $\mbf{G}_0(\mbb{A})$ satisfying the assumptions in \S \ref{StatementOfMainResultsSSecIntro} and the start of \S \ref{ResultsInFamiliesSSecIntro};
            \item the specialisation of this Hecke eigensystem at $\lambda$ coincides with the Hecke eigensystem associated with $\pi$.
        \end{itemize}
        In particular, such a family $\underline{\pi}$ determines (up to scalar) a $p$-regular $p$-stabilisation in $\underline{\pi}_x$ for any $x \in \Upsilon$ which is new away from $p$.
        \item There exists a locally analytic distribution $\mathscr{L}_p(\underline{\pi}, -) \in \mathscr{D}^{\opn{la}}(\Gal(F_{p^{\infty}}/F), \mathscr{O}_{\Omega})$ such that for any $x \in \Upsilon$:
        \[
        \opn{sp}_x \mathscr{L}_p(\underline{\pi}, -) = \mathscr{L}_p(\underline{\pi}_x, -)
        \]
        where $\opn{sp}_x$ denotes the specialisation at the point $x$ and $\mathscr{L}_p(\underline{\pi}_x, -)$ is the locally distribution in Theorem \ref{FirstTheoremIntro} (associated with the appropriately normalised $p$-regular $p$-stabilisation in $\underline{\pi}_x$ which is new away from $p$).
        \item The set $\Sigma_{\underline{\pi}} \defeq \bigcup_{x \in \Upsilon} \{ x \} \times \Sigma_{\underline{\pi}_x}$ is Zariski dense in the fibre product of $\Omega$ with the weight space of continuous characters of $\Gal(F_{p^{\infty}}/F)$, hence $\mathscr{L}_p(\underline{\pi}, -)$ is uniquely determined by its interpolation property at points in $\Sigma_{\underline{\pi}}$.
    \end{enumerate}
\end{theoremx}

\subsubsection{Expected relation with Euler systems} \label{ExpRelationWithESIntro}

In analogy with the Heegner cycle setting, one expects that the $p$-adic $L$-functions in Theorems \ref{FirstTheoremIntro} and \ref{SecondTheoremIntro} should be related to Euler system classes outside the regions of interpolation. More precisely, suppose that $\pi$ is a cuspidal automorphic representation of $\mbf{G}_0(\mbb{A})$ satisfying the conditions in \S \ref{StatementOfMainResultsSSecIntro} \emph{including} those in Remark \ref{CPRRemarkIntro}. Furthermore, suppose that the sign of the functional equation of $L(\Pi, s)$ is $-1$, and for simplicity, suppose that $\Pi$ is cuspidal and the eigensystem $\{ \alpha_{i, \tau} \}$ is Borel-ordinary. Let $L_p(\Pi, -; \Sigma_{\tau_0,+}^{\opn{an}}) \defeq \mathscr{L}_{p, \phi}(\pi, -)^2$ denote the square of the $p$-adic $L$-function in Theorem \ref{FirstTheoremIntro} with respect to a suitable choice of test data. Let $\rho_{\Pi} \colon \Gal(\overline{F}/F) \to \opn{GL}_{2n}(L)$ denote the $p$-adic semisimple Galois representation associated with $\Pi$, as constructed by Chenevier--Harris \cite{ChenevierHarris}.

Then, continuing with Remark \ref{CPRRemarkIntro}, one expects that $L_p(\Pi, -; \Sigma_{\tau_0,+}^{\opn{an}})$ is a $p$-adic measure on $\Gal(F_{p^{\infty}}/F)$ which interpolates the central critical $L$-values $L(\Pi \otimes \chi, 1/2) = L(\rho_{\Pi}(n) \otimes \hat{\chi}, 0)$ as $\chi$ runs through the interpolating set:
\[
\Sigma_{\tau_0, +}^{\opn{an}} \defeq \left\{ \chi \text{ anticyclotomic } : \begin{array}{c} \opn{cond}(\chi) \text{ is divisible only by primes above } p \\ \text{ the } \infty\text{-type of } \chi \text{ at } \tau \in \Psi \text{ is } (\ell_{\tau}, -\ell_{\tau}) \\ \text{ with } \lambda_{n, \tau_0}+1 \leq \ell_{\tau_0} \leq \lambda_{n-1, \tau_0}+1 \\ \text{ and } |\ell_{\tau}| \leq \lambda_{n, \tau} 
\text{ for } \tau \neq \tau_0 \end{array} \right\}
\]
Here we have changed notation slightly, and this interpolating set is larger than $\Sigma_{\pi}$ to include unramified characters and also the condition $-\lambda_{n, \tau} \leq j_{\tau} \leq 0$ for $\tau \neq \tau_0$. By our assumptions, the sign of the functional equation for $L(\rho_{\Pi}(n) \otimes \hat{\chi}, s)$ is $+1$ for $\chi \in \Sigma_{\tau_0, +}^{\opn{an}}$.

On the other hand, one can also consider the behaviour of the Galois representation $\rho_{\Pi}(n) \otimes \hat{\chi}$ for $\chi$ lying in a different ``geometric'' region of twists
\[
\Sigma^{\opn{geom}} \defeq \left\{ \chi \text{ anticyclotomic } : \begin{array}{c} \opn{cond}(\chi) \text{ is divisible only by primes above } p \\ \text{ the } \infty\text{-type of } \chi \text{ at } \tau \in \Psi \text{ is } (\ell_{\tau}, -\ell_{\tau}) \\ \text{ with } |\ell_{\tau}| \leq \lambda_{n, \tau} 
\text{ for all } \tau \in \Psi \end{array} \right\}
\]
which is disjoint from $\Sigma_{\tau_0, +}^{\opn{an}}$. In this region, the sign of the functional equation is $-1$, and by generalising the construction in \cite{ACES}, one can construct split anticyclotomic Euler systems for $\rho_{\Pi}(n) \otimes \hat{\chi}$ with $\chi \in \Sigma^{\opn{geom}}$. The idea behind this construction is to consider the ($p$-adic) \'{e}tale regulators of cycles arising from the pair of unitary Shimura varieties associated with $\mbf{G}_0$ and $\mbf{H}_0$ with appropriately chosen level subgroups. One can then extend this construction to all anticyclotomic characters by Soul\'{e} twisting. Let $z_{\chi} \in \opn{H}^1(F, \rho_{\Pi}(n) \otimes \hat{\chi})$ denote the bottom class of the Euler system (after inverting $p$); this class lies in the Bloch--Kato Selmer group if $\chi \in \Sigma^{\opn{geom}}$.

In this situation, one expects to be able to prove two kinds of \emph{explicit reciprocity laws}:
\begin{itemize}
    \item If $\chi \in \Sigma^{\opn{geom}}$, then 
    \[
    L_p(\Pi, \hat{\chi}; \Sigma_{\tau_0, +}^{\opn{an}}) = (\star)_{\chi} \cdot \opn{log}_{\opn{BK}, \eta}(\opn{loc}_{\ide{p}_{\tau_0}}(z_{\chi}))^2 ,
    \]
    generalising \cite[Theorem 5.13]{bertolini2013}. Here $\opn{log}_{\opn{BK}, \eta}$ is a certain linear functional on $\opn{H}^1_f(F_{\ide{p}_{\tau_0}}, \rho_{\Pi}(n) \otimes \hat{\chi})$ constructed from the Bloch--Kato logarithm and a fixed choice $\eta \in \opn{Fil}^0\mbf{D}_{\opn{dR}}(\rho_{\Pi}^*(1-n)|_{G_{F_{\ide{p}_{\tau_0}}}})$, and $(\star)_{\chi}$ is an explicit factor (involving the $\gamma$-factors for $\theta_{n, \tau_0} \chi_{\ide{p}_{\bar{\tau}_0}}^{-1}$ and $\theta_{n+1, \tau_0} \chi_{\ide{p}_{\bar{\tau}_0}}$).
    \item If $\chi \in \Sigma_{\tau_0, +}^{\opn{an}}$, then 
    \[
    L_p(\Pi, \hat{\chi}; \Sigma_{\tau_0, +}^{\opn{an}}) = (\star)_{\chi} \cdot \opn{exp}^*_{\opn{BK}, \eta}(\opn{loc}_{\ide{p}_{\tau_0}}(z_{\chi}))^2 ,
    \]
    generalising \cite[Corollary 5.8]{CastellaHsieh}. Here $(\star)_{\chi}$ and $\eta$ are as in the previous bullet point, and $\opn{exp}^*_{\opn{BK}, \eta}$ is a certain linear functional on $\opn{H}^1(F_{\ide{p}_{\tau_0}}, \rho_{\Pi}(n) \otimes \hat{\chi})$ built from the Bloch--Kato dual exponential map.
\end{itemize}
The key strategy for the first explicit reciprocity law will be to analyse the syntomic regulators of the cycles associated with the above pair of Shimura varieties (c.f., \cite{LZBK20} for an instance where this strategy is carried out for automorphic representations of $\opn{GSp}_4$). One will then be able to use $p$-adic deformation arguments to obtain the second explicit reciprocity law. For this latter step, the fact that the $p$-adic $L$-functions interpolate as one varies $\pi$ in a Coleman family (Theorem \ref{SecondTheoremIntro}) will be an important ingredient.

The benefit of establishing such explicit reciprocity laws is that one can obtain results towards the Bloch--Kato conjectures for anticyclotomic twists of $\rho_{\Pi}(n)$. More precisely, using forthcoming work of Jetchev--Nekov\'{a}\v{r}--Skinner \cite{JNS}\footnote{A summary of their work can be found in \cite[\S 8.1]{Alonso_Castella_Rivero_2023}.} (and the forthcoming work of Leslie--Xiao--Zhang \cite{LXZufjIII}), and under the usual ``big image'' assumptions on $\rho_{\Pi}(n)$ (see \cite[\S 8.1]{Alonso_Castella_Rivero_2023}), one will be able to deduce that:
\begin{itemize}
    \item If $\chi \in \Sigma^{\opn{geom}}$ and $L_p(\Pi, \hat{\chi}; \Sigma_{\tau_0, +}^{\opn{an}}) \neq 0$, then the Bloch--Kato Selmer group
    \[
    \opn{H}^1_f\left( F, \rho_{\Pi}(n) \otimes \hat{\chi} \right)
    \]
    is one-dimensional and generated by the class $z_{\chi}$.
    \item If $\chi \in \Sigma^{\opn{an}}_{\tau_0, +}$ and the $L$-value $L(\rho_{\Pi}(n) \otimes \hat{\chi}, 0)$ is non-zero, then the Bloch--Kato Selmer group
    \[
    \opn{H}^1_f\left( F, \rho_{\Pi}(n) \otimes \hat{\chi} \right)
    \]
    vanishes.
\end{itemize}
Indeed, for the first bullet point, the condition $L_p(\Pi, \hat{\chi}; \Sigma_{\tau_0, +}^{\opn{an}}) \neq 0$ and the first explicit reciprocity law implies that $z_{\chi} \neq 0$; hence \cite[Theorem 8.3]{Alonso_Castella_Rivero_2023} implies that $\opn{H}^1_f\left( F, \rho_{\Pi}(n) \otimes \hat{\chi} \right)$ is generated by $z_{\chi}$. For the second bullet point, if $L(\rho_{\Pi}(n) \otimes \hat{\chi}, 0)$ is non zero, then the relation between unitary Friedberg--Jacquet periods and $L$-values, in combination with the interpolation property of the $p$-adic $L$-function, implies that $L_p(\Pi, \hat{\chi}; \Sigma_{\tau_0, +}^{\opn{an}}) \neq 0$. Hence, the second explicit reciprocity law implies that $z_{\chi} \neq 0$, and \cite[Theorem 8.3]{Alonso_Castella_Rivero_2023} implies that a certain Selmer group $\opn{Sel}_{\opn{Gr}}(F, \rho_{\Pi}(n) \otimes \hat{\chi})$ (different from the Bloch--Kato Selmer group) is one-dimensional. One can then deduce that $\opn{H}^1_f\left( F, \rho_{\Pi}(n) \otimes \hat{\chi} \right)$ vanishes from an argument involving the Poitou--Tate long exact sequence (c.f., \cite[Theorem 7.9]{CastellaHsieh} or \cite[\S 9.2]{Alonso_Castella_Rivero_2023}).

We note that there is also some extra symmetry involved in the construction of the $p$-adic $L$-function in Theorem \ref{FirstTheoremIntro}. To be more precise, one could repeat the whole construction in this article to produce a $p$-adic $L$-function $L_p(\Pi, -; \Sigma_{\tau_0, -}^{\opn{an}})$ which interpolates the $L$-values $L(\rho_{\Pi}(n) \otimes \hat{\chi}, 0)$ for $\chi$ in the region
\[
\Sigma_{\tau_0, -}^{\opn{an}} \defeq \left\{ \chi \text{ anticyclotomic } : \begin{array}{c} \opn{cond}(\chi) \text{ is divisible only by primes above } p \\ \text{ the } \infty\text{-type of } \chi \text{ at } \tau \in \Psi \text{ is } (\ell_{\tau}, -\ell_{\tau}) \\ \text{ with } -(\lambda_{n-1, \tau_0}+1) \leq \ell_{\tau_0} \leq -(\lambda_{n, \tau_0}+1) \\ \text{ and } |\ell_{\tau}| \leq \lambda_{n, \tau} 
\text{ for } \tau \neq \tau_0 \end{array} \right\} .
\]
Furthermore, if we write $\Psi = \{ \tau_0, \dots, \tau_{d-1}\}$ with $d = [F^+:\mbb{Q}]$, then there is no reason to choose $\tau_0$ as the privileged place where the Hermitian space has signature $(1, 2n-1)$; one could also repeat the same construction to produce $p$-adic $L$-functions $L_p(\Pi, -; \Sigma_{\tau_i, \pm}^{\opn{an}})$ interpolating the $L$-values $L(\rho_{\Pi}(n) \otimes \hat{\chi}, 0)$ for $\chi$ in the regions
\[
\Sigma_{\tau_i, \pm}^{\opn{an}} \defeq \left\{ \chi \text{ anticyclotomic } : \begin{array}{c} \opn{cond}(\chi) \text{ is divisible only by primes above } p \\ \text{ the } \infty\text{-type of } \chi \text{ at } \tau \in \Psi \text{ is } (\pm \ell_{\tau}, \mp \ell_{\tau}) \\ \text{ with } \lambda_{n, \tau_i}+1 \leq \ell_{\tau_i} \leq \lambda_{n-1, \tau_i}+1 \\ \text{ and } |\ell_{\tau}| \leq \lambda_{n, \tau} 
\text{ for } \tau \neq \tau_i \end{array} \right\}
\]
(note that, by switching the behaviour of the signatures at the places in $\Psi$, the Hermitian space changes but the Galois representation $\rho_{\Pi}$ does not). With exactly the same methods, one should be able to prove explicit reciprocity laws for these extra $p$-adic $L$-functions  and obtain similar applications to the Bloch--Kato conjecture. For example, when $d=2$, one can produce $p$-adic $L$-functions in the four regions adjacent (horizontally and vertically) to $\Sigma^{\opn{geom}}$ in Figure \ref{RegionsOfTwistFig}, and all of these should be related to Euler systems classes in $\Sigma^{\opn{geom}}$ via an explicit reciprocity law. 

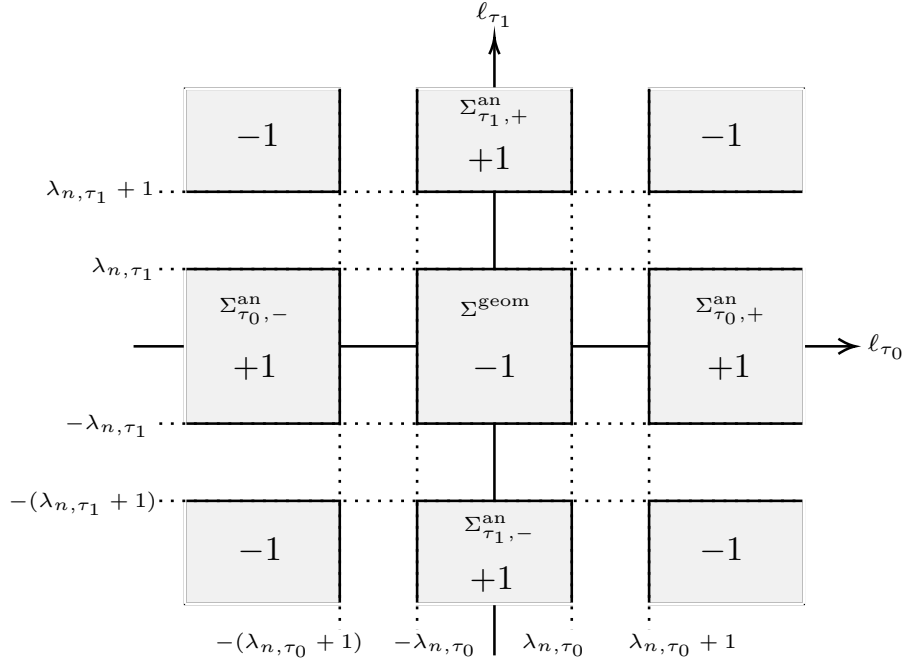
\begin{figure}[t]
    \centering
\caption{Regions of twists when $[F^+:\mbb{Q}] = 2$.}
\scalebox{1.3}{\tikzset{every picture/.style={line width=0.75pt}} %set default line width to 0.75pt        

\begin{tikzpicture}[x=0.75pt,y=0.75pt,yscale=-1,xscale=1]
%uncomment if require: \path (0,479); %set diagram left start at 0, and has height of 479

%Straight Lines [id:da8030756754913774] 
\draw [color={rgb, 255:red, 0; green, 0; blue, 0 }  ,draw opacity=0.42 ]   (350,370) -- (350,132) ;
\draw [shift={(350,130)}, rotate = 90] [color={rgb, 255:red, 0; green, 0; blue, 0 }  ,draw opacity=0.42 ][line width=0.75]    (6.56,-1.97) .. controls (4.17,-0.84) and (1.99,-0.18) .. (0,0) .. controls (1.99,0.18) and (4.17,0.84) .. (6.56,1.97)   ;
%Straight Lines [id:da3404679016580523] 
\draw [color={rgb, 255:red, 0; green, 0; blue, 0 }  ,draw opacity=0.42 ]   (210,250) -- (488,250) ;
\draw [shift={(490,250)}, rotate = 180] [color={rgb, 255:red, 0; green, 0; blue, 0 }  ,draw opacity=0.42 ][line width=0.75]    (6.56,-2.94) .. controls (4.17,-1.38) and (1.99,-0.4) .. (0,0) .. controls (1.99,0.4) and (4.17,1.38) .. (6.56,2.94)   ;
%Shape: Rectangle [id:dp5222848207862026] 
\draw  [fill={rgb, 255:red, 241; green, 241; blue, 241 }  ,fill opacity=1 ] (320,220) -- (380,220) -- (380,280) -- (320,280) -- cycle ;
%Shape: Rectangle [id:dp8909183423064054] 
\draw  [fill={rgb, 255:red, 241; green, 241; blue, 241 }  ,fill opacity=1 ] (320,150) -- (380,150) -- (380,190) -- (320,190) -- cycle ;
%Shape: Rectangle [id:dp9913161301068618] 
\draw  [fill={rgb, 255:red, 241; green, 241; blue, 241 }  ,fill opacity=1 ] (320,310) -- (380,310) -- (380,350) -- (320,350) -- cycle ;
%Shape: Rectangle [id:dp009910227281157202] 
\draw  [fill={rgb, 255:red, 241; green, 241; blue, 241 }  ,fill opacity=1 ] (230,220) -- (290,220) -- (290,280) -- (230,280) -- cycle ;
%Shape: Rectangle [id:dp4635401085780513] 
\draw  [fill={rgb, 255:red, 241; green, 241; blue, 241 }  ,fill opacity=1 ] (410,220) -- (470,220) -- (470,280) -- (410,280) -- cycle ;
%Shape: Rectangle [id:dp22462603258835656] 
\draw  [fill={rgb, 255:red, 241; green, 241; blue, 241 }  ,fill opacity=1 ] (230,150) -- (290,150) -- (290,190) -- (230,190) -- cycle ;
%Shape: Rectangle [id:dp2990787553140589] 
\draw  [fill={rgb, 255:red, 241; green, 241; blue, 241 }  ,fill opacity=1 ] (410,150) -- (470,150) -- (470,190) -- (410,190) -- cycle ;
%Shape: Rectangle [id:dp68588637236812] 
\draw  [fill={rgb, 255:red, 241; green, 241; blue, 241 }  ,fill opacity=1 ] (230,310) -- (290,310) -- (290,350) -- (230,350) -- cycle ;
%Shape: Rectangle [id:dp9281042768758939] 
\draw  [fill={rgb, 255:red, 241; green, 241; blue, 241 }  ,fill opacity=1 ] (410,310) -- (470,310) -- (470,350) -- (410,350) -- cycle ;
%Straight Lines [id:da28715006376553465] 
\draw [color={rgb, 255:red, 255; green, 255; blue, 255 }  ,draw opacity=1 ]   (230,150) -- (230,350) ;
%Straight Lines [id:da014166668295681917] 
\draw [color={rgb, 255:red, 255; green, 255; blue, 255 }  ,draw opacity=1 ]   (230,150) -- (470,150) ;
%Straight Lines [id:da8558103211493455] 
\draw [color={rgb, 255:red, 255; green, 255; blue, 255 }  ,draw opacity=1 ]   (230,350) -- (470,350) ;
%Straight Lines [id:da9655516692130622] 
\draw [color={rgb, 255:red, 255; green, 255; blue, 255 }  ,draw opacity=1 ]   (470,150) -- (470,350) ;
%Straight Lines [id:da5788945398558377] 
\draw  [dash pattern={on 0.84pt off 2.51pt}]  (290,150) -- (290,360) ;
%Straight Lines [id:da115646847664849] 
\draw  [dash pattern={on 0.84pt off 2.51pt}]  (320,150) -- (320,360) ;
%Straight Lines [id:da5318636145184872] 
\draw  [dash pattern={on 0.84pt off 2.51pt}]  (380,150) -- (380,360) ;
%Straight Lines [id:da8721106219751598] 
\draw  [dash pattern={on 0.84pt off 2.51pt}]  (410,150) -- (410,360) ;
%Straight Lines [id:da16602745542767594] 
\draw  [dash pattern={on 0.84pt off 2.51pt}]  (220,190) -- (470,190) ;
%Straight Lines [id:da5305287843689546] 
\draw  [dash pattern={on 0.84pt off 2.51pt}]  (220,220) -- (470,220) ;
%Straight Lines [id:da20346730567210247] 
\draw  [dash pattern={on 0.84pt off 2.51pt}]  (220,280) -- (470,280) ;
%Straight Lines [id:da2388891786236319] 
\draw  [dash pattern={on 0.84pt off 2.51pt}]  (220,310) -- (470,310) ;

% Text Node
\draw (350.68,234.95) node  [font=\tiny] [align=left] {$\displaystyle \Sigma ^{\opn{geom}}$};
% Text Node
\draw (340,252) node [anchor=north west][inner sep=0.75pt]   [align=left] {$\displaystyle -1$};
% Text Node
\draw (196.89,189.95) node  [font=\tiny] [align=left] {$\displaystyle \lambda _{n,\tau _{1}} +1$};
% Text Node
\draw (206.23,219.95) node  [font=\tiny] [align=left] {$\displaystyle \lambda _{n,\tau _{1}}$};
% Text Node
\draw (190.3,310.95) node  [font=\tiny] [align=left] {$\displaystyle -( \lambda _{n,\tau _{1}} +1)$};
% Text Node
\draw (199.81,280.95) node  [font=\tiny] [align=left] {$\displaystyle -\lambda _{n,\tau _{1}}$};
% Text Node
\draw (422.89,365.95) node  [font=\tiny] [align=left] {$\displaystyle \lambda _{n,\tau _{0}} +1$};
% Text Node
\draw (270.3,365.95) node  [font=\tiny] [align=left] {$\displaystyle -( \lambda _{n,\tau _{0}} +1)$};
% Text Node
\draw (373.23,365.95) node  [font=\tiny] [align=left] {$\displaystyle \lambda _{n,\tau _{0}}$};
% Text Node
\draw (327,366) node  [font=\tiny] [align=left] {$\displaystyle -\lambda _{n,\tau _{0}}$};
% Text Node
\draw (502.3,250.31) node  [font=\scriptsize] [align=left] {$\displaystyle \ell _{\tau _{0}}$};
% Text Node
\draw (350.3,121.31) node  [font=\scriptsize] [align=left] {$\displaystyle \ell _{\tau _{1}}$};
% Text Node
\draw (248,162) node [anchor=north west][inner sep=0.75pt]   [align=left] {$\displaystyle -1$};
% Text Node
\draw (428,162) node [anchor=north west][inner sep=0.75pt]   [align=left] {$\displaystyle -1$};
% Text Node
\draw (428,323) node [anchor=north west][inner sep=0.75pt]   [align=left] {$\displaystyle -1$};
% Text Node
\draw (249,323) node [anchor=north west][inner sep=0.75pt]   [align=left] {$\displaystyle -1$};
% Text Node
\draw (247,252) node [anchor=north west][inner sep=0.75pt]   [align=left] {$\displaystyle +1$};
% Text Node
\draw (431,252) node [anchor=north west][inner sep=0.75pt]   [align=left] {$\displaystyle +1$};
% Text Node
\draw (339,334) node [anchor=north west][inner sep=0.75pt]   [align=left] {$\displaystyle +1$};
% Text Node
\draw (338,171) node [anchor=north west][inner sep=0.75pt]   [align=left] {$\displaystyle +1$};
% Text Node
\draw (441.5,236) node  [font=\tiny] [align=left] {$ \Sigma_{\tau_0,+}^{\opn{an}}$};
% Text Node
\draw (352,321) node  [font=\tiny] [align=left] {$ \Sigma_{\tau_1,-}^{\opn{an}}$};
% Text Node
\draw (257,236) node  [font=\tiny] [align=left] {$ \Sigma_{\tau_0,-}^{\opn{an}}$};
% Text Node
\draw (350,159) node  [font=\tiny] [align=left] {$ \Sigma_{\tau_1,+}^{\opn{an}}$};

\end{tikzpicture}}
\label{RegionsOfTwistFig}
\end{figure}

Of course, there are many other regions of twists disjoint from $\Sigma^{\opn{geom}}$ and $\Sigma_{\tau, \pm}^{\opn{an}}$ one could consider (depending on how the infinity-type of $\chi$ intertwines with the Hodge--Tate weights of $\rho_{\Pi}(n)$) -- for example, the four regions in the corners of Figure \ref{RegionsOfTwistFig}. It would also be interesting to understand the behaviour of the Galois representation in these regions and whether one can construct Euler systems/$p$-adic $L$-functions for these twists. 

\begin{remark}
    One could also ask what happens when one instead imposes the condition that the sign of the functional equation for $L(\Pi, s)$ is $+1$. In this setting, one expects to be able to produce a $p$-adic $L$-function which interpolates the values $L(\rho_{\Pi}(n) \otimes \hat{\chi}, 0)$ for $\chi \in \Sigma^{\opn{geom}}$ by considering the $p$-adic variation of unitary Friedberg--Jacquet periods for a pair of \emph{definite} unitary groups. This situation is more closely aligned with the formalism of bipartite Euler systems (generalising the strategy in \cite{BDIMC}) and will be studied in forthcoming work of Murilo Corato-Zanarella \cite{MCZbipartite}.
\end{remark}

\subsection{Method of proof}

We now describe the main ingredients that go into the proofs of Theorem \ref{FirstTheoremIntro} and Theorem \ref{SecondTheoremIntro}. The first step is to reinterpret the unitary Friedberg--Jacquet periods as a pairing in the coherent cohomology of unitary Shimura varieties. For this, we work with the corresponding unitary similitude groups, since these give rise to PEL-type Shimura varieties. Let $\mbf{G}$ denote the unitary similitude group associated with the Hermitian space $W$ (with similitude in $\mbb{G}_m$), and let $\mbf{H} \subset \mbf{G}$ denote the subgroup preserving the decomposition $W = W_1 \oplus W_2$. 

Let $\pi$ denote a cuspidal automorphic representation of $\mbf{G}(\mbb{A})$ satisfying the analogous hypotheses as in \S \ref{StatementOfMainResultsSSecIntro} (this is made precise in \S \ref{FinalSectionAssumptionsSSec}). It is also convenient to work with eigenvectors for the transpose $U_p$-Hecke operators. More precisely, for an integer $\beta \geq 1$, let $K^G_{\opn{Iw}}(p^{\beta}) \subset \mbf{G}(\mbb{Q}_p) = \mbb{Q}_p^{\times} \times \prod_{\tau \in \Psi} \opn{GL}_{2n}(\mbb{Q}_p)$ denote the upper-triangular modulo $p^{\beta}$ Iwahori subgroup. Then associated with the $p$-stabilisation $\phi_p$, we can define eigenvectors $\psi_{p, \beta} \in \pi_p^{K^G_{\opn{Iw}}(p^{\beta})}$ satisfying the property:
\[
[K^G_{\opn{Iw}}(p^{\beta}) \cdot s_{p, i, \tau} \cdot K^G_{\opn{Iw}}(p^{\beta})] \cdot \psi_{p, \beta} = \alpha_{i, \tau} \psi_{p, \beta}
\]
where $s_{p, i, \tau}$ denotes the conjugate of $t_{p, i, \tau}$ by the longest Weyl element of $\mbf{G}(\mbb{Q}_p)$. Let $\psi_{\beta}^{[j]} = (\Delta^{[j]}_{\kappa} \cdot \phi_{\infty}) \otimes \phi_S \otimes \bigotimes_{\ell \not\in S \cup \{p\}} \phi_{\ell} \otimes \hat{\gamma} \psi_{p, \beta} \in \pi$, where $\hat{\gamma}$ is a certain representative of the Zariski open orbit of the upper-triangular Borel subgroup acting on $\mbf{H}(\mbb{Q}_p) \backslash \mbf{G}(\mbb{Q}_p)$ (see Definition \ref{NewDefOfGamma}). Then, up to an explicit non-zero rational number, the period $\mathscr{P}_{\pi, \chi}(\phi^{[j]}_e)$ is equal to $\mathscr{P}_{\pi, \chi}(\psi^{[j]}_{\beta})$, provided that $\beta \geq \opn{max}_{\tau}(e_{\tau})$. Therefore, it suffices to analyse these latter periods.

Let $S_{\mbf{G}, \opn{Iw}}(p^{\beta})$ (resp. $S_{\mbf{H}, \diamondsuit}(p^{\beta})$) denote the Shimura variety associated with $\mbf{G}$ (resp. $\mbf{H}$) of level $K^p K^G_{\opn{Iw}}(p^{\beta})$ (resp. $(K^p \cap \mbf{H}(\mbb{A}_f^p)) (\hat{\gamma} K^G_{\opn{Iw}}(p^{\beta}) \hat{\gamma}^{-1} \cap \mbf{H}(\mbb{Q}_p))$). There is a natural finite unramified morphism
\[
\hat{\iota} \colon S_{\mbf{H}, \diamondsuit}(p^{\beta}) \to S_{\mbf{G}, \opn{Iw}}(p^{\beta})
\]
induced from right-translation by $\hat{\gamma}$. By the work of Su \cite{Su19}, the test data $\psi_{\beta}^{[0]}$ can be encoded as a coherent cohomology class
\[
\eta_{\beta} \in \opn{H}^{n-1}\left( S_{\mbf{G}, \opn{Iw}}(p^{\beta})(\mbb{C}), \mathscr{M}_{G,\kappa^*} \right)
\]
where $\mathscr{M}_{G, \kappa^*}$ denotes the automorphic vector bundle with highest weight $\kappa^* = w_n \cdot (\lambda + \rho) - \rho$. After possibly rescaling this test data, this cohomology class is defined over a finite extension of the reflex field of $S_{\mbf{G}, \opn{Iw}}(p^{\beta})$. 

In general, to view the test data $\psi^{[j]}_{\beta}$ as coherent cohomology classes, one needs to enlarge the automorphic vector bundle $\mathscr{M}_{G, \kappa^*}$ to a sheaf of ``nearly holomorphic forms''. More precisely, let $P_{\mbf{G}}^{\opn{std}} \subset \opn{GL}_1 \times \prod_{\tau \in \Psi} \opn{GL}_{2n}$ denote the standard lower-triangular parabolic subgroup with Levi $M_{\mbf{G}} = \opn{GL}_1 \times (\opn{GL}_{1} \times \opn{GL}_{2n-1}) \times \prod_{\tau \neq \tau_0} \opn{GL}_{2n}$. Then there exists a $P_{\mbf{G}}^{\opn{std}}$-torsor
\[
\pi \colon P_{\mbf{G}, \opn{dR}} \to S_{\mbf{G}, \opn{Iw}}(p^{\beta})
\]
such that automorphic vector bundles arise as the associated sheaves of this torsor. Let $\mathscr{N}_G \defeq \pi_* \mathcal{O}_{P_{\mbf{G}, \opn{dR}}}$ and $\mathscr{N}_{G, \kappa^*} = \left(\mathscr{N}_G \otimes V_{\kappa}^*\right)^{M_{\mbf{G}}}$, where $V_{\kappa}^*$ denotes the algebraic representation of $M_{\mbf{G}}$ with highest weight $\kappa^*$. Then we have an embedding $\mathscr{M}_{G,\kappa^*} \subset \mathscr{N}_{G, \kappa^*}$ which identifies with the subspace of elements killed by the action of the Lie algebra of the unipotent radical of $P_{\mbf{G}}^{\opn{std}}$ (obtained by differentiating the torsor structure). As explained in \S \ref{DmodulesOnFLsection}, the quasi-coherent sheaf $\mathscr{N}_G$ is in fact a $\mathcal{D}$-module on $S_{\mbf{G}, \opn{Iw}}(p^{\beta})$, and carries an action of $(2n-1)$ commuting derivations $\{ \nabla_1, \dots, \nabla_{2n-1} \}$ which are algebraic interpretations of the Maass--Shimura differential operators on $S_{\mbf{G}, \opn{Iw}}(p^{\beta})(\mbb{C})$. Furthermore, we can package these operators together into an $M_{\mbf{G}}$-equivariant algebra action
\begin{equation} \label{PolyActionIntro}
C^{\opn{pol}}(\mbb{G}_a^{\oplus 2n-1}, \mbb{G}_a) \otimes \mathscr{N}_G \to \mathscr{N}_G
\end{equation}
of polynomial functions, such that the projection to the $i$-th component acts through the operator $\nabla_i$. The action of $M_{\mbf{G}}$ on $C^{\opn{pol}}(\mbb{G}_a^{\oplus 2n-1}, \mbb{G}_a)$ is given by the adjoint action on the argument, by identifying $\mbb{G}_a^{\oplus 2n-1}$ with the unipotent radical of $P_{\mbf{G}}^{\opn{std}}$. One can make similar definitions for the group $\mbf{H}$ and there is an analogous $M_{\mbf{H}}$-equivariant action of polynomial functions on $\mathscr{N}_{H}$.

With this formalism of sheaves of ``nearly holomorphic forms'', we can reinterpret the period $\mathscr{P}_{\pi, \chi}(\psi_{\beta}^{[j]})$ algebraically as follows. Given $\chi \in \Sigma_{\pi}$, one can construct a certain non-zero vector
\[
\delta_{\kappa, j} \in V_{\kappa} \otimes C^{\opn{pol}}(\mbb{G}_a^{\oplus 2n-1}, \mbb{G}_a)
\]
which is an eigenvector for the action of $M_{\mbf{H}}$ with eigencharacter given by a certain character $\sigma_{\kappa}^{[j], -1} \colon M_{\mbf{H}} \to \mbb{G}_m$ (see \S \ref{BranchingLawsPreliminarySection} for its definition) which, up to a shift by the sum of the positive roots of $\mbf{H}$ not lying in $M_{\mbf{H}}$, encodes the infinity-type of $\chi$. If we write $\delta_{\kappa, j} = \sum_{l} v_l \otimes p_l$ as a sum of pure tensors, then we can build a map
\begin{align*} 
\mathscr{N}_{G, \kappa^*} = \opn{Hom}_{M_{\mbf{G}}}(V_{\kappa}, \mathscr{N}_G) &\xrightarrow{\vartheta_{\kappa, j, \beta}} \opn{Hom}_{M_{\mbf{H}}}(\sigma_{\kappa}^{[j], -1}, \hat{\iota}_*\mathscr{N}_H) = \hat{\iota}_*\mathscr{N}_{H, \sigma_{\kappa}^{[j]}} \\
F &\mapsto \left( 1 \mapsto \sum_l \hat{\iota}^*(p_l \star F(v_l)) \right)
\end{align*}
where $\star$ denotes the action in (\ref{PolyActionIntro}) and $\hat{\iota}^*$ denotes the natural pullback map $\mathscr{N}_G \to \hat{\iota}_* \mathscr{N}_H$. By analysing the representation theory and relevant branching laws for the pair of groups $(M_{\mbf{G}}, M_{\mbf{H}})$, one can show that this morphism takes ``holomorphic forms'' for $\mbf{G}$ to ``holomorphic forms'' for $\mbf{H}$, i.e., one obtains an induced morphism $\vartheta_{\kappa, j, \beta} \colon \mathscr{M}_{G, \kappa^*} \to \hat{\iota}_* \mathscr{M}_{H, \sigma_{\kappa}^{[j]}}$. Finally, by Shimura reciprocity, the character $\chi$ can be interpreted as a cohomology class $[\chi] \in \opn{H}^0\left( S_{\mbf{H}, \diamondsuit}(p^{\beta}), \mathscr{M}_{H, \sigma_{\kappa}^{[j]}}^{\vee} \right)$, where $(-)^{\vee}$ denotes the Serre dual of a vector bundle. Therefore, for a sufficiently large finite extension $L/F$, we can define an ``evaluation map'':
\begin{align*}
    \opn{Ev}_{\kappa, j, \chi, \beta} \colon \opn{H}^{n-1}\left( S_{\mbf{G}, \opn{Iw}}(p^{\beta}), \mathscr{M}_{G, \kappa^*} \right) &\to L \\
    \eta &\mapsto \langle \vartheta_{\kappa, j, \beta}(\eta), [\chi] \rangle . 
\end{align*}
Up to explicit non-zero factors, the automorphic period $\mathscr{P}_{\pi, \chi}(\psi_{\beta}^{[j]})$ is equal to the image of $\eta_{\beta}$ under the evaluation map $\opn{Ev}_{\kappa, j, \chi, \beta}$, which is the desired algebraic interpretation of the unitary Friedberg--Jacquet periods.

The benefit of this reinterpretation is that we can study the $p$-adic behaviour of these evaluation maps, and hence the unitary Friedberg--Jacquet periods. In particular, if we (abusively) let $L/\mbb{Q}_p$ denote the finite extension obtained as the completion of $\iota_p(L)$ and let $\mathcal{S}_{G, \opn{Iw}}(p^{\beta})$ denote the corresponding adic Shimura variety over $L$, then by rigid GAGA we can view $\eta_{\beta}$ (resp. $\opn{Ev}_{\kappa, j, \chi, \beta}$) as a cohomology class in (resp. linear form on) the cohomology group $\opn{H}^{n-1}\left( \mathcal{S}_{G, \opn{Iw}}(p^{\beta}), \mathscr{M}_{G,\kappa^*} \right)$. Using the recently developed higher Coleman theory of Boxer--Pilloni \cite{BoxerPilloni}, one can $p$-adically interpolate these cohomology groups and cohomology classes, so the strategy for $p$-adically interpolating $\opn{Ev}_{\kappa, j, \chi, \beta}$ is:
\begin{enumerate}
    \item to define overconvergent versions of $\mathscr{N}_G$ and $\mathscr{N}_H$ with actions of differential operators which can be $p$-adically iterated;
    \item construct a $p$-adic version of the map $\vartheta_{\kappa, j, \beta}$ by $p$-adically interpolating the eigenvectors $\delta_{\kappa, j}$.
\end{enumerate}
For (1), we generalise the results of \cite{DiffOps} to the setting of unitary Shimura varieties, and for (2), we generalise the $p$-adic interpolation of branching laws appearing in \cite[Appendix A]{UFJ}. One then builds an overconvergent version of the morphism $\vartheta_{\kappa, j, \beta}$ in a similar way as above. This method can be described rather abstractly and we explain how to accomplish this in \S \ref{TheMainConstructionSubSec}. To be able to $p$-adically interpolate the eigenvectors $\delta_{\kappa, j}$ one must first perform a certain ``$p$-depletion'', and in \S \ref{PropsOfEvMapsInterpolationSubSec} we explain how this is related to the appearance of the factor $\mathscr{E}_p(\pi, \chi)$ in the interpolation property in Theorem \ref{FirstTheoremIntro}. Additionally, this whole strategy can be extended to the setting where we allow the automorphic representation $\pi$ to vary in Coleman families, which leads to the construction of the $p$-adic $L$-function in Theorem \ref{SecondTheoremIntro}. 

\subsection{Structure of the article}

This article comprises roughly of three parts. The material in \S \ref{AbstractComputationsSection}--\S \ref{CohomologyAndCorrespondencesSection} is mostly preliminary and lays the foundations for the constructions of the evaluation maps. More precisely, in \S \ref{AbstractComputationsSection} we describe the abstract results from representation theory which generalise \cite[Appendix A]{UFJ}, and give an overview for constructing the ($p$-adic versions of the) morphisms $\vartheta_{\kappa, j, \beta}$. In \S \ref{ContinuousOpsOnBanachSpacesSec}, we summarise (and adapt to our setting) the main results on continuous operators on Banach spaces appearing in \cite{DiffOps}. In \S \ref{CohomologyAndCorrespondencesSection}, we describe a minor generalisation of the theory of Hecke operators and cohomology with support appearing in \cite{BoxerPilloni, HHTBoxerPilloni}, which is necessary for the (slightly more general) version of higher Coleman theory needed in this article.

The second part of this article (\S \ref{NearlyHolAutFormsSection}--\S \ref{PadicIterationOfDiffOpsSection}) involves the study of the geometry and cohomology of unitary Shimura varieties and the construction of the ($p$-adic) evaluation maps. In \S \ref{NearlyHolAutFormsSection}, we introduce the Shimura varieties (and the associated moduli spaces of abelian varieties with extra structure), the sheaves of ``nearly holomorphic forms'', and the evaluation maps $\opn{Ev}_{\kappa, j, \beta}$ (and their relation to unitary Friedberg--Jacquet periods). We then study the $p$-adic versions of the sheaves (the sheaves of ``nearly overconvergent forms'') and the $p$-adic geometry of the adic Shimura varieties in \S \ref{ThePADICtheorySection}, and describe how to $p$-adically interpolate the differential operators on these sheaves in \S \ref{PadicIterationOfDiffOpsSection}. We end \S \ref{PadicIterationOfDiffOpsSection} by giving the construction of $p$-adic versions of the evaluation maps in families.

Finally, the third part of this article (\S \ref{HOpsAndHCTChapter}--\S \ref{AutoRepsSection}) describes the construction of the $p$-adic $L$-functions in Theorem \ref{FirstTheoremIntro} and Theorem \ref{SecondTheoremIntro}. More precisely, in \S \ref{HOpsAndHCTChapter} we introduce the action of Hecke operators on the cohomology groups appearing in higher Coleman theory, prove an interpolation property for the $p$-adic evaluation maps, and develop the abstract machinery for producing the $p$-adic $L$-function. In \S \ref{AutoRepsSection}, we apply this general machinery to the relevant automorphic representations and prove Theorem \ref{FirstTheoremIntro} and Theorem \ref{SecondTheoremIntro}. To obtain the final interpolation formula, one also needs to incorporate some miscellaneous representation theoretic results on branching laws for general linear groups which we provide in Appendix \ref{SomeRepTheoryAppendix} and Appendix \ref{EquivariantLinearFuncsAppendix}.

\subsection{Acknowledgements}

The author would like to thank Daniel Disegni and Rob Rockwood for comments on an earlier version of this article. The author would also like to thank the anonymous referee whose comments and suggestions have significantly improved this article.

Part of this work was carried out during the author's visits to MIT (February 2022) and MSRI/SLMath (March 2023). The author would like to thank both of these institutions for their hospitality, and also the Cecil King Memorial Foundation (London Mathematical Society) for funding the former visit. This work was (partly) funded by UK Research and Innovation grant MR/V021931/1. For the purpose of Open Access, the author has applied a CC BY public copyright licence to any Author Accepted Manuscript (AAM) version arising from this submission.

\subsection{Notation and conventions} \label{NotationsAndConventionsIntro}

We fix the following notation and conventions throughout the article.

\begin{itemize}
    \item We let $F/F^+$ be a CM extension with maximal totally real subfield $F^+ \neq \mbb{Q}$. We assume that $F$ contains an imaginary quadratic extension $E/\mbb{Q}$. We fix a CM type $\Psi$ of $F$ and a distinguished place $\tau_0 \in \Psi$ (i.e., an embedding $\tau_0 \colon F \hookrightarrow \mbb{C}$). We let $F^{\opn{cl}}$ denote the Galois closure of $F$ in $\mbb{C}$ via the embedding $\tau_0$.
    \item We fix an odd prime $p$ which splits completely in $F/\mbb{Q}$ and doesn't divide the discriminant of $F/\mbb{Q}$. We will also impose Assumption \ref{AssumpPoddAndDoesntDivide} throughout the majority of the article.
    \item We fix an identification $\iota_p \colon \mbb{C} \cong \Qpb$, and for any embedding $\sigma \colon F \hookrightarrow \mbb{C}$ we let $\ide{p}_{\sigma}$ denote the prime of $F$ lying above $p$ determined by the embedding $\iota_p \circ \sigma \colon F \hookrightarrow \Qpb$.
    \item We fix an integer $n \geq 1$ throughout the article (which will be half the dimension of the Hermitian space $W$). For convenience we assume that $n \geq 2$, but note that, with suitable modifications, one can also adapt the methods in this article to treat the case $n=1$.
    \item For any number field $\Phi$, we let $\mbb{A}_{\Phi}$ (resp. $\mbb{A}_{\Phi}^{\times}$) denote the ad\`{e}les (resp. id\`{e}les) of $\Phi$. If $\Phi = \mbb{Q}$, we will often omit this from the notation. If $S$ is a finite set of places of $\Phi$, we let $\mbb{A}_{\Phi}^S$ (resp. $\mbb{A}_{\Phi, S}$) denote the ad\`{e}les away from $S$ (resp. at $S$). We let $\mbb{A}_{\Phi, f}$ denote the ad\`{e}les at finite places.
    \item Unless specified otherwise, we will use additive notation for the multiplication of characters. For a character $\kappa$, we will often write $(-)^{\kappa} \defeq \kappa(-)$.
    \item By a ``Tate affinoid algebra'' $(R, R^+)$ we mean a complete uniform Huber pair $(R, R^+)$ over $(\mbb{Q}_p, \mbb{Z}_p)$. We will always assume that $(R, R^+)$ is sheafy, i.e., $\opn{Spa}(R, R^+)$ is an adic space. 
    \item For a category $\mathcal{C}$, let $\opn{Ind}\mathcal{C}$ denote the ind-category of $\mathcal{C}$, i.e., the category whose objects (which we refer to as ind-systems) are small filtered inductive systems $X = (X_i)_{i \in I}$, and whose morphisms are
\[
\opn{Hom}(X, Y) \defeq \varprojlim_{i \in I} \varinjlim_{j \in J} \opn{Hom}(X_i, Y_j)
\]
where $Y = (Y_j)_{j \in J}$ and the transition maps are the obvious ones induced from pre-/post-composition. By a representative of a morphism $f \in \opn{Hom}(X, Y)$, we mean the data of a map $\alpha \colon I \to J$, and morphisms $f_i 
\in \opn{Hom}(X_i, Y_{\alpha(i)})$ such that $f = (f_i)_{i \in I}$ (so in particular, for any $i' \geq  i$ there exists $j \geq \alpha(i), \alpha(i')$ such that the morphisms $X_i \xrightarrow{f_i} Y_{\alpha(i)} \to Y_{j}$ and $X_i \to X_{i'} \xrightarrow{f_{i'}} Y_{\alpha(i')} \to Y_j$ coincide). If $F \colon \mathcal{C} \to \mathcal{D}$ is a functor, then we will also use the notation $F \colon \opn{Ind}\mathcal{C} \to \opn{Ind}\mathcal{D}$ to denote the natural functor given by $F(X) \defeq (F(X_i))_{i \in I}$. Whenever we say that a morphism $f \in \opn{Hom}(X, Y)$ satisfies an additional property (such as equivariance with respect to a group action) we mean that $f$ has a representative $\{ f_i \}$ where each $f_i$ satisfy this property.
\item For a split reductive group $G$ over a characteristic zero field, we let $w_G^{\opn{max}}$ denote the longest element of the Weyl group of $G$. If $G = \opn{GL}_d$, we will usually take $w_{\opn{GL}_d}^{\opn{max}}$ to be the antidiagonal $(d \times d)$-matrix with $1$s along the antidiagonal.  
\item We fix once and for all a choice $i = \sqrt{-1} \in \mbb{C}$. This determines a canonical choice of $p$-th power roots of unity $e^{2 \pi i/p^h} \in \mbb{C}$ for $h \geq 0$, and hence a system of $p$-th power roots of unity $\iota_p(e^{2 \pi i/p^h}) \in \Qpb$. Any construction depending on a choice of $p$-th power roots of unity will be with respect to this canonical choice (e.g. when defining Gauss sums associated with finite order characters of $p$-th power conductor).

\item If $s \geq 1$ is an integer, $(R, R^+)$ is a Tate affinoid algebra, and $X$ is a $p$-adic manifold (that is, isomorphic to an open subset of $\mbb{Q}_p^{r}$ for some $r \geq 1$), then we say that a continuous function $f \colon X \to R$ is $s$-analytic if: for every $\underline{a} = (a_1, \dots, a_r) \in X$ there exists $\alpha_{n_1, \dots, n_r} \in R$ with the property that $p^{s(n_1+ \cdots + n_r)} \alpha_{n_1, \dots, n_r} \to 0$ as $n_1+\cdots + n_r \to + \infty$ such that, for every $(x_1, \dots, x_r) \in X \cap B(\underline{a}, p^{-s})$, one has
\[
f(x_1, \dots, x_r) = \sum_{n_1, \dots, n_r \geq 0} \alpha_{n_1, \dots, n_r} (x_1 - a_1)^{n_1} \cdots (x_r - a_r)^{n_r} 
\]
(c.f., \cite[\S 3.2.1]{UrbanEigenvarieties}). Here $B(\underline{a}, p^{-s}) \subset \mbb{Q}_p^r$ denotes the ball of radius $p^{-s}$ with centre $\underline{a}$. This is independent of the choice of identification of $X$ with an open subset of $\mbb{Q}_p^r$.
\end{itemize}

%----------------------------------------

\section{Abstract computations} \label{AbstractComputationsSection}

\subsection{Recap on the groups and notation}

In this section, we recall some notation for the groups and representations appearing in \cite{UFJ}.

\subsubsection{Algebraic and analytic groups}

Let $G$ denote the algebraic group over $\mbb{Q}_p$ given by $G = \opn{GL}_1 \times \prod_{\tau \in \Psi} \opn{GL}_{2n}$ and let $H \subset G$ denote the subgroup $H = \opn{GL}_1 \times \prod_{\tau \in \Psi} \left(\opn{GL}_n \times \opn{GL}_n \right)$ embedded block diagonally. If $A \subset G$ (resp. $A \subset H$) is any subgroup and $\tau \in \Psi$, then we will call the projection of $A$ to the $\opn{GL}_{2n}$-factor (resp. $(\opn{GL}_n \times \opn{GL}_n)$-factor) indexed by $\tau$ the \emph{$\tau$-factor (or $\tau$-component) of $A$}. We will also call the projection of $A$ to the first $\opn{GL}_1$-factor the \emph{similitude factor of $A$}. If $d \geq 1$ is an integer and $a_1 + \cdots + a_l = d$ is a partition of $d$, then we say that $A \subset \opn{GL}_d$ is the standard upper-triangular (resp. lower-triangular) parabolic with Levi $\opn{GL}_{a_1} \times \cdots \times \opn{GL}_{a_l}$ if $A$ has the following block matrix description:
\begin{equation} \label{Eqn:BlockMatrixParabolic}
A = \smat{\opn{GL}_{a_1} & * & * & \cdots & * \\ & \opn{GL}_{a_2} & * & \cdots & * \\ & & \ddots & & * \\ & & & \opn{GL}_{a_{l-1}} & * \\ & & & & \opn{GL}_{a_l} }, \quad \quad \text{ (resp. } A = \smat{\opn{GL}_{a_1} &  &  &  &  \\ * & \opn{GL}_{a_2} &  &  &  \\ \vdots & \vdots & \ddots & &  \\ * & * & & \opn{GL}_{a_{l-1}} &  \\ * & * & * & * & \opn{GL}_{a_l} } \text{ )},
\end{equation}
In this case, we will also occasionally call $A$ the standard upper-triangular (resp. lower-triangular) parabolic of type $(a_1, \dots, a_l)$. This definition extends to the groups $H$ and $G$ by specifying the similitude and $\tau$-factors of the parabolic subgroup as above. If $a_1 = \cdots = a_l = 1$, then we say $A$ is the standard upper-triangular (lower-triangular) Borel subgroup, and call its Levi subgroup the standard diagonal torus. Finally, if $B$ is an algebraic group satisfying $\opn{GL}_{a_1} \times \cdots \times \opn{GL}_{a_l} \subset B \subset \opn{GL}_d$, then we call a parabolic subgroup $A \subset B$ the standard upper-triangular (lower-triangular) parabolic with Levi $\opn{GL}_{a_1} \times \cdots \times \opn{GL}_{a_l}$ if it is the intersection of $B$ with the block subgroup as in (\ref{Eqn:BlockMatrixParabolic}).

Let $P_G \subset G$ denote the upper-triangular parabolic subgroup with Levi subgroup given by
\[
M_G = \opn{GL}_1 \times \left( \opn{GL}_1 \times \opn{GL}_{2n-1} \right) \times \prod_{\tau \neq \tau_0} \opn{GL}_{2n}
\]
and we set $P_H = P_G \cap H$ and $M_H = M_G \cap H$. Let $Q_{M_H}$ denote the standard upper-triangular parabolic of $M_H$, which in the $\tau_0$-factor has Levi subgroup
\[
\opn{GL}_1 \times \opn{GL}_{n-1} \times \opn{GL}_1 \times \opn{GL}_{n-1},
\]
and in the $\tau$-factor has Levi subgroup equal to $\opn{GL}_n \times \opn{GL}_n$, for $\tau \neq \tau_0$. All of these groups have obvious integral models over $\mbb{Z}_p$ which we will denote by the same letters.

We recall some special elements that appear in \cite[\S 2.4]{UFJ}, however note that our definition of $\gamma$ below is slightly different from that in \emph{loc.cit.}. The reason for this change is that the definition below is more suitable for describing the functoriality between the Igusa towers appearing in \S \ref{IgusaVarietiesSection}. As seen in Lemma \ref{OpenOrbitForgammauvLemma} below, all of the properties in \cite{UFJ} that hold for the old definition of $\gamma$ (such as being an open orbit representative for the spherical pair $(G, H)$) are still satisfied for the new definition of $\gamma$ in this article. Let $W_?$ denote the Weyl group of a reductive group $?$.

\begin{definition} \label{NewDefOfGamma}
    Let $w_n$ denote the unique minimal length representative for $W_{M_G} \backslash W_G$ of length $n$. We view $w_n \in G(\mbb{Z}_p)$ via the explicit matrix in \cite[Definition 2.4.1]{UFJ}. 
    \begin{enumerate}
        \item Let $u \in M_G(\mbb{Z}_p)$ denote the element as defined in \cite[Definition 2.4.2]{UFJ}.
        \item Let $\gamma = 1 \times \prod_{\tau} \gamma_{\tau} \in P_G(\mbb{Z}_p)$ denote the element given by
        \[
        \gamma_{\tau_0} = u_{\tau_0} \tbyt{1}{x_{\tau_0}}{}{1}, \quad \gamma_\tau = u_{\tau} \; (\tau \neq \tau_0)
        \]
        where $x_{\tau_0}$ is the $1 \times (2n-1)$-matrix with $1$ in the $n$-th place and $0$ elsewhere. Here $\tbyts{1}{x_{\tau_0}}{}{1}$ is a block matrix with diagonal block sizes $1 \times 1$ and $(2n-1) \times (2n-1)$. We set $\hat{\gamma} = \gamma \cdot w_n \in G(\mbb{Z}_p)$.
        \item Let $v \in M_H(\mbb{Z}_p)$ be the element which is the identity outside the $\tau_0$-component, and in the $\tau_0$-component is equal to 
\[
1 \times 1 \times \tbyt{1}{}{y}{1} \in \opn{GL}_1 \times \opn{GL}_{n-1} \times \opn{GL}_n
\]
where the last matrix has upper-left block size $1 \times 1$, the bottom-right block has size $(n-1) \times (n-1)$, and $y$ is the column vector with $1$ as every entry.
    \end{enumerate}
\end{definition}

The following lemma is the analogue of \cite[Lemma 2.4.3]{UFJ}.

\begin{lemma} \label{OpenOrbitForgammauvLemma}
    \begin{enumerate}
        \item Let $B_G$ denote the standard upper-triangular Borel subgroup of $G$. Then the subset $H \cdot \hat{\gamma} \cdot B_G$ is Zariski open and dense in $G$ (over $\opn{Spec}\mbb{Z}_p$).
        \item Let $B_{M_G}$ denote the standard upper-triangular Borel subgroup of $M_G$. Then the subset 
                \[
                M_H \cdot (u, v) \cdot \left( B_{M_G} \times Q_{M_H} \right) \subset M_{G} \times M_H
                \]
            with $M_H \subset M_G \times M_H$ diagonally, is Zariski open and dense in $M_G$ (over $\opn{Spec}\mbb{Z}_p$).
    \end{enumerate}
\end{lemma}
\begin{proof}
    The proof of (1) is identical to \cite[Lemma 2.4.3(2)]{UFJ} but instead we have a similar decomposition 
    \[
    \hat{\gamma}_{\tau_0} = \tbyt{X}{}{}{1} \tbyt{1}{}{w_{\opn{GL}_n}^{\opn{max}}}{1} \tbyt{1}{Y}{}{Z}
    \]
    as in \emph{loc.cit.} with $X w_{\opn{GL}_n}^{\opn{max}} = 1 \times w_{\opn{GL}_{n-1}}^{\opn{max}}$. Part (2) is very similar to \cite[Lemma 2.4.3(1)]{UFJ}. More precisely, it is enough to establish this for each $\tau$-factor with $\tau \in \Psi$. For $\tau \neq \tau_0$ this follows immediately from \emph{loc.cit.}. For $\tau = \tau_0$, one can explicitly calculate the stabiliser $M_H \cap (u, v) (B_{M_G} \times Q_{M_H}) (u, v)^{-1}$ in the $\tau_0$-component; namely it is equal to all diagonal matrices of the form
\[
\opn{diag}(x, y, \dots, y) \subset \opn{GL}_1 \times \opn{GL}_{n-1} \times \opn{GL}_n .
\]
This stabiliser has the minimal possible dimension, which forces the subset to be Zariski open and dense. 
\end{proof}

We now introduce the relevant $p$-adic analytic groups. Let $\beta \geq 1$ be an integer. We let $M^G_{\opn{Iw}}(p^{\beta})$ denote the upper-triangular depth $p^{\beta}$ Iwahori subgroup in $M_G(\mbb{Z}_p)$ (i.e. all elements which lie in $B_{M_G}$ modulo $p^\beta$). Let $T$ denote the standard diagonal torus in $G$. We denote by $(x; y_{1, \tau}, \dots, y_{2n, \tau})$ the element of $T$ given by
\[
x \times \prod_{\tau \in \Psi} \opn{diag}(y_{1, \tau}, \dots, y_{2n, \tau}) .
\]
Let $T^{\diamondsuit} \subset T$ denote the subtorus of elements which satisfy $y_{1, \tau_0} = y_{n+1, \tau_0}$, $y_{i, \tau_0} = y_{2n+2-i, \tau_0}$ for $i=2, \dots, n$, and $y_{i, \tau} = y_{2n+1-i, \tau}$ for $i=1, \dots, n$ and $\tau \neq \tau_0$. We let $M^H_{\diamondsuit}(p^{\beta}) \subset M_H(\mbb{Z}_p)$ denote the subgroup of elements which land in $T^{\diamondsuit}$ modulo $p^{\beta}$. Note that $u^{-1} M^H_{\diamondsuit}(p^{\beta}) u \subset M^G_{\opn{Iw}}(p^{\beta})$.

\subsubsection{Algebraic and locally analytic weights}

Let $X^*(T)$ denote the group of algebraic characters $T \to \mbb{G}_m$. Any such character $\kappa \in X^*(T)$ can be described as a tuple of integers $(\kappa_0; \kappa_{1, \tau}, \dots, \kappa_{2n, \tau})$ such that
\[
\kappa( x ; y_{1, \tau}, \dots, y_{2n, \tau} ) = x^{\kappa_0} \prod_{\tau \in \Psi} \prod_{i=1}^{2n} y_{i, \tau}^{\kappa_{i, \tau}} .
\]
Let $S = \prod_{\tau \in \Psi} \mbb{G}_m$, so that any element $j \in X^*(S)$ is described as a tuple of integers $j = (j_{\tau})_{\tau \in \Psi}$. In this article we will consider the following cone of algebraic weights.

\begin{definition} \label{DefOfAlgWeightsE}
    Let $\mathcal{E} \subset X^*(T \times S)$ denote the submonoid of algebraic characters $(\kappa, j)$ with $\kappa \in X^*(T)$ and $j \in X^*(S)$, satisfying the following properties:
    \begin{itemize}
        \item $\kappa$ is $M_G$-dominant, i.e. $\kappa_{2, \tau_0} \geq \cdots \geq \kappa_{2n, \tau_0}$ and $\kappa_{1, \tau} \geq \cdots \geq \kappa_{2n, \tau}$ for all $\tau \neq \tau_0$.
        \item There exists an integer $w \in \mbb{Z}_{\leq 0}$ such that
        \[
        \kappa_{i, \tau_0} + \kappa_{2n+2-i, \tau_0} = w
        \]
        for all $i=2, \dots, n$, and $\kappa_{n+1, \tau_0} \leq w$.
        \item For all $\tau \neq \tau_0$ and $i=1, \dots, n$, we have
        \[
        \kappa_{i, \tau} + \kappa_{2n+1-i, \tau} = 0 .
        \]
        \item The tuple $j = (j_{\tau})$ satisfies $0 \leq j_{\tau_0} \leq \kappa_{n+1, \tau_0} - \kappa_{n+2, \tau_0}$ and
        \[
        0 \leq j_{\tau} \leq \kappa_{n, \tau}
        \]
        for all $\tau \neq \tau_0$.
    \end{itemize}
\end{definition}

Let 
\[
\{\mu_0, \mu_w \} \cup \{ \mu_{i, \tau_0} : i=1, \dots, n+1 \} \cup \{ \mu_{i, \tau} : i=1, \dots, n \text{ and } \tau \neq \tau_0 \} \subset X^*(T)
\]
denote the collection of characters given by
\begin{itemize}
    \item $\mu_0(x; y_{1, \tau}, \dots, y_{2n, \tau}) = x$ and $\mu_w(x; y_{1, \tau}, \dots, y_{2n, \tau}) = \prod_{i=n+1}^{2n} y_{i, \tau_0}^{-1}$.
    \item $\mu_{1, \tau_0}(x; y_{1, \tau}, \dots, y_{2n, \tau}) = y_{1, \tau_0}$.
    \item For all $i=2, \dots, n$, one has $\mu_{i, \tau_0}(x; y_{1, \tau}, \dots, y_{2n, \tau}) = \prod_{j=2}^{i} y_{j, \tau_0} y_{2n+2-j, \tau_0}^{-1}$.
    \item $\mu_{n+1, \tau_0}(x; y_{1, \tau}, \dots, y_{2n, \tau}) = y_{n+1, \tau_0}^{-1} \prod_{j=2}^{n} y_{j, \tau_0} y_{2n+2-j, \tau_0}^{-1}$.
    \item For all $i=1, \dots, n$ and $\tau \neq \tau_0$, $\mu_{i, \tau}(x; y_{1, \tau}, \dots, y_{2n, \tau}) = \prod_{j=1}^{i} y_{j, \tau} y_{2n+1-j, \tau}^{-1}$.
\end{itemize}
For any $\tau \in \Psi$, let $1_{\tau} \in X^*(S)$ denote the character given by $1_{\tau} = (0, \dots, 0, 1, 0, \dots, 0)$ where $1$ is in the $\tau$-component. This collection of characters provides an explicit generating set of $\mathcal{E}$, as explained in the following lemma:

\begin{lemma} \label{ConeOfWeightsForELemma}
    For any $(\kappa, j) \in \mathcal{E}$, there exist unique integers $a_0, a_{1, \tau_0} \in \mbb{Z}$, $a_w, a_{i, \tau} \in \mbb{Z}_{\geq 0}$ for $(i, \tau) \neq (1, \tau_0)$, and $b_{\tau} \in \mbb{Z}_{\geq 0}$ such that
    \[
    (\kappa, j) = a_0 (\mu_0, 0) + a_w (\mu_w, 0) + a_{n+1, \tau_0}(\mu_{n+1, \tau_0}, 0) + \sum_{i=1}^n \sum_{\tau \in \Psi} a_{i, \tau}(\mu_{i, \tau}, 0) + \sum_{\tau \in \Psi} b_{\tau} (\mu_{n, \tau}, 1_{\tau}) .
    \]
\end{lemma}
\begin{proof}
    We can take $a_0 = \kappa_0$, $a_w = -(\kappa_{n, \tau_0} + \kappa_{n+2, \tau_0})$, $a_{n+1, \tau_0} = \kappa_{n, \tau_0} - \kappa_{n+1, \tau_0} + \kappa_{n+2, \tau_0}$, $a_{1, \tau_0} = \kappa_{1, \tau_0}$, $a_{i, \tau_0} = \kappa_{i, \tau_0} - \kappa_{i+1, \tau_0}$ ($i \in \{2, \dots, n-1\}$), $a_{n, \tau_0} = \kappa_{n+1, \tau_0} - \kappa_{n+2, \tau_0} - j_{\tau_0}$, $a_{i, \tau} = \kappa_{i, \tau} - \kappa_{i+1, \tau}$ ($i \in \{1, \dots, n-1 \}$, $\tau \neq \tau_0$), $a_{n, \tau} = \kappa_{n, \tau} - j_{\tau}$ ($\tau \neq \tau_0$), and $b_{\tau} = j_{\tau}$ ($\tau \in \Psi$). This is the only possible choice that works. See also, \cite[p.1175]{UFJ}.
\end{proof}

We now introduce the $p$-adic weights that will appear in this article. Let $(R, R^+)$ denote a complete Tate affinoid algebra over $(\mbb{Q}_p, \mbb{Z}_p)$. Let $s \geq 1$ be an integer. We say that a continuous character on $\mbb{Z}_p^{\times}$ (or several copies of $\mbb{Z}_p^{\times}$) is $s$-analytic if it is $s$-analytic as a continuous function (as in \S \ref{NotationsAndConventionsIntro}).

\begin{definition}
    Let $s \geq 1$ be an integer. We let $\mathcal{X}_{R, s}$ denote the group of $s$-analytic characters $(\kappa, j) \colon T(\mbb{Z}_p) \times S(\mbb{Z}_p) \to (R^+)^{\times}$ which satisfy:
    \begin{itemize}
        \item There exists an $s$-analytic character $w \colon \mbb{Z}_p^{\times} \to (R^+)^{\times}$ such that
        \[
        \kappa_{i, \tau_0} + \kappa_{2n+2-i,\tau_0} = w
        \]
        for all $i=2, \dots, n$.
        \item For all $i=1, \dots, n$ and $\tau \neq \tau_0$, we have
        \[
        \kappa_{i, \tau} + \kappa_{2n+1 - i, \tau} = 0 .
        \]
    \end{itemize}
    We let $\mathcal{X}_R = \cup_{s \geq 1} \mathcal{X}_{R, s}$ denote the corresponding space of locally analytic characters. 
\end{definition}

\begin{remark} \label{LAConeDecompositionWeightsRem}
    As in Lemma \ref{ConeOfWeightsForELemma}, for any $(\kappa, j) \in \mathcal{X}_{R, s}$, we can write
    \begin{align*} 
    (\kappa, j) = a_0 \circ (\mu_0, 0) + a_w \circ (\mu_w, 0) &+ a_{n+1, \tau_0} \circ (\mu_{n+1, \tau_0}, 0) \\ &+ \sum_{i=1}^n \sum_{\tau \in \Psi} a_{i, \tau} \circ (\mu_{i, \tau}, 0) + \sum_{\tau \in \Psi} b_{\tau} \circ (\mu_{n, \tau}, 1_{\tau}) 
    \end{align*}
    for unique $s$-analytic characters $a_0, a_w, a_{i, \tau}, b_{\tau} \colon \mbb{Z}_p^{\times} \to (R^+)^{\times}$.
\end{remark}

\subsection{Branching laws} \label{BranchingLawsPreliminarySection}

In this section, we extend the branching laws in \cite[Appendix A]{UFJ} to cover more general anticyclotomic twists. To be more precise, if we take $j_{\tau_0} = 0$ (i.e., $j_{\tau_0}$ is the trivial character) in Proposition \ref{SanDeltaDaggerIsAnEvectorProp} below, then we recover \cite[Theorem A.5.10]{UFJ}. The key point of this section is to incorporate non-trivial $j_{\tau_0}$ into the branching laws; for the applications to $p$-adic $L$-functions, this will allow us to vary the infinity-type of the anticyclotomic twist freely (without the restriction in \cite[\S 8.1]{UFJ}).

\subsubsection{Classical branching laws} \label{ClassicalBranchingLawsSubSubSec}

Before discussing the classical branching law in this situation, we first recall a classical branching result for algebraic representations of general linear groups known as Pieri's rule (or more generally, the Littlewood--Richardson rule).

\begin{lemma} \label{PierisRuleLemma}
Let $d \geq 1$ be an integer, and $\kappa$ a dominant weight for the diagonal torus inside $\opn{GL}_d$, which we describe as a tuple of integers $(\kappa_1, \dots, \kappa_d)$ in the usual way. Let $V_{\kappa}$ be the algebraic representation of $\opn{GL}_d$ of highest weight $\kappa$, and for $j \geq 0$, let $W_{-j}$ denote the algebraic representation of $\opn{GL}_d$ of highest weight $(0, \dots, 0, -j)$.
\begin{enumerate}
    \item One has a $\opn{GL}_d$-equivariant decomposition
    \[
    V_{\kappa} \otimes W_{-j} = \bigoplus_{\kappa'} V_{\kappa'}
    \]
    where the sum is over all dominant $\kappa' = (\kappa_1 - t_1, \dots, \kappa_d - t_d)$ with $t_i \geq 0$ for all $i$, $\sum_{i=1}^d t_i = j$, and $t_{i} \leq \kappa_i - \kappa_{i+1}$ for all $i=1, \dots, d-1$.
    \item Suppose that $d = 2c -1$ is odd and consider the subgroup $\opn{GL}_{c-1} \times \opn{GL}_c \subset \opn{GL}_d$ embedded block diagonally. Let $S_{-j}$ denote the algebraic representation of $\opn{GL}_c$ of highest weight $(0, \dots, 0, -j)$. Then
    \begin{equation} \label{1GLcnonvanishing}
    \left( V_{\kappa}|_{1 \times \opn{GL}_c} \otimes S_{-j} \otimes \opn{det}^k \right)^{1 \times \opn{GL}_c} \neq 0, \quad \quad \quad \text{ for some } k \in \mbb{Z},
    \end{equation}
    implies that $0 \leq k + \kappa_c \leq j$.
\end{enumerate}
\end{lemma}
\begin{proof}
Part (1) follows from Pieri's rule (see \cite[Corollary 9.2.4]{GoodmanWallach}). For part (2), we know that if (\ref{1GLcnonvanishing}) holds then we have
\[
\left( V_{\kappa} \otimes W_{-j} \otimes \opn{det}^k \right)^{1 \times \opn{GL}_c} \neq 0 .
\]
By part (1), this implies $(V_{\kappa'} \otimes \opn{det}^k )^{1 \times \opn{GL}_c} \neq 0$ for some $\kappa'$ as in part (1). But by \cite[Theorem 2.1]{Knapp}, this implies that $k + \kappa'_c = k + \kappa_c - t_c = 0$ as required.
\end{proof}

Recall the definition of $\mathcal{E}$ from Definition \ref{DefOfAlgWeightsE}. For any $(\kappa, j) \in \mathcal{E}$, we let $\sigma_{\kappa}^{[j]} \colon M_H \to \mbb{G}_m$ denote the character given by sending a general element $(x; y_1, y_2, y_3; z_{1, \tau}, z_{2, \tau})_{\tau \neq \tau_0}$ to 
    \begin{equation} \label{ShapeOfsigmakjEqn}
    x^{-\kappa_0} y_1^{-\kappa_{1, \tau_0} - j_{\tau_0}} \opn{det}y_2^{\kappa_{n+1, \tau_0} - j_{\tau_0} - w} \opn{det}y_3^{-\kappa_{n+1, \tau_0} + j_{\tau_0}} \prod_{\tau \neq \tau_0} \opn{det} z_{1, \tau}^{-j_{\tau}} \opn{det}z_{2, \tau}^{j_{\tau}} .
    \end{equation} 

\begin{theorem} \label{TheoremForClassicalBranching}
Let $(\kappa, j) \in \mathcal{E}$, and let $S_{-j}$ denote the algebraic representation of $M_H$ with highest weight trivial outside the $\tau_0$-component, and $(j_{\tau_0},0, \dots, 0, -j_{\tau_0})$ at $\tau_0$. Let $V_{\kappa}$ denote the algebraic representation of $M_G$ with highest weight $\kappa$. We equip $V_{\kappa} \otimes S_{-j}$ with the action $\star$ of $u^{-1}M_Hu$ given on pure tensors by
\[
    (u^{-1} m u) \star (a \otimes b) \defeq [(u^{-1}mu) \cdot a] \otimes [m \cdot b], \quad \quad m \in M_H .
\]
Then there exists a unique vector $x_{\kappa}^{[j]} \in V_{\kappa} \otimes S_{-j}$ with the following properties:
\begin{enumerate}
    \item $x_{\kappa}^{[j]}$ is an eigenvector for the action of $u^{-1}M_Hu$ via $\star$, with eigencharacter given by the inverse of the character $\sigma_{\kappa}^{[j]}$, i.e. we have 
    \[
    (u^{-1}m u) \star x_{\kappa}^{[j]} = \sigma_{\kappa}^{[j]}(m)^{-1} x_{\kappa}^{[j]}
    \]
    for all $m \in M_H$.
    \item If we identify $V_{\kappa} \otimes S_{-j}$ with the space of all algebraic functions $f \colon M_G \times M_H \to \mbb{A}^1$ satisfying the following transformation property:
    \[
    f((m_1, m_2) \cdot (b, q)) = w_{M_G}^{\opn{max}} \kappa(b^{-1}) \cdot q_1^{-j_{\tau_0}} q_3^{j_{\tau_0}} \cdot f(m_1, m_2)
    \]
    where $(b, q) \in B_{M_G} \times Q_{M_H}$ and the projection of $q$ to the $\tau_0$-component of the Levi is $(q_1, q_2, q_3, q_4)$, then we have
    \[
    x_{\kappa}^{[j]}(1, v) = 1 .
    \]
    Here the $M_G \times M_H$ action is via $[(g, h) \cdot f](m_1, m_2) = f(g^{-1}m_1, h^{-1}m_2)$. 
\end{enumerate}
\end{theorem}
\begin{proof}
Let $W_{-j}$ denote the algebraic representation of $M_G$ with the same highest weight as $S_{-j}$. Then Lemma \ref{PierisRuleLemma}(1) implies that $V_{\kappa'}$ appears in $V_{\kappa} \otimes W_{-j}$ with multiplicity one, where $\kappa'$ is the weight equal to $\kappa$ except for $\kappa_{1, \tau_0}' = \kappa_{1, \tau_0}+j_{\tau_0}$ and $\kappa_{n+1, \tau_0}' = \kappa_{n+1, \tau_0} - j_{\tau_0}$. It is shown in \cite[Theorem A.5.4]{UFJ} that there exists a (non-zero) vector $v_{\kappa}^{[j]} \in V_{\kappa'}$ with multiplicity one on which $M_H$ acts through the inverse of the character $\sigma_{\kappa}^{[j]}$. In fact it has multiplicity one in $V_{\kappa} \otimes W_{-j}$ (by Lemma \ref{PierisRuleLemma}(1), $\kappa'$ is the unique weight appearing in the decomposition with $\kappa_{n+1, \tau_0}' = \kappa_{n+1, \tau_0} - j_{\tau_0}$). 

Decompose $W_{-1}|_{M_H} = T_{-1} \oplus S_{-1}$, where $T_{-1}$ is the algebraic representation of $M_H$ with highest weight in the $\tau_0$-component given by $(1, 0, \dots,0, -1, 0, \dots, 0)$ (where the $-1$ is in the $n$-th place). We have the following decomposition 
\[
W_{-j}|_{M_H} = \bigoplus_{i=0}^{j_{\tau_0}} \opn{Sym}^i(T_{-1}) \otimes \opn{Sym}^{j_{\tau_0}-i} (S_{-1}) = \bigoplus_{i=0}^{j_{\tau_0}} \opn{Sym}^i(T_{-1}) \otimes S_{-(j-i)} 
\]
where $S_{-(j-i)}$ has highest weight trivial outside the $\tau_0$-component, and $(j_{\tau_0} - i, 0, \dots, 0, -(j_{\tau_0} - i))$ in the $\tau_0$-component. Then we must have that $v_{\kappa}^{[j]}$ appears in $V_{\kappa} \otimes \opn{Sym}^i(T_{-1}) \otimes S_{-(j-i)}$ for some $0 \leq i \leq j_{\tau_0}$, which implies that
\[
\left( V_{\kappa}|_{M_H} \otimes S_{-(j-i)} \otimes \opn{det}_{\tau_0}^{-\kappa_{n+1, \tau_0} + j_{\tau_0}} \right)^{M'} \neq 0.
\]
where $M' \subset M_H$ denotes the subgroup which in the $\tau_0$-component is equal to $\opn{GL}_1 \times \{1\} \times \opn{GL}_n$, and is trivial outside the $\tau_0$-component. Here $\opn{det}_{\tau_0}$ denotes the determinant character in the $\tau_0$-component. By Lemma \ref{PierisRuleLemma}(2), this implies that $0 \leq j_{\tau_0} \leq j_{\tau_0}-i$, hence $i = 0$. We now define $x_{\kappa}^{[j]}$ as $(u^{-1}, 1) \cdot v_{\kappa}^{[j]}$ and the rest of the theorem follows from Lemma \ref{OpenOrbitForgammauvLemma}.
\end{proof}

We note the following crucial property of the vectors in the above theorem. Let $\overline{N}_{M_G}$ and $\overline{N}_{M_H}$ denote the opposite unipotent radicals of $B_{M_G}$ and $Q_{M_H}$ respectively. For an integer $\beta \geq 1$, let $\overline{N}_{M_G, \beta}^1 \subset \overline{N}_{M_G}(\mbb{Z}_p)$ denote the subgroup of elements which are congruent to the identity modulo $p^{\beta}$.

\begin{corollary} \label{xkappajValuedInCor}
Let $\beta \geq 1$. Consider the $B_{M_G} \times M_H$-equivariant map $V_{\kappa} \otimes S_{-j} \to w_{M_G}^{\opn{max}}\kappa \otimes S_{-j}$ given by restricting an algebraic function on $M_G \times M_H$ to $1 \times M_H$. Then the image of $x_{\kappa}^{[j]}$ under this map is a lowest weight vector. Furthermore, this implies that
\[
x_{\kappa}^{[j]}(n, m) \in 1 + p^{\beta} \mbb{Z}_p \subset \mbb{Z}_p^{\times}
\]
for $(n, m) \in \overline{N}_{M_G, \beta}^1 \times \overline{N}_{M_H}(\mbb{Z}_p)$. 
\end{corollary}
\begin{proof}
Let $T^{\clubsuit} \subset T$ denote the subtorus of elements of the form $(t_0; t_{1, \tau}, \dots, t_{2n, \tau})$ with the condition that $t_{i, \tau_0} = t_{2n+2-i, \tau_0}$ for $i=2, \dots, n$ and $t_{i, \tau} = t_{2n+1 - i, \tau}$ for all $\tau \neq \tau_0$ and $i=1, \dots, n$. Let $\xi$ denote the highest weight of $S_{-j}$. 

We first claim that the eigenspace $S_{-j}[T^{\clubsuit} = w_{M_H}^{\opn{max}}\xi] \subset S_{-j}$ where $T^{\clubsuit}$ acts through the character $w_{M_H}^{\opn{max}}\xi$ is the lowest weight subspace. Indeed, the action of $M_H$ factors through the projection $M_H \twoheadrightarrow M'$, where $M'$ is as in the proof of Theorem \ref{TheoremForClassicalBranching}, and the image of $T^{\clubsuit}$ under this projection is the maximal torus.

Now we note that $\sigma^{[j], -1}_{\kappa}(t) = w_{M_G}^{\opn{max}}\kappa(t) w_{M_H}^{\opn{max}}\xi(t)$ for all $t \in T^{\clubsuit}$. Furthermore, for any $t \in T^{\clubsuit}$, we have $u^{-1} t u = t$, hence we see that $x_{\kappa}^{[j]}$ maps to a lowest weight vector. It is non-zero because $x_{\kappa}^{[j]}(1, v) = 1$.

Finally, we claim that $x_{\kappa}^{[j]}$ extends to an algebraic function $M_{G, \mbb{Z}_p} \times M_{H, \mbb{Z}_p} \to \mbb{A}^1_{\mbb{Z}_p}$. This follows the same argument as in \cite[Proposition 3.2.6]{LRZ21}, namely the function $x_{\kappa}^{[j]}$ is regular on the union of $M_{G, \mbb{Q}_p} \times M_{H, \mbb{Q}_p}$ and the spherical cell $[(u^{-1}M_{H, \mbb{Z}_p} u) \times M_{H, \mbb{Z}_p}] \cdot (1, v) \cdot [B_{M_{G}, \mbb{Z}_p} \times Q_{M_{H}, \mbb{Z}_p}]$, and the complement of this union has codimension $\geq 2$ in $M_{G, \mbb{Z}_p} \times M_{H, \mbb{Z}_p}$ by Lemma \ref{OpenOrbitForgammauvLemma}. Hence $x_{\kappa}^{[j]}$ extends to a regular function on $M_{G, \mbb{Z}_p} \times M_{H, \mbb{Z}_p}$ by Hartogs' lemma. We therefore see that the vector
\[
x_{\kappa}^{[j]}|_{\overline{N}_{M_G}(\mbb{Z}_p) \times \overline{N}_{M_H}(\mbb{Z}_p)}
\]
is a polynomial with $\mbb{Z}_p$-coefficients in the coordinates corresponding to the opposite unipotents, and $x_{\kappa}^{[j]}|_{1 \times \overline{N}_{M_H}(\mbb{Z}_p)}$ is identically equal to $1$, by our normalisations and the fact that it is a lowest weight vector. The result follows.
\end{proof}

\subsubsection{Differential operators}

In this section, we introduce certain locally analytic functions which will act on the space of nearly holomorphic/overconvergent automorphic forms (defined in \S \ref{ClassicalNearlyHolFormsSection} and \S \ref{NearlyOverconvergentAutoFormsSection}) through differential operators. We construct these functions by $p$-adically interpolating the branching law in the previous section. Throughout this section, we fix an integer $\beta \geq 1$.

We first discuss the classical case. Let $C^{\opn{pol}}(\mbb{Q}_p^{2n-1}, \mbb{Q}_p)$ denote the $\mbb{Q}_p$-algebra of polynomial functions $\mbb{Q}_p^{2n-1} \to \mbb{Q}_p$. This naturally carries an action of $M_G(\mbb{Q}_p)$ as follows. We can view a (row) vector $a \in \mbb{Q}_p^{2n-1}$ as a lower triangular block matrix $\tbyts{1}{}{a^t}{1}$ in (the $\tau_0$-component of) $\overline{U}_G$, where $U_G$ denotes the unipotent radical of $P_G$. The action is then given by
\begin{equation} \label{ActionOnCpolEqnPrelim}
(m \cdot \phi)(a) = \phi( m^{-1} \tbyt{1}{}{a^t}{1} m ), \quad \phi \in C^{\opn{pol}}(\mbb{Q}_p^{2n-1}, \mbb{Q}_p), m \in M_G(\mbb{Q}_p).
\end{equation}

\begin{remark}
With this action, $C^{\opn{pol}}(\mbb{Q}_p^{2n-1}, \mbb{Q}_p)$ is naturally isomorphic to the universal enveloping algebra of $\ide{u}_G =\opn{Lie}U_G$ (equipped with the adjoint action).
\end{remark}

Given a partition $\sum_{i=n+1}^{2n} t_i = j_{\tau_0}$ with $t_i \geq 0$, we can consider the polynomial function 
\[
(a_2, \dots, a_{2n}) \mapsto \prod_{i=n+1}^{2n} a_i^{t_i} \in C^{\opn{pol}}(\mbb{Q}_p^{2n-1}, \mbb{Q}_p).
\]
The subspace spanned by these functions is stable under $M_H(\mbb{Q}_p) \subset M_G(\mbb{Q}_p)$ and is isomorphic to the representation $S_{-j}$. When $t_{n+1} = j_{\tau_0}$ (and $t_i = 0$ for $i > n+1$), this is a lowest weight vector in this representation.

Recall that we view $S_{-j}$ as the space of algebraic functions $f \colon M_H \to \mbb{A}^1$ satisfying the transformation property:
\[
f(m q) = q_1^{-j_{\tau_0}} q_3^{j_{\tau_0}} f(m)
\]
for all $m \in M_H$, $q \in Q_{M_H}$, where $(q_1, q_2, q_3, q_4)$ denotes the projection of $q$ to the $\tau_0$-component of the Levi of $Q_{M_H}$. We let $v^{\opn{can}} \in S_{-j}$ denote the lowest weight vector which satisfies $v^{\opn{can}}(1) = 1$.

\begin{convention} \label{ConventionForPhiBeta}
We normalise the $M_H$-equivariant map $S_{-j} \subset C^{\opn{pol}}(\mbb{Q}_p^{2n-1}, \mbb{Q}_p)$ described above by sending the lowest weight vector $v^{\opn{can}}$ to the function 
\[
(a_2, \dots, a_{2n}) \mapsto a_{n+1}^{j_{\tau_0}} 
\]
and we denote the resulting embedding by $\Phi \colon S_{-j} \to C^{\opn{pol}}(\mbb{Q}_p^{2n-1}, \mbb{Q}_p)$. In particular, if $f \in S_{-j}$ is an algebraic function and $\Phi(f) \in C^{\opn{pol}}(\mbb{Q}_p^{2n-1}, \mbb{Q}_p)$ denotes its image under this embedding, then we have 
\[
(m^{-1} \cdot \Phi(f))((0, \dots,0, 1, 0, \dots, 0)) = f(m), \quad \quad m \in M_H(\mbb{Q}_p)
\]
where $1$ is in the $n$-th place (i.e. $a_{n+1} = 1$). 
\end{convention}

We make the following definition:

\begin{definition} \label{DefinitionOfDeltakappaj}
Let $(\kappa, j) \in \mathcal{E}$. We let $\delta_{\kappa, j} \in V_{\kappa} \otimes C^{\opn{pol}}(\mbb{Q}_p^{2n-1}, \mbb{Q}_p)$ denote the image of $v_{\kappa}^{[j]}$ under the map 
\[
1 \times \Phi \colon V_{\kappa} \otimes S_{-j} \to V_{\kappa} \otimes C^{\opn{pol}}(\mbb{Q}_p^{2n-1}, \mbb{Q}_p) .
\]
Note that $\delta_{\kappa, j}$ is an eigenvector for the diagonal action of $M_{H}(\mbb{Q}_p)$ with eigencharacter $\sigma_{\kappa}^{[j], -1}$.
\end{definition}

We now discuss the $p$-adic analogue of these polynomial functions. Recall that $M^G_{\opn{Iw}}(p^{\beta})$ denotes the depth $p^{\beta}$ upper-triangular Iwahori subgroup in $M_G(\mbb{Z}_p)$.

\begin{definition} \label{DefOfUGbetacircSupport}
We let $U_{G, \beta} = p^{-\beta} \mbb{Z}_p^{\oplus n} \oplus \mbb{Z}_p^{\oplus n-1} \subset \mbb{Q}_p^{\oplus 2n-1}$, and we let
\[
U_{G, \beta}^{\circ} = \mbb{Z}_p^{\oplus n-1} \oplus \mbb{Z}_p^{\times} \oplus (p\mbb{Z}_p)^{\oplus n-1} \subset U_{G, \beta}.
\]
Note that both $U_{G, \beta}$ and $U_{G, \beta}^{\circ}$ are stable under the conjugation action of $M^G_{\opn{Iw}}(p^{\beta})$.
\end{definition}

We have the following important lemma, which allows us to $p$-adically interpolate the polynomial functions above.

\begin{lemma} \label{uudeltaInZpstarLemma}
If we view $\delta_{\kappa, j}$ as a polynomial function on $M_G(\mbb{Q}_p) \times \mbb{Q}_p^{\oplus 2n-1}$ satisfying the transformation property $\delta_{\kappa, j}(mb, a) = (w_{M_G}^{\opn{max}}\kappa)(b^{-1}) \delta_{\kappa, j}(m, a)$ for all $b \in B_{M_G}$, then
\[
[(u^{-1}, u^{-1}) \cdot \delta_{\kappa, j}] (M^G_{\opn{Iw}}(p^{\beta}) \times U_{G, \beta}^{\circ}) \subset \mbb{Z}_p^{\times} .
\]
\end{lemma}
\begin{proof}
Let $b = (a_2, \dots, a_n, a_{n+1}, a_{n+2}, \dots, a_{2n}) \in U^{\circ}_{G, \beta}$ (so $a_i \in \mbb{Z}_p$, $a_{n+1} \in \mbb{Z}_p^{\times}$ and $a_{n+2}, \dots, a_{2n} \in p\mbb{Z}_p$). Consider the element $z \in \overline{Q}_{M_H}(\mbb{Z}_p)$ which is the identity outside the $\tau_0$-component, and in the $\tau_0$-component is given by
\[
1 \times 1 \times \tbyt{a_{n+1}}{}{c}{1}
\]
where $\tbyt{a_{n+1}}{}{c}{1}$ is the $n \times n$-matrix block matrix with $1 \times 1$ upper left block and $(n-1) \times (n-1)$-lower block, with $c = (a_{n} + a_{n+2}, a_{n-1} + a_{n+3}, \dots, a_2 + a_{2n})^t$. 

By Convention \ref{ConventionForPhiBeta} and the definition of $x_{\kappa}^{[j]}$, we see that for all $i \in M^G_{\opn{Iw}}(p^{\beta})$
\begin{align*}
[(u^{-1}, u^{-1}) \cdot \delta_{\kappa, j}](i, b) &= (1 \times \Phi)(x_{\kappa}^{[j]})(i, (a_2, \dots, a_{n+1}, a_n + a_{n+2}, \dots, a_2 + a_{2n})) \\ &= (1 \times \Phi)(x_{\kappa}^{[j]})(i, (0, \dots,0, a_{n+1}, a_n + a_{n+2}, \dots, a_2 + a_{2n})) \\ &= (1 \times z^{-1}) \cdot (1 \times \Phi)(x_{\kappa}^{[j]})(i, (0, \dots,0, 1, 0, \dots, 0)) \\ &= x_{\kappa}^{[j]}(i, z).
\end{align*}
This is an element of $\mbb{Z}_p^{\times}$ by Corollary \ref{xkappajValuedInCor}.
\end{proof}

We make the following definition:

\begin{definition} \label{DeltaDaggerKJBDef}
    For $(\kappa, j) \in \mathcal{E}$, let $\delta^{\dagger}_{\kappa, j, \beta} \in V_{\kappa} \otimes C^{\opn{pol}}(U_{G, \beta}, \mbb{Q}_p)$ denote the restriction of $(u^{-1}, u^{-1}) \cdot \delta_{\kappa, j}$ to $M_G(\mbb{Q}_p) \times U_{G, \beta}$.
\end{definition}

Let $s \geq 1$ be an integer and $(R, R^+)$ a complete Tate affinoid algebra over $(\mbb{Q}_p, \mbb{Z}_p)$. For any $(\kappa, j) \in \mathcal{X}_{R, s}$, we let $\sigma_{\kappa}^{[j]} \colon M_H(\mbb{Z}_p) \to R^{\times}$ denote the corresponding $s$-analytic character defined by the same formula as in (\ref{ShapeOfsigmakjEqn}). Let $C^{s\opn{-an}}(U_{G, \beta}, R)$ denote the $R$-algebra of $s$-analytic functions $U_{G, \beta} \to R$, which comes equipped with an action of $M^G_{\opn{Iw}}(p^{\beta})$ via the same formula as in (\ref{ActionOnCpolEqnPrelim}). We can $p$-adically interpolate the vectors in Definition \ref{DeltaDaggerKJBDef} by following the same strategy in \cite[Appendix A]{UFJ}. More precisely, let $V^{s\opn{-an}}_{\kappa}$ denote the $s$-analytic induction 
\[
V_{\kappa}^{s\opn{-an}} = \left\{ f \colon M^G_{\opn{Iw}}(p^{\beta}) \to \mbb{Q}_p : \begin{array}{c} f \text{ is } s\text{-analytic } \\ f(- \cdot b) = (w_{M_G}^{\opn{max}} \kappa)(b^{-1}) f(-) \text{ for all } b \in B_{M_G}(\mbb{Q}_p) \cap M^G_{\opn{Iw}}(p^{\beta}) \end{array} \right\}
\]
which carries an action of $m \in M^G_{\opn{Iw}}(p^{\beta})$ via the formula $m \cdot f(-) = f(m^{-1}\cdot -)$. 

\begin{definition}
Let $(\kappa, j) \in \mathcal{X}_{R, s}$ and write
    \begin{align*} 
    (\kappa, j) = a_0 \circ (\mu_0, 0) + a_w \circ (\mu_w, 0) &+ a_{n+1, \tau_0} \circ (\mu_{n+1, \tau_0}, 0) \\ &+ \sum_{i=1}^n \sum_{\tau \in \Psi} a_{i, \tau} \circ (\mu_{i, \tau}, 0) + \sum_{\tau \in \Psi} b_{\tau} \circ (\mu_{n, \tau}, 1_{\tau}) 
    \end{align*}
    for unique $s$-analytic characters $a_0, a_w, a_{i, \tau}, b_{\tau} \colon \mbb{Z}_p^{\times} \to (R^+)^{\times}$ (see Remark \ref{LAConeDecompositionWeightsRem}). Then we define $\delta^{\dagger, s\opn{-an}}_{\kappa, j, \beta} \colon M^G_{\opn{Iw}}(p^{\beta}) \times U_{G, \beta} \to R$ as the extension-by-zero of the function supported on $M^G_{\opn{Iw}}(p^{\beta}) \times U_{G, \beta}^{\circ}$ given by
\[
(\delta^{\dagger}_{\mu_0, 0, \beta})^{a_0} \cdot (\delta^{\dagger}_{\mu_w,0, \beta})^{a_w} \cdot (\delta^{\dagger}_{\mu_{n+1, \tau_0},0, \beta})^{a_{n+1, \tau_0}} \cdot \prod_{\substack{i=1, \dots, n \\ \tau \in \Psi}} (\delta^{\dagger}_{\mu_{i, \tau},0, \beta})^{a_{i, \tau}} \cdot \prod_{\tau \in \Psi} (\delta^{\dagger}_{\mu_{n, \tau},1_{\tau}, \beta})^{b_{\tau}}
\]
which is well-defined by Lemma \ref{uudeltaInZpstarLemma}. This function defines an element $\delta^{\dagger, s\opn{-an}}_{\kappa, j, \beta} \in V_{\kappa}^{s\opn{-an}} \hatot C^{s\opn{-an}}(U_{G, \beta}, R)$, and note that $\delta^{\dagger, s\opn{-an}}_{\kappa, j, \beta} \in V_{\kappa} \otimes C^{s\opn{-an}}(U_{G, \beta}, \mbb{Q}_p)$ whenever $(\kappa, j) \in \mathcal{E}$.
\end{definition}

These vectors satisfy the following properties:

\begin{proposition} \label{SanDeltaDaggerIsAnEvectorProp}
    We have the following:
    \begin{enumerate}
        \item The element $\delta^{\dagger, s\opn{-an}}_{\kappa, j, \beta}$ is an eigenvector for the action of $u^{-1}M^H_{\diamondsuit}(p^{\beta})u \subset M^G_{\opn{Iw}}(p^{\beta})$ with eigencharacter $\sigma_{\kappa}^{[j], -1}$.
        \item The construction of $\delta^{\dagger, s\opn{-an}}_{\kappa, j, \beta}$ is compatible with changing the Tate affinoid algebra $(R, R^+)$ and the radius $s$ of analyticity (in the obvious way).
        \item Let $(\kappa, j) \in \mathcal{E}$ and let $\chi = (\chi_{\tau}) \colon \prod_{\tau \in \Psi} \mbb{Z}_p^{\times} \to L^{\times}$ be a finite order character, for some finite extension $L/\mbb{Q}_p$. Let $\beta \geq 1$ be any integer such that $\chi_{\tau}$ is trivial on $1 + p^{\beta}\mbb{Z}_p$ for all $\tau \in \Psi$. Let $1_{U_{G, \beta}^{\circ}, \chi}$ denote the weighted indicator function given by
        \[
        1_{U_{G, \beta}^{\circ}, \chi}(a_2, \dots, a_{2n}) \defeq \left\{ \begin{array}{cc} \chi_{\tau_0}(a_{n+1}) & \text{ if } (a_2, \dots, a_{2n}) \in U_{G,\beta}^{\circ} \\ 0 & \text{ otherwise } \end{array} \right. \quad \in C^{s\opn{-an}}(U_{G, \beta}, L)
        \]
        with $s \geq 1$ sufficiently large. Then $\delta_{\kappa, j+\chi, \beta}^{\dagger, s\opn{-an}} = 1_{U_{G, \beta}^{\circ}, \chi} \cdot \delta_{\kappa, j, \beta}^{\dagger}$.
        \item Let $\beta' \geq \beta$. Then the vector $\delta^{\dagger, s\opn{-an}}_{\kappa, j, \beta'}$ is equal to the image of $\delta^{\dagger, s\opn{-an}}_{\kappa, j, \beta}$ under the map induced from restricting a function on $M^G_{\opn{Iw}}(p^{\beta}) \times U_{G, \beta}$ to $M^G_{\opn{Iw}}(p^{\beta'}) \times U_{G, \beta}$ and then extending by zero to $M^G_{\opn{Iw}}(p^{\beta'}) \times U_{G, \beta'}$.
    \end{enumerate}
\end{proposition}
\begin{proof}
For part (1), note that for $(\kappa, j) \in \mathcal{E}$, the element $\delta^{\dagger}_{\kappa, j, \beta}$ is an eigenvector for the action of $u^{-1}M^H_{\diamondsuit}(p^{\beta})u$ with eigencharacter $\sigma_{\kappa}^{[j], -1} \colon M^H_{\diamondsuit}(p^{\beta}) \to \mbb{Z}_p^{\times}$. Hence, for $(\kappa, j) \in \mathcal{X}_{R, s}$, the element $\delta^{\dagger, s\opn{-an}}_{\kappa, j, \beta}$ is an eigenvector for the action of $u^{-1}M^H_{\diamondsuit}(p^{\beta})u$ with eigencharacter
\[
\sigma_{\mu_0}^{[0],-a_0} \cdot \sigma_{\mu_w}^{[0], -a_w} \cdot \sigma_{\mu_{n+1, \tau_0}}^{[0], -a_{n+1, \tau_0}} \cdot \prod_{\substack{i=1, \dots, n \\ \tau \in \Psi}} \sigma_{\mu_{i, \tau}}^{[0], -a_{i, \tau}} \cdot \prod_{\tau \in \Psi} \sigma_{\mu_{n, \tau}}^{[1_{\tau}], -b_{\tau}} = \sigma_{\kappa}^{[j], -1} .
\]
Parts (2) and (4) follow from the analogous compatibilities for $\delta^{\dagger}_{\kappa, j, \beta}$ with $(\kappa, j) \in \mathcal{E}$ (recall that $\delta^{\dagger}_{\kappa, j, \beta}$ is just the restriction of the algebraic function $(u^{-1}, u^{-1}) \cdot \delta_{\kappa, j}$ to $M_G(\mbb{Q}_p) \times U_{G, \beta}$, and $U^{\circ}_{G, \beta}$ is independent of $\beta$).

For part (3), the claim is clear when $\chi_{\tau_0}$ is trivial, so without loss of generality we may assume that $\chi_{\tau}$ is trivial for $\tau \neq \tau_0$. It is enough to show that $\delta^{\dagger, s\opn{-an}}_{0, \chi, \beta}(i, -) \in C^{s\opn{-an}}(U_{G, \beta}, L)$ is equal to $1_{U_{G, \beta}^{\circ}, \chi}$ for any $i \in \overline{N}^1_{M_G, \beta}$. Clearly $\delta^{\dagger, s\opn{-an}}_{0, \chi, \beta}(i,-)$ is supported on $U_{G, \beta}^{\circ}$. Let $\underline{a} = (a_2, \dots, a_{2n}) \in U_{G, \beta}^{\circ}$. Then, by construction, we have
    \[
    \delta^{\dagger, s\opn{-an}}_{0, \chi, \beta}(i,\underline{a}) = \chi_{\tau_0}\left( \delta^{\dagger}_{\mu_{n, \tau_0}, 1_{\tau_0}, \beta}(i, \underline{a}) \cdot \delta^{\dagger}_{\mu_{n, \tau_0}, 0, \beta}(i, \underline{a})^{-1} \right)
    \]
    where we view $\delta^{\dagger}_{\mu_{n, \tau_0}, 0, \beta}$ and $\delta^{\dagger}_{\mu_{n, \tau_0}, 1_{\tau_0}, \beta}$ as functions on $M_G(\mbb{Q}_p) \times U_{G, \beta}$. Let $z(\underline{a}) \in \overline{Q}_{M_H}(\mbb{Z}_p)$ denote the element which is the identity outside the $\tau_0$-component, and in the $\tau_0$-component is equal to the block matrix
    \[
    1 \times 1 \times \tbyt{a_{n+1}}{}{c}{1}
    \]
    with $c = (a_n + a_{n+2}, a_{n-1}+ a_{n+3}, \dots, a_2 + a_{2n})^t$. From the proof of Lemma \ref{uudeltaInZpstarLemma} and Corollary \ref{xkappajValuedInCor}, we see that $\delta^{\dagger}_{\mu_{n, \tau_0}, 0, \beta}(i, \underline{a}) = x_{\mu_{n, \tau_0}}^{[0]}(i, z(\underline{a})) \in 1+p^{\beta}\mbb{Z}_p$ and
    \[
    \delta^{\dagger}_{\mu_{n, \tau_0}, 1_{\tau_0}, \beta}(i, \underline{a}) = x_{\mu_{n, \tau_0}}^{[1_{\tau_0}]}(i, z(\underline{a})) \quad \in \quad a_{n+1} + p^{\beta}\mbb{Z}_p
    \]
    where the containment uses the transformation properties of the vector $x_{\mu_{n, \tau_0}}^{[1_{\tau_0}]}$ under right-translation by $Q_{M_H}$. The result follows. 
\end{proof}

\begin{remark}
If $j$ is the trivial character on the $\tau_0$-component and $j' \colon \prod_{\tau \neq \tau_0} \mbb{Z}_p^{\times} \to (R^+)^{\times}$ is the character away from $\tau_0$, then $\delta_{\kappa,j, \beta}^{\dagger, s\opn{-an}} = x_{\kappa}^{[j']} \otimes 1_{U_{G, \beta}^{\circ}}$ where $x_{\kappa}^{[j']} \in V_{\kappa}^{s\opn{-an}}$ is the vector constructed in \cite[Theorem A.5.10]{UFJ} (in the notation of \emph{loc.cit.}).
\end{remark}

\subsection{The main construction} \label{TheMainConstructionSubSec}

We now describe the key construction that will be used in the definition of the ``evaluation maps'' (which, in turn, will give rise to the $p$-adic $L$-functions). This abstract construction will be applied in \S \ref{ClassicalEvaluationMapsSubSec}, \S \ref{LACechVersionMainConstructionSSec} and can be skipped on first reading. As in the previous subsection, we fix an integer $\beta \geq 1$ which we will often omit from the notation. We also allow ourselves to work over a fixed finite extension $L/\mbb{Q}_p$. 

\subsubsection{Nearly holomorphic forms}

Suppose that we have a topological $L$-module $\mathcal{N}_G$ which comes equipped with an $L$-algebra action 
\begin{equation} \label{AMpol}
C^{\opn{pol}}(\mbb{Q}_p^{\oplus 2n-1}, L) \otimes \mathcal{N}_G \to \mathcal{N}_G .
\end{equation}
We assume that $\mathcal{N}_G$ has a left action of $M_G(\mbb{Q}_p)$ and the action map (\ref{AMpol}) is equivariant for the diagonal action of $M_G(\mbb{Q}_p)$. Suppose that we have a topological $L$-module $\mathcal{N}_H$ with a left action of $M_H(\mbb{Q}_p)$, and a continuous $M_H(\mbb{Q}_p)$-equivariant map $p \colon \mathcal{N}_G \to \mathcal{N}_H$. These spaces mimic the properties of (the sheaves of) nearly holomorphic modular forms appearing in \S \ref{NearlyHoloFormsSubSec}.

We consider the following construction.

\begin{definition} \label{NHOLDefAbstractTheta}
Let $(\kappa, j) \in \mathcal{E}$. Set $\mathcal{N}_{G, \kappa^*} = \left(\mathcal{N}_G \otimes V_{\kappa}^* \right)^{M_G(\mbb{Q}_p)}$ and $\mathcal{N}_{H, \sigma_{\kappa}^{[j]}} = \left( \mathcal{N}_H \otimes \sigma_{\kappa}^{[j]} \right)^{M_H(\mbb{Q}_p)}$. Then we define an $L$-linear map
\[
\vartheta_{\kappa, j, \beta} \colon \mathcal{N}_{G, \kappa^*} \to \mathcal{N}_{H, \sigma_{\kappa}^{[j]}}
\]
induced from passing to $M_H(\mbb{Q}_p)$-invariants of the map
\begin{equation} \label{EMmap}
\sigma_{\kappa}^{[j], -1} \otimes \mathcal{N}_G \otimes V_{\kappa}^* \to \mathcal{N}_H  
\end{equation}
which is the composition of the following $M_{H}(\mbb{Q}_p)$-equivariant maps:
\begin{itemize}
    \item The morphism $\sigma_{\kappa}^{[j], -1} \otimes \mathcal{N}_G \otimes V_{\kappa}^* \to V_{\kappa} \otimes C^{\opn{pol}}(\mbb{Q}_p^{\oplus 2n-1}, L) \otimes \mathcal{N}_G \otimes V_{\kappa}^*$ induced from sending the first factor to $\delta_{\kappa, j}$.
    \item The morphism $V_{\kappa} \otimes C^{\opn{pol}}(\mbb{Q}_p^{\oplus 2n-1}, L) \otimes \mathcal{N}_G \otimes V_{\kappa}^* \to V_{\kappa} \otimes \mathcal{N}_G \otimes V_{\kappa}^*$ induced from the action map in (\ref{AMpol}).
    \item The morphism $V_{\kappa} \otimes \mathcal{N}_G \otimes V_{\kappa}^* \to \mathcal{N}_G$ induced from the natural map $V_{\kappa} \otimes V^*_{\kappa} \to L$.
    \item The morphism $p \colon \mathcal{N}_G \to \mathcal{N}_H$.
\end{itemize}
\end{definition}

\subsubsection{Nearly overconvergent forms of classical weight}

We now suppose that we have an ind-system $\mathcal{N}^{\dagger}_G$ of topological $L$-modules which are equipped with a left action of $M^G_{\opn{Iw}}(p^{\beta})$. We suppose that we have an $L$-algebra action
\begin{equation} \label{AMla}
C^{\opn{pol}}(U_{G, \beta}, L) \otimes \mathcal{N}^{\dagger}_G \to \mathcal{N}^{\dagger}_G
\end{equation}
equivariant for the diagonal action of $M^G_{\opn{Iw}}(p^{\beta})$ (i.e. (\ref{AMla}) is a morphism in the ind-category of topological $L$-modules equipped with a continuous action of $M^G_{\opn{Iw}}(p^{\beta})$). We also suppose that we have an ind-system $\mathcal{N}^{\dagger}_H$ of topological $L$-modules which are equipped with a left action of $M^H_{\diamondsuit}(p^{\beta})$, and a continuous $L$-linear morphism $p^{\dagger} \colon \mathcal{N}^{\dagger}_G \to \mathcal{N}^{\dagger}_H$ which is equivariant for the action of $M^H_{\diamondsuit}(p^{\beta})$, where $M^H_{\diamondsuit}(p^{\beta})$ acts on $\mathcal{N}^{\dagger}_G$ through the inclusion $u^{-1}M^H_{\diamondsuit}(p^{\beta})u \subset M^G_{\opn{Iw}}(p^{\beta})$. These spaces mimic (the ind-sheaves of) nearly overconvergent forms that appear in \S \ref{NearlyOverconvergentAutoFormsSection}. We will always view $M^H_{\diamondsuit}(p^{\beta})$ as acting on $\mathcal{N}^{\dagger}_{G}$ or any representation of $M^G_{\opn{Iw}}(p^{\beta})$ through the inclusion $u^{-1}M^H_{\diamondsuit}(p^{\beta})u \subset M^G_{\opn{Iw}}(p^{\beta})$. 

We consider the following construction:

\begin{definition} \label{AbstractDefOfThetadagger}
Let $(\kappa, j) \in \mathcal{E}$. Set $\mathcal{N}^{\dagger}_{G, \kappa^*} = \left(\mathcal{N}^{\dagger}_G \otimes V_{\kappa}^* \right)^{ M^G_{\opn{Iw}}(p^{\beta})}$ and $\mathcal{N}^{\dagger}_{H, \sigma_{\kappa}^{[j]}} = \left( \mathcal{N}^{\dagger}_H \otimes \sigma_{\kappa}^{[j]} \right)^{M^H_{\diamondsuit}(p^{\beta})}$. Then we define an $L$-linear map
\[
\vartheta^{\dagger}_{\kappa, j, \beta} \colon \mathcal{N}^{\dagger}_{G, \kappa^*} \to \mathcal{N}^{\dagger}_{H, \sigma_{\kappa}^{[j]}}
\]
induced by taking $M^H_{\diamondsuit}(p^{\beta})$-invariants of the map
\[
\sigma_{\kappa}^{[j], -1} \otimes \mathcal{N}^{\dagger}_G \otimes V_{\kappa}^* \to \mathcal{N}^{\dagger}_H  
\]
which is the composition of the following $M^H_{\diamondsuit}(p^{\beta})$-equivariant maps:
\begin{itemize}
    \item The morphism $\sigma_{\kappa}^{[j], -1} \otimes \mathcal{N}^{\dagger}_G \otimes V_{\kappa}^* \to V_{\kappa} \otimes C^{\opn{pol}}(U_{G, \beta}, L) \otimes \mathcal{N}^{\dagger}_G \otimes V_{\kappa}^*$ induced from sending the first factor to $\delta^{\dagger}_{\kappa, j, \beta}$ (this is indeed $M^H_{\diamondsuit}(p^{\beta})$-equivariant by our conventions above and Proposition \ref{SanDeltaDaggerIsAnEvectorProp}).
    \item The morphism $V_{\kappa} \otimes C^{\opn{pol}}(U_{G, \beta}, L) \otimes \mathcal{N}^{\dagger}_G \otimes V_{\kappa}^* \to V_{\kappa} \otimes \mathcal{N}^{\dagger}_G \otimes V_{\kappa}^*$ induced from the action map in (\ref{AMla}).
    \item The morphism $V_{\kappa} \otimes \mathcal{N}^{\dagger}_G \otimes V_{\kappa}^* \to \mathcal{N}^{\dagger}_G$ induced from the natural map $V_{\kappa} \otimes V^*_{\kappa} \to L$.
    \item The morphism $p^{\dagger} \colon \mathcal{N}^{\dagger}_G \to \mathcal{N}^{\dagger}_H$.
\end{itemize}
\end{definition}

The operators $\vartheta_{\kappa, j, \beta}$ and $\vartheta_{\kappa, j, \beta}^{\dagger}$ will be compatible in the following way. Suppose that we have a commutative diagram of continuous $L$-linear maps
\[
\begin{tikzcd}
\mathcal{N}_G \arrow[d, "\iota_G"'] \arrow[r, "p \circ u"] & \mathcal{N}_H \arrow[d, "\iota_H"] \\
\mathcal{N}^{\dagger}_G \arrow[r, "p^{\dagger}"]           & \mathcal{N}_H^{\dagger}           
\end{tikzcd}
\]
where the vertical maps are equivariant for the actions of $M^G_{\opn{Iw}}(p^{\beta})$ and $M^H_{\diamondsuit}(p^{\beta})$ respectively. In particular, the maps $\iota_{\bullet}$ induce morphisms on isotypic pieces for the actions of these groups. We also assume that $\iota_G$ intertwines the actions (\ref{AMpol}) and (\ref{AMla}) via the natural restriction map $C^{\opn{pol}}(\mbb{Q}_p^{\oplus 2n-1}, L) \to C^{\opn{pol}}(U_{G, \beta}, L)$.

\begin{proposition} \label{DaggerAndClassicalUnderBPisosProp}
    We have the following relation $\vartheta^{\dagger}_{\kappa, j, \beta} \circ \iota_G = \iota_H \circ \vartheta_{\kappa, j, \beta}$.
\end{proposition}
\begin{proof}
To prove this, we will show that we can reinterpret $\vartheta_{\kappa, j, \beta}$ using the map $p \circ u$ instead of $p$. More precisely, by replacing the first bullet point in Definition \ref{NHOLDefAbstractTheta} with the map induced from $(u^{-1}, u^{-1}) \cdot \delta_{\kappa, j}$ and replacing the last bullet point in Definition \ref{NHOLDefAbstractTheta} with $p \circ u$, we obtain an induced map
\begin{equation} \label{TwistedEM}
    \sigma_{\kappa}^{[j], -1} \otimes \mathcal{N}_G \otimes V_{\kappa}^* \to \mathcal{N}_H
\end{equation}
and one can verify that (\ref{TwistedEM})$(a \otimes u^{-1}b \otimes u^{-1}c) =$ (\ref{EMmap})$(a \otimes b \otimes c)$. Thus $\vartheta_{\kappa, j, \beta}$ is induced from passing to $M_H(\mbb{Q}_p)$-invariants of the map (\ref{TwistedEM}). 
The proposition follows, because $\delta^{\dagger}_{\kappa, j, \beta}$ is just the restriction of $(u^{-1}, u^{-1}) \cdot \delta_{\kappa, j}$ to $U_{G, \beta}$ (and the map $\iota_G$ is equivariant for the actions (\ref{AMpol}) and (\ref{AMla})).
\end{proof}

\subsubsection{Nearly overconvergent forms of \texorpdfstring{$p$}{p}-adic weight}

Continuing with the same notation and conventions as above, we now suppose that $\mathcal{N}^{\dagger}_G$ is an ind-system of Fr\'{e}chet spaces over $L$. We also suppose that (\ref{AMla}) extends to a continuous $L$-algebra action 
\begin{equation} \label{AMla2}
    C^{s\opn{-an}}(U_{G, \beta}, L) \; \hatot \; \mathcal{N}^{\dagger}_G \to \mathcal{N}^{\dagger}_G
\end{equation}
which is equivariant for the diagonal action of $M^G_{\opn{Iw}}(p^{\beta})$, for some integer $s \geq 1$ (recall our conventions on functors between ind-categories in \S \ref{NotationsAndConventionsIntro}). For an $s$-analytic character $\kappa \colon T(\mbb{Z}_p) \to R^{\times}$, let $D^{s\opn{-an}}_{\kappa^*}$ denote the continuous $R$-dual of 
\[
V^{\circ, s\opn{-an}}_{\kappa} \defeq \left\{ f \colon M^G_{\opn{Iw}}(p^{\beta}) \to R : \begin{array}{c} f \text{ is } (s+\varepsilon)\text{-analytic for all } \varepsilon > 0 \\ f(- \cdot b) = (w_{M_G}^{\opn{max}} \kappa)(b^{-1}) f(-) \text{ for all } b \in B_{M_G}(\mbb{Q}_p) \cap M^G_{\opn{Iw}}(p^{\beta}) \end{array}  \right\} .
\]
We consider the following construction:

\begin{definition} \label{AbstractDefOfThetadaggersan}
Let $(R, R^+)$ be a complete Tate affinoid algebra over $(L, \mathcal{O}_L)$, and let $(\kappa, j) \in \mathcal{X}_{R, s}$. Set $\mathcal{N}^{\dagger, s\opn{-an}}_{G, \kappa^*} = \left(\mathcal{N}^{\dagger}_G \hatot D^{s\opn{-an}}_{\kappa^*} \right)^{M^G_{\opn{Iw}}(p^{\beta})}$ and $\mathcal{N}^{\dagger, \opn{an}}_{H, \sigma_{\kappa}^{[j]}} = \left( \mathcal{N}^{\dagger}_H \hatot \sigma_{\kappa}^{[j]} \right)^{M^H_{\diamondsuit}(p^{\beta})}$. Then we define an $R$-linear map
\[
\vartheta^{\dagger, s\opn{-an}}_{\kappa,j, \beta} \colon \mathcal{N}^{\dagger, s\opn{-an}}_{G, \kappa^*} \to \mathcal{N}^{\dagger, \opn{an}}_{H, \sigma_{\kappa}^{[j]}}
\]
induced by taking $M^H_{\diamondsuit}(p^{\beta})$-invariants of the map
\[
\sigma_{\kappa}^{[j], -1} \; \hatot \; \mathcal{N}^{\dagger}_G \; \hatot \; D^{s\opn{-an}}_{\kappa^*} \to \mathcal{N}^{\dagger}_H  
\]
which is the composition of the following $M^H_{\diamondsuit}(p^{\beta})$-equivariant maps:
\begin{itemize}
    \item The morphism $\sigma_{\kappa}^{[j], -1} \hatot \mathcal{N}^{\dagger}_G \hatot D^{s\opn{-an}}_{\kappa^*} \to V_{\kappa}^{s\opn{-an}} \hatot C^{s\opn{-an}}(U_{G, \beta}, R) \hatot \mathcal{N}^{\dagger}_G \hatot D^{s\opn{-an}}_{\kappa^*}$ induced from sending the first factor to $\delta^{\dagger, s\opn{-an}}_{\kappa, j, \beta}$.
    \item The morphism $V_{\kappa}^{s\opn{-an}} \hatot C^{s\opn{-an}}(U_{G, \beta}, R) \hatot \mathcal{N}^{\dagger}_G \hatot D^{s\opn{-an}}_{\kappa^*} \to V_{\kappa}^{s\opn{-an}} \hatot \mathcal{N}^{\dagger}_G \hatot D^{s\opn{-an}}_{\kappa^*}$ induced from the action map in (\ref{AMla}).
    \item The morphism $V_{\kappa}^{s\opn{-an}} \hatot \mathcal{N}^{\dagger}_G \hatot D^{s\opn{-an}}_{\kappa^*} \to \mathcal{N}^{\dagger}_G$ induced from the natural map $V_{\kappa}^{s\opn{-an}} \hatot D^{s\opn{-an}}_{\kappa^*} \to R$.
    \item The morphism $p^{\dagger} \colon \mathcal{N}^{\dagger}_G \to \mathcal{N}^{\dagger}_H$.
\end{itemize}
\end{definition}

We will also need a version for locally algebraic characters.

\begin{definition} \label{abstractcirclalgTheta}
    Let $(\kappa, j) \in \mathcal{E}$ and let $\chi = (\chi_{\tau}) \colon \prod_{\tau \in \Psi} \mbb{Z}_p^{\times} \to L^{\times}$ be a finite-order character such that $\chi_{\tau}$ is trivial on $1 + p^{\beta}\mbb{Z}_p$ for all $\tau \in \Psi$. Then we define an $L$-linear map
    \[
    \vartheta_{\kappa, j+\chi, \beta}^{\dagger, \circ} \colon \mathcal{N}^{\dagger}_{G, \kappa^*} \to \mathcal{N}^{\dagger}_{H, \sigma_{\kappa}^{[j]}} = \mathcal{N}^{\dagger}_{H, \sigma_{\kappa}^{[j+\chi]}}
    \]
    by taking the $M^H_{\diamondsuit}(p^{\beta})$-invariants of the map $\sigma_{\kappa}^{[j+\chi], -1} \; \hatot \; \mathcal{N}^{\dagger}_G \; \hatot \; V_{\kappa}^* \to \mathcal{N}^{\dagger}_H$ defined in exactly the same way as in Definition \ref{AbstractDefOfThetadaggersan} (note that $\delta^{\dagger, s\opn{-an}}_{\kappa, j+\chi, \beta} \in V_{\kappa} \hatot C^{s\opn{-an}}(U_{G, \beta}, L)$ by Proposition \ref{SanDeltaDaggerIsAnEvectorProp}(3)). If $\chi$ is the trivial character, we will denote this morphism by $\vartheta^{\dagger, \circ}_{\kappa, j, \beta}$.
\end{definition}

As in the previous section, there exist relations between $\vartheta^{\dagger, s\opn{-an}}_{\kappa, j, \beta}$, $\vartheta^{\dagger, \circ}_{\kappa, j+\chi, \beta}$, and $\vartheta^{\dagger}_{\kappa, j, \beta}$, and also compatibility as the Tate affinoid algebra $(R, R^+)$ varies.

\begin{proposition}
    We have the following:
    \begin{enumerate}
        \item Let $(R, R^+) \to (R', (R')^+)$ be a morphism of complete Tate affinoid algebras over $(L, \mathcal{O}_L)$, and $s' \geq s \geq 1$ integers. Let $(\kappa, j) \in \mathcal{X}_{R, s}$ and let $(\kappa', j') \in \mathcal{X}_{R', s'}$ denote its image under the natural map $\mathcal{X}_{R, s} \to \mathcal{X}_{R', s'}$. Then we have a commutative diagram
        \[
\begin{tikzcd}
{\mathcal{N}^{\dagger, s'\opn{-an}}_{G, (\kappa')^*}} \arrow[d, "{\vartheta^{\dagger, s'\opn{-an}}_{\kappa', j', \beta}}"'] & {\mathcal{N}^{\dagger, s'\opn{-an}}_{G, \kappa^*}} \arrow[l] \arrow[r] \arrow[d, "{\vartheta^{\dagger, s'\opn{-an}}_{\kappa, j, \beta}}"'] & {\mathcal{N}^{\dagger, s\opn{-an}}_{G, \kappa^*}} \arrow[d, "{\vartheta^{\dagger, s\opn{-an}}_{\kappa, j, \beta}}"] \\
{\mathcal{N}^{\dagger, \opn{an}}_{H, \sigma_{\kappa'}^{[j']}}}                                                              & {\mathcal{N}^{\dagger, \opn{an}}_{H, \sigma_{\kappa}^{[j]}}} \arrow[l] \arrow[r, equals]                                                           & {\mathcal{N}^{\dagger, \opn{an}}_{H, \sigma_{\kappa}^{[j]}}}                                                       
\end{tikzcd}
        \]
        where the unlabelled arrows are the natural ones.
        \item Let $(\kappa, j) \in \mathcal{E}$ and let $\chi = (\chi_{\tau}) \colon \prod_{\tau \in \Psi} \mbb{Z}_p^{\times} \to L^{\times}$ be a finite-order character such that $\chi_{\tau}$ is trivial on $1 + p^{\beta}\mbb{Z}_p$ for all $\tau \in \Psi$. Let $s \geq 1$ be a sufficiently large integer such that $\chi$ is $s$-analytic. Then 
        \[
        \vartheta^{\dagger, s\opn{-an}}_{\kappa, j+\chi, \beta} = \vartheta^{\dagger, \circ}_{\kappa, j+\chi, \beta} 
        \]
        via the natural maps $\mathcal{N}^{\dagger, s\opn{-an}}_{G, \kappa^*} \to \mathcal{N}^{\dagger}_{G, \kappa^*}$ and $\mathcal{N}^{\dagger, \opn{an}}_{H, \sigma_{\kappa}^{[j]}} = \mathcal{N}^{\dagger}_{H, \sigma_{\kappa}^{[j]}}$. If $\chi$ is the trivial character, then
        \[
        \vartheta^{\dagger, \circ}_{\kappa, j, \beta} = \vartheta^{\dagger}_{\kappa, j, \beta} \circ 1_{U_{G, \beta}^{\circ}}
        \]
        where $1_{U_{G, \beta}^{\circ}} \colon \mathcal{N}^{\dagger}_{G, \kappa^*} \to \mathcal{N}^{\dagger}_{G, \kappa^*}$ denotes the action of the indicator function of $U_{G,\beta}^{\circ}$ (which preserves the weight).
    \end{enumerate}
\end{proposition}
\begin{proof}
    The first part just follows from the fact that the vectors $\delta^{\dagger, s\opn{-an}}_{\kappa, j, \beta}$ are compatible with changing the Tate affinoid algebra and the radius of analyticity (see Proposition \ref{SanDeltaDaggerIsAnEvectorProp}).  Indeed, we have the following commutative diagram:
    \[
    \begin{tikzcd}
{\sigma_{\kappa'}^{[j'], -1} \hatot \mathcal{N}^{\dagger}_G \hatot D^{s'\opn{-an}}_{(\kappa')^*}} \arrow[r]                & {V_{\kappa'}^{s'\opn{-an}} \hatot C^{s'\opn{-an}}(U_{G, \beta}, R') \hatot \mathcal{N}^{\dagger}_G \hatot D^{s'\opn{-an}}_{(\kappa')^*}} \arrow[r]               & V_{\kappa'}^{s'\opn{-an}} \hatot \mathcal{N}^{\dagger}_G \hatot D^{s'\opn{-an}}_{(\kappa')^*} \arrow[r]                & \mathcal{N}^{\dagger}_H                     \\
{\sigma_{\kappa}^{[j], -1} \hatot \mathcal{N}^{\dagger}_G \hatot D^{s'\opn{-an}}_{\kappa^*}} \arrow[u] \arrow[r]           & {V_{\kappa}^{s'\opn{-an}} \hatot C^{s'\opn{-an}}(U_{G, \beta}, R) \hatot \mathcal{N}^{\dagger}_G \hatot D^{s'\opn{-an}}_{\kappa^*}} \arrow[u] \arrow[r]         & V_{\kappa}^{s'\opn{-an}} \hatot \mathcal{N}^{\dagger}_G \hatot D^{s'\opn{-an}}_{\kappa^*} \arrow[u] \arrow[r]          & \mathcal{N}^{\dagger}_H \arrow[u]           \\
{\sigma_{\kappa}^{[j], -1} \hatot \mathcal{N}^{\dagger}_G \hatot D^{s'\opn{-an}}_{\kappa^*}} \arrow[r] \arrow[u, equals] \arrow[d] & {V_{\kappa}^{s\opn{-an}} \hatot C^{s\opn{-an}}(U_{G, \beta}, R) \hatot \mathcal{N}^{\dagger}_G \hatot D^{s'\opn{-an}}_{\kappa^*}} \arrow[r] \arrow[u] \arrow[d] & V_{\kappa}^{s\opn{-an}} \hatot \mathcal{N}^{\dagger}_G \hatot D^{s'\opn{-an}}_{\kappa^*} \arrow[r] \arrow[u] \arrow[d] & \mathcal{N}^{\dagger}_H \arrow[u, equals] \arrow[d] \\
{\sigma_{\kappa}^{[j], -1} \hatot \mathcal{N}^{\dagger}_G \hatot D^{s\opn{-an}}_{\kappa^*}} \arrow[r]                      & {V_{\kappa}^{s\opn{-an}} \hatot C^{s\opn{-an}}(U_{G, \beta}, R) \hatot \mathcal{N}^{\dagger}_G \hatot D^{s\opn{-an}}_{\kappa^*}} \arrow[r]                      & V_{\kappa}^{s\opn{-an}} \hatot \mathcal{N}^{\dagger}_G \hatot D^{s\opn{-an}}_{\kappa^*} \arrow[r]                      & \mathcal{N}^{\dagger}_H                    
\end{tikzcd}
    \]
    where the compositions of the horizontal maps in the top, second, and bottom row describe $\vartheta^{\dagger, s'\opn{-an}}_{\kappa', j', \beta}$, $\vartheta^{\dagger, s'\opn{-an}}_{\kappa, j, \beta}$, and $\vartheta^{\dagger, s\opn{-an}}_{\kappa, j, \beta}$ respectively. Here the leftmost horizontal map is the first bullet point in Definition \ref{AbstractDefOfThetadaggersan}, the middle horizontal map is the second bullet point, and the rightmost horizontal map is the composition of the third and fourth bullet points. The top, middle, and bottom left squares commute by the compatibility properties in Proposition \ref{SanDeltaDaggerIsAnEvectorProp}, and the remaining squares commute by the compatibility in $s$ of the action maps \eqref{AMla2} as well as the natural pairings between $s$-analytic functions and distributions.
    
    For the second part, this follows from the fact that $\delta^{\dagger, s\opn{-an}}_{\kappa, j+\chi, \beta} = 1_{U_{G, \beta}^{\circ}, \chi} \cdot \delta^{\dagger}_{\kappa, j, \beta}$.  Indeed, we have the following commutative diagram:

\[
\begin{tikzcd}
{\sigma_{\kappa}^{[j+\chi], -1} \hatot \mathcal{N}^{\dagger}_G \hatot D^{s\opn{-an}}_{\kappa^*}} \arrow[r]                       & {V_{\kappa}^{s\opn{-an}} \hatot C^{s\opn{-an}}(U_{G, \beta}, L) \hatot \mathcal{N}^{\dagger}_G \hatot D^{s\opn{-an}}_{\kappa^*}} \arrow[r]                       & V_{\kappa}^{s\opn{-an}} \hatot \mathcal{N}^{\dagger}_G \hatot D^{s\opn{-an}}_{\kappa^*} \arrow[r]          & \mathcal{N}^{\dagger}_H                     \\
{\sigma_{\kappa}^{[j+\chi], -1} \hatot \mathcal{N}^{\dagger}_G \hatot D^{s\opn{-an}}_{\kappa^*}} \arrow[u, equals] \arrow[r] \arrow[d] & {V_{\kappa} \hatot C^{s\opn{-an}}(U_{G, \beta}, L) \hatot \mathcal{N}^{\dagger}_G \hatot D^{s\opn{-an}}_{\kappa^*}} \arrow[u] \arrow[r] \arrow[d]             & V_{\kappa} \hatot \mathcal{N}^{\dagger}_G \hatot D^{s\opn{-an}}_{\kappa^*} \arrow[u] \arrow[r] \arrow[d] & \mathcal{N}^{\dagger}_H \arrow[u, equals] \arrow[d, equals] \\
{\sigma_{\kappa}^{[j+\chi], -1} \hatot \mathcal{N}^{\dagger}_G \hatot V^*_{\kappa}} \arrow[r]                                  & {V_{\kappa} \hatot C^{s\opn{-an}}(U_{G, \beta}, L) \hatot \mathcal{N}^{\dagger}_G \hatot V^*_{\kappa}} \arrow[r]                                              & V_{\kappa} \hatot \mathcal{N}^{\dagger}_G \hatot V^*_{\kappa} \arrow[r]                                  & \mathcal{N}^{\dagger}_H                     \\
{\sigma_{\kappa}^{[j], -1} \otimes \mathcal{N}^{\dagger}_G \otimes V^*_{\kappa}} \arrow[r] \arrow[u, equals, dashed]                             & {V_{\kappa} \otimes C^{\opn{pol}}(U_{G, \beta}, L) \otimes \mathcal{N}^{\dagger}_G \otimes V^*_{\kappa}} \arrow[r] \arrow[u, "{1_{U_{G, \beta}^{\circ}, \chi}}"] & V_{\kappa} \otimes \mathcal{N}^{\dagger}_G \otimes V^*_{\kappa} \arrow[u, equals] \arrow[r]                        & \mathcal{N}^{\dagger}_H \arrow[u, equals]          
\end{tikzcd}
\]

where the compositions of the horizontal maps in the top, third, and bottom rows describe $\vartheta^{\dagger, s\opn{-an}}_{\kappa, j+\chi, \beta}$, $\vartheta^{\dagger, \circ}_{\kappa, j+\chi, \beta}$, and $\vartheta^{\dagger}_{\kappa, j, \beta}$ respectively. Here the dotted equals sign in the bottom left means we only consider this when $\chi$ is trivial, and the labelled arrow is induced from the natural map $C^{\opn{pol}}(U_{G, \beta}, L) \to C^{s\opn{-an}}(U_{G, \beta}, L)$ given by multiplication by $1_{U_{G, \beta}^{\circ}, \chi}$.
\end{proof}

\begin{remark}
    One should view the action of $1_{U_{G, \beta}^{\circ}}$ as a kind of ``$p$-depletion'' which is necessary for $p$-adically interpolating the maps $\vartheta^{\dagger}_{\kappa, j, \beta}$. This will be made more precise in \S \ref{PropsOfEvMapsInterpolationSubSec}.
\end{remark}

%--------------------------------------------

\section{Continuous operators on Banach spaces} \label{ContinuousOpsOnBanachSpacesSec}

In this section we establish the abstract results on locally analytic actions needed to $p$-adically interpolate the Gauss--Manin connection (see \S \ref{OrdExpStrictNhoodGSec}--\ref{TheConstuctionForHSec}). This section is essentially an adaptation of \cite[\S 4]{DiffOps} to the setting of this article.

\subsection{Function spaces}

Let $L/\mbb{Q}_p$ be a finite extension with ring of integers $\mathcal{O}_L$, equipped with the $p$-adic norm $| \cdot |$ such that $|p| = p^{-1}$. We let $C_{\opn{cont}}(\mbb{Z}_p, \mathcal{O}_L)$ (resp. $C_{\opn{cont}}(\mbb{Z}_p, L)$) denote the $\mathcal{O}_L$-algebra (resp. $L$-algebra) of continuous functions $\mbb{Z}_p \to \mathcal{O}_L$ (resp. $\mbb{Z}_p \to L$). For an integer $k \geq 0$, consider the continuous function
\[
\bincoeff{x}{k} = \frac{ x(x-1)\cdots (x-k+1)}{k!} \in C_{\opn{cont}}(\mbb{Z}_p, \mathcal{O}_L) .
\]
Then it is well-known that the collection $\{ \bincoeff{x}{k} \}_{k \geq 0}$ forms an orthonormal basis for $C_{\opn{cont}}(\mbb{Z}_p, L)$ (see \cite[Corollaire I.2.4]{ColmezFonctions} for example).

For an integer $h \geq 0$, let $C^{h\opn{-an}}(\mbb{Z}_p, L) \subset C_{\opn{cont}}(\mbb{Z}_p, L)$ denote the subspace of functions which are analytic on discs of radius $p^{-h}$, and let $C^{\opn{la}}(\mbb{Z}_p, L) = \varinjlim_{h \geq 0} C^{h\opn{-an}}(\mbb{Z}_p, L)$ denote the space of locally analytic functions equipped with the direct limit topology. In terms of the orthonormal basis above, $C^{\opn{la}}(\mbb{Z}_p, L)$ is the subspace of functions 
\begin{equation} \label{MahlerExpansionEqn}
f(x) = \sum_{k \geq 0} a_k \bincoeff{x}{k} \quad \in \quad C_{\opn{cont}}(\mbb{Z}_p, L)
\end{equation}
such that $p^{k \varepsilon} |a_k| \to 0$ as $k \to +\infty$ for some $\varepsilon > 0$.

\begin{definition}
    For $\varepsilon > 0$, let $C_{\varepsilon}(\mbb{Z}_p, L) \subset C_{\opn{cont}}(\mbb{Z}_p, L)$ denote the subspace of functions as in (\ref{MahlerExpansionEqn}) satisfying the condition that $p^{k \varepsilon} |a_k| \to 0$ as $k \to +\infty$. This is an $L$-Banach algebra with norm given by
    \[
    |\! | f | \! | = \opn{sup}\{ p^{k \varepsilon}|a_k| : k \geq 0 \} .
    \]
\end{definition}

We introduce the following terminology:

\begin{definition} \label{Def:ContLAEpsilonAction}
    Let $W$ be a topological $L$-vector space and $T \in \opn{End}_L(W)$ a continuous operator. We say that $T$ extends to a continuous (resp. locally analytic, resp. $\varepsilon$-analytic) action if there exists a continuous $L$-bilinear map
    \[
    C \times W \to W, \quad C = C_{\opn{cont}}(\mbb{Z}_p, L) \text{ (resp. } C = C^{\opn{la}}(\mbb{Z}_p, L)\text{, resp. } C = C_{\varepsilon}(\mbb{Z}_p, L) \text{ )}
    \]
such that
\begin{itemize}
    \item the induced map $C \to \opn{End}_L(W)$ is a morphism of $L$-algebras
    \item the structural map $\mbb{Z}_p \hookrightarrow L$ acts as $T$.
\end{itemize}
    Here $C \times W$ is equipped with the product topology.
\end{definition}

\subsection{A pertubation lemma}

We will need the following lemma which describes the convergence properties of two operators which are congruent modulo a power of $p$.

\begin{lemma} \label{PertubationLemma}
    Let $V$ be an $L$-Banach space and $T_1, T_2 \in \opn{End}_L(V)$ two continuous operators. Suppose that $T_1$ extends to a locally analytic action and $|\!| T_2 | \!| \leq 1$ (i.e. $T_2$ preserves the unit ball in $V$). Then for any $\varepsilon > 0$, there exists an integer $n(\varepsilon) \geq 1$ (depending on $T_1$ only) such that for all $n \geq n(\varepsilon)$, the operator $T_1 + p^n T_2$ extends to an $\varepsilon$-analytic action on $V$ (as in Definition \ref{Def:ContLAEpsilonAction}).
\end{lemma}
\begin{proof}
    This is {\cite[Proposition 4.2.4]{DiffOps}}.
\end{proof}

\subsection{Nilpotent operators} \label{NilpotentOperatorsSubSec}

Let $S^+$ be an admissible $\mathcal{O}_L$-algebra, and set $S = S^+[1/p]$. Let $|\!|\cdot |\!|$ denote the Banach norm on $S$ such that $S^+$ is the unit ball. Consider the following two-variable Tate algebra $V^+ = S^+\langle X, Y \rangle$ and set $V = V^+[1/p]$. Equip $V$ with the Banach norm such that $V^+$ is the unit ball, i.e. equip $V$ with the Banach norm
\[
| \! | f | \! | = \opn{sup}\{ |\!| s_{a, b} |\!| : a, b \geq 0 \}
\]
where $f(X, Y) = \sum_{a, b \geq 0} s_{a, b} X^a Y^b$.

Suppose that we have an $\mathcal{O}_L$-linear derivation $D \colon S^+ \to S^+$. Fix an element $\lambda \in \mathcal{O}_L$ and consider the $\mathcal{O}_L$-linear derivation $T_D = T_{D, \lambda} \colon V^+ \to V^+$ uniquely determined by the following properties:
\begin{itemize}
    \item $T_D(s) = D(s)$ for any $s \in S^+$
    \item $T_D(X) = \lambda Y$
    \item $T_D(Y) = 0$ .
\end{itemize}
Concretely, the derivation $T_D$ is given by:
\[
T_D\left( \sum_{a, b \geq 0} s_{a, b} X^a Y^b \right) \defeq \sum_{a, b \geq 0} D(s_{a, b}) X^a Y^b + \sum_{\substack{a \geq 1 \\ b \geq 0}} (\lambda \cdot a \cdot s_{a, b}) X^{a-1} Y^{b+1}  
\]
which is well-defined because $V^+ = S^+\langle X, Y \rangle$ is a Tate algebra. It is uniquely determined by the above three properties because $T_D$ is uniquely determined by its values on $s_{a, b}X^aY^b$ (by $\mathcal{O}_L$-linearity) and hence its values on $s_{a, b}$, $X$, and $Y$ (by the Leibniz rule).

\begin{proposition} \label{LANilpToLAIterationProp}
    Suppose that $D$ extends to a locally analytic action on $S$. Then $T_D$ extends to a locally analytic action on $V$.
\end{proposition}
\begin{proof}
    We first introduce some polynomial functions that will be useful in the proof of this proposition. For any integers $k \geq 1$ and $0 \leq r \leq k$, let $\Sigma_{k, r}$ denote the set of subsets of $\{0, \dots, k-1 \}$ of size $r$. For any non-empty $I \in \Sigma_{k, r}$, let $k_1, \dots, k_{\ell}$ be the lengths of the largest blocks of consecutive integers in $I$, so that $\sum_{i=1}^\ell k_i = r$. In other words, $I$ can be written as 
    \[
    I = \bigcup_{1 \leq j \leq \ell} \{ i_j, i_j +1, \dots, i_j + k_j - 1 \} = \bigcup_{1 \leq j \leq \ell} I_j
    \]
    with no adjacent intervals (i.e. $i_j > i_{j-1} + k_{j-1}$ for all $2 \leq j \leq \ell$). For any $I \in \Sigma_{k, r}$, consider the following polynomial function $f_I \colon \mbb{Z}_p \to \mathcal{O}_L$ given by the formula
    \[
    f_I(x) = \prod_{1 \leq j \leq \ell} \frac{1}{k_j!} \prod_{i \in I_j} (x-i), \quad \quad x \in \mbb{Z}_p
    \]
    with the convention that $f_{\varnothing} = 1$. We also set $f_k = f_{\{0, \dots, k-1\}}$. Then, by a simple induction argument on $k$ (see \cite[Lemma 4.3.2]{DiffOps}), we have the following formula:
    \begin{equation} \label{ExplicitFormulaForfk}
    f_k(T_D)(s X^a Y^b) = \sum_{r=0}^{\opn{min}(k, a)} \sum_{I \in \Sigma_{k, k-r}} \bincoeff{k-r}{k_1, \dots, k_{\ell}}^{-1} \bincoeff{k}{r}^{-1} \bincoeff{a}{r} f_I(D)(s) \lambda^r X^{a-r}Y^{b+r}  
    \end{equation}
    for $s \in S$ and $a, b \geq 0$.

    Let $\varepsilon > 0$. To show that $T_D$ extends to a locally analytic action, it suffices to show that there exists a constant $C_{\varepsilon} \in \mbb{R}_{>0}$ such that for all $k \geq 0$, $s \in S$ and $a, b \geq 0$, we have 
    \[
    p^{-k \varepsilon} | \! | f_k(T_D)(sX^aY^b) | \! | \leq C_{\varepsilon} | \! | s X^a Y^b | \! | = C_{\varepsilon} |\!| s |\!|.
    \]
    Indeed, this implies that the operator norm satisfies $| \! | f_k(T_D) |\! | \leq p^{k \varepsilon} C_{\varepsilon}$, and hence any expression $\sum_{k \geq 0} a_k f_k(T_D)$, with $p^{k \varepsilon}|a_k| \to 0$ as $k \to +\infty$, converges (under the operator norm) to a well-defined operator on $V$. 

    By the formula in (\ref{ExplicitFormulaForfk}), it suffices to show that for each $0 \leq r \leq \opn{min}(k, a)$ and $I \in \Sigma_{k, k-r}$, we have
    \[
    p^{-k\varepsilon} \left| \! \left| \bincoeff{k-r}{k_1, \dots, k_{\ell}}^{-1} \bincoeff{k}{r}^{-1} \bincoeff{a}{r} f_I(D)(s) \lambda^r \right| \! \right| \leq C_{\varepsilon} |\!|s|\!| .
    \]
    Since $D$ extends to a locally analytic action on $S$ and $f_I(-)$ is valued in $\mathcal{O}_L$, there exists a constant $C_{\varepsilon/2}'$ such that $p^{-k\varepsilon/2}|\!|f_I(D)(s)|\!| \leq C_{\varepsilon/2}' |\!|s|\!|$ for all $s \in S$ and $k \geq 0$. Therefore, we have 
    \begin{align*}
    p^{-k \varepsilon} \left| \! \left| \bincoeff{k-r}{k_1, \dots, k_{\ell}}^{-1} \bincoeff{k}{r}^{-1} \bincoeff{a}{r} f_I(D)(s) \lambda^r \right| \! \right| &\leq p^{-k\varepsilon/2}\left| \! \left| \bincoeff{k-r}{k_1, \dots, k_{\ell}}^{-1} \bincoeff{k}{r}^{-1} \right| \! \right| p^{-k\varepsilon/2} |\!|f_I(D)(s) |\!| \\
     &\leq p^{-k\varepsilon/2 + \opn{log}_p(k-r) + \opn{log}_p(k)} C'_{\varepsilon/2} |\!|s|\!| .
    \end{align*}
    Since $\varepsilon > 0$, we have $-k\varepsilon/2 + \opn{log}_p(k-r) + \opn{log}_p(k) \to -\infty$ as $k \to +\infty$, hence we can find a constant $C_{\varepsilon}$, independent of $k, s, a, b$, such that
    \[
    p^{-k\varepsilon/2 + \opn{log}_p(k-r) + \opn{log}_p(k)} C'_{\varepsilon/2} \leq C_{\varepsilon}
    \]
    as required.
\end{proof}

\subsection{Overconvergence}

To conclude this section, we discuss a result which will allow us to extend $\varepsilon$-actions to ``overconvergent neighbourhoods''  (see Remark \ref{Rem:OCintuition}). Suppose that we have a sequence of $L$-Banach spaces
\[
V_0 \to V_1 \to \cdots \to V_{\infty}
\]
and denote the Banach norm on $V_r$ by $|\!| \cdot |\!|_r$. By abuse of notation, if $v \in V_r$, then for any $r \leq s \leq \infty$ we write $|\!|v|\!|_s$ for the Banach norm of the image of $v$ under the map $V_r \to V_s$. Finally, suppose that we have continuous operators $T = T_r \colon V_r \to V_r$ which are all compatible with each other under the maps above. We assume that $|\!| v |\!|_s \leq |\!| v |\!|_r$ for all $v \in V_r$ and $r \leq s \leq \infty$.

\begin{proposition} \label{CloseToIsomGeneralProp}
    Assume that the following property holds: for any $0 < \delta < 1$ and $r \in \mbb{N}$, there exists $s = s(\delta) \geq r$ such that, for all $c \in \mbb{Q}$, $h \in \mbb{N}$ and $v \in V_r$ we have
\begin{equation} \label{Equationproperty}
 |\!|v|\!|_r \leq p^c \text{ and } |\!|v|\!|_{\infty} \leq p^{c - h} \implies |\!|v|\!|_s \leq p^{c -\delta h}.
\end{equation}
Assume that, for some $\varepsilon > 0$, the operator $T$ extends to an $\varepsilon$-analytic action on $V_\infty$  (as in Definition \ref{Def:ContLAEpsilonAction}). Then, for any $\gamma > \varepsilon$, there exists $s \in \mbb{N}$ (depending only on $\varepsilon$, $\gamma$ and the operator norm $|\!|T|\!|_r$) such that, for any $v\in V_r$,
\[ 
p^{-k \gamma} |\!|f_k(\nabla)(v)|\!|_s \to 0 \text{ as } k \to +\infty.
\]
\end{proposition}
\begin{proof}
    This is {\cite[Proposition 4.4.1, Remark 4.4.4]{DiffOps}}.
\end{proof}

\begin{remark} \label{Rem:OCintuition}
    Let us briefly describe the rationale behind Proposition \ref{CloseToIsomGeneralProp}. Suppose that $\mathcal{X}$ is a qcqs adic space over $\opn{Spa}(L, \mathcal{O}_L)$, and suppose that $h \in \opn{H}^0(\mathcal{X}, \mathcal{L})$ is a non-zero section of an invertible sheaf $\mathcal{L}$ of $\mathcal{O}_{\mathcal{X}}^+$-modules. Let $\mathcal{U}_{\infty} \subset \mathcal{X}$ denote the open locus where $|h| = 1$. Then we can define a system of ``overconvergent neighbourhoods'' of $\mathcal{U}_{\infty}$:
    \[
    \mathcal{U}_{\infty} \subset \cdots \subset \mathcal{U}_r \subset \cdots \subset \mathcal{U}_{1} \subset \mathcal{U}_0 \subset \mathcal{X}
    \]
    where $\mathcal{U}_r \subset \mathcal{X}$ is the open locus where $|h|^{p^{r+1}} \geq |p|$. One has the property that $\mathcal{U}_r$ contains the closure of $\mathcal{U}_{s}$ (in $\mathcal{X}$) for $r < s \leq \infty$. If we let $V_r = \opn{H}^0(\mathcal{U}_r, \mathcal{O}_{\mathcal{U}_r})$, then we obtain a sequence of $L$-Banach spaces $V_0 \to V_1 \to \cdots \to V_{\infty}$ which satisfies the assumptions of Proposition \ref{CloseToIsomGeneralProp}. Hence if we have compatible continuous operators $T \colon V_r \to V_r$ such that $T$ extends to an $\varepsilon$-analytic action on $V_{\infty}$, then the conclusion of Proposition \ref{CloseToIsomGeneralProp} implies this action ``overconverges'' to a $\gamma$-analytic action on $\varinjlim_{r < \infty} V_r$ for any $\gamma > \varepsilon$. In practice (see \S \ref{OrdExpStrictNhoodGSec}--\ref{TheConstuctionForHSec}), the space $\mathcal{X}$ will be a Shimura variety, $\{\mathcal{U}_r\}_{r < \infty}$ will be overconvergent neighbourhoods of a component of the ordinary locus in $\mathcal{X}$, and we will take $T$ to be the Gauss--Manin connection.
\end{remark}

%---------------------------------------------

\section{Cohomology and correspondences} \label{CohomologyAndCorrespondencesSection}

In \S \ref{HOpsAndHCTChapter}, we will describe the action of certain Hecke correspondences on the cohomology of unitary Shimura varieties with partial compact support. In particular, we will summarise the main results from higher Coleman theory \cite{BoxerPilloni} that we will need for the construction of the $p$-adic $L$-function. Unfortunately, the support conditions we need to consider for these Hecke correspondence are slightly different from those in \emph{op.cit.}, and it will therefore be useful to have a slight generalisation of this theory. 

Consider the following correspondence of smooth adic spaces (over a finite extension of $\mbb{Q}_p$)
\[
\begin{tikzcd}
            & \mathcal{C} \arrow[ld, "p_1"'] \arrow[rd, "p_2"] &             \\
\mathcal{X} &                                                  & \mathcal{X}
\end{tikzcd}
\]
where $p_1$ and $p_2$ are finite flat. For any subset $\mathcal{Y} \subset \mathcal{X}$, we set $T(\mathcal{Y}) = p_2 p_1^{-1}(\mathcal{Y})$ and $T^t(\mathcal{Y}) = p_1 p_2^{-1}(\mathcal{Y})$. By our assumptions, $T$ and $T^t$ take (quasi-compact) open subsets to (quasi-compact) open subsets, and closed subsets to closed subsets. We now introduce a version of support conditions suitable for our purposes.

\begin{definition} \label{SystemOfSupportConditionsDef}
    A system of support conditions for the above correspondence is a collection of open subsets $\{ \mathcal{U}_k \}_{k \in \mbb{N}}$ and closed subsets $\{ \mathcal{Z}_m \}_{m \in \mbb{N}}$ of $\mathcal{X}$ such that:
    \begin{enumerate}
        \item $\{ \mathcal{U}_k \}_{k \in \mbb{N}}$ and $\{ \mathcal{Z}_m \}_{m \in \mbb{N}}$ are nested (i.e. $\mathcal{U}_{k+1} \subset \mathcal{U}_k$ and $\mathcal{Z}_{m+1} \subset \mathcal{Z}_m$ for all $k, m \in \mbb{N}$).
        \item For any $k,m \in \mbb{N}$, both $\mathcal{U}_k$ and the complement of $\mathcal{Z}_m$ are finite unions of quasi-Stein\footnote{Recall that an adic space $\mathcal{Y}$ over $\opn{Spa}(\mbb{Q}_p, \mbb{Z}_p)$ is quasi-Stein if it is a countable increasing union $\mathcal{Y} = \bigcup_{i \geq 0} \mathcal{Y}_i$ of finite-type affinoid adic spaces $\mathcal{Y}_i \to \opn{Spa}(\mbb{Q}_p, \mbb{Z}_p)$ such that the restriction maps $\mathcal{O}_{\mathcal{Y}_{i+1}} \to \mathcal{O}_{\mathcal{Y}_i}$ have dense image (see \cite[Definition 2.5.14]{BoxerPilloni}). This condition implies that the cohomology complexes $R\Gamma_{\mathcal{U}_k \cap \mathcal{Z}_m}(\mathcal{U}_k, \mathscr{F})$ in (\ref{UZbulletDefinition}) can be represented by \v{C}ech complexes formed from projective systems of $\mbb{Q}_p$-Banach spaces (see \cite[Lemma 2.5.21]{BoxerPilloni}).} open subspaces.
        \item For any integer $k \in \mbb{N}$, there exist integers $k', m \in \mbb{N}$ such that
        \[
        T(\mathcal{U}_{k'}) \cap \mathcal{Z}_m \subset \mathcal{U}_k .
        \]
        \item For any $m \in \mbb{N}$, we have $T^t(\mathcal{Z}_m) \subset \mathcal{Z}_m$.
    \end{enumerate}
    For any locally projective Banach sheaf $\mathscr{F}$ on $\mathcal{U}_1$, we set
    \begin{equation} \label{UZbulletDefinition}
    R\Gamma(\mathcal{U}_{\bullet}, \mathcal{Z}_{\bullet}; \mathscr{F}) \defeq \varinjlim_k \varprojlim_m R\Gamma_{\mathcal{U}_k \cap \mathcal{Z}_m}(\mathcal{U}_k, \mathscr{F})
    \end{equation}
    where the direct limit is with respect to the natural restriction maps and the (derived) inverse limit is with respect to the natural corestriction maps.
\end{definition}

Suppose that we have a system of support conditions and a locally projective Banach sheaf $\mathscr{F}$ as in Definition \ref{SystemOfSupportConditionsDef}. Fix integers $a, b \in \mbb{N}$ such that $T(\mathcal{U}_a) \cap \mathcal{Z}_b \subset \mathcal{U}_1$ and suppose that we have a continuous morphism $\phi \colon p_2^* \mathscr{F} \to p_1^* \mathscr{F}$ defined on a open neighbourhood $p_1^{-1}(\mathcal{U}_a) \cap p_2^{-1}(\mathcal{Z}_b) \subset \mathcal{V} \subset p_1^{-1}(\mathcal{U}_a) \cap p_2^{-1}(\mathcal{U}_1)$. Then, for any integer $k \in \mbb{N}$ and integers $k' \geq a$, $m \geq b$ such that $T(\mathcal{U}_{k'}) \cap \mathcal{Z}_m \subset \mathcal{U}_k$, we obtain an operator $T$ defined by the following composition:
\begin{align*} 
R \Gamma_{\mathcal{U}_k \cap \mathcal{Z}_m}(\mathcal{U}_k, \mathscr{F}) &\xrightarrow{p_2^*} R\Gamma_{p_2^{-1}(\mathcal{U}_k) \cap p_2^{-1}(\mathcal{Z}_m)}(p_2^{-1}(\mathcal{U}_k), p_2^*\mathscr{F}) \\
 &\xrightarrow{\opn{res}} R\Gamma_{p_1^{-1}(\mathcal{U}_{k'}) \cap p_2^{-1}(\mathcal{Z}_m)}(p_2^{-1}(\mathcal{U}_k) \cap p_1^{-1}(\mathcal{U}_{k'}) \cap \mathcal{V}, p_2^*\mathscr{F}) \\
 &\xrightarrow{\phi} R\Gamma_{p_1^{-1}(\mathcal{U}_{k'}) \cap p_2^{-1}(\mathcal{Z}_m)}(p_2^{-1}(\mathcal{U}_k) \cap p_1^{-1}(\mathcal{U}_{k'}) \cap \mathcal{V}, p_1^*\mathscr{F}) \\
 &= R\Gamma_{p_1^{-1}(\mathcal{U}_{k'}) \cap p_2^{-1}(\mathcal{Z}_m)}(p_1^{-1}(\mathcal{U}_{k'}), p_1^*\mathscr{F}) \\
 &\xrightarrow{\opn{Tr}_{p_1}} R\Gamma_{\mathcal{U}_{k'} \cap T^t(\mathcal{Z}_m)}(\mathcal{U}_{k'}, \mathscr{F}) \\
 &\xrightarrow{\opn{cores}} R\Gamma_{\mathcal{U}_{k'} \cap \mathcal{Z}_m}(\mathcal{U}_{k'}, \mathscr{F})
\end{align*}
where: 
\begin{itemize}
    \item ``$\opn{res}$'' denotes the restriction map from $p_2^{-1}(\mathcal{U}_k)$ with support in $p_2^{-1}(\mathcal{U}_k) \cap p_2^{-1}(\mathcal{Z}_m)$ to $p_2^{-1}(\mathcal{U}_k) \cap p_1^{-1}(\mathcal{U}_{k'}) \cap \mathcal{V}$ with support in 
    \begin{equation} \label{Eqn:SupportEquality}
    p_2^{-1}(\mathcal{U}_k) \cap p_2^{-1}(\mathcal{Z}_m) \cap p_2^{-1}(\mathcal{U}_k) \cap p_1^{-1}(\mathcal{U}_{k'}) \cap \mathcal{V} = p_1^{-1}(\mathcal{U}_{k'}) \cap p_2^{-1}(\mathcal{Z}_m),
    \end{equation}
    where this equality follows from $T(\mathcal{U}_{k'}) \cap \mathcal{Z}_m \subset \mathcal{U}_k$ and $p_1^{-1}(\mathcal{U}_{k'}) \cap p_2^{-1}(\mathcal{Z}_m) \subset \mathcal{V}$ (because $k' \geq a$ and $m \geq b$);
    \item the equality is the excision isomorphism using the fact that $p_1^{-1}(\mathcal{U}_{k'}) \cap p_2^{-1}(\mathcal{Z}_m)$ is closed in both $p_2^{-1}(\mathcal{U}_k) \cap p_1^{-1}(\mathcal{U}_{k'}) \cap \mathcal{V}$ and $p_1^{-1}(\mathcal{U}_{k'})$ (by \eqref{Eqn:SupportEquality});
    \item the trace map $\opn{Tr}_{p_1}$ in the one constructed in \cite[Lemma 2.1.2]{BoxerPilloni};
    \item ``$\opn{cores}$'' denotes corestriction from cohomology with support in $\mathcal{U}_{k'} \cap T^t(\mathcal{Z}_m)$ to cohomology with support in $\mathcal{U}_{k'} \cap \mathcal{Z}_m$ (which makes sense by Definition \ref{SystemOfSupportConditionsDef}(4)).
\end{itemize}
We have the following compatibility with changing $(k, k', m, a, b)$:

\begin{lemma}
    For any $(k, k', m, a, b)$ as above, we have commutative diagrams:
\[
\begin{tikzcd}
{R \Gamma_{\mathcal{U}_k \cap \mathcal{Z}_m}(\mathcal{U}_k, \mathscr{F})} \arrow[r, "T"]                              & {R \Gamma_{\mathcal{U}_{k'} \cap \mathcal{Z}_m}(\mathcal{U}_{k'}, \mathscr{F})}                              & {R \Gamma_{\mathcal{U}_k \cap \mathcal{Z}_m}(\mathcal{U}_k, \mathscr{F})} \arrow[d, "\opn{res}"'] \arrow[r, "T"] & {R \Gamma_{\mathcal{U}_{k'} \cap \mathcal{Z}_m}(\mathcal{U}_{k'}, \mathscr{F})} \arrow[d, equals] \\
{R \Gamma_{\mathcal{U}_k \cap \mathcal{Z}_{m+1}}(\mathcal{U}_k, \mathscr{F})} \arrow[u, "\opn{cores}"] \arrow[r, "T"] & {R \Gamma_{\mathcal{U}_{k'} \cap \mathcal{Z}_{m+1}}(\mathcal{U}_{k'}, \mathscr{F})} \arrow[u, "\opn{cores}"] & {R \Gamma_{\mathcal{U}_{k+1} \cap \mathcal{Z}_m}(\mathcal{U}_{k+1}, \mathscr{F})} \arrow[r, "T"]                 & {R \Gamma_{\mathcal{U}_{k'} \cap \mathcal{Z}_m}(\mathcal{U}_{k'}, \mathscr{F})}                     
\end{tikzcd}
\]
where, for the right-hand diagram, we assume that $T(\mathcal{U}_{k'}) \cap \mathcal{Z}_m \subset \mathcal{U}_{k+1}$. As a consequence, we obtain a well-defined operator
\[
T \colon R\Gamma(\mathcal{U}_{\bullet}, \mathcal{Z}_{\bullet}; \mathscr{F}) \to R\Gamma(\mathcal{U}_{\bullet}, \mathcal{Z}_{\bullet}; \mathscr{F}) .
\]
This operator is independent of the choice of $a, b$ and $\mathcal{V}$.
\end{lemma}
\begin{proof}
    This follows from the various compatibilities of each of the maps in the definition of $T$, which we now explain.
    \begin{itemize}
        \item The map $p_2^*$ is compatible with changing $k$ because pullback commutes with the restriction maps with respect to the embeddings $\mathcal{U}_{k+1} \subset \mathcal{U}_k$ and $p_2^{-1}(\mathcal{U}_{k+1}) \subset p_2^{-1}(\mathcal{U}_k)$. Furthermore, $p_2^*$ commutes with the corestriction maps with respect to $\mathcal{U}_k \cap \mathcal{Z}_{m+1} \subset \mathcal{U}_k \cap \mathcal{Z}_{m}$ and $p_2^{-1}(\mathcal{U}_k) \cap p_2^{-1}(\mathcal{Z}_{m+1}) \subset p_2^{-1}(\mathcal{U}_k) \cap p_2^{-1}(\mathcal{Z}_{m})$ because the pullback of $(\mathcal{U}_k \cap \mathcal{Z}_m) - (\mathcal{U}_k \cap \mathcal{Z}_{m+1})$ under $p_2$ is just $(p_2^{-1}(\mathcal{U}_k) \cap p_2^{-1}(\mathcal{Z}_m)) - (p_2^{-1}(\mathcal{U}_k) \cap p_2^{-1}(\mathcal{Z}_{m+1}))$.
        \item The restriction map clearly commutes with restriction along $p_2^{-1}(\mathcal{U}_{k+1}) \subset p_2^{-1}(\mathcal{U}_{k})$ and $p_2^{-1}(\mathcal{U}_{k+1}) \cap p_1^{-1}(\mathcal{U}_{k'}) \cap \mathcal{V} \subset p_2^{-1}(\mathcal{U}_{k}) \cap p_1^{-1}(\mathcal{U}_{k'}) \cap \mathcal{V}$, and commutes with corestriction induced from $\mathcal{Z}_{m+1} \subset \mathcal{Z}_m$ for similar reasons as in the bullet point above.
        \item The compatibility for $\phi$ is clear, and the compatibility for the excision map is for the same reasons as in the preceding bullet point (the excision map is induced from a restriction map).
        \item The map $\opn{Tr}_{p_1}$ is compatible with the corestriction maps with respect to $p_1^{-1}(\mathcal{U}_{k'}) \cap p_2^{-1}(\mathcal{Z}_{m+1}) \subset p_1^{-1}(\mathcal{U}_{k'}) \cap p_2^{-1}(\mathcal{Z}_{m})$ and $\mathcal{U}_{k'} \cap T^t(\mathcal{Z}_{m+1}) \subset \mathcal{U}_{k'} \cap T^t(\mathcal{Z}_{m})$ because pushforwards commute with corestriction maps.
        \item The compatibility of the corestriction map with the corestriction maps induced from $\mathcal{Z}_{m+1} \subset \mathcal{Z}_m$ follows from the same reasons as in the preceding bullet point.
    \end{itemize}
    Putting this all together, we see that the diagrams in the statement of the lemma are commutative. The independence in $a, b, \mathcal{V}$ is immediate.
\end{proof}

We now compare two choices of support conditions.

\begin{definition} \label{IntertwinedSupportConditionsDef}
    We say that two systems of support conditions $(\mathcal{U}_{\bullet}, \mathcal{Z}_{\bullet})$ and $(\mathcal{U}_{\bullet}', \mathcal{Z}_{\bullet}')$ are intertwined if:
    \begin{itemize}
        \item For any pair of integers $k_1, m_1 \in \mbb{N}$, there exist a pair of integers $k_2, m_2 \in \mbb{N}$ such that
        \[
        \mathcal{U}'_{k_2} \cap \mathcal{Z}'_{m_2} \subset \mathcal{U}_{k_1} \cap \mathcal{Z}_{m_1} .
        \]
        \item For any pair of integers $k_1, m_1 \in \mbb{N}$, there exist a pair of integers $k_2, m_2 \in \mbb{N}$ such that
        \[
        \mathcal{U}_{k_2} \cap \mathcal{Z}_{m_2} \subset \mathcal{U}'_{k_1} \cap \mathcal{Z}'_{m_1} .
        \]
    \end{itemize}
    In particular, we have $\bigcap_{k, m} \left(\mathcal{U}_k \cap \mathcal{Z}_m \right) = \bigcap_{k, m} \left(\mathcal{U}'_k \cap \mathcal{Z}'_m \right)$.
\end{definition}

\begin{lemma}
    Suppose that $(\mathcal{U}_{\bullet}, \mathcal{Z}_{\bullet})$ and $(\mathcal{U}_{\bullet}', \mathcal{Z}_{\bullet}')$ are two systems of support conditions which are intertwined, and suppose that $\mathscr{F}$ is a locally projective Banach sheaf on $\mathcal{U}_1 \cup \mathcal{U}_1'$. Then there is a natural quasi-isomorphism
    \begin{equation} \label{NaturalQIsoIntertwinedEqn}
    R\Gamma( \mathcal{U}_{\bullet}, \mathcal{Z}_{\bullet}; \mathscr{F} ) \xrightarrow{\sim} R\Gamma( \mathcal{U}'_{\bullet}, \mathcal{Z}'_{\bullet}; \mathscr{F} ) .
    \end{equation}
    functorial in $\mathscr{F}$. If $\mathcal{V}$ and $\mathcal{V}'$ are open neighbourhoods as above (for the systems $(\mathcal{U}_{\bullet}, \mathcal{Z}_{\bullet})$ and $(\mathcal{U}_{\bullet}', \mathcal{Z}_{\bullet}')$ respectively) and $\phi \colon p_2^* \mathscr{F} \to p_1^*\mathscr{F}$ is continuous morphism defined on $\mathcal{V} \cup \mathcal{V}'$, then (\ref{NaturalQIsoIntertwinedEqn}) is $T$-equivariant.
\end{lemma}
\begin{proof}
    Let $\mathcal{U}''_k = \mathcal{U}_k \cap \mathcal{U}'_k$ and $\mathcal{Z}''_m = \mathcal{Z}_m \cap \mathcal{Z}_m'$, then the two systems  $(\mathcal{U}_{\bullet}, \mathcal{Z}_{\bullet})$ and $(\mathcal{U}_{\bullet}'', \mathcal{Z}_{\bullet}'')$ are intertwined and we are reduced to proving the statement for these two systems (with the morphism $\phi$ restricted to $\mathcal{V}$). Furthermore, we can reduce this two further simpler cases: 
    \begin{enumerate}
        \item The statement for the systems  $(\mathcal{U}_{\bullet}, \mathcal{Z}_{\bullet})$ and $(\mathcal{U}_{\bullet}'', \mathcal{Z}_{\bullet})$
        \item The statement for the systems  $(\mathcal{U}_{\bullet}'', \mathcal{Z}_{\bullet})$ and $(\mathcal{U}_{\bullet}'', \mathcal{Z}_{\bullet}'')$.
    \end{enumerate}
    Suppose we are in case (1). Then one can show that the restriction maps $R\Gamma_{\mathcal{U}_k \cap \mathcal{Z}_m}\left( \mathcal{U}_k, \mathscr{F} \right) \to R\Gamma_{\mathcal{U}_k'' \cap \mathcal{Z}_m}\left( \mathcal{U}_k'', \mathscr{F} \right)$ are compatible as one varies $k, m$ and induce a $T$-equivariant map
    \begin{equation} \label{ResForIntertwined}
    R\Gamma(\mathcal{U}_{\bullet}, \mathcal{Z}_{\bullet}; \mathscr{F}) \xrightarrow{\opn{res}} R\Gamma(\mathcal{U}''_{\bullet}, \mathcal{Z}_{\bullet}; \mathscr{F}) .
    \end{equation}
    We claim that this is a quasi-isomorphism. Let $k_1, m \in \mbb{N}$ be integers, and $k_2 \geq k_1$ such that
    \[
    \mathcal{U}_{k_2} \cap \mathcal{Z}_{m} \subset \mathcal{U}''_{k_1} \cap \mathcal{Z}_m \subset \mathcal{U}_{k_1} \cap \mathcal{Z}_m.
    \]
    Then we have a commutative diagram:
    \[
\begin{tikzcd}
{R\Gamma_{\mathcal{U}_{k_1}\cap\mathcal{Z}_m}(\mathcal{U}_{k_1}, \mathscr{F})} \arrow[d] \arrow[r] & {R\Gamma_{\mathcal{U}_{k_2}\cap\mathcal{Z}_m}(\mathcal{U}_{k_2}, \mathscr{F})} \arrow[d, "\sim"] \arrow[rd]       &                                                                                    \\
{R\Gamma_{\mathcal{U}''_{k_1}\cap\mathcal{Z}_m}(\mathcal{U}''_{k_1}, \mathscr{F})} \arrow[r]       & {R\Gamma_{\mathcal{U}_{k_2}\cap\mathcal{Z}_m}(\mathcal{U}_{k_1}'' \cap \mathcal{U}_{k_2}, \mathscr{F})} \arrow[r] & {R\Gamma_{\mathcal{U}''_{k_2}\cap\mathcal{Z}_m}(\mathcal{U}''_{k_2}, \mathscr{F})}
\end{tikzcd}
    \]
    where all the maps are induced from restriction. The right-hand vertical arrow is an isomorphism by excision. This implies that we have factorisations 
    \[
    R\Gamma_{\mathcal{U}_{k_1}\cap\mathcal{Z}_m}(\mathcal{U}_{k_1}, \mathscr{F}) \to R\Gamma_{\mathcal{U}''_{k_1}\cap\mathcal{Z}_m}(\mathcal{U}''_{k_1}, \mathscr{F}) \to R\Gamma_{\mathcal{U}_{k_2}\cap\mathcal{Z}_m}(\mathcal{U}_{k_2}, \mathscr{F}) \to R\Gamma_{\mathcal{U}''_{k_2}\cap\mathcal{Z}_m}(\mathcal{U}''_{k_2}, \mathscr{F}) \to \cdots 
    \]
    functorial in $m$, which implies that the map (\ref{ResForIntertwined}) is a quasi-isomorphism.

    For case (2), the proof of the statement is very similar to (1), but now the quasi-isomorphism 
    \[
    R\Gamma(\mathcal{U}_{\bullet}'', \mathcal{Z}_{\bullet}''; \mathscr{F}) \xrightarrow{\opn{cores}} R\Gamma(\mathcal{U}_{\bullet}'', \mathcal{Z}_{\bullet}; \mathscr{F})
    \]
    is built up from corestriction maps. We leave the details to the reader. Finally, we now see that the map (\ref{NaturalQIsoIntertwinedEqn}) is the following composition
    \[
    R\Gamma(\mathcal{U}_{\bullet}, \mathcal{Z}_{\bullet}; \mathscr{F}) \xrightarrow{\sim} R\Gamma(\mathcal{U}_{\bullet}'', \mathcal{Z}_{\bullet}; \mathscr{F}) \xleftarrow{\sim} R\Gamma(\mathcal{U}_{\bullet}'', \mathcal{Z}_{\bullet}''; \mathscr{F}) \xrightarrow{\sim} R\Gamma(\mathcal{U}_{\bullet}'', \mathcal{Z}_{\bullet}'; \mathscr{F}) \xleftarrow{\sim} R\Gamma(\mathcal{U}_{\bullet}', \mathcal{Z}_{\bullet}'; \mathscr{F})
    \]
    which is $T$-equivariant.
\end{proof}

\begin{example} \label{HHTsupportEqualsThisExample}
    Suppose that $(\mathcal{U}, \mathcal{Z})$ is an open/closed support condition for the correspondence $\mathcal{C}$ as in \cite[Definition 6.1.3]{HHTBoxerPilloni}. Then the collection $\mathcal{U}_k \defeq T^{k-1}(\mathcal{U})$, $\mathcal{Z}_m \defeq (T^t)^{m-1}(\mathcal{Z})$ is a system of support conditions as in Definition \ref{SystemOfSupportConditionsDef}.
\end{example}

%----------------------------------

\section{Nearly holomorphic automorphic forms} \label{NearlyHolAutFormsSection}

In this section we introduce the unitary Shimura varieties that will be used throughout this article. We will also describe the space of nearly holomorphic automorphic forms equipped with an action of differential operators, and construct classical ``evaluation maps'' which are closely related to unitary Friedberg--Jacquet periods.

\subsection{PEL data and torsors}

In this section we describe the abelian varieties with extra structure that will feature in our moduli problems. Suppose $E = \mbb{Q}(\sqrt{-d})$ for a fixed choice of square-root $\sqrt{-d}$.

\subsubsection{The PEL data}

Let $n \geq 2$ be an integer and, as in \cite[\S 2]{UFJ}, we fix a $2n$-dimensional Hermitian space $W$ over $F$ with signature given by the generalised CM-type of rank $2n$
\[
\boldsymbol\Psi = \tau_0 + (2n-1)\tau_0^c + \sum_{\tau \in \Psi - \{\tau_0\}} 2n \tau^c .
\]
Let $\langle \cdot, \cdot \rangle_W$ denote the Hermitian pairing on $W$, and let $\langle \cdot, \cdot \rangle \colon W \times W \to \mbb{Q}$ denote the induced alternating bilinear pairing given by $\opn{tr}_{F^+/\mbb{Q}} \circ \opn{Im}_{\sqrt{-d}} \langle \cdot, \cdot \rangle_W$, where $\opn{Im}_{\sqrt{-d}} \colon F \to F^+$ is the map given by $x \mapsto (2 \sqrt{-d})^{-1} (x + \bar{x})$, and $\opn{tr}_{F^+} \colon F^+ \to \mbb{Q}$ is the trace map. 

Let $\mbf{G}$ be the unitary similitude group associated with $W$ as in \cite[Definition 2.0.1]{UFJ} (note that we can replace $\langle \cdot , \cdot \rangle_W$ with $\langle \cdot, \cdot \rangle$ in the definition). We suppose that we have a decomposition $W = W_1 \oplus W_2$ into the direct sum of $n$-dimensional Hermitian spaces, with respective generalised CM-types
\[
\boldsymbol\Psi_1 = \tau_0 + (n-1)\tau_0^c + \sum_{\tau \in \Psi - \{\tau_0\}} n \tau^c, \quad \quad \boldsymbol\Psi_2 = \sum_{\tau \in \Psi} n \tau^c ,
\]
and we let $\mbf{H} \subset \mbf{G}$ denote the subgroup preserving this decomposition.

Recall we have a decomposition $W \otimes_{\mbb{Q}} \mbb{R} = W^+ \oplus W^-$ into maximal subspaces where the Hermitian pairing is $\pm$-definite (we may assume they are orthogonal to each other). We let
\[
h_{\mbf{G}} \colon \mbb{C} \to \opn{End}_{\mbb{R}}(W \otimes_{\mbb{Q}} \mbb{R})
\]
be the map sending $h_{\mbf{G}}(z)$ to the element which acts as scalar multiplication by $z$ (resp. $\bar{z}$) on $W^+$ (resp. $W^-$). Then (with the convention that the Hermitian form is antilinear in the first-variable) the conditions in \cite[Definition 1.2.1.2]{LanArithmetic} are satisfied. Similarly, we obtain a morphism 
\[
h_{\mbf{H}} \colon \mbb{C} \to \opn{End}_{\mbb{R}}(W_1 \otimes_{\mbb{Q}} \mbb{R}) \oplus \opn{End}_{\mbb{R}}(W_2 \otimes_{\mbb{Q}} \mbb{R})
\]
satisfying the conditions in \cite[Definition 1.2.1.2]{LanArithmetic}. Without loss of generality, we may assume that $h_{\mbf{G}}$ is the composition of $h_{\mbf{H}}$ with the natural inclusion $\opn{End}_{\mbb{R}}(W_1 \otimes \mbb{R}) \oplus \opn{End}_{\mbb{R}}(W_2 \otimes \mbb{R}) \subset \opn{End}_{\mbb{R}}(W \otimes \mbb{R})$.

Finally, we fix $\mathcal{O}_F$-lattices $L_i \subset W_i$ ($i=1, 2$) and set $L = L_1 \oplus L_2 \subset W$. We let $L^{\#}$ denote the dual lattice under the pairing $\langle \cdot, \cdot \rangle$ above.

\begin{assumption} \label{AssumpPoddAndDoesntDivide}
Fix an odd prime $p$ not dividing $\opn{Disc}_{\ordd_F/\mbb{Z}} [L^{\#} : L ]$ that splits completely in $F/\mbb{Q}$.
\end{assumption}

\begin{remark}
    The pairs $(\mbf{G}, h_{\mbf{G}})$ and $(\mbf{H}, h_{\mbf{H}})$ are precisely the PEL Shimura--Deligne data considered in \cite{UFJ}. In \emph{op.cit.} we worked with the Shimura--Deligne varieties associated with these data, however in this article it seems more natural to work with the associated PEL moduli problems since these moduli problems have natural integral models. For $\mbf{G}$ these two spaces are the same, but they can be different for $\mbf{H}$ due to the potential failure of the Hasse principle. This is harmless in practice however, since the Shimura--Deligne variety for $\mbf{H}$ is a connected component of the PEL moduli space (see Lemma \ref{LemmaSDintoPEL} below).
\end{remark}

We introduce some useful notation:

\begin{notation}
Let $G$ denote the reductive group over $\mbb{Z}_p$ of symplectic $\mathcal{O}_F \otimes_{\mbb{Z}} \mbb{Z}_p$-equivariant similitudes of $L \otimes_{\mbb{Z}} \mbb{Z}_p$, and $H \subset G$ the reductive subgroup preserving the decomposition 
\[
L \otimes_{\mbb{Z}} \mbb{Z}_p = (L_1 \otimes_{\mbb{Z}} \mbb{Z}_p) \oplus (L_2 \otimes_{\mbb{Z}} \mbb{Z}_p ).
\]
We can and do identify 
\[
H = \opn{GL}_1 \times \prod_{\tau \in \Psi} (\opn{GL}_n \times \opn{GL}_n) \subset \opn{GL}_1 \times \prod_{\tau \in \Psi} \opn{GL}_{2n} = G
\]
with the embedding being the block diagonal one described in \cite[Remark 2.0.2]{UFJ}.
\end{notation}

\subsubsection{Unitary abelian varieties}

We now introduce unitary abelian varieties. Recall that $F^{\opn{cl}}$ denotes the Galois closure of $F$.

\begin{definition}
Let $S$ be a locally Noetherian scheme over $\opn{Spec}\ordd_{F^{\opn{cl}}, (p)}$. We say that a tuple $(A, \lambda, i)$ is a unitary abelian scheme of signature $\boldsymbol\Psi$ (resp. $\boldsymbol\Psi_1$, resp. $\boldsymbol\Psi_2$) if:
\begin{enumerate}
    \item $A$ is an abelian scheme over $S$
    \item $\lambda \colon A \to A^{\vee}$ is a $\mbb{Z}_{(p)}^{\times}$-polarisation
    \item $i \colon \mathcal{O}_F \otimes_{\mbb{Z}} \mbb{Z}_{(p)} \to \opn{End}_S(A) \otimes_{\mbb{Z}} (\mbb{Z}_{(p)})_S$ is a ring homomorphism such that the restriction of the Rosati involution to $\mathcal{O}_F \otimes_{\mbb{Z}} \mbb{Z}_{(p)}$ coincides with complex conjugation
    \item The characteristic polynomial of $i(a)$ for $a \in \mathcal{O}_F$ on $\opn{Lie}_{A/S}$ is given by
    \begin{align*} 
     &(T - \tau_0(a))(T-\tau_0^c(a))^{2n-1} \prod_{\tau \in \Psi - \{\tau_0\}} (T - \tau^c(a))^{2n} \\
    ( \text{ resp. } &(T - \tau_0(a))(T-\tau_0^c(a))^{n-1} \prod_{\tau \in \Psi - \{\tau_0\}} (T - \tau^c(a))^n  ,\text{ resp. } \prod_{\tau \in \Psi} (T - \tau^c(a))^n  ).
    \end{align*}
\end{enumerate}
Note that condition (4) implies that $A/S$ has relative dimension $2n [F^+ \colon \mbb{Q}]$ (resp. $n[F^+ : \mbb{Q}]$). For brevity, we will often simply say $A$ is a $\boldsymbol\Psi$-unitary (resp. $\boldsymbol\Psi_1$-unitary, resp. $\boldsymbol\Psi_2$-unitary) abelian scheme (over $S$). 
\end{definition}

We also introduce prime-to-$p$ level structures.

\begin{definition}
    Let $K^p \subset \mbf{G}(\mbb{A}_f^p)$ be a neat compact open subgroup and $(A, \lambda, i)$ a $\boldsymbol\Psi$-unitary abelian scheme over $S$. Then a $K^p$-level structure $\eta^p$ for $A$ is the data of a $\pi_1(S, s)$-invariant $K^p$-orbit of $F \otimes_{\mbb{Q}} \mbb{A}_f^p$-equivariant symplectic isomorphisms 
    \[
    \eta^p_s \colon W \otimes_{\mbb{Q}} \mbb{A}_f^p \xrightarrow{\sim} V^p A_s
    \]
    for each geometric point $s \in S$, which are compatible in changing $s$ (see \cite[Definition 1.3.8.7]{LanArithmetic} for more details).
    
    Similarly, let $U^p \subset \mbf{H}(\mbb{A}_f^p)$ be a neat compact open subgroup and let $(A_1, \lambda_1, i_1)$ and $(A_2, \lambda_2, i_2)$ be $\boldsymbol\Psi_1$-unitary and $\boldsymbol\Psi_2$-unitary abelian schemes over $S$ respectively. Set 
    \[
    (A = A_1 \oplus A_2, \lambda = \lambda_1 \oplus \lambda_2, i = i_1 + i_2)
    \]
    which is a $\boldsymbol\Psi$-unitary abelian scheme over $S$. Then a $U^p$-level structure $\eta^p$ is the data of a $\pi_1(S, s)$-invariant $U^p$-orbit of $F \otimes_{\mbb{Q}} \mbb{A}_f^p$-equivariant symplectic isomorphisms 
    \[
    \eta^p_s \colon (W_1 \otimes_{\mbb{Q}} \mbb{A}_f^p) \oplus (W_2 \otimes_{\mbb{Q}} \mbb{A}_f^p) \xrightarrow{\sim} V^p (A_1)_s \oplus V^p (A_2)_s
    \]
    for each geometric point $s \in S$, which are compatible in changing $s$ and respect the decompositions on both sides.
\end{definition}

We now introduce the moduli of these unitary abelian schemes.

\begin{definition} \label{HyperspecialVarietiesDefinition}
Let $K^p \subset \mbf{G}(\mbb{A}_f^p)$ (resp. $U^p \subset \mbf{H}(\mbb{A}_f^p)$) be a neat compact open subgroup. We let $X_{\mbf{G}}$ (resp. $X_{\mbf{H}}$) denote the functor taking a locally Noetherian scheme $S$ over $\opn{Spec}\ordd_{F^{\opn{cl}}, (p)}$ to the set of $\boldsymbol\Psi$-unitary abelian schemes (resp. pairs of $\boldsymbol\Psi_1$-unitary and $\boldsymbol\Psi_2$-unitary abelian schemes) over $S$ equipped with a $K^p$ (resp $U^p$) level structure, up to equivalence (as in \cite[\S 1.4.2]{LanArithmetic}). This is representable by a smooth projective scheme over $\opn{Spec}\ordd_{F^{\opn{cl}}, (p)}$ of dimension $2n-1$ (resp. $n-1$).

If $U^p \subset K^p$, we let $\iota \colon X_{\mbf{H}} \to X_{\mbf{G}}$ denote the natural finite unramified morphism sending $(A_1, \lambda_1, i_1, A_2, \lambda_2, i_2, \eta^p)$ to $(A_1 \oplus A_2, \lambda_1 \oplus \lambda_2, i_1 + i_2, \eta^p K^p)$.
\end{definition}

\subsubsection{Deeper level at $p$} \label{DeeperLevelAtpSubSec}

We now introduce some additional level structure at the prime $p$. Let $S$ be a locally Noetherian scheme over $\opn{Spec}F^{\opn{cl}}$. Let $\beta \geq 1$ be an integer. Since $p \nmid [L^{\#} : L]$, we have an induced symplectic pairing on $L/p^{\beta}L$ which we will continue to denote by $\langle \cdot, \cdot \rangle$. Since $p$ splits completely in $F/\mbb{Q}$, we have a decomposition 
\[
L/p^{\beta}L = \bigoplus_{\tau \in \Psi} \left( L/\ide{p}_{\tau}^{\beta}L \oplus L/\ide{p}_{\bar{\tau}}^{\beta}L \right)
\]
and similarly for $L_1$ and $L_2$. Each factor is a symplectic space with pairing induced from $\langle \cdot, \cdot \rangle$, and both $L/\ide{p}_{\tau}^{\beta}L$ and $L/\ide{p}_{\bar{\tau}}^{\beta}L$ are free of rank $2n$ over $\mbb{Z}/p^{\beta}\mbb{Z}$. The decomposition $L/\ide{p}_{\tau}^{\beta}L \oplus L/\ide{p}_{\bar{\tau}}^{\beta}L$ is a decomposition into maximal isotropic subspaces for the symplectic pairing, in particular $L/\ide{p}_{\bar{\tau}}^{\beta}L$ is in perfect duality with $L/\ide{p}_{\tau}^{\beta}L$ via $\langle \cdot, \cdot \rangle$.

Let $A/S$ be a $\boldsymbol\Psi$-unitary abelian scheme. Then the polarisation and endomorphism structure (and the fact that $p$ is invertible in $F^{\opn{cl}}$) imply that 
\[
A[p^{\beta}] = \bigoplus_{\tau \in \Psi} \left( A[\ide{p}_{\tau}^{\beta}] \oplus A[\ide{p}_{\bar{\tau}}^{\beta}] \right)
\]
is an \'{e}tale rank $2n[F^+ : \mbb{Q}]$ symplectic $\mathcal{O}_F/p^{\beta}\mathcal{O}_F$ group scheme over $S$. Hence it must be \'{e}tale locally isomorphic to $\left( L/p^{\beta}L \right)_S$ (since $p$ splits completely in $F/\mbb{Q}$). Similar isomorphisms exist for $\boldsymbol\Psi_i$-unitary schemes ($i=1,2$).

\begin{definition} \label{DefinitionOfDeeperLevelAtpVarieties}
Let $\beta \geq 1$ be an integer and let $\hat{\gamma} = \gamma w_n \in G(\mbb{Z}_p)$ denote the element introduced in Definition \ref{NewDefOfGamma}.
\begin{enumerate}
    \item Let $X_{\mbf{G}, \opn{Iw}, F^{\opn{cl}}}(p^{\beta}) \to X_{\mbf{G}, F^{\opn{cl}}}$ denote the finite \'{e}tale cover which, for a given point $(A, \lambda, i, \eta^p) \in X_{\mbf{G}, F^{\opn{cl}}}(S)$, parameterises flags
    \[
    0 = C_{0, \tau} \subset C_{1, \tau} \subset \cdots \subset C_{2n, \tau} = A[\ide{p}^{\beta}_{\tau}], \quad \quad \tau \in \Psi 
    \]
    where $C_{i, \tau}$ is a finite flat (\'{e}tale) subgroup scheme of rank $p^{i\beta}$, and $C_{i, \tau}/C_{i-1, \tau}$ is \'{e}tale locally isomorphic to $(\mbb{Z}/p^{\beta}\mbb{Z})_S$ (for all $i=1, \dots, 2n$). Equivalently, it parameterises $B_G(\mbb{Z}/p^{\beta}\mbb{Z})$-orbits of symplectic $\mathcal{O}_F/p^{\beta}\mathcal{O}_F$-equivariant isomorphisms $\left( L/p^{\beta} L \right)_S \cong A[p^{\beta}]$. Here $B_G \subset \mbf{G}_{\mbb{Q}_p}$ denotes the upper-triangular Borel subgroup.
    \item Let $X_{\mbf{H}, \diamondsuit, F^{\opn{cl}}}(p^{\beta}) \to X_{\mbf{H}, F^{\opn{cl}}}$ denote the finite \'{e}tale cover which, for a given point $(A_1, A_2, \eta^p) \in X_{\mbf{H}, F^{\opn{cl}}}(S)$, parameterises $\hat{\gamma} B_G(\mbb{Z}/p^{\beta}\mbb{Z}) \hat{\gamma}^{-1} \cap H(\mbb{Z}/p^{\beta}\mbb{Z})$-orbits of symplectic $\mathcal{O}_F/p^{\beta}\mathcal{O}_F$-equivariant isomorphisms 
    \[
    \left( L_1/p^{\beta} L_1 \right)_S \oplus \left( L_2/p^{\beta} L_2 \right)_S \cong A_1[p^{\beta}] \oplus A_2[p^{\beta}]
    \]
    respecting the decomposition on both sides.
\end{enumerate}
\end{definition}

We have a natural finite unramified map $\hat{\iota} \colon X_{\mbf{H}, \diamondsuit, F^{\opn{cl}}}(p^{\beta}) \to X_{\mbf{G}, \opn{Iw}, F^{\opn{cl}}}(p^{\beta})$ given by sending a $\hat{\gamma} B_G(\mbb{Z}/p^{\beta}\mbb{Z}) \hat{\gamma}^{-1} \cap H(\mbb{Z}/p^{\beta}\mbb{Z})$-orbit of symplectic $\mathcal{O}_F/p^{\beta}\mathcal{O}_F$-equivariant isomorphisms 
    \[
    \alpha \colon \left( L_1/p^{\beta} L_1 \right)_S \oplus \left( L_2/p^{\beta} L_2 \right)_S \xrightarrow{\sim} A_1[p^{\beta}] \oplus A_2[p^{\beta}]
    \]
to the orbit of isomorphisms $[\alpha \circ \hat{\gamma}]$ (for the unitary abelian scheme $A = A_1 \oplus A_2$). In particular, we have a commutative diagram:
\[
\begin{tikzcd}
{X_{\mbf{H}, \diamondsuit, F^{\opn{cl}}}(p^{\beta+1})} \arrow[d] \arrow[r, "\hat{\iota}"] & {X_{\mbf{G}, \opn{Iw}, F^{\opn{cl}}}(p^{\beta+1})} \arrow[d] \\
{X_{\mbf{H}, \diamondsuit, F^{\opn{cl}}}(p^{\beta})} \arrow[d] \arrow[r, "\hat{\iota}"]   & {X_{\mbf{G}, \opn{Iw}, F^{\opn{cl}}}(p^{\beta})} \arrow[d]   \\
{X_{\mbf{H}, F^{\opn{cl}}}} \arrow[r, "\iota"]                                            & {X_{\mbf{G}, F^{\opn{cl}}}}                                 
\end{tikzcd}
\]
for any $\beta \geq 1$. The top square is Cartesian (see \cite[Lemma 2.5.3]{UFJ}).

\subsubsection{Relation to Shimura--Deligne varieties}

The moduli spaces considered above are closely related to the Shimura--Deligne varieties associated with $\mbf{G}$ and $\mbf{H}$. More precisely, let $G(\mbb{Z}_p) \subset \mbf{G}(\mbb{Q}_p)$ (resp. $H(\mbb{Z}_p) \subset \mbf{H}(\mbb{Q}_p)$) denote the subgroup of similitudes which preserve the lattice $L = L_1 \oplus L_2$. These are hyperspecial subgroups. We let $S_{\mbf{G}}$ (resp. $S_{\mbf{H}}$) denote the Shimura--Deligne variety over $F^{\opn{cl}}$ associated with the datum $(\mbf{G}, h_{\mbf{G}})$ (resp. $(\mbf{H}, h_{\mbf{H}})$) of level $K^p G(\mbb{Z}_p)$ (resp. $U^p H(\mbb{Z}_p)$) for some neat compact open subgroup $K^p \subset \mbf{G}(\mbb{A}_f^p)$ (resp. $U^p \subset \mbf{H}(\mbb{A}_f^p)$).

For an integer $\beta \geq 1$, let $K^G_{\opn{Iw}}(p^{\beta}) \subset G(\mbb{Z}_p)$ and $K^H_{\diamondsuit}(p^{\beta}) \subset H(\mbb{Z}_p)$ denote the compact open subgroups defined in \cite[Definition 2.5.2]{UFJ}. We let $S_{\mbf{G}, \opn{Iw}}(p^\beta)$ (resp. $S_{\mbf{H}, \diamondsuit}(p^{\beta})$) denote the Shimura--Deligne variety of level $K^p K^G_{\opn{Iw}}(p^{\beta})$ (resp. $U^p K_{\diamondsuit}^H(p^{\beta})$).

\begin{lemma} \label{LemmaSDintoPEL}
    There exist Cartesian diagrams:
    \[
\begin{tikzcd}
{S_{\mbf{H}, \diamondsuit}(p^{\beta})} \arrow[d] \arrow[r, hook] & {X_{\mbf{H}, \diamondsuit, F^{\opn{cl}}}(p^{\beta})} \arrow[d] &  & {S_{\mbf{G}, \opn{Iw}}(p^{\beta})} \arrow[d] \arrow[r, "\sim"] & {X_{\mbf{G}, \opn{Iw}, F^{\opn{cl}}}(p^{\beta})} \arrow[d] \\
S_{\mbf{H}} \arrow[r, hook]                                      & {X_{\mbf{H}, F^{\opn{cl}}}}                                    &  & S_{\mbf{G}} \arrow[r, "\sim"]                                  & {X_{\mbf{G}, F^{\opn{cl}}}}                               
\end{tikzcd}
    \]
    where the horizontal arrows are open and closed embeddings. Furthermore, the morphisms 
    \[
    \hat{\iota} \colon X_{\mbf{H}, \diamondsuit, F^{\opn{cl}}}(p^{\beta}) \to X_{\mbf{G}, \opn{Iw}, F^{\opn{cl}}}(p^{\beta}), \quad \quad \hat{\iota} \colon S_{\mbf{H}, \diamondsuit}(p^{\beta}) \to S_{\mbf{G}, \opn{Iw}}(p^{\beta})
    \]
    (the latter given by right-translation by $\hat{\gamma}$) are compatible under these embeddings.
\end{lemma}
\begin{proof}
    Since $p$ splits completely in $F/\mbb{Q}$, there exists a unique (up to isomorphism) $\mathcal{O}_F \otimes_{\mbb{Z}} \mbb{Z}_p$-Hermitian space of a given dimension. In particular, this implies that (over any locally Noetherian scheme $S \to \opn{Spec}F^{\opn{cl}}$) a prime-to-$p$ level structure can be lifted to an orbit of trivialisations of the Tate module at all finite places. Hence all the moduli problems in this lemma can be replaced with the versions which include level stucture at $p$. 
    
    The commutative diagrams then follow from \cite[\S 2.5]{Lan-Toroid} -- they are Cartesian because the vertical maps are finite \'{e}tale covers and have the same degree. Finally, we note that the horizontal arrows in the diagram for $\mbf{G}$ are isomorphisms. Indeed, the moduli spaces for $\mbf{G}$ are a disjoint union of Shimura varieties for unitary similitude groups associated with Hermitian spaces which are locally isomorphic (up to similitude) to $W$ at all places. By the Hasse principle for \emph{even dimensional} unitary groups (see \cite[\S 7]{Kottwitz}), there is only one such Hermitian space (up to similitude), namely $W$.
\end{proof}

\subsubsection{Torsors}

We now describe the de Rham torsor that will play an important role in the discussion of nearly holomorphic/overconvergent automorphic forms. Let $P_{\mbf{G}} \subset \mbf{G}_{F^{\opn{cl}}}$ denote the parabolic subgroup associated with the Shimura--Deligne datum for $\mbf{G}$, as defined in \cite[\S 2]{UFJ}. Let $M_{\mbf{G}}$ denote its Levi subgroup. We let $P_{\mbf{H}} = P_{\mbf{G}} \cap \mbf{H}_{F^{\opn{cl}}}$ and $M_{\mbf{H}}$ its Levi subgroup. Recall that we have identifications
\begin{align*} 
M_{\mbf{G}} &= \opn{GL}_1 \times (\opn{GL}_1 \times \opn{GL}_{2n-1}) \times \prod_{\tau \in \Psi - \{ \tau_0 \}} \opn{GL}_{2n} \\
M_{\mbf{H}} &= \opn{GL}_1 \times (\opn{GL}_1 \times \opn{GL}_{n-1} \times \opn{GL}_n) \times \prod_{\tau \in \Psi - \{\tau_0 \}} ( \opn{GL}_n \times \opn{GL}_n ) .
\end{align*}
We let $P_{\mbf{G}}^{\opn{std}}$ and $P_{\mbf{H}}^{\opn{std}}$ denote the opposite parabolic subgroups with respect to the standard upper triangular Borel subgroups of $\mbf{G}$ and $\mbf{H}$ respectively.

\begin{notation}
For a $\boldsymbol\Psi$-unitary abelian scheme $A$ over a locally Noetherian scheme $S \to \opn{Spec}\mathcal{O}_{F^{\opn{cl}}, (p)}$, let $\mathcal{H}_A \defeq \mathcal{H}_1^{\opn{dR}}(A/S)$ denote the first relative de Rham homology of $A/S$. This is a vector bundle of rank $4n [F^+ : \mbb{Q}]$ and comes equipped with the Hodge filtration
\begin{equation} \label{HodgeFiltrationForUAS}
0 \to \omega_{A^D} \to \mathcal{H}_A \to \opn{Lie}(A/S) \to 0
\end{equation}
where $\omega_{A^D} = \pi_* \Omega^1_{A^D/S}$ denotes the Hodge bundle (with $\pi \colon A^D \to S$) and $\opn{Lie}(A/S)$ denotes the Lie algebra of $A/S$. 

The $\mathcal{O}_F$-endomorphism action induces a decomposition 
\[
\mathcal{H}_A = \bigoplus_{\tau \in \Psi} (\mathcal{H}_{A, \tau} \oplus \mathcal{H}_{A, \bar{\tau}} )
\]
into isotypic pieces (where on each piece, $\mathcal{O}_F$ acts through scalar multiplication via the corresponding embedding). One has a similar exact sequence to (\ref{HodgeFiltrationForUAS}) on isotypic pieces. The same discussion applies to $\boldsymbol\Psi_i$-unitary abelian schemes ($i=1, 2$) and we will use similar notation for the de Rham homology and Hodge bundles. 
\end{notation}

We now introduce the de Rham torsors.

\begin{definition} \label{DefinitionOfGdrPdR}
Let $\beta \geq 1$ be an integer.
    \begin{enumerate}
        \item Let $G_{\opn{dR}} \to X_{\mbf{G}, \opn{Iw}, F^{\opn{cl}}}(p^{\beta})$ denote the right $\mbf{G}_{F^{\opn{cl}}}$-torsor parameterising isomorphisms 
        \[
        \mathcal{O}_S \otimes_{\mbb{Q}} W \xrightarrow{\sim} \mathcal{H}_A, \quad \quad A \in X_{\mbf{G}, \opn{Iw}, F^{\opn{cl}}}(p^{\beta})(S)
        \]
        which respect the symplectic and endomorphism structures on both sides. Similarly, let $H_{\opn{dR}} \to X_{\mbf{H}, \diamondsuit, F^{\opn{cl}}}(p^{\beta})$ denote the right $\mbf{H}_{F^{\opn{cl}}}$-torsor parameterising isomorphisms 
        \[
        (\mathcal{O}_S \otimes_{\mbb{Q}} W_1) \otimes_{\mathcal{O}_S} (\mathcal{O}_S \otimes_{\mbb{Q}} W_2) \xrightarrow{\sim} \mathcal{H}_{A_1} \oplus \mathcal{H}_{A_2}, \quad \quad (A_1, A_2) \in X_{\mbf{H}, \diamondsuit, F^{\opn{cl}}}(p^{\beta})(S)
        \]
        respecting the symplectic, endomorphism structures and the decompositions on both sides.
        \item Let $P_{G, \opn{dR}} \to X_{\mbf{G}, \opn{Iw}, F^{\opn{cl}}}(p^{\beta})$ and $P_{H, \opn{dR}} \to X_{\mbf{H}, \diamondsuit, F^{\opn{cl}}}(p^{\beta})$ denote the $P_{\mbf{G}}^{\opn{std}}$ and $P_{\mbf{H}}^{\opn{std}}$ reductions of $G_{\opn{dR}}$ and $H_{\opn{dR}}$ respectively, given by trivialisations which respect the Hodge filtration. 
        \item Let $M_{G, \opn{dR}} \to X_{\mbf{G}, \opn{Iw}, F^{\opn{cl}}}(p^{\beta})$ and $M_{H, \opn{dR}} \to X_{\mbf{H}, \diamondsuit, F^{\opn{cl}}}(p^{\beta})$ denote the pushouts of $P_{G, \opn{dR}}$ and $P_{H, \opn{dR}}$ along the projection maps $P_{\mbf{G}}^{\opn{std}} \to M_{\mbf{G}}$ and $P_{\mbf{H}}^{\opn{std}} \to M_{\mbf{H}}$ respectively.
    \end{enumerate}
    Note that $G_{\opn{dR}}$ (resp. the pullback of $H_{\opn{dR}}$ along the open and closed embedding $S_{\mbf{H}, \diamondsuit}(p^{\beta}) \hookrightarrow X_{\mbf{H}, \diamondsuit, F^{\opn{cl}}}(p^{\beta})$ from Lemma \ref{LemmaSDintoPEL}) coincides with the standard principal bundle in \cite[\S III.3]{MilneCanonicalMixed}. Furthermore, $P_{G, \opn{dR}}$ and $P_{H, \opn{dR}}$ are the usual reductions of structure appearing in the theory of automorphic vector bundles (see \cite[\S 2]{CS17} for example).
\end{definition}

\subsection{D-modules on flag varieties} \label{DmodulesOnFLsection}

In this section, we describe the construction of $\mathcal{D}$-modules on flag varieties from $(\ide{g}, P_{\mbf{G}}^{\opn{std}})$ or $(\ide{h}, P_{\mbf{H}}^{\opn{std}})$ representations. We will then explain how to transport these modules to $\mathcal{D}$-modules on Shimura varieties, and prove a key result (Corollary \ref{KeyCorForHolOnX}) which will be used throughout the rest of this section. To ease notation, we will only describe the construction for $\mbf{G}$ -- the construction for $\mbf{H}$ (or indeed any well-behaved Shimura variety) follows exactly the same arguments.

\subsubsection{Running notation}

In this subsection, we fix some notation that will be used in the rest of this section. Since we will only consider the setting for $\mbf{G}$, we will drop the group from the notation for almost all objects in this section.

Throughout, we will let $k = F^{\opn{cl}}$. We let $\opn{FL} = \opn{FL}_{\mbf{G}}^{\opn{std}} = \mbf{G}_k/P_{\mbf{G}}^{\opn{std}}$ denote the partial flag variety over $\opn{Spec}(k)$. Also, to simplify notation, we will write $G = \mbf{G}_k$, $P = P_{\mbf{G}}$, $\overline{P} = P_{\mbf{G}}^{\opn{std}}$, and $M = M_{\mbf{G}}$. We will denote the Lie algebra of $\mbf{G}$ (resp. $P$, resp. $\overline{P}$) by $\ide{g}$ (resp. $\ide{p}$, resp. $\overline{\ide{p}}$), and let $\ide{u}$ denote the (upper-triangular) nilpotent part of $\ide{p}$. The adjoint action of a group on its Lie algebra will be denoted by $\opn{Ad}$.

\begin{definition}
    Let $\opn{Rep}_k(\ide{g}, \overline{P})$ denote the category of $(\ide{g}, \overline{P})$-representations, i.e., algebraic representations $V$ of $\overline{P}$ which come equipped with an action of $\ide{g}$ satisfying the following properties:
    \begin{enumerate}
        \item For any $p \in \overline{P}$, $X \in \ide{g}$ and $v \in V$, one has
        \[
        p \cdot (X \cdot v) = (\opn{Ad}(p)X) \cdot (p \cdot v) .
        \]
        \item For any $X \in \overline{\ide{p}}$ and $v \in V$, one has 
        \[
        X \cdot v = \left. \left( \frac{d}{dt} \opn{exp}(tX) \cdot v \right) \right|_{t=0} .
        \]
    \end{enumerate}
\end{definition}

For any scheme $Y$ over $\opn{Spec}(k)$, we let $\mathcal{T}_Y$ denote its tangent bundle and $\mathcal{D}_Y$ the sheaf of differential operators. If $\mathcal{A}$ is a sheaf of $\mathcal{O}_Y$-algebras, then we let $\opn{Der}_k(\mathcal{A}, \mathcal{A})$ denote the $k$-module of $k$-linear derivations $\mathcal{A} \to \mathcal{A}$. Finally, we fix an integer $\beta \geq 1$ and set $X = X_{\mbf{G}, \opn{Iw}, F^{\opn{cl}}}(p^{\beta})$.

\subsubsection{Actions on the flag variety} \label{ActionsOnFLsubsec}

We will consider the following sheaves of Lie algebras on the flag variety $\opn{FL}$. 

\begin{definition}
Set $\mathfrak{g}^0 = \mathcal{O}_{\opn{FL}} \otimes_k \ide{g}$. We denote by $\overline{\mathfrak{p}}^0$ the vector bundle on $\opn{FL}$ which on any open $U \subset \opn{FL}$ is given by
\[
\overline{\ide{p}}^0(U) = \{ f \colon U \to \mathfrak{g} : f(x) \in \overline{\ide{p}}_x \text{ for all } x \in U  \}
\]
where $\overline{\ide{p}}_x$ denotes the Lie algebra of the parabolic corresponding to the point $x \in U \subset \opn{FL}$. The tangent bundle of $\opn{FL}$ then coincides with $\mathcal{T}_{\opn{FL}} = \ide{g}^0 / \overline{\ide{p}}^0$.
\end{definition}

Let $\pi \colon G \to \opn{FL}$ be the natural right $\overline{P}$-torsor. Then, as explained in \cite[Appendix A]{DiffOps}, to any $(\ide{g}, \overline{P})$-representation $V$ one has a natural action of $\ide{g}^0$ on $\mathcal{V} = \left(\pi_*\mathcal{O}_G \otimes_k V \right)^{\overline{P}}$ (denoted $\star_{\mathcal{D}}$) factoring through $\mathcal{T}_{\opn{FL}}$. Let us recall how this action is constructed.

\begin{definition} \label{Def:StarDaction}
    Let $\mathcal{V}$ be as above. Then we define an action of $\mathfrak{g}^0$ on $\mathcal{V}$ by the following formula: for $\lambda \in \ide{g}^0(U)$ viewed as a function $U \to \ide{g}$, and $F \in \mathcal{V}(U)$ viewed as a function $F \colon \pi^{-1}(U) \to V$, we set
    \[
    (\lambda \star_{\mathcal{D}} F)(g) = (\lambda(g) \star_1 F)(g) - (\opn{Ad}(g^{-1})\lambda(g) \star_2 F)(g)
    \]
    where $\star_1$ denotes the action of $\ide{g}$ induced from differentiating the $G$-equivariant structure $(h \cdot F)(-) = F(h^{-1}\cdot -)$ ($h \in G$), and $\star_2$ denotes the $\ide{g}$-action on $V$.
\end{definition}

We explain how to construct a ``horizontal action'' out of this for a specific choice of $V$.

\begin{definition}
    Let $k[\overline{P}]$ denote the $k$-algebra of algebraic functions on $\overline{P}$. This comes equipped with two actions of $\overline{P}$, namely for $p \in \overline{P}$ we set
    \[
    (p \star_l f)(-) = f(p^{-1}\cdot -), \quad (p \star_r f)(-) = f(- \cdot p), \quad f \in k[\overline{P}] .
    \]
    The $\star_l$-action can be enhanced to a $(\ide{g}, \overline{P})$-action, by considering the open immersion $\overline{P} \hookrightarrow G/U$, where $U$ is the unipotent radical of $P$, and differentiating the $G$-action on $G/U$. This representation comes with some additional structure, namely the action of $\ide{g}$ (resp. $\overline{P}$) is through derivations (resp. algebra automorphisms). 
\end{definition}

\begin{remark}
Let $\kappa \in X^*(T)$ and let $W_{\kappa}$ denote the algebraic representation of $M$ of highest weight $\kappa$. Then the (relative) Verma module of weight $\kappa$ is given by $\mathcal{U}(\ide{g}) \otimes_{\mathcal{U}(\overline{\ide{p}})} W_{\kappa}$ and the dual of this representation in the Bernstein--Gelfand--Gelfand category $\mathcal{O}$ coincides with the representation $(k[\overline{P}] \otimes_k W_{\kappa} )^{(M, \star_r)}$ (invariants of the action of $M$ under $\star_r \otimes \cdot$).
\end{remark}

We now construct the desired action.

\begin{lemma} \label{HorzActionOnFL}
Let $V = k[\overline{P}]$. Then the action of $\ide{g}^0$ on $\mathcal{V}$ can be extended to an action of $\pi_{*}\mathcal{O}_G \otimes_{\mathcal{O}_{\opn{FL}}} \mathfrak{g}^0$ factoring through the quotient $(\pi_*\mathcal{O}_G \otimes_{\mathcal{O}_{\opn{FL}}} \mathfrak{g}^0) / (\pi_*\mathcal{O}_G \otimes_{\mathcal{O}_{\opn{FL}}} \overline{\mathfrak{p}}^0) = \pi_*\mathcal{O}_G \otimes_{\mathcal{O}_{\opn{FL}}} \mathcal{T}_{\opn{FL}}$. In particular, we obtain an induced Lie algebra homomorphism
\begin{equation} \label{uGactionEqn}
\mathfrak{u} \to \opn{Der}_k(\mathcal{V}, \mathcal{V})
\end{equation}
where $\ide{u}$ denotes the Lie algebra of the unipotent radical of $P$.
\end{lemma}
\begin{proof}
Let $\star_{\mathcal{D}}$ denote the action of $\ide{g}^0$ on $\mathcal{V}$. Note that we have 
\begin{align*} 
\mathcal{V}(U) &= \{ f \colon \pi^{-1}(U) \to k[\overline{P}] : f(-\cdot p) = p^{-1} \star_l f(-) \text{ for all } p \in \overline{P} \} \\
 &= \{ F \colon \pi^{-1}(U) \times \overline{P} \to k : F(- \cdot p, -) = F(-, p \cdot -) \text{ for all } p \in \overline{P} \} \\
 &= \pi_* \mathcal{O}_G(U)
\end{align*}
so we get an action of $\pi_*\mathcal{O}_G \otimes_{\mathcal{O}_{\opn{FL}}} \mathfrak{g}^0$ by using the multiplication structure on $\mathcal{V} = \pi_*\mathcal{O}_G$ (i.e. $x \otimes y$ acts as $x \cdot (y \star_{\mathcal{D}} -)$). Clearly this action factors through $\pi_*\mathcal{O}_G \otimes_{\mathcal{O}_{\opn{FL}}} \mathcal{T}_{\opn{FL}}$.

Alternatively, one can view an element $\lambda \in \pi_*\mathcal{O}_G \otimes_{\mathcal{O}_{\opn{FL}}} \mathfrak{g}^0$ as a function $\lambda \colon \pi^{-1}(U) \to \ide{g}$ and define the action as
\[
(\lambda \; \tilde{\star}_{\mathcal{D}} \; F)(g, p) \defeq ( \lambda(gp) \star_{1} F)(g, p) - (\opn{Ad}(g^{-1})\lambda(gp) \star_2 F)(g, p)
\]
where $\star_1$ is the $\ide{g}$-action induced from left-translation on the first variable, and $\star_2$ arises from the $\ide{g}$-action on $k[\overline{P}]$. Now the action in (\ref{uGactionEqn}) is given as follows. For $X \in \ide{u}$, we consider the function $\lambda_X \colon \pi^{-1}(U) \to \ide{g}$ given by $\lambda_X(g) = -\opn{Ad}(g)X$. The action of $X$ is then given by the action of $\lambda_X$. Explicitly, we have
\begin{equation} \label{lambdaXTildeEqn}
(\lambda_X \; \tilde{\star}_{\mathcal{D}} \; F)(g, p) = (-\opn{Ad}(gp)X \star_1 F)(g, p) + (\opn{Ad}(p)X \star_2 F)(g, p) .
\end{equation}
Note that the second term in (\ref{lambdaXTildeEqn}) vanishes because $X \in \ide{u}$, and we can rewrite the action in (\ref{lambdaXTildeEqn}) as
\[
\lambda_X \; \tilde{\star}_{\mathcal{D}} \; F =  X \star_r F, \quad \quad \quad F \in \pi_*\mathcal{O}_G(U) = \mathcal{O}_G(\pi^{-1}U),
\]
where $\star_r$ denotes the action of $\ide{g}$ on $\mathcal{O}_G$ given by differentiating the action of $G$ given by right-translation of the argument. This automatically implies that $X \mapsto \lambda_{X} \tilde{\star}_{\mathcal{D}} -$ is a Lie algebra homomorphism (i.e., the actions of $\lambda_X$ and $\lambda_Y$ commute for $X, Y \in \ide{u}$).
\end{proof}

Note that $\mathcal{V} = \pi_*\mathcal{O}_G$ comes with an additional action of $\overline{\ide{p}}$ given by differentiating $\star_r$ on $k[\overline{P}]$ (which coincides with the action of $\star_r$ on $\mathcal{O}_G$ above). We will denote this action by $\star_{\overline{P}}$. We have the following key relation:

\begin{proposition} \label{KeyPropforHolOnFL}
Let $\mathcal{V} = \pi_*\mathcal{O}_G$. Let $X \in \overline{\ide{p}}$, $Y \in \ide{u}$, and let $[X, Y] \in \ide{g}$ denote the Lie bracket as elements of $\ide{g}$. Suppose that $[X, Y] \in \overline{\ide{p}}$. Then for any $F \in \mathcal{V}$, we have the relation
\[
X \star_{\overline{P}} (Y \star_{\ide{u}} F) = Y \star_{\ide{u}} (X \star_{\overline{P}} F) + [X, Y] \star_{\overline{P}} F
\]
where $\star_{\ide{u}}$ denotes the action of $\ide{u}$ constructed in Lemma \ref{HorzActionOnFL}. 
\end{proposition}
\begin{proof}
    This follows immediately from the fact that $\star_{\ide{u}}$ and $\star_{\overline{P}}$ are both induced from the action of $\ide{g}$ on $\mathcal{O}_G$ via $\star_r$.
\end{proof}

\subsubsection{Passage to Shimura varieties} \label{PassageToShVars}

In this section we describe the relation between the tangent bundles of Shimura varieties and flag varieties.  The main references we follow for this are \cite[\S 3.4.2]{HarrisAVBI} and \cite[\S 3.1]{Hormann_2024}, although note that our situation is significantly simpler than the latter because our Shimura varieties are compact (so we do not need to consider toroidal compactifications and log differentials).

Recall that $X = X_{\mbf{G}, \opn{Iw}, F^{\opn{cl}}}(p^{\beta})$ denotes the Shimura--Deligne variety associated with $\mbf{G}$ of level $K^pK^G_{\opn{Iw}}(p^{\beta})$. Consider the following diagram
\[
\begin{tikzcd}
  & G_{\opn{dR}} \arrow[ld, "p"'] \arrow[rd, "q"] &                        \\
X &                                               & \opn{FL}
\end{tikzcd}
\]
where $G_{\opn{dR}}$ denotes the right $G$-torsor (the standard principal bundle) parameterising frames of the first de Rham homology respecting PE-structures (see Definition \ref{DefinitionOfGdrPdR}). Here $p$ denotes the natural map, and $q$ denotes the map measuring the position of the Hodge filtration with respect to such a frame. The map $q$ is $G$-equivariant, i.e., $q(x g) = g^{-1} q(x)$.

Since $G_{\opn{dR}} \to X$ is a principal $G$-torsor, we have a $G$-equivariant short exact sequence
\begin{equation} \label{Xsestangent}
    0 \to \mathcal{T}_{G_{\opn{dR}}}^{p\text{-vert}} \to \mathcal{T}_{G_{\opn{dR}}} \xrightarrow{dp} p^*\mathcal{T}_{X} \to 0
\end{equation}
where $\mathcal{T}_{G_{\opn{dR}}}^{p\text{-vert}} = \mathcal{O}_{G_{\opn{dR}}} \otimes_k \ide{g}$ with the $G$-equivariant structure given by the diagonal action (i.e., the $G$-equivariant structure on $\mathcal{O}_{G_{\opn{dR}}}$ and the adjoint action on $\ide{g}$). Since the relative de Rham homology comes equipped with a connection (the Gauss--Manin connection), we obtain a $G$-equivariant splitting $s_X \colon p^*\mathcal{T}_X \to \mathcal{T}_{G_{\opn{dR}}}$ of (\ref{Xsestangent}) whose image we denote by $\mathcal{T}_{G_{\opn{dR}}}^{\opn{horz}}$. Since the connection is integrable, the subbundle $\mathcal{T}_{G_{\opn{dR}}}^{\opn{horz}}$ is closed under the $G$-equivariant Lie bracket on $\mathcal{T}_{G_{\opn{dR}}}$  (this is ``axiom (F)'' in \cite[\S 3.1.4]{Hormann_2024}). In particular, the section $s_X$ induces a $G$-equivariant Lie bracket structure $[\cdot, \cdot]_X$ on $p^*\mathcal{T}_X$ -- concretely this is described on pure tensors as 
\[
[f \otimes v, f' \otimes v']_X = s_X(f \otimes v) f' \otimes v' - s_X(f' \otimes v') f \otimes v + f f' \otimes [v, v']_X
\]
for $f \otimes v$ and $f' \otimes v'$ elements of $\mathcal{O}_{G_{\opn{dR}}} \otimes_{p^{-1}\mathcal{O}_X} p^{-1}\mathcal{T}_X$ (and we also use the notation $[\cdot, \cdot]_X$ to denote the Lie bracket on $\mathcal{T}_X$). As seen from this description, this Lie bracket extends the one on $\mathcal{T}_X$, so the notation is justified. 

On the other hand, let $\mathcal{T}_{G_{\opn{dR}}}^{q\text{-vert}}$ denote the kernel of the $G$-equivariant map $dq \colon \mathcal{T}_{G_{\opn{dR}}} \to q^* \mathcal{T}_{\opn{FL}}$. The morphism $dq$ is surjective and we have a G-equivariant short exact sequence
\begin{equation} \label{FLestangent}
     0 \to \mathcal{T}_{G_{\opn{dR}}}^{q\text{-vert}} \to \mathcal{T}_{G_{\opn{dR}}} \xrightarrow{dq} q^*\mathcal{T}_{\opn{FL}} \to 0
\end{equation}
We have the following result:

\begin{proposition} \label{TXTFLrelation}
One has an isomorphism $dq \colon \mathcal{T}_{G_{\opn{dR}}}^{\opn{horz}} \xrightarrow{\sim} q^*\mathcal{T}_{\opn{FL}}$. In particular, we obtain a $G$-equivariant splitting $s_{\opn{FL}} \colon q^*\mathcal{T}_{\opn{FL}} \to \mathcal{T}_{G_{\opn{dR}}}$ of (\ref{FLestangent}).
\end{proposition}
\begin{proof}
This is proved in \cite{HarrisAVBI}. The result is first established over $\mbb{C}$ (as complex manifolds) and then is shown to be algebraic and descends to $F^{\opn{cl}}$.  This result corresponds to the ``Torelli axiom (T)'' in \cite[\S 3.1.9]{Hormann_2024}.
\end{proof}

As a consequence of the splitting $s_{\opn{FL}}$, we obtain a $G$-equivariant Lie bracket on $q^*\mathcal{T}_{\opn{FL}}$ which extends the bracket on $\mathcal{T}_{\opn{FL}}$. We have an induced $G$-equivariant Lie algebra isomorphism $p^*\mathcal{T}_{X} \xrightarrow{s_X} \mathcal{T}_{G_{\opn{dR}}}^{\opn{horz}} \xrightarrow{dq} q^*\mathcal{T}_{\opn{FL}}$, which encodes the Kodaira--Spencer isomorphism  (see \cite[Theorem 10.4]{Voisin_2002}).

We now recall how one passes from $\mathcal{D}_{\opn{FL}}$-modules on $\opn{FL}$ to $\mathcal{D}_X$-modules on $X$. 

\begin{definition}
    Let $\mathcal{V}$ be a quasi-coherent $G$-equivariant sheaf on $\opn{FL}$. We set $[\mathcal{V}] \defeq (p_* q^*\mathcal{V})^G$, which defines a quasi-coherent sheaf on $X$.
\end{definition}

\begin{example}
As seen from Proposition \ref{TXTFLrelation}, one has $\mathcal{T}_X \cong [\mathcal{T}_{\opn{FL}}]$.
\end{example}

Suppose that $\mathcal{V}$ is a quasi-coherent $G$-equivariant sheaf on $\opn{FL}$ equipped with a $G$-equivariant Lie algebra action of $\mathcal{T}_{\opn{FL}}$, i.e., a $G$-equivariant Lie algebra morphism
\[
\mathcal{T}_{\opn{FL}} \to \underline{\opn{End}}_k(\mathcal{V})
\]
where the latter is equipped with the commutator bracket. Then we obtain a $G$-equivariant Lie algebra action of $q^*\mathcal{T}_{\opn{FL}}$ on $q^*\mathcal{V}$. Explicitly, given $f \otimes v \in \mathcal{O}_{G_{\opn{dR}}} \otimes_{q^{-1}\mathcal{O}_{\opn{FL}}} q^{-1}\mathcal{T}_{\opn{FL}}$ and $\lambda \otimes \gamma \in \mathcal{O}_{G_{\opn{dR}}} \otimes_{q^{-1}\mathcal{O}_{\opn{FL}}} q^{-1}\mathcal{V}$, the action is given by
\[
(f \otimes v) \cdot (\lambda \otimes \gamma) = s_{\opn{FL}}(f \otimes v)(\lambda) \otimes \gamma + f \lambda \otimes v \cdot \gamma .
\]
This induces a $G$-equivariant Lie algebra action of $p^*\mathcal{T}_X \cong q^*\mathcal{T}_{\opn{FL}}$ on $q^*\mathcal{V}$, and by passing to $G$-invariants, we obtain a Lie algebra action of $\mathcal{T}_X$ on $[\mathcal{V}]$.

\begin{example}
Let $\pi \colon G \to \opn{FL}$ be the natural (right) $\overline{P}$-torsor, and let $\mathcal{V}$ denote the $G$-equivariant sheaf on $\opn{FL}$ given by the bundle associated with the standard representation of $\overline{P}$ (via this torsor). This carries an action of $\mathcal{T}_{\opn{FL}}$ (see \S \ref{ActionsOnFLsubsec}) and the induced action of $\mathcal{T}_X$ on $[\mathcal{V}]$ simply corresponds to the Gauss--Manin connection on the first de Rham homology. 
\end{example}

We will analyse this construction further for a specific choice of $\mathcal{V}$. For the rest of this subsection let $\mathcal{V} = \pi_*\mathcal{O}_G$, where $\pi \colon G \to \opn{FL}$ is the natural $\overline{P}$-torsor. This comes equipped with a $G$-equivariant Lie algebra action of $\mathcal{T}_{\opn{FL}}$ (as explained in \S \ref{ActionsOnFLsubsec}). We consider the following $G$-equivariant sheaves of Lie algebras:
\begin{itemize}
    \item $\mathcal{V} \otimes_{\mathcal{O}_{\opn{FL}}} \mathcal{T}_{\opn{FL}}$ with $G$-equivariant Lie bracket induced by the Leibniz rule, i.e., for $f \otimes v$ and $f' \otimes v'$ elements of $\mathcal{V} \otimes \mathcal{T}_{\opn{FL}}$ the bracket is given 
    \[
    [f \otimes v, f' \otimes v']_{\mathcal{D}} = f (v\cdot f') \otimes v' - (v' \cdot f)f' \otimes v + f f' \otimes [v, v']_{\opn{FL}} .
    \]
    As indicated by the notation, this bracket is induced from a similar bracket $[\cdot, \cdot]_{\mathcal{D}}$ on $\pi_*\mathcal{O}_G \otimes \ide{g}^0$ via the constructions in \S \ref{ActionsOnFLsubsec}. 
    \item $q^*(\mathcal{V} \otimes \mathcal{T}_{\opn{FL}}) = q^*\mathcal{V} \otimes_{\mathcal{O}_{G_{\opn{dR}}}} q^*\mathcal{T}_{\opn{FL}} \cong q^*\mathcal{V} \otimes_{\mathcal{O}_{G_{\opn{dR}}}} p^*\mathcal{T}_{X}$ with $G$-equivariant Lie bracket defined similarly. 
    \item $[\mathcal{V}] \otimes_{\mathcal{O}_X} \mathcal{T}_X$ with Lie bracket induced from taking $G$-invariants of the Lie bracket in the above bullet point.
\end{itemize}

We have a natural action of $\mathcal{V} \otimes \mathcal{T}_{\opn{FL}}$ on $\mathcal{V}$ given by $(f \otimes v) \cdot x = f(v\cdot x)$ for $f \otimes v \in \mathcal{V} \otimes \mathcal{T}_{\opn{FL}}$, which is a $G$-equivariant Lie algebra action. By pulling back along $q$ and taking $G$-invariants, we obtain a Lie algebra action of $[\mathcal{V}]\otimes \mathcal{T}_{X}$ on $[\mathcal{V}]$ extending the action of $\mathcal{T}_X$.

We recall that we are in the following situation. Let $\ide{u}$ denote the Lie algebra of the unipotent radical of $P_{\mbf{G}}$, which comes equipped with the trivial Lie bracket and the trivial action of $G$. Then, as explained in the proof of Lemma \ref{HorzActionOnFL}, we have a $G$-equivariant map of Lie algebras:
\[
\ide{u} \to \left( \mathcal{V} \otimes_{\mathcal{O}_{\opn{FL}}} \mathcal{T}_{\opn{FL}} \right)(\opn{FL})
\]
i.e., commuting $G$-invariant global sections of $\mathcal{V} \otimes_{\mathcal{O}_{\opn{FL}}} \mathcal{T}_{\opn{FL}}$. Then, by pulling back under $q$ and taking $G$-invariants, we obtain a Lie algebra morphism $\ide{u} \to \left( [\mathcal{V}] \otimes \mathcal{T}_X \right)$, and hence a Lie algebra action
\[
\ide{u} \to \opn{End}_k([\mathcal{V}]).
\]
In fact this action map factors through $\opn{Der}_k([\mathcal{V}])$ (the space of derivations $[\mathcal{V}] \to [\mathcal{V}]$), which makes sense because $[\mathcal{V}]$ carries an algebra structure. One can easily show that $[\mathcal{V}] \cong \pi_* \mathcal{O}_{P_{\opn{dR}}}$, where $\pi \colon P_{\opn{dR}} \to X$ denotes the $\overline{P}$-reduction of $G_{\opn{dR}}$ (frames of the first relative de Rham homology preserving the Hodge filtration -- see Definition \ref{DefinitionOfGdrPdR}). We denote this action by $\star_{\ide{u}}$.

Recall $\overline{\ide{p}}$ denotes the Lie algebra of $\overline{P}$. Then $[\mathcal{V}]$ has an $\mathcal{O}_{X}$-linear action of $\overline{\ide{p}}$ given by differentiating the torsor structure, which we denote by $\star_{\overline{P}}$. We have the following analogue of Proposition \ref{KeyPropforHolOnFL}.

\begin{proposition} \label{KeyPropforHolOnX}
Let $\gamma \in \overline{\ide{p}}$, $\delta \in \ide{u}$ and let $[\gamma, \delta] \in \ide{g}$ denote the Lie bracket of $\gamma$ and $\delta$ viewed as elements of $\ide{g}$. Suppose $[\gamma, \delta] \in \overline{\ide{p}}$. Then for any $F \in [\mathcal{V}]$, we have the relation
\[
\gamma \star_{\overline{P}} ( \delta \star_{\ide{u}} F ) = \delta \star_{\ide{u}} ( \gamma \star_{\overline{P}} F ) + [\gamma, \delta] \star_{\overline{P}} F .
\]
\end{proposition}
\begin{proof}
By Proposition \ref{KeyPropforHolOnFL}, the analogous relation holds for $\mathcal{V} = \pi_*\mathcal{O}_G$. Since the action of $\overline{\ide{p}}$ on $\mathcal{V}$ is $G$-equivariant and $\mathcal{O}_{\opn{FL}}$-linear, the result therefore follows by pulling back under $q$ and passing to $G$-invariants.
\end{proof}

It will be useful to iterate this relation.

\begin{corollary} \label{KeyCorForHolOnX}
For $i, j \in \{ 1, \dots, 2n \}$, let $E_{ij} \in \mathfrak{g}$ denote the elementary matrix with $1$ in the $(i, j)$-th place in the $\tau_0$-component, and $0$ outside the $\tau_0$-component. For $j \in \{n+1, \dots, 2n \}$, let $x_j \colon [\mathcal{V}] \to [\mathcal{V}]$ denote the endomorphism given by $E_{1, j} \star_{\ide{u}} -$. Then for any polynomial $p \in k[x_{n+1}, \dots, x_{2n}]$ and any $i \in \{2, \dots, n \}$, one has
\[
E_{i, 1} \star_{\overline{P}} ( p \cdot F ) = p \cdot ( E_{i, 1} \star_{\overline{P}} F ) + \sum_{j = n+1}^{2n} \frac{\partial p}{\partial x_j} \cdot ( E_{i, j} \star_{\overline{P}} F )
\]
for any $F \in [\mathcal{V}]$.
\end{corollary}
\begin{proof}
    By linearity, it is enough to prove this when $p$ is a monomial. Note that $[E_{i,1}, E_{1, j}] = E_{i, j}$ and $[E_{i, j}, E_{1, j'}] = 0$ for any $i \in \{2, \dots, n \}$, $j, j' \in \{n+1, \dots, 2n \}$. The result then follows from Proposition \ref{KeyPropforHolOnX} and a simple induction argument on the degree of $p$, using the Leibniz rule.
\end{proof}

\subsection{Classical nearly holomorphic forms} \label{ClassicalNearlyHolFormsSection}

In this section we introduce the sheaves of nearly holomorphic forms for $G$ and $H$, and describe a classical ``evaluation map'' on the cohomology of these sheaves which encodes the twisted unitary Friedberg--Jacquet periods.

\subsubsection{Nearly holomorphic forms} \label{NearlyHoloFormsSubSec}

Recall from Definition \ref{DefinitionOfGdrPdR} that we have torsors
\[
\pi_G \colon P_{G, \opn{dR}} \to X_{\mbf{G}, \opn{Iw}, F^{\opn{cl}}}(p^{\beta}) \quad \text{ and } \quad \pi_H \colon P_{H, \opn{dR}} \to X_{\mbf{H}, \diamondsuit, F^{\opn{cl}}}(p^{\beta}).
\]
We introduce the following notation:

\begin{notation}
Let $\mathscr{N}_G = \mathscr{O}_{P_{G, \opn{dR}}} \defeq (\pi_G)_*\mathcal{O}_{P_{G, \opn{dR}}}$ and $\mathscr{N}_H = \mathscr{O}_{P_{H, \opn{dR}}} \defeq (\pi_H)_*\mathcal{O}_{P_{H, \opn{dR}}}$, which we refer to as the sheaves of nearly holomorphic forms for $G$ and $H$ respectively. Given an algebraic representation $V$ of $M_{\mbf{G}}$ of highest weight $\kappa$, we let $\mathscr{N}_{G, \kappa} = (\mathscr{N}_G \otimes V)^{M_{\mbf{G}}}$ where the invariants are with respect to the natural diagonal action of $M_{\mbf{G}} \subset P_{\mbf{G}}^{\opn{std}}$. We use similar notation for $\mathscr{N}_H$ with respect to algebraic representations of $M_{\mbf{H}}$.
\end{notation}

\begin{remark}
    The sheaf $\mathscr{N}_{G, \kappa}$ is the quasi-coherent sheaf on $X_{\mbf{G}, \opn{Iw}, F^{\opn{cl}}}(p^{\beta})$ associated with the dual Verma module of weight $\kappa$, via the construction in \S \ref{PassageToShVars}. A similar assertion is true for $H$. 
\end{remark}

The sheaf $\mathscr{N}_G$ (resp. $\mathscr{N}_H$) comes equipped with an action of $\overline{\ide{p}}_G = \opn{Lie}P_{\mbf{G}}^{\opn{std}}$ (resp. $\overline{\ide{p}}_H = \opn{Lie}P_{\mbf{H}}^{\opn{std}}$) given by differentiating the torsor structure. We let $\mathscr{M}_G \subset \mathscr{N}_G$ (resp. $\mathscr{M}_H \subset \mathscr{N}_H$) denote the subsheaf of elements killed by the action of the nilpotent subalgebra $\overline{\ide{u}}_G \subset \overline{\ide{p}}_G$ (resp. $\overline{\ide{u}}_H \subset \overline{\ide{p}}_H$) -- these are the sheaves of holomorphic forms for $G$ and $H$ respectively. We have identifications $\mathscr{M}_G = (\pi'_G)_* \mathcal{O}_{M_{G, \opn{dR}}}$ and $\mathscr{M}_H = (\pi'_H)_* \mathcal{O}_{M_{H, \opn{dR}}}$ where $\pi'_G \colon M_{G, \opn{dR}} \to X_{\mbf{G}, \opn{Iw}, F^{\opn{cl}}}(p^{\beta})$ and $\pi'_H \colon M_{H, \opn{dR}} \to X_{\mbf{H}, \diamondsuit, F^{\opn{cl}}}(p^{\beta})$ denote the torsors in Definition \ref{DefinitionOfGdrPdR}. We will also use the notation $\mathscr{M}_{G, \kappa}$ if we wish to specify the weight, and similarly for $H$.

As explained in \S \ref{DmodulesOnFLsection}, the sheaf $\mathscr{N}_G$ is in fact a $\mathcal{D}_{X_{\mbf{G}, \opn{Iw}, F^{\opn{cl}}}(p^{\beta})}$-module on $X_{\mbf{G}, \opn{Iw}, F^{\opn{cl}}}(p^{\beta})$ and we have an action of $\ide{u}_G$ (the Lie algebra of the unipotent of $P_{\mbf{G}}$) on $\mathscr{N}_G$ through derivations. Recall that $\ide{u}_G \cong \mbb{G}_a^{\oplus 2n-1}$ corresponding to the Lie algebra of the unipotent of the standard $(1, 2n-1)$-parabolic in $\opn{GL}_{2n}$ (using the identification of $\mbf{G}_{F^{\opn{cl}}}$ as a product of general linear groups). Here is a more concrete description of this action:

\begin{lemma} \label{LemmaConcreteDescriptionNablai}
For $i=1, \dots, 2n-1$, let
\[
\nabla_i \colon \mathscr{N}_G \to \mathscr{N}_G 
\]
denote the derivation given by the action of the element $(0, \dots, 0, 1, 0, \dots, 0)$ of $\ide{u}_G = \mbb{G}_a^{\oplus 2n-1}$ with $1$ in the $i$-th place. Then $\nabla_i$ can be described as the composition of:
\begin{itemize}
    \item The $\mathcal{D}$-module structure $\mathscr{N}_G \to \mathscr{N}_G \otimes_{\mathcal{O}_{X_{\mbf{G}, \opn{Iw}, F^{\opn{cl}}}(p^\beta)}} \Omega^1_{X_{\mbf{G}, \opn{Iw}, F^{\opn{cl}}}(p^\beta)}$
    \item The inverse of the Kodaira--Spencer map $\omega_{A^D, \tau_0} \otimes \omega_{A, \tau_0} \xrightarrow{\sim} \Omega^1_{X_{\mbf{G}, \opn{Iw}, F^{\opn{cl}}}(p^\beta)}$, where $A \to X_{\mbf{G}, \opn{Iw}, F^{\opn{cl}}}(p^\beta)$ denotes the universal $\boldsymbol\Psi$-unitary abelian scheme, and $A^D$ its dual 
    \item The adjoint map $\omega_{A^D, \tau_0} \otimes \omega_{A, \tau_0} \to \mathscr{N}_G$ to the morphism
    \[
    \pi_G^* \left( \omega_{A^D, \tau_0} \otimes \omega_{A, \tau_0} \right) \to \mathcal{O}_{P_{G, \opn{dR}}}
    \]
    which is obtained as follows. Over $P_{G,\opn{dR}}$, we have universal trivialisations $\psi_1 \colon \mathcal{O}_{P_{G, \opn{dR}}} \xrightarrow{\sim} \pi_G^* (\omega_{A, \tau_0}^{\vee})$ and $\psi_2 \colon \mathcal{O}_{P_{G, \opn{dR}}}^{\oplus 2n-1} \xrightarrow{\sim} \pi_G^*(\omega_{A^D, \tau_0})$. Let $\pi_G^*( \omega_{A^D, \tau_0}) \otimes \pi_G^* (\omega_{A, \tau_0})  \to \mathcal{O}_{P_{G, \opn{dR}}}$ denote the map $\opn{pr}_i \circ \; \psi_2^{-1} \otimes \psi_1^{\vee}$, where $\opn{pr}_i \colon \mathcal{O}_{P_{G, \opn{dR}}}^{\oplus 2n-1} \to \mathcal{O}_{P_{G, \opn{dR}}}$ denotes projection to the $i$-th component.
    \item The multiplication map $\mathscr{N}_G \otimes \mathscr{N}_G \to \mathscr{N}_G$.
\end{itemize}
We have a similar description for the $\ide{u}_H$ action on $\mathscr{N}_H$; in particular we will use the notation $\nabla_i \colon \mathscr{N}_H \to \mathscr{N}_H$ ($i=1, \dots, n-1$) for the derivations coming from the standard basis vectors of $\ide{u}_H = \mbb{G}_a^{\oplus n-1} \subset \ide{u}_G$.
\end{lemma}
\begin{proof}
It will be enough to prove this on the level of flag varieties. Let $V = V_0 \boxplus (\boxplus_{\tau \in \Psi} V_{\tau})$ denote the standard representation of $G = \opn{GL}_1 \times \prod_{\tau} \opn{GL}_n$, and let $e_0$, $\{ e_{\tau, 1}, \dots, e_{\tau, 2n} \}$ denote the standard bases of $V_0$, $V_{\tau}$ respectively. Let $L_{\tau_{0}} = \langle e_{\tau_0, 2}, \dots, e_{\tau_0, 2n}\rangle \subset V_{\tau_0}$ and $L_{\tau_0}' = V_{\tau_0}/L_{\tau_0}$, both of which are representations of $\overline{P}_G$ (through the projection to the $\tau_0$-component). Let $G$ act on $V_{\tau_0}$ through its projection to the $\tau_0$-component. Let $\mathcal{V}_{\tau_0}$, $\mathcal{L}_{\tau_0}$, and $\mathcal{L}_{\tau_0}'$ denote the corresponding $G$-equivariant bundles on $\opn{FL}$. Note that $\Omega^1_{\opn{FL}}$ is the $G$-equivariant bundle associated with the $\overline{P}_G$-representation $\left( \ide{g}/\overline{\ide{p}}_G \right)^{\vee}$.

Let $\{ E_0 \}$, $\{ E_{\tau, i, j} : i,j =1, \dots, 2n \}$ denote the standard bases of the similitude and $\tau$-factors of $\ide{g}$ respectively. Let $U \subset \opn{FL}$ be an open such that we have a section $s \colon U \to \pi^{-1}U$ of the natural torsor $\pi \colon G \to \opn{FL}$. Let $i \in \{1, \dots, 2n-1\}$, and consider the following function:
\begin{align*}
\mu_i \colon \pi^{-1}U &\to \ide{g}^{\vee} \\ 
g &\mapsto \left( X \mapsto - E_{\tau_0, 1, i+1}^*(\opn{Ad}(g^{-1})X) \right) ,
\end{align*}
where $E_{\tau_0, 1, i+1}^*$ denotes the dual basis element. Note that $\mu_i(g)$ is trivial on $\opn{Ad}(g)(\overline{\ide{p}}_G)$. Let $\nu_i \colon \pi^{-1}U \to \ide{g}^{\vee}$ denote the function $\nu_i(g) = g^{-1} \cdot \mu_i(s(\pi(g)))$. Then $\nu_i(g)$ is trivial on $\overline{\ide{p}}_G$, and defines a function $\nu_i \colon \pi^{-1}U \to \left( \ide{g}/\overline{\ide{p}}_G \right)^{\vee}$ which satisfies $\nu_i(- \cdot p) = p^{-1} \nu_i(-)$ for $p \in \overline{P}_G$ (so $\nu_i \in \Omega^1_{\opn{FL}}(U)$). One can show that $\{ \nu_i : i=1, \dots, 2n-1 \}$ is a basis of $\Omega^{1}_{\opn{FL}}(U)$. Let $\lambda_i \colon \pi^{-1}U \to \ide{g}$ denote the function $\lambda_i(g) = -\opn{Ad}(g) E_{\tau_0, 1, i+1}$, and let $\partial_i \colon \pi^{-1}U \to \ide{g}/\overline{\ide{p}}_G$ denote the function where $\partial_i(g)$ denotes the image of $g^{-1} \cdot \lambda_i(s(\pi(g)))$ in the quotient $\ide{g}/\overline{\ide{p}}_G$. Then $\{ \partial_i : i=1, \dots, 2n-1\}$ forms a basis of $\mathcal{T}_{\opn{FL}}(U)$ which is dual to the basis $\nu_i$. 

As explained in Definition \ref{Def:StarDaction}, one has a connection $\nabla$ on $\mathcal{V}_{\tau_0}$, which induces an isomorphism:
\begin{equation} \label{FLKSiso}
\mathcal{L}_{\tau_0} \otimes (\mathcal{L}'_{\tau_0})^{\vee} \xrightarrow{\sim} \Omega^1_{\opn{FL}} 
\end{equation}
given by $f \otimes z \mapsto (z \otimes 1)(\overline{\nabla(f)})$, where $\overline{\nabla(f)}$ denotes the image of $\nabla(f)$ in $\mathcal{L}'_{\tau_0} \otimes \Omega^1_{\opn{FL}}$. This is the Kodaira--Spencer isomorphism on the level of flag varieties. Let $\epsilon_{\tau_0, i} \colon \pi^{-1}U \to V_{\tau_0}$ denote the function given by $\epsilon_{\tau_0, i}(g) = g^{-1} s(\pi(g)) \cdot e_{\tau_0, i}$ (so $\{\epsilon_{\tau_0, 2}, \dots, \epsilon_{\tau_0, 2n} \}$ is a basis of $\mathcal{L}_{\tau_0}(U)$ and $\epsilon_{\tau_0, 1}$ gives a basis of $\mathcal{L}_{\tau_0}'$). Then, for $j \in \{2, \dots, 2n\}$,
\begin{align*}
((\lambda_i \circ s) \star_{\mathcal{D}} \epsilon_{\tau_0, j})(s(\pi(g))) &\equiv E_{\tau_0, 1, i+1} \star_2 \epsilon_{\tau_0, j}(s\pi(g)) \\ &= E_{\tau_0, 1, i+1} \star_2 e_{\tau_0, j} \\ &= \left\{ \begin{array}{cc} e_{\tau_0, 1} & \text{ if } j=i+1 \\ 0 & \text{ otherwise} \end{array} \right.
\end{align*}
modulo $L_{\tau_0}$. Hence $\epsilon_{\tau_0, i+1} \otimes \epsilon_{\tau_0, 1}^*$ is sent to $\nu_i$ under the isomorphism (\ref{FLKSiso}). This implies that the composition of the second and third bullet points in the statement of the lemma (on the level of flag varieties) is the morphism
\[
\Omega^{1}_{\opn{FL}}(U) \to \mathcal{O}_G(\pi^{-1}U)
\]
which sends $\nu_i$ to $1$ and $\nu_j$ to zero if $j \neq i$.

This implies that the composition $\mathscr{N}_G \to \mathscr{N}_G$ of all of the bullet points in the statement of the lemma is identified, on the level of flag varieties and over $U$, with the map 
\[
\lambda_i \tilde{\star}_{\mathcal{D}} - \colon \pi_*\mathcal{O}_G \to \pi_*\mathcal{O}_G
\]
which is precisely $\nabla_i$ (on the level of flag varieties) by definition.
\end{proof}

Let $k = F^{\opn{cl}}$ and let $C^{\opn{pol}}(k^{\oplus 2n-1}, k)$ denote the space of polynomial functions $k^{\oplus 2n-1} \to k$, which we equip with an action of $M_{\mbf{G}}$ via the formula:
\begin{equation} \label{ActionOnCpolEqn}
(m \cdot \phi)(a) = \phi( m_{\tau_0}^{-1} \tbyt{1}{}{a^t}{1} m_{\tau_0} ), \quad \phi \in C^{\opn{pol}}(k^{\oplus 2n-1}, k), m \in M_{\mbf{G}}.
\end{equation}
Here we view the row vector $a \in k^{\oplus 2n-1}$ as a lower triangular matrix $\tbyt{1}{}{a^t}{1}$ in the lower triangular unipotent of the $(1, 2n-1)$-parabolic of $\opn{GL}_{2n}$, and $m_{\tau_0}$ is the projection of $m$ to the $\tau_0$-component of $M_{\mbf{G}}$. Note that $C^{\opn{pol}}(k^{\oplus 2n-1}, k)$ is naturally isomorphic to the universal enveloping algebra of $\ide{u}_G$ equipped with the adjoint action of $M_{\mbf{G}}$. We therefore obtain an action map
\begin{equation} \label{ActionMapForScr}
C^{\opn{pol}}(k^{\oplus 2n-1}, k) \otimes \mathscr{N}_G \to \mathscr{N}_G
\end{equation}
which one can verify is $M_{\mbf{G}}$-equivariant by the explicit description of the action of $\ide{u}_G$ in \S \ref{DmodulesOnFLsection}. 

\subsubsection{Classical evaluations maps} \label{ClassicalEvaluationMapsSubSec}

Let $k = F^{\opn{cl}}$ for ease of notation. We now describe the construction of a linear functional on the cohomology of $X_{\mbf{G}, \opn{Iw}, k}(p^{\beta})$ which will be shown to recover unitary Friedberg--Jacquet periods in the following subsection.

The first step is to apply the general branching law construction in \S \ref{TheMainConstructionSubSec} to construct a map of sheaves from nearly holomorphic forms for $\mbf{G}$ to nearly holomorphic forms for $\mbf{H}$ incorporating the action of certain differential operators. Recall from \S \ref{DeeperLevelAtpSubSec} that we have a finite unramified morphism:
\[
\hat{\iota} \colon X_{\mbf{H}, \diamondsuit, k}(p^{\beta}) \to X_{\mbf{G}, \opn{Iw}, k}(p^{\beta}) .
\]
\begin{lemma} \label{ClassicalReductionOfStructuresOfPdRLemma}
    The natural $P_{\mbf{H}}^{\opn{std}}$-equivariant map $P_{H, \opn{dR}} \to P_{G, \opn{dR}}$ given by sending a trivialisation of $\mathcal{H}_{A_1} \oplus \mathcal{H}_{A_2}$ to the induced trivialisation of $\mathcal{H}_{A_1 \oplus A_2}$ gives a reduction of structure
    \[
    \hat{\iota}^* P_{G, \opn{dR}} = P_{H, \opn{dR}} \times^{P_{\mbf{H}}^{\opn{std}}} P_{\mbf{G}}^{\opn{std}} .
    \]
    Hence we obtain a natural $P_{\mbf{H}}^{\opn{std}}$-equivariant map of sheaves $\mathscr{N}_G \to \hat{\iota}_* \mathscr{N}_H$.
\end{lemma}
\begin{proof}
    Immediate from the definitions.
\end{proof}

We now construct the aforementioned map of sheaves. Recall the definition of $\mathcal{E}$ from \S \ref{ClassicalBranchingLawsSubSubSec}, i.e., pairs $(\kappa, j)$ where $\kappa \in X^*(T)$ is a $M_{\mbf{G}}$-dominant weight that is ``pure of non-positive weight'' and $j = (j_{\tau})_{\tau \in \Psi}$ is a tuple of integers, such that $\kappa$ and $j$ satisfy a certain intertwining property. For $(\kappa, j) \in \mathcal{E}$, let $V_{\kappa}$ denote the algebraic representation of $M_{\mbf{G}}$ of highest weight $\kappa$, and recall that we have an eigenvector $\delta_{\kappa, j} \in V_{\kappa} \otimes_k C^{\opn{pol}}(k^{\oplus 2n-1}, k)$ for the diagonal action of $M_{\mbf{H}}$ with eigencharacter $\sigma_{\kappa}^{[j], -1}$ (see Definition \ref{DefinitionOfDeltakappaj}).

\begin{definition} \label{DefinitionOfScrPullbackNGNH}
    Let $(\kappa, j) \in \mathcal{E}$ and $\beta \geq 1$. Then we define a $k$-linear map
    \[
    \vartheta_{\kappa, j, \beta} \colon \mathscr{N}_{G, \kappa^*} \to \hat{\iota}_* \mathscr{N}_{H, \sigma_{\kappa}^{[j]}}
    \]
    of abelian sheaves over $X_{\mbf{G}, \opn{Iw}, k}(p^{\beta})$ as the $M_{\mbf{H}}$-invariants of the map
    \[
    \sigma_{\kappa}^{[j], -1} \otimes \mathscr{N}_G \otimes V_{\kappa}^* \to \hat{\iota}_* \mathscr{N}_H
    \]
    arising as the composition of the following $M_{\mbf{H}}$-equivariant maps:
    \begin{itemize}
    \item The morphism $\sigma_{\kappa}^{[j], -1} \otimes \mathscr{N}_G \otimes V_{\kappa}^* \to V_{\kappa} \otimes C^{\opn{pol}}(k^{\oplus 2n-1}, k) \otimes \mathscr{N}_G \otimes V_{\kappa}^*$ induced from sending the first factor to $\delta_{\kappa, j}$.
    \item The morphism $V_{\kappa} \otimes C^{\opn{pol}}(k^{\oplus 2n-1}, k) \otimes \mathscr{N}_G \otimes V_{\kappa}^* \to V_{\kappa} \otimes \mathscr{N}_G \otimes V_{\kappa}^*$ induced from the action map in (\ref{ActionMapForScr}).
    \item The morphism $V_{\kappa} \otimes \mathscr{N}_G \otimes V_{\kappa}^* \to \mathscr{N}_G$ induced from the natural map $V_{\kappa} \otimes V^*_{\kappa} \to k$.
    \item The morphism $\mathscr{N}_G \to \hat{\iota}_*\mathscr{N}_H$.
\end{itemize}
\end{definition}

We have the following result which crucially uses the property in Corollary \ref{KeyCorForHolOnX}.

\begin{proposition} \label{ClassicalThetaPreservesHolProp}
Let $(\kappa, j) \in \mathcal{E}$ and $\beta \geq 1$. Then the morphism $\vartheta_{\kappa, j, \beta}$ induces a morphism
\[
\vartheta_{\kappa, j, \beta} \colon \mathscr{M}_{G, \kappa^*} \to \hat{\iota}_* \mathscr{M}_{H, \sigma_{\kappa}^{[j]}} ,
\]
i.e., the morphism takes holomorphic forms for $\mbf{G}$ to holomorphic forms for $\mbf{H}$.
\end{proposition}
\begin{proof}
    We can (and do) assume that $j_{\tau_0} > 0$, otherwise there is nothing to prove. We first note the following property. Let $E_{i,k} \in \ide{gl}_{2n}$ denote the elementary matrix with $1$ in the $(i, k)$-th position and $0$ elsewhere. We can view $E_{i, k} \in \ide{g}$ as elements only concentrated in the $\tau_0$-component. 
    
    If $x_i \in C^{\opn{pol}}(k^{\oplus 2n-1}, k)$ is the polynomial function sending $(a_2, \dots, a_{2n}) \mapsto a_i$, then the action of $x_i$ on $\mathscr{N}_G$ corresponds to the action of $E_{1, i} \in \ide{u}_G$ under $\star_{\ide{u}}$. Note that we have $E_{i, 1} \in \overline{\ide{u}}_H$ for $i=2, \dots, n$. We have the commutator relations
    \[
    [E_{i, 1}, E_{1, k}] = E_{i, k} \in \opn{Lie}M_{\mbf{G}} \quad \quad \text{ for } i \in \{2, \dots, n\} \text{ and } k \in \{ n+1, \dots, 2n \}
    \]
    and
    \[
    [E_{i, k}, E_{1, k'}] = 0 \quad \quad \text{ for } i \in \{ 2, \dots, n \} \text{ and } k, k' \in \{ n+1, \dots, 2n \} .
    \]
    Therefore, by Corollary \ref{KeyCorForHolOnX}, we find that for any polynomial $p \in C^{\opn{pol}}(k^{\oplus 2n-1}, k)$ in the coordinates $x_{n+1}, \dots, x_{2n}$, any $i=2, \dots, n$ and any $F \in \mathscr{N}_G$, we have 
    \begin{equation} \label{EstarpstarFHolomorphicCase}
    E_{i, 1} \star_{\overline{P}} (p \star_{\ide{u}} F) = p \star_{\ide{u}} ( E_{i, 1} \star_{\overline{P}} F) + \sum_{k = n+1}^{2n} \frac{\partial p}{\partial x_k} \star_{\ide{u}} (E_{i, k} \star_{\overline{P}} F)
    \end{equation} 
    where we have also used the notation $\star_{\ide{u}}$ to denote the action map in (\ref{ActionMapForScr}).
    
    Now suppose that we have an element $F \in (\mathscr{M}_{G} \otimes V_{\kappa}^* )^{M_{\mbf{G}}} \subset (\mathscr{N}_{G} \otimes V_{\kappa}^* )^{M_{\mbf{G}}}$. We can also view this as a $M_{\mbf{G}}$-equivariant map $\tilde{F} \colon V_{\kappa} \to \mathscr{N}_G$. For any finite multiset $T \subset \{ 2, \dots, 2n \}$, we let $x^T = \prod_{i \in T} x_i \in C^{\opn{pol}}(k^{\oplus 2n-1}, k)$. These form a $k$-basis of $C^{\opn{pol}}(k^{\oplus 2n-1}, k)$. Let $\{ v_l \}$ denote a fixed basis for $V_{\kappa}$ and $\{ v_l^* \} \subset V_{\kappa}^*$ the dual basis. We note that $\delta_{\kappa, j}$ can be expressed as a linear combination
    \[
    \delta_{\kappa, j} = \sum_l \sum_{\substack{T \subset \{n+1, \dots, 2n\} \\ \#T = j_{\tau_0}}} \lambda_{l, T} (v_l \otimes x^T)
    \]
    for some $\lambda_{l, T} \in k$.
    
    Let $\vartheta_{\kappa, j, \beta}' \colon \mathscr{N}_G \otimes V_{\kappa}^* \to \mathscr{N}_G$ denote the composition of the maps in Definition \ref{DefinitionOfScrPullbackNGNH}, except for the final bullet point (we omit the notation for the character $\sigma_{\kappa}^{[j], -1}$). Then we have $h \cdot \vartheta_{\kappa, j, \beta}'(F) = \sigma_{\kappa}^{[j], -1}(h) \vartheta_{\kappa, j, \beta}'(F)$ for all $h \in M_{\mbf{H}}$. Since the map $\mathscr{N}_G \to \hat{\iota}_* \mathscr{N}_H$ is $P_{\mbf{H}}^{\opn{std}}$-equivariant, it is enough to show that $E_{i, 1} \star_{\overline{P}} \vartheta'_{\kappa, j, \beta}(F) = 0$ for all $i \in \{2, \dots, n \}$.
    
    Now if we write $F = \sum_l F_l \otimes v_l^*$ for some (unique) $F_l \in \mathscr{M}_G \subset \mathscr{N}_G$ (i.e., $F_l = \tilde{F}(v_l)$), then we have the formulae:
    \begin{align*} 
    \vartheta_{\kappa, j, \beta}'(F) &= \sum_l \sum_{\substack{T \subset \{n+1, \dots, 2n\} \\ \# T = j_{\tau_0}}} \lambda_{l, T} ( x^T \star_{\ide{u}} F_l )  \\
    E_{i, 1} \star_{\overline{P}} \vartheta'_{\kappa, j, \beta}(F) &=  \sum_l \sum_{\substack{T \subset \{n+1, \dots, 2n\} \\ \# T = j_{\tau_0}}} \lambda_{l, T} \left( \sum_{k=n+1}^{2n} \frac{\partial x^T}{\partial x_k} \star_{\ide{u}} (E_{i, k} \star_{\overline{P}} F_l ) \right)
    \end{align*}
    using the fact that $E_{i, 1} \star_{\overline{P}} F_l = 0$ because $F_l \in \mathscr{M}_G$. 
    
    With notation as in Theorem \ref{TheoremForClassicalBranching} and Convention \ref{ConventionForPhiBeta}, we have a $M_{\mbf{H}}$-equivariant map 
    \[
    g \colon S_{-(j-1)} \otimes V_{\kappa} \to \mathscr{N}_G
    \]
    given by $g( y \otimes z) = \Phi(y) \star_{\ide{u}} \tilde{F}(z)$ on pure tensors. Since $\frac{\partial x^T}{\partial x_k}$ is either zero or a monomial of degree $j_{\tau_0} - 1$, there exist unique elements $w_{k, T} \in S_{-(j-1)}$ such that $\Phi(w_{k, T}) = \frac{\partial x^T}{\partial x_k}$. We consider the elements $A_i \in S_{-(j-1)} \otimes V_{\kappa}$ given by
    \[
    A_i =  \sum_l \sum_{\substack{T \subset \{n+1, \dots, 2n\} \\ \# T = j_{\tau_0}}} \lambda_{l, T} \left( \sum_{k=n+1}^{2n} w_{k, T} \otimes (E_{i, k} \star_{\ide{m}_G} v_l ) \right)
    \]
    and we see that $g(A_i) = E_{i, 1} \star_{\overline{P}} \vartheta'_{\kappa, j, \beta}(F)$ (because $E_{i, k} \in \ide{m}_G = \opn{Lie}M_{\mbf{G}}$ and $\tilde{F}$ is $M_{\mbf{G}}$-equivariant).
    
    On the other hand, let $W$ denote the algebraic representation of $M_{\mbf{H}}$ with trivial highest weight outside the $\tau_0$-component, and weight $(-1, 1, 0, \dots, 0 )$ in the $\tau_0$-component. If $\{ e_1, \dots, e_{2n} \}$ denotes the standard basis of the standard representation of $\opn{GL}_{2n}$ (seen as the $\tau_0$-component of $\mbf{G}$), we have a $M_{\mbf{H}}$-equivariant map
    \begin{align*}
        h \colon W \otimes \sigma_{\kappa}^{[j], -1} &\to \mathscr{N}_G \\
        e_1^* e_i &\mapsto E_{i, 1} \star_{\overline{P}} \vartheta'_{\kappa, j, \beta}(F) .
    \end{align*}
    We define a $k$-linear map $q \colon W \otimes \sigma_{\kappa}^{[j], -1} \to S_{-(j-1)} \otimes V_{\kappa}$ sending $e_1^* e_i$ to $A_i$. Then clearly we have $g \circ q = h$ and the induced map 
    \[
    \bar{q} \colon W \otimes \sigma_{\kappa}^{[j], -1} \to (S_{-(j-1)} \otimes V_{\kappa}) / \opn{ker}(g)
    \]
    is $M_{\mbf{H}}$-equivariant. Since the category of algebraic representations of $M_{\mbf{H}}$ is semisimple (as the group is reductive), this gives rise to a $M_{\mbf{H}}$-equivariant morphism $W \otimes \sigma_{\kappa}^{[j], -1} \to S_{-(j-1)} \otimes V_{\kappa}$ -- if we show that any such morphism has to be zero, this will imply that $E_{i, 1} \star_{\overline{P}} \vartheta'_{\kappa, j, \beta}(F) = 0$, as required.
    
    Restricting to the $\tau_0$-component and considering the subgroup $1 \times 1 \times \opn{GL}_n \subset (M_{\mbf{H}})_{\tau_0}$, it is enough to show that
    \[
    \opn{Hom}_{1 \times 1 \times \opn{GL}_n} ( \sigma_{\kappa}^{[j], -1}, S_{-(j-1)} \otimes V_{\kappa} ) = 0
    \]
    since $W$ is the trivial representation when restricted to this group. But we have already shown this in the proof of Theorem \ref{TheoremForClassicalBranching}.
\end{proof}

By passing to cohomology and restricting along the open and closed embeddings in Lemma \ref{LemmaSDintoPEL}, we obtain a $k$-linear map
\[
\vartheta_{\kappa, j, \beta} \colon \opn{H}^{n-1}\left( S_{\mbf{G}, \opn{Iw}}(p^{\beta}), \mathscr{M}_{G, \kappa^*} \right) \to \opn{H}^{n-1}\left( S_{\mbf{H}, \diamondsuit}(p^{\beta}), \mathscr{M}_{H, \sigma_{\kappa}^{[j]}} \right) .
\]
The following proposition describes the compatibility of these maps as one varies $\beta$. For any $\beta \geq 1$, let $p_{G, \beta} \colon S_{\mbf{G}, \opn{Iw}}(p^{\beta+1}) \to S_{\mbf{G}, \opn{Iw}}(p^{\beta})$ and $p_{H, \beta} \colon S_{\mbf{H}, \diamondsuit}(p^{\beta+1}) \to S_{\mbf{H}, \diamondsuit}(p^{\beta})$ denote the natural finite \'{e}tale maps.

\begin{proposition} \label{TraceCompatibilityOfThetaProp}
Let $(\kappa, j) \in \mathcal{E}$. Then we have a commutative diagram:
\[
\begin{tikzcd}
{\opn{H}^{n-1}\left( S_{\mbf{G}, \opn{Iw}}(p^{\beta + 1}), \mathscr{M}_{G, \kappa^*} \right)} \arrow[d, "{\opn{Tr}_{p_{G, \beta}}}"'] \arrow[r, "{\vartheta_{\kappa, j, \beta + 1}}"] & {\opn{H}^{n-1}\left( S_{\mbf{H}, \diamondsuit}(p^{\beta + 1}), \mathscr{M}_{H, \sigma_{\kappa}^{[j]}} \right)} \arrow[d, "{\opn{Tr}_{p_{H, \beta}}}"] \\
{\opn{H}^{n-1}\left( S_{\mbf{G}, \opn{Iw}}(p^{\beta}), \mathscr{M}_{G, \kappa^*} \right)} \arrow[r, "{\vartheta_{\kappa, j, \beta}}"]                                                 & {\opn{H}^{n-1}\left( S_{\mbf{H}, \diamondsuit}(p^{\beta}), \mathscr{M}_{H, \sigma_{\kappa}^{[j]}} \right)}                                           
\end{tikzcd}
\]
where $\opn{Tr}_{\bullet}$ denotes the trace maps associated with the corresponding finite \'{e}tale morphisms. 
\end{proposition}
\begin{proof}
    In this proof only, we will add the subscript $\beta$ or $\beta + 1$ to the objects to indicate which level we are working at. Note that $p_{G, \beta}^* \mathscr{N}_{G, \beta} = \mathscr{N}_{G, \beta}$ and similarly for $\mbf{H}$, so the trace maps in the statement of the proposition do indeed exist. Furthermore, it is enough to prove the statements over the varieties $X_{\mbf{G}, \opn{Iw}, F^{\opn{cl}}}(p^{\beta})$ and $X_{\mbf{H}, \diamondsuit, F^{\opn{cl}}}(p^{\beta})$ (we will continue to use the notation $p_{G, \beta}$ and $p_{H, \beta}$ for the natural maps. Recall that the diagram 
    \[
\begin{tikzcd}
{X_{\mbf{H}, \diamondsuit, k}(p^{\beta + 1})} \arrow[d, "{p_{H, \beta}}"'] \arrow[r, "\hat{\iota}_{\beta + 1}"] & {X_{\mbf{G}, \opn{Iw}, k}(p^{\beta + 1})} \arrow[d, "{p_{G, \beta}}"] \\
{X_{\mbf{H}, \diamondsuit, k}(p^{\beta})} \arrow[r, "\hat{\iota}_{\beta}"]                                      & {X_{\mbf{G}, \opn{Iw}, k}(p^{\beta})}                                
\end{tikzcd}
    \]
    is Cartesian  because $p_{H, \beta}$ and $p_{G, \beta}$ have the same degree -- see \cite[Lemma 2.5.3]{UFJ}.
    
    The claim in the proposition is a local one. Let $U \subset X_{\mbf{G}, \opn{Iw}, k}(p^{\beta})$ be an open, and let $V = p_{H, \beta}^{-1} \hat{\iota}_{\beta}^{-1} U = \hat{\iota}_{\beta + 1}^{-1} p_{G, \beta}^{-1} U$. We first note that the following diagram is commutative because the above square is Cartesian:
    \[
\begin{tikzcd}
{\mathscr{N}_{G, \beta+1}(p_{G, \beta}^{-1}U)} \arrow[r, "\hat{\iota}_{\beta + 1}^*"] \arrow[d, "{\opn{Tr}_{p_{G, \beta}}}"'] & {\mathscr{N}_{H, \beta + 1}(V)} \arrow[d, "{\opn{Tr}_{p_{H, \beta}}}"] \\
{\mathscr{N}_{G, \beta}(U)} \arrow[r, "\hat{\iota}_{\beta}^*"]                                                                & {\mathscr{N}_{H, \beta}(\hat{\iota}_{\beta}^{-1}U)}                   
\end{tikzcd}
\]
Furthermore, we claim that for any $\phi \in C^{\opn{pol}}(k^{\oplus 2n-1}, k)$, the following diagram is commutative:
\[
\begin{tikzcd}
{\mathscr{N}_{G, \beta+1}(p_{G, \beta}^{-1}U)} \arrow[r, "\phi \star_{\ide{u}} -"] \arrow[d, "{\opn{Tr}_{p_{G, \beta}}}"'] & {\mathscr{N}_{G, \beta+1}(p_{G, \beta}^{-1}U)} \arrow[d, "{\opn{Tr}_{p_{G, \beta}}}"] \\
{\mathscr{N}_{G, \beta}(U)} \arrow[r, "\phi \star_{\ide{u}} -"]                                                            & {\mathscr{N}_{G, \beta}(U)}                                                          
\end{tikzcd}
\]
Indeed, it is enough to check this statement after pulling back to $G_{\opn{dR}}$, and then the claim follows from the description of $\nabla_i$ in Lemma \ref{LemmaConcreteDescriptionNablai}, the fact that the morphisms $p_{G, \beta}$ and $p_{H, \beta}$ are finite \'{e}tale, and the fact that the sheaves $\mathscr{N}_{G, \bullet}$ arise from the same $G$-equivariant sheaf over the flag variety. The remainder of the proposition now follows from unwinding the definition of $\vartheta_{\kappa, j, \beta}$ in Definition \ref{DefinitionOfScrPullbackNGNH}. 
\end{proof}

We are now in a position to define the classical evaluation maps. We first recall, from \cite[\S 7]{UFJ}, how to associate cohomology classes to anticyclotomic algebraic Hecke characters. Let $\opn{Res}_{F^+/\mbb{Q}} \opn{U}(1)$ denote the unitary group associated with the one dimensional Hermitian space over $F$. Then we have a morphism
\begin{align*}
    \nu \colon \mbf{H} &\to \opn{Res}_{F^+/\mbb{Q}} \opn{U}(1) \\
    (h_1, h_2) &\mapsto \opn{det}h_2/\opn{det}h_1
\end{align*}
which is open and surjective on $\mbb{A}_f$-points. We also have a natural map $\mathcal{N} \colon \opn{Res}_{F/\mbb{Q}} \mbb{G}_m \to \opn{Res}_{F^+/\mbb{Q}} \opn{U}(1)$ given by $\mathcal{N}(z) = \bar{z}/z$. 

\begin{definition} \label{DefinitionOfMathfrakNbeta}
    Let $K_{\beta} \subset \mbf{H}(\mbb{A}_f)$ denote the level of the Shimura--Deligne variety $S_{\mbf{H}, \diamondsuit}(p^{\beta})$. We let $\ide{N}_{\beta}$ denote the smallest ideal of $\mathcal{O}_F$ such that $\nu(K_{\beta}) \subset \mathcal{N}((\widehat{\mathcal{O}}_{F^+} + \ide{N}_{\beta}\widehat{\mathcal{O}}_F)^{\times})$. Note that $\ide{N}_{\beta} = \ide{N}p^{\beta}$ for some ideal $\ide{N} \subset \mathcal{O}_{F^+}$ prime to $p$.
\end{definition}

We now introduce the space of anticyclotomic characters that we are interested in. For this, we need to make an additional assumption on the pair of weights $(\kappa, j)$.

\begin{assumption} \label{AssumpOnKJforACchar}
    Let $(\kappa, j) \in \mathcal{E}$. We assume that $\kappa_0 = 0$, $\kappa_{1, \tau_0} + \kappa_{n+1, \tau_0} = n-1$, and $w \defeq \kappa_{2, \tau_0} + \kappa_{2n, \tau_0}  = -1$.
\end{assumption}

For any $(\kappa, j) \in \mathcal{E}$ satisfying Assumption \ref{AssumpOnKJforACchar}, let $\sigma_{\kappa}^{[j], \vee} = - w_{M_H}^{\opn{max}}\sigma_{\kappa}^{[j]} - 2\rho_{H, \opn{nc}}$ denote the Serre dual of $\sigma_{\kappa}^{[j]}$. By Assumption \ref{AssumpOnKJforACchar}, $\sigma_{\kappa}^{[j], \vee}$ extends to a character of $\mbf{H}$.

\begin{definition} \label{DefOfSigmaKJNB}
    Let $(\kappa, j) \in \mathcal{E}$ satisfying Assumption \ref{AssumpOnKJforACchar}. We let $\Sigma_{\kappa, j}(\ide{N}_{\beta})$ denote the set of algebraic Hecke characters $\chi \colon \mbb{A}_F^{\times}/F^{\times} \to \mbb{C}^{\times}$ satisfying:
    \begin{enumerate}
        \item $\chi$ is anticyclotomic, i.e., $\chi$ is trivial on $\mbb{A}^{\times}_{F^+}$.
        \item The infinity type of $\chi$ is $(j_{\tau_0} - \kappa_{n+1, \tau_0}, -(j_{\tau_0} - \kappa_{n+1, \tau_0}))$ in the $\tau_0$-component and $(j_{\tau}, -j_{\tau})$ for $\tau \neq \tau_0$, i.e.,
\[
\chi( z ) = z_{\tau_0}^{\kappa_{n+1, \tau_0} - j_{\tau_0}} \bar{z}_{\tau_0}^{j_{\tau_0} - \kappa_{n+1, \tau_0}} \cdot \prod_{\tau \in \Psi - \{\tau_0\}} z_{\tau}^{-j_{\tau}} \bar{z}_{\tau}^{j_{\tau}} 
\]
for all $z = (z_{\tau}) \in (\mbb{R} \otimes_{\mbb{Q}} F)^{\times} = \prod_{\tau \in \Psi} \mbb{C}^{\times}$.
\item The conductor of $\chi$ divides $\ide{N}_{\beta}$.
    \end{enumerate}
    We let $\chi' \colon (\opn{Res}_{F^+/\mbb{Q}}\opn{U}(1))(\mbb{Q}) \backslash (\opn{Res}_{F^+/\mbb{Q}}\opn{U}(1))(\mbb{A}) \to \mbb{C}^{\times}$ denote unique character such that $\chi = \chi' \circ \mathcal{N}$.
\end{definition}

For any $\chi \in \Sigma_{\kappa, j}(\ide{N}_{\beta})$ let $F^{\opn{cl}}(\chi)$ denote the smallest number field containing $F^{\opn{cl}}$ and over which the finite part of $\chi$ is defined. Then, by the discussion in \cite[\S 7]{UFJ}, one can associate a cohomology class
\[
[\chi] \in \opn{H}^0\left( S_{\mbf{H}, \diamondsuit}(p^{\beta})_{F^{\opn{cl}}(\chi)}, \mathscr{M}_{H, \sigma_{\kappa}^{[j], \vee}} \right)
\]
which, after base-changing to $\mbb{C}$, coincides with the function
\begin{align*}
H_{\opn{dR}}(\mbb{C}) = \mbf{H}(\mbb{Q}) \backslash X \times \mbf{H}(\mbb{C}) \times \mbf{H}(\mbb{A}_f)/K_{\beta} &\to \mbb{C} \\
[x, h, h'] &\mapsto \xi^{[j]}_{\kappa}(h) \chi'(\nu(h'))
\end{align*}
where $X$ denotes the Hermitian space associated with the Shimura--Deligne datum $(\mbf{H}, h_{\mbf{H}})$, and $\xi_{\kappa}^{[j]} \colon \mbf{H}(\mbb{C}) = \mbb{C}^{\times} \times \prod_{\tau \in \Psi} \left( \opn{GL}_n(\mbb{C}) \times \opn{GL}_n(\mbb{C}) \right) \to \mbb{C}^{\times}$ denotes the character 
\[
(h_0; h_{1, \tau}, h_{2, \tau}) \mapsto (\opn{det}h_{2, \tau_0}/\opn{det}h_{1, \tau_0})^{j_{\tau_0} - \kappa_{n+1, \tau_0}}\prod_{\tau \neq \tau_0}(\opn{det}h_{2, \tau}/\opn{det}h_{1, \tau})^{j_{\tau}} .
\]

\begin{definition} \label{ClassicalEvMapDef}
    Let $(\kappa, j) \in \mathcal{E}$ satisfying Assumption \ref{AssumpOnKJforACchar}, and suppose that $\chi \in \Sigma_{\kappa, j}(\ide{N}_{\beta})$. For any field extension $L/F^{\opn{cl}}$ contained in $\mbb{C}$ (resp. $\Qpb$), let 
    \[
    \opn{Ev}_{\kappa, j, \chi, \beta} \colon \opn{H}^{n-1}\left( S_{\mbf{G}, \opn{Iw}}(p^{\beta})_L, \mathscr{M}_{G, \kappa^*} \right) \to L(\chi)
    \]
    denote the $L$-linear morphism given by $\opn{Ev}_{\kappa, j, \chi, \beta}(\eta) = \langle \vartheta_{\kappa, j, \beta}(\eta), [\chi] \rangle$, where $L(\chi)$ denotes the compositum of $L$ and $F^{\opn{cl}}(\chi)$ in $\mbb{C}$ (resp. in $\Qpb$) and $\langle \cdot, \cdot \rangle$ denotes the pairing given by Serre duality.
\end{definition}

We have the following compatibility with varying $\beta$.

\begin{proposition} \label{EVBeta1EqualsEVBetaProp}
    With notation as in Definition \ref{ClassicalEvMapDef} (and $\chi \in \Sigma_{\kappa, j}(\ide{N}_{\beta})$), we have $\opn{Ev}_{\kappa, j, \chi, \beta+1} = \opn{Ev}_{\kappa, j, \chi, \beta} \circ \opn{Tr}_{p_{G, \beta}}$.
\end{proposition}
\begin{proof}
    Since $\chi$ has conductor dividing $\ide{N}_{\beta}$, we see that $[\chi] \in \opn{H}^0\left( S_{\mbf{H}, \diamondsuit}(p^{\beta+1})_{F^{\opn{cl}}(\chi)}, \mathscr{M}_{H, \sigma_{\kappa}^{[j], \vee}} \right)$ is the image of the class at level $\beta$ under $p_{H, \beta}^*$. The claim now follows from Proposition \ref{TraceCompatibilityOfThetaProp}.
\end{proof}

\subsubsection{Relation with unitary Friedberg--Jacquet periods} \label{RelationWithUFJperiodsSubSubSec}

In this section, we relate the \emph{algebraic} construction of the morphisms $\vartheta_{\kappa, j, \beta}$ in the previous section with an \emph{analytic} construction over $\mbb{C}$ (Lemma \ref{ComplexThetaDefCoincidesWithOldDefLemma}). In particular, this will allow us to reinterpret the maps $\vartheta_{\kappa, j, \beta}$ on cohomology (over $\mbb{C}$) in terms of Lie algebra cohomology (Lemma \ref{ChevalleyEComplex}), and hence relate the evaluation maps in Definition \ref{ClassicalEvMapDef} (over $\mbb{C}$) with automorphic periods for the groups $\mbf{H} \subset \mbf{G}$ (Proposition \ref{ClassicalComplexUFJrelationProp}). The key input for this is the description of the coherent cohomology of Shimura varieties in terms of automorphic representations, following \cite{HarrisPartial} and \cite{Su19}.

For the rest of this section, we work over $L = \mbb{C}$, but will often omit this from the notation. In particular, gothic letters will denote the complexification of the Lie algebra of the corresponding group, unless specified otherwise. 

Let $K_{\infty} \subset \mbf{G}(\mbb{R})$ denote the maximal compact-mod-centre subgroup whose complexification is equal to $M_{\mbf{G}}(\mbb{C})$. Let $K_{\circ} \subset K_{\infty}$ denote the maximal compact subgroup, and let $A_{\mbf{G}}$ denote the maximal $\mbb{Q}$-split torus in the centre of $\mbf{G}$. Then one can verify that $K_{\infty} = K_{\circ} A_{\mbf{G}}(\mbb{R})^{\circ}$, where $A_{\mbf{G}}(\mbb{R})^{\circ} \subset A_{\mbf{G}}(\mbb{R})$ denotes the connected component of the identity. Let $\overline{\ide{p}}_G$ (resp. $\ide{a}_G$) denote the Lie algebra of $P_{\mbf{G}}^{\opn{std}}$ (resp. $A_{\mbf{G}}$), and set $\overline{\ide{p}}_{\circ} \defeq \overline{\ide{p}}_G \cap \opn{Lie}(\mbf{G}_0)$. Then one has $\overline{\ide{p}}_G = \overline{\ide{p}}_{\circ} \oplus \ide{a}_G$.

Let $[\mbf{G}] \defeq \mbf{G}(\mbb{Q}) \backslash \mbf{G}(\mbb{A}) / A_{\mbf{G}}(\mbb{R})^{\circ}$, and let $K = K^p K^G_{\opn{Iw}}(p^{\beta}) \subset \mbf{G}(\mbb{A}_f)$ denote the level subgroup for $S_{\mbf{G}, \opn{Iw}}(p^{\beta})$. Following \cite{Su19}, consider the following (right) $K_{\circ}$-torsor
\[
\pi_{\circ} \colon [\mbf{G}]/K \to S_{\mbf{G}, \opn{Iw}}(p^{\beta})(\mbb{C}) .
\]
Let $V$ be an algebraic representation of $P_{\mbf{G}}^{\opn{std}}(\mbb{C})$ on which $A_{\mbf{G}}(\mbb{C})$ acts trivially. Then the automorphic sheaf $\mathcal{V}$ on $S_{\mbf{G}, \opn{Iw}}(p^{\beta})(\mbb{C})$ associated with $V$ satisfies 
\[
\mathcal{V}(U) \cong \left( C^{\opn{\infty}}(\pi_{\circ}^{-1}U) \otimes V \right)^{(\overline{\ide{p}}_{\circ}, K_{\circ})}
\]
where the isomorphism also respects equivariant structures. Here $C^{\opn{\infty}}(-)$ denotes the space of complex-valued smooth functions on a real-analytic manifold, and $U \subset S_{\mbf{G}, \opn{Iw}}(p^{\beta})(\mbb{C})$ is an open subspace. We have a similar description for automorphic vector bundles on $S_{\mbf{H}, \diamondsuit}(p^{\beta})(\mbb{C})$. In particular, if $W$ is an algebraic representation of $P_{\mbf{H}}^{\opn{std}}$ on which $A_{\mbf{H}}(\mbb{C}) = A_{\mbf{G}}(\mbb{C})$ acts trivially, and $\mathcal{W}$ denotes the associated sheaf on $S_{\mbf{H}, \diamondsuit}(p^{\beta})$, then
\[
\hat{\iota}_*\mathcal{W}(U) \cong \left( C^{\opn{\infty}}(\hat{\iota}^{-1}\pi_{\circ}^{-1}U) \otimes W \right)^{(\overline{\ide{p}}_{\circ} \cap \ide{h}, K_{\circ} \cap \mbf{H}(\mbb{R}))}
\]
where $\hat{\iota} \colon [H]/(\hat{\gamma} K \hat{\gamma}^{-1} \cap \mbf{H}(\mbb{A}_f)) \to [G]/K$ is the natural map induced from right-translation by $\hat{\gamma}$.

We equip the sheaf $U \mapsto C^{\infty}(\pi_{\circ}^{-1}U)$ with an action of $\ide{g}$ in the following way. For any $X \in \opn{Lie}\mbf{G}(\mbb{R})$ and $f \in C^{\infty}(\pi_{\circ}^{-1}U)$, we define
\[
(X \star f)(-) \defeq  \left. \frac{d}{dt}\right|_{t=0} f( - \cdot \opn{exp}(tX))
\]
where $\opn{exp}(tX) \in \mbf{G}(\mbb{R})$ denotes the exponential of $tX$. We then extend this linearly to an action of $\ide{g} = \opn{Lie}\mbf{G}(\mbb{C})$. This induces an action of $\mathcal{U}(\ide{u}_G) \cong C^{\opn{pol}}(\mbb{C}^{\oplus 2n-1}, \mbb{C})$ such that the action map is equivariant for $K_{\circ}$ (via the action in (\ref{ActionOnCpolEqn})). 

\begin{lemma} \label{ComplexThetaDefCoincidesWithOldDefLemma}
    Let $(\kappa, j) \in \mathcal{E}$. Consider the following $\mbb{C}$-linear morphism 
    \[
    \vartheta_{\kappa, j, \beta} \colon C^{\opn{\infty}}(\pi_{\circ}^{-1}U) \otimes V_{\kappa}^* \to C^{\opn{\infty}}(\hat{\iota}^{-1}\pi_{\circ}^{-1}U) \otimes \sigma_{\kappa}^{[j]} 
    \]
    defined on pure tensors as:
    \[
    \vartheta_{\kappa, j, \beta}( f \otimes \lambda ) = \left. \left[ \lambda(\delta_{\kappa, j}) \star f \right] \right|_{\hat{\iota}^{-1}\pi_{\circ}^{-1}U} 
    \]
    where $f \in C^{\opn{\infty}}(\pi_{\circ}^{-1}U)$, $\lambda \in V_{\kappa}^*$, and $\delta_{\kappa, j}$ is defined in Definition \ref{DefinitionOfDeltakappaj}. Then:
    \begin{enumerate}
        \item The morphism $\vartheta_{\kappa, j, \beta}$ is $K_{\circ} \cap \mbf{H}(\mbb{R})$-equivariant.
        \item For any $m \geq 0$, the following diagram commutes
        \[
\begin{tikzcd}[column sep=large]
{\left( C^{\opn{\infty}}(\pi_{\circ}^{-1}U) \otimes V_{\kappa}^* \otimes \bigwedge^m \overline{\ide{u}}_G^* \right)^{K_{\circ}}} \arrow[d, "\alpha_1"'] \arrow[r, "{\vartheta_{\kappa, j, \beta} \otimes \opn{pr}}"]     & {C^{\opn{\infty}}(\hat{\iota}^{-1}\pi_{\circ}^{-1}U) \otimes \sigma_{\kappa}^{[j]} \otimes \bigwedge^m \overline{\ide{u}}_H^*} \arrow[d, "\alpha_2"] \\
C^{\opn{\infty}}(\pi_{\circ}^{-1}U) \otimes V_{\kappa}^* \otimes \bigwedge^m \overline{\ide{u}}_G^* \otimes \overline{\ide{u}}_H^* \arrow[r, "{\vartheta_{\kappa, j, \beta} \otimes \opn{pr} \otimes \opn{id}}"] & {C^{\opn{\infty}}(\hat{\iota}^{-1}\pi_{\circ}^{-1}U) \otimes \sigma_{\kappa}^{[j]} \otimes \bigwedge^m \overline{\ide{u}}_H^* \otimes \overline{\ide{u}}_H^*}            
\end{tikzcd}
        \]
        where $\alpha_i$ satisfies 
        \[
        \langle \alpha_i(f \otimes \lambda \otimes \gamma), X \rangle = (X \star f) \otimes \lambda \otimes \gamma \text{ for } \left\{ \begin{array}{cc} f \in C^{\opn{\infty}}(\pi_{\circ}^{-1}U), \lambda \in V_{\kappa}^*, \gamma \in \bigwedge^m \overline{\ide{u}}_G^*, X \in \overline{\ide{u}}_H & \text{ if } i=1 \\ f \in C^{\opn{\infty}}(\hat{\iota}^{-1}\pi_{\circ}^{-1}U), \lambda \in \sigma_{\kappa}^{[j]}, \gamma \in \bigwedge^m \overline{\ide{u}}_H^*, X \in \overline{\ide{u}}_H & \text{ if } i=2 \end{array} \right.
        \]
        and $\opn{pr} \colon \bigwedge^m \overline{\ide{u}}_G^* \to \bigwedge^m \overline{\ide{u}}_H^*$ is the map induced from the dual of the inclusion $\overline{\ide{u}}_H \hookrightarrow \overline{\ide{u}}_G$.
        \item The morphism $\vartheta_{\kappa, j, \beta}$ induces a morphism
    \[
    \mathscr{M}_{G, \kappa^*}(U) \cong \left( C^{\opn{\infty}}(\pi_{\circ}^{-1}U) \otimes V_{\kappa}^* \right)^{(\overline{\ide{p}}_{\circ}, K_{\circ})} \to \left( C^{\opn{\infty}}(\hat{\iota}^{-1}\pi_{\circ}^{-1}U) \otimes \sigma_{\kappa}^{[j]} \right)^{(\overline{\ide{p}}_{\circ} \cap \ide{h}, K_{\circ} \cap \mbf{H}(\mbb{R}))} \cong \hat{\iota}_* \mathscr{M}_{H, \sigma_{\kappa}^{[j]}}(U)
    \]
    which coincides with the morphism in Proposition \ref{ClassicalThetaPreservesHolProp}.
    \end{enumerate}
\end{lemma}
\begin{proof}
    The $K_{\circ} \cap \mbf{H}(\mbb{R})$-equivariance is clear. We claim that part (2) follows from the same argument as in the proof of Proposition \ref{ClassicalThetaPreservesHolProp}. Indeed, let $F \in {\left( C^{\opn{\infty}}(\pi_{\circ}^{-1}U) \otimes V_{\kappa}^* \otimes \bigwedge^m \overline{\ide{u}}_G^* \right)^{K_{\circ}}}$ and consider the element:
    \[
    F' \defeq \left[ \alpha_2 \circ ( \vartheta_{\kappa, j, \beta} \otimes \opn{pr}) \right] (F) - \left[ (\vartheta_{\kappa, j, \beta} \otimes \opn{pr} \otimes \opn{id}) \circ \alpha_1 \right](F)
    \]
    which is fixed by the action of $K_{\circ} \cap \mbf{H}(\mbb{R})$. This is therefore equivalent to a $K_{\circ} \cap \mbf{H}(\mbb{R})$-equivariant morphism:
    \[
    F' \colon \sigma_{\kappa}^{[j], -1} \otimes \bigwedge^m \overline{\ide{u}}_H \otimes \overline{\ide{u}}_H \to C^{\opn{\infty}}(\hat{\iota}^{-1}\pi_{\circ}^{-1}U)
    \]
    which we wish to show is zero. Consider the elementary matrices $E_{i, 1} \in \overline{\ide{u}}_H$ with $i \in \{2, \dots, n \}$, which form a basis of this Lie algebra. For any such elementary matrix $E_{i, 1}$ (with $i \in \{2, \dots, n \}$) and any polynomial $Q \in C^{\opn{pol}}(\mbb{C}^{\oplus 2n-1}, \mbb{C})$ in the last $n$ coordinates $X_{n+1}, \dots, X_{2n}$ of $\mbb{C}^{\oplus 2n-1}$, we have
    \[
    E_{i, 1} \star ( Q \star f ) = Q \star (E_{i, 1} \star f) + \sum_{j=n+1}^{2n} \frac{\partial Q}{\partial X_j} \star (E_{i, j} \star f) 
    \]
    for any $f \in C^{\opn{\infty}}(\pi_{\circ}^{-1}U)$. Fix a basis $\{ v_l \}$ of $V_{\kappa}$ and let $\{v_l^*\}$ denote the dual basis. By viewing $F$ as a $K_{\circ}$-equivariant morphism $V_{\kappa} \otimes \bigwedge^m \overline{\ide{u}}_G \to C^{\opn{\infty}}(\pi_{\circ}^{-1}U)$, we see that the morphism $F'$ satisfies 
    \begin{align*} 
    F'(1 \otimes \gamma \otimes E_{i, 1}) &= \sum_{l}  \left. \left[E_{i, 1} \star v_l^*(\delta_{\kappa, j}) \star F(v_l \otimes \gamma) - v_l^*(\delta_{\kappa, j}) \star E_{i, 1} \star F(v_l \otimes \gamma) \right] \right|_{\hat{\iota}^{-1}\pi_{\circ}^{-1}U} \\
    &= \sum_l  \sum_{j=n+1}^{2n} \frac{\partial v_l^*(\delta_{\kappa, j})}{\partial X_j} \star ( E_{i, j} \star F(v_l \otimes \gamma) )|_{\hat{\iota}^{-1}\pi_{\circ}^{-1}U}
    \end{align*} 
    where $\gamma \in \bigwedge^m \overline{\ide{u}}_H$ (note that $\gamma$ is killed by $E_{i, j}$ for $i \in \{2, \dots, n\}$ and $j \in \{n+1, \dots, 2n\}$). We can define a $K_{\circ} \cap \mbf{H}(\mbb{R})$-equivariant map
    \[
    F'' \colon V_{\kappa} \otimes \bigwedge^m \overline{\ide{u}}_H \otimes S_{-(j-1)} \to C^{\opn{\infty}}(\hat{\iota}^{-1}\pi_{\circ}^{-1}U)
    \]
    such that the image of $F'$ is contained in the image of $F''$. Then, by the same argument as in Proposition \ref{ClassicalThetaPreservesHolProp}, one can then show that the obstruction to inducing a morphism $\mathscr{M}_{G, \kappa^*} \to \hat{\iota}_* \mathscr{M}_{H, \sigma_{\kappa}^{[j]}}$ is given by an element in 
    \[
    \opn{Hom}_{K_{\circ} \cap \mbf{H}(\mbb{R})}\left( \sigma_{\kappa}^{[j], -1} \otimes \overline{\ide{u}}_H, V_{\kappa} \otimes S_{-(j-1)} \right) .
    \]
    But this hom-space is trivial for the same reasons as in Proposition \ref{ClassicalThetaPreservesHolProp}. This also implies that the morphism $F'$ is zero, and hence completes the proof of part (2).

    It remains to show that the induced morphism in part (3) coincides with the one in Proposition \ref{ClassicalThetaPreservesHolProp}. But by the calculations in \S \ref{DmodulesOnFLsection}, the action of $C^{\opn{pol}}(\mbb{C}^{\oplus 2n-1}, \mbb{C})$ is induced from the action of $\ide{u}_G$ in Lemma \ref{HorzActionOnFL} (denoted $\tilde{\star}_{\mathcal{D}}$). Indeed, let $\pi_{\mbb{C}} \colon \mbf{G}(\mbb{C}) \to \mbf{G}(\mbb{C})/P_{\mbf{G}}^{\opn{std}}(\mbb{C})$ denote the natural $P_{\mbf{G}}^{\opn{std}}(\mbb{C})$-torsor. Then, from Lemma \ref{HorzActionOnFL}, the action of $\ide{u}_G$ on $\pi_{\mbb{C}, *}\mathcal{O}_{\mbf{G}(\mbb{C})}$ satisfies $X \tilde{\star}_{\mathcal{D}} f = X \star_r f$ (with $X \in \ide{u}_G$ and $f \in \pi_{\mbb{C}, *}\mathcal{O}_{\mbf{G}(\mbb{C})}$), where $\star_r$ denotes the action induced from right-translation of the argument. One then sees that the action of $\ide{u}_G$ on $\pi_{\circ, *}C^{\infty}(-)$ (defined just before this lemma) extends $\tilde{\star}_{\mathcal{D}}$. The compatibility between the two morphisms now follows.
\end{proof}

For any $m \geq 0$, let $\iota \colon \bigwedge^{m} \overline{\ide{u}}_H \hookrightarrow \bigwedge^{m} \overline{\ide{u}}_G$ denote the natural map induced from $\mbf{H} \hookrightarrow \mbf{G}$.

\begin{lemma} \label{ChevalleyEComplex}
    Let $(\kappa, j) \in \mathcal{E}$. Then $\vartheta_{\kappa, j, \beta}$ induces a morphism of Chevalley-Eilenberg complexes:
    \begin{align*} 
    \opn{Hom}_{K_{\circ}}\left( \wedge^{\bullet}\overline{\ide{u}}_G, C^{\infty}(\pi_{\circ}^{-1}-) \otimes V_{\kappa}^* \right) &\to \opn{Hom}_{K_{\circ} \cap \mbf{H}(\mbb{R})}\left( \wedge^{\bullet}\overline{\ide{u}}_H, C^{\infty}(\hat{\iota}^{-1}\pi_{\circ}^{-1}-) \otimes \sigma_{\kappa}^{[j]} \right) \\ 
    f &\mapsto \vartheta_{\kappa, j, \beta} \circ f \circ \iota 
    \end{align*} 
    and, after passing to cohomology, coincides with the map
    \[
    \vartheta_{\kappa, j, \beta} \colon \opn{H}^{n-1}\left( S_{\mbf{G}, \opn{Iw}}(p^{\beta}), \mathscr{M}_{G, \kappa^*} \right) \to \opn{H}^{n-1}\left( S_{\mbf{H}, \diamondsuit}(p^{\beta}), \mathscr{M}_{H, \sigma_{\kappa}^{[j]}} \right)
    \]
    from \S \ref{ClassicalEvaluationMapsSubSec} via the identifications in \cite[\S 1.1]{Su19}.
\end{lemma}
\begin{proof}
Noting that $\overline{\ide{u}}_G$ and $\overline{\ide{u}}_H$ are abelian, this follows from Lemma \ref{ComplexThetaDefCoincidesWithOldDefLemma}.
\end{proof}

Finally, we can relate the evaluation maps above to unitary Friedberg--Jacquet periods. Set $[\mbf{H}] \defeq \mbf{H}(\mbb{Q}) \backslash \mbf{H}(\mbb{A}) / A_{\mbf{G}}(\mbb{R})^{\circ}$ (note that $A_{\mbf{G}} \subset \mbf{H}$). Let $dh$ denote the Tamagawa measure on $[\mbf{H}]$. Fix bases $\alpha_{+}$ and $\alpha_{-}$ of $\bigwedge^{n-1} \ide{u}_H$ and $\bigwedge^{n-1} \overline{\ide{u}}_H$ respectively, and recall the definition of $\Delta_{\kappa}^{[j]} \in \mathcal{U}(\ide{g})$ from Definition \ref{AppendixDefOfDELTAkappaJDiffOp}. 

\begin{proposition} \label{ClassicalComplexUFJrelationProp}
    Let $(\kappa, j) \in \mathcal{E}$ and $\chi \in \Sigma_{\kappa, j}(\ide{N}_{\beta})$. Let $\eta \in \opn{H}^{n-1}\left( S_{\mbf{G}, \opn{Iw}}(p^{\beta}), \mathscr{M}_{G, \kappa^*} \right)$, which we view as a $K_{\circ}$-equivariant morphism $F \colon \bigwedge^{n-1}\overline{\ide{u}}_G \otimes V_{\kappa} \to C^{\infty}([\mbf{G}]/K)$. Set $\phi = F(\iota(\alpha_{-}) \otimes v_{\kappa}^{[0]})$. Then
    \[
    \opn{Ev}_{\kappa, j, \chi, \beta}(\eta) = (2 \pi i)^{-(n-1)} \opn{Vol}(K_{H, \beta}; dh)^{-1} \int_{[\mbf{H}]} \left(\Delta_{\kappa}^{[j]} \cdot \phi \right)(h \hat{\gamma}) \chi'\left( \frac{\opn{det}h_2}{\opn{det}h_1} \right) dh 
    \]
    where we write $h = (h_1, h_2)$ for the components preserving the corresponding factor in the decomposition $W = W_1 \oplus W_2$, we set $K_{H, \beta} \defeq \left( K_{\circ} \cap \mbf{H}(\mbb{R}) \right) \cdot \left(K^p \cap \mbf{H}(\mbb{A}_f^p) \right) \cdot K^H_{\diamondsuit}(p^{\beta})$, and we view $\hat{\gamma} \in \mbf{G}(\mbb{Q}_p) \subset \mbf{G}(\mbb{A}_f)$.
\end{proposition}
\begin{proof}
    By Lemma \ref{ComplexThetaDefCoincidesWithOldDefLemma}, Lemma \ref{ChevalleyEComplex}, and the definition of $\Delta_{\kappa}^{[j]}$, we have 
    \[
    \vartheta_{\kappa, j, \beta}(F)(\alpha_{-}) = \left. \left[ \Delta^{[j]}_{\kappa} \cdot F(\iota(\alpha_{-}) \otimes v_{\kappa}^{[0]}) \right] \right|_{[\mbf{H}]/(\hat{\gamma} K \hat{\gamma}^{-1} \cap \mbf{H}(\mbb{A}_f))} 
    \]
    where the restriction is via the map $\hat{\iota}$. Here we have used the fact that $\iota(\alpha_{-})$ is killed by any $E_{1, k, \tau_0}$ for $k \in \{n+1, \dots, 2n\}$. Similarly, we can view $[\chi]$ as a homomorphism 
    \[
    [\chi] \colon \sigma_{\kappa}^{[j]} \otimes \bigwedge^{n-1} \ide{u}_H \to C^{\infty}([\mbf{H}]/(\hat{\gamma} K \hat{\gamma}^{-1} \cap \mbf{H}(\mbb{A}_f))) 
    \]
    which satisfies $[\chi](\alpha_+)(h) = \chi'(\opn{det}h_2/\opn{det}h_1)$. The cup product $\vartheta_{\kappa, j, \beta}(F) \smile [\chi]$ then corresponds to the $K_{H, \beta}$-invariant volume form
    \[
    h \mapsto \left(\Delta_{\kappa}^{[j]} \cdot \phi \right)(h \hat{\gamma}) \chi' \left( \frac{\opn{det}h_2}{\opn{det}h_1} \right) \alpha_{+}^* \wedge \alpha_{-}^* .
    \]
    The result now follows from \cite[Proposition 3.8]{HarrisPartial}.
\end{proof}

%----------------------------------

\section{The \texorpdfstring{$p$}{p}-adic theory} \label{ThePADICtheorySection}

In the section we describe the $p$-adic analogue of the previous section. Recall we have fixed a prime $p$ which splits completely in $F/\mbb{Q}$ (see Assumption \ref{AssumpPoddAndDoesntDivide}). 

\subsection{Igusa varieties} \label{IgusaVarietiesSection}

To be able to define the appropriate ordinary strata in the Shimura--Deligne varieties we consider, we need to introduce certain Igusa varieties. Let $\mathcal{X}_{G}$ and $\mathcal{X}_{H}$ denote the adic spaces over $\mbb{Q}_p$ associated with $X_{\mbf{G}, \mbb{Q}_p}$ and $X_{\mbf{H}, \mbb{Q}_p}$, where we have base-changed along the morphism $\mathcal{O}_{F^{\opn{cl}}, (p)} \to \mbb{Q}_p$ induced from the fixed embedding $F \hookrightarrow \Qpb$ induced by $\tau_0$. Let $\ide{X}_G$ and $\ide{X}_G$ denote the formal completions of $X_{\mbf{G}, \mbb{Z}_p}$ and $X_{\mbf{H}, \mbb{Z}_p}$ respectively along the special fibre, where we have base-changed along the morphism $\mathcal{O}_{F^{\opn{cl}}, (p)} \to \mbb{Z}_p$ induced by $\tau_0$. Since $X_{\mbf{G}}$ and $X_{\mbf{H}}$ are proper, the adic generic fibres of $\ide{X}_G$ and $\ide{X}_H$ are precisely $\mathcal{X}_G$ and $\mathcal{X}_H$. 

\subsubsection{Caraiani--Scholze Igusa varieties}

We now introduce the Igusa varieties considered in \cite[\S 4.3]{CS17}. For $\tau \in \Psi$, consider the following $p$-divisible group over $\opn{Spf}\mbb{Z}_p$:
\[
\mbb{X}_{\opn{ord}, \tau} = \left\{ \begin{array}{cc} \mu_{p^{\infty}} \oplus (\mbb{Q}_p/\mbb{Z}_p)^{\oplus 2n-1} & \tau = \tau_0 \\ (\mbb{Q}_p/\mbb{Z}_p)^{\oplus 2n} & \tau \neq \tau_0 \end{array} \right.
\]
For any $\tau \in \Psi$, we have a decomposition $\mbb{X}_{\opn{ord}, \tau} = \mbb{X}_{1, \tau} \oplus \mbb{X}_{2, \tau}$, where $\mbb{X}_{1, \tau_0} = \mu_{p^{\infty}} \oplus (\mbb{Q}_p/\mbb{Z}_p)^{\oplus n-1}$ and $\mbb{X}_{i, \tau} = (\mbb{Q}_p/\mbb{Z}_p)^{\oplus n}$ for either $\tau \neq \tau_0$ and $i = 1, 2$, or $(i, \tau) = (2, \tau_0)$. We also let $\tilde{\mbb{X}}_{\opn{ord}, \tau} = \varprojlim_{\times p} \mbb{X}_{\opn{ord}, \tau}$ and $\tilde{\mbb{X}}_{i, \tau} = \varprojlim_{\times p} \mbb{X}_{i, \tau}$ denote the universal covers, where the inverse limit is over multiplication by $p$. Let $\opn{Nilp}_{\mbb{Z}_p}$ denote the category of $\mbb{Z}_p$-algebras on which $p$ is nilpotent. In what follows, if $M$ is a locally profinite group, we let $\underline{M}$ denote the fpqc sheaf\footnote{That is, a sheaf for the fpqc topology as in \cite[Tag 03NV]{stacks-project}.} on $\opn{Nilp}^{\opn{op}}_{\mbb{Z}_p}$ given by $\underline{M}(R) = \opn{Cont}(\opn{Spec}R, M)$ (continuous maps $\opn{Spec}R \to M$ for the Zariski topology and locally profinite topology on the source and target respectively). This is consistent with the notation in \cite[\S 4.1.1]{howe2020unipotent} for example.

\begin{definition}
    For $\tau \in \Psi$, let $J_{G, \opn{ord}, \tau}$ and $J^+_{G, \opn{ord}, \tau}$ denote the fpqc sheaves on $\opn{Nilp}_{\mbb{Z}_p}^{\opn{op}}$ given by
    \[
    J_{G, \opn{ord}, \tau}(R) = \opn{Aut}( \tilde{\mbb{X}}_{\opn{ord}, \tau, R} ), \quad \quad J^+_{G, \opn{ord}, \tau}(R) = \opn{Aut}( \mbb{X}_{\opn{ord}, \tau, R} ).
    \]
    We set $J_{G, \opn{ord}} = \underline{\mbb{Q}_p^{\times}} \times \prod_{\tau \in \Psi} J_{G, \opn{ord}, \tau}$ and $J^+_{G, \opn{ord}} = \underline{\mbb{Z}_p^{\times}} \times \prod_{\tau \in \Psi} J^+_{G, \opn{ord}, \tau}$. Let $J_{H, \opn{ord}, \tau} \subset J_{G, \opn{ord}, \tau}$ (resp. $J^+_{H, \opn{ord}, \tau} \subset J^+_{G, \opn{ord}, \tau}$) denote the sub-sheaves preserving the decomposition $\tilde{\mbb{X}}_{\opn{ord}, \tau} = \tilde{\mbb{X}}_{1, \tau} \oplus \tilde{\mbb{X}}_{2, \tau}$ (resp. $\mbb{X}_{\opn{ord}, \tau} = \mbb{X}_{1, \tau} \oplus \mbb{X}_{2, \tau}$). We set $J_{H, \opn{ord}} = \underline{\mbb{Q}_p^{\times}} \times \prod_{\tau \in \Psi} J_{H, \opn{ord}, \tau}$ and $J^+_{H, \opn{ord}} = \underline{\mbb{Z}_p^{\times}} \times \prod_{\tau \in \Psi} J^+_{H, \opn{ord}, \tau}$.
\end{definition}

\begin{remark}
    Concretely, we have the following descriptions
    \[
    J_{G, \opn{ord}} = \underline{\mbb{Q}_p^{\times}} \times  \tbyt{\underline{\mbb{Q}_p^{\times}}}{\tilde{\mu_{p^{\infty}}}^{\oplus 2n-1}}{}{\underline{\opn{GL}_{2n-1}(\mbb{Q}_p)}} \times \prod_{\tau \neq \tau_0} \underline{\opn{GL}_{2n}(\mbb{Q}_p)}
    \]
    and 
    \[
    J_{G, \opn{ord}}^+ = \underline{\mbb{Z}_p^{\times}} \times  \tbyt{\underline{\mbb{Z}_p^{\times}}}{T_p\mu_{p^{\infty}}^{\oplus 2n-1}}{}{\underline{\opn{GL}_{2n-1}(\mbb{Z}_p)}} \times \prod_{\tau \neq \tau_0} \underline{\opn{GL}_{2n}(\mbb{Z}_p)}
    \]
    where $T_p\mu_{p^{\infty}}$ denotes the Tate module. The first factor will correspond to the similitude factor (see Definition \ref{DefOfCSIgusa} and Remark \ref{RemarkOnCSIgusa}).
\end{remark}

Let $\ide{X}_G^{\opn{ord}}$ and $\ide{X}_H^{\opn{ord}}$ denote the ordinary loci in $\ide{X}_G$ and $\ide{X}_H$ respectively. We now introduce the Igusa varieties:

\begin{definition} \label{DefOfCSIgusa}
    Let $\ide{IG}_G \to \ide{X}_G^{\opn{ord}}$ denote the functor on $\opn{Nilp}_{\mbb{Z}_p}^{\opn{op}}$ given by
    \[
    \ide{IG}_G(R) = \left\{ (A, \lambda, i, \eta^p, s, f_{\tau}) : \begin{array}{c} (A, \lambda, i, \eta^p) \in \ide{X}_G^{\opn{ord}}(R), s \in \underline{\mbb{Z}_p^{\times}}(R), \\  f_{\tau} \colon \mbb{X}_{\opn{ord}, \tau, R} \xrightarrow{\sim} A[\ide{p}_{\tau}^{\infty}] \end{array}  \right\}. 
    \]
    Similarly, we let $\ide{IG}_H \to \ide{X}_H^{\opn{ord}}$ denote the functor on $\opn{Nilp}^{\opn{op}}_{\mbb{Z}_p}$ such that $\ide{IG}_H(R)$ consists of tuples $(A_1, A_2, \lambda, i, \eta^p, s, f_{\tau})$ with $(A_1, A_2, \lambda, i, \eta^p) \in \ide{X}_H^{\opn{ord}}(R)$, $s \in \underline{\mbb{Z}_p^{\times}}(R)$, and $f_{\tau} \colon \mbb{X}_{\opn{ord}, \tau, R} \xrightarrow{\sim} A[\ide{p}_{\tau}^{\infty}]$ are isomorphisms preserving the decompositions $\mbb{X}_{\opn{ord}, \tau} = \mbb{X}_{1, \tau} \oplus \mbb{X}_{2, \tau}$ and $A[\ide{p}_{\tau}^{\infty}] = A_1[\ide{p}_{\tau}^{\infty}] \oplus A_2[\ide{p}_{\tau}^{\infty}]$.
\end{definition}

\begin{remark} \label{RemarkOnCSIgusa}
    Set $\mbb{X}_{\opn{ord}} = \bigoplus_{\tau \in \Psi} (\mbb{X}_{\opn{ord}, \tau} \oplus \mbb{X}_{\opn{ord}, \tau}^D)$, where $(-)^D$ denotes the dual $p$-divisible group. Then there is a natural symplectic pairing on $\mbb{X}_{\opn{ord}}$ given by
    \[
    \langle \sum_{\tau} x_{\tau}+y_{\tau}, \sum_{\tau} x'_{\tau} + y_{\tau}' \rangle_{\opn{std}} = \sum_{\tau} \langle x_{\tau}, y'_{\tau} \rangle - \langle y_{\tau}, x'_{\tau} \rangle 
    \]
    where $x_{\tau}, x_{\tau}' \in \mbb{X}_{\opn{ord}, \tau}$, $y_{\tau}, y_{\tau}' \in \mbb{X}_{\opn{ord}, \tau}^D$, and the pairings $\langle -, - \rangle$ on the right-hand side denote the natural ones. It also has a natural endomorphism structure via the identification $\mathcal{O}_F \otimes_{\mbb{Z}} \mbb{Z}_p = \bigoplus_{\tau \in \Psi} (\mbb{Z}_p \oplus \mbb{Z}_p)$. Given a tuple $(A, \lambda, i, \eta^p, s, f_{\tau}) \in \ide{IG}_G(R)$, we therefore obtain a trivialisation
    \begin{equation} \label{Eqn:BigXordTriv}
    \mbb{X}_{\opn{ord}} = \bigoplus_{\tau \in \Psi} (\mbb{X}_{\opn{ord}, \tau} \oplus \mbb{X}_{\opn{ord}, \tau}^D) \xrightarrow{\sim} \bigoplus_{\tau \in \Psi} (A[\ide{p}_{\tau}^{\infty}] \oplus A[\ide{p}_{\tau^c}^{\infty}]) = A[p^{\infty}]
    \end{equation}
    given by $\bigoplus_{\tau} (f_{\tau} \; \oplus \; s (f_{\tau}^D)^{-1})$, where we are identifying $A[\ide{p}_{\tau^c}^{\infty}]$ with $A[\ide{p}_{\tau}^{\infty}]^D$ via the Weil pairing on $A$. By design, the trivialisation in \eqref{Eqn:BigXordTriv} respects the endomorphism and symplectic structures up to the similitude $s$; hence the moduli problem $\ide{IG}_G$ is equivalent to (the formal version of) the one in \cite[Definition 4.3.1]{CS17} parameterising trivialisations of $A[p^{\infty}]$ respecting the endomorphism and symplectic structure up to similitude. The same is true for $\ide{IG}_H$. 
\end{remark}

We have the following properties of these functors.

\begin{proposition}
We have:
\begin{enumerate}
    \item $\ide{IG}_G$ and $\ide{IG}_H$ are representable by flat affine $p$-adic formal schemes over $\mbb{Z}_p$.
    \item $\ide{IG}_G \to \ide{X}_G^{\opn{ord}}$ and $\ide{IG}_H \to \ide{X}_H^{\opn{ord}}$ are fpqc torsors under the groups $J^+_{G, \opn{ord}}$ and $J^+_{H, \opn{ord}}$ respectively.
    \item The action of $J^+_{G, \opn{ord}}$ (resp. $J^+_{H, \opn{ord}}$) on $\ide{IG}_G$ (resp. $\ide{IG}_H$) extends to an action of $J_{G, \opn{ord}}$ (resp. $J_{H, \opn{ord}}$).
\end{enumerate}
\end{proposition}
\begin{proof}
    Parts (1) and (3) follow from \cite[p.718]{CS17} and \cite[Corollary 4.3.5]{CS17} respectively. Part (2) follows from the same proof as in \cite[Lemma 5.1.1]{howe2020unipotent} (the cover $\ide{IG}_{\bullet} \to \ide{X}_{\bullet}^{\opn{ord}}$ is pro-finite-flat, hence fpqc).
\end{proof}

\subsubsection{Quotients of Igusa varieties} \label{QuotientsOfIgVarSubSec}

We recall some notation from \cite{UFJ}. Let $P^G_{\opn{Iw}}(p^{\beta}) = w_n K^G_{\opn{Iw}}(p^{\beta}) w_n^{-1} \cap P_G(\mbb{Q}_p)$ and let $M^G_{\opn{Iw}}(p^{\beta})$ denote its image under the map $P_G(\mbb{Q}_p) \to M_G(\mbb{Q}_p)$. Let $N^G_{\opn{Iw}}(p^{\beta}) \subset J^+_{G, \opn{ord}}$ denote the sub fpqc sheaf given by
\[
N^G_{\opn{Iw}}(p^{\beta}) =  \underline{\{ 1 \}} \times \tbyt{\underline{\{ 1 \}}}{(p^{\beta}T_p \mu_{p^{\infty}})^{\oplus n} \oplus T_p \mu_{p^{\infty}}^{\oplus n-1}}{}{\underline{\{ 1 \}}} \times \prod_{\tau \neq \tau_0} \underline{\{1 \}}  .
\]
and set $J^G_{\opn{Iw}}(p^{\beta}) = N^G_{\opn{Iw}}(p^{\beta}) \rtimes \underline{M^G_{\opn{Iw}}(p^{\beta})} \subset J^+_{G, \opn{ord}}$. Here $\{1 \}$ denotes the trivial group. Also, set $P^H_{\diamondsuit}(p^{\beta}) = K^H_{\diamondsuit}(p^{\beta}) \cap P_H(\mbb{Q}_p)$ and $M^H_{\diamondsuit}(p^{\beta})$ its image under the map $P_H(\mbb{Q}_p) \to M_H(p^{\beta})$. Let $N^H_{\diamondsuit}(p^{\beta}) = N^G_{\opn{Iw}}(p^{\beta}) \cap J^+_{H, \opn{ord}}$ and $J^H_{\diamondsuit}(p^{\beta}) = N^H_{\diamondsuit}(p^{\beta}) \rtimes \underline{M^H_{\diamondsuit}(p^{\beta})} \subset J^+_{H, \opn{ord}}$.

\begin{definition}
    We introduce the following quotients:
    \begin{enumerate}
        \item Let $\ide{IG}_{G, w_n}(p^{\beta})$ (resp. $\ide{X}_{G, w_n}(p^{\beta})$) denote the flat $p$-adic formal scheme obtained as the quotient of $\ide{IG}_G$ by $N^G_{\opn{Iw}}(p^{\beta})$ (resp. $J^G_{\opn{Iw}}(p^{\beta})$). The map $\ide{IG}_{G, w_n}(p^{\beta}) \to \ide{X}_{G, w_n}(p^{\beta})$ is a pro\'{e}tale $M^G_{\opn{Iw}}(p^{\beta})$-torsor.
        \item Let $\ide{IG}_{H, \opn{id}}(p^{\beta})$ (resp. $\ide{X}_{H, \opn{id}}(p^{\beta})$) denote the flat $p$-adic formal scheme obtained as the quotient of $\ide{IG}_H$ by $N^H_{\diamondsuit}(p^{\beta})$ (resp. $J^H_{\diamondsuit}(p^{\beta})$). The map $\ide{IG}_{H, \opn{id}}(p^{\beta}) \to \ide{X}_{H, \opn{id}}(p^{\beta})$ is a pro\'{e}tale $M^H_{\diamondsuit}(p^{\beta})$-torsor.
    \end{enumerate}
\end{definition}

We have the following important properties of these quotients.

\begin{proposition} \label{PropositionLociIntClosedInGenFibre}
Let $\beta \geq 1$.
    \begin{enumerate}
        \item The natural maps $\ide{X}_{G, w_n}(p^{\beta}) \to \ide{X}_G^{\opn{ord}}$ and $\ide{X}_{H, \opn{id}}(p^{\beta}) \to \ide{X}_H^{\opn{ord}}$ are finite flat.
        \item $\ide{X}_{G, w_n}(p^{\beta})$ and $\ide{X}_{H, \opn{id}}(p^{\beta})$ are smooth $p$-adic formal schemes over $\opn{Spf}\mbb{Z}_p$.
        \item If we let $\mathcal{X}_{G, w_n}(p^{\beta})$ and $\mathcal{X}_{H, \opn{id}}(p^{\beta})$ denote the adic generic fibres of $\ide{X}_{G, w_n}(p^{\beta})$ and $\ide{X}_{H, \opn{id}}(p^{\beta})$, then $\ide{X}_{G, w_n}(p^{\beta})$ and $\ide{X}_{H, \opn{id}}(p^{\beta})$ are integrally closed in $\mathcal{X}_{G, w_n}(p^{\beta})$ and $\mathcal{X}_{H, \opn{id}}(p^{\beta})$ respectively, in the sense of \cite[\S 1.1]{PilloniStroh}.
    \end{enumerate}
\end{proposition} 
\begin{proof}
    Part (1) is clear because $\ide{IG}_{\bullet} \to \ide{X}_{\bullet}^{\opn{ord}}$ are pro-finite-flat torsors. Part (3) follows from part (2). Indeed, suppose that $\mathfrak{X}_{G, w_n}(p^{\beta}) \to \opn{Spf}\mbb{Z}_p$ is smooth. Then $\mathcal{X}_{G, w_n}(p^{\beta})$ is reduced, hence we are in the setting of \cite[\S 1.1]{PilloniStroh}. Suppose that $\ide{U} = \opn{Spf}R \subset \ide{X}_{G, w_n}(p^{\beta})$ is an open affine subspace with $R$ a smooth $\mbb{Z}_p$-algebra. Note that $R$ is $p$-torsion free. We want to show that $R$ is integrally closed in $R[1/p]$. It suffices to check this after localising at open primes ideals of $R$, i.e. if $\ide{p} \subset R$ is a prime ideal containing $p$, then we need to show that $R_{\ide{p}}$ is integrally closed in $R_{\ide{p}}[1/p]$. But since $R/p$ is a smooth $\mbb{F}_p$-algebra, $R_{\ide{p}}/p$ is regular, and since $R_{\ide{p}}$ is $p$-torsion free, $R_{\ide{p}}$ is regular and hence normal. Here we have used the fact that $R_{\ide{p}}$ is a flat $\mbb{Z}_p$-algebra, hence $\opn{dim}R_{\ide{p}} = 1 + \opn{dim}(R_{\ide{p}}/p)$. This implies that $R_{\ide{p}}$ is integrally closed in $R_{\ide{p}}[1/p]$. Therefore it remains to prove part (2).

     Let us begin by describing the method for proving part (2). The first attempt would be to show that the morphism $\ide{X}_{G, w_n}(p^{\beta}) \to \ide{X}_{G}^{\opn{ord}}$ is smooth (which would imply the claim because $\ide{X}_{G}^{\opn{ord}}$ is smooth over $\mbb{Z}_p$). Unfortunately this is false; so we have to modify this strategy. The idea is to factorise the map $\ide{X}_{G, w_n}(p^{\beta}) \to \ide{X}_{G}^{\opn{ord}}$ as
    \[
    \ide{X}_{G, w_n}(p^{\beta}) \to \ide{Y} \to \ide{X}_{G}^{\opn{ord}},
    \]
    and show two properties: the space $\ide{Y}$ is smooth over $\mbb{Z}_p$ and the morphism $\ide{X}_{G, w_n}(p^{\beta}) \to \ide{Y}$ is finite \'{e}tale (and hence smooth). The key idea for the first property is to construct a \emph{second map} $\ide{Y} \xrightarrow{q} \ide{X}_{G}^{\opn{ord}}$ which is smooth, hence one can deduce that $\ide{Y}$ is smooth over $\mbb{Z}_p$ from the fact that $\ide{X}_{G}^{\opn{ord}}$ is smooth over $\mbb{Z}_p$.\footnote{It does not seem immediate to the author how to construct such a map $q$ for $\ide{Y} = \ide{X}_{G, w_n}(p^{\beta})$ which is smooth, hence the need for an intermediate space $\ide{Y}$ which is genuinely different from $\ide{X}_{G, w_n}(p^{\beta})$ and $\ide{X}_{G}^{\opn{ord}}$.} For the second property, we will show that $\ide{X}_{G, w_n}(p^{\beta})$ and $\ide{Y}$ are both suitable quotients of the pro-\'{e}tale torsor $\ide{IG}_{G, w_n}(p^{\beta})$, which will automatically imply the map $\ide{X}_{G, w_n}(p^{\beta}) \to \ide{Y}$ is finite \'{e}tale. \\
    \\
    \textbf{Step 1: (The construction of $\ide{Y}$ and $q$)}
    
    Let $\ide{Y} \to \ide{X}_G^{\opn{ord}}$ denote the moduli space parameterising subgroup schemes $C \subset A[\ide{p}_{\tau_0}^{\beta}]$ which are \'{e}tale locally isomorphic to $\left(\mbb{Z}/p^{\beta}\mbb{Z} \right)^{\oplus n}$. For any such subgroup $C$, let $C' = C \oplus C^{\perp} \subset A[p^{\beta}]$ where $C^{\perp}$ is the orthogonal complement of $C$ under the symplectic pairing $A[\ide{p}_{\tau_0}^{\beta}] \times A[\ide{p}_{\tau_0^c}^{\beta}] \to \mu_{p^{\beta}}$ induced from the polarisation and Weil pairing. Then $C'$ is an $\mathcal{O}_F$-stable totally isotropic finite flat subgroup, hence $A/C'$ is naturally a $\boldsymbol\Psi$-unitary abelian scheme with an induced prime-to-$p$ level structure. Note that the natural map $\ide{Y} \to \ide{X}_G^{\opn{ord}}$ is just given by forgetting $C$. As described above, in order to show $\ide{Y}$ is a smooth $p$-adic formal scheme over $\opn{Spf}\mbb{Z}_p$, it is enough to construct a different morphism $q \colon \ide{Y} \to \ide{X}_G^{\opn{ord}}$ and show this is smooth. The candidate morphism we will consider is the finite flat map $q \colon \ide{Y} \to \ide{X}_G^{\opn{ord}}$ given by $q(A, C) = A/C'$.\\
    \\
    \textbf{Step 2: (Showing $q$ is smooth)}

    To show that $q$ is smooth, it is enough to prove it is formally smooth, i.e., we need to show for any commutative diagram
    \[
\begin{tikzcd}
T_0 \arrow[r] \arrow[d]        & \ide{Y} \arrow[d, "q"]  \\
T \arrow[r] \arrow[ru, dashed] & \ide{X}_{G}^{\opn{ord}}
\end{tikzcd}
    \]
    the dotted arrow exists, where $T = \opn{Spec}A$ is an affine scheme with $A$ a $\mbb{Z}/p^r\mbb{Z}$-algebra (for any $r \geq 1$), $J \subset A$ is a square-zero ideal, and $T_0 = \opn{Spec}A/J$. Rephrasing, we need to show that for any point $(A, C) \in \ide{Y}(T_0)$ and any deformation $B \in \ide{X}_{G}^{\opn{ord}}(T)$ such that $B|_{T_0} = A/C'$, there exists a point $(\tilde{B}, \tilde{C}) \in \ide{Y}(T)$ such that $\tilde{B}/\tilde{C}' = B$. 
    
    Let $K = \overline{A[\ide{p}_{\tau_0}^{\beta}]}$ denote the image of $A[\ide{p}_{\tau_0}^{\beta}]$ in $A/C'$ (which is a finite flat group scheme \'{e}tale locally isomorphic to $\mu_{p^{\beta}} \oplus \left(\mbb{Z}/p^{\beta}\mbb{Z}\right)^{\oplus n-1}$). Then we have a short exact sequence
    \begin{equation} \label{SESoffflatgp}
    0 \to K \to (A/C')[\ide{p}_{\tau_0}^{\beta}] \to L \to 0
    \end{equation}
    where $L$ is finite flat \'{e}tale locally isomorphic to $\left(\mbb{Z}/p^{\beta}\mbb{Z}_p\right)^{\oplus n}$. By Illusie's deformation theory \cite{Illusie} and since \'{e}tale group schemes deform uniquely, there exists a finite flat group scheme $\mathcal{L}$, \'{e}tale locally isomorphic to $\left(\mbb{Z}/p^{\beta}\mbb{Z}\right)^{\oplus n}$, such that $\mathcal{L}|_{T_0} = L$, and a morphism $B[\ide{p}_{\tau_0}] \twoheadrightarrow \mathcal{L}$ deforming the right-hand map in (\ref{SESoffflatgp}). Let $\mathcal{K}$ denote the kernel of the map $B[\ide{p}_{\tau_0}] \twoheadrightarrow \mathcal{L}$. Then the required point is given by $\tilde{B} = B/\mathcal{K}'$ and $\tilde{C}$ is the image of $B[\ide{p}_{\tau_0}^{\beta}]$ under $B \to \tilde{B}$. This proves that $q$ and hence $\ide{Y}$ is smooth.\\
    \\
    \textbf{Step 3: (Constructing a finite \'{e}tale map $\ide{X}_{G, w_n}(p^{\beta}) \to \ide{Y}$)}
    
    For $i=2, \dots, n+1$, let $e_i \in \mbb{X}_{\opn{ord}, \tau_0}[p^{\beta}] = \mu_{p^{\beta}} \oplus \left(\mbb{Z}/p^{\beta}\mbb{Z} \right)^{\oplus 2n-1}$ denote the basis vector of the $\mbb{Z}/p^{\beta}\mbb{Z}$-factor in the $i$-th place. Let $C_{\opn{std}} \subset \mbb{X}_{\opn{ord}, \tau_0}[p^{\beta}]$ denote the finite flat subgroup scheme generated by $\{ e_i : i=2, \dots, n+1 \}$. We have a natural map $\ide{IG}_G \to \ide{Y}$ given by sending $(A, \lambda, i, \eta^p, s, f_{\tau})$ to $(A, C)$ (which the induced extra structure) where $C = f_{\tau_0}(C_{\opn{std}})$. This is a fpqc torsor for a subgroup of the form $N^G_{\opn{Iw}}(p^{\beta}) \rtimes \underline{M} \subset J^+_{G, \opn{ord}}$, for some pro-\'{e}tale subgroup scheme $\underline{M} \subset J^+_{G, \opn{ord}}$ which contains $\underline{M^G_{\opn{Iw}}(p^{\beta})}$. We have a factorisation
    \[
    \ide{IG}_G \to \ide{IG}_{G, w_n}(p^{\beta}) \to \ide{X}_{G, w_n}(p^{\beta}) \to \ide{Y} .
    \]
    Since the map $\ide{IG}_{G, w_n} \to \ide{Y}$ is a pro-\'{e}tale $M$-torsor, we see that $\ide{X}_{G, w_n}(p^{\beta}) \to \ide{Y}$ is finite \'{e}tale. Since we have already shown $\ide{Y}$ is smooth, this implies that $\ide{X}_{G, w_n}(p^{\beta})$ is smooth. \\
    \\
    The proof for $\ide{X}_{H, \opn{id}}(p^{\beta})$ is very similar, using the moduli space $\ide{Y} \to \ide{X}_{H}^{\opn{ord}}$ parameterising finite flat subgroup schemes $C \subset A_1[\ide{p}_{\tau_0}^{\beta}]$ which are \'{e}tale locally isomorphic to $\left(\mbb{Z}/p^{\beta}\mbb{Z}\right)^{\oplus n-1}$.
\end{proof}

Let $\mbb{Q}_p^{\opn{cycl}}$ denote the $p$-adic completion of $\mbb{Q}_p(\mu_{p^{\infty}})$ with ring of integers $\mbb{Z}_p^{\opn{cycl}}$, and fix a basis $\varepsilon \in T_p\mu_{p^{\infty}}(\mbb{Z}_p^{\opn{cycl}})$. Let $u \in M_{\mbf{G}}(\mbb{Z}_p)$ denote the element in \cite[Definition 2.4.2]{UFJ}. Then we consider the following element $\gamma = 1 \times \prod_{\tau \in \Psi} \gamma_{\tau} \in J_{G, \opn{ord}}^+(\mbb{Z}_p^{\opn{cycl}})$, where $\gamma_{\tau} = u_{\tau}$ for $\tau \neq \tau_0$ and 
\[
\gamma_{\tau_0} = u_{\tau_0} \cdot \tbyt{1}{x_{\tau_0}}{}{1}
\]
where $x_{\tau_0} \in (T_p\mu_{p^{\infty}}(\mbb{Z}_p^{\opn{cycl}}))^{\oplus 2n-1}$ is the $(1 \times 2n-1)$-matrix whose $n$-th entry is $\varepsilon \in T_p\mu_{p^{\infty}}(\mbb{Z}_p^{\opn{cycl}})$ and the rest are $1 \in T_p\mu_{p^{\infty}}(\mbb{Z}_p^{\opn{cycl}})$. Then we have a morphism
\[
\ide{IG}_{H, \mbb{Z}_p^{\opn{cycl}}} \xrightarrow{\cdot \gamma} \ide{IG}_{G, \mbb{Z}_p^{\opn{cycl}}}
\]
given by right-translation by $\gamma$. If we let $\opn{Unip}(J_{H, \opn{ord}}) \cong \tilde{\mu_{p^{\infty}}}^{\oplus n-1}$ denote the unipotent part of $J_{H, \opn{ord}}$, then this morphism is equivariant for the actions of $\opn{Unip}(J_{H, \opn{ord}})$ on both sides. Furthermore, we have $N_{\diamondsuit}^H(p^{\beta}) = \gamma N^G_{\opn{Iw}}(p^{\beta}) \gamma^{-1} \cap J_{H, \opn{ord}}^+$ and $J_{\diamondsuit}^H(p^{\beta}) = \gamma J_{\opn{Iw}}^G(p^{\beta}) \gamma^{-1} \cap J_{H, \opn{ord}}^+$, so we obtain induced morphisms
\[
\hat{\iota} \colon \ide{IG}_{H, \opn{id}}(p^{\beta}) \to \ide{IG}_{G, w_n}(p^{\beta}), \quad \quad \hat{\iota} \colon \ide{X}_{H, \opn{id}}(p^{\beta}) \to \ide{X}_{G, w_n}(p^{\beta}) 
\]
Both of these morphisms only depend on the image of $\varepsilon$ in $\mu_{p^{\beta}}$, hence descend to morphisms over $\mbb{Z}_p[\mu_{p^{\beta}}]$. The reason for the notation $\hat{\iota}$ will be explained in Proposition \ref{HlociisHasseProp}.

\subsubsection{Differential operators} \label{DiffOpsCcontsubsec}

We now explain how one obtains a unipotent action on the Igusa towers $\ide{IG}_{G, w_n}(p^{\beta})$ and $\ide{IG}_{H, \opn{id}}(p^{\beta})$ following the strategy in \cite{howe2020unipotent}. 

Let $\opn{Unip}(J_{G, \opn{ord}})$ (resp. $\opn{Unip}(J_{H, \opn{ord}})$) denote the unipotent subgroup of $J_{G, \opn{ord}}$ (resp. $J_{H, \opn{ord}}$). Then we have 
\[
\opn{Unip}(J_{G, \opn{ord}})/N_{\opn{Iw}}^G(p^{\beta}) \cong \left( \tilde{\mu_{p^{\infty}}}/p^{\beta}T_p\mu_{p^{\infty}} \right)^{\oplus n} \oplus \left( \tilde{\mu_{p^{\infty}}}/T_p\mu_{p^{\infty}} \right)^{\oplus n-1}
\]
and 
\[
\opn{Unip}(J_{H, \opn{ord}})/N_{\diamondsuit}^H(p^{\beta}) \cong \left( \tilde{\mu_{p^{\infty}}}/p^{\beta}T_p\mu_{p^{\infty}} \right)^{\oplus n-1}
\]
both of which are isomorphic to copies of the formal torus $\widehat{\mbb{G}}_m$.

\begin{definition} \label{DefOfUGbetaUHbeta}
    Let $U_{G, \beta} = (p^{-\beta} \mbb{Z}_p)^{\oplus n} \oplus \mbb{Z}_p^{\oplus n-1}$ and $U_{H, \beta} = (p^{-\beta} \mbb{Z}_p)^{\oplus n-1}$. We view $U_{H, \beta} \subset U_{G, \beta}$ by including in the first $n-1$ factors. Let $C_{\opn{cont}}(U_{\bullet, \beta}, \mbb{Z}_p)$ denote the algebra of continuous functions $U_{\bullet, \beta} \to \mbb{Z}_p$.
\end{definition}

By $p$-adic Fourier theory, we have identifications of $\mbb{Z}_p$-algebras
\begin{align*} 
C_{\opn{cont}}(U_{G, \beta}, \mbb{Z}_p) &= \opn{Hom}_{\mbb{Z}_p}(\ordd(\opn{Unip}(J_{G, \opn{ord}})/N_{\opn{Iw}}^G(p^{\beta})), \mbb{Z}_p) \\ C_{\opn{cont}}(U_{H, \beta}, \mbb{Z}_p) &= \opn{Hom}_{\mbb{Z}_p}(\ordd(\opn{Unip}(J_{H, \opn{ord}})/N_{\diamondsuit}^H(p^{\beta})), \mbb{Z}_p)
\end{align*}
where the algebra action on the right-hand side is induced from the co-algebra structure on $\mathcal{O}(\cdots)$.  Indeed, by a change of coordinates these identifications are induced from the identification $C_{\opn{cont}}(\mbb{Z}_p, \mbb{Z}_p) = \opn{Hom}_{\mbb{Z}_p}(\mathcal{O}(\widehat{\mbb{G}}_m), \mbb{Z}_p)$, which in turn arises from the Amice transform identifying sections $\mathcal{O}(\widehat{\mbb{G}}_m) \cong \mbb{Z}_p[\![T]\!]$ with measures on $\mbb{Z}_p$ (see \cite[Corollaire I.2.4]{ColmezFonctions}; or for a more general result for families of $p$-divisible groups stated in a similar form as above, see \cite[\S 7]{FourierTheory}).

Suppose that $R$ is $p$-adically complete and separated $\mbb{Z}_p$-algebra and let $\tilde{\zeta} = (\tilde{\zeta}_i) \in \tilde{\mu_{p^{\infty}}}^{\oplus 2n-1}(R)$. We can view $\tilde{\zeta}_i = (\zeta_{i, k})_{k \geq 0} \in \varprojlim_{\times p} \mu_{p^{\infty}}(R)$ where $\zeta_{i, k} \in \mu_{p^{\infty}}(R)$ and $\zeta_{i, k+1}^p = \zeta_{i, k}$.\footnote{Note that we are viewing $\mu_{p^{\infty}}$ as a $p$-divisible group over $\opn{Spf}\mbb{Z}_p$. In particular, $\mu_{p^{\infty}}(R) = 1 + R^{00}$ where $R^{00}$ denotes the topological nilpotent elements. C.f. \cite[Remark 2.1.3]{howe2020unipotent}.} Consider the continuous function $\chi_{\tilde{\zeta}} \colon U_{G, \beta} \to R$ given by 
\[
\chi_{\tilde{\zeta}}( \frac{a_1}{p^{\beta}}, \dots, \frac{a_n}{p^{\beta}}, a_{n+1}, \dots, a_{2n-1}) = \zeta_{1, \beta}^{a_1} \cdots \zeta_{n, \beta}^{a_n} \cdot \zeta_{n+1, 0}^{a_{n+1}} \cdots \zeta_{2n-1, 0}^{a_{2n-1}}
\]
which only depends on the image of $\tilde{\zeta}$ in $\opn{Unip}(J_{G, \opn{ord}})/N_{\opn{Iw}}^G(p^{\beta})$. Then the above identifications are normalised so that $\chi_{\tilde{\zeta}}$ corresponds to the homomorphism $\mathcal{O}(\opn{Unip}(J_{G, \opn{ord}})/N_{\opn{Iw}}^G(p^{\beta})) \to R$ given by evaluating a section on $\tilde{\zeta}$. We have a similar description for $H$.

Since $J_{\bullet, \opn{ord}}$ acts on $\ide{IG}_{\bullet}$, we obtain right actions 
\begin{align*}
    \opn{Unip}(J_{G, \opn{ord}})/N_{\opn{Iw}}^G(p^{\beta}) \times \ide{IG}_{G, w_n}(p^{\beta}) &\to \ide{IG}_{G, w_n}(p^{\beta}) \\
    \opn{Unip}(J_{H, \opn{ord}})/N_{\diamondsuit}^H(p^{\beta}) \times \ide{IG}_{H, \opn{id}}(p^{\beta}) &\to \ide{IG}_{H, \opn{id}}(p^{\beta}) .
\end{align*}
Therefore, by passing to the associated co-actions and using the identifications above, we obtain $\mbb{Z}_p$-algebra actions
\begin{align*}
    C_{\opn{cont}}(U_{G, \beta}, \mbb{Z}_p) \times \mathcal{O}_{\ide{IG}_{G, w_n}(p^{\beta})} &\to \mathcal{O}_{\ide{IG}_{G, w_n}(p^{\beta})} \\
    C_{\opn{cont}}(U_{H, \beta}, \mbb{Z}_p) \times \mathcal{O}_{\ide{IG}_{H, \opn{id}}(p^{\beta})} &\to \mathcal{O}_{\ide{IG}_{H, \opn{id}}(p^{\beta})}.
\end{align*}
In particular, the action of $\chi_{\tilde{\zeta}}$ corresponds right-translation of a section by $\tilde{\zeta}$.

Finally, we note that the pullback map $\ordd_{\ide{IG}_{G, w_n}(p^{\beta})} \to \hat{\iota}_* \mathcal{O}_{\ide{IG}_{H, \opn{id}}(p^{\beta})}$ is equivariant for the action of $C_{\opn{cont}}(U_{H, \beta}, \mbb{Z}_p)$, where we view $C_{\opn{cont}}(U_{H, \beta}, \mbb{Z}_p) \subset C_{\opn{cont}}(U_{G, \beta}, \mbb{Z}_p)$ in the natural way (via the split inclusion $U_{H, \beta} \subset U_{G, \beta}$).

\subsection{Integral models and overconvergent neighbourhoods} \label{IntegralModelsandOCNbhdsSec}

Let $\mathcal{X}_{G, \opn{Iw}}(p^{\beta})$ and $\mathcal{X}_{H, \diamondsuit}(p^{\beta})$ denote the smooth proper adic spaces over $\mbb{Q}_p$ associated with $X_{\mbf{G}, \opn{Iw}, \mbb{Q}_p}(p^{\beta})$ and $X_{\mbf{H}, \diamondsuit, \mbb{Q}_p}(p^{\beta})$ respectively. In this section we will construct certain integral models for these adic spaces, via normalisation, which will be useful in the $p$-adic interpolation of differential operators (see \S \ref{PadicIterationOfDiffOpsSection}). Furthermore, we will explain how $\ide{X}_{G, w_n}(p^{\beta})$ and $\ide{X}_{H, \opn{id}}(p^{\beta})$ appear as ordinary strata inside these formal models. To be able to do this for $H$, it will be necessary to work over a finite extension $L/\mbb{Q}_p$ which contains $\mu_{p^{\beta}}$. \emph{We therefore assume that all of the spaces we consider in this section have been base-changed to $L$ or $\mathcal{O}_L$, however we will omit this from the notation.}

\subsubsection{Integral models for \texorpdfstring{$G$}{G}}

Fix $\beta \geq 1$ (and $L/\mbb{Q}_p$ as above). Recall from Definition \ref{DefinitionOfDeeperLevelAtpVarieties} that $\mathcal{X}_{G, \opn{Iw}}(p^{\beta})$ parameterises flags
\begin{equation} \label{ShapeOfFlags}
0 = C_{0, \tau} \subset C_{1, \tau} \subset \cdots \subset C_{2n, \tau} = A[\ide{p}_{\tau}^{\beta}] 
\end{equation}
of finite flat group schemes $C_{i, \tau}$ of order $p^{i \beta}$, such that each graded piece $C_{i, \tau}/C_{i-1, \tau}$ is cyclic (of order $p^{\beta}$). From the moduli description of $\ide{IG}_G$, we also see that $\ide{X}_{G, w_n}(p^{\beta})$ parameterises ordinary $\boldsymbol\Psi$-abelian schemes $A$ (with extra structure) and flags of finite flat subgroups as in (\ref{ShapeOfFlags}), with
\begin{itemize}
    \item $C_{i, \tau}/C_{i-1, \tau}$ is \'{e}tale locally isomorphic to $\mbb{Z}/p^{\beta}\mbb{Z}$ for $(i, \tau) \neq (n+1, \tau_0)$
    \item $C_{n+1, \tau_0}/C_{n, \tau_0}$ is \'{e}tale locally isomorphic to $\mu_{p^{\beta}}$.
\end{itemize}
We therefore clearly have a morphism $\mathcal{X}_{G, w_n}(p^{\beta}) \to \mathcal{X}_{G, \opn{Iw}}(p^{\beta})$.  We claim that this is an open immersion and can be identified with the locus where a certain invertible $\mathcal{O}_{\mathcal{X}_{G, \opn{Iw}}(p^{\beta})}^+$-module coincides with the structural sheaf $\mathcal{O}_{\mathcal{X}_{G, \opn{Iw}}(p^{\beta})}^+$. This will allow us to explicitly describe overconvergent neighbourhoods of $\mathcal{X}_{G, w_n}(p^{\beta})$ in $\mathcal{X}_{G, \opn{Iw}}(p^{\beta})$ (see Lemma \ref{XGwnisOpenImmersionLemma} and Definition \ref{Def:XGwnOCneighbourhood}).

To prove this, we need to consider the following morphisms:

\begin{definition}
    For any $(i, \tau)$, let $A_{i, \tau} \defeq A/C_{i, \tau}'$, where $C_{i, \tau}' = C_{i, \tau} \oplus C_{i, \tau}^{\perp} \subset A[p^{\beta}]$ with $C_{i, \tau}^{\perp} \subset A[\ide{p}_{\tau^c}^{\beta}]$ the orthogonal complement of $C_{i, \tau}$ under the Weil pairing. Note that the extra structure on $A$ naturally descends to $A_{i, \tau}$. We therefore obtain a finite morphism
    \begin{align*}
        q_{i, \tau} \colon \mathcal{X}_{G, \opn{Iw}}(p^{\beta}) &\to \mathcal{X}_G \\
        (A, C_{\bullet, \bullet}) &\mapsto A_{i, \tau} .
    \end{align*}
    Let $q \colon \mathcal{X}_{G, \opn{Iw}}(p^{\beta}) \to \mathcal{X}_G \times \prod_{(i, \tau)} \mathcal{X}_G$ denote the induced finite morphism, where the map to the first factor is the forgetful map and $\opn{pr}_{i, \tau} \circ q = q_{i, \tau}$, where $\opn{pr}_{i, \tau}$ denotes projection to the $(i, \tau)$-factor.
    
    We define $\ide{X}_{G, \opn{Iw}}(p^{\beta})$ to be the normalisation of $\ide{X}_G \times \prod_{(i, \tau)} \mathfrak{X}_G$ under the morphism $q \colon \mathcal{X}_{G, \opn{Iw}}(p^{\beta}) \to \mathcal{X}_G \times \prod_{(i, \tau)} \mathcal{X}_G$.
\end{definition}

We have the following lemma which says that the universal flags of finite flat subgroup schemes over $\mathcal{X}_{G, \opn{Iw}}(p^{\beta})$ extend to the normalisation. 

\begin{lemma} \label{ExtensionOfFFfiltrationLemma}
    Let $\mathcal{A}$ denote the pullback of the universal $\boldsymbol\Psi$-unitary abelian scheme under the map $\ide{X}_{G, \opn{Iw}}(p^{\beta}) \to \ide{X}_G \times \prod_{(i, \tau)} \mathfrak{X}_G \to \ide{X}_G$, where the second map is projection to the first component. Then there exist flags
    \[
    0 = \mathcal{C}_{0, \tau} \subset \mathcal{C}_{1, \tau} \subset \cdots \subset \mathcal{C}_{2n, \tau} = \mathcal{A}[\ide{p}_{\tau}^{\beta}]
    \]
    of finite flat subgroup schemes $\mathcal{C}_{i, \tau}$ of order $p^{i \beta}$ which agree with the universal flags (as in (\ref{ShapeOfFlags})) on the adic generic fibre.
\end{lemma}
\begin{proof}
    Let $\mathcal{A}_{i, \tau}$ denote the pullback of the universal $\boldsymbol\Psi$-unitary abelian scheme under the morphism
    \[
    \ide{X}_{G, \opn{Iw}}(p^{\beta}) \to \ide{X}_G \times \prod_{(i, \tau)} \mathfrak{X}_G \xrightarrow{\opn{pr}_{i, \tau}} \ide{X}_G
    \]
    where the second map is projection to the $(i, \tau)$-component. This agrees with the abelian scheme $A_{i, \tau}$ on the adic generic fibre. We claim that the isogeny $A \to A_{i, \tau}$ extends to an isogeny $\mathcal{A}[p^{\infty}] \to \mathcal{A}_{i, \tau}[p^{\infty}]$ of $p$-divisible groups. If we show this, then we can define the finite flat subgroup scheme $\mathcal{C}_{i, \tau}$ to be the kernel of the induced isogeny $\mathcal{A}[\ide{p}_{\tau}^{\infty}] \to \mathcal{A}_{i, \tau}[\ide{p}_{\tau}^{\infty}]$ of $p$-divisible groups.

    Let $\beta' \geq \beta$ be an integer and let $L \subset A[p^{\beta'}] \oplus A_{i, \tau}[p^{\beta'}]$ denote the graph of the morphism induced from the isogeny $A \to A_{i, \tau}$. Note that projection to the first factor induces an isomorphism $L \xrightarrow{\sim} A[p^{\beta'}]$. Let $\mathcal{L} \subset \mathcal{A}[p^{\beta'}] \oplus \mathcal{A}_{i, \tau}[p^{\beta'}]$ denote the Zariski closure of $L$ in $\mathcal{A}[p^{\beta'}] \oplus \mathcal{A}_{i, \tau}[p^{\beta'}]$, which is a closed subscheme. Consider the induced morphism $\mathcal{L} \to \mathcal{A}[p^{\beta'}]$ given by projecting to the first factor. Since this an isomorphism on generic fibres, the induced map $\mathcal{O}_{\mathcal{A}[p^{\beta'}]} \to \mathcal{O}_\mathcal{L}$ is injective. We wish to show this map is surjective, hence an isomorphism. By Nakayama's lemma, it is enough to check this after specialising at rank one points.

    Let $x \colon \opn{Spa}(K, \mathcal{O}_K) \to \mathcal{X}_{G, \opn{Iw}}(p^{\beta})$ be a rank one point inducing a morphism $\mathcal{O}_K \to \ide{X}_{G, \opn{Iw}}(p^{\beta})$. We will denote base-change along this morphism by adding a subscript $x$ to the object. Then $\mathcal{L}_x \subset \mathcal{A}_x[p^{\beta'}] \oplus \mathcal{A}_{i, \tau, x}[p^{\beta'}]$ is finite flat and equal to the Zariski closure of the graph of the morphism $A_x[p^{\beta'}] \to A_{i, \tau, x}[p^{\beta'}]$. Let $\opn{deg}(-)$ denote the degree of a finite flat subgroup scheme over $\mathcal{O}_K$ as defined in \cite{FarguesHN}. The morphism $\mathcal{L}_x \to \mathcal{A}_x[p^{\beta'}]$ is an isomorphism on generic fibres, hence we must have 
    \begin{equation} \label{degreeEqn}
    \opn{deg}(\mathcal{L}_x) \leq \opn{deg}(\mathcal{A}_x[p^{\beta'}]) .
    \end{equation}
    On the other hand, let $N$ denote the kernel of the isogeny $A_x \to A_{i, \tau, x}$ and let $\mathcal{N} \subset A_x[p^{\beta}]$ denote its Zariski closure, which is a finite flat subgroup scheme. Then $\mathcal{A}_x/\mathcal{N}$ is an integral model for $A_x/N \cong A_{i, \tau, x}$ over $\mathcal{O}_K$. By the valuative criterion of properness for $\ide{X}_{G, \opn{Iw}}(p^{\beta})$, we must have $\mathcal{A}_x/ \mathcal{N} \cong \mathcal{A}_{i, \tau, x}$, hence the isogeny $A_x \to A_{i, \tau, x}$ extends to an isogeny $f \colon \mathcal{A}_x \to \mathcal{A}_{i, \tau, x}$. We obtain an induced morphism
    \[
    \left(\mathcal{A}_x[p^{\beta'}] \oplus \mathcal{A}_{i, \tau, x}[p^{\beta'}] \right)/\mathcal{L}_x \xrightarrow{f - \opn{id}} \mathcal{A}_{i, \tau, x}[p^{\beta'}]
    \]
    which is an isomorphism on generic fibres. This implies that (\ref{degreeEqn}) is in fact an equality, and so the map $\mathcal{L}_x \to \mathcal{A}_x[p^{\beta'}]$ (and hence $\mathcal{L} \to \mathcal{A}[p^{\beta'}]$) is an isomorphism.

    Finally, for any $\beta'$ we obtain an induced morphism
    \[
    \mathcal{A}[p^{\beta'}] \xleftarrow{\sim} \mathcal{L} \to \mathcal{A}[p^{\beta'}] \oplus \mathcal{A}_{i, \tau}[p^{\beta'}] \to \mathcal{A}_{i, \tau}[p^{\beta'}]
    \]
    where the last map is projection to the second component. These maps are compatible as $\beta'$ varies, hence we obtain an induced isogeny $\mathcal{A}[p^{\infty}] \to \mathcal{A}_{i, \tau}[p^{\infty}]$ as required.
\end{proof}

We can use the previous lemma to define certain  invertible $\mathcal{O}_{\ide{X}_{G, \opn{Iw}}(p^{\beta})}$-modules which cut out $\ide{X}_{G, w_n}(p^{\beta})$. With notation as in the proof of Lemma \ref{ExtensionOfFFfiltrationLemma}, we set $G_{\tau} \defeq \mathcal{A}[\ide{p}_{\tau}^{\infty}]$ and $G_{i, \tau} \defeq \mathcal{A}_{i, \tau}[\ide{p}_{\tau}^{\infty}]$, which are both $p$-divisible groups of height $2n$ and dimension $1$ (resp. $0$) if $\tau = \tau_0$ (resp. $\tau \neq \tau_0$) by the signature condition. Note that $\mathcal{C}_{i, \tau}$ was constructed as the kernel of an isogeny $G_{\tau} \to G_{i, \tau}$. We also note that we have factorisations:
\begin{equation} \label{ChainOfGitau}
G_{\tau} = G_{0, \tau} \to G_{1, \tau} \to G_{2, \tau} \to \dots \to G_{2n-1, \tau} \to G_{\tau} = G_{2n, \tau}
\end{equation}
where the total composition is multiplication by $p^{\beta}$ (so has kernel $G_{\tau}[p^{\beta}] = \mathcal{A}[\ide{p}_{\tau}^{\beta}]$).

\begin{definition} \label{DefinitionOfDeltaHasseIdeals}
    For $i=1, \dots, 2n$, let 
    \begin{equation} \label{deltaimap}
    \mathcal{O}_{\ide{X}_{G, \opn{Iw}}(p^{\beta})} \to \opn{det}\omega_{G_{i-1, \tau_0}} \otimes \opn{det}\omega_{G_{i, \tau_0}}^{\vee}
    \end{equation}
    denote the morphism arising from the determinant of the map $\omega_{G_{i, \tau_0}} \to \omega_{G_{i-1, \tau_0}}$ induced from the isogeny $G_{i-1, \tau_0} \to G_{i, \tau_0}$. Let $\delta_{G, i} \subset \mathcal{O}_{\ide{X}_{G, \opn{Iw}}(p^{\beta})}$ denote the annihilator of the cokernel of (\ref{deltaimap}), which is an invertible ideal.
\end{definition}

These ideals satisfy the following properties.

\begin{lemma} \label{LemmaOnDeltaHasseIdeals}
    We have:
    \begin{enumerate}
        \item $\prod_{i=1}^{2n} \delta_{G, i} = p^{\beta} \mathcal{O}_{\ide{X}_{G, \opn{Iw}}(p^{\beta})}$
        \item For any rank one point $x \colon \opn{Spa}(K, \mathcal{O}_K) \to \mathcal{X}_{G, \opn{Iw}}(p^{\beta})$, we have $|\delta_{G, i}|_x = 1$ if and only if $\mathcal{C}_{i, \tau_0, x}/\mathcal{C}_{i-1, \tau_0, x}$ is \'{e}tale locally isomorphic to $\mbb{Z}/p^{\beta}\mbb{Z}$
        \item For any rank one point $x \colon \opn{Spa}(K, \mathcal{O}_K) \to \mathcal{X}_{G, \opn{Iw}}(p^{\beta})$, we have $|\delta_{G, i}|_x = |p^{\beta}|_x$ if and only if $\mathcal{C}_{i, \tau_0, x}/\mathcal{C}_{i-1, \tau_0, x}$ is \'{e}tale locally isomorphic to $\mu_{p^{\beta}}$.
    \end{enumerate}
\end{lemma}
\begin{proof}
    Part (1) follows from the fact that the composition of maps in (\ref{ChainOfGitau}) is equal to multiplication by $p^{\beta}$. For parts (2) and (3), we may normalise so that $|p|_x = p^{-1}$. Then the degree of $\mathcal{C}_{i, \tau_0,x}/\mathcal{C}_{i-1, \tau_0,x}$ satisfies
    \[
    p^{-\opn{deg}(\mathcal{C}_{i, \tau_0,x}/\mathcal{C}_{i-1, \tau_0,x})} = |\delta_{G, i}|_x  .
    \]
    In particular, $|\delta_{G, i}|_x = 1$ (resp. $|\delta_{G, i}|_x = p^{-\beta}$) if and only if $\mathcal{C}_{i, \tau_0,x}/\mathcal{C}_{i-1, \tau_0,x}$ is \'{e}tale (resp. multiplicative), since the height of $\mathcal{C}_{i, \tau_0,x}/\mathcal{C}_{i-1, \tau_0,x}$ is $\beta$ (see \cite[\S 3, Exemple 2]{FarguesHN}). We obtain the claim about cyclicity because, on the generic fibre, $C_{i, \tau_0, x}/C_{i-1, \tau_0, x}$ is \'{e}tale locally isomorphic to $\mbb{Z}/p^{\beta}\mbb{Z}$.
\end{proof}

Set $\hat{\delta}_{G, n+1} = \prod_{i \neq n+1} \delta_{G, i}$. 

\begin{lemma} \label{XGwnisOpenImmersionLemma}
    The morphism $\mathcal{X}_{G, w_n}(p^{\beta}) \to \mathcal{X}_{G, \opn{Iw}}(p^{\beta})$ extends to an open immersion $\ide{X}_{G, w_n}(p^{\beta}) \hookrightarrow \ide{X}_{G, \opn{Iw}}(p^{\beta})$ whose image is identified with the open subscheme $V \subset \ide{X}_{G, \opn{Iw}}(p^{\beta})$ where $\hat{\delta}_{G, n+1}|_V = \mathcal{O}_V$. In particular, the morphism $\mathcal{X}_{G, w_n}(p^{\beta}) \to \mathcal{X}_{G, \opn{Iw}}(p^{\beta})$ is an open immersion.
\end{lemma}
\begin{proof}
Consider the chain of isogenies 
\[
G_{0, \tau_0} \to G_{1, \tau_0} \to \cdots \to G_{2n, \tau_0}
\]
over $\mathcal{X}_{G, \opn{Iw}}(p^{\beta})$, where $G_{i, \tau_0} = A_{i, \tau_0}[\ide{p}_{\tau_0}^{\infty}]$. We can define an analytic version of the ideals in Definition \ref{DefinitionOfDeltaHasseIdeals} as follows. Consider the induced morphism $\omega_{G_{i, \tau_0}}^+ \to \omega_{G_{i-1, \tau_0}}^+$ on integral invariant differentials, and let $\mathcal{O}^+_{\mathcal{X}_{G, \opn{Iw}}(p^{\beta})} \to \opn{det}\omega_{G_{i-1, \tau_0}}^+ \otimes \opn{det}\omega_{G_{i, \tau_0}}^{+, \vee}$ denote the morphism induced from its determinant. Let $\delta^+_{G, i} \subset \mathcal{O}^+_{\mathcal{X}_{G, \opn{Iw}}(p^{\beta})}$ denote the annihilator of the cokernel of this morphism, which is an invertible ideal. We set $\hat{\delta}^+_{G, n+1} = \prod_{i \neq n+1} \delta^+_{G, i}$. 

Let $V_{\eta}$ denote the adic generic fibre of $V$. Then $V_{\eta} \subset \mathcal{X}_{G, \opn{Iw}}(p^{\beta})$ is identified with the quasi-compact open subspace where $\hat{\delta}^+_{G, n+1}|_{V_{\eta}} = \mathcal{O}_{V_{\eta}}^+$. Consider the morphism $f \colon \mathcal{X}_{G, w_n}(p^{\beta}) \to \mathcal{X}_{G, \opn{Iw}}(p^{\beta})$. Then, from the moduli description of $\ide{X}_{G, w_n}(p^{\beta})$ and Lemma \ref{LemmaOnDeltaHasseIdeals}, we see that $f^* \hat{\delta}_{G, n+1}^+ = \mathcal{O}_{\mathcal{X}_{G, w_n}(p^{\beta})}^+$, hence $f$ factors through the inclusion $V_{\eta} \subset \mathcal{X}_{G, \opn{Iw}}(p^{\beta})$.

On the other hand, the pullback $\mathcal{C}_{n+1, \tau_0}/\mathcal{C}_{n, \tau_0}|_V$ must be \'{e}tale locally isomorphic to $\mu_{p^{\beta}}$, which implies that $\mathcal{A}|_{V}$ is ordinary and the inclusion $V_{\eta} \subset \mathcal{X}_{G, \opn{Iw}}(p^{\beta})$ factors as
\[
V_{\eta} \xrightarrow{g} \mathcal{X}_{G, w_n}(p^{\beta}) \xrightarrow{f} \mathcal{X}_{G, \opn{Iw}}(p^{\beta}) .
\]
This implies that $f \colon \mathcal{X}_{G, w_n}(p^{\beta}) \to \mathcal{X}_{G, \opn{Iw}}(p^{\beta})$ is identified with the open immersion $V_{\eta} \subset \mathcal{X}_{G, \opn{Iw}}(p^{\beta})$. Note that the morphism $g$ extends to a morphism $V \xrightarrow{g} \ide{X}_{G, w_n}(p^{\beta})$. 

We have a natural map $V_{\eta} \cong \mathcal{X}_{G, w_n}(p^{\beta}) \to \mathcal{X}_{G}^{\opn{ord}} \times \prod_{(i, \tau)} \mathcal{X}_{G}^{\opn{ord}}$, and $V$ is identified with the normalisation of $\ide{X}_{G}^{\opn{ord}} \times \prod_{(i, \tau)}\ide{X}_{G}^{\opn{ord}}$ under this map. Since the resulting map $V \to \ide{X}_{G}^{\opn{ord}} \times \prod_{(i, \tau)}\ide{X}_{G}^{\opn{ord}}$ factors as
\[
V \xrightarrow{g} \ide{X}_{G, w_n}(p^{\beta}) \to \ide{X}_{G}^{\opn{ord}} \times \prod_{(i, \tau)}\ide{X}_{G}^{\opn{ord}}
\]
and $\ide{X}_{G, w_n}(p^{\beta})$ is integrally closed in its generic fibre (see Proposition \ref{PropositionLociIntClosedInGenFibre}), the map $g \colon V \to \ide{X}_{G, w_n}(p^{\beta})$ must be an isomorphism. This completes the proof.
\end{proof}

As a consequence of this lemma, we can define overconvergent neighbourhoods of $\mathcal{X}_{G, w_n}(p^{\beta})$ inside $\mathcal{X}_{G, \opn{Iw}}(p^{\beta})$. We will also need a certain closed subset of $\mathcal{X}_{G, \opn{Iw}}(p^{\beta})$ containing $\mathcal{X}_{G, w_n}(p^{\beta})$ to define the support conditions for coherent cohomology. We summarise this in the following definition.

\begin{definition} \label{Def:XGwnOCneighbourhood}
    Let $r \geq 1$ be an integer. 
    \begin{enumerate}
        \item With notation as in the proof of Lemma \ref{XGwnisOpenImmersionLemma}, let $\mathcal{X}_{G, w_n}(p^{\beta})_r \subset \mathcal{X}_{G, \opn{Iw}}(p^{\beta})$ denote the quasi-compact open rational subset defined by the inequality $|\hat{\delta}_{G, n+1}^+|^{p^{r+1}} \geq |p|$. This is a strict neighbourhood of $\mathcal{X}_{G, w_n}(p^{\beta})$, i.e. it contains the closure of $\mathcal{X}_{G, w_n}(p^{\beta})$ in $\mathcal{X}_{G, \opn{Iw}}(p^{\beta})$.
        \item Let $\delta_{G, > n+1}^+ = \prod_{i > n+1} \delta^+_{G, i}$. Let $\mathcal{Z}_{G, >n+1}(p^{\beta}) \subset \mathcal{X}_{G, \opn{Iw}}(p^{\beta})$ denote the closure of the quasi-compact open subset defined by the condition $|\delta_{G, > n+1}^+| = 1$. Clearly we have $\mathcal{X}_{G, w_n}(p^{\beta}) \subset \mathcal{Z}_{G, > n+1}(p^{\beta})$.
    \end{enumerate}
\end{definition}

\begin{remark} \label{RemarkOCFormalIntModelsG}
    The open $\mathcal{X}_{G, w_n}(p^{\beta})_r$ has an integral model $\ide{X}_{G, w_n}(p^{\beta})_r$ given by an open in the formal admissible blow-up of $\ide{X}_{G, \opn{Iw}}(p^{\beta})$ along the ideal $\hat{\delta}_{G, n+1}^{p^{r+1}} + (p)$, where the open is defined by the condition that the pulback of $\hat{\delta}_{G, n+1}^{p^{r+1}} + (p)$ is generated by the pullback of $\hat{\delta}_{G, n+1}^{p^{r+1}}$ (see \cite[\S 8.2, Proposition 7]{Bosch}). Furthermore, the subset $\mathcal{Z}_{G, > n+1}(p^{\beta})$ should be thought of as the closure of the locus where $A_x[\ide{p}_{\tau_0}^{\beta}]/C_{n+1, \tau_0, x}$ extends to an \'{e}tale group scheme over $\mathcal{O}_K$, for any rank one point $x \colon \opn{Spa}(K, \mathcal{O}_K) \to \mathcal{X}_{G, \opn{Iw}}(p^{\beta})$.
\end{remark}

\subsubsection{Integral models for \texorpdfstring{$H$}{H}}

We now define an integral model for $\mathcal{X}_{H, \diamondsuit}(p^{\beta})$ for which $\ide{X}_{H, \opn{id}}(p^{\beta})$ is an open subspace cut out by a certain  invertible $\mathcal{O}_{\ide{X}_{H, \diamondsuit}(p^{\beta})}$-module. We remind the reader that we are working over a finite extension $L/\mbb{Q}_p$ containing $\mu_{p^{\beta}}$.

\begin{definition}
    We define $\ide{X}_{H, \diamondsuit}(p^{\beta})$ to be the normalisation of $\ide{X}_{G, \opn{Iw}}(p^{\beta})$ under the finite map $\hat{\iota} \colon \mathcal{X}_{H, \diamondsuit}(p^{\beta}) \to \mathcal{X}_{G, \opn{Iw}}(p^{\beta})$. We also denote the resulting finite map $\ide{X}_{H, \diamondsuit}(p^{\beta}) \to \ide{X}_{G, \opn{Iw}}(p^{\beta})$ by $\hat{\iota}$.
\end{definition}

\begin{proposition} \label{HlociisHasseProp}
    Let $\delta_{H, i} = \hat{\iota}^* \delta_{G, i} \subset \mathcal{O}_{\ide{X}_{H, \diamondsuit}(p^{\beta})}$ and $\delta_{H, i}^+ = \hat{\iota}^* \delta_{G, i}^+ \subset \mathcal{O}_{\mathcal{X}_{H, \diamondsuit}(p^{\beta})}^+$.
    \begin{enumerate}
        \item For any $1 \leq i \leq n$, we have $\delta_{H, i} = \mathcal{O}_{\ide{X}_{H, \diamondsuit}(p^{\beta})}$ and $\delta_{H, i}^+ = \mathcal{O}^+_{\mathcal{X}_{H, \diamondsuit}(p^{\beta})}$.
        \item We have a Cartesian diagram
        \[
\begin{tikzcd}
{\mathcal{X}_{H, \opn{id}}(p^{\beta})} \arrow[d, hook] \arrow[r, "\hat{\iota}"] & {\mathcal{X}_{G, w_n}(p^{\beta})} \arrow[d, hook] \\
{\mathcal{X}_{H, \diamondsuit}(p^{\beta})} \arrow[r, "\hat{\iota}"]             & {\mathcal{X}_{G, \opn{Iw}}(p^{\beta})}           
\end{tikzcd}
        \]
        where the top map is defined at the end of \S \ref{QuotientsOfIgVarSubSec}. In particular, $\mathcal{X}_{H, \opn{id}}(p^{\beta})$ is identified with the quasi-compact open in $\mathcal{X}_{H, \diamondsuit}(p^{\beta})$ defined by the condition $|\hat{\delta}_{H, n+1}^+| = 1$ (here $\hat{\delta}_{H, n+1}^+ = \prod_{i \neq n+1} \delta_{H, i}^+ = \prod_{i > n+1} \delta_{H, i}^+$).
        \item The Cartesian diagram in (2) extends to a Cartesian diagram
        \[
\begin{tikzcd}
{\mathfrak{X}_{H, \opn{id}}(p^{\beta})} \arrow[d, hook] \arrow[r, "\hat{\iota}"] & {\mathfrak{X}_{G, w_n}(p^{\beta})} \arrow[d, hook] \\
{\mathfrak{X}_{H, \diamondsuit}(p^{\beta})} \arrow[r, "\hat{\iota}"]             & {\mathfrak{X}_{G, \opn{Iw}}(p^{\beta})}           
\end{tikzcd}
        \]
        on formal models, where the top map is defined at the end of \S \ref{QuotientsOfIgVarSubSec}. In particular, $\ide{X}_{H, \opn{id}}(p^{\beta})$ is identified with the open in $\ide{X}_{H, \diamondsuit}(p^{\beta})$ where $\hat{\delta}_{H, n+1} = \prod_{i \neq n+1} \delta_{H, i} = \prod_{i > n+1} \delta_{H, i}$ equals the structure sheaf.
    \end{enumerate}
\end{proposition}
\begin{proof}
    Let $p \colon \ide{X}_{G, \opn{Iw}}(p^{\beta}) \to \ide{X}_G \times \prod_{(i, \tau)} \ide{X}_G \to \ide{X}_G$ denote the forgetful map, where the second map is projection to the first factor. We use similar notation for the map on adic generic fibres. Then, since $\ide{X}_{G, \opn{Iw}}(p^{\beta})$ maps to normalisation of $\ide{X}_G$ under $p \colon \mathcal{X}_{G, \opn{Iw}}(p^{\beta}) \to \mathcal{X}_G$, the commutative diagram
    \[
\begin{tikzcd}
{\mathcal{X}_{H, \diamondsuit}(p^{\beta})} \arrow[d] \arrow[r, "\hat{\iota}"] & {\mathcal{X}_{G, \opn{Iw}}(p^{\beta})} \arrow[d, "p"] \\
\mathcal{X}_H \arrow[r]                                                       & \mathcal{X}_G                                        
\end{tikzcd}
    \]
    extends to a commutative diagram on formal integral models. In particular, the pullback of $\mathcal{A}$ along $\hat{\iota} \colon \ide{X}_{H, \diamondsuit}(p^{\beta}) \to \ide{X}_{G, \opn{Iw}}(p^{\beta})$ coincides with the pullback of $\mathcal{A}_1 \oplus \mathcal{A}_2$ under the map $\ide{X}_{H, \diamondsuit}(p^{\beta}) \to \ide{X}_H$, where $\mathcal{A}_1$ (resp. $\mathcal{A}_2$) denotes the universal $\boldsymbol\Psi_1$-unitary (resp. $\boldsymbol\Psi_2$-unitary) abelian scheme over $\ide{X}_H$.

    For part (1), it is enough to prove this for the $+$-sheaves at rank one points. Let $x \colon \opn{Spa}(K, \mathcal{O}_K) \to \mathcal{X}_{H, \diamondsuit}(p^{\beta})$ be a rank one point and let $y = \hat{\iota} \circ x$. Consider the corresponding filtration determined by the point $y$:
    \[
    0 = C_{0, \tau_0, y} \subset C_{1, \tau_0, y} \subset \cdots \subset C_{2n, \tau_0, y} = A_{1, x}[\ide{p}_{\tau_0}^{\beta}] \oplus A_{2, x}[\ide{p}_{\tau_0}^{\beta}] = A_y[\ide{p}_{\tau_0}^{\beta}] .
    \]
    Consider the \'{e}tale group scheme $\left(\mbb{Z}/p^{\beta}\mbb{Z} \right)^{\oplus n} \oplus \left(\mbb{Z}/p^{\beta}\mbb{Z} \right)^{\oplus n}$ with standard basis denoted $e_1, \dots, e_{2n}$. The point $x$ gives rise to a certain orbit of isomorphisms
    \[
    \phi \colon \left(\mbb{Z}/p^{\beta}\mbb{Z} \right)^{\oplus n} \oplus \left(\mbb{Z}/p^{\beta}\mbb{Z} \right)^{\oplus n} \cong A_{1, x}[\ide{p}_{\tau_0}^{\beta}] \oplus A_{2, x}[\ide{p}_{\tau_0}^{\beta}]
    \]
    under the $\tau_0$-component of the group $T^{\diamondsuit}(\mbb{Z}/p^{\beta}\mbb{Z}) = \hat{\gamma} B_G(\mbb{Z}/p^{\beta}\mbb{Z}) \hat{\gamma}^{-1} \cap H(\mbb{Z}/p^{\beta}\mbb{Z})$. In particular, by using the explicit description of $\hat{\gamma}$ in Definition \ref{NewDefOfGamma}, we see that $C_{n, \tau_0, y}$ is the subgroup generated by the elements:
    \[
    \{ \phi(e_i) + \phi(e_{2n+2-i}) : i=2, \dots, n \} \cup \{ \phi(e_1) + \phi(e_{n+1}) \} .
    \]
    In particular, we see that the induced maps
    \[
    r_i \colon C_{n, \tau_0, y} \subset A_{1, x}[\ide{p}_{\tau_0}^{\beta}] \oplus A_{2, x}[\ide{p}_{\tau_0}^{\beta}] \twoheadrightarrow A_{i, x}[\ide{p}_{\tau_0}^{\beta}], \quad \quad i=1, 2
    \]
    are isomorphisms, so by Goursat's lemma for \'{e}tale group schemes in characteristic zero, we see that $C_{n, \tau_0, y}$ is the graph of an isomorphism $\alpha \colon A_{2, x}[\ide{p}_{\tau_0}^{\beta}] \xrightarrow{\sim} A_{1, x}[\ide{p}_{\tau_0}^{\beta}]$. Since $\mathcal{A}_{2, x}[\ide{p}_{\tau_0}^{\beta}]$ is \'{e}tale (by the signature condition), after possibly replacing $(K, \mathcal{O}_K)$ with a finite \'{e}tale extension, the morphism $\alpha$ extends to a morphism $\boldsymbol\alpha \colon \mathcal{A}_{2, x}[\ide{p}_{\tau_0}^{\beta}] \to \mathcal{A}_{1, x}[\ide{p}_{\tau_0}^{\beta}]$. The Zariski closure of $C_{n, \tau_0, y}$ in $\mathcal{A}_{1, x}[\ide{p}_{\tau_0}^{\beta}] \oplus \mathcal{A}_{2, x}[\ide{p}_{\tau_0}^{\beta}]$, denoted $\mathcal{C}_{n, \tau_0, y}$, is therefore equal to the graph of $\boldsymbol\alpha$. But this implies that $\opn{deg}(\mathcal{C}_{n, \tau_0, y}) \leq \opn{deg}(\mathcal{A}_{2, x}[\ide{p}_{\tau_0}^{\beta}]) = 0$, hence $\mathcal{C}_{n, \tau_0, y}$ is an \'{e}tale group scheme, as required.

    We now prove part (2). Let $W \subset \ide{X}_{H, \diamondsuit}(p^{\beta})$ denote the open subspace where $\hat{\delta}_{H, n+1}|_W = \mathcal{O}_W$, and let $W_{\eta} \subset \mathcal{X}_{H, \diamondsuit}(p^{\beta})$ denote its adic generic fibre (which is the locus where $|\hat{\delta}_{H, n+1}^+| = 1$). Recall that $\ide{X}_{H, \opn{id}}(p^{\beta}) \to \ide{X}_{H}^{\opn{ord}}$ parameterises $\underline{T^{\diamondsuit}(\mbb{Z}/p^{\beta}\mbb{Z})}$-orbits of isomorphisms 
    \begin{equation} \label{UpsilonTauIsos}
    \upsilon_{\tau} \colon \mbb{X}_{1, \tau}[p^{\beta}] \oplus \mbb{X}_{2, \tau}[p^{\beta}] \xrightarrow{\sim} A_1[\ide{p}_{\tau}^{\beta}] \oplus A_2[\ide{p}_{\tau}^{\beta}]
    \end{equation}
    respecting the decompositions on both sides (and a similitude factor $s \in \underline{\left(\mbb{Z}/p^{\beta}\mbb{Z}\right)^{\times}}$ ). By passing to generic fibres and fixing a $p^{\beta}$-root of unity $\varepsilon \in \mu_{p^{\beta}}$, we have an isomorphism $\left(\mbb{Z}/p^{\beta}\mbb{Z}\right)^{\oplus n} \cong \mbb{X}_{1, \tau_0}^{\eta}[p^{\beta}]$, hence we obtain a morphism $\mathcal{X}_{H, \opn{id}}(p^{\beta}) \to \mathcal{X}_{H, \diamondsuit}(p^{\beta})$.

    On the other hand, consider the morphism $\hat{\iota} \colon \ide{X}_{H, \opn{id}}(p^{\beta}) \to \ide{X}_{G, w_n}(p^{\beta})$ as at the end of \S \ref{QuotientsOfIgVarSubSec}, defined with respect to the same choice $\varepsilon \in \mu_{p^{\beta}}$. It is straightforward to verify that one has a commutative diagram
    \[
\begin{tikzcd}
{\mathcal{X}_{H, \opn{id}}(p^{\beta})} \arrow[d] \arrow[r] & {\mathcal{X}_{G, w_n}(p^{\beta})} \arrow[d] \\
{\mathcal{X}_{H, \diamondsuit}(p^{\beta})} \arrow[r]       & {\mathcal{X}_{G, \opn{Iw}}(p^{\beta})}     
\end{tikzcd}
    \]
    hence there exists a morphism $\mathcal{X}_{H, \opn{id}}(p^{\beta}) \to W_{\eta}$ (as $W_{\eta}$ is the pullback of $\mathcal{X}_{G, w_n}(p^{\beta})$ along $\hat{\iota} \colon \mathcal{X}_{H, \diamondsuit}(p^{\beta}) \to \mathcal{X}_{G, \opn{Iw}}(p^{\beta})$, using Lemma \ref{XGwnisOpenImmersionLemma}).

    Consider the moduli space $W' \to W$ parameterising isomorphisms $(v_{\tau})_{\tau \in \Psi}$ as in (\ref{UpsilonTauIsos}) and a similitude $s \in \underline{\left(\mbb{Z}/p^{\beta}\mbb{Z}\right)^{\times}}$ such that the flags determined by $(v_{\tau} \circ \gamma_{\tau})_{\tau \in \Psi}$ coincide with the pullback of the flags $\mathcal{C}_{\bullet, \tau}$ from Lemma \ref{ExtensionOfFFfiltrationLemma} to $W$. We claim this is a finite flat $\underline{T^{\diamondsuit}(\mbb{Z}/p^{\beta}\mbb{Z} )}$-torsor. Indeed, let $S \to W$ be a finite flat cover over which the group schemes $\mathcal{A}_i[\ide{p}_{\tau}^{\beta}]$ trivialise. We can find isomorphisms $f_{\tau} \colon \mbb{X}_{\opn{ord}, \tau}[p^{\beta}] \cong \mathcal{A}[\ide{p}_{\tau}^{\beta}]$ lifting the filtrations $\mathcal{C}_{\bullet, \tau}$, and we know on the generic fibre that $f_{\tau} \circ \gamma_{\tau}^{-1}$ preserves the decompositions 
    \[
    \mbb{X}_{\opn{ord}, \tau}^{\eta}[p^{\beta}] = \mbb{X}_{1, \tau}^{\eta}[p^{\beta}] \oplus \mbb{X}_{2, \tau}^{\eta}[p^{\beta}] \text{ and } A[\ide{p}_{\tau}^{\beta}] = A_1[\ide{p}_{\tau}^{\beta}] \oplus A_2[\ide{p}_{\tau}^{\beta}].
    \]
    Hence $f_{\tau} \circ \gamma_{\tau}^{-1}$ must also preserve the decomposition integrally and $W' \to W$ has a section over $S$. Furthermore, one can easily see that the fibre of $W' \to W$ over $S$ is a principal homogeneous space for $T^{\diamondsuit}(\mbb{Z}/p^{\beta}\mbb{Z})$. This implies we have a map $g \colon W \to \ide{X}_{H, \opn{id}}(p^{\beta})$ giving a factorisation
    \[
    W \xrightarrow{g} \ide{X}_{H, \opn{id}}(p^{\beta}) \to \ide{X}_{G, w_n}(p^{\beta})
    \]
    of the natural map $W \to \ide{X}_{G, w_n}(p^{\beta})$. Hence, on generic fibres, the map $g \colon W_{\eta} \to \mathcal{X}_{H, \opn{id}}(p^{\beta})$ must factor the inclusion $W_{\eta} \subset \mathcal{X}_{H, \diamondsuit}(p^{\beta})$, which proves part (2).

    To conlude the proof, we note that $W$ is identified with the normalisation of $\ide{X}_{G, w_n}(p^{\beta})$ under the map $W_{\eta} \cong \mathcal{X}_{H, \opn{id}}(p^{\beta}) \to \mathcal{X}_{G, w_n}(p^{\beta})$, and $\ide{X}_{H, \opn{id}}(p^{\beta})$ is already integrally closed in its generic fibre (see Proposition \ref{PropositionLociIntClosedInGenFibre}), so the map $g \colon W \to \ide{X}_{H, \opn{id}}(p^{\beta})$ must be an isomorphism. This gives part (3).  
\end{proof}

We introduce the following overconvergent neighbourhoods.

\begin{definition} \label{DeltaHneighbourhoodsDef}
    For $r \geq 1$, let $\mathcal{X}_{H, \opn{id}}(p^{\beta})_r \subset \mathcal{X}_{H, \opn{id}}(p^{\beta})$ denote the quasi-compact open rational subset defined by the inequality $|\hat{\delta}_{H, n+1}^+|^{p^{r+1}} \geq |p|$. It has a formal integral model $\ide{X}_{H, \opn{id}}(p^{\beta})_r$ given by an open in the formal admissible blow-up of $\ide{X}_{H, \diamondsuit}(p^{\beta})$ along the ideal $\hat{\delta}_{H, n+1}^{p^{r+1}} + (p)$, where the open is the locus where the pullback of $\hat{\delta}_{H, n+1}^{p^{r+1}} + (p)$ is equal to the pullback of $\hat{\delta}_{H, n+1}^{p^{r+1}}$ (see \cite[\S 8.2, Proposition 7]{Bosch}). We also let $\mathcal{Z}_{H, \opn{id}}(p^{\beta}) \subset \mathcal{X}_{H, \diamondsuit}(p^{\beta})$ denote the closure of $\mathcal{X}_{H, \opn{id}}(p^{\beta})$.
\end{definition}

\subsubsection{Some special open subspaces}

We now introduce certain collections of open subspaces which will be useful in \S \ref{PadicIterationOfDiffOpsSection}. Let $P_{G, \opn{dR}}^{\opn{an}} \to \mathcal{X}_{G, \opn{Iw}}(p^{\beta})$ and $M_{G, \opn{dR}}^{\opn{an}} \to \mathcal{X}_{G, \opn{Iw}}(p^{\beta})$ denote the analytifications of the torsors $P_{G, \opn{dR}}$ and $M_{G, \opn{dR}}$ defined in Definition \ref{DefinitionOfGdrPdR}. Similarly, let $P_{H, \opn{dR}}^{\opn{an}} \to \mathcal{X}_{H, \diamondsuit}(p^{\beta})$ and $M_{H, \opn{dR}}^{\opn{an}} \to \mathcal{X}_{H, \diamondsuit}(p^{\beta})$ denote the analytifications of the torsors $P_{H, \opn{dR}}$ and $M_{H, \opn{dR}}$.

\begin{definition} \label{DefOfCGCH}
    Let
    \begin{enumerate}
        \item $\mathcal{C}_G$ denote the collection of quasi-compact open affinoid subspaces $\opn{Spa}(A, A^+) \subset \mathcal{X}_{G, \opn{Iw}}(p^{\beta})$ such that $\opn{Spa}(A_{\opn{ord}}, A^+_{\opn{ord}}) \defeq \opn{Spa}(A, A^+) \cap \mathcal{X}_{G, w_n}(p^{\beta})$ is the adic generic fibre of an open affine subspace $\opn{Spf}A_{\opn{ord}}^+ \subset \ide{X}_{G, w_n}(p^{\beta})$.  We also impose the condition that the ideals $\delta_{G, i}^+$ ($i=1, \dots, 2n$) and the $+$-versions of $\mathcal{H}_A$ and the graded pieces of its Hodge filtration trivialise over any $\opn{Spa}(A, A^+) \in \mathcal{C}_G$. In particular, this implies that the torsors $P_{G, \opn{dR}}^{\opn{an}}$, $M_{G, \opn{dR}}^{\opn{an}}$ trivialise over any $\opn{Spa}(A, A^+) \in \mathcal{C}_G$.
        \item $\mathcal{C}_H$ denote the collection of quasi-compact open affinoid subspaces $\opn{Spa}(A, A^+) \subset \mathcal{X}_{H, \diamondsuit}(p^{\beta})$ such that $\opn{Spa}(A_{\opn{ord}}, A^+_{\opn{ord}}) \defeq \opn{Spa}(A, A^+) \cap \mathcal{X}_{H, \opn{id}}(p^{\beta})$ is the adic generic fibre of an open affine subspace $\opn{Spf}A_{\opn{ord}}^+ \subset \ide{X}_{H, \opn{id}}(p^{\beta})$.  We also impose the condition that the ideals $\delta_{H, i}^+$ ($i=1, \dots, 2n$) and the $+$-versions of $\mathcal{H}_{A_1}$, $\mathcal{H}_{A_2}$ and the graded pieces of their Hodge filtrations trivialise over any $\opn{Spa}(A, A^+) \in \mathcal{C}_H$. In particular, the torsors $P_{H, \opn{dR}}^{\opn{an}}$, $M_{H, \opn{dR}}^{\opn{an}}$ trivialise over any $\opn{Spa}(A, A^+) \in \mathcal{C}_H$.
        
        \item Let $\mathcal{C}_{G, H}$ denote the collection of all $U \in \mathcal{C}_G$ such that the pullback of $U$ along the finite unramified morphism 
    \[
    \hat{\iota} \colon \mathcal{X}_{H, \diamondsuit}(p^{\beta}) \to \mathcal{X}_{G, \opn{Iw}}(p^{\beta})
    \]
    is contained in $\mathcal{C}_H$.
        
    \end{enumerate}
\end{definition}

\begin{remark} 
    Note that there exists a finite cover of $\mathcal{X}_{G, \opn{Iw}}(p^{\beta})$ (resp. $\mathcal{X}_{H, \diamondsuit}(p^{\beta})$) by elements of $\mathcal{C}_G$ (resp. $\mathcal{C}_H$), because we can always choose $\opn{Spa}(A, A^+)$ to be the adic generic fibre of an open in $\ide{X}_{G, \opn{Iw}}(p^{\beta})$ (resp. $\ide{X}_{H, \diamondsuit}(p^{\beta})$).
\end{remark}

\begin{lemma} \label{Lem:SpecialCoverFromCGH}
    There exists a finite cover of $\mathcal{X}_{G, \opn{Iw}}(p^{\beta})$ by elements of $\mathcal{C}_{G, H}$.
\end{lemma}
\begin{proof}
    We start with the following basic result in commutative algebra. \\
    \\
    \textbf{Claim:} Let $\phi \colon A \to B$ be a finite morphism of commutative rings, and let $M$ be a finite projective $B$-module of constant rank $d$. Then there exists a finite collection of elements $f_1, \dots, f_r \in A$ which generate the unit ideal (in $A$) such that: $M \otimes_B B[\phi(f_i)^{-1}]$ is free of rank $d$ over $B[\phi(f_i)^{-1}]$ for all $i=1, \dots, r$. \\
    \\
    \textbf{Proof of claim:} Let $\ide{m} \subset A$ be a maximal ideal. Then, because the map $\phi$ is finite, there are finitely many maximal ideals $\ide{n}_1, \dots, \ide{n}_s$ of $B$ which lie above $\ide{m}$ (i.e., $\phi^{-1}(\ide{n}_j) = \ide{m}$ for all $j=1, \dots, s$). Let $S = A - \ide{m}$ and $T = \cap_{j=1}^s (B - \ide{n}_j)$. Then $\phi(S) \subset T$, $\phi^{-1}(T) = S$, and $T^{-1}B$ is a semi-local ring. Since finite projective modules of constant rank over (commutative) semi-local rings are free (see, e.g., \cite{Hinohara}), we see that $M \otimes_B T^{-1}B$ is free of rank $d$ over $T^{-1}B$. Furthermore, we have $\phi(S)^{-1}B = T^{-1}B$ because $T$ is contained in the saturation of $\phi(S)$ in $B$ (the saturation of $\phi(S)$ is necessarily the complement of a union of prime ideals \cite[p.\ 2, Theorem 2]{Kaplansky}, and any prime ideal $\ide{p}$ which intersects $\phi(S)$ trivially must be contained in some $\ide{n}_i$ by the ``going up'' property \cite[p.\ 29, Theorem 44]{Kaplansky}). Since $M$ is finitely-generated, we can therefore find $f_{\ide{m}} \in S$ such that $M \otimes_B B[\phi(f_{\ide{m}})^{-1}]$ is free of rank $d$ over $B[\phi(f_{\ide{m}})^{-1}]$. Repeating this for every maximal ideal $\ide{m}$, we obtain a collection of elements $\{ f_{\ide{m}} \}$ and we can clearly pick out a finite collection $f_1, \dots, f_r$ satisfying the statement of the claim. \hfill \qedsymbol \\
    \\
    We note the important consequence of this claim, namely: if $\phi \colon X \to Y$ is a finite morphism of schemes with $Y$ quasi-compact, and $\mathcal{M}_1, \dots, \mathcal{M}_t$ is a finite collection of locally free $\mathcal{O}_X$-modules of (finite) constant rank, then there exists a finite open affine cover $Y = \bigcup_{i=1}^r U_i$ such that $\mathcal{M}_{j}$ is free over $\phi^{-1}U_i$ for all $i=1, \dots, r$ and $j=1, \dots, t$. \\
    \\
    We now return to the proof of the lemma. Consider the finite morphism $\iota \colon X_{\mbf{H}} \to X_{\mbf{G}}$ over $\opn{Spec}\mathcal{O}_L$ from Definition \ref{HyperspecialVarietiesDefinition}. Then the consequence of the above claim implies that we can find a finite open affine cover $X_{\mbf{G}} = \bigcup_{i \in I} U_i$ such that $\opn{H}_1^{\opn{dR}}(A_1/X_{\mbf{H}})$, $\opn{H}_1^{\opn{dR}}(A_2/X_{\mbf{H}})$ and the graded pieces of their Hodge filtrations become free over the cover $\{ \iota^{-1}(U_i) \}_{i \in I}$. By refining $\{ U_i \}_{i \in I}$ we may assume, without loss of generality, that $\opn{H}_1^{\opn{dR}}(A/X_{\mbf{G}})$ and the graded pieces of its Hodge filtration become free over the cover $\{ U_i \}_{i \in I}$.

    Let $\varpi$ be a uniformiser of $\mathcal{O}_L$. Let $\ide{X}_{H}$, $\ide{X}_G$ and $\ide{U}_i$ denote the $\varpi$-adic completions of $X_{\mbf{H}}$, $X_{\mbf{G}}$, $U_i$ respectively, which are formal schemes over $\opn{Spf}\mathcal{O}_L$ (and $\{ \ide{U}_i \}_{i \in I}$ is a finite open affine cover of $\ide{X}_G$). Note all the formal schemes have the $\varpi$-adic topology. Recall from the proof of Proposition \ref{HlociisHasseProp} that we have a commutative diagram
    \[
\begin{tikzcd}
{\ide{X}_{H, \diamondsuit}(p^{\beta})} \arrow[r, "\hat{\iota}"] \arrow[d] & {\ide{X}_{G, \opn{Iw}}(p^{\beta})} \arrow[d, "p"] \\
\ide{X}_H \arrow[r, "\iota"]                                              & \ide{X}_G                                        
\end{tikzcd}
    \]
    where the vertical arrows are the forgetful maps. Let $p^{-1}\ide{U}_i$ denote the pullback of $\ide{U}_i$ under the right-hand vertical map. Recall $\ide{X}_{G, \opn{Iw}}(p^{\beta})$ is quasi-compact. Then, we can find a finite open affine refinement $\{ \ide{U}_i' \}_{i \in I'}$ of the cover $\{ p^{-1}\ide{U}_i \}_{i \in I}$ such that $\mathcal{H}_A$ and the graded pieces of its Hodge filtration are free over $\{ \ide{U}_i' \}_{i \in I'}$ and $\mathcal{H}_{A_1}$, $\mathcal{H}_{A_2}$, and the graded pieces of their Hodge filtrations are free over $\{ \hat{\iota}^{-1}(\ide{U}_i') \}_{i \in I'}$. By refining the cover $\{ \ide{U}'_i \}_{i \in I'}$ further, we may also assume, without loss of generality, that the ideals $\{ \delta_{G, i}^+ : i=1, \dots, 2n \}$ are free over $\{ \ide{U}_i' \}_{i \in I'}$. Since $\delta_{H, i} = \hat{\iota}^* \delta_{G, i}$, we immediately see that the ideals $\{ \delta_{H, i} : i=1, \dots, 2n \}$ are free over the affine cover $\{ \hat{\iota}^{-1}(\ide{U}_i')\}_{i \in I'}$. The desired cover of $\mathcal{X}_{G, \opn{Iw}}(p^{\beta})$ is then $\{ \mathcal{U}_i' \}_{i \in I'}$, where $\mathcal{U}_i'$ is the adic generic fibre of $\ide{U}_i'$. 
\end{proof}

\subsection{Comparison between the Gauss--Manin connection and Atkin--Serre operators}

We now compare the operators $\nabla_i$ in Lemma \ref{LemmaConcreteDescriptionNablai} to the actions of $C_{\opn{cont}}(U_{\bullet, \beta}, \mbb{Z}_p)$ constructed in \S \ref{DiffOpsCcontsubsec}. We continue to adopt the same conventions as in \S \ref{IntegralModelsandOCNbhdsSec}. First we note the following:

\begin{lemma} \label{IGPdrComparisonsLemma}
    One has commutative diagrams
    \[
\begin{tikzcd}
                                                                  & {P_{G, \opn{dR}}^{\opn{an}}} \arrow[d] &  &                                                                        & {P_{H, \opn{dR}}^{\opn{an}}} \arrow[d]     \\
{\mathcal{IG}_{G, w_n}(p^{\beta})} \arrow[r] \arrow[ru] \arrow[d] & {M_{G, \opn{dR}}^{\opn{an}}} \arrow[d] &  & {\mathcal{IG}_{H, \opn{id}}(p^{\beta})} \arrow[ru] \arrow[r] \arrow[d] & {M_{H, \opn{dR}}^{\opn{an}}} \arrow[d]     \\
{\mathcal{X}_{G, w_n}(p^{\beta})} \arrow[r]                       & {\mathcal{X}_{G, \opn{Iw}}(p^{\beta})} &  & {\mathcal{X}_{H, \opn{id}}(p^{\beta})} \arrow[r]                       & {\mathcal{X}_{H, \diamondsuit}(p^{\beta})}
\end{tikzcd}
    \]
    where the top diagonal and horizontal maps are $M^G_{\opn{Iw}}(p^{\beta})$-equivariant and $M^H_{\diamondsuit}(p^{\beta})$-equivariant respectively, and the top horizontal maps provide reductions of structure of $M_{G, \opn{dR}}^{\opn{an}} \times_{\mathcal{X}_{G, \opn{Iw}}(p^{\beta})} \mathcal{X}_{G, w_n}(p^{\beta})$ and $M_{H, \opn{dR}}^{\opn{an}} \times_{\mathcal{X}_{H, \diamondsuit}(p^{\beta})} \mathcal{X}_{H, \opn{id}}(p^{\beta})$.
\end{lemma}
\begin{proof}
    The Igusa tower $\mathcal{IG}_{G, w_n}(p^{\beta})$ parameterises points $(A, \mathcal{C}_{\bullet, \bullet}) \in \mathcal{X}_{G, w_n}(p^{\beta})$ and trivialisations of (the graded pieces of) the canonical filtration on $A[p^{\infty}]$ which are compatible with the filtrations $\mathcal{C}_{\bullet, \bullet}$. One then obtains a trivialisation of $\mathcal{H}_A$ respecting the Hodge filtration in the usual way, namely by considering the ``dlog'' morphisms for each graded piece of the canonical filtration (to obtain a map to $M_{G,\opn{dR}}^{\opn{an}}$), and using the unit root splitting (to extend this to a map to $P_{G,\opn{dR}}^{\opn{an}}$).

    Similarly, the Igusa tower $\mathcal{IG}_{H, \opn{id}}(p^{\beta})$ parameterises tuples $(A_1, A_2, \lambda, i, \eta^p, s, [f_{\tau}])$ with $(A_1, A_2, \lambda, i, \eta^p) \in \ide{X}_H^{\opn{ord}}$, $s \in \underline{\mbb{Z}_p^{\times}}$ and $[f_{\tau}]$ is an $N^H_{\diamondsuit}(p^{\beta})$-orbit of isomorphisms $f_{\tau} \colon \mbb{X}_{\opn{ord}, \tau} \xrightarrow{\sim} A[\ide{p}_{\tau}^{\infty}]$ respecting the decompositions $\mbb{X}_{\opn{ord}, \tau} = \mbb{X}_{1, \tau} \oplus \mbb{X}_{2, \tau}$ and $A[\ide{p}_{\tau}^{\infty}] = A_1[\ide{p}_{\tau}^{\infty}] \oplus A_2[\ide{p}_{\tau}^{\infty}]$. Given such a point $(A_1, A_2, \lambda, i, \eta^p, s, [f_{\tau}])$ we can: obtain trivialisations of the Hodge filtrations on $\mathcal{H}_{A_1}$ and $\mathcal{H}_{A_2}$ using the ``dlog'' morphisms and the trivialisations of the graded pieces of the canonical filtration induced from $f_{\tau}$ (this is independent of the choice of representatives of the orbits $[f_{\tau}]$); obtain a $(\hat{\gamma} B_G(\mbb{Z}/p^{\beta}\mbb{Z}) \hat{\gamma}^{-1} \cap H(\mbb{Z}/p^{\beta}\mbb{Z}))$-orbit of isomorphisms
    \[
    (\underline{\mbb{Z}/p^{\beta}\mbb{Z}})^{\oplus n} \oplus (\underline{\mbb{Z}/p^{\beta}\mbb{Z}})^{\oplus n} \cong  \mbb{X}_{1, \tau}[p^{\beta}] \oplus \mbb{X}_{2, \tau}[p^{\beta}] \xrightarrow{f_{\tau}[p^{\beta}]} A_1[\ide{p}_{\tau}^{\beta}] \oplus A_2[\ide{p}_{\tau}^{\beta}] 
    \]
    respecting the decompositions on both sides (note that we have an isomorphism $\mu_{p^{\beta}} \cong \underline{\mbb{Z}/p^{\beta}\mbb{Z}}$ of group schemes over $L$ because we are assuming $L$ contains the $p^{\beta}$-roots of unity; in particular $\hat{\gamma} B_G(\mbb{Z}/p^{\beta}\mbb{Z}) \hat{\gamma}^{-1} \cap H(\mbb{Z}/p^{\beta}\mbb{Z})$ is identified with $J^H_{\diamondsuit}(p^{\beta})$ modulo the subgroup of $J^+_{H, \opn{ord}}$ consisting of elements congruent to the identity modulo $p^{\beta}$). This describes the desired map $\mathcal{IG}_{H, \opn{id}}(p^{\beta}) \to M_{H, \opn{dR}}^{\opn{an}}$. We then extend this to a map to $P_{H, \opn{dR}}^{\opn{an}}$ using the unit root splittings of the Hodge filtrations on $\mathcal{H}_{A_1}$ and $\mathcal{H}_{A_2}$.
\end{proof}

Let $\mathscr{O}_{P_{G,\opn{dR}}^{\opn{an}}}$ and $\mathscr{O}_{\mathcal{IG}_{G, w_n}(p^{\beta})}$ denote the pushforwards of the structure sheaves of $P_{G, \opn{dR}}^{\opn{an}}$ and $\mathcal{IG}_{G, w_n}(p^{\beta})$ respectively to $\mathcal{X}_{G, \opn{Iw}}(p^{\beta})$. We define $\mathscr{O}_{P_{H,\opn{dR}}^{\opn{an}}}$ and $\mathscr{O}_{\mathcal{IG}_{H, \opn{id}}(p^{\beta})}$ similarly (by pushing forward to $\mathcal{X}_{H, \diamondsuit}(p^{\beta})$). Lemma \ref{IGPdrComparisonsLemma} implies that we have natural restriction maps $\mathscr{O}_{P_{G,\opn{dR}}^{\opn{an}}} \to \mathscr{O}_{\mathcal{IG}_{G, w_n}(p^{\beta})}$ and $\mathscr{O}_{P_{H,\opn{dR}}^{\opn{an}}} \to \mathscr{O}_{\mathcal{IG}_{H, \opn{id}}(p^{\beta})}$. Note that for any $U = \opn{Spa}(A, A^+) \in \mathcal{C}_G$, we have an action
\[
C_{\opn{cont}}(U_{G, \beta}, L) \times \mathscr{O}_{\mathcal{IG}_{G, w_n}(p^{\beta})}(U) \to \mathscr{O}_{\mathcal{IG}_{G, w_n}(p^{\beta})}(U)
\]
which is functorial in $U$. The same is true for $H$.

\begin{proposition} \label{IntertwiningOfNablaThetaiProp}
    For $i=1, \dots, 2n-1$ and $U \in \mathcal{C}_G$, let $\theta_i \colon \mathscr{O}_{\mathcal{IG}_{G, w_n}(p^{\beta})}(U) \to \mathscr{O}_{\mathcal{IG}_{G, w_n}(p^{\beta})}(U)$ denote the morphism given by the action of the continuous map
    \[
    (a_1, \dots, a_{2n-1}) \mapsto a_i \quad \in \quad C_{\opn{cont}}(U_{G, \beta}, L) .
    \]
    For $i=1, \dots, n-1$ and $U \in \mathcal{C}_H$, we also use the notation $\theta_i \colon \mathscr{O}_{\mathcal{IG}_{H, \opn{id}}(p^{\beta})}(U) \to \mathscr{O}_{\mathcal{IG}_{H, \opn{id}}(p^{\beta})}(U)$ to denote the morphism constructed analogously.
    \begin{enumerate}
        \item For any $U \in \mathcal{C}_G$ and $i=1, \dots, 2n-1$, one has a commutative diagram
        \[
\begin{tikzcd}
{\mathscr{O}_{P_{G, \opn{dR}}^{\opn{an}}}(U)} \arrow[r, "\nabla_i"] \arrow[d] & {\mathscr{O}_{P_{G, \opn{dR}}^{\opn{an}}}(U)} \arrow[d] \\
{\mathscr{O}_{\mathcal{IG}_{G, w_n}(p^{\beta})}(U)} \arrow[r, "\theta_i"]     & {\mathscr{O}_{\mathcal{IG}_{G, w_n}(p^{\beta})}(U)}    
\end{tikzcd}
        \]
        \item For any $U \in \mathcal{C}_H$ and $i=1, \dots, n-1$, one has a commutative diagram
        \[
\begin{tikzcd}
{\mathscr{O}_{P_{H, \opn{dR}}^{\opn{an}}}(U)} \arrow[r, "\nabla_i"] \arrow[d]  & {\mathscr{O}_{P_{H, \opn{dR}}^{\opn{an}}}(U)} \arrow[d]  \\
{\mathscr{O}_{\mathcal{IG}_{H, \opn{id}}(p^{\beta})}(U)} \arrow[r, "\theta_i"] & {\mathscr{O}_{\mathcal{IG}_{H, \opn{id}}(p^{\beta})}(U)}
\end{tikzcd}
        \]
    \end{enumerate}
\end{proposition} 
\begin{proof}
    Let $\mathcal{A}$ denote the universal $\boldsymbol\Psi$-unitary abelian scheme over $\ide{IG}_{G, w_n}(p^{\beta})$. Then (using the universal trivialisations of the graded pieces) the canonical filtration for $\mathcal{A}[\ide{p}_{\tau_0}^{\infty}]$ is of the form
    \[
    0 \to \mu_{p^{\infty}} \to \mathcal{A}[\ide{p}_{\tau_0}^{\infty}] \to \left(\mbb{Q}_p/\mbb{Z}_p \right)^{\oplus 2n-1} \to 0 .
    \]
    Let $\mbb{E}(\mathcal{A}[\ide{p}_{\tau_0}^{\infty}])$ denote the universal vector extension of $\mathcal{A}[\ide{p}_{\tau_0}^{\infty}]$. Then $\opn{Lie}\mbb{E}(\mathcal{A}[\ide{p}_{\tau_0}^{\infty}])$ is identified with the first relative de Rham homology of $\mathcal{A}[\ide{p}_{\tau_0}^{\infty}]$, and the Hodge filtration is given by the exact sequence:
    \[
    0 \to \omega_{\mathcal{A}[\ide{p}_{\tau_0}^{\infty}]^D} \to \opn{Lie}\mbb{E}(\mathcal{A}[\ide{p}_{\tau_0}^{\infty}]) \to \opn{Lie} \mathcal{A}[\ide{p}_{\tau_0}^{\infty}] \to 0 .
    \]
    The ``dlog'' map for $T_p\mathcal{A}[\ide{p}_{\tau_0}^{\infty}]$ induces an isomorphism
    \[
    T_p\mathcal{A}[\ide{p}_{\tau_0}^{\infty}]^{\et} \otimes_{\mbb{Z}_p} \mathcal{O}_{\ide{IG}_{G, w_n}(p^{\beta})} \cong \mathcal{O}_{\ide{IG}_{G, w_n}(p^{\beta})}^{\oplus 2n-1} \xrightarrow{\sim} \omega_{\mathcal{A}[\ide{p}_{\tau_0}^{\infty}]^D}
    \]
    where the first isomorphism arises from the universal trivialisation of the \'{e}tale part of the canonical filtration. In particular, we obtain a basis $\{ \omega_{\opn{can}, i} : i = 1, \dots, 2n-1 \}$ of $\omega_{\mathcal{A}[\ide{p}_{\tau_0}^{\infty}]^D}$ by considering the image of the $i$-th basis vector of $\mathcal{O}_{\ide{IG}_{G, w_n}(p^{\beta})}^{\oplus 2n-1}$.

    On the other hand, the inclusion $\mu_{p^{\infty}} \hookrightarrow \mathcal{A}[\ide{p}_{\tau_0}^{\infty}]$ induces an injective map
    \[
    \opn{Lie} \mu_{p^{\infty}} = \opn{Lie} \mbb{E}(\mu_{p^{\infty}}) \hookrightarrow \opn{Lie}\mbb{E}(\mathcal{A}[\ide{p}_{\tau_0}^{\infty}])
    \]
    which splits the Hodge filtration (it is the ``unit root splitting''). Let $u_{\opn{can}} \in \opn{Lie}\mbb{E}(\mathcal{A}[\ide{p}_{\tau_0}^{\infty}])$ denote the image of the canonical tangent vector $t \partial_t \in \opn{Lie}\mu_{p^{\infty}}$ under this map.

    As explained in \cite[\S 2.4]{howe2020unipotent}, we have a crystalline connection $\nabla_{\opn{cris}}$ on $\opn{Lie}\mbb{E}(\mathcal{A}[\ide{p}_{\tau_0}^{\infty}])$ which is block upper nilpotent in the basis $\{ u_{\opn{can}}, \omega_{\opn{can}, 1}, \dots, \omega_{\opn{can}, 2n-1} \}$. Let $\mathcal{K}_i \in \Omega^1_{\ide{IG}_{G, w_n}(p^{\beta})}$ denote the unique differential such that $\nabla_{\opn{cris}}(\omega_{\opn{can}, i}) = u_{\opn{can}} \otimes \mathcal{K}_i$. Let $\theta^{\opn{cris}}_i \colon \mathcal{O}_{\mathcal{IG}_{G, w_n}(p^{\beta})} \to \mathcal{O}_{\mathcal{IG}_{G, w_n}(p^{\beta})}$ denote the (unique) derivation such that
    \[
    \langle \mathcal{K}_i, \theta^{\opn{cris}}_i \rangle = 1 \quad \text{ and } \quad \langle \mathcal{K}_j, \theta^{\opn{cris}}_i \rangle = 0 \text{ for } j \neq i.
    \]
    Since the crystalline and Gauss--Manin connections coincide (see \cite[Theorem 2.6.1]{howe2020unipotent}), one sees from the description in Lemma \ref{LemmaConcreteDescriptionNablai} that $\nabla_i$ and $\theta^{\opn{cris}}_i$ are compatible under the map $\mathscr{O}_{P_{G, \opn{dR}}^{\opn{an}}} \to \mathscr{O}_{\mathcal{IG}_{G, w_n}(p^{\beta})}$. Therefore, it is enough to show that $\theta^{\opn{cris}}_i = \theta_i$ on $\mathscr{O}_{\mathcal{IG}_{G, w_n}(p^{\beta})}(U)$.

    For $i=1, \dots, 2n-1$, let 
    \[
    \beta_i = \left\{ \begin{array}{cc} \beta & \text{ if } i=1, \dots, n \\ 0 & \text{ if } i=n+1, \dots, 2n-1 \end{array} \right.
    \]
    so that $U_{G, \beta} = \prod_{i=1}^{2n-1} p^{-\beta_i}\mbb{Z}_p$. Let $L_i \defeq \tilde{\mu_{p^{\infty}}} / p^{\beta_i}T_p\mu_{p^{\infty}}$, then we have an isomorphism
    \[
    L_i \xrightarrow{\sim} \mu_{p^{\infty}}, \quad (\zeta_{k})_{k \geq 0} \mapsto \zeta_{\beta_i}
    \]
    hence we have a natural map $D = \opn{Spf}\mbb{Z}_p[\epsilon]/\epsilon^2 \to L_i$ which, when composed with this isomorphism, corresponds to the point $1+\epsilon \in \mu_{p^{\infty}}(\mbb{Z}_p[\epsilon]/\epsilon^2)$. We therefore obtain a vector field
    \[
    t_i \colon D \times \ide{IG}_{G, w_n}(p^{\beta}) \to L_i \times \ide{IG}_{G, w_n}(p^{\beta}) \to \ide{IG}_{G, w_n}(p^{\beta})
    \]
    where the second map is the action map. Then
    \begin{itemize}
        \item The derivation $\mathcal{O}_{\ide{IG}_{G, w_n}(p^{\beta})} \to \mathcal{O}_{\ide{IG}_{G, w_n}(p^{\beta})}$ corresponding to the vector field $t_i$ is given by the action of 
        \[
        \left. \frac{d}{d\epsilon} \right|_{\epsilon = 0} [ ( \frac{a_j}{p^{\beta_j}} ) \mapsto (1+ \epsilon)^{a_i} ] = [ ( \frac{a_j}{p^{\beta_j}} ) \mapsto a_i ] \in C_{\opn{cont}}(U_{G, \beta}, \mbb{Z}_p)
        \]
        hence is equal to $p^{\beta_i} \theta_i$.
        \item By the same proof as in \cite[Theorem 5.3.1]{howe2020unipotent}, one can show that the crystalline connection induces an isomorphism
        \[
        \nabla_{\opn{cris}, t_i} \colon t_i^* \opn{Lie}\mbb{E}(\mathcal{A}[\ide{p}_{\tau_0}^{\infty}]) \xrightarrow{\sim} 0^* \opn{Lie}\mbb{E}(\mathcal{A}[\ide{p}_{\tau_0}^{\infty}]) 
        \]
        which satisfies 
        \[
        \nabla_{\opn{cris}, t_i}(\omega_{\opn{can}, j}) = \left\{ \begin{array}{cc} \omega_{\opn{can}, i} + p^{\beta_i} \epsilon u_{\opn{can}} & \text{ if } j=i \\ \omega_{\opn{can}, j} & \text{ otherwise } \end{array} \right. .
        \]
        This implies that $\langle \mathcal{K}_j, t_i \rangle$ equals $p^{\beta_i}$ if $j=i$, and is equal to zero otherwise. Hence, the derivation induced from the vector field $t_i$ is equal to $p^{\beta_i} \theta^{\opn{cris}}_i$.
    \end{itemize}
    Combining these bullet points, we see that $\theta_i = \theta_i^{\opn{cris}}$ as required. The proof of part (2) is identical.
\end{proof}

\subsection{Nearly overconvergent automorphic forms} \label{NearlyOverconvergentAutoFormsSection}

We now introduce the ind-sheaves of nearly overconvergent automorphic forms and their extra structures. In \S \ref{PadicIterationOfDiffOpsSection}, we will explain how one can ``overconverge'' the operators $\theta_i$ to these ind-sheaves. As in \S \ref{IntegralModelsandOCNbhdsSec}, we continue to work over a finite extension $L/\mbb{Q}_p$ containing $\mu_{p^{\beta}}$, but omit this from the notation. We fix an integer $\beta \geq 1$ throughout.

\begin{definition}
     For any open subspace $U \subset P_{G, \opn{dR}}^{\opn{an}}$ or $U \subset M_{G, \opn{dR}}^{\opn{an}}$ (resp. $U \subset P_{H, \opn{dR}}^{\opn{an}}$ or $U \subset M_{H, \opn{dR}}^{\opn{an}}$) we let $\mathscr{O}_U$ denote the pushforward of the structure sheaf of $U$ along the natural map to $\mathcal{X}_{G, \opn{Iw}}(p^{\beta})$ (resp. $\mathcal{X}_{H, \diamondsuit}(p^{\beta})$).
    \begin{enumerate}
        \item We define the ind-sheaf of nearly overconvergent forms on $G$ (resp. $H$) to be
        \[
        \mathscr{N}_G^{\dagger} \defeq (\mathscr{O}_U)_U, \quad \quad \text{ (resp. } \mathscr{N}^{\dagger}_H \defeq (\mathscr{O}_U)_U \text{ )},
        \]
        where the inductive system is over all quasi-compact open subspaces of $P_{G, \opn{dR}}^{\opn{an}}$ (resp. $P_{G, \opn{dR}}^{\opn{an}}$) containing the closure of $\mathcal{IG}_{G, w_n}(p^{\beta})$ (resp. $\mathcal{IG}_{H, \opn{id}}(p^{\beta})$), ordered by inclusion (so the transition maps in the inductive systems are given by restriction).
        \item We define the ind-sheaf of overconvergent forms on $G$ (resp. $H$) to be
        \[
        \mathscr{M}_G^{\dagger} \defeq (\mathscr{O}_U)_U, \quad \quad \text{ (resp. } \mathscr{M}^{\dagger}_H \defeq (\mathscr{O}_U)_U \text{ )},
        \]
        where the inductive system is over all quasi-compact open subspaces of $M_{G, \opn{dR}}^{\opn{an}}$ (resp. $M_{G, \opn{dR}}^{\opn{an}}$) containing the closure of $\mathcal{IG}_{G, w_n}(p^{\beta})$ (resp. $\mathcal{IG}_{H, \opn{id}}(p^{\beta})$), ordered by inclusion (so the transition maps in the inductive system are given by restriction).
    \end{enumerate}
\end{definition}

We have the following relation with the torsors appearing in higher Coleman theory \cite{BoxerPilloni}. Let $\mathcal{M}_G$ and $\mathcal{M}_H$ denote the completions of $M_G$ and $M_H$ along the special fibre.

\begin{lemma} \label{OctorsorsReductionLemma}
    For any integer $k \geq 1$, let $\mathcal{M}^1_{G, k} \subset \mathcal{M}_G$ denote the affinoid subgroup of elements which reduce to the identity modulo $p^k$. Set $\mathcal{M}^{\square}_{G, k} = \mathcal{M}^1_{G, k} \cdot M^G_{\opn{Iw}}(p^{\beta})$. Similarly, let $\mathcal{M}^1_{H, k} \subset \mathcal{M}_H$ denote the affinoid subgroup of elements which reduce to the identity modulo $p^k$, and set $\mathcal{M}^{\diamondsuit}_{H, k} = \mathcal{M}^1_{H, k} \cdot M^H_{\diamondsuit}(p^{\beta})$. 
    \begin{enumerate}
        \item For any integer $k \geq 1$, there exists an integer $r \geq 1$ (depending on $k$) such that $M_{G, \opn{dR}}^{\opn{an}} \times_{\mathcal{X}_{G, \opn{Iw}}(p^{\beta})} \mathcal{X}_{G, w_n}(p^{\beta})_{r}$ has a reduction of structure to an \'{e}tale $\mathcal{M}^{\square}_{G, k}$-torsor, denoted ${^\mu \mathcal{M}_{G, \opn{HT}, k}}$. Furthermore the torsors ${^\mu \mathcal{M}_{G, \opn{HT}, k}} \subset M_{G, \opn{dR}}^{\opn{an}}$ form a cofinal system of quasi-compact open subspaces containing the closure of $\mathcal{IG}_{G, w_n}(p^{\beta})$, so 
        \[
        \mathscr{M}^{\dagger}_G \cong (\mathscr{M}_G^{(r,k)})_{r,k} , \quad \quad \mathscr{M}_G^{(r,k)} \defeq \mathscr{O}_{{^\mu \mathcal{M}_{G, \opn{HT}, k}}},
        \]
        where the inductive system is over all possible $r, k$ with transition maps given by restriction.
        \item For any integer $k \geq 1$, there exists an integer $r \geq 1$ (depending on $k$) such that $M_{H, \opn{dR}}^{\opn{an}} \times_{\mathcal{X}_{H, \diamondsuit}(p^{\beta})} \mathcal{X}_{H, \opn{id}}(p^{\beta})_r$ has a reduction of structure to an \'{e}tale $\mathcal{M}^{\diamondsuit}_{H, k}$-torsor, denoted ${^\mu \mathcal{M}_{H, \opn{HT}, k}}$. Furthermore the torsors ${^\mu \mathcal{M}_{H, \opn{HT}, k}} \subset M_{H, \opn{dR}}^{\opn{an}}$ form a cofinal system of quasi-compact open subspaces containing the closure of $\mathcal{IG}_{H, \opn{id}}(p^{\beta})$, so 
        \[
        \mathscr{M}^{\dagger}_H \cong (\mathscr{M}^{(r, k)}_H)_{r, k} , \quad \quad \mathscr{M}^{(r, k)}_H \defeq \mathscr{O}_{{^\mu \mathcal{M}_{H, \opn{HT}, k}}},
        \]
        where the inductive system is over all possible $r, k$ with transition maps given by restriction.
    \end{enumerate}
\end{lemma}
\begin{proof}
    The proof of this lemma is very similar to the constructions in \cite[\S 5]{UFJ}, although the torsors are for slightly different groups and strata in the Shimura--Deligne varieties. More precisely, the torsor ${^\mu \mathcal{M}_{G, \opn{HT}, k}}$ is the twist along the central Hodge cocharacter $\mu \colon \mbb{Z}_p^{\times} \to \mathcal{M}_{G, k}^{\square}$ of the pullback of a torsor over the flag variety $\mathcal{P}_G \backslash \mathcal{G}$ under the (truncated) Hodge--Tate period morphism. 

    For the torsor ${^\mu \mathcal{M}_{H, \opn{HT}, k}}$, the construction is similar, however we must twist the pro\'{e}tale torsor $\mathcal{M}_{H, \opn{HT}, k}$ along the restriction of the Hodge cocharacter to $1 + p^{\beta} \mbb{Z}_p$, i.e. we define
    \[
    {^\mu \mathcal{M}_{H, \opn{HT}, k}} = \mathcal{M}_{H, \opn{HT}, k} \times^{[1 + p^{\beta}\mbb{Z}_p, \mu]} \mathcal{T}^{\times}_{\beta}
    \]
    where $\mathcal{T}^{\times}_{\beta} \subset \mathcal{T}^{\times} = \underline{\opn{Isom}}(\mbb{Z}_p, \mbb{Z}_p(1))$ is the sub-$(1+p^{\beta}\mbb{Z}_p)$-torsor of isomorphisms which map $1$ to a fixed $p^{\beta}$-root of unity modulo $p^{\beta}$ (and the twist is along the central homomorphism $\mu \colon 1 + p^{\beta}\mbb{Z}_p \to \mathcal{M}^{\diamondsuit}_{H, k}$). We refer the reader to \cite[\S 4.2]{UFJ} for more details on this twisting construction.
\end{proof}

It will be useful to extend the above lemma to reductions of structure of $P_{\bullet, \opn{dR}}^{\opn{an}}$. Let $\overline{\mathcal{P}}_G$ (resp. $\overline{\mathcal{P}}_H$) denote the formal completion of $\overline{P}_G$ (resp. $\overline{P}_H$) along its special fibre. For any integer $k \geq 1$, let $\overline{\mathcal{P}}^1_{G, k} \subset \overline{\mathcal{P}}_G$ (resp. $\overline{\mathcal{P}}^1_{H, k} \subset \overline{\mathcal{P}}_H$) denote the affinoid subgroup of elements which reduce to the identity modulo $p^k$. Set $\overline{P}^{\square}_{G, k} = \overline{\mathcal{P}}^1_{G, k} \cdot M^G_{\opn{Iw}}(p^{\beta})$ and $\overline{P}^{\diamondsuit}_{H, k} = \overline{\mathcal{P}}^1_{H, k} \cdot M^H_{\diamondsuit}(p^{\beta})$.

\begin{proposition} \label{NOctorsorsReductionProp}
    For any integer $k \geq 1$, there exists an integer $r \geq 1$ (depending on $k$) such that:
    \begin{enumerate}
        \item $P_{G, \opn{dR}}^{\opn{an}} \times_{\mathcal{X}_{G, \opn{Iw}}(p^{\beta})} \mathcal{X}_{G, w_n}(p^{\beta})_r$ (resp. $P_{H, \opn{dR}}^{\opn{an}} \times_{\mathcal{X}_{H, \diamondsuit}(p^{\beta})} \mathcal{X}_{H, \opn{id}}(p^{\beta})_r$) has a reduction of structure to an \'{e}tale $\overline{P}^{\square}_{G, k}$-torsor (resp. $\overline{P}^{\diamondsuit}_{H, k}$-torsor) which is denoted $\mathcal{P}_{G, \opn{dR}, k}$ (resp. $\mathcal{P}_{H,\opn{dR},k}$). Furthermore, $\mathcal{IG}_{G, w_n}(p^{\beta})$ (resp. $\mathcal{IG}_{H, \opn{id}}(p^{\beta})$) provides a reduction of structure of $\mathcal{P}_{G, \opn{dR}, k}$ (resp. $\mathcal{P}_{H,\opn{dR},k}$) over $\mathcal{X}_{G, w_n}(p^{\beta})$ (resp. $\mathcal{X}_{H, \opn{id}}(p^{\beta})$).
        \item One has $\mathcal{P}_{G,\opn{dR},k} \times^{\overline{\mathcal{P}}^{\square}_{G, k}} \mathcal{M}^{\square}_{G, k} = {^\mu \mathcal{M}_{G, \opn{HT}, k} }$ and $\mathcal{P}_{H,\opn{dR},k} \times^{\overline{\mathcal{P}}^{\diamondsuit}_{H, k}} \mathcal{M}^{\diamondsuit}_{H, k} = {^\mu \mathcal{M}_{H, \opn{HT}, k}}$.
    \end{enumerate}
    In particular, if $\mathscr{F} \to \mathcal{X}_{G, w_n}(p^{\beta})_r$ and $\mathscr{F}' \to \mathcal{X}_{G, w_n}(p^{\beta})_{r'}$ (resp. $\mathscr{F} \to \mathcal{X}_{H, \opn{id}}(p^{\beta})_r$ and $\mathscr{F}' \to \mathcal{X}_{H, \opn{id}}(p^{\beta})_{r'}$) are two \'{e}tale torsors satisfying (1), then there exists an integer $r'' \geq \opn{max}(r, r')$ such that
    \begin{align*}
    \mathscr{F} \times_{\mathcal{X}_{G, w_n}(p^{\beta})_r} \mathcal{X}_{G, w_n}(p^{\beta})_{r''} &= \mathscr{F}' \times_{\mathcal{X}_{G, w_n}(p^{\beta})_{r'}} \mathcal{X}_{G, w_n}(p^{\beta})_{r''} \\
    \text{(resp. } \mathscr{F} \times_{\mathcal{X}_{H, \opn{id}}(p^{\beta})_r} \mathcal{X}_{H, \opn{id}}(p^{\beta})_{r''} &= \mathscr{F}' \times_{\mathcal{X}_{H, \opn{id}}(p^{\beta})_{r'}} \mathcal{X}_{H, \opn{id}}(p^{\beta})_{r''} \text{ )}
    \end{align*}
    viewed as subsets of $P_{G, \opn{dR}}^{\opn{an}}$ (resp. $P_{H, \opn{dR}}^{\opn{an}}$).

    The collection of such torsors form a cofinal system of quasi-compact open subspaces containing the closure of the Igusa tower, which implies that
    \[
    \mathscr{N}^{\dagger}_G \cong (\mathscr{N}^{(r, k)}_{G})_{r, k}, \quad \quad \mathscr{N}^{\dagger}_H \cong (\mathscr{N}^{(r, k)}_{H})_{r, k},
    \]
    where $\mathscr{N}^{(r, k)}_{G} \defeq \mathscr{O}_{\mathcal{P}_{G, \opn{dR}, k}}$, $\mathscr{N}^{(r, k)}_{H} \defeq \mathscr{O}_{\mathcal{P}_{H, \opn{dR}, k}}$, and the transition maps in the inductive systems (running over all possible $r, k$) are induced from restriction. 
\end{proposition}
\begin{proof}
     In this proof only, let $\mathscr{G}_{k}$ denote the image of $M^G_{\opn{Iw}}(p^{\beta})$ (resp. $M^H_{\diamondsuit}(p^{\beta})$) in $M_G(\mbb{Z}/p^{k}\mbb{Z})$ (resp. $M_H(\mbb{Z}/p^{k}\mbb{Z})$). Let $\mathcal{IG}_{k, \infty}$ denote the quotient of $\mathcal{IG}_{G, w_n}(p^{\beta})$ (resp. $\mathcal{IG}_{H, \opn{id}}(p^{\beta})$) by the kernel of the map $M^G_{\opn{Iw}}(p^{\beta}) \to M_G(\mbb{Z}/p^{k}\mbb{Z})$ (resp. $M^H_{\diamondsuit}(p^{\beta}) \to M_H(\mbb{Z}/p^{k}\mbb{Z})$); so in particular, $\mathcal{IG}_{k, \infty}$ is a finite \'{e}tale $\mathscr{G}_{k}$-torsor over $\mathcal{X}_{G, w_n}(p^{\beta})$ (resp. $\mathcal{X}_{H, \opn{id}}(p^{\beta})$). Set $\mathcal{X}_{\infty} = \mathcal{X}_{G, w_n}(p^{\beta})$ and $\mathcal{X}_r = \mathcal{X}_{G, w_n}(p^{\beta})_r$ (resp. $\mathcal{X}_{\infty} = \mathcal{X}_{H, \opn{id}}(p^{\beta})$ and $\mathcal{X}_r = \mathcal{X}_{H, \opn{id}}(p^{\beta})_r$). For $r$ sufficiently large, we have a Cartesian diagram:
    \[
\begin{tikzcd}
{\mathcal{IG}_{k, \infty}} \arrow[d] \arrow[r] & {\mathcal{IG}_{k, r}} \arrow[d] \\
\mathcal{X}_{\infty} \arrow[r]                     & \mathcal{X}_r                      
\end{tikzcd}
    \]
    where $\mathcal{IG}_{k, r}$ denotes the quotient of ${^\mu \mathcal{M}_{\bullet, \opn{HT}, k}}$ by $\mathcal{M}^1_{\bullet, k}$. In particular, ${^\mu \mathcal{M}_{\bullet, \opn{HT}, k}}$ has sections locally on $\mathcal{IG}_{k, r}$. The rest of the proof is now identical to \cite[Proposition 6.2.1]{DiffOps}, exploiting the fact that $\mathcal{X}_{\infty}$ is the locus in $\mathcal{X}_r$ where $|\hat{\delta}_{\bullet, n+1}^+| = 1$ (a local generator of the ideal $\hat{\delta}_{\bullet, n+1}^+$ plays the role of the element ``$h$'' in \emph{loc.cit.}).
\end{proof}

Note that $\mathscr{N}^{\dagger}_{\bullet}$ and $\mathscr{M}^{\dagger}_{\bullet}$ are ind-sheaves of Banach spaces, so the terms in the inductive systems have a natural topology. Furthermore, Lemma \ref{OctorsorsReductionLemma} and Proposition \ref{NOctorsorsReductionProp} allow us to view $\mathscr{N}^{\dagger}_G$, $\mathscr{M}^{\dagger}_G$ and $\mathscr{N}^{\dagger}_H$, $\mathscr{M}^{\dagger}_{H}$ as ind-sheaves of Banach spaces equipped with continuous actions of $M^G_{\opn{Iw}}(p^{\beta})$ and $M^H_{\diamondsuit}(p^{\beta})$ respectively. We can also differentiate the torsor structures on $\mathcal{P}_{\bullet, \opn{dR}, k}$, to obtain actions of $\overline{\ide{p}}_G = \opn{Lie}\overline{P}_G$ and $\overline{\ide{p}}_H = \opn{Lie}\overline{P}_H$ on $\mathscr{N}^{\dagger}_G$ and $\mathscr{N}^{\dagger}_H$ respectively. If we let $\overline{\ide{u}}_{\bullet}$ denote the nilpotent part of $\overline{\ide{p}}_{\bullet}$, then the sub-ind-sheaf of $\mathscr{N}^{\dagger}_{\bullet}$ consisting of elements killed by $\overline{\ide{u}}_{\bullet}$ is precisely the ind-sheaf of overconvergent forms $\mathscr{M}^{\dagger}_{\bullet}$.

 In the following definition, we introduce the ind-sheaves of nearly overconvergent and overconvergent automorphic forms for $G$ and $H$ with a fixed weight  given by either an algebraic or locally analytic representation of $M^G_{\opn{Iw}}(p^{\beta})$ or $M^H_{\diamondsuit}(p^{\beta})$. The cohomology (with partial compact support) of the ind-sheaves of overconvergent forms will recover the spaces of overconvergent forms considered in higher Coleman theory \cite{BoxerPilloni}; the cohomology of the nearly overconvergent ind-sheaves will provide an extension of \emph{op.cit.}\ which incorporates the action of Maass--Shimura differential operators.

\begin{definition} \label{DefinitionOfNOCwithweight}
    \begin{enumerate}
        \item Let $V_{\kappa}$ be an algebraic representation of $M_G$ (resp. $M_H$) of highest weight $\kappa \in X^*(T)$. We define
        \begin{align*} 
        \mathscr{N}^{\dagger}_{G, \kappa} &\defeq (\mathscr{N}^{\dagger}_G \hatot V_{\kappa})^{M^G_{\opn{Iw}}(p^{\beta})} \defeq \left( (\mathscr{N}^{(r, k)}_G \hatot V_{\kappa})^{M^G_{\opn{Iw}}(p^{\beta})} \right)_{r, k} = \left( \mathscr{N}^{(r, k)}_{G, \kappa} \right)_{r, k} \\
        \mathscr{M}^{\dagger}_{G, \kappa} &\defeq (\mathscr{M}^{\dagger}_G \hatot V_{\kappa})^{M^G_{\opn{Iw}}(p^{\beta})} \defeq \left( (\mathscr{M}^{(r, k)}_G \hatot V_{\kappa})^{M^G_{\opn{Iw}}(p^{\beta})} \right)_{r, k} = \left( \mathscr{M}^{(r, k)}_{G, \kappa} \right)_{r, k} \\
        \text{(resp. } \mathscr{N}^{\dagger}_{H, \kappa} &\defeq (\mathscr{N}^{\dagger}_H \hatot V_{\kappa})^{M^H_{\diamondsuit}(p^{\beta})} \defeq \left( (\mathscr{N}^{(r, k)}_H \hatot V_{\kappa})^{M^H_{\diamondsuit}(p^{\beta})} \right)_{r, k} = \left( \mathscr{N}^{(r, k)}_{H, \kappa} \right)_{r, k} \\
        \mathscr{M}^{\dagger}_{H, \kappa} &\defeq (\mathscr{M}^{\dagger}_H \hatot V_{\kappa})^{M^H_{\diamondsuit}(p^{\beta})} \defeq \left( (\mathscr{M}^{(r, k)}_H \hatot V_{\kappa})^{M^H_{\diamondsuit}(p^{\beta})} \right)_{r, k} = \left( \mathscr{M}^{(r, k)}_{H, \kappa} \right)_{r, k} \text{ )}
        \end{align*} 
        to be the ind-sheaves of nearly overconvergent and overconvergent forms respectively, where the invariants are with respect to the diagonal action. 
        \item Let $(R, R^+)$ be a Tate affinoid algebra over $(L, \mathcal{O}_L)$ and suppose that $\kappa \colon T(\mbb{Z}_p) \to (R^+)^{\times}$ is an $s$-analytic character, for some $s \geq 1$. Let $V^{\circ, s\opn{-an}}_{G, \kappa}$ denote the $s$-analytic induction as in \cite[\S 6.2.20]{BoxerPilloni}, and let $D^{s\opn{-an}}_{G, \kappa^*}$ denote its continuous $R$-linear dual. We set
        \begin{align*} 
        \mathscr{N}^{\dagger, s\opn{-an}}_{G, \kappa^*} &\defeq \left( \mathscr{N}^{\dagger}_G \hatot D^{s\opn{-an}}_{G, \kappa^*} \right)^{M^G_{\opn{Iw}}(p^{\beta})} \defeq \left( \left( \mathscr{N}^{(r, k)}_G \hatot D^{s\opn{-an}}_{G, \kappa^*} \right)^{M^G_{\opn{Iw}}(p^{\beta})} \right)_{r, k \geq s+1} = \left( \mathscr{N}^{(r, k), s\opn{-an}}_{G, \kappa^*} \right)_{r, k \geq s+1}, \\
        \mathscr{M}^{\dagger, s\opn{-an}}_{G, \kappa^*} &\defeq \left( \mathscr{M}^{\dagger}_G \hatot D^{s\opn{-an}}_{G, \kappa^*} \right)^{M^G_{\opn{Iw}}(p^{\beta})} \defeq \left( \left( \mathscr{M}^{(r, k)}_G \hatot D^{s\opn{-an}}_{G, \kappa^*} \right)^{M^G_{\opn{Iw}}(p^{\beta})} \right)_{r, k \geq s+1} = \left( \mathscr{M}^{(r, k), s\opn{-an}}_{G, \kappa^*} \right)_{r, k \geq s+1} .
        \end{align*}
        \item Let $(R, R^+)$ be a Tate affinoid algebra over $(L, \mathcal{O}_L)$ and suppose that $\sigma \colon M^H_{\diamondsuit}(p^{\beta}) \to (R^+)^{\times}$ is an $s$-analytic character, for some $s \geq 1$. We set
        \begin{align*}
        \mathscr{N}^{\dagger, \opn{an}}_{H, \sigma} &\defeq \left( \mathscr{N}^{\dagger}_H \hatot \sigma \right)^{M^H_{\diamondsuit}(p^{\beta})} \defeq \left( \left( \mathscr{N}^{(r, k)}_H \hatot \sigma \right)^{M^H_{\diamondsuit}(p^{\beta})} \right)_{r, k \geq s} = \left( \mathscr{N}^{(r, k), \opn{an}}_{H, \sigma} \right)_{r, k \geq s}, \\
        \mathscr{M}^{\dagger, \opn{an}}_{H, \sigma} &\defeq \left( \mathscr{M}^{\dagger}_H \hatot \sigma \right)^{M^H_{\diamondsuit}(p^{\beta})} \defeq \left( \left( \mathscr{M}^{(r, k)}_H \hatot \sigma \right)^{M^H_{\diamondsuit}(p^{\beta})} \right)_{r, k \geq s} = \left( \mathscr{M}^{(r, k), \opn{an}}_{H, \sigma} \right)_{r, k \geq s},
        \end{align*}
        both of which are independent of $s$ (up to isomorphism in the ind-category).
    \end{enumerate}
\end{definition}

\subsubsection{Acyclicity}

We will need the following acyclicity lemma.

\begin{lemma} \label{AcyclicityCoverLemma}
    Let $\mathscr{F}^{(r, k)} \in \{ \mathscr{M}^{(r, k)}_{G, \kappa}, \hat{\iota}_* \mathscr{M}^{(r, k)}_{H, \kappa}, \mathscr{M}^{(r, k), s\opn{-an}}_{G, \kappa^*}, \hat{\iota}_*\mathscr{M}^{(r, k), \opn{an}}_{H, \sigma} \}$, where the sheaves are as in Definition \ref{DefinitionOfNOCwithweight} and $\hat{\iota} \colon \mathcal{X}_{H, \diamondsuit}(p^{\beta}) \to \mathcal{X}_{G, \opn{Iw}}(p^{\beta})$ is the (analytification of the) morphism from \S \ref{DeeperLevelAtpSubSec}. Then, after possibly increasing $r$, there exists a finite open affinoid cover $\ide{U} = \{ U_i \}_{i \in I}$ of $\mathcal{X}_{G, \opn{Iw}}(p^{\beta})$ such that:
    \begin{itemize}
        \item $U_i \in \mathcal{C}_{G, H}$ for all $i \in I$
        \item For any subset $J \subset I$, with $U_J \defeq \cap_{i \in J} U_i$, we have
        \[
        R^j \Gamma(U_J, \mathscr{F}^{(r, k)}) = 0 \quad \text{ for all } \quad j \neq 0 
        \]
        for any $(r, k)$ such that $\mathscr{F}^{(r, k)}$ is well-defined.
        \item For any subset $J \subset I$, we have
    \[
        R^j \Gamma(V_J, \mathscr{F}^{(r,k)}) = 0 \quad \text{ for all } \quad j \neq 0 
    \]
    for any $(r, k)$ such that $\mathscr{F}^{(r, k)}$ is well-defined, where $V_J \defeq U_J \cap \left( \mathcal{X}_{G, \opn{Iw}}(p^{\beta}) - \mathcal{Z}_{G, >n+1}(p^{\beta}) \right)$.
    \end{itemize}
\end{lemma}
\begin{proof}
    Let $\mathscr{F}^{(r, k)} \in \{ \mathscr{M}^{(r, k)}_{G, \kappa}, \mathscr{M}^{(r, k), s\opn{-an}}_{G, \kappa^*} \}$. Let ${^\mu \mathcal{M}_{G, \opn{HT}, k}} \to \mathcal{X}_{G, w_n}(p^{\beta})_r$ and ${^\mu \mathcal{M}_{G, \opn{HT}, k+1}} \to \mathcal{X}_{G, w_n}(p^{\beta})_{r'}$ be the torsors from Lemma \ref{OctorsorsReductionLemma}, for some $r' \geq r$. Let $\mathcal{IG}_{r', k+1} \to \mathcal{X}_{G, w_n}(p^{\beta})_{r'}$ denote the pushout of ${^\mu \mathcal{M}_{G, \opn{HT}, k+1}}$ along the map $\mathcal{M}^{\square}_{G, k+1} \twoheadrightarrow \mathcal{M}^{\square}_{G, k+1}/\mathcal{M}^1_{G, k+1} =: \mathscr{G}_{k+1}$. Let $\mathcal{IG}_{\opn{ord}, k+1} = \mathcal{IG}_{r', k+1} \times_{\mathcal{X}_{G, w_n}(p^{\beta})_{r'}} \mathcal{X}_{G, w_n}(p^{\beta})$.
    
    Let $\ide{U} = \{ U_i \}_{i \in I}$ be any finite open affinoid cover of $\mathcal{X}_{G, \opn{Iw}}(p^{\beta})$  as in Lemma \ref{Lem:SpecialCoverFromCGH}, which gives rise to sections $s_i \colon U_i \to P_{G,\opn{dR}}^{\opn{an}}$. In particular, the $+$-versions of $\mathcal{H}_A$ and its Hodge filtration trivialise over $U_i$. Then, since ${^\mu \mathcal{M}_{G, \opn{HT}, k+1}} \times_{\mathcal{X}_{G, w_n}(p^{\beta})_{r'}} \mathcal{X}_{G, w_n}(p^{\beta})$ has an explicit moduli description (locally it is the pushout to the Levi of the torsors $\mathcal{U}_{\opn{HT}, k+1}$ in \S \ref{RedOfStrOverOrdLocSSec}), one can find sections 
    \[
    t_i \colon U_i \times_{\mathcal{X}_{G, \opn{Iw}}(p^{\beta})} \mathcal{IG}_{\opn{ord}, k+1} \to {^\mu \mathcal{M}_{G, \opn{HT}, k+1}} \times_{\mathcal{X}_{G, w_n}(p^{\beta})_{r'}} \mathcal{X}_{G, w_n}(p^{\beta}) .
    \]
    Since ${^\mu \mathcal{M}_{G, \opn{HT}, k}}$ is a strict neighbourhood of ${^\mu \mathcal{M}_{G, \opn{HT}, k+1}} \times_{\mathcal{X}_{G, w_n}(p^{\beta})_{r'}} \mathcal{X}_{G, w_n}(p^{\beta})$ inside $M_{G, \opn{dR}}^{\opn{an}}$ and the cover $\ide{U}$ is finite, we can find a sufficiently large integer $r'' \geq r'$ such that $t_i$ extend to sections 
    \[
    t_i \colon U_i'' \to {^\mu \mathcal{M}_{G, \opn{HT}, k}} \times_{\mathcal{X}_{G, w_n}(p^{\beta})_r} \mathcal{X}_{G, w_n}(p^{\beta})_{r''}
    \]
    where $U_i'' \defeq U_i \times_{\mathcal{X}_{G, \opn{Iw}}(p^{\beta})} \mathcal{IG}_{r'', k+1}$. Set $U_i' = U_i \times_{\mathcal{X}_{G, \opn{Iw}}(p^{\beta})} \mathcal{X}_{G, w_n}(p^{\beta})_{r''}$. Then $U_i'' \to U_i'$ is a finite \'{e}tale Galois cover with Galois group $\mathscr{G}_{k+1}$. By the proof of \cite[Proposition 6.3.3]{BoxerPilloni}, this implies that $\mathscr{F}^{(r'',k)}|_{U_i'} = \mathscr{F}^{(r'',k)}(U_i') \hatot_{\mathcal{O}(U_i')} \mathcal{O}_{U_i'}$ with $\mathscr{F}^{(r'',k)}(U_i')$ a projective Banach $\mathcal{O}(U_i')$-module. In particular, since $U_i$ is finite-type, Proposition 2.5.17 in \emph{op.cit.} implies that $R^j \Gamma( V, \mathscr{F}^{(r'',k)} ) = 0$ for all $j \neq 0$ for any quasi-Stein open subspace $V \subset U_i'$. The result now follows for $\mathscr{F}^{(r'', k)}$ (note that both $V_J$ and $U_J$ are quasi-Stein and the morphism $\mathcal{X}_{G, w_n}(p^{\beta})_{r''} \to \mathcal{X}_{G, \opn{Iw}}(p^{\beta})$ is affinoid).

    The proof for $\mathscr{F}^{(r, k)} \in \{ \hat{\iota}_* \mathscr{M}^{(r, k)}_{H, \kappa}, \hat{\iota}_*\mathscr{M}^{(r, k), \opn{an}}_{H, \sigma} \}$ follows a similar argument on the space $\mathcal{X}_{H, \diamondsuit}(p^{\beta})$ by using the fact that $\hat{\iota}_*$ is an exact functor (as the morphism is finite)  and that $\hat{\iota}^{-1}\ide{U}$ is a cover consisting of elements in $\mathcal{C}_H$.
\end{proof}

\subsubsection{Overconvergent cohomologies}

We now define the overconvergent cohomologies that will appear in this article. To ease notation, set $\mathcal{Z}_{G, >n+1}(p^{\beta})^c = \mathcal{X}_{G, \opn{Iw}}(p^{\beta}) - \mathcal{Z}_{G, >n+1}(p^{\beta})$ and $\mathcal{Z}_{H, \opn{id}}(p^{\beta})^c = \mathcal{X}_{H, \diamondsuit}(p^{\beta}) - \mathcal{Z}_{H, \opn{id}}(p^{\beta})$.

\begin{definition} \label{OverconvergentCohomologiesDefinition}
With notation as in Definition \ref{DefinitionOfNOCwithweight}, we define:
\begin{align}
        R\Gamma_{w_n}^G(\kappa; \beta)^{(-, \dagger)} &= \varinjlim_{r,k} R\Gamma_{\mathcal{Z}_{G, >n+1}(p^{\beta})}( \mathcal{X}_{G, \opn{Iw}}(p^{\beta}), \mathscr{M}^{(r,k)}_{G, \kappa} ) \nonumber \\
        R\Gamma_{w_n, s\opn{-an}}^G(\kappa^*; \beta)^{(-, \dagger)} &= \varinjlim_{r,k} R\Gamma_{\mathcal{Z}_{G, >n+1}(p^{\beta})}( \mathcal{X}_{G, \opn{Iw}}(p^{\beta}), \mathscr{M}^{(r,k), s\opn{-an}}_{G, \kappa^*} ) \nonumber \\
        R\Gamma_{\opn{id}}^H(\kappa; \beta)^{(-, \dagger)} &= \varinjlim_{r,k} R\Gamma_{\mathcal{Z}_{H, \opn{id}}(p^{\beta})}(\mathcal{X}_{H, \diamondsuit}(p^{\beta}), \mathscr{M}^{(r,k)}_{H, \kappa}) = R\Gamma_{\mathcal{Z}_{H, \opn{id}}(p^{\beta})}(\mathcal{X}_{H, \diamondsuit}(p^{\beta}), \mathscr{M}_{H, \kappa}) \label{ExcisionEqnlabel} \\
        R\Gamma_{\opn{id}, \opn{an}}^H(\sigma; \beta)^{(-, \dagger)} &= \varinjlim_{r, k} R\Gamma_{\mathcal{Z}_{H, \opn{id}}(p^{\beta})}(\mathcal{X}_{H, \diamondsuit}(p^{\beta}), \mathscr{M}^{(r,k), \opn{an}}_{H, \sigma}) \nonumber
    \end{align}
where the last equality in (\ref{ExcisionEqnlabel}) follows from excision. The last two cohomology complexes should be thought of as the compactly supported cohomology of the dagger space associated with $\mathcal{X}_{H, \opn{id}}(p^{\beta})$, whereas the first two cohomology complexes should be thought of as the cohomology of the dagger space associated with $\mathcal{X}_{G, w_n}(p^{\beta})$ with \emph{partial} compact support.

We also set $R\Gamma^H_{\opn{id}, \opn{an}}( \mathcal{S}_{H, \diamondsuit}(p^{\beta}), \sigma)^{(-. \dagger)}$ to be the cohomology complex defined in \cite[\S 5.4]{UFJ} -- in the notation of \emph{loc.cit.}, but replacing $t$ with $\beta$ and $\sigma$ with $\sigma_n^{[\beta]}(\underline{\lambda})$, this is given by 
\[
\varprojlim_m R \Gamma_{\mathcal{Z}^H_m(p^{\beta})}( \mathcal{U}_k^H(p^{\beta}), [\sigma] ).
\]
This is simply the analogue of $R\Gamma_{\opn{id}, \opn{an}}^H(\sigma; \beta)^{(-, \dagger)}$ defined using the Shimura--Deligne variety $\mathcal{S}_{H, \diamondsuit}(p^{\beta})$ instead of $\mathcal{X}_{H, \diamondsuit}(p^{\beta})$, and we have a natural map
\[
R\Gamma_{\opn{id}, \opn{an}}^H(\sigma; \beta)^{(-, \dagger)} \to R\Gamma^H_{\opn{id}, \opn{an}}( \mathcal{S}_{H, \diamondsuit}(p^{\beta}), \sigma)^{(-. \dagger)}
\]
induced from pullback along the open and closed embedding in Lemma \ref{LemmaSDintoPEL}.
\end{definition}

\begin{remark} \label{CanBeComputeUsingCechRem}
    Let $\mathscr{F}^{(r, k)} \in \{ \mathscr{M}^{(r, k)}_{G, \kappa}, \hat{\iota}_* \mathscr{M}^{(r, k)}_{H, \kappa}, \mathscr{M}^{(r, k), s\opn{-an}}_{G, \kappa^*}, \hat{\iota}_*\mathscr{M}^{(r, k), \opn{an}}_{H, \sigma} \}$ and let $\ide{U} = (U_i)_{i \in I}$ be an open cover of $\mathcal{X}_{G, \opn{Iw}}(p^{\beta})$ as in Lemma \ref{AcyclicityCoverLemma}. Let $\ide{V} = (V_i)_{i \in I}$ be the induced cover of $\mathcal{Z}_{G, >n+1}(p^{\beta})^c$ (i.e. $V_i = U_i \cap \mathcal{Z}_{G, >n+1}(p^{\beta})^c$). Let $\opn{Cech}(\mathscr{F}^{(r, k)};\ide{U})$ and $\opn{Cech}(\mathscr{F}^{(r,k)}; \ide{V})$ denote the \v{C}ech complexes representing $R\Gamma(\mathcal{X}_{G, \opn{Iw}}(p^{\beta}), \mathscr{F}^{(r, k)} )$ and $R\Gamma(\mathcal{Z}_{G, >n+1}(p^{\beta})^c, \mathscr{F}^{(r,k)})$ respectively, which makes sense by Lemma \ref{AcyclicityCoverLemma}. Then each of the cohomology complexes in Definition \ref{OverconvergentCohomologiesDefinition} can be computed as
    \[
    \opn{Cone}\left( \varinjlim_{r, k}\opn{Cech}(\mathscr{F}^{(r, k)};\ide{U}) \to \varinjlim_{r, k} \opn{Cech}(\mathscr{F}^{(r,k)}; \ide{V}) \right)[-1] .
    \]
    Indeed, the claim for $G$ is clear, using the exact triangle for cohomology with support in a closed subspace and the fact that passing to the inductive limit is exact. For $H$, the proof of the claim is similar, however one must use the additional fact that
    \[
    R\Gamma(\mathcal{X}_{G, \opn{Iw}}(p^{\beta}), \hat{\iota}_*\mathscr{G} ) = R\Gamma(\mathcal{X}_{H, \diamondsuit}(p^{\beta}), \mathscr{G}), \quad R\Gamma(\mathcal{Z}_{G, > n+1}(p^{\beta})^c, \hat{\iota}_*\mathscr{G} ) = R\Gamma(\mathcal{Z}_{H, \opn{id}}(p^{\beta})^c, \mathscr{G})
    \]
    for a sheaf $\mathscr{G}$. The latter follows from the fact that $\hat{\iota}^{-1}(\mathcal{Z}_{G, > n+1}(p^{\beta})^c) = \mathcal{Z}_{H, \opn{id}}(p^{\beta})^c$ by Proposition \ref{HlociisHasseProp}(1).
\end{remark}

\begin{remark}
    Let $k \geq 1$ and $r' \gg r_k$ (where $r=r_k$ is as in Lemma \ref{OctorsorsReductionLemma}). Let $\mathcal{Z}_{G, > n+1}(p^{\beta})_{r'}$ denote the closure of the locus where $|\delta^+_{G, >n+1}|^{p^{r'+1}} \geq |p|$. Then we can write
    \[
    R\Gamma_{w_n}^G(\kappa; \beta)^{(-, \dagger)} = \varinjlim_k \varprojlim_{r'} R\Gamma_{\mathcal{Z}_{G, >n+1}(p^{\beta})_{r'} \cap \mathcal{X}_{G, w_n}(p^{\beta})_{r_k}}( \mathcal{X}_{G, w_n}(p^{\beta})_{r_k}, \mathscr{M}_{G, \kappa} )
    \]
    where the inverse limit is (a priori) derived. In fact, by the Mittag--Leffler property for the \v{C}ech complexes representing $R\Gamma_{\mathcal{Z}_{G, >n+1}(p^{\beta})_{r'} \cap \mathcal{X}_{G, w_n}(p^{\beta})_{r_k}}( \mathcal{X}_{G, w_n}(p^{\beta})_{r_k}, \mathscr{M}_{G, \kappa} )$ as $r'$ varies, we see that 
    \[
    \opn{H}^i_{w_n}(\kappa; \beta)^{(-, \dagger)} = \varinjlim_k \varprojlim_{r'} \opn{H}^i_{\mathcal{Z}_{G, >n+1}(p^{\beta})_{r'} \cap \mathcal{X}_{G, w_n}(p^{\beta})_{r_k}}( \mathcal{X}_{G, w_n}(p^{\beta})_{r_k}, \mathscr{M}_{G, \kappa} )
    \]
    where the left-hand side denotes the cohomology of $R\Gamma_{w_n}^G(\kappa; \beta)^{(-, \dagger)}$ (i.e. the inverse limit has vanishing higher derived functors on these complexes). This expresses the overconvergent cohomologies $\opn{H}^i_{w_n}(\kappa; \beta)^{(-, \dagger)}$ as the colimit of a limit of overconvergent cohomology groups with fixed radii of overconvergence, which is more in spirit with the cohomology groups considered in \cite{BoxerPilloni}. Similar descriptions can be obtained for the other complexes in Definition \ref{OverconvergentCohomologiesDefinition}.
\end{remark}

\subsubsection{The Gauss--Manin connection} \label{TheGMConnectionOnNdaggerSubSec}

Recall from Lemma \ref{LemmaConcreteDescriptionNablai} that we have commuting derivations
\[
\nabla_i \colon (\pi_G)_* \mathcal{O}_{P_{G, \opn{dR}}} \to (\pi_G)_* \mathcal{O}_{P_{G, \opn{dR}}} .
\]
Since the map $\pi_G$ is affine, these induce derivations $\nabla_i \colon \mathcal{O}_{P_{G, \opn{dR}}} \to \mathcal{O}_{P_{G, \opn{dR}}}$, and hence their analytifications immediately induce commuting continuous derivations $\nabla_i \colon \mathscr{N}^{\dagger}_G \to \mathscr{N}^{\dagger}_G$. We have similar operators on $\mathscr{N}^{\dagger}_H$. We summarise the properties of these derivations in the following lemma:

\begin{lemma}
    With notation as above:
    \begin{enumerate}
        \item We have actions of $C^{\opn{pol}}(\mbb{Q}_p^{\oplus 2n-1}, L)$ and $C^{\opn{pol}}(\mbb{Q}_p^{\oplus n-1}, L)$ on $\mathscr{N}^{\dagger}_G$ and $\mathscr{N}^{\dagger}_{H}$ respectively, such that action of 
        \[
        (a_1, \dots, a_{2n-1}) \mapsto a_i \in C^{\opn{pol}}(\mbb{Q}_p^{\oplus 2n-1}, L) \quad \quad \text{ (resp. } (a_1, \dots, a_{n-1}) \mapsto a_i \in C^{\opn{pol}}(\mbb{Q}_p^{\oplus n-1}, L) \text{ )}
        \]
        is given by $\nabla_i$ on $\mathscr{N}^{\dagger}_G$ (resp. $\mathscr{N}^{\dagger}_H$).
        \item The action maps $C^{\opn{pol}}(\mbb{Q}_p^{\oplus 2n-1}, L) \times \mathscr{N}^{\dagger}_G \to \mathscr{N}^{\dagger}_G$ and $C^{\opn{pol}}(\mbb{Q}_p^{\oplus n-1}, L) \times \mathscr{N}^{\dagger}_H \to \mathscr{N}^{\dagger}_H$ are $M^G_{\opn{Iw}}(p^{\beta})$ and $M^H_{\diamondsuit}(p^{\beta})$ equivariant respectively, where the action on polynomial functions is given in \S \ref{NearlyHoloFormsSubSec}.
    \end{enumerate}
\end{lemma}
\begin{proof}
    Part (1) is immediate, and part (2) follows from the fact that $\mathscr{O}_{P_{\bullet, \opn{dR}}^{\opn{an}}}$ is dense in $\mathscr{N}^{\dagger}_{\bullet}$.
\end{proof}

\subsubsection{Functoriality} \label{FunctorialityOfGMConnectionOnNOCSubSec}

We also note that we have the following functoriality results between nearly overconvergent forms for $G$ and $H$. We fix a choice of $p^{\beta}$-root of unity in $L$.

\begin{lemma}
     We have the following commutative cube:
    \[
\begin{tikzcd}
                                                                       &                                                    & {P_{H,\opn{dR}}^{\opn{an}}} \arrow[r, "\hat{\iota}"] \arrow[d, dashed]  & {P_{G, \opn{dR}}^{\opn{an}}} \arrow[d]      \\
{\mathcal{IG}_{H, \opn{id}}(p^{\beta})} \arrow[rru] \arrow[d] \arrow[r, "\hat{\iota}"] & {\mathcal{IG}_{G, w_n}(p^{\beta})} \arrow[rru] \arrow[d] & {\mathcal{X}_{H, \diamondsuit}(p^{\beta})} \arrow[r] & {\mathcal{X}_{G, \opn{Iw}}(p^{\beta})} \\
{\mathcal{X}_{H, \opn{id}}(p^{\beta})} \arrow[r] \arrow[rru, dashed]         & {\mathcal{X}_{G, w_n}(p^{\beta})} \arrow[rru]            &                                                &                                 
\end{tikzcd}
    \]
    with the bottom square Cartesian. Here $\hat{\iota} \colon \mathcal{IG}_{H, \opn{id}}(p^{\beta}) \to \mathcal{IG}_{G, w_n}(p^{\beta})$ is the morphism defined at the end of \S \ref{QuotientsOfIgVarSubSec}, and $\hat{\iota} \colon P_{H, \opn{dR}}^{\opn{an}} \to P_{G, \opn{dR}}^{\opn{an}}$ is the composition
    \[
    P_{H, \opn{dR}}^{\opn{an}} \to P_{G, \opn{dR}}^{\opn{an}} \xrightarrow{u} P_{G, \opn{dR}}^{\opn{an}}
    \]
    where the first map is the natural one and the second map is the action of $u \in \overline{P}_G^{\opn{an}}$ through the torsor structure. The front and back squares provide reduction of structures, i.e. 
    \[
    \hat{\iota}^*\mathcal{IG}_{G, w_n}(p^{\beta}) = \mathcal{IG}_{H, \opn{id}}(p^{\beta}) \times^{[M^H_{\diamondsuit}(p^{\beta}), u]} M^G_{\opn{Iw}}(p^{\beta}), \quad \hat{\iota}^* P_{G, \opn{dR}}^{\opn{an}} = P_{H, \opn{dR}}^{\opn{an}} \times^{[\overline{P}_H^{\opn{an}}, u]} \overline{P}_{G}^{\opn{an}}
    \]
    using the inclusions $u^{-1}M^H_{\diamondsuit}(p^{\beta}) u \subset M^G_{\opn{Iw}}(p^{\beta})$ and $u^{-1} \overline{P}_H^{\opn{an}} u \subset \overline{P}_{G}^{\opn{an}}$. We have a similar cube replacing $P_{\bullet, \opn{dR}}^{\opn{an}}$ with $M_{\bullet, \opn{dR}}^{\opn{an}}$. 
\end{lemma}
\begin{proof}
    This follows immediately from the definitions. 
\end{proof}

The above lemma implies that we have natural maps of ind-sheaves $\mathscr{N}^{\dagger}_G \to \hat{\iota}_*\mathscr{N}^{\dagger}_H$ and $\mathscr{M}^{\dagger}_{G} \to \hat{\iota}_* \mathscr{M}^{\dagger}_H$ both of which are $M^H_{\diamondsuit}(p^{\beta})$-equivariant via the inclusion $u^{-1} M^H_{\diamondsuit}(p^{\beta}) u \subset M^G_{\opn{Iw}}(p^{\beta})$. Note that we have a commutative diagram
\[
\begin{tikzcd}
\mathscr{N}_G \arrow[r, "u"] \arrow[d] & \hat{\iota}_* \mathscr{N}_H \arrow[d] \\
\mathscr{N}_G^{\dagger} \arrow[r]      & \hat{\iota}_*\mathscr{N}^{\dagger}_H 
\end{tikzcd}
\]
where the top map is the action of $u$ followed by the natural map $\mathscr{N}_G \to \hat{\iota}_*\mathscr{N}_H$ in Lemma \ref{ClassicalReductionOfStructuresOfPdRLemma}. Furthemore, the map $\mathscr{N}^{\dagger}_G \to \hat{\iota}_*\mathscr{N}^{\dagger}_H$ is equivariant for the action of $\overline{\ide{p}}_H$ through the inclusion $\opn{Ad}(u^{-1}) \overline{\ide{p}}_H \subset \overline{\ide{p}}_G$. Finally, we note that the map $\mathscr{N}^{\dagger}_G \to \hat{\iota}_*\mathscr{N}^{\dagger}_H$ is equivariant for the action of $C^{\opn{pol}}(\mbb{Q}_p^{\oplus n-1}, L) \subset C^{\opn{pol}}(\mbb{Q}_p^{\oplus 2n-1}, L)$ because this subspace of polynomial functions is fixed (pointwise) by the action of $u$. 

%------------------------------------

\section{\texorpdfstring{$p$}{p}-adic iteration of differential operators} \label{PadicIterationOfDiffOpsSection}

The first goal of this section is to show that the actions of differential operators on $\mathscr{N}^{\dagger}_G$ and $\mathscr{N}^{\dagger}_H$ can be $p$-adically iterated.  This will provide an extension of \cite{DiffOps} (which covers the $\opn{GL}_2$-case) to the setting of unitary Shimura varieties. The method is very similar to \emph{op.cit.}, albeit with significantly more cumbersome notation. After this we will describe the construction of the $p$-adic evaluation maps (which will ultimately give rise to the $p$-adic $L$-function), extending the construction in \cite{UFJ} to cover more general anticyclotomic twists.

We continue to keep the same conventions as in the previous section. In particular, we fix an integer $\beta \geq 1$ and work over a finite extension $L/\mbb{Q}_p$ containing $\mu_{p^{\beta}}$. Recall the definitions of $U_{G, \beta}$ and $U_{H, \beta}$ from Definition \ref{DefOfUGbetaUHbeta}. We let $C^{\opn{la}}(U_{\bullet, \beta}, L)$ denote the $L$-algebra of locally analytic functions $U_{\bullet, \beta} \to L$. For a real number $\varepsilon > 0$, let $C_{\varepsilon}(U_{G, \beta}, L) \subset C_{\opn{cont}}(U_{G, \beta}, L)$ denote the subalgebra of continuous functions
\begin{equation} \label{MultivariableMahlerEqn}
    f(x_1, \dots, x_{2n-1}) = \sum_{\underline{k}} a_{\underline{k}} \bincoeff{p^{\beta_1} x_1}{k_1} \cdots \bincoeff{p^{\beta_{2n-1}} x_{2n-1}}{k_{2n-1}}, \quad \quad (x_1, \dots, x_{2n-1}) \in U_{G, \beta}
\end{equation}
which satisfy $p^{(k_1+\cdots + k_{2n-1})\varepsilon}|a_{\underline{k}}| \to 0$ as $k_1 + \cdots + k_{2n-1} \to +\infty$, where the sum runs over all tuples of integers $\underline{k} = (k_1, \dots, k_{2n-1})$ with $k_i \geq 0$ for all $i=1, \dots, 2n-1$. Here (\ref{MultivariableMahlerEqn}) is the multivariable Mahler expansion of a continuous function $U_{G, \beta} \to L$. Note that, via the product map, we can view
\[
C_{\varepsilon}(p^{-\beta_1}\mbb{Z}_p, L) \otimes_L \cdots \otimes_L C_{\varepsilon}(p^{-\beta_{2n-1}}\mbb{Z}_p, L) \subset C_{\varepsilon}(U_{G, \beta}, L)
\]
and this subspace is dense. We define $C_{\varepsilon}(U_{H, \beta}, L)$ similarly. 

The main theorem that we wish to prove is the following:

\begin{theorem} \label{MainTheoremOnPadicIteration}
    Recall the definitions of  $\mathcal{C}_G$, $\mathcal{C}_H$, and $\mathcal{C}_{G, H}$ from Definition \ref{DefOfCGCH}. Let $\varepsilon > 0$.
    \begin{enumerate}
        \item If $U \in \mathcal{C}_G$, then there exists a continuous $L$-algebra action
        \[
        C_{\varepsilon}(U_{G, \beta}, L) \times \mathscr{N}^{\dagger}_G(U) \to \mathscr{N}^{\dagger}_G(U)
        \]
        uniquely extending the action of $C^{\opn{pol}}(\mbb{Q}_p^{\oplus 2n-1}, L)$. This action map is equivariant for the action of $M^G_{\opn{Iw}}(p^{\beta})$, where the action of $M^G_{\opn{Iw}}(p^{\beta})$ on $C_{\varepsilon}(U_{G, \beta}, L)$ is given by the same formula as in (\ref{ActionOnCpolEqn}). This action is functorial in $U$ and $\varepsilon$.
        \item If $U \in \mathcal{C}_H$, then there exists a continuous $L$-algebra action
        \[
        C_{\varepsilon}(U_{H, \beta}, L) \times \mathscr{N}^{\dagger}_H(U) \to \mathscr{N}^{\dagger}_H(U)
        \]
        uniquely extending the action of $C^{\opn{pol}}(\mbb{Q}_p^{\oplus n-1}, L)$. This action map is equivariant for the action of $M^H_{\diamondsuit}(p^{\beta})$. This action is functorial in $U$ and $\varepsilon$.
        \item If $U \in \mathcal{C}_{G, H}$, the pullback map $\mathscr{N}^{\dagger}_G(U) \to \mathscr{N}^{\dagger}_H(\hat{\iota}^{-1}U)$ is equivariant for the action of $C_{\varepsilon}(U_{H, \beta}, L) \subset C_{\varepsilon}(U_{G, \beta}, L)$ (note that $\hat{\iota}^{-1}U \in \mathcal{C}_H$).
    \end{enumerate}
\end{theorem}

The proof of this theorem will occupy \S \ref{OrdExpStrictNhoodGSec}--\ref{TheConstuctionForHSec} following the strategy in \cite{DiffOps} (supplemented with the results in \S \ref{ContinuousOpsOnBanachSpacesSec}).

\subsection{Ordinary explicit strict neighbourhoods for \texorpdfstring{$G$}{G}} \label{OrdExpStrictNhoodGSec}

In this section, we fix $\opn{Spa}(A, A^+) \in \mathcal{C}_G$. To ease notation, for any integer $k \geq 1$, we let $\mathscr{G}_k$ denote the image of $M^G_{\opn{Iw}}(p^{\beta})$ under the map $M_G(\mbb{Z}_p) \to M_G(\mbb{Z}/p^{k}\mbb{Z})$. We have a decomposition $\mathscr{G}_{k} = \mathscr{G}_{k, \opn{sim}} \times \prod_{\tau \in \Psi} \mathscr{G}_{k, \tau}$ (where ``$\opn{sim}$'' stands for similitude), where
\[
\mathscr{G}_{k, \opn{sim}} = \left(\mbb{Z}/p^k\mbb{Z}\right)^{\times}, \quad \mathscr{G}_{k, \tau_0} = \left(\mbb{Z}/p^{k}\mbb{Z}\right)^{\times} \times \opn{Iw}_{\opn{GL}_{2n-1}}(p^{\beta})_{\opn{mod} p^k}, \quad \mathscr{G}_{k, \tau} = \opn{Iw}_{\opn{GL}_{2n}}(p^{\beta})_{\opn{mod} p^k} \; \; \; (\tau \neq \tau_0) .
\]
Here $\opn{Iw}_{\opn{GL}_d}(p^{\beta})_{\opn{mod} p^k}$ denotes the image of the depth $p^{\beta}$ upper-triangular Iwahori subgroup of $\opn{GL}_d$ under the map $\opn{GL}_d(\mbb{Z}_p) \to \opn{GL}_d(\mbb{Z}/p^k\mbb{Z})$, and recall that $M_G = \opn{GL}_1 \times \left( \opn{GL}_1 \times \opn{GL}_{2n-1} \right) \times \prod_{\tau \neq \tau_0} \opn{GL}_{2n}$.

We have a tower of finite \'{e}tale torsors
\[
\cdots \to \ide{IG}_{G, w_n}(p^{\beta})_k \to \cdots \to \ide{IG}_{G, w_n}(p^{\beta})_1 \to \ide{X}_{G, w_n}(p^{\beta})
\]
where $\ide{IG}_{G, w_n}(p^\beta)_k \to \ide{X}_{G, w_n}(p^{\beta})$ is a torsor for the group $\mathscr{G}_k$ and $\ide{IG}_{G, w_n}(p^{\beta})_{\infty} \defeq \ide{IG}_{G, w_n}(p^{\beta}) = \varprojlim_k \ide{IG}_{G, w_n}(p^{\beta})_k$.

\begin{notation}
    We let $\opn{Spa}(A_{\opn{ord}}, A^+_{\opn{ord}})$ denote the pullback of $\opn{Spa}(A, A^+) \subset \mathcal{X}_{G, \opn{Iw}}(p^{\beta})$ to $\mathcal{X}_{G, w_n}(p^{\beta})$, which is the adic generic fibre of an open $\opn{Spf}(A^+_{\opn{ord}}) \subset \ide{X}_{G, w_n}(p^{\beta})$ by assumption. We let $\opn{Spf}(A^+_{\opn{ord}, k})$ and $\opn{Spf}(A^+_{\opn{ord}, \infty})$ denote the pullbacks of $\opn{Spf}(A^+_{\opn{ord}})$ to $\ide{IG}_{G, w_n}(p^{\beta})_k$ and $\ide{IG}_{G, w_n}(p^{\beta})$ respectively.
    
    We fix a generator $h$ of $\hat{\delta}^+_{G, n+1}$ over $\opn{Spa}(A, A^+)$, so in particular $A^+_{\opn{ord}} = A^+\langle 1/h \rangle$. We also have that $A^+_{\opn{ord}, \infty}$ is the $p$-adic completion of $\varinjlim_k A^+_{\opn{ord}, k}$.
\end{notation}

We set $A_{\opn{ord}, k} = A^+_{\opn{ord}, k}[1/p]$ ($k \leq +\infty$) so in particular $\opn{Spa}(A_{\opn{ord}, k}, A_{\opn{ord}, k}^+)$ is the pullback of $\opn{Spa}(A, A^+) \subset \mathcal{X}_{G, \opn{Iw}}(p^{\beta})$ to $\mathcal{IG}_{G, w_n}(p^{\beta})_k$.

\subsubsection{Canonical bases}

We now fix several bases of the relative de Rham homology of the universal $\boldsymbol\Psi$-unitary abelian scheme, respecting the Hodge filtration.

\begin{notation} \label{FixedBasisOverSpfA+Notation}
    For $\tau \in \Psi$, fix bases $\{ e_{1, \tau}, \dots, e_{2n, \tau} \}$ and $\{ e_{2n, \bar{\tau}}, \dots, e_{1, \bar{\tau}} \}$ of $\mathcal{H}_{\mathcal{A}, \tau}$ and $\mathcal{H}_{\mathcal{A}, \bar{\tau}}$ over $\opn{Spf}A^+$ respectively, preserving the Hodge filtrations and satisfying the property $\langle e_{i, \tau}, e_{j, \bar{\tau}} \rangle = \delta_{i j}$.\footnote{Here $\delta_{ij}$ is the Kronecker delta function.} In particular, this defines a section of $P^{\opn{an}}_{G, \opn{dR}, A} \defeq P_{G,\opn{dR}}^{\opn{an}} \times_{\mathcal{X}_{G, \opn{Iw}}(p^{\beta})} \opn{Spa}(A, A^+) \to \opn{Spa}(A, A^+)$. 
\end{notation}

We also have canonical bases over the layers of the Igusa tower. Indeed, let $k \geq 1$ be an integer, then from the universal object over $\opn{Spf}(A_{\opn{ord}, k}^+) \subset \ide{IG}_{G, w_n}(p^{\beta})_k$, we obtain the data $(s, f^{(1)}_{\tau_0}, f^{(2)}_{\tau_0}, f_{\tau} \; (\tau \neq \tau_0))$ where $s \in \underline{\left(\mbb{Z}/p^k \mbb{Z} \right)^{\times}}(A_{\opn{ord}, k}^+)$ and we have isomorphisms
\begin{align*}
    f^{(1)}_{\tau_0} \colon \mu_{p^k} &\xrightarrow{\sim} \mathcal{A}[\ide{p}_{\tau_0}^k]^{\circ} \\
    f^{(2)}_{\tau_0} \colon \left( \mbb{Z}/p^{k}\mbb{Z} \right)^{\oplus 2n-1} &\xrightarrow{\sim} \mathcal{A}[\ide{p}_{\tau_0}^k]/\mathcal{A}[\ide{p}_{\tau_0}^k]^{\circ} \\
    f_{\tau} \colon \left( \mbb{Z}/p^{k}\mbb{Z} \right)^{\oplus 2n} &\xrightarrow{\sim} \mathcal{A}[\ide{p}_{\tau}^k]
\end{align*}
over $\opn{Spf}(A_{\opn{ord}, k}^+)$, where $(-)^{\circ}$ denotes the connected component of a finite flat group scheme. The isomorphisms are compatible with the flags $\mathcal{C}_{\bullet, \bullet}$ of finite flat groups schemes in the universal $\boldsymbol\Psi$-unitary abelian scheme $\mathcal{A}$ over $\opn{Spf}(A_{\opn{ord}, k}^+)$  in the following sense: for $1 \leq i \leq n$, the image of the first $i$ standard basis elements of $\left( p^{\opn{max}(0, k-\beta)}\mbb{Z}/p^{k}\mbb{Z} \right)^{\oplus 2n-1}$ under $f^{(2)}_{\tau_0}$ generate $\mathcal{C}_{i, \tau_0}$ modulo $\mathcal{A}[\ide{p}_{\tau_0}^{\opn{min}(k, \beta)}]^{\circ}$; for $n+1 \leq i \leq 2n-1$, the image of the first $i$ standard basis elements of $\left( p^{\opn{max}(0, k-\beta)}\mbb{Z}/p^{k}\mbb{Z} \right)^{\oplus 2n-1}$ under $f^{(2)}_{\tau_0}$ generate $\mathcal{C}_{i+1, \tau_0}$ modulo $\mathcal{A}[\ide{p}_{\tau_0}^{\opn{min}(k, \beta)}]^{\circ}$; and for $1 \leq i \leq 2n$ and $\tau \neq \tau_0$, the image of the first $i$ standard basis elements of $\left( p^{\opn{max}(0, k-\beta)}\mbb{Z}/p^{k}\mbb{Z} \right)^{\oplus 2n}$ under $f_{\tau}$ generate $\mathcal{C}_{i, \tau}$ modulo $\mathcal{A}[\ide{p}_{\tau}^{\opn{min}(k, \beta)}]^{\circ}$. The reason for this is because the pullback of the universal flags $\mathcal{C}_{\bullet, \bullet}$ under the natural map $\ide{IG}_G \to \ide{X}_{G, w_n}(p^{\beta})$ satisfies
\[
\mathcal{C}_{i, \tau} = \left\{ \begin{array}{cc} f_{\tau_0}(\{1\} \oplus (p^{-\beta}\mbb{Z}_p/\mbb{Z}_p)^{\oplus i} \oplus \{0\}^{\oplus 2n-1-i}) & \text{ if } 1 \leq i \leq n, \tau = \tau_0 \\ f_{\tau_0}(\mu_{p^{\beta}} \oplus (p^{-\beta}\mbb{Z}_p/\mbb{Z}_p)^{\oplus i-1} \oplus \{0\}^{\oplus 2n-i}) & n+1 \leq i \leq 2n, \tau = \tau_0 \\ f_{\tau}((p^{-\beta}\mbb{Z}_p/\mbb{Z}_p)^{\oplus i} \oplus \{0\}^{\oplus 2n-i}) & 1 \leq i \leq 2n, \tau \neq \tau_0 \end{array} \right.
\]
where $f_{\tau} \colon \mbb{X}_{\opn{ord}, \tau} \xrightarrow{\sim} \mathcal{A}[\ide{p}_{\tau}^{\infty}]$ denotes the universal trivialisations in Definition \ref{DefOfCSIgusa}.

We introduce the following basis elements:
\begin{itemize}
    \item Let $e_{1, \tau_0, k}^{\opn{can}} \in \mathcal{H}_{\mathcal{A}, \tau_0}/p^k$ denote the image of 
    \[
    \opn{Lie}(f_{\tau_0}^{(1)})(t \partial_t ) \in \opn{Lie}(\mathcal{A}[\ide{p}_{\tau_0}^k]^{\circ}) = \opn{Lie}(\mathcal{A})_{\tau_0}/p^k
    \]
    under the unit root splitting, where $t \partial_t \in \opn{Lie}\mu_{p^k}$ denotes the canonical tangent vector.
    \item For $i=2, \dots, 2n$, let $e_{i, \tau_0, k}^{\opn{can}} \in \mathcal{H}_{\mathcal{A}, \tau_0}$ denote the element
    \[
    \opn{dlog} \circ f^{(2)}_{\tau_0}(0, \dots, 0, 1, 0, \dots, 0) \in \omega_{\mathcal{A}^D, \tau_0}/p^k \subset \mathcal{H}_{\mathcal{A}, \tau_0}/p^k
    \]
    where $1$ is in the $(i-1)$-th place.
    \item For $\tau \neq \tau_0$ and $i=1, \dots, 2n$, let $e_{i, \tau, k}^{\opn{can}}$ denote the element
    \[
    \opn{dlog} \circ f_{\tau}(0, \dots, 0, 1, 0, \dots, 0) \in \omega_{\mathcal{A}^D, \tau}/p^k = \mathcal{H}_{\mathcal{A}, \tau}/p^k
    \]
    where the $1$ is in the $i$-th place.
    \item For any $i=1, \dots, 2n$ and $\tau \in \Psi$, let $e_{i, \bar{\tau}, k}^{\opn{can}} \in \mathcal{H}_{\mathcal{A}, \bar{\tau}}$ denote the unique element such that $\langle e_{i, \tau, k}^{\opn{can}}, e_{j, \bar{\tau}, k}^{\opn{can}} \rangle = s \delta_{ij}$. 
\end{itemize}

\begin{definition} \label{DefOfCanonicalBasesOverIgTower}
    Let $\ide{B}^{\opn{can}}_k$ denote the basis of $\mathcal{H}_{\mathcal{A}}/p^k$ over $\opn{Spf}A_{\opn{ord}, k}^+$ given by
    \[
    \bigcup_{\tau \in \Psi} \{ e_{1, \tau, k}^{\opn{can}}, \dots, e_{2n, \tau, k}^{\opn{can}} \} \cup \{ e_{2n, \bar{\tau}, k}^{\opn{can}}, \dots, e_{1, \bar{\tau}, k}^{\opn{can}} \}.
    \]
    This respects the Hodge filtration and symplectic structure. If $k = +\infty$, then we can define a basis 
    \[
    \ide{B}^{\opn{can}}_{\infty} = \bigcup_{\tau \in \Psi} \{ e_{1, \tau, \infty}^{\opn{can}}, \dots, e_{2n, \tau, \infty}^{\opn{can}} \} \cup \{ e_{2n, \bar{\tau}, \infty}^{\opn{can}}, \dots, e_{1, \bar{\tau}, \infty}^{\opn{can}} \}
    \]
    of $\mathcal{H}_{\mathcal{A}}$ over $\opn{Spf}A_{\opn{ord}, \infty}^+$ in a similar way as above, by replacing the finite flat groups schemes by their respective $p$-divisible groups.
\end{definition}

\subsubsection{Reductions of structure over the ordinary locus} \label{RedOfStrOverOrdLocSSec}

We now construct certain reductions of structure of $P_{G, \opn{dR}}^{\opn{an}}$ over $\mathcal{X}_{G, w_n}(p^{\beta})$. Let $\overline{\mathcal{P}}^1_{G, k} \subset \overline{P}_G$ denote the formal subgroup scheme of elements which are congruent to the identity modulo $p^k$. By abuse of notation, we will also denote its adic generic fibre by $\overline{\mathcal{P}}^1_{G, k}$.

\begin{definition}
    Let $k \geq 1$ be an integer. Let $\ide{U}_{\opn{HT}, k} \to \opn{Spf}A_{\opn{ord}, k}^+$ denote the functor 
    \[
    \ide{U}_{\opn{HT}, k}(S \xrightarrow{f} \opn{Spf}A_{\opn{ord}, k}^+) = \left\{ \begin{array}{c} \text{ bases of } f^*\mathcal{H}_{\mathcal{A}} \text{ respecting the Hodge filtration and endomorphism } \\ \text{ and symplectic structure which are congruent to } \ide{B}_k^{\opn{can}} \text{ modulo } p^k  \end{array} \right\} .
    \]
    This defines an analytic torsor under the group $\overline{\mathcal{P}}^1_{G, k}$. Furthermore, $\ide{U}_{\opn{HT}, k} \to \opn{Spf}A_{\opn{ord}}^+$ is an \'{e}tale $\overline{\mathcal{P}}^{\square}_{G, k} = \overline{\mathcal{P}}^1_{G, k} \cdot M^G_{\opn{Iw}}(p^{\beta})$-torsor. If $\mathcal{U}_{\opn{HT}, k} \to \opn{Spa}(A_{\opn{ord}}, A_{\opn{ord}}^+)$ denotes its adic generic fibre, then this defines a reduction of structure of $P^{\opn{an}}_{G, \opn{dR}, A_{\opn{ord}}} = P^{\opn{an}}_{G, \opn{dR}} \times_{\mathcal{X}_{G, \opn{Iw}}(p^{\beta})} \opn{Spa}(A_{\opn{ord}}, A_{\opn{ord}}^+)$.
\end{definition}

Clearly one has a natural map $\opn{Spf}A_{\opn{ord}, \infty}^+ \to \ide{U}_{\opn{HT}, k}$ given by the canonical basis $\ide{B}_{\infty}^{\opn{can}}$, and the collection $\{ \mathcal{U}_{\opn{HT}, k} \}_{k \geq 1}$ defines a cofinal collection of quasi-compact open subspaces of $P^{\opn{an}}_{G, \opn{dR}, A_{\opn{ord}}}$ containing the closure of $\opn{Spa}(A_{\opn{ord}, \infty}, A_{\opn{ord}, \infty}^+)$ in $P^{\opn{an}}_{G, \opn{dR}, A_{\opn{ord}}}$.

\begin{lemma}
    One has
    \[
    \mathcal{U}_{\opn{HT}, k} \times_{\opn{Spa}(A_{\opn{ord}, k}, A_{\opn{ord}, k}^+)} \opn{Spa}(A_{\opn{ord, \infty}}, A_{\opn{ord}, \infty}^+) \cong \opn{Spa}\left( A_{\opn{ord}, \infty}\langle \overline{\mathcal{P}}^1_{G, k} \rangle, A_{\opn{ord}, \infty}^+\langle \overline{\mathcal{P}}^1_{G, k} \rangle \right)
    \]
    where $A_{\opn{ord}, \infty}\langle \overline{\mathcal{P}}^1_{G, k}\rangle$ (resp. $A_{\opn{ord}, \infty}^+\langle \overline{\mathcal{P}}^1_{G, k} \rangle$) denotes global sections of the group $\overline{\mathcal{P}}^1_{G, k}$ over $A_{\opn{ord}, \infty}$ (resp. $A_{\opn{ord}, \infty}^+$), i.e. convergent power series on $\overline{\mathcal{P}}^1_{G, k}$.
\end{lemma}
\begin{proof}
    This simply follows from the isomorphism
    \begin{align*}
    \overline{\mathcal{P}}^1_{G, k} \times \opn{Spa}(A_{\opn{ord, \infty}}, A_{\opn{ord}, \infty}^+) &\xrightarrow{\sim} \mathcal{U}_{\opn{HT}, k} \times_{\opn{Spa}(A_{\opn{ord}, k}, A_{\opn{ord}, k}^+)} \opn{Spa}(A_{\opn{ord, \infty}}, A_{\opn{ord}, \infty}^+) \\
    p &\mapsto \ide{B}_{\infty}^{\opn{can}} \cdot p .
    \end{align*}
\end{proof}

\subsubsection{The Gauss--Manin connection (ordinary case)}

Fix an integer $i=1, \dots, 2n-1$ and recall the definition of the integer $\beta_i$ from the proof of Proposition \ref{IntertwiningOfNablaThetaiProp} (i.e. $\beta_i = \beta$ if $i=1, \dots, n$ and $0$ if $i=n+1, \dots, 2n-1$). Recall from \S \ref{TheGMConnectionOnNdaggerSubSec} that we have an operator
\[
\nabla_i \colon \mathcal{O}_{P_{G, \opn{dR}}^{\opn{an}}} \to \mathcal{O}_{P_{G, \opn{dR}}^{\opn{an}}} .
\]
Let $k \geq 1$ be an integer. Since $\mathcal{U}_{\opn{HT}, k}$ is open in $P_{G, \opn{dR}}^{\opn{an}}$, we obtain an induced operator
\[
\nabla_i \colon \mathcal{O}_{\mathcal{U}_{\opn{HT}, k}}(\mathcal{U}_{\opn{HT}, k}) \to \mathcal{O}_{\mathcal{U}_{\opn{HT}, k}}(\mathcal{U}_{\opn{HT}, k}) .
\]
Furthermore, since $\opn{Spa}(A_{\opn{ord}, \infty}, A_{\opn{ord}, \infty}^+) \to \opn{Spa}(A_{\opn{ord}, k}, A_{\opn{ord}, k}^+)$ is a pro-\'{e}tale torsor, this operator extends to an operator $\nabla_i \colon \mathcal{O}_{\mathcal{U}_{\opn{HT}, k, \infty}}(\mathcal{U}_{\opn{HT}, k, \infty}) \to \mathcal{O}_{\mathcal{U}_{\opn{HT}, k, \infty}}(\mathcal{U}_{\opn{HT}, k, \infty})$, where 
\[
\mathcal{U}_{\opn{HT}, k, \infty} = \mathcal{U}_{\opn{HT}, k} \times_{\opn{Spa}(A_{\opn{ord}, k}, A_{\opn{ord}, k}^+)} \opn{Spa}(A_{\opn{ord}, \infty}, A_{\opn{ord}, \infty}^+) .
\]
The following proposition will allow us to apply the general results on continuous operators in \S \ref{ContinuousOpsOnBanachSpacesSec}.

\begin{proposition} \label{StandardFormpbetaNablaiProp}
    With notation as above:
    \begin{enumerate}
        \item The operators 
        \begin{align*} 
        p^{\beta_i}\nabla_i \colon \mathcal{O}_{\mathcal{U}_{\opn{HT}, k}}(\mathcal{U}_{\opn{HT}, k}) &\to \mathcal{O}_{\mathcal{U}_{\opn{HT}, k}}(\mathcal{U}_{\opn{HT}, k}) \\ p^{\beta_i}\nabla_i \colon \mathcal{O}_{\mathcal{U}_{\opn{HT}, k, \infty}}(\mathcal{U}_{\opn{HT}, k, \infty}) &\to \mathcal{O}_{\mathcal{U}_{\opn{HT}, k, \infty}}(\mathcal{U}_{\opn{HT}, k, \infty})
        \end{align*}
        are integral, i.e., they preserve $\mathcal{O}^+_{\mathcal{U}_{\opn{HT}, k}}(\mathcal{U}_{\opn{HT}, k})$ and $\mathcal{O}^+_{\mathcal{U}_{\opn{HT}, k, \infty}}(\mathcal{U}_{\opn{HT}, k, \infty})$ respectively.
        \item One has a continuous $A^+_{\opn{ord}, \infty}$-algebra isomorphism
        \[
        \mathcal{O}^+_{\mathcal{U}_{\opn{HT}, k, \infty}}(\mathcal{U}_{\opn{HT}, k, \infty}) \cong A_{\opn{ord}, \infty}^+ \langle X_{\opn{sim}}, X_{a, b, \tau} \rangle
        \]
        where $X_{\opn{sim}}$ and $X_{a, b, \tau}$ are variables, $\tau$ runs over $\Psi$, and $1 \leq a, b \leq 2n$ are integers, omitting the variables $X_{a, b, \tau_0}$ for $(a, b) \in \{ (1, 2), \dots, (1, 2n) \}$.
        \item Under the isomorphism in (2), the operator $p^{\beta_i}\nabla_i$ satisfies the following properties:
        \begin{itemize}
            \item $p^{\beta_i}\nabla_i (x) \equiv p^{\beta_i}\theta_i(x)$ modulo $p^k$ for all $x \in A_{\opn{ord}, \infty}^+$
            \item $p^{\beta_i}\nabla_i(X_{\opn{sim}}) \equiv 0$ and $p^{\beta_i}\nabla_i (X_{a, b, \tau}) \equiv 0$ modulo $p^k$, for all $(a, b, \tau)$ with $\tau \neq \tau_0$
            \item $p^{\beta_i}\nabla_i (X_{1, 1, \tau_0}) \equiv -p^{\beta_i} X_{i+1, 1, \tau_0}$ modulo $p^k$
            \item $p^{\beta_i}\nabla_i(X_{i+1, i+1, \tau_0}) \equiv 0$ and $p^{\beta_i}\nabla_i(X_{a, b, \tau_0}) \equiv 0$ modulo $p^k$, for all $a \geq 2$ and $b \neq i+1$
            \item $p^{\beta_i}\nabla_i(X_{a, i+1, \tau_0}) \equiv p^{\beta_i} X_{a, 1, \tau_0}$ modulo $p^k$, for all $a \geq 2$ with $a \neq i+1$.
        \end{itemize}
        Here, $\theta_i$ denotes the operator from Proposition \ref{IntertwiningOfNablaThetaiProp} arising from the action of $C_{\opn{cont}}(U_{G, \beta}, L)$ on $A_{\opn{ord}, \infty}$.
    \end{enumerate}
\end{proposition}
\begin{proof}
    Recall that we have a canonical basis $\ide{B}^{\opn{can}}_{\infty}$ of $\mathcal{H}_{\mathcal{A}}$ over $A_{\opn{ord}, \infty}^+$. Then the universal basis of $\mathcal{H}_{\mathcal{A}}$ over $\mathcal{U}_{\opn{HT}, k, \infty}$ is given by
    \[
    \ide{B}_{\infty}^{\opn{can}} \cdot \underline{X}'
    \]
    where $\underline{X}' = X_{\opn{sim}}' \times \prod_{\tau} (X_{a, b, \tau}')_{a, b}$ is the universal element in $\overline{P}^{\opn{an}}_G$ over $\mathcal{U}_{\opn{HT}, k, \infty}$. This gives coordinates for $\mathcal{O}^+_{\mathcal{U}_{\opn{HT}, k, \infty}}(\mathcal{U}_{\opn{HT}, k, \infty})$ and we find that
    \[
    \mathcal{O}^+_{\mathcal{U}_{\opn{HT}, k, \infty}}(\mathcal{U}_{\opn{HT}, k, \infty}) \cong A_{\opn{ord}, \infty}^+ \langle \frac{X_{\opn{sim}}' - 1}{p^k}, \frac{X_{a, b, \tau}' - \delta_{a b}}{p^k} \rangle .
    \]
    Let $\underline{Y} = (\underline{X}')^{-1} = Y_{\opn{sim}} \times \prod_{\tau} (Y_{a, b, \tau})_{a, b}$.
    
    Now let $\mathcal{K}_i \in \Omega^1_{A_{\opn{ord}, \infty}^+/\mathcal{O}_L}$ denote the unique differential such that
    \[
    \nabla^{\opn{GM}}(e_{i+1, \tau_0, \infty}^{\opn{can}}) = e_{1, \tau_0, \infty} \otimes \mathcal{K}_i
    \]
    and note that $\theta_i \colon A_{\opn{ord}, \infty} \to A_{\opn{ord}, \infty}$ is the continuous derivation dual to $\mathcal{K}_i$ (with respect to the basis $\{\mathcal{K}_1, \dots, \mathcal{K}_{2n-1} \}$ of $\Omega^1_{A_{\opn{ord}, \infty}/L}$ -- see the proof of Proposition \ref{IntertwiningOfNablaThetaiProp}). Let $E_{1, i+1} \in \ide{g}$ denote the element such that
    \[
    \nabla^{\opn{GM}}_{\theta_i}(\ide{B}_{\infty}^{\opn{can}}) = \ide{B}_{\infty}^{\opn{can}} \cdot E_{1, i+1} .
    \]
    As explained in the proof of Proposition \ref{IntertwiningOfNablaThetaiProp}, $E_{1, i+1}$ is equal to the block upper nilpotent matrix defined in Corollary \ref{KeyCorForHolOnX}, hence the notation. 
    
    By the explicit description of $\nabla_i$ in Lemma \ref{LemmaConcreteDescriptionNablai} and the description of the $\mathcal{D}$-module structure in \S \ref{DmodulesOnFLsection}, we see that for $F \in A_{\opn{ord}, \infty} \langle \frac{X_{\opn{sim}}' - 1}{p^k}, \frac{X_{a, b, \tau}' - \delta_{a b}}{p^k} \rangle$, one has
    \begin{equation} \label{ExplcitEqnforNablaiOrd}
    \nabla_i(F)(X'_{\opn{sim}}, X'_{a, b, \tau}) = \left( \theta_i \cdot F(X'_{\opn{sim}}, X'_{a, b, \tau}) + E_{1, i+1} \star_{l} F(X'_{\opn{sim}}, X'_{a, b, \tau})) \right) \cdot Y_{1, 1, \tau_0}^{-1}Y_{i+1, i+1, \tau_0}
    \end{equation}
    where $\theta_i \cdot F$ means act by the derivation on the coefficients, $\star_l$ is the $\ide{g}$-action on $\mathcal{O}(\overline{\mathcal{P}}^1_{G, k})$ viewing $F \in A_{\opn{ord}, \infty} \hatot \mathcal{O}(\overline{\mathcal{P}}^1_{G, k})$, and the last factor arises from the final three bullet points in Lemma \ref{LemmaConcreteDescriptionNablai}. 

    From this description, one immediately sees that $p^{\beta_i}\nabla_i$ preserves $\mathcal{O}^+_{\mathcal{U}_{\opn{HT}, k, \infty}}(\mathcal{U}_{\opn{HT}, k, \infty})$.  Since $A_{\opn{ord}, k}^+ \to A_{\opn{ord}, \infty}^+$ is a profinite pro-\'{e}tale torsor, the induced morphism $\ide{U}_{\opn{HT}, k, \infty} = \ide{U}_{\opn{HT}, k} \times_{\opn{Spf}(A_{\opn{ord}, k}^+)} \opn{Spf}(A_{\opn{ord}, \infty}^+) \to \ide{U}_{\opn{HT}, k}$ is also a profinite pro-\'{e}tale torsor. This implies that the induced map $\mathcal{O}(\ide{U}_{\opn{HT}, k})/p^{\ell} \to \mathcal{O}(\ide{U}_{\opn{HT}, k, \infty})/p^{\ell}$ is injective for any $\ell \geq 1$ (since the source is identified with the invariants of the target under the action of a profinite pro-\'{e}tale group scheme). This implies that the map
    \begin{equation} \label{UHTktoinftyIsometry}
    \mathcal{O}^+_{\mathcal{U}_{\opn{HT}, k}}(\mathcal{U}_{\opn{HT}, k}) \hookrightarrow \mathcal{O}^+_{\mathcal{U}_{\opn{HT}, k}}(\mathcal{U}_{\opn{HT}, k}) \hatot_{A_{\opn{ord, k}}^+} A_{\opn{ord}, \infty}^+ = \mathcal{O}^+_{\mathcal{U}_{\opn{HT}, k, \infty}}(\mathcal{U}_{\opn{HT}, k, \infty})
    \end{equation}
    is injective modulo $p^{\ell}$ for any $\ell \geq 1$, and hence the operator $p^{\beta_i}\nabla_i$ also preserves $\mathcal{O}^+_{\mathcal{U}_{\opn{HT}, k}}(\mathcal{U}_{\opn{HT}, k})$. This proves part (1).

    Now using the explicit formula in (\ref{ExplcitEqnforNablaiOrd}), we can compute the action of $p^{\beta_i}\nabla_i$ in terms of the coordinates $\underline{X}'$. Note that $Y_{1, 1, \tau_0}^{-1}Y_{i+1, i+1, \tau_0} \equiv 1$ modulo $p^k$. An explicit matrix calculation of the action $E_{1, i+1} \star_l -$ shows that:
    \begin{itemize}
        \item $p^{\beta_i}\nabla_i (x) \equiv p^{\beta_i}\theta_i(x)$ modulo $p^k$, for any $x \in A_{\opn{ord}, \infty}^+$
        \item $p^{\beta}\nabla_i (X'_{\opn{sim}}) \equiv 0$ and $p^{\beta_i}\nabla_i(X'_{a, b, \tau}) \equiv 0$ modulo $p^k$ for any $(a, b, \tau)$ with $\tau \neq \tau_0$
        \item $p^{\beta_i}\nabla_i(X'_{1, 1, \tau_0}) \equiv -p^{\beta_i} X'_{i+1, 1, \tau_0}$ modulo $p^k$
        \item $p^{\beta_i}\nabla_i(X'_{j, 1, \tau_0}) \equiv 0$ modulo $p^k$, for all $2 \leq j \leq 2n$
        \item $p^{\beta_i}\nabla_i(X'_{a, b, \tau_0}) \equiv p^{\beta_i} X'_{a, 1, \tau_0} (X'_{1, 1, \tau_0})^{-1} X'_{i+1, b, \tau_0}$ modulo $p^k$, for all $2 \leq a, b \leq 2n$.
    \end{itemize}
    Now consider the following change of coordinates:
    \begin{itemize}
        \item $X_{\opn{sim}} = \frac{X'_{\opn{sim}} - 1}{p^k}$ and $X_{a, b, \tau} = \frac{X'_{a, b, \tau} - \delta_{ab}}{p^k}$ for all $(a, b, \tau)$ with $\tau \neq \tau_0$
        \item $X_{a, b, \tau_0} = \frac{X'_{a, b, \tau_0} - \delta_{ab}}{p^k}$ for all $(a, b)$ (except those in the set $\{(1, 2), \dots, (1, 2n)\}$) with $(a, b) \neq (i+1, i+1)$
        \item $X_{i+1, i+1, \tau_0} = X_{1, 1, \tau_0} + \frac{X_{i+1, i+1, \tau_0} - 1}{p^k}$ .
    \end{itemize}
    This change of coordinates gives rise to a continuous $A^+_{\opn{ord}, \infty}$-algebra isomorphism 
    \[
    \mathcal{O}^+_{\mathcal{U}_{\opn{HT}, k, \infty}}(\mathcal{U}_{\opn{HT}, k, \infty}) \cong A^+_{\opn{ord}, \infty}\langle X_{\opn{sim}}, X_{a, b, \tau} \rangle 
    \]
    and one can easily verify that the properties in (3) are satisfied.
\end{proof}

We now apply the general results in \S \ref{ContinuousOpsOnBanachSpacesSec}. 

\begin{proposition} \label{epsilonactionOrdProp}
    Let $\varepsilon > 0$. There exists an integer $k(\varepsilon) \geq 1$ such that, for any $k \geq k(\varepsilon)$, there exists a unique continuous $L$-algebra action
    \[
    C_{\varepsilon}(p^{-\beta_i} \mbb{Z}_p, L) \times \mathcal{O}_{\mathcal{U}_{\opn{HT}, k}}(\mathcal{U}_{\opn{HT}, k}) \to \mathcal{O}_{\mathcal{U}_{\opn{HT}, k}}(\mathcal{U}_{\opn{HT}, k})
    \]
    such that the action of the natural inclusion $p^{-\beta_i}\mbb{Z}_p \hookrightarrow L$ is given by $\nabla_i$.
\end{proposition}
\begin{proof}
    Note that $C_{\varepsilon}(p^{-\beta_i}\mbb{Z}_p, L)$ is isomorphic to $C_{\varepsilon}(\mbb{Z}_p, L)$ isometrically, and the function $p^{-\beta_i}\mbb{Z}_p \to \mathcal{O}_L$ given by $p^{-\beta_i}x \mapsto x$ is identified with the natural inclusion $\mbb{Z}_p \hookrightarrow \mathcal{O}_L$. Hence it is enough to show that there is a unique action
    \[
    C_{\varepsilon}(\mbb{Z}_p, L) \times \mathcal{O}_{\mathcal{U}_{\opn{HT}, k}}(\mathcal{U}_{\opn{HT}, k}) \to \mathcal{O}_{\mathcal{U}_{\opn{HT}, k}}(\mathcal{U}_{\opn{HT}, k})
    \]
    such that the natural map $\mbb{Z}_p \hookrightarrow L$ corresponds to the action of $p^{\beta_i}\nabla_i$. We first prove a similar claim for $\mathcal{O}_{\mathcal{U}_{\opn{HT}, k, \infty}}(\mathcal{U}_{\opn{HT}, k, \infty})$ by iterating Proposition \ref{LANilpToLAIterationProp}.

    Consider the Tate algebra $V \defeq A_{\opn{ord},  \infty}\langle X_{\opn{sim}}, X_{a,b, \tau} \rangle$ from Proposition \ref{StandardFormpbetaNablaiProp}(2), which is independent of $k$. We will partition the coordinates in the following way. Let
    \[
    \Sigma_1 = \{ X_{\opn{sim}}, X_{i+1, i+1, \tau_0} \} \cup \left\{ X_{a, b, \tau_0} : \begin{array}{c} 2 \leq a \leq 2n \\ b \neq 1, i+1 \end{array} \right\} \cup \left\{ X_{a, b, \tau} : \begin{array}{c} 1 \leq a, b \leq 2n \\ \tau \neq \tau_0 \end{array} \right\} 
    \]
    and set $S_1^+ \defeq A_{\opn{ord}, \infty}^+\langle \Sigma_1 \rangle$ and $S_1 = S_1^+[1/p]$. Let $D_1 = p^{\beta_i}\theta_i \colon S_1^+ \to S_1^+$ denote the derivation which acts only on the coefficients in $A_{\opn{ord}, \infty}^+$ (i.e. the coordinates in $\Sigma_1$ are all horizontal). By \S \ref{DiffOpsCcontsubsec}, the operator $D_1$ extends to a continuous (and hence locally analytic) action on $S_1$. 

    Now, for $2 \leq r \leq i$, set $S_r^+ = S_{r-1}^+\langle X_{r, i+1, \tau_0}, X_{r, 1, \tau_0} \rangle$ and $S_r = S_r^+[1/p]$. Consider the derivation $D_r \defeq T_{D_{r-1}, p^{\beta_i}}$ on $S_r^+$ as constructed in \S \ref{NilpotentOperatorsSubSec} (i.e. it acts on $S_{r-1}^+$ as $D_{r-1}$, and we have $D_r(X_{r, i+1, \tau_0}) = p^{\beta_i}X_{r, 1, \tau_0}$ and $D_r(X_{r, 1, \tau_0}) = 0$). Similarly, we set $S_{i+1}^+ = S_i^+\langle X_{1, 1, \tau_0}, X_{i+1, 1, \tau_0} \rangle$, $S_{i+1} = S_{i+1}^+[1/p]$, and $D_{i+1} \defeq T_{D_i, -p^{\beta_i}}$ acting on $S_{i+1}^+$. Finally, for $i+2 \leq r \leq 2n$, we set $S_r^+ = S_{r-1}^+\langle X_{r, i+1, \tau_0}, X_{r, 1, \tau_0} \rangle$, $S_r = S_r^+[1/p]$, and $D_r \defeq T_{D_{r-1}, p^{\beta_i}}$ acting on $S_r^+$. Note that $V = S_{2n}$.

    Now iteratively applying Proposition \ref{LANilpToLAIterationProp} to the tuples $(S_{r-1}, D_{r-1}, S_r, D_r)$ and using the fact that $D_1$ extends to a locally analytic action on $S_1$, we see that $T \defeq D_{2n}$ extends to a locally analytic action on $V$. By Lemma \ref{PertubationLemma}, there exists an integer $k(\varepsilon) \geq 1$, such that for any $k \geq k(\varepsilon)$ and continuous operator $T' \colon V \to V$ preserving the unit ball $V^+ = A_{\opn{ord}, \infty}^+\langle X_{\opn{sim}}, X_{a, b, \tau} \rangle$, the operator $T + p^k T'$ extends to an $\varepsilon$-analytic action.

    Let $k \geq k(\varepsilon)$. Then, by Proposition \ref{StandardFormpbetaNablaiProp}, we have an identification $\mathcal{O}^+_{\mathcal{U}_{\opn{HT}, k, \infty}}(\mathcal{U}_{\opn{HT}, k, \infty}) = V^+$ and the operator $p^{\beta_i}\nabla_i$ is congruent to $T$ modulo $p^k$. Since $V^+$ is $p$-torsion free, this implies that $p^{\beta_i}\nabla_i = T + p^k T'$ for some continuous operator $T' \colon V^+ \to V^+$, hence $p^{\beta_i}\nabla_i$ extends to an $\varepsilon$-analytic action on $\mathcal{O}_{\mathcal{U}_{\opn{HT}, k, \infty}}(\mathcal{U}_{\opn{HT}, k, \infty})$. 

    Finally, since $A_{\opn{ord}, k}^+ \to A_{\opn{ord}, \infty}^+$ is a pro-\'{e}tale torsor, the map $\mathcal{O}_{\mathcal{U}_{\opn{HT}, k}}(\mathcal{U}_{\opn{HT}, k}) \hookrightarrow \mathcal{O}_{\mathcal{U}_{\opn{HT}, k, \infty}}(\mathcal{U}_{\opn{HT}, k, \infty})$ is an isometry, and therefore the operator $p^{\beta_i}\nabla_i$ extends to an $\varepsilon$-analytic action on $\mathcal{O}_{\mathcal{U}_{\opn{HT}, k}}(\mathcal{U}_{\opn{HT}, k})$.
\end{proof}

\subsection{Overconvergent neighbourhoods for \texorpdfstring{$G$}{G}} \label{OCneighbourhoodsGSec}

In this section we will prove an analogous result to Proposition \ref{epsilonactionOrdProp} for overconvergent neighbourhoods of the Igusa tower inside $P_{G,\opn{dR}}^{\opn{an}}$. We will continue to use the notation introduced in the previous section. In particular, we will continue to work locally with respect to a choice of $\opn{Spa}(A, A^+) \in \mathcal{C}_G$.

\subsubsection{Overconvergent extensions}

Recall that $h$ is a fixed generator of $\hat{\delta}_{G, n+1}^+$ over $\opn{Spa}(A,A^+)$. For an integer $r \geq 1$, let $\opn{Spa}(A_r,A^+_r) = \opn{Spa}(A\langle \frac{p}{h^{p^{r+1}}} \rangle, A^+\langle \frac{p}{h^{p^{r+1}}} \rangle) \subset \mathcal{X}_{G, w_n}(p^{\beta})_r$ denote the pullback of $\opn{Spa}(A,A^+)$ under the affinoid morphism 
\[
\mathcal{X}_{G, w_n}(p^{\beta})_r \to \mathcal{X}_{G, \opn{Iw}}(p^{\beta})
\]
(see Remark \ref{RemarkOCFormalIntModelsG}). 

\begin{definition}
    Let $k \geq 1$ be an integer. We say that a quasi-compact open affinoid subspace $\mathcal{U} \subset P_{G, \opn{dR}, A}^{\opn{an}}$ is an overconvergent extension of $\mathcal{U}_{\opn{HT}, k}$ if
    \begin{itemize}
        \item One has $\mathcal{U} \cap P_{G, \opn{dR}, A_{\opn{ord}}}^{\opn{an}} = \mathcal{U}_{\opn{HT}, k}$, where the intersection is taken inside $P_{G, \opn{dR}, A}^{\opn{an}}$
        \item $\mathcal{U}$ contains the closure of $\opn{Spa}(A_{\opn{ord}, \infty}, A_{\opn{ord}, \infty}^+)$ inside $P_{G, \opn{dR}, A}^{\opn{an}}$.
    \end{itemize}
    Given an overconvergent extension $\mathcal{U}$ and an integer $r \geq 1$, we set $\mathcal{U}_r \defeq \mathcal{U} \cap P_{G, \opn{dR}, A_r}^{\opn{an}}$, which is again an overconvergent extension.
\end{definition}

The following proposition shows that overconvergent extensions actually exist.

\begin{proposition} \label{OcextensionsExistProp}
    For any integer $k \geq 1$, there exists an overconvergent extension of $\mathcal{U}_{\opn{HT}, k}$. Moreover, the collection of all overconvergent extensions (as $k$ varies) forms a cofinal system of quasi-compact open neighbourhoods of the closure of $\opn{Spa}(A_{\opn{ord}, \infty}, A_{\opn{ord}, \infty}^+)$ inside $P_{G, \opn{dR}, A}^{\opn{an}}$.
\end{proposition}
\begin{proof}
    Since $\mathcal{U}_{\opn{HT}, k}$ is the pushout of of $\mathcal{IG}_{G, w_n}(p^{\beta}) \times_{\mathcal{X}_{G, w_n}(p^{\beta})} \opn{Spa}(A_{\opn{ord}}, A^+_{\opn{ord}})$ along the natural map $M^G_{\opn{Iw}}(p^{\beta}) \hookrightarrow \overline{\mathcal{P}}^{\square}_{G, k}$, the result follows from Proposition \ref{NOctorsorsReductionProp}.
\end{proof}

We now establish a key property for overconvergent extensions as we vary the radius of overconvergence. Let $k \geq 1$ be an integer and $\mathcal{U}$ an overconvergent extension of $\mathcal{U}_{\opn{HT}, k}$. Then we have a chain of quasi-compact open affinoid neighbourhoods
\[
\mathcal{U}_1 \supset \mathcal{U}_2 \supset \cdots \supset \mathcal{U}_{\infty} \defeq \mathcal{U}_{\opn{HT}, k}
\]
such that $\mathcal{U}_{\infty}$ is the locus inside $\mathcal{U}_r$ where $|h| = 1$. On sections, this chain of neighbourhoods is induced from the chain of continuous maps
\[
B_{1}^+ \to B_2^+ \to \cdots \to B_{\infty}^+
\]
where $B_r^+ = \mathcal{O}^+(\mathcal{U}_r)$. Let $B_r = \mathcal{O}(\mathcal{U}_r) = B_r^+[1/p]$, and let $|\!| \cdot |\!|_r$ denote the $L$-Banach norm on $B_r$ for which $B_r^+$ is the unit ball. 

\begin{lemma} \label{OCclosetobeingisomLemma}
    Let $r \geq 1$. For any real number $0 < \delta < 1$, there exists an integer $s= s(\delta) \geq r$ such that: for any $v \in B_r$, $m \in \mbb{N}$ and $c \in \mbb{Q}$, one has
    \[
    |\!| v |\!|_{\infty} \leq p^{c-m} \text{ and } |\!|v|\!|_r \leq p^c \quad \Rightarrow \quad |\!| v |\!|_s \leq p^{c- \delta m} .
    \]
\end{lemma}
\begin{proof}
    Since $\mathcal{U}_r$ is an overconvergent extension, we have $B^+_\infty = B_r^+\langle 1/h \rangle$, hence $\left( B^+_{\infty}/p \right) = \left( B_r^+/p \right)[1/h]$. Therefore, any element in the kernel of the map $B_r^+/p \to B_{\infty}^+/p$ is killed by some power of $h$. Since $B_r^+/p$ is Noetherian, we see that there is a common power $h^M$ which kills any element in the kernel of the map $B_r^+/p \to B_{\infty}^+/p$. Since $B_{\infty}^+$ is $p$-torsion free, this implies (by a simple induction argument) that $h^{m M}$ kills the kernel of $B_r^+/p^m \to B^+_{\infty}/p^m$ for any integer $m \geq 1$.

    Now, by raising $v$ to an integral power, it is enough to prove the claim for $c \in \mbb{Z}$, and by rescaling, it is enough to prove the claim when $c = 0$. Therefore, we suppose we have an element $v \in B_r^+$ whose image lies in $p^m B_{\infty}^+$. By the paragraph above, we must therefore have
    \[
    |\!| (p^{-1} h^M)^m v |\!|_s \leq |\!| (p^{-1} h^M)^m v |\!|_r \leq 1
    \]
    for any $s \geq r$ and $m \geq 1$. Since $|\!|h^{-1} |\!|_s \to 1$ as $s \to +\infty$, taking $s$ large enough such that $|\!|h^{-1}|\!|_s \leq p^{\frac{(1-\delta)}{M}}$, we obtain
    \[
    |\!| v |\!|_s \leq |\!| (p h^{-M})^m |\!|_s \leq p^{-\delta m}
    \]
    as required.
\end{proof}

\subsubsection{The Gauss--Manin connection (overconvergent case)} \label{TheGMConnectionOCcaseSubSec}

We now prove an overconvergent analogue of Proposition \ref{epsilonactionOrdProp}. 

\begin{proposition} \label{epactionOCversionProp}
    Let $\varepsilon > 0$. For any quasi-compact open neighbourhood $U$ of the closure of $\opn{Spa}(A_{\opn{ord}, \infty}, A_{\opn{ord}, \infty}^+)$ inside $P_{G, \opn{dR}, A}^{\opn{an}}$, there exists a quasi-compact open neighbourhood $V$ of the closure of $\opn{Spa}(A_{\opn{ord}, \infty}, A_{\opn{ord}, \infty}^+)$ inside $P_{G, \opn{dR}, A}^{\opn{an}}$ and a unique continuous $L$-linear action
    \begin{equation} \label{epactionOCversionEqn}
    C_{\varepsilon}(U_{G, \beta}, L) \times \mathcal{O}_{P_{G, \opn{dR}, A}^{\opn{an}}}(U) \to \mathcal{O}_{P_{G, \opn{dR}, A}^{\opn{an}}}(V)
    \end{equation}
    extending the action of polynomial functions in $C_{\varepsilon}(U_{G, \beta}, L)$ induced from the operators $\{ \nabla_i : i=1, \dots, 2n-1 \}$. These actions are compatible as one varies $\varepsilon$, $U$, and $V$.
\end{proposition}
\begin{proof}
    Let $i \in \{1, \dots, 2n-1\}$ and let $k \geq 1$ be an integer such that $k \geq k(\varepsilon/2)$, where $k(\varepsilon/2)$ is the integer in Proposition \ref{epsilonactionOrdProp} (depending on the integer $i$). Then, increasing $k$ if necessary, there exists an overconvergent extension $\mathcal{U}$ of $\mathcal{U}_{\opn{HT}, k}$ with $\mathcal{U} \subset U$ because overconvergent extensions are cofinal. Without loss of generality, we may assume $\mathcal{U} = \mathcal{U}_1$. Then, we consider the chain of $L$-Banach spaces
    \[
    \mathcal{O}_{P_{G, \opn{dR}, A}^{\opn{an}}}(\mathcal{U}_1) \to \cdots \to \mathcal{O}_{P_{G, \opn{dR}, A}^{\opn{an}}}(\mathcal{U}_{\infty})
    \]
    where $\mathcal{U}_{\infty} = \mathcal{U}_{\opn{HT}, k}$. By Proposition \ref{epsilonactionOrdProp}, the operator $p^{\beta_i}\nabla_i$ extends to an $\varepsilon/2$-analytic action on $\mathcal{O}_{P_{G, \opn{dR}, A}^{\opn{an}}}(\mathcal{U}_{\infty})$. Combining this with Lemma \ref{OCclosetobeingisomLemma}, we see that the hypotheses of Proposition \ref{CloseToIsomGeneralProp} are satisfied, hence there exists an integer $r \geq 1$ such that $p^{\beta_i}\nabla_i$ extends to an $\varepsilon$-analytic action on $\mathcal{O}_{P_{G, \opn{dR}, A}^{\opn{an}}}(\mathcal{U}_{r})$. In other words, if we set $V_i =  \mathcal{U}_r$, then there exists a continuous $L$-linear action
    \[
    C_{\varepsilon}(p^{-\beta_i}\mbb{Z}_p, L) \times \mathcal{O}_{P_{G, \opn{dR}, A}^{\opn{an}}}(U) \to \mathcal{O}_{P_{G, \opn{dR}, A}^{\opn{an}}}(V_i)
    \]
    extending the action of polynomial functions in $p^{\beta_i} \nabla_i$. This action is unique because it can be computed on Mahler expansions.

    Now take $V = \cap_{i=1}^{2n-1} V_i$. Using the fact that the operators $\nabla_i$ commute with each other and the fact that
    \[
    C_{\varepsilon}(p^{-\beta_1} \mbb{Z}_p, L) \otimes_L \cdots \otimes_L C_{\varepsilon}(p^{-\beta_{2n-1}} \mbb{Z}_p, L) \subset C_{\varepsilon}(U_{G, \beta}, L)
    \]
    is dense, we see that there is a unique action as in (\ref{epactionOCversionEqn}) extending the action of polynomial functions. This action is functorial in $\varepsilon$, $U$ and $V$ by unicity.
\end{proof}

We now prove the first part of Theorem \ref{MainTheoremOnPadicIteration}. Let $U = \opn{Spa}(A, A^+) \in \mathcal{C}_G$. Then, from the functoriality properties in Proposition \ref{epactionOCversionProp}, we obtain a continuous $L$-algebra action
\[
C_{\varepsilon}(U_{G, \beta}, L) \times \mathscr{N}^{\dagger}_G(U) \to \mathscr{N}^{\dagger}_G(U)
\]
extending the action of polynomial functions. This action is functorial in $\varepsilon$. We claim that this action is unique. Indeed, by Proposition \ref{NOctorsorsReductionProp}, there exist integers $r_1 < r_2 < \cdots $ and \'{e}tale $\overline{\mathcal{P}}^{\square}_{G, k}$-torsors $\mathcal{U}_{r_k, k} \to \opn{Spa}(A_{r_k}, A_{r_k}^+)$ with $\mathcal{U}_{r_k, k}$ an overconvergent extension of $\mathcal{U}_{\opn{HT}, k}$, such that 
\[
\mathcal{U}_{r_1, 1} \supset \mathcal{U}_{r_2, 2} \supset \cdots \supset \mathcal{U}_{r_k, k} \supset \cdots \supset \opn{Spa}(A_{\opn{ord}, \infty}, A_{\opn{ord}, \infty}^+) .
\]
Furthermore, the collection $\{ \mathcal{U}_{r_k, k} \}_{k \geq 1}$ is a cofinal system of quasi-compact open neighbourhoods of $P_{G, \opn{dR}, A}^{\opn{an}}$ containing the closure of $\opn{Spa}(A_{\opn{ord}, \infty}, A_{\opn{ord}, \infty}^+)$, and all of these torsors are necessarily reductions of structure of each other. This implies that
\begin{equation} \label{InjectiveDirectSystemForNdagger}
\mathscr{N}^{\dagger}_G(U) \cong \left( \mathcal{O}_{P_{G, \opn{dR}, A}^{\opn{an}}}(\mathcal{U}_{r_k, k}) \right)_{k \geq 1} .
\end{equation}

Suppose that $\star_1$ and $\star_2$ are two continuous actions of $C_{\varepsilon}(U_{G, \beta}, L)$ on $\mathscr{N}^{\dagger}_G(U)$ extending the action of polynomial functions induced from the operators $\{ \nabla_i : i=1, \dots, 2n-1 \}$. Let $f \in C_{\varepsilon}(U_{G, \beta}, L)$ which we can write as a limit $f = \lim_m f_m$ of polynomial functions, and let $x \in \mathcal{O}_{P_{G, \opn{dR}, A}^{\opn{an}}}(\mathcal{U}_{r_k, k})$ for some $k \geq 1$. By assumption and our definition of an action on ind-sheaves (see \S \ref{NotationsAndConventionsIntro}), there exists an integer $k' \geq 1$ and continuous maps
\begin{equation} \label{CepToNdaggerEqn}
C_{\varepsilon}(U_{G, \beta}, L) \to \mathcal{O}_{P_{G, \opn{dR}, A}^{\opn{an}}}(\mathcal{U}_{r_{k'}, k'}), \quad g \mapsto g \star_j x \quad (j=1, 2) .
\end{equation}
Since $\mathcal{O}_{P_{G, \opn{dR}, A}^{\opn{an}}}(\mathcal{U}_{r_{k'}, k'})$ is Hausdorff and both actions extend the same action of polynomial functions, we find that
\[
f \star_1 x = \lim_m f_m \star_1 x = \lim_m f_m \star_2 x = f \star_2 x
\]
with the limit taking place in $\mathcal{O}_{P_{G, \opn{dR}, A}^{\opn{an}}}(\mathcal{U}_{r_{k'}, k'})$. This implies that the actions $\star_1$ and $\star_2$ coincide (as morphisms of ind-sheaves), thereby proving uniqueness.

To conclude the proof of Theorem \ref{MainTheoremOnPadicIteration}(1), we must show that the action is functorial in $U \in \mathcal{C}_G$ and equivariant for the action of $M^G_{\opn{Iw}}(p^{\beta})$. Let $U_1, U_2 \in \mathcal{C}_G$ with $U_1 \supset U_2$, and consider the natural (continuous) restriction map $\pi \colon \mathscr{N}^{\dagger}_G(U_1) \to \mathscr{N}^{\dagger}_G(U_2)$. Let $x \in \mathscr{N}^{\dagger}_G(U_1)$ and consider the following two continuous maps $C_{\varepsilon}(U_{G, \beta}, L) \to \mathscr{N}^{\dagger}_G(U_2)$ given by
\begin{align}
    f &\mapsto \pi ( f \star_{U_1} x ) \label{ftopifU1} \\
    f &\mapsto f \star_{U_2} \pi(x) \label{ftofU_2pi}
\end{align}
where $\star_{U_j}$ denotes the action of $C_{\varepsilon}(U_{G, \beta}, L)$ on $\mathscr{N}^{\dagger}_G(U_j)$. By the same argument above, we can write $\mathscr{N}^{\dagger}_G(U_2) \cong (V_k)_{k \geq 1}$ as a filtered inductive system of Banach spaces and both (\ref{ftopifU1}) and (\ref{ftofU_2pi}) factor through some $V_k$ (by the definition of an action on ind-sheaves). Write $f = \lim_m f_m$ as a limit of polynomial functions. Since we know that the action of the operators $\{\nabla_i \}$ (and hence the action of polynomial functions) is already functorial, we have
\[
\pi( f \star_{U_1} x ) = \lim_m \pi(f_m \star_{U_1} x) = \lim_m (f_m \star_{U_2} \pi(x) ) = f \star_{U_2} \pi(x)
\]
with the limit taking place in $V_k$. This implies that (\ref{ftopifU1}) $=$ (\ref{ftofU_2pi}). The proof for the equivariance under the action of $M^G_{\opn{Iw}}(p^{\beta})$ is very similar, and can again be deduced from the density of polynomial functions.  

\subsection{The construction for \texorpdfstring{$H$}{H}} \label{TheConstuctionForHSec}

In this section we finish the proof of Theorem \ref{MainTheoremOnPadicIteration}. Since the construction of the action in Theorem \ref{MainTheoremOnPadicIteration}(2) follows exactly the same strategy as in \S \ref{OrdExpStrictNhoodGSec}--\ref{OCneighbourhoodsGSec}, we do not repeat the construction in the same amount of detail. Instead, we simply highlight the differences between the constructions for $G$ and $H$ and leave the details to the interested reader.

To define the analogue of the torsors $\mathcal{U}_{\opn{HT}, k}$, we consider quotients of the Igusa tower $\ide{IG}_{H, \opn{id}}(p^{\beta})$ by the subgroup of $M^H_{\diamondsuit}(p^{\beta})$ consisting of elements which are congruent to the identity modulo $p^k$. Over these quotients of the Igusa tower, we again have a canonical basis of $\mathcal{H}_{\mathcal{A}_1}/p^k \oplus \mathcal{H}_{\mathcal{A}_2}/p^k$ arising from the universal trivialisations, and we can consider the torsor parameterising bases of $\mathcal{H}_{\mathcal{A}_1} \oplus \mathcal{H}_{\mathcal{A}_2}$ respecting the symplectic structure and Hodge filtration which reduce to the canonical basis modulo $p^k$. Working locally, one can show via the same method as in Proposition \ref{OcextensionsExistProp} that these torsors overconverge (and hence we can define overconvergent extensions in this setting). To define the system of overconvergent neighbourhoods of $\mathcal{X}_{H, \opn{id}}(p^{\beta})$ inside $\mathcal{X}_{H, \diamondsuit}(p^{\beta})$ locally, we consider the exact same congruences but now involving a generator of $\hat{\delta}^+_{H, n+1}$ (see Definition \ref{DeltaHneighbourhoodsDef}).

We now proceed in exactly the same way for the construction the action for $H$. Indeed, one can choose coordinates of $\mathcal{U}_{\opn{HT}, k}$ base-changed to $\mathcal{IG}_{H, \opn{id}}(p^{\beta})$ so that the action of $p^{\beta_i} \nabla_i$ ($i=1, \dots, n-1$) is congruent modulo $p^k$ to a simple operator on a Tate algebra (which is an iterated version of the construction in \S \ref{NilpotentOperatorsSubSec}). One then applies the general results in \S \ref{ContinuousOpsOnBanachSpacesSec} in the same way as in \S \ref{OrdExpStrictNhoodGSec}--\ref{OCneighbourhoodsGSec}, and this results in the construction of the action in Theorem \ref{MainTheoremOnPadicIteration}(2). The uniqueness and functoriality results follow the same proof as at the end of \S \ref{TheGMConnectionOCcaseSubSec}.

To prove Theorem \ref{MainTheoremOnPadicIteration}(3), we again use the same strategy as at the end of \S \ref{TheGMConnectionOCcaseSubSec}. More precisely, we already know from \S \ref{FunctorialityOfGMConnectionOnNOCSubSec} that the pullback map is equivariant on polynomial functions, and we use the fact that $\mathscr{N}^{\dagger}_H(\hat{\iota}^{-1}U)$ is a filtered inductive system of Banach spaces. Hence equivariance on polynomial functions is enough to deduce equivariance for the $\varepsilon$-analytic actions.

\subsection{\texorpdfstring{$p$}{p}-adic evaluation maps}

We now construct the $p$-adic analogues of the evaluation maps in \S \ref{ClassicalEvaluationMapsSubSec} following the general construction in \S \ref{TheMainConstructionSubSec}. 

\subsubsection{Compatibility with support conditions} \label{CompatWithSupportConditionsSubSec}

Fix $U = \opn{Spa}(A, A^+) \in \mathcal{C}_G$ and let $V \defeq U \cap (\mathcal{X}_{G, \opn{Iw}}(p^{\beta}) - \mathcal{Z}_{G, > n+1}(p^{\beta}))$. Before constructing the evaluation maps, one must first show that the structures on $\mathscr{N}^{\dagger}_G(U)$ extend over $V$. Recall the definitions of $\mathscr{M}^{(r, k)}_{\bullet}$ and $\mathscr{N}^{(r, k)}_{\bullet}$ from Lemma \ref{OctorsorsReductionLemma} and Proposition \ref{NOctorsorsReductionProp} respectively, and also their versions with a fixed weight from Definition \ref{DefinitionOfNOCwithweight}. 

Note that $\mathscr{M}^{(r,k)}_G(U) \subset \mathscr{N}^{(r,k)}_G(U)$ (resp. $\mathscr{M}^{(r,k)}_G(V) \subset \mathscr{N}^{(r,k)}_G(V)$) is the subspace killed by the action of $\overline{\ide{u}}_G$ (resp. $\overline{\ide{u}}_H$) under $\star_{\overline{P}}$. Furthermore, $\mathscr{N}^{(r,k)}_G(U)$ and $\mathscr{M}^{(r,k)}_G(U)$ are Banach $L$-algebras, and $\mathscr{N}^{(r,k)}_G(V)$ and $\mathscr{M}^{(r,k)}_G(V)$ are Fr\'{e}chet $L$-algebras. Indeed, let $\delta$ be a fixed generator of $\delta_{G, >n+1}^+$ over $U$, and let $\{ \gamma_i \}_{i \in \mbb{N}}$ be a monotonically decreasing sequence of rational numbers $\gamma_i > 0$ such that $\gamma_i \to 0$ as $i \to +\infty$. Let $V_{\gamma_i} \subset V \subset U$ denote the open affinoid subspace where $|\delta| \leq |p|^{\gamma_i}$. Then $\{ V_{\gamma_i} \}$ is a cover of $V$ with compact restriction maps with dense image. We have
\[
\mathscr{N}^{(r,k)}_G(V) = \varprojlim_i \mathscr{N}^{(r,k)}_G(V_{\gamma_i}), \quad \quad \mathscr{M}^{(r,k)}_G(V) = \varprojlim_i \mathscr{M}^{(r,k)}_G(V_{\gamma_i})
\]
where each term in the inverse limit is a Banach $L$-algebra and the transition maps are compact and have dense image.

\begin{lemma} \label{LaActionExtendsToVLemma}
    Let $U \in \mathcal{C}_G$ and $V \defeq U \cap (\mathcal{X}_{G, \opn{Iw}}(p^{\beta}) - \mathcal{Z}_{G, > n+1}(p^{\beta}))$. Let $\varepsilon > 0$. Then the action in Theorem \ref{MainTheoremOnPadicIteration}(1) extends uniquely to a continuous $L$-algebra action
    \[
    C_{\varepsilon}(U_{G, \beta}, L) \times \mathscr{N}^{\dagger}_{G}(V) \to \mathscr{N}^{\dagger}_{G}(V)
    \]
    which is functorial in $U$, $\varepsilon$, and is $M^G_{\opn{Iw}}(p^{\beta})$-equivariant.
\end{lemma}
\begin{proof}
    Without loss of generality, we may assume that the rational numbers above satisfy $\gamma_i = 1/i$. Then we have
    \[
    \mathscr{N}^{(r, k)}_{G}(V_{\gamma_i}) = \mathscr{N}^{(r, k)}_G(U) \langle \frac{\delta^i}{p} \rangle
    \]
    for all $k \geq 1$. Let $\varepsilon > 0$. Let $\nabla \in \{ p^{\beta_1}\nabla_1, \dots, p^{\beta_{2n-1}}\nabla_{2n-1} \}$ be one of the operators appearing at the start of \S \ref{TheGMConnectionOnNdaggerSubSec}, normalised by the power of $p$ defined in Proposition \ref{IntertwiningOfNablaThetaiProp}. Then there exists $k' \geq k$ and $r' \geq r$ such that
    \begin{itemize}
        \item We have a continuous $L$-linear action
        \[
        C_{\varepsilon/2}(U_{G, \beta}, L) \times \mathscr{N}^{(r, k)}_G(U) \to \mathscr{N}^{(r', k')}_G(U)
        \]
        extending the action of polynomial functions. In particular, there exists a constant $C \in \mbb{R}_{>0}$ such that for any $f \in C^{\opn{pol}}(U_{G, \beta}, \mathcal{O}_L)$ of degree $\leq a$ and $y \in \mathscr{N}^{(r, k)}_G(U)$, we have $p^{-a\varepsilon/2}|\!|f \star y |\!|_{k'} \leq C |\!| y |\!|_{k}$, where $|\!| \cdot |\!|_k$ (resp. $|\!| \cdot |\!|_{k'}$) denotes the Banach norm on $\mathscr{N}^{(r, k)}_G(U)$ (resp. $\mathscr{N}^{(r', k')}_G(U)$). 
        \item $|\!| \nabla(\delta) \delta^{-1} |\!|_{k'} \leq p^{\varepsilon/4}$ for any choice of $\nabla$. Indeed, $\nabla$ is integral over the ordinary locus (see Proposition \ref{StandardFormpbetaNablaiProp}), so we can increase $k'$ if necessary so that this is satisfied. 
    \end{itemize}
    Let $x = \sum_{j=0}^{\infty} a_j \left( \frac{\delta^i}{p} \right)^j \in \mathscr{N}^{(r, k)}_{G}(V_{\gamma_i})$, with $a_j \in \mathscr{N}^{(r, k)}_G(U)$ converging to zero. Then, by a similar argument as in the proof of Proposition \ref{LANilpToLAIterationProp}, we have (for $a \geq 1$)
    \[
    \bincoeff{\nabla}{a} ( a_j \left( \frac{\delta^i}{p} \right)^j ) = \sum_{b=0}^{\opn{min}(a, j)}\sum \left( \begin{array}{c} \text{ multinomial } \\ \text{ coefficients } \end{array} \right) \cdot \left( \begin{array}{c} \text{ element of } \\ C^{\opn{pol}}(U_{G, \beta}, \mathcal{O}_L) \text{ of } \\ \text{ degree } \leq a \end{array} \right)(a_j) \cdot \nabla\left( \frac{\delta^i}{p} \right)^b \cdot \left( \frac{\delta^i}{p} \right)^{j-b} 
    \]
    where the unlabelled sum is indexed by subsets of $\{0, \dots, a-1\}$ of size $a-b$ (c.f. formula (\ref{ExplicitFormulaForfk})). Let $|\!| \cdot |\!|_{k', i}$ denote the Banach norm on $\mathscr{N}^{(r', k')}_{G}(V_{\gamma_i})$. Then, following the proof of Proposition \ref{LANilpToLAIterationProp}, we see that
    \begin{eqnarray}
    \left|\! \left| \bincoeff{\nabla}{a} ( a_j \left( \frac{\delta^i}{p} \right)^j ) \right| \! \right|_{k', i} & \leq & \opn{max}\{ p^{2\opn{log}_p(a)} |\!| f \star a_j |\!|_{k'} |\!| \nabla\left( \frac{\delta^i}{p} \right) |\!|_{k'}^b \} \label{BoundOnDeltaa1} \\ 
    & \leq & \opn{max}\{ p^{2\opn{log}_p(a)} |\!| f \star a_j |\!|_{k'} |\!|  \nabla(\delta)\delta^{-1}  |\!|_{k'}^b \} \label{BoundOnDeltaa2} \\
    & \leq & p^{2 \opn{log}_p(a)}p^{a \varepsilon/2} C |\!| a_j |\!|_k p^{a \varepsilon/4} \label{BoundOnDeltaa3}
    \end{eqnarray}
    where the maxima in (\ref{BoundOnDeltaa1}) and (\ref{BoundOnDeltaa2}) are over $0 \leq b \leq a$ and polynomials $f \in C^{\opn{pol}}(U_{G, \beta}, \mathcal{O}_L)$ of degree $\leq a$. For the inequality in (\ref{BoundOnDeltaa2}) we have used the fact that $\nabla(\delta^i/p) = i \nabla(\delta) \delta^{-1} (\delta^i/p)$, and for the inequality in (\ref{BoundOnDeltaa3}), we have used the fact that $|\!| f \star a_j |\!|_{k'} \leq p^{a\varepsilon/2} C |\!| a_j |\!|_k$ (by the first bullet point above) and $|\!|  \nabla(\delta)\delta^{-1}  |\!|_{k'}^b \leq p^{b \varepsilon/4} \leq p^{a \varepsilon/4}$ (by the second bullet point above). This implies that
    \[
    p^{-a\varepsilon} \left|\! \left| \bincoeff{\nabla}{a} (x) \right|\! \right|_{k', i} \leq p^{-a \varepsilon/4 + 2\opn{log}_p(a)} C \opn{max}_j |\!| a_j |\!|_k
    \]
    and hence the left-hand side converges to zero as $a \to +\infty$. This implies that the action $C_{\varepsilon}(U_{G, \beta}, L) \times \mathscr{N}^{(r, k)}_G(U) \to \mathscr{N}^{(r', k')}_G(U)$ extends uniquely to an action
    \[
    C_{\varepsilon}(U_{G, \beta}, L) \times \mathscr{N}^{(r, k)}_G(V_{\gamma_i}) \to \mathscr{N}^{(r', k')}_G(V_{\gamma_i})
    \]
    for any $\gamma_i$, and hence induces an action
    \[
    C_{\varepsilon}(U_{G, \beta}, L) \times \varprojlim_i \mathscr{N}^{(r, k)}_G(V_{\gamma_i}) \to \varprojlim_i \mathscr{N}^{(r', k')}_G(V_{\gamma_i})
    \]
    as required. The rest of the lemma follows.
\end{proof}

\begin{remark}
    By a similar argument as in Lemma \ref{LaActionExtendsToVLemma} above, one can check that Theorem \ref{MainTheoremOnPadicIteration}(2) extends to $\mathscr{N}^{\dagger}_H(\hat{\iota}^{-1}V)$, and we also have an analogue of Theorem \ref{MainTheoremOnPadicIteration}(3).
\end{remark}

We will also need the following important result.

\begin{lemma} \label{IndicatorFunctionsPreserveOCLemma}
    Fix $U \in \mathcal{C}_G$ and set $V \defeq U \cap (\mathcal{X}_{G, \opn{Iw}}(p^{\beta}) - \mathcal{Z}_{G, > n+1}(p^{\beta}))$. Let
    \[
    \mathcal{M}^{\dagger}_G \in \{ \mathscr{M}^{\dagger}_G(U), \mathscr{M}^{\dagger}_G(V) \} .
    \]
    Let $f \in C^{\opn{la}}(U_{G, \beta}, L)$ be any locally constant function. Then for any $x \in \mathcal{M}^{\dagger}_G$, one has $f \star x \in \mathcal{M}^{\dagger}_G$, i.e. the action of the locally constant function $f$ takes overconvergent forms to overconvergent forms. We also have an analogous statement for overconvergent forms for $H$.
\end{lemma}
\begin{proof}
    We first prove the claim for $\mathcal{M}^{\dagger}_G = \mathscr{M}^{\dagger}_G(U)$. Fix an integer $i=1, \dots, 2n-1$ and let $\xi \in \mbb{Z}_p$. Let $f_{i, \xi}^{(s)} \in C^{\opn{la}}(U_{G, \beta}, L)$ denote the indicator function of the following subset:
    \[
    p^{-\beta_1}\mbb{Z}_p \times \cdots \times p^{-\beta_{i-1}}\mbb{Z}_p \times p^{-\beta_i} \left( \xi + p^s\mbb{Z}_p \right) \times p^{-\beta_{i+1}}\mbb{Z}_p \times \cdots \times p^{-\beta_{2n-1}} \mbb{Z}_p \subset U_{G, \beta}.
    \]
    where $s \geq 0$ is an integer. It is enough to show that these indicator functions preserve overconvergent forms. We will prove this by induction on the radius of analyticity $s \geq 0$. Note that the claim for $s=0$ is immediate, because $f^{(0)}_{i, \xi} \equiv 1$ for any $i$, $\xi$.
    
    Let $s \geq 1$ be an integer. Let $\varepsilon > 0$ be any real number such that $f_{i, \xi}^{(s')} \in C_{\varepsilon}(U_{G, \beta}, L)$ for any choice of $i$, $\xi$ and $0 \leq s' \leq s$. With notation as above, let $k$ be a sufficiently large integer such that we have an action
    \[
    C_{\varepsilon}(U_{G, \beta}, L) \times \mathscr{N}^{(r, k)}_G(U) \to \mathscr{N}^{(r', k)}_G(U)
    \]
    for some integer $r' \geq r$, and such that we have an action of $C_{\varepsilon}(U_{G, \beta}, L)$ on $\mathcal{O}(\mathcal{U}_{\opn{HT}, k})$ where $\mathcal{U}_{\opn{HT}, k}$ is defined in \S \ref{OrdExpStrictNhoodGSec}. Since the natural restriction map
    \[
    \mathscr{N}^{(r', k)}_G(U) \to \mathcal{O}(\mathcal{U}_{\opn{HT}, k})
    \]
    is injective and equivariant for the action of $\overline{\ide{u}}_G$, it suffices to show for any $x \in \mathcal{O}(\mathcal{U}_{\opn{HT}, k})$ killed by $\overline{\ide{u}}_G$, the element $f_{i, \xi}^{(s)} \star x$ is also killed by $\overline{\ide{u}}_G$. Furthermore, since the natural map
    \[
    \mathcal{O}(\mathcal{U}_{\opn{HT}, k}) \to \mathcal{O}(\mathcal{U}_{\opn{HT}, k, \infty} ) 
    \]
    is also injective and equivariant for the action of $\overline{\ide{u}}_G$, it suffices to prove the same statement for elements of $\mathcal{O}(\mathcal{U}_{\opn{HT}, k, \infty})$. 
    
    Recall from Proposition \ref{StandardFormpbetaNablaiProp} that we have an identification 
    \[
    \mathcal{O}^+(\mathcal{U}_{\opn{HT}, k, \infty}) = A^+_{\opn{ord}, \infty}\langle X_{\opn{sim}}, X_{a, b, \tau} \rangle 
    \]
    and an operator $T \defeq p^{\beta_i}\nabla_i$ on this space. For the rest of this proof, we freely use notation from the proof of Proposition \ref{StandardFormpbetaNablaiProp}. 
    
    Let $\mathcal{I} \subset \mbb{Z}_p \langle X_{\opn{sim}}, X_{a, b, \tau} \rangle$ denote the ideal generated by the coordinates $X_{\opn{sim}}, X_{a, b, \tau}$ (as $a, b, \tau$ vary). Then note that $Y_{1, 1, \tau_0}^{-1} Y_{i+1, i+1, \tau_0} \in 1 + \mathcal{I} \subset \mbb{Z}_p \langle X_{\opn{sim}}, X_{a, b, \tau} \rangle$. Using the description of the operator in (\ref{ExplcitEqnforNablaiOrd}) and explicitly calculating $\star_l$, we see that:
    \begin{itemize}
        \item $T(X_{\opn{sim}}) = 0$ and $T(X_{a, b, \tau}) = 0$ for all $(a, b, \tau)$ with $\tau \neq \tau_0$.
        \item $T(X_{1, 1, \tau_0}) = -p^{\beta_i}X_{i+1, 1, \tau_0} Y_{1, 1, \tau_0}^{-1} Y_{i+1, i+1, \tau_0} \in -p^{\beta_i}X_{i+1, 1, \tau_0} + \mathcal{I}^2$.
        \item $T(X_{a, 1, \tau_0}) = 0$ for all $2 \leq a \leq 2n$
        \item For $2 \leq a, b \leq 2n$ with $b \neq i+1$, we have
        \[
        T(X_{a, b, \tau_0}) = p^{\beta_i + k} X_{a, 1, \tau_0} (X_{1, 1, \tau_0}')^{-1} X_{i+1, b, \tau_0} Y_{1, 1, \tau_0}^{-1} Y_{i+1, i+1, \tau_0} \in \mathcal{I}^2 .
        \]
        \item For $2 \leq a \leq 2n$ with $a \neq i+1$, we have
        \[
        T(X_{a, i+1, \tau_0}) = p^{\beta_i} X_{a, 1, \tau_0} (X'_{1, 1, \tau_0})^{-1} X_{i+1, i+1}' Y_{1, 1, \tau_0}^{-1} Y_{i+1, i+1, \tau_0} \in p^{\beta_i}X_{a, 1, \tau_0} + \mathcal{I}^2 .
        \]
        \item $T(X_{i+1, i+1, \tau_0}) = p^{\beta_i} X_{i+1, 1, \tau_0} ((X_{1, 1, \tau_0}')^{-1} X_{i+1, i+1, \tau_0}' - 1) Y_{1, 1, \tau_0}^{-1} Y_{i+1, i+1, \tau_0} \in \mathcal{I}^2$. 
    \end{itemize}
    This implies that $T^2(\mathcal{I}) \subset \mathcal{I}^2$. By continually applying the Leibniz rule, we see that for any integer $h \geq 1$, one has
    \[
    T^{h+1}(\mathcal{I}^h) \subset \mathcal{I}^{h+1} .
    \]
    Let $\tilde{\star}$ denote the action of $C_{\opn{cont}}(U_{G, \beta}, \mathcal{O}_L)$ on $A_{\opn{ord}, \infty}^+\langle X_{\opn{sim}}, X_{a, b, \tau} \rangle$ given by simply acting on the coefficients (via the action in \S \ref{DiffOpsCcontsubsec}). Let $1 \leq s' \leq s$. Suppose that the action of $f^{(s'-1)}_{i, \xi}$ on $A_{\opn{ord}, \infty}\langle X_{\opn{sim}}, X_{a, b, \tau} \rangle$ under $\star$ is equal to the action of $f^{(s'-1)}_{i, \xi}$ under $\tilde{\star}$. Let $\theta \colon A_{\opn{ord}, \infty}^+\langle X_{\opn{sim}}, X_{a, b, \tau} \rangle \to A_{\opn{ord}, \infty}^+\langle X_{\opn{sim}}, X_{a, b, \tau} \rangle$ denote the derivation which acts as $p^{\beta_i}\theta_i$ on $A_{\opn{ord}, \infty}^+$ and $\theta(X_{\opn{sim}}) = \theta(X_{a, b, \tau}) = 0$. Then:
    \begin{itemize}
        \item For any $x \in A_{\opn{ord}, \infty}\langle X_{\opn{sim}}, X_{a, b, \tau} \rangle$
        \[
        p^{-(s'-1)p^{l-1}(p-1)} (T-\xi)^{p^{l-1}(p-1)} ( f^{(s'-1)}_{i, \xi} \star x )
        \]
        converges to $f^{(s'-1)}_{i, \xi} \star x - f^{(s')}_{i, \xi} \star x$ as $l \to +\infty$.
        \item For any $x \in A_{\opn{ord}, \infty}^+\langle X_{\opn{sim}}, X_{a, b, \tau} \rangle$
        \[
        p^{-(s'-1)p^{l-1}(p-1)} (\theta-\xi)^{p^{l-1}(p-1)} ( f^{(s'-1)}_{i, \xi} \star x )
        \]
        converges to $f^{(s'-1)}_{i, \xi} \star x - f^{(s')}_{i, \xi} \tilde{\star} x$ as $l \to +\infty$. In particular 
        \[
        (\theta-\xi)^{M} ( f^{(s'-1)}_{i, \xi} \star x ) \in p^{(s'-1)M}  A_{\opn{ord}, \infty}^+\langle X_{\opn{sim}}, X_{a, b, \tau} \rangle
        \]
        for any $M \geq 1$.
        \item The operators $T$ and $\theta$ commute, which implies that the idempotent operators $F \defeq f^{(s')}_{i, \xi} \star -$ and $\tilde{F} \defeq f^{(s')}_{i, \xi} \tilde{\star} -$ commute. This implies that, for any odd integer $w \geq 1$, one has
        \begin{equation} \label{FFtildeoddidempotent}
        (F - \tilde{F})^w = F - \tilde{F} .
        \end{equation}
    \end{itemize}

    Let $a \in A_{\opn{ord}, \infty}^+$ and $y \in \mathcal{I}^h$ (for some $h \geq 0$, with the convention that $\mathcal{I}^0 = \mbb{Z}_p \langle X_{\opn{sim}}, X_{a, b, \tau} \rangle$). Then we have
    \[
    p^{-(s'-1)M}(T-\xi)^M(f^{(s'-1)}_{i, \xi} \star a y) \equiv \sum_{j=0}^{h} \binom{M}{j} \left[ p^{-(s'-1)M}(\theta - \xi)^{M-j}(f^{(s'-1)}_{i, \xi} \star a) \right] T^j(y) 
    \]
    modulo $\mathcal{I}^{h+1} A_{\opn{ord}, \infty}\langle X_{\opn{sim}}, X_{a, b, \tau} \rangle$, for any $M \geq h$. Here we have used that fact that $Y_{1, 1, \tau_0}^{-1} Y_{i+1, i+1, \tau_0} \in 1 + \mathcal{I}$, $T^{h+1}(\mathcal{I}^h) \subset \mathcal{I}^{h+1}$, and $T(\mathcal{I}^{h'}A_{\opn{ord}, \infty}^+\langle X_{\opn{sim}}, X_{a, b, \tau} \rangle) \subset \mathcal{I}^{h'} A_{\opn{ord}, \infty}^+\langle X_{\opn{sim}}, X_{a, b, \tau} \rangle$ for any integer $h' \geq 0$.

    For any $1 \leq j \leq h$, we have that
    \[
    \binom{p^{l-1}(p-1)}{j} \left[ p^{-(s'-1)p^{l-1}(p-1)}(\theta - \xi)^{p^{l-1}(p-1)-j}(f^{(s'-1)}_{i, \xi} \star a) \right] T^j(y)
    \]
    converges to $0$ as $l \to + \infty$. Since $\mathcal{I}^{h+1} A_{\opn{ord}, \infty}\langle X_{\opn{sim}}, X_{a, b, \tau} \rangle$ is a closed ideal, by taking $M = p^{l-1}(p-1)$ and passing to the limit as $l \to +\infty$, this implies that 
    \[
    F(ay) - \tilde{F}(a y) \in \mathcal{I}^{h+1} A_{\opn{ord}, \infty}\langle X_{\opn{sim}}, X_{a, b, \tau} \rangle
    \]
    and hence 
    \[
    (F - \tilde{F})(\mathcal{I}^{h} A_{\opn{ord}, \infty}\langle X_{\opn{sim}}, X_{a, b, \tau} \rangle) \subset \mathcal{I}^{h+1} A_{\opn{ord}, \infty}\langle X_{\opn{sim}}, X_{a, b, \tau} \rangle
    \]
    for any $h \geq 0$. By using the idempotent property in (\ref{FFtildeoddidempotent}), we therefore see that
    \[
    (F - \tilde{F})(x) \in \bigcap_{h \geq 0} \mathcal{I}^{h} A_{\opn{ord}, \infty}\langle X_{\opn{sim}}, X_{a, b, \tau} \rangle = \{ 0 \}
    \]
    for any $x \in A_{\opn{ord}, \infty}\langle X_{\opn{sim}}, X_{a, b, \tau} \rangle$. Hence $F = \tilde{F}$. Therefore, by an induction argument on $0 \leq s' \leq s$, we see that $f^{(s)}_{i, \xi} \star -$ is equal to $f^{(s)}_{i, \xi} \tilde{\star} -$ for any $i$, $\xi$.

    But the subspace of $A_{\opn{ord}, \infty}^+\langle X_{\opn{sim}}, X_{a, b, \tau} \rangle$ killed by the action of $\overline{\ide{u}}_G$ is identified with the subspace of power series which are constant in the variables $\{ X_{a, 1, \tau_0} : 2 \leq a \leq 2n \}$, hence we see that $f^{(s)}_{i, \xi} \star - = f^{(s)}_{i, \xi} \tilde{\star} -$ preserves this subspace as required. This completes the proof of the claim when $\mathcal{M}^{\dagger}_G = \mathscr{M}^{\dagger}_G(U)$.
    
    The claim for $\mathcal{M}^{\dagger}_G = \mathscr{M}^{\dagger}_G(V)$ now follows from the fact that $\mathscr{N}^{(r, k)}_G(U)$ is dense in $\mathscr{N}^{(r, k)}_G(V_{\gamma_i})$, and the fact that the actions of $f$ and $\overline{\ide{u}}_G$ are continuous. Indeed, if $x \in \mathscr{M}^{(r, k)}_G(V_{\gamma_i})$ and $k' \geq k$ is such that the action of $f$ induces a continuous operator
    \[
    \mathscr{N}^{(r, k)}_G(V_{\gamma_i}) \xrightarrow{f \star -} \mathscr{N}^{(r', k')}_G(V_{\gamma_i})
    \]
    then we can write $x = \lim_{l \to +\infty} x_l$ with $x_l \in \mathscr{M}^{(r, k)}_G(U)$ and we see that
    \[
    X \star_{\overline{P}} ( f \star x) = \lim_{l \to + \infty} X \star_{\overline{P}} ( f \star x_l) = 0
    \]
    for any $X \in \overline{\ide{u}}_G$.

    The proof of the analogous claim for $H$ follows exactly the same argument. 
\end{proof}

\begin{example} \label{ExampleOfIndicatorFcnPreservingOC}
    Lemma \ref{IndicatorFunctionsPreserveOCLemma} implies that the actions of the locally constant functions $1_{U^{\circ}_{G, \beta}}$ and $1_{U^{\circ}_{G, \beta}, \chi}$ (introduced in \S \ref{AbstractComputationsSection}) preserve overconvergent forms.
\end{example}

\subsubsection{The main construction} \label{LACechVersionMainConstructionSSec}

Following \S \ref{TheMainConstructionSubSec}, we now construct certain morphisms of sheaves of overconvergent forms needed to defined the $p$-adic evaluation maps. We will freely use the notation in \S \ref{TheMainConstructionSubSec}.

\begin{definition} \label{NOCDefinitionOfThetadaggers}
    Let $U \in \mathcal{C}_{G,H}$ and $V \defeq U \cap (\mathcal{X}_{G, \opn{Iw}}(p^{\beta}) - \mathcal{Z}_{G, > n+1}(p^{\beta}))$. Let $s \geq 1$ be an integer. 
    \begin{enumerate}
        \item For $(\kappa, j) \in \mathcal{E}$, let 
        \begin{align*} 
        \vartheta^{\dagger}_{\kappa, j, \beta} \colon \mathscr{N}^{\dagger}_{G, \kappa^*}(U) &\to \mathscr{N}^{\dagger}_{H, \sigma_{\kappa}^{[j]}}(\hat{\iota}^{-1}U) \\
        \vartheta^{\dagger}_{\kappa, j, \beta} \colon \mathscr{N}^{\dagger}_{G, \kappa^*}(V) &\to \mathscr{N}^{\dagger}_{H, \sigma_{\kappa}^{[j]}}(\hat{\iota}^{-1}V)
        \end{align*}
        denote the $L$-linear morphisms as constructed in Definition \ref{AbstractDefOfThetadagger} (with $\mathcal{N}^{\dagger}_{G} = \mathscr{N}^{\dagger}_G(U)$ and $\mathcal{N}^{\dagger}_H = \mathscr{N}^{\dagger}_H(\hat{\iota}^{-1}U)$ in the former case, and $\mathcal{N}^{\dagger}_G = \mathscr{N}^{\dagger}_{G}(V)$ and $\mathcal{N}^{\dagger}_H = \mathscr{N}^{\dagger}_{H}(\hat{\iota}^{-1}V)$ in the latter case). These morphisms are functorial in $U$.
        \item Let $(R, R^+)$ be a Tate affinoid algebra over $(L, \mathcal{O}_L)$ and $(\kappa, j) \in \mathcal{X}_{R, s}$. We let
        \begin{align*}
        \vartheta^{\dagger, s\opn{-an}}_{\kappa, j, \beta} \colon \mathscr{N}^{\dagger, s\opn{-an}}_{G, \kappa^*}(U) &\to \mathscr{N}^{\dagger, \opn{an}}_{H, \sigma_{\kappa}^{[j]}}(\hat{\iota}^{-1}U) \\
        \vartheta^{\dagger, s\opn{-an}}_{\kappa, j, \beta} \colon \mathscr{N}^{\dagger, s\opn{-an}}_{G, \kappa^*}(V) &\to \mathscr{N}^{\dagger, \opn{an}}_{H, \sigma_{\kappa}^{[j]}}(\hat{\iota}^{-1}V)
        \end{align*} 
        denote the $R$-linear morphisms, which are functorial in $U$, as constructed in Definition \ref{AbstractDefOfThetadaggersan}. 
        \item Let $(\kappa, j) \in \mathcal{E}$ and let $\chi = (\chi_{\tau}) \colon \prod_{\tau \in \Psi} \mbb{Z}_p^{\times} \to L^{\times}$ be a finite-order character such that $\chi_{\tau}$ is trivial on $1 + p^{\beta}\mbb{Z}_p$ for all $\tau \in \Psi$. We let 
        \begin{align*} 
        \vartheta^{\dagger, \circ}_{\kappa, j+\chi, \beta} \colon \mathscr{N}^{\dagger}_{G, \kappa^*}(U) &\to \mathscr{N}^{\dagger}_{H, \sigma_{\kappa}^{[j]}}(\hat{\iota}^{-1}U) \\
        \vartheta^{\dagger, \circ}_{\kappa, j+\chi, \beta} \colon \mathscr{N}^{\dagger}_{G, \kappa^*}(V) &\to \mathscr{N}^{\dagger}_{H, \sigma_{\kappa}^{[j]}}(\hat{\iota}^{-1}V)
        \end{align*}
        denote the $L$-linear morphisms constructed in Definition \ref{abstractcirclalgTheta}, which are functorial in $U$.
    \end{enumerate}
\end{definition}

\begin{remark} \label{CompatibilityRelationsForThetadaggersRem}
    Let $1_{U_{G, \beta}^{\circ}} \in C^{\opn{la}}(U_{G, \beta}, L)$ denote the indicator function of $U_{G, \beta}^{\circ} \subset U_{G, \beta}$ (see Definition \ref{DefOfUGbetacircSupport}). Then we have the following compatibility properties:
    \begin{itemize}
        \item Let $(\kappa, j) \in \mathcal{E}$ and let $\chi = (\chi_{\tau}) \colon \prod_{\tau \in \Psi} \mbb{Z}_p^{\times} \to L^{\times}$ be a finite-order character such that $\chi_{\tau}$ is trivial on $1 + p^{\beta}\mbb{Z}_p$ for all $\tau \in \Psi$. Then we have commutative diagrams:
        \[
\begin{tikzcd}
{\mathscr{N}^{\dagger, s\opn{-an}}_{G, \kappa^*}(U)} \arrow[r, "{\vartheta^{\dagger, s\opn{-an}}_{\kappa, j+\chi, \beta}}"] \arrow[d] & {\mathscr{N}^{\dagger, \opn{an}}_{H, \sigma_{\kappa}^{[j]}}(\hat{\iota}^{-1}U)} \arrow[d, equals] &  & {\mathscr{N}^{\dagger}_{G, \kappa^*}(U)} \arrow[d, "{1_{U_{G, \beta}^{\circ}} \star}"'] \arrow[r, "{\vartheta^{\dagger, \circ}_{\kappa, j, \beta}}"] & {\mathscr{N}^{\dagger}_{H, \sigma_{\kappa}^{[j]}}(\hat{\iota}^{-1}U)} \arrow[d, equals] \\
{\mathscr{N}^{\dagger}_{G, \kappa^*}(U)} \arrow[r, "{\vartheta^{\dagger, \circ}_{\kappa, j+\chi, \beta}}"]                            & {\mathscr{N}^{\dagger}_{H, \sigma_{\kappa}^{[j]}}(\hat{\iota}^{-1}U)}                     &  & {\mathscr{N}^{\dagger}_{G, \kappa^*}(U)} \arrow[r, "{\vartheta^{\dagger}_{\kappa, j, \beta}}"]                                                       & {\mathscr{N}^{\dagger}_{H, \sigma_{\kappa}^{[j]}}(\hat{\iota}^{-1}U)}          
\end{tikzcd}        
        \]
        Furthermore, we have a commutative diagram:
        \[
\begin{tikzcd}
{\mathscr{N}_{G, \kappa^*}(U)} \arrow[d] \arrow[r, "{\vartheta_{\kappa, j, \beta}}"]           & {\mathscr{N}_{H, \sigma_{\kappa}^{[j]}}(\hat{\iota}^{-1}U)} \arrow[d] \\
{\mathscr{N}^{\dagger}_{G, \kappa^*}(U)} \arrow[r, "{\vartheta^{\dagger}_{\kappa, j, \beta}}"] & {\mathscr{N}^{\dagger}_{H, \sigma_{\kappa}^{[j]}}(\hat{\iota}^{-1}U)}
\end{tikzcd}
        \]
        where $\vartheta_{\kappa, j, \beta}$ is defined in Definition \ref{DefinitionOfScrPullbackNGNH}.
        \item If $s' \geq s$ and $(\kappa, j) \in \mathcal{X}_{R, s} \subset \mathcal{X}_{R, s'}$, then we have a commutative diagram:
        \[
\begin{tikzcd}
{\mathscr{N}^{\dagger, s'\opn{-an}}_{G, \kappa^*}(U)} \arrow[d] \arrow[r, "{\vartheta^{\dagger, s'\opn{-an}}_{\kappa, j, \beta}}"] & {\mathscr{N}^{\dagger, \opn{an}}_{H, \sigma_\kappa^{[j]}}(\hat{\iota}^{-1}U)} \\
{\mathscr{N}^{\dagger, s\opn{-an}}_{G, \kappa^*}(U)} \arrow[ru, "{\vartheta^{\dagger, s\opn{-an}}_{\kappa, j, \beta}}"']           &                                                                              
\end{tikzcd}
        \]
        \item If $(R, R^+) \to (R', (R')^+)$ is a morphism of Tate affinoid algebras over $(L, \mathcal{O}_L)$, and $(\kappa', j') \in \mathcal{X}_{R', s}$ denotes the image of $(\kappa, j) \in \mathcal{X}_{R, s}$ under the natural map $\mathcal{X}_{R, s} \to \mathcal{X}_{R', s}$, then $\vartheta^{\dagger, s\opn{-an}}_{\kappa, j, \beta}$ and $\vartheta^{\dagger, s\opn{-an}}_{\kappa', j', \beta}$ are compatible under the natural maps $\mathscr{N}^{\dagger, s\opn{-an}}_{G, \kappa^*}(U) \to \mathscr{N}^{\dagger, s\opn{-an}}_{G, (\kappa')^*}(U)$ and $\mathscr{N}^{\dagger, \opn{an}}_{H, \sigma_{\kappa}^{[j]}}(\hat{\iota}^{-1}U) \to \mathscr{N}^{\dagger, \opn{an}}_{H, \sigma_{\kappa'}^{[j']}}(\hat{\iota}^{-1}U)$.
    \end{itemize}
    We have similar compatibility relations for the modules $\mathscr{N}_{G, \kappa^*}(V)$, $\mathscr{N}^{\dagger}_{G, \kappa^*}(V)$, $\mathscr{N}^{\dagger, s\opn{-an}}_{G, \kappa^*}(V)$ etc.
\end{remark}

We now prove an analogous result to Proposition \ref{ClassicalThetaPreservesHolProp}, namely that the morphisms in Definition \ref{NOCDefinitionOfThetadaggers} take overconvergent forms on $G$ to overconvergent forms on $H$.

\begin{proposition} \label{DaggerThetaTakesOCtoOCProp}
    Let $U \in \mathcal{C}_{G,H}$ and $V \defeq U \cap (\mathcal{X}_{G, \opn{Iw}}(p^{\beta}) - \mathcal{Z}_{G, > n+1}(p^{\beta}))$.
    \begin{enumerate}
        \item Let $(\kappa, j) \in \mathcal{E}$. Then the morphisms $\vartheta^{\dagger}_{\kappa, j, \beta}$ induce $L$-linear morphisms
        \begin{align*} 
        \vartheta^{\dagger}_{\kappa, j, \beta} \colon \mathscr{M}^{\dagger}_{G, \kappa^*}(U) &\to \mathscr{M}^{\dagger}_{H, \sigma_{\kappa}^{[j]}}(\hat{\iota}^{-1}U) \\
        \vartheta^{\dagger}_{\kappa, j, \beta} \colon \mathscr{M}^{\dagger}_{G, \kappa^*}(V) &\to \mathscr{M}^{\dagger}_{H, \sigma_{\kappa}^{[j]}}(\hat{\iota}^{-1}V) .
        \end{align*}
        \item Let $(\kappa, j) \in \mathcal{E}$ and let $\chi = (\chi_{\tau}) \colon \prod_{\tau \in \Psi} \mbb{Z}_p^{\times} \to L^{\times}$ be a finite-order character such that $\chi_{\tau}$ is trivial on $1 + p^{\beta}\mbb{Z}_p$ for all $\tau \in \Psi$. Then the morphisms $\vartheta^{\dagger, \circ}_{\kappa, j+\chi, \beta}$ induce $L$-linear morphisms
        \begin{align*} 
        \vartheta^{\dagger, \circ}_{\kappa, j+\chi, \beta} \colon \mathscr{M}^{\dagger}_{G, \kappa^*}(U) &\to \mathscr{M}^{\dagger}_{H, \sigma_{\kappa}^{[j]}}(\hat{\iota}^{-1}U) \\
        \vartheta^{\dagger, \circ}_{\kappa, j+\chi, \beta} \colon \mathscr{M}^{\dagger}_{G, \kappa^*}(V) &\to \mathscr{M}^{\dagger}_{H, \sigma_{\kappa}^{[j]}}(\hat{\iota}^{-1}V) .
        \end{align*}
        \item Let $(R, R^+)$ be a Tate affinoid algebra over $(L, \mathcal{O}_L)$ and set $\Omega = \opn{Spa}(R, R^+)$. Let $(\kappa, j) \in \mathcal{X}_{R, s}$, and suppose that there exists a Zariski dense subset $\Sigma \subset \Omega(\mbb{C}_p)$ such that for any point $x \in \Sigma$ (corresponding to a morphism $x \colon R \to \mbb{C}_p$), the induced character $(\kappa_x, j_x) \defeq x \circ (\kappa, j)$ lies in $\mathcal{E}$. Then the morphisms $\vartheta^{\dagger, s\opn{-an}}_{\kappa, j, \beta}$ induce $R$-linear morphisms
        \begin{align*}
        \vartheta^{\dagger, s\opn{-an}}_{\kappa, j, \beta} \colon \mathscr{M}^{\dagger, s\opn{-an}}_{G, \kappa^*}(U) &\to \mathscr{M}^{\dagger, \opn{an}}_{H, \sigma_{\kappa}^{[j]}}(\hat{\iota}^{-1}U) \\
        \vartheta^{\dagger, s\opn{-an}}_{\kappa, j, \beta} \colon \mathscr{M}^{\dagger, s\opn{-an}}_{G, \kappa^*}(V) &\to \mathscr{M}^{\dagger, \opn{an}}_{H, \sigma_{\kappa}^{[j]}}(\hat{\iota}^{-1}V) .
        \end{align*}
    \end{enumerate}
\end{proposition}
\begin{proof}
    Parts (1) and (2) follow the same strategy as in Proposition \ref{ClassicalThetaPreservesHolProp}. More precisely, let
    \[
    (\mathcal{M}^{\dagger}_G, \mathcal{N}^{\dagger}_G, \mathcal{M}^{\dagger}_{G, \kappa^*}, \mathcal{N}^{\dagger}_{G, \kappa^*}, \mathcal{N}^{\dagger}_H )
    \]
    be one of the following tuples:
    \begin{align*}
        (\mathscr{M}^{\dagger}_G(U), \; \mathscr{N}^{\dagger}_G(U), \; &\mathscr{M}^{\dagger}_{G, \kappa^*}(U), \; \mathscr{N}^{\dagger}_{G, \kappa^*}(U), \; \mathscr{N}^{\dagger}_{H}(\hat{\iota}^{-1}U) ) \\
        (\mathscr{M}^{\dagger}_G(V), \; \mathscr{N}^{\dagger}_G(V), \; &\mathscr{M}^{\dagger}_{G, \kappa^*}(V), \; \mathscr{N}^{\dagger}_{G, \kappa^*}(V), \; \mathscr{N}^{\dagger}_{H}(\hat{\iota}^{-1}V) )  
    \end{align*}
    and let $(\vartheta, \xi) \in \{ (\vartheta^{\dagger}_{\kappa, j, \beta}, 1_{U_{G, \beta}}), (\vartheta^{\dagger, \circ}_{\kappa, j+\chi, \beta}, 1_{U_{G, \beta}^{\circ}, \chi}) \}$. Let $F \in \mathcal{M}^{\dagger}_{G, \kappa^*} \subset \mathcal{N}^{\dagger}_{G, \kappa^*}$, so in particular, $F$ is killed by the action of $\overline{\ide{u}}_G$ under $\star_{\overline{P}}$. Recall from Definition \ref{AbstractDefOfThetadagger} or Definition \ref{abstractcirclalgTheta}, that the morphism $\vartheta$ is the $M^H_{\diamondsuit}(p^{\beta})$-invariants of the composition of two maps:
    \[
    \mathcal{N}^{\dagger}_{G} \otimes V_{\kappa}^* \xrightarrow{\pi_1} \mathcal{N}^{\dagger}_{G} \otimes \sigma_{\kappa}^{[j]} \xrightarrow{\pi_2} \mathcal{N}^{\dagger}_{H} \otimes \sigma_{\kappa}^{[j]}
    \]
    where the first map $\pi_1$ is the composition of the first three bullet points in Definition \ref{AbstractDefOfThetadagger} (or the analogous maps in Definition \ref{abstractcirclalgTheta}), and $\pi_2$ is the natural pullback map (as discussed at the end of \S \ref{FunctorialityOfGMConnectionOnNOCSubSec}). Here the action of $M^H_{\diamondsuit}(p^{\beta})$ on $\mathcal{N}^{\dagger}_G$ and $V_{\kappa}^*$ is through the embedding $u^{-1}M^H_{\diamondsuit}(p^{\beta}) u \subset M^G_{\opn{Iw}}(p^{\beta})$, and we have twisted the maps in Definition \ref{AbstractDefOfThetadagger} by $\sigma_{\kappa}^{[j]}$. 
    
    With notation as in the proof of Proposition \ref{ClassicalThetaPreservesHolProp}, we wish to show $E_{i, 1} \star_{\overline{P}} \vartheta(F) = 0$ for all $i=2, \dots, n$. Since $\pi_2$ is equivariant for the action of $\overline{\ide{u}}_H$ through the embedding $\opn{Ad}(u^{-1})\overline{\ide{u}}_H  \subset \overline{\ide{u}}_G$, it suffices to prove $\opn{Ad}(u^{-1})E_{i, 1} \star_{\overline{P}} \pi_1(F) = 0$.

    Fix a basis $\{v_l \}$ of $V_{\kappa}$ and note that $F$ can be viewed as a $M^G_{\opn{Iw}}(p^{\beta})$-equivariant map $\tilde{F} \colon V_{\kappa} \to \mathcal{N}^{\dagger}_G$. Then, with notation as in Proposition \ref{ClassicalThetaPreservesHolProp}, we have
    \[
    \pi_1(F) = \sum_l \sum_{\substack{T \subset \{n+2, \dots, 2n\} \\ \# T = j_{\tau_0}}} \lambda_{l, T} \left[ p_T \star_{\ide{u}} (\xi \star_{\ide{u}} \tilde{F}(u^{-1} \cdot v_l) ) \right]
    \]
    where $p_T$ denotes the restriction of $u^{-1} \cdot x^T$ to a polynomial on $U_{G, \beta}$. Note that $\frac{\partial p_T}{\partial x_k}$ is the restriction of $u^{-1} \frac{\partial x^T}{\partial x_k}$ to $U_{G, \beta}$, for all $k=n+1, \dots, 2n$ (since the change of coordinates given by the action of $u^{-1}$ only shifts the coordinates by $x_i$ for $i=2, \dots, n$). Here $\star_{\ide{u}}$ denotes the action in Theorem \ref{MainTheoremOnPadicIteration} or Lemma \ref{LaActionExtendsToVLemma}. By acting on both sides of (\ref{EstarpstarFHolomorphicCase}) by $u^{-1}$ and using the density of $\mathscr{N}_G(U)$ in $\mathcal{N}^{\dagger}_G$, we have for any $F' \in \mathcal{N}^{\dagger}_G$
    \[
    \opn{Ad}(u^{-1})E_{i, 1} \star_{\overline{P}} ( p_T \star_{\ide{u}} F') = p_T \star_{\ide{u}} ( \opn{Ad}(u^{-1})E_{i, 1} \star_{\overline{P}} F') + \sum_{k=n+1}^{2n} \frac{\partial p_T}{\partial x_k} \star_{\ide{u}} ( \opn{Ad}(u^{-1})E_{i, k} \star_{\overline{P}} F') .
    \]
    Then, using the fact that $\overline{\ide{u}}_G$ kills $\xi \star_{\ide{u}} \tilde{F}(u^{-1} \cdot v_l)$ (because $F \in \mathcal{M}^{\dagger}_G$ and $\xi$ preserves overconvergent forms -- see Example \ref{ExampleOfIndicatorFcnPreservingOC}), we have
    \[
    \opn{Ad}(u^{-1})E_{i, 1} \star_{\overline{P}} \pi_1(F) =  \sum_l \sum_{\substack{T \subset \{n+2, \dots, 2n\} \\ \# T = j_{\tau_0}}} \lambda_{l, T} \left( \sum_{k=n+1}^{2n} \frac{\partial p_T}{\partial x_k} \star_{\ide{u}} \left[ \opn{Ad}(u^{-1})E_{i, k} \star_{\overline{P}} (\xi \star_{\ide{u}} \tilde{F}(u^{-1} \cdot v_l) ) \right]  \right) .
    \]
    Now we consider the following map
    \[
    g \colon S_{-(j-1)} \otimes V_{\kappa} \to \mathcal{N}^{\dagger}_G
    \]
    given by $g( y \otimes z) = (u^{-1} \cdot \Phi_{\beta}(y) )|_{U_{G, \beta}} \star_{\ide{u}} (\xi \star_{\ide{u}} \tilde{F}(u^{-1} \cdot z))$. This is $M^H_{\diamondsuit}(p^{\beta})$-equivariant in the following sense: for any $h \in M^H_{\diamondsuit}(p^{\beta})$, we have
    \[
    g( h \cdot y \otimes h \cdot z) = (u^{-1}h u) \cdot g(y \otimes z) .
    \]
    Here we are using the fact that $(u^{-1}hu) \cdot \xi = \sigma_0^{[\chi]}(h)^{-1} \xi = \xi$. Note that there are elements $A_i \in S_{-(j-1)} \otimes V_{\kappa}$ ($i=2, \dots, n$) such that $g(A_i) = \opn{Ad}(u^{-1})E_{i, 1} \star_{\overline{P}} \pi_1(F)$ (because $\xi$ is killed under the action of $\ide{m}_G$). 

    On the other hand, let $W$ be as in Proposition \ref{ClassicalThetaPreservesHolProp}, then we define a morphism 
    \begin{align*} 
    h \colon W \otimes \sigma_{\kappa}^{[j], -1} &\to \mathcal{N}^{\dagger}_{G} \\
    e_1^* e_i &\mapsto \opn{Ad}(u^{-1})E_{i, 1} \star_{\overline{P}} \pi_1(F) 
    \end{align*} 
    and extending linearly. This is $M^H_{\diamondsuit}(p^{\beta})$-equivariant in the following sense: for any $m \in M^H_{\diamondsuit}(p^{\beta})$ and $w \in W \otimes \sigma_{\kappa}^{[j], -1}$, we have
    \[
    h( m \cdot w ) = (u^{-1}m u) \cdot h(w) .
    \]
    As in Proposition \ref{ClassicalThetaPreservesHolProp}, we therefore obtain a $M^H_{\diamondsuit}(p^{\beta})$-equivariant morphism
    \[
    \bar{q} \colon W \otimes \sigma_{\kappa}^{[j],-1} \to \left( S_{-(j-1)} \otimes V_{\kappa} \right)/\opn{ker}(g) .
    \]
    We now claim that the natural map $S_{-(j-1)} \otimes V_{\kappa} \twoheadrightarrow \left( S_{-(j-1)} \otimes V_{\kappa} \right)/\opn{ker}(g)$ is actually $M_H(\mbb{Q}_p)$-equivariant. For this, it is enough to show that $\opn{ker}(g)$ is stable under $M_H(\mbb{Q}_p)$. Choose a basis $\{ s_1, \dots, s_a \}$ of $S_{-(j-1)} \otimes V_{\kappa}$ such that $\{ s_1, \dots, s_b \}$ is a basis of $\opn{ker}(g)$ for some $1 \leq b \leq a$. Let $\mu_i \colon S_{-(j-1)} \otimes V_{\kappa} \to \mbb{Q}_p$ denote the linear functional projecting to the coefficient of $s_i$. Fix $x \in \opn{ker}(g)$ and consider the map
    \begin{align*} 
    Q_i \colon M_H(\mbb{Q}_p) &\to \mbb{Q}_p \\
    m &\mapsto \mu_i(m \cdot x ) .
    \end{align*} 
    Since $S_{-(j-1)} \otimes V_{\kappa}$ is an algebraic representation, the maps $Q_i$ are algebraic (i.e. they extend to algebraic morphisms $M_H \to \mbb{A}^1$). We already know that for $b+1 \leq i \leq a$, the functions $Q_i$ vanish on $M^H_{\diamondsuit}(p^{\beta})$. But this subgroup is Zariski dense in $M_H$, hence we must have $Q_i = 0$ for all $i=b+1, \dots, a$. This implies $\opn{ker}(g)$ is stable under $M_H(\mbb{Q}_p)$. More generally, this argument shows that any $M^H_{\diamondsuit}(p^{\beta})$-stable subspace of a finite-dimensional algebraic representation of $M_H$, is actually stable under $M_H(\mbb{Q}_p)$.

    A similar argument shows that $\bar{q}$ is in fact $M_H(\mbb{Q}_p)$-equivariant.  Indeed, for any $x \in W \otimes \sigma_{\kappa}^{[j],-1}$, consider the following morphism $Q_x \colon M_H(\mbb{Q}_p) \to \mbb{Q}_p$ given by $m \mapsto \bar{q}(m \cdot x) - m \cdot \bar{q}(x)$, which makes sense because the target of $\bar{q}$ carries an action of $M_H(\mbb{Q}_p)$. Since $\bar{q}$ is linear, and the actions of $M_H(\mbb{Q}_p)$ on both the source and target are algebraic, the morphisms $Q_x$ are algebraic. But the morphism $\bar{q}$ is $M^H_{\diamondsuit}(p^{\beta})$-equivariant, hence $Q_x$ vanishes on $M^H_{\diamondsuit}(p^{\beta})$. By density, this implies that $Q_x$ is identically zero, and since $x$ is arbitrary, this implies that $\bar{q}$ is $M_H(\mbb{Q}_p)$-equivariant.

    To complete the proof of (1) and (2), we now note from the proof of Proposition \ref{ClassicalThetaPreservesHolProp} that the morphism $\bar{q}$ must be zero; hence the map $h$ is zero as required.

    We now prove part (3). Firstly note that for any $x \in \Sigma$, any $i \in \{2, \dots, n\}$, and any $z \in \mathscr{M}^{\dagger, s\opn{-an}}_{G, \kappa_x^*}(U)$ or $\mathscr{M}^{\dagger, s\opn{-an}}_{G, \kappa_x^*}(V)$
    \[
    E_{i, 1} \star_{\overline{P}} \vartheta^{\dagger, s\opn{-an}}_{\kappa_x, j_x, \beta}(z) = E_{i, 1} \star_{\overline{P}} \vartheta^{\dagger}_{\kappa_x, j_x, \beta}(1_{U_{G, \beta}^{\circ}} \star z) = 0
    \]
    by the compatibility relations in Remark \ref{CompatibilityRelationsForThetadaggersRem} and the fact that the action of $1_{U_{G, \beta}^{\circ}}$ preserves overconvergent forms (Example \ref{ExampleOfIndicatorFcnPreservingOC}). Furthermore, let $z \in \mathscr{M}^{(r, k), s\opn{-an}}_{G, \kappa^*}(U)$ (resp. $z \in \mathscr{M}^{(r, k), s\opn{-an}}_{G, \kappa^*}(V_{\gamma_i})$). Then unwinding the definitions in Definition \ref{NOCDefinitionOfThetadaggers}, we see that there exist integers $r', k'$ and a rational number $\gamma > 0$ such that
    \begin{itemize}
        \item $\vartheta^{\dagger, s\opn{-an}}_{\kappa, j, \beta}(z) \in \mathscr{N}^{(r', k'), \opn{an}}_{H, \sigma_{\kappa}^{[j]}}(\hat{\iota}^{-1}U)$ (resp. $\vartheta^{\dagger, s\opn{-an}}_{\kappa, j, \beta}(z) \in \mathscr{N}^{(r', k'), \opn{an}}_{H, \sigma_{\kappa}^{[j]}}(\hat{\iota}^{-1}V_{\gamma})$). 
        \item For any $X \in \overline{\ide{u}}_H$, the specialisation of $X \star_{\overline{P}} \vartheta^{\dagger, s\opn{-an}}_{\kappa, j, \beta}(z)$ at any point $x \in \Sigma$ is zero. 
    \end{itemize}
    By \cite[Proposition 6.3.3]{BoxerPilloni}, $\mathscr{N}^{(r', k'), \opn{an}}_{H, \sigma_{\kappa}^{[j]}}(\hat{\iota}^{-1}U)$ (resp. $\mathscr{N}^{(r', k'), \opn{an}}_{H, \sigma_{\kappa}^{[j]}}(\hat{\iota}^{-1}V_{\gamma})$) is a projective Banach $C \hatot R$-module, where $C$ denotes the sections of an appropriate quasi-compact open affinoid subspace of $\mathcal{X}_{H, \diamondsuit}(p^{\beta})$. This implies that $X \star_{\overline{P}} \vartheta^{\dagger, s\opn{-an}}_{\kappa, j, \beta}(z) = 0$. Indeed, this follows from the general fact: if $M$ is a projective Banach $C \hatot R$-module, then any element $m \in M$ whose specialisations satisfy $m_x = 0$ for any $x \in \Sigma$ must satisfy $m=0$. This can be proven by reducing to the case where $M$ is orthonormalisable, and then to the setting where $M = C \hatot R$, where the claim follows from Zariski density. 
\end{proof}

We now consider the induced morphisms on cohomology. We consider three cases:
\begin{enumerate}
    \item Let $(\kappa, j) \in \mathcal{E}$. In this case, we set $\mathscr{F}^{(r,k)} = \mathscr{M}^{(r,k)}_{G, \kappa^*}$ and $\mathscr{G}^{(r,k)} = \hat{\iota}_* \mathscr{M}^{(r,k)}_{H, \sigma_{\kappa}^{[j]}}$. We also let $\vartheta = \vartheta^{\dagger}_{\kappa, j, \beta}$.
    \item Let $(R, R^+)$ be a Tate affinoid algebra over $(L, \mathcal{O}_L)$ and $(\kappa, j) \in \mathcal{X}_{R, s}$ satisfying the conditions in Proposition \ref{DaggerThetaTakesOCtoOCProp}(2). In this case, we set $\mathscr{F}^{(r,k)} = \mathscr{M}^{(r,k), s\opn{-an}}_{G, \kappa^*}$ and $\mathscr{G}^{(r,k)} = \hat{\iota}_* \mathscr{M}^{(r,k), \opn{an}}_{H, \sigma_{\kappa}^{[j]}}$ (for $k \geq s+1$). We also let $\vartheta = \vartheta^{\dagger, s\opn{-an}}_{\kappa, j, \beta}$.
    \item Let $(\kappa, j) \in \mathcal{E}$ and let $\chi = (\chi_{\tau}) \colon \prod_{\tau \in \Psi} \mbb{Z}_p^{\times} \to L^{\times}$ be a finite-order character such that $\chi_{\tau}$ is trivial on $1 + p^{\beta}\mbb{Z}_p$ for all $\tau \in \Psi$. In this case, we set $\mathscr{F}^{(r,k)} = \mathscr{M}^{(r,k)}_{G, \kappa^*}$ and $\mathscr{G}^{(r,k)} = \hat{\iota}_* \mathscr{M}^{(r,k)}_{H, \sigma_{\kappa}^{[j]}}$. We also let $\vartheta = \vartheta^{\dagger, \circ}_{\kappa, j+\chi, \beta}$.
\end{enumerate}
Let $\ide{U} = (U_i)_{i \in I}$ be an open cover of $\mathcal{X}_{G, \opn{Iw}}(p^{\beta})$ as in Lemma \ref{AcyclicityCoverLemma} (for both $\mathscr{F}^{(r, k)}$ and $\mathscr{G}^{(r, k)}$) and let $\ide{V} = (V_i)_{i \in I}$ denote the induced cover of $\mathcal{X}_{G, \opn{Iw}}(p^{\beta}) - \mathcal{Z}_{G, >n+1}(p^{\beta})$ (i.e. we set $V_i = U_i \cap (\mathcal{X}_{G, \opn{Iw}}(p^{\beta}) - \mathcal{Z}_{G, >n+1}(p^{\beta}))$). Then since the morphisms $\vartheta$ are functorial in $U \in \mathcal{C}_{G,H}$ and take overconvergent forms to overconvergent forms, we obtain a commutative diagram of \v{C}ech complexes
\begin{equation} \label{CommutativeDiagOfCechComplex}
\begin{tikzcd}
\varinjlim_{r, k}\opn{Cech}(\mathscr{F}^{(r, k)}; \ide{U}) \arrow[d, "\vartheta"'] \arrow[r] & \varinjlim_{r, k} \opn{Cech}(\mathscr{F}^{(r,k)}; \ide{V} ) \arrow[d, "\vartheta"] \\
\varinjlim_{r, k}\opn{Cech}(\mathscr{G}^{(r, k)}; \ide{U}) \arrow[r]                         & \varinjlim_{r,k} \opn{Cech}(\mathscr{G}^{(r,k)}; \ide{V} ) .                      
\end{tikzcd}
\end{equation}
By Remark \ref{CanBeComputeUsingCechRem}, this induces morphisms:
\begin{align*}
    \vartheta^{\dagger}_{\kappa, j, \beta} \colon R\Gamma^G_{w_n}(\kappa^*; \beta)^{(-, \dagger)} \to R\Gamma^H_{\opn{id}}(\sigma_{\kappa}^{[j]}; \beta)^{(-, \dagger)} \to R\Gamma_{\mathcal{Z}_{H, \opn{id}}(p^{\beta})}\left( \mathcal{S}_{H, \diamondsuit}(p^{\beta}), \mathscr{M}_{H, \sigma_{\kappa}^{[j]}} \right) \\
    \vartheta^{\dagger, s\opn{-an}}_{\kappa, j, \beta} \colon R\Gamma^G_{w_n, s\opn{-an}}(\kappa^*; \beta)^{(-, \dagger)} \to R\Gamma^H_{\opn{id}, \opn{an}}(\sigma_{\kappa}^{[j]}; \beta)^{(-, \dagger)} \to R\Gamma^H_{\opn{id}, \opn{an}}(\mathcal{S}_{H, \diamondsuit}(p^{\beta}), \sigma_{\kappa}^{[j]})^{(-, \dagger)} \\
    \vartheta^{\dagger, \circ}_{\kappa, j+\chi, \beta} \colon R\Gamma^G_{w_n}(\kappa^*; \beta)^{(-, \dagger)} \to R\Gamma^H_{\opn{id}}(\sigma_{\kappa}^{[j]}; \beta)^{(-, \dagger)} \to R\Gamma_{\mathcal{Z}_{H, \opn{id}}(p^{\beta})}\left( \mathcal{S}_{H, \diamondsuit}(p^{\beta}), \mathscr{M}_{H, \sigma_{\kappa}^{[j]}} \right)
\end{align*}
in cases (1), (2), and (3) respectively. In all three compositions, the first map is induced from the commutative diagram (\ref{CommutativeDiagOfCechComplex}) and the second map is just restriction along the open and closed embedding in Lemma \ref{LemmaSDintoPEL} (note we have used the excision property in (\ref{ExcisionEqnlabel}) for the first and third map, and by abuse of notation we also use the notation $\mathcal{Z}_{H, \opn{id}}(p^{\beta})$ for its intersection with $\mathcal{S}_{H, \diamondsuit}(p^{\beta})$).

\subsubsection{Overconvergent evaluation maps for classical and locally algebraic weights}

For $(\kappa, j) \in \mathcal{E}$ or $\mathcal{X}_{R, s}$, let $\sigma_{\kappa}^{[j], \vee}$ denote the Serre dual of $\sigma_{\kappa}^{[j]}$, i.e
\[
\sigma_{\kappa}^{[j], \vee} \defeq - w_{M_H}^{\opn{max}} \sigma_{\kappa}^{[j]} - 2 \rho_{H, \opn{nc}}
\]
where $\rho_{H, \opn{nc}}$ denotes the half-sum of positive roots of $H$ not lying in $M_H$.

\begin{definition}
    \begin{enumerate}
        \item For $(\kappa, j) \in \mathcal{E}$, let 
        \[
        R\Gamma^H_{\opn{id}}( \mathcal{S}_{H, \diamondsuit}(p^{\beta}), \sigma_{\kappa}^{[j], \vee} )^{(+, \dagger)} \defeq \varinjlim_U R \Gamma( U, \mathscr{M}_{H, \sigma_{\kappa}^{[j], \vee}} ) = \varinjlim_{r, k} R\Gamma\left( \mathcal{S}_{H, \diamondsuit}(p^{\beta}), \mathscr{M}^{(r, k)}_{H, \sigma_{\kappa}^{[j], \vee}} \right)
        \]
        where the colimit is over all open neighbourhoods of $\mathcal{Z}_{H, \opn{id}}(p^{\beta})$ inside $\mathcal{S}_{H, \diamondsuit}(p^{\beta})$ (with transition maps given by restriction).
        \item Let $(\kappa, j) \in \mathcal{X}_{R, s}$. Then we define
        \[
        R\Gamma^H_{\opn{id}, \opn{an}}( \mathcal{S}_{H, \diamondsuit}(p^{\beta}), \sigma_{\kappa}^{[j], \vee} )^{(+, \dagger)} \defeq \varinjlim_{r, k} R\Gamma \left( \mathcal{S}_{H, \diamondsuit}(p^{\beta}), \mathscr{M}^{(r, k), \opn{an}}_{H, \sigma_{\kappa}^{[j], \vee}} \right) .
        \]
    \end{enumerate}
    Note that $R\Gamma^H_{\opn{id}}( \mathcal{S}_{H, \diamondsuit}(p^{\beta}), \sigma_{\kappa}^{[j], \vee} )^{(+, \dagger)} = R\Gamma^H_{\opn{id}, \opn{an}}( \mathcal{S}_{H, \diamondsuit}(p^{\beta}), \sigma_{\kappa}^{[j], \vee} )^{(+, \dagger)}$ when $(\kappa, j) \in \mathcal{E}$.
\end{definition}

\begin{remark} \label{UFJsSerreDualityRem}
    As explained in \cite[\S 5.4]{UFJ}, we have Serre duality pairings
    \begin{align*}
        \opn{H}^{n-1-i}_{\opn{id}}( \mathcal{S}_{H, \diamondsuit}(p^{\beta}), \sigma_{\kappa}^{[j]} )^{(-, \dagger)} \times \opn{H}^{i}_{\opn{id}}( \mathcal{S}_{H, \diamondsuit}(p^{\beta}), \sigma_{\kappa}^{[j], \vee} )^{(+, \dagger)} &\to L \\
        \opn{H}^{n-1-i}_{\opn{id}, \opn{an}}( \mathcal{S}_{H, \diamondsuit}(p^{\beta}), \sigma_{\kappa}^{[j]} )^{(-, \dagger)} \times \opn{H}^{i}_{\opn{id}, \opn{an}}( \mathcal{S}_{H, \diamondsuit}(p^{\beta}), \sigma_{\kappa}^{[j], \vee} )^{(+, \dagger)} &\to R 
    \end{align*}
    for $i=0, \dots, n-1$, which are compatible with changing $R$. 
\end{remark}

Let $(\kappa, j) \in \mathcal{E}$ satisfying Assumption \ref{AssumpOnKJforACchar} and let $\chi \in \Sigma_{\kappa, j}(\ide{N}_{\beta})$ (see Definition \ref{DefOfSigmaKJNB}). Let $L/\mbb{Q}_p$ be a finite extension containing $F^{\opn{cl}}(\chi)$ and $\mbb{Q}_p(\mu_{p^{\beta}})$. All constructions are made over this finite extension $L/\mbb{Q}_p$, which we will once again omit from the notation. Let $\chi_{p} = (\chi_{p, \tau}) \colon \prod_{\tau \in \Psi} \mbb{Z}_p^{\times} \to L^{\times}$ denote the restriction of $\chi$ to $\prod_{\tau \in \Psi} \mathcal{O}^{\times}_{F_{\bar{\ide{p}}_{\tau}}} \cong \prod_{\tau \in \Psi} \mbb{Z}_p^{\times}$. Note that, for any $\tau \in \Psi$, $\chi_{p, \tau}$ is trivial on $1 + p^{\beta}\mbb{Z}_p$ because the conductor of $\chi$ divides $\ide{N}_{\beta} = \ide{N} p^{\beta}$. 

Recall from \S \ref{ClassicalEvaluationMapsSubSec} that we have a cohomology class
\[
[\chi] \in \opn{H}^0\left( \mathcal{S}_{H, \diamondsuit}(p^{\beta}), \mathscr{M}_{H, \sigma_{\kappa}^{[j], \vee}} \right) 
\]
associated with the anticyclotomic character $\chi$. Let $\opn{res}[\chi] \in \opn{H}^0_{\opn{id}}(\mathcal{S}_{H, \diamondsuit}(p^{\beta}), \sigma_{\kappa}^{[j], \vee})^{(+, \dagger)}$ denote the image of $[\chi]$ under the natural restriction map $\opn{H}^0\left( \mathcal{S}_{H, \diamondsuit}(p^{\beta}), \mathscr{M}_{H, \sigma_{\kappa}^{[j], \vee}} \right) \to \opn{H}^0_{\opn{id}}(\mathcal{S}_{H, \diamondsuit}(p^{\beta}), \sigma_{\kappa}^{[j], \vee})^{(+, \dagger)}$.

\begin{definition}
    With notation as above, we consider the following $L$-linear maps:
    \begin{align*}
    \opn{Ev}^{\dagger}_{\kappa, j, \chi, \beta} \colon \opn{H}^{n-1}_{w_n}(\kappa^*; \beta)^{(-, \dagger)} &\to L  \\
    \eta &\mapsto \langle \vartheta^{\dagger}_{\kappa, j, \beta}(\eta), \opn{res}[\chi] \rangle \\
    \opn{Ev}^{\dagger, \circ}_{\kappa, j, \chi, \beta} \colon \opn{H}^{n-1}_{w_n}(\kappa^*; \beta)^{(-, \dagger)} &\to L  \\
    \eta &\mapsto \langle \vartheta^{\dagger, \circ}_{\kappa, j+\chi_p, \beta}(\eta), \opn{res}[\chi] \rangle
    \end{align*}
    where, in both cases, $\langle \cdot, \cdot \rangle$ denotes the Serre duality pairing in Remark \ref{UFJsSerreDualityRem}.
\end{definition}

\subsubsection{Overconvergent evaluation maps for \texorpdfstring{$p$}{p}-adic weights} \label{OverconvergentEVMapsPAdicWtSSec}

Let $\mathcal{W}$ denote the adic space over $\opn{Spa}(\mbb{Q}_p, \mbb{Z}_p)$ parameterising continuous characters of $\prod_{\tau \in \Psi} \mbb{Z}_p^{\times}$. This has an increasing cover $\{ \mathcal{W}_{h} \}_{h \geq 1}$ by quasi-compact open affinoid subpaces where the universal character of $\mathcal{W}_h$ is $h$-analytic. We can (and do) assume that $\mathcal{W}_1$ contains all classical weights, i.e. the characters of the form 
\begin{align*}
    \prod_{\tau \in \Psi} \mbb{Z}_p^{\times} &\to \mbb{Q}_p^{\times} \\
    (x_{\tau})_{\tau \in \Psi} &\mapsto \prod_{\tau \in \Psi} x_{\tau}^{m_{\tau}}
\end{align*}
where $m_{\tau}$ are integers.

Let $\mathcal{W}(\ide{N} p^{\infty})$ denote the adic space over $\opn{Spa}(\mbb{Q}_p, \mbb{Z}_p)$ parameterising continuous characters of $\Gal(F_{\ide{N}p^{\infty}}/F)$, where $F_{\ide{N}p^{\infty}}/F$ denotes the abelian extension corresponding to 
\[
C_{\ide{N}p^{\infty}} \defeq F^{\times} \backslash \left( \mbb{A}_{F, f}^{\times} / \mbb{A}_{F^+, f}^{\times} \cdot (\widehat{\mathcal{O}}^{(p)}_{F^+} + \ide{N} \widehat{\mathcal{O}}_F^{(p)} )^{\times} \right)
\]
via Artin reciprocity, where $\widehat{\mathcal{O}}_{F}^{(p)} = \prod_{\substack{v \text{ finite } \\ v \nmid p}} \mathcal{O}_{F, v}$ and similarly for $\widehat{\mathcal{O}}_{F^+}^{(p)}$. Note that, in general, $F_{\ide{N}p^{\infty}}$ can be smaller than the ring class field $F[\ide{N}p^{\infty}]$. We define $F_{\ide{N}}/F$ in the same way, but replacing $\widehat{\mathcal{O}}_{F^+}^{(p)}$ and $\widehat{\mathcal{O}}_{F}^{(p)}$ with $\widehat{\mathcal{O}}_{F^+}$ and $\widehat{\mathcal{O}}_{F}$ respectively. We have an injective map
\begin{equation} \label{ProdZpstarIncl}
\prod_{\tau \in \Psi} \mbb{Z}_p^{\times} \cong \prod_{\tau \in \Psi} \mathcal{O}_{F_{\bar{\ide{p}}_{\tau}}}^{\times} \hookrightarrow C_{\ide{N}p^{\infty}}
\end{equation}
where the second map is induced from the natural inclusion of ideles, and this subgroup is identified with $\Gal(F_{\ide{N}p^{\infty}}/F_{\ide{N}})$. One has a natural map
\[
\jmath \colon \mathcal{W}(\ide{N}p^{\infty}) \to \mathcal{W}
\]
given by sending a continuous character $C_{\ide{N}p^{\infty}} \to R^{\times}$ to its restriction to $\prod_{\tau \in \Psi} \mbb{Z}_p^{\times}$ via (\ref{ProdZpstarIncl}). 

\begin{lemma}
    The map $\jmath$ is an \'{e}tale torsor under the character group scheme $X^*(\Gal(F_{\ide{N}}/F))$. In particular, the map $\jmath$ is finite \'{e}tale.
\end{lemma}
\begin{proof}
    Recall our convention is that multiplication of characters is written additively. The action of the group scheme is given by
    \begin{align*} 
    X^*(\Gal(F_{\ide{N}}/F))(S) \times \mathcal{W}(\ide{N}p^{\infty})(S) &\to \mathcal{W}(\ide{N}p^{\infty})(S) \\
    (\chi_1, \chi_2) \mapsto \chi_1 + \chi_2
    \end{align*}
    for an adic space $S \to \mathcal{W}$. This is clearly free and transitive, so we just need to show the map $\jmath$ has sections locally in the \'{e}tale topology. Let $S = \opn{Spa}(R, R^+) \hookrightarrow \mathcal{W}$ be an open affinoid, which corresponds to a continuous character $\chi \colon \prod_{\tau \in \Psi} \mbb{Z}_p^{\times} \to R^{\times}$. Consider the short exact sequence (in category $\opn{Ab}$ of abelian groups):
    \[
    0 \to \prod_{\tau \in \Psi} \mbb{Z}_p^{\times} \to C_{\ide{N}p^{\infty}} \to \Gal(F_{\ide{N}}/F) \to 0 .
    \]
    This gives rise to the exact sequence:
    \begin{equation} \label{Hom2Ext}
    \opn{Hom}_{\opn{Ab}}\left( C_{\ide{N}p^{\infty}}, R^{\times} \right) \to \opn{Hom}_{\opn{Ab}}\left( \prod_{\tau \in \Psi} \mbb{Z}_p^{\times}, R^{\times} \right) \to \opn{Ext}^1_{\opn{Ab}}\left(\Gal(F_{\ide{N}}/F), R^{\times} \right)
    \end{equation}
    which is functorial in $R$. Since $\Gal(F_{\ide{N}}/F)$ is finite abelian group, there exist positive integers $m_1, \dots, m_k$ such that
    \[
    \opn{Ext}^1_{\opn{Ab}}\left(\Gal(F_{\ide{N}}/F), R^{\times} \right) \cong \bigoplus_{i=1}^k R^{\times}/\left( R^{\times} \right)^{m_i} 
    \]
    functorially in $R$. The image of the continuous character $\chi$ under the second map in (\ref{Hom2Ext}) therefore gives rise to a class
    \[
    \opn{ob}(\chi) = ( x_i ) \in \bigoplus_{i=1}^k R^{\times}/\left( R^{\times} \right)^{m_i} .
    \]
    Let $T = R[X_1, \dots, X_k]/(X_1^{m_1} - x_1, \dots, X_k^{m_k} - x_k)$, and let $T^+ \subset T$ denote the integral closure of $R^+$ under the map $R \to T$. Then $\opn{Spa}(T, T^+) \to S$ is finite \'{e}tale, and the image of $\opn{ob}(\chi)$ under the map
    \[
    \bigoplus_{i=1}^k R^{\times}/\left( R^{\times} \right)^{m_i} \to \bigoplus_{i=1}^k T^{\times}/\left( T^{\times} \right)^{m_i}
    \]
    is zero. This implies that $\chi$ can be lifted to a homomorphism $C_{\ide{N}p^{\infty}} \to T^{\times}$. It is automatically continuous because its restriction $\prod_{\tau \in \Psi} \mbb{Z}_p^{\times} \to T^{\times}$ is continuous (and $\prod_{\tau \in \Psi} \mbb{Z}_p^{\times}$ is an open subgroup of $C_{\ide{N}p^{\infty}}$). This implies that $\jmath$ has a section over $\opn{Spa}(T, T^+)$ as required.
\end{proof}

We can view the characters introduced in Definition \ref{DefOfSigmaKJNB} as points on the adic space $\mathcal{W}(\ide{N}p^{\infty})$. More precisely, for any $(\kappa, j) \in \mathcal{E}$ satisfying Assumption \ref{AssumpOnKJforACchar}, there is an injective map
\[
\Sigma_{\kappa, j}(\ide{N}_{\beta}) \hookrightarrow \mathcal{W}(\ide{N}p^{\infty})(\mbb{C}_p), \quad \chi \mapsto \hat{\chi}
\]
where $\hat{\chi} \colon C_{\ide{N}p^{\infty}} \to \mbb{C}_p^{\times}$ is the continuous character defined by the formula:
\[
\hat{\chi}(z) = \iota_p(\chi(z)) \cdot z_{\ide{p}_{\tau_0}}^{\kappa_{n+1, \tau_0} - j_{\tau_0}} z_{\ide{p}_{\bar{\tau}_0}}^{j_{\tau_0}-\kappa_{n+1, \tau_0}} \cdot \prod_{\tau \neq \tau_0} z_{\ide{p}_{\tau}}^{- j_{\tau_0}} z_{\ide{p}_{\bar{\tau}}}^{j_{\tau_0}}, \quad \quad z \in \mbb{A}_{F, f}^{\times} .
\]
Here, for an embedding $\sigma \colon F \hookrightarrow \mbb{C}$, $z_{\ide{p}_{\sigma}}$ denotes the component of $z$ at the prime $\ide{p}_{\sigma}$ above $p$ determined by the embedding $F \xrightarrow{\sigma} \mbb{C} \xrightarrow{\iota_p} \Qpb$. 

We let $\mathcal{W}(\ide{N}p^{\infty})_h = \jmath^{-1}(\mathcal{W}_h)$ which, by the above lemma, is a quasi-compact open affinoid subspace. Let $L/\mbb{Q}_p$ be a finite extension containing $\mbb{Q}_p(\mu_{p^{\beta}})$ and $\iota_p(F_{\ide{N}})$, and let $\Omega = \opn{Spa}(\mathscr{O}_{\Omega}, \mathscr{O}_{\Omega}^+)$ be the adic spectrum of a Tate affinoid adic space over $(L, \mathcal{O}_L)$. For $h \geq 1$, we set
\[
\Omega_h \defeq \Omega \times_{\mbb{Q}_p} \mathcal{W}(\ide{N}p^{\infty})_h = \opn{Spa}(\mathscr{O}_{\Omega} \hatot_{\mbb{Q}_p} \mathscr{O}_{\mathcal{W}(\ide{N}p^{\infty})_h}) = \opn{Spa}(\mathscr{O}_{\Omega_h}).
\]
Let $s \geq \opn{max}(h, \beta)$ and let $\kappa \colon T(\mbb{Z}_p) \to \mathscr{O}_{\Omega}^{\times}$ be an $s$-analytic character satisfying Assumption \ref{AssumpOnKJforACchar}. Let $\chi \colon \Gal(F_{\ide{N}p^{\infty}}/F) \to \mathscr{O}_{\mathcal{W}(\ide{N}p^{\infty})_h}^{\times}$ denote the universal character. We let $\tilde{\kappa}$ and $\tilde{\chi}$ denote the characters $\kappa$ and $\chi$ respectively, viewed as homomorphisms valued in $\mathscr{O}_{\Omega_h}^{\times}$. Note that $\tilde{\kappa}$ and $\jmath(\tilde{\chi})$ are both $s$-analytic. We let
\begin{align*} 
j \colon \prod_{\tau \in \Psi} \mbb{Z}_p^{\times} &\to \mathscr{O}_{\Omega_h}^{\times} \\
(x_{\tau}) &\mapsto \tilde{\kappa}_{n+1, \tau_0}(x_{\tau_0})\prod_{\tau \in \Psi} \jmath(\tilde{\chi})(x_{\tau})
\end{align*} 
which is $s$-analytic.

We introduce some notation.

\begin{notation} \label{NotationForSigmaSigmaPrime}
    Let $\Sigma'_{\beta} \subset \Omega_h(\mbb{C}_p)$ denote a subset satisfying the following properties:
    \begin{enumerate}
        \item Let $\kappa_x \colon T(\mbb{Z}_p) \to \mbb{C}_p^{\times}$ and $j_x \colon \prod_{\tau \in \Psi} \mbb{Z}_p^{\times} \to \mbb{C}_p^{\times}$ denote the specialisations of $\tilde{\kappa}$ and $j$ at $x \in \Sigma'_{\beta}$ respectively. We assume that there exists a finite-order character $\chi_{x, p} \colon \prod_{\tau \in \Psi} \mbb{Z}_p^{\times} \to \mbb{C}_p^{\times}$ which is trivial on $\prod_{\tau \in \Psi} (1 + p^{\beta}\mbb{Z}_p)$ such that $(\kappa_x, j_x - \chi_{x, p}) \in \mathcal{E}$.
        \item We assume that the projection of $x \in \Sigma'_{\beta}$ to a point in $\mathcal{W}(\ide{N}p^{\infty})(\mbb{C}_p)$ lies in $\Sigma_{\kappa_x, j_x - \chi_{x, p}}(\ide{N}_{\beta})$. We let $\chi_x$ denote the corresponding character (so that $\hat{\chi}_x$ is equal to the specialisation of $\chi$ at $x$).
    \end{enumerate}
    We also let $\Sigma \subset \Sigma'_{\beta}$ denote a subset such that $\chi_x$ has conductor dividing $\ide{N}$ for any $x \in \Sigma$ (which implies that $\chi_{x, p}$ is trivial). Note that one automatically has $\Sigma \subset \Omega_1(\mbb{C}_p)$.
\end{notation}

\begin{lemma} \label{TheUniversalChiCharacterLemma}
    There exists a cohomology class $[\chi] \in \opn{H}^0_{\opn{id}, \opn{an}}\left( \mathcal{S}_{H, \diamondsuit}(p^{\beta}), \sigma_{\tilde{\kappa}}^{[j], \vee} \right)^{(+, \dagger)}$ such that: for any $x \in \Sigma'_{\beta}$, the image of $[\chi]$ under the specialisation map
    \[
    \opn{H}^0_{\opn{id}, \opn{an}}\left( \mathcal{S}_{H, \diamondsuit}(p^{\beta}), \sigma_{\tilde{\kappa}}^{[j], \vee} \right)^{(+, \dagger)} \to \opn{H}^0_{\opn{id}}\left( \mathcal{S}_{H, \diamondsuit}(p^{\beta}), \sigma_{\kappa_x}^{[j_x], \vee} \right)^{(+, \dagger)}
    \]
    is equal to $\opn{res}[\chi_x]$. If $\Sigma'_{\beta}$ is Zariski dense in $\Omega_h$, then $[\chi]$ is uniquely determined by this specialisation property.
\end{lemma}
\begin{proof}
    This follows from the construction in \cite[\S 7]{UFJ}. More precisely, let $(\mbf{R}, h_{\mbf{R}})$ denote the PEL Shimura datum as in \S 7.1 in \emph{op.cit.}, and let $\mathcal{R}_s \subset R^{\opn{an}}$ denote the affinoid subgroup
    \[
    \mathcal{R}_s = \mbb{Z}_p^{\times}\left( 1 + \mathcal{B}_s \right) \times \prod_{\tau \in \Psi} \left( \mbb{Z}_p^{\times}\left( 1 + \mathcal{B}_s \right) \times \mbb{Z}_p^{\times}\left( 1 + \mathcal{B}_s \right) \right) .
    \]
    Let $\Delta$ denote the adic Shimura variety over $\mbb{Q}_p$ associated with $(\mbf{R}, h_{\mbf{R}})$ of level $C \defeq \opn{det}(U) \subset \mbf{R}(\mbb{A}_f)$, where $U \subset \mbf{H}(\mbb{A}_f)$ denotes the level of $\mathcal{S}_{H, \diamondsuit}(p^{\beta})$ and $\opn{det} \colon \mbf{H} \to \mbf{R}$ is determinant map in \S 7.1 of \emph{op.cit.}. The compact open subgroup $C$ decomposes as $C = C^p C_p$ with $C^p \subset \mbf{R}(\mbb{A}_f^p)$ and $C_p \subset \mbf{R}(\mbb{Q}_p)$. By Shimura reciprocity and the fact that $L$ contains $\mbb{Q}_p(\mu_{p^{\beta}})$, we see that there is a finite unramified extension $L' / \mbb{Q}_p$ such that
    \[
    \Delta_{\Phi} = \mbf{R}(\mbb{Q}) \backslash \mbf{R}(\mbb{A}_f) / C
    \]
    where $\Phi = L \cdot L' \subset \Qpb$. If $F_p \in \Gal(\mbb{Q}_p^{\opn{ab}}/\mbb{Q}_p)$ denotes the geometric Frobenius corresponding $p \in \mbb{Q}_p^{\times}$ under the Artin reciprocity map of local class field theory, then (after choosing a lift to $\Gal(\Qpb /\mbb{Q}_p)$) there is an integer $m \geq 1$ such that $F^m_p$ generates the cyclic Galois group $\Gal(\Phi / L)$. By Shimura reciprocity, the action of $F^m_p$ on $[x] \in \Delta(\Phi)$ (with $x \in \mbf{R}(\mbb{A}_f)$) is given by
    \[
    F_p^m \cdot [x] = [y \cdot x]
    \]
    where $y \in \mbf{R}(\mbb{A}_f)$ denotes the point which is trivial outside the $p$-component, and at the $p$-component is given by $y_p =(y_0; y_{1, \tau}, y_{2, \tau})_{\tau \in \Psi}$ with $y_0 = p^m$, $y_{1, \tau_0} = p^m$, $y_{2, \tau_0} = 1$, and $y_{i, \tau} = 1$ for all $i=1, 2$ and $\tau \neq \tau_0$.

    The torsor $R_{\opn{dR}}^{\opn{an}}$ has a reduction of structure ${^\mu \mathcal{R}_{\opn{HT},s}}$ which can similarly be described as 
    \[
    {^\mu \mathcal{R}_{\opn{HT},s, \Phi}} = \left[ \mbf{R}(\mbb{Q}) \backslash \left( \mbf{R}(\mbb{A}_f^p)/C^p \times \mbf{R}(\mbb{Q}_p) \right) \times \mathcal{R}_{s, \Phi} \right] / C_p
    \]
    where $C_p$ acts diagonally with $C_p$ acting on $\mathcal{R}_{s, \Phi}$ through the map $C_p \to \mathcal{R}_{s, \Phi}$, $z \mapsto z^{-1}$. The action of $F_p^m$ on $[x, x', x''] \in {^\mu \mathcal{R}_{\opn{HT},s}}(\Phi)$ (with $x \in \mbf{R}(\mbb{A}_f^p)$, $x' \in \mbf{R}(\mbb{Q}_p)$, $x'' \in \mathcal{R}_s(\Phi)$) is given by
    \[
    F_p^m \cdot [x, x', x''] = [x, y_p \cdot x', F_p^m(x'')]
    \]
    where $F_p^m(x'')$ denotes the natural Galois action.

    Let $\lambda \colon \mbf{R}(\mbb{Q})\backslash \mbf{R}(\mbb{A}_f) \to \mathscr{O}_{\Omega_h}^{\times}$ denote the unique character such that $\lambda(z_1, z_2) = \tilde{\chi}'(z_2/z_1)$, where $\tilde{\chi}'$ denotes the unique $p$-adic Hecke character on $\opn{Res}_{F^+/\mbb{Q}}\opn{U}(1)$ satisfying $\tilde{\chi} = \tilde{\chi}' \circ \mathcal{N}$. Let $f \colon {^\mu \mathcal{R}_{\opn{HT},s, \Omega_{h,\Phi}}} \to \mbb{A}^{1, \opn{an}}_{\Omega_{h,\Phi}}$ denote the (well-defined) function given by
    \[
    f([x, x', x'']) = \lambda(x)\lambda(x')\lambda(x'')
    \]
    which makes sense because $\lambda$ restricted to $R(\mbb{Z}_p)$ is $s$-analytic (because $\jmath(\chi)$ is). Then $f$ corresponds to a class $[\lambda] \in \opn{H}^0( \Delta_{\Omega_{h, \Phi}}, \mathcal{F}_{\lambda^{-1}} )$, where $\mathcal{F}_{\lambda^{-1}}$ denotes the line bundle associated with the character $\lambda^{-1} \colon \mathcal{R}_s \to \mbb{G}_m^{\opn{an}}$. Note $\lambda(y_p) = 1$ because $L$ contains $\iota_p(F_{\ide{N}})$ and $\tilde{\chi}$ is trivial on $F^{\times}\mbb{A}_{F^+, f}^{\times} (\widehat{\mathcal{O}}^{(p)}_{F^+} + \ide{N} \widehat{\mathcal{O}}_F^{(p)} )^{\times}$. By the description of the Galois actions above, we therefore see that $[\lambda]$ descends to a class $[\lambda] \in \opn{H}^0( \Delta_{\Omega_{h}}, \mathcal{F}_{\lambda^{-1}} )$. As explained in \cite[\S 7.3]{UFJ}, there is a natural map
    \[
    \opn{det} \colon {^\mu \mathcal{M}_{H, \opn{HT}, s}} \to {^\mu \mathcal{R}_{\opn{HT}, s}}
    \]
    and we define $[\chi] \defeq \opn{det}^*[\lambda] \in \opn{H}^0_{\opn{id}, \opn{an}}\left( \mathcal{S}_{H, \diamondsuit}(p^{\beta}), \sigma_{\tilde{\kappa}}^{[j], \vee} \right)^{(+, \dagger)}$ (note that $\sigma_{\tilde{\kappa}}^{[j], \vee} = \lambda^{-1} \circ \opn{det}$). One can easily verify that this class interpolates $\opn{res}[\chi_x]$ for $x \in \Sigma'_{\beta}$.

    Recall that ${^\mu \mathcal{M}_{H, \opn{HT}, s}}$ is a torsor over the affinoid $\mathcal{S}_{H, \opn{id}}(p^{\beta})_s$ (the pullack of $\mathcal{X}_{H, \opn{id}}(p^{\beta})_s$ along the open and closed embedding in Lemma \ref{LemmaSDintoPEL}), and note that the pullback of ${^\mu \mathcal{R}_{\opn{HT}, s}}$ along the map $\opn{det} \colon \mathcal{S}_{H, \opn{id}}(p^{\beta})_s \to \Delta$ is the pushout of ${^\mu \mathcal{M}_{H, \opn{HT}, s}}$ along $\opn{det} \colon \mathcal{M}_{H, s}^{\diamondsuit} \to \mathcal{R}_s$. Fix a set $Z = \{ z \in \mbf{R}(\mbb{A}_f) \}$ of representatives of $\Delta_{\Phi}$, and let $\mathcal{S}_{H, \opn{id}}(p^{\beta})_{s, [z]}$ denote the preimage of $[z] \in \Delta_{\Phi}$ under $\opn{det} \colon \mathcal{S}_{H, \opn{id}}(p^{\beta})_{s, \Phi} \to \Delta_{\Phi}$. Then we have an identification
    \begin{align}
    \opn{H}^0( \mathcal{S}_{H, \opn{id}}(p^{\beta})_{s, \Omega_{h,\Phi}}, \sigma_{\tilde{\kappa}}^{[j], \vee}) &= \opn{H}^0(\Delta_{\Omega_{h,\Phi}}, (\opn{det})_*\mathcal{O}_{\mathcal{S}_{H, \opn{id}}(p^{\beta})_{s, \Phi}} \hatot \mathcal{F}_{\lambda^{-1}}) \nonumber \\
     &= \bigoplus_{z \in Z} \opn{H}^0(\mathcal{S}_{H, \opn{id}}(p^{\beta})_{s, [z]}, \mathcal{O}_{\mathcal{S}_{H, \opn{id}}(p^{\beta})_{s, [z]}}) \; \hatot_{\Phi} \; \mathscr{O}_{\Omega_{h, \Phi}} \label{Eqn:ConnectedCompSHdaggerSpace}
    \end{align}
    given by evaluating a section of ${^\mu \mathcal{R}_{\opn{HT}, s, \Omega_{h, \Phi}}}$ at $[z^p, z_p, 1]$ for $z \in Z$ (for brevity, we are writing $\sigma_{\tilde{\kappa}}^{[j], \vee}$ instead of $\mathscr{M}^{\opn{an}}_{H, \sigma_{\tilde{\kappa}}^{[j], \vee}}$). Here we are using the fact that $\mathcal{S}_{H, \opn{id}}(p^{\beta})_s$ is affinoid. The specialisation map at some $x \in \Sigma_{\beta}'$ is identified (via \eqref{Eqn:ConnectedCompSHdaggerSpace}) with the map
    \begin{equation} \label{Eqn:SpecMapDirectSums}
    \bigoplus_{z \in Z} \opn{H}^0(\mathcal{S}_{H, \opn{id}}(p^{\beta})_{s, [z]}, \mathcal{O}_{\mathcal{S}_{H, \opn{id}}(p^{\beta})_{s, [z]}}) \; \hatot_{\Phi} \; \mathscr{O}_{\Omega_{h, \Phi}} \to \bigoplus_{z \in Z} \opn{H}^0(\mathcal{S}_{H, \opn{id}}(p^{\beta})_{s, [z]}, \mathcal{O}_{\mathcal{S}_{H, \opn{id}}(p^{\beta})_{s, [z]}}) \; \hatot_{\Phi} \; \mbb{C}_p
    \end{equation}
    induced from the specialisation map $\mathscr{O}_{\Omega_{h, \Phi}} \to \mbb{C}_p$. Since $\opn{H}^0(\mathcal{S}_{H, \opn{id}}(p^{\beta})_{s, [z]}, \mathcal{O}_{\mathcal{S}_{H, \opn{id}}(p^{\beta})_{s, [z]}})$ is a projective Banach $\Phi$-module, if we assume that $\Sigma_{\beta}'$ is Zariski dense, then an element of the left-hand side of \eqref{Eqn:SpecMapDirectSums} is uniquely determined by its specialisations at $x \in \Sigma_{\beta}'$. Hence any section of $\opn{H}^0( \mathcal{S}_{H, \opn{id}}(p^{\beta})_{s, \Omega_{h,\Phi}}, \sigma_{\tilde{\kappa}}^{[j], \vee})$ is uniquely determined by its specialisations at points in $\Sigma_{\beta}'$, implying the uniqueness claim of the lemma.
\end{proof}

We now define the overconvergent evaluation map in this setting.

\begin{definition}
    With notation as above, suppose that $\Sigma$ is Zariski dense in $\Omega_h$ (or equivalently, Zariski dense in $\Omega_1$). We define the following $\mathscr{O}_{\Omega}$-linear map
    \begin{align*} 
    \opn{Ev}^{\dagger, s\opn{-an}}_{\kappa, h, \beta} \colon \opn{H}^{n-1}_{w_n, s\opn{-an}}\left( \kappa^* ; \beta \right)^{(-, \dagger)} &\to \mathscr{O}_{\Omega_h} \\
    \eta &\mapsto \langle \vartheta^{\dagger, s\opn{-an}}_{\tilde{\kappa}, j, \beta}(\tilde{\eta}), [\chi] \rangle
    \end{align*} 
    where $\tilde{\eta}$ denotes the pullback of $\eta$ to a class in $\opn{H}^{n-1}_{w_n, s\opn{-an}}\left( \tilde{\kappa}^* ; \beta \right)^{(-, \dagger)}$ and $[\chi]$ denotes the class in Lemma \ref{TheUniversalChiCharacterLemma}. Note that the assumptions in Proposition \ref{DaggerThetaTakesOCtoOCProp}(3) are satisfied.
\end{definition}

\subsubsection{Compatibilities} \label{CompatibilitiesOfEvMapSSEC}

One can easily verify that $\opn{Ev}^{\dagger, s\opn{-an}}_{\kappa, h, \beta}$ is compatible with changing $s$ and $h$. In particular, by passing to the limit over $s$ and $h$, one obtains an $\mathscr{O}_{\Omega}$-linear map
\[
\opn{Ev}^{\dagger, \opn{la}}_{\kappa, \beta} \colon \opn{H}^{n-1}_{w_n, \opn{la}}\left( \kappa^* ; \beta \right)^{(-, \dagger)} \defeq \varprojlim_s \opn{H}^{n-1}_{w_n, s\opn{-an}}\left( \kappa^* ; \beta \right)^{(-, \dagger)} \to \mathscr{D}^{\opn{la}}\left( \Gal(F_{\ide{N}p^{\infty}}/F), \mathscr{O}_{\Omega} \right)
\]
where $\mathscr{D}^{\opn{la}}\left( \Gal(F_{\ide{N}p^{\infty}}/F), \mathscr{O}_{\Omega} \right) = \varprojlim_h \mathscr{O}_{\Omega_h}$ denotes the $\mathscr{O}_{\Omega}$-module of locally analytic distributions on $\Gal(F_{\ide{N}p^{\infty}}/F)$. Furthermore, one has the interpolation property:
\[
\opn{Ev}^{\dagger, s\opn{-an}}_{\kappa, h, \beta}(\eta)(x) = \opn{Ev}^{\dagger, \circ}_{\kappa_x, j_x - \chi_{x, p}, \chi_x, \beta}(\eta_x)
\]
for any $x \in \Sigma'_{\beta}$, where $\eta_x$ denotes the image of $\eta$ under the specialisation map $\opn{H}^{n-1}_{w_n, s\opn{-an}}\left( \kappa^* ; \beta \right)^{(-, \dagger)} \to \opn{H}^{n-1}_{w_n}\left( \kappa_x^* ; \beta \right)^{(-, \dagger)}$. 

%-----------------------------------------

\section{Hecke operators and higher Coleman theory} \label{HOpsAndHCTChapter}

The construction in the previous section will be used for the $p$-adic $L$-function. However to find the suitable test data and prove the interpolation property, we need to understand the action of Hecke operators on the space of overconvergent forms. 

\subsection{Notations for the flag variety}

Consider the following flag varieties $\mathtt{FL}^G = P_G^{\opn{an}} \backslash G^{\opn{an}} = \mathcal{P}_G \backslash \mathcal{G}$ and $\mathtt{FL}^H = P_H^{\opn{an}} \backslash H^{\opn{an}} = \mathcal{P}_H \backslash \mathcal{H}$ which are adic spaces over $\opn{Spa}(\mbb{Q}_p, \mbb{Z}_p)$. Recall from Definition \ref{HyperspecialVarietiesDefinition} that we have fixed neat compact open subgroups $K^p \subset \mbf{G}(\mbb{A}_f^p)$ and $U^p \subset \mbf{H}(\mbb{A}_f^p)$. To simplify notation in this section, we will let $\mathcal{S}_{G, K_p}$ (resp. $\mathcal{S}_{H, K_p}$) denote the adic Shimura variety for $\mbf{G}$ (resp. $\mbf{H}$) over $\opn{Spa}(\mbb{Q}_p, \mbb{Z}_p)$ of level $K^pK_p \subset \mbf{G}(\mbb{A}_f)$ (resp. $U^pK_p \subset \mbf{H}(\mbb{A}_f)$) for any compact open subgroup $K_p \subset G(\mbb{Q}_p)$ (resp. $K_p \subset H(\mbb{Q}_p)$). In particular, $\mathcal{S}_{G, \opn{Iw}}(p^{\beta}) = \mathcal{S}_{G, K_p}$ for $K_p = K^G_{\opn{Iw}}(p^{\beta})$, and similarly for $\mathcal{S}_{H, \diamondsuit}(p^{\beta})$.

As usual, let $\beta \geq 1$ be an integer and $L/\mbb{Q}_p$ a finite extension containing $\mu_{p^{\beta}}$. We will often consider the base-change of the above flag varieties and Shimura varieties to $\opn{Spa}(L, \mathcal{O}_L)$ but omit this from the notation. Recall that for any compact open $K_p \subset G(\mbb{Z}_p)$ (resp. $K_p \subset H(\mbb{Z}_p)$) we have truncated Hodge--Tate period maps (which we simply view as maps of topological spaces):
\begin{align*}
    \pi_{\opn{HT}, G, K_p} \colon \mathcal{S}_{G, K_p} &\to \mathtt{FL}^G/K_p \\
    \text{(resp. }\pi_{\opn{HT}, H, K_p} \colon \mathcal{S}_{H, K_p} &\to \mathtt{FL}^H/K_p \text{ )}
\end{align*}
which we will often denote simply by $\pi_{\opn{HT}, K_p}$ when the context is clear.

Recall from \cite[Definition 3.0.1]{UFJ} that $\mathtt{FL}^G \cong \mbb{P}^{2n-1}$ and in coordinates the right-action of $g \in G^{\opn{an}}$ (denoted $\star$) is given by
\[
[x_0 : x_1 : \cdots : x_{2n-1}] \star g = [x_0 : x_1 : \cdots : x_{2n-1}] \cdot {^t g_{\tau_0}^{-1}}
\]
where the right hand side is the usual action on row vectors and $g_{\tau_0} \in \opn{GL}_{2n}^{\opn{an}}$ denotes the $\tau_0$-component of $g$. We have a similar description for $\mathtt{FL}^H$. Finally, we let $]C^?_{w}[_{m, k} \subset \mathtt{FL}^?$ (and their partially compactified versions) denote the tubes of Bruhat cells as in \cite[Definition 3.2.3]{UFJ}.

\begin{definition} \label{DefinitionOfCurlyQm}
    For an integer $m \geq 1$, let $\mathcal{Q}_m \subset \mathcal{G}$ denote the subgroup which coincides with $\mathcal{P}_G$ outside the $\tau_0$-component, and in the $\tau_0$-component is equal to the subgroup of block matrices
    \[
    \tbyt{A}{B}{C}{D}
    \]
    with $A \in \mathcal{GL}_{n+1}$, $B \in \mathcal{M}_{(n+1) \times (n-1)}$, $C \in \mathcal{B}_m^{\circ} \mathcal{M}_{(n-1) \times (n+1)}$ and $D \in \mathcal{GL}_{n-1}$. Here $\mathcal{M}_{r \times s}$ denotes the adic group scheme with $\mathcal{M}_{r \times s}(\opn{Spa}(A, A^+))$ equal to the group of $r \times s$-matrices with coefficients in $A^+$, and $\mathcal{B}_m^{\circ}$ is the ``open disc'' defined in \cite[Definition 3.2.1]{UFJ}.
\end{definition}

We have the following useful lemma.

\begin{lemma}
    For any integer $k \geq 1$ and compact open subgroup $K_p \subset G(\mbb{Z}_p)$ (resp. $K_p \subset H(\mbb{Z}_p)$), set $\mathcal{U}^G_{K_p, k} \defeq \pi_{\opn{HT}, G, K_p}^{-1}\left( ]C^G_{w_n}[_{k, k} K_p \right) \subset \mathcal{S}_{G, K_p}$ (resp. $\mathcal{U}^H_{K_p, k} \defeq \pi_{\opn{HT}, H, K_p}^{-1}\left( ]C^H_{\opn{id}}[_{k, k} K_p \right) \subset \mathcal{S}_{H, K_p}$). Then
    \begin{enumerate}
        \item For $K_p = K^G_{\opn{Iw}}(p^{\beta})$, we have 
        \[
        \overline{\mathcal{S}_{G, w_n}(p^{\beta})} = \bigcap_{k \geq 1} \mathcal{U}^G_{K_p, k}
        \]
        where $\overline{\mathcal{S}_{G, w_n}(p^{\beta})}$ denotes the closure of $\mathcal{S}_{G, w_n}(p^{\beta}) \defeq \mathcal{X}_{G, w_n}(p^{\beta})$ inside $\mathcal{S}_{G, \opn{Iw}}(p^{\beta}) = \mathcal{X}_{G, \opn{Iw}}(p^{\beta})$. In particular $\{ \mathcal{U}^G_{K_p, k} \}_{k \geq 1}$ is a cofinal system of (quasi-Stein) open neighbourhoods of $\overline{\mathcal{S}_{G, w_n}(p^{\beta})}$.
        \item Let $K_p = K^H_{\diamondsuit}(p^{\beta})$ and let $\mathcal{S}_{H, \opn{id}}(p^{\beta})$ denote the pullback of $\mathcal{X}_{H, \opn{id}}(p^{\beta})$ along the open and closed embedding $\mathcal{S}_{H, \diamondsuit}(p^{\beta}) \subset \mathcal{X}_{H, \diamondsuit}(p^{\beta})$ (see Lemma \ref{LemmaSDintoPEL}). Then
        \[
        \overline{\mathcal{S}_{H, \opn{id}}(p^{\beta})} = \bigcap_{k \geq 1} \mathcal{U}^H_{K_p, k}
        \]
        where $\overline{\mathcal{S}_{H, \opn{id}}(p^{\beta})}$ denotes the closure of $\mathcal{S}_{H, \opn{id}}(p^{\beta})$ inside $\mathcal{S}_{H, \diamondsuit}(p^{\beta})$. In particular $\{ \mathcal{U}^H_{K_p, k} \}_{k \geq 1}$ is a cofinal system of (quasi-Stein) open neighbourhoods of $\overline{\mathcal{S}_{H, \opn{id}}(p^{\beta})}$.
        \item Let $K_p = K^G_{\opn{Iw}}(p^{\beta})$. Then we have
        \[
        \mathcal{Z}_{G, >n+1}(p^{\beta}) = \bigcap_{m \geq 1} \pi^{-1}_{\opn{HT}, G, K_p}\left( \mathcal{P}_G \backslash \mathcal{P}_G \mathcal{Q}_m K_p \right) .
        \]
    \end{enumerate}
\end{lemma}
\begin{proof}
    Let $x \colon \opn{Spa}(F, \mathcal{O}_F) \to \mathcal{S}_{G, \opn{Iw}}(p^{\beta})$ be a rank one point. Then $\pi_{\opn{HT},G, K_p}(x)$ can be described as follows. The point $x$ corresponds to a tuple $(A, \lambda, i, \eta^p)$ (a $\boldsymbol\Psi$-unitary abelian scheme over $F$ with $K^p$-level structure) and flags
    \[
    0 = C_{0, \tau} \subset C_{1, \tau} \subset \cdots \subset C_{2n, \tau} = A[\ide{p}_{\tau}^{\beta}], \quad \quad \tau \in \Psi 
    \]
    of finite-flat group schemes satisfying the conditions in Definition \ref{DefinitionOfDeeperLevelAtpVarieties}(1). Assume without loss of generality that $F = \widehat{F^{\opn{sep}}}$ is equal to the completion of its separable closure. We know that $(A, \lambda, i, \eta^p)$ extends to a $\boldsymbol\Psi$-unitary abelian scheme $\mathcal{A}$ over $\mathcal{O}_F$. Let $\mathcal{C}_{i, \tau}$ denote the Zariski closure of $C_{i, \tau}$ inside $\mathcal{A}[\ide{p}_{\tau}^{\beta}]$ (which are finite flat). Let $T_{\ide{p}_{\tau}}(A) = T_p(A[\ide{p}_{\tau}^{\infty}]) = \varprojlim_{\beta' \geq 1}\mathcal{A}[\ide{p}_{\tau}^{\beta'}](F) = \varprojlim_{\beta' \geq 1}\mathcal{A}[\ide{p}_{\tau}^{\beta'}](\mathcal{O}_F)$ denote the corresponding Tate modules, and set $V_{\ide{p}_{\tau}}(A) = T_{\ide{p}_{\tau}}(A) \otimes_{\mbb{Z}_p} \mbb{Q}_p$. Since $F$ is separably closed, we can choose bases $\{ e_{1, \tau}, \dots, e_{2n, \tau} \}$ of $T_{\ide{p}_{\tau}}(A)$ which are compatible with the flags $C_{\bullet, \tau}$ modulo $p^{\beta}$ (i.e. $e_{1, \tau}$ modulo $p^{\beta}$ generates $C_{1, \tau}(F) = \mathcal{C}_{1, \tau}(\mathcal{O}_F)$ etc.). Any other such choice of basis differs from this one by elements of the depth $p^{\beta}$ upper-triangular Iwahori subgroups of $\opn{GL}_{2n}(\mbb{Z}_p) = \opn{Aut}_{\mbb{Z}_p}(T_{\ide{p}_{\tau}}(A))$.

    On the other hand, for any $\beta' \geq 1$, we have ``dlog'' morphisms 
    \[
    \opn{dlog}_{\beta', \tau} \colon \mathcal{A}[\ide{p}_{\tau}^{\beta'}](\mathcal{O}_F) \to \omega_{\mathcal{A}^D, \tau}/p^{\beta'}
    \]
    which are compatible as $\beta'$ varies. Hence, we obtain morphisms $\opn{dlog}_{\tau} \colon T_{\ide{p}_{\tau}}(A) \to \omega_{\mathcal{A}^D, \tau}$ and 
    \[
    \opn{HT}_{\tau} \colon V_{\ide{p}_{\tau}}(A) \otimes_{\mbb{Q}_p} F = T_{\ide{p}_{\tau}}(A) \otimes_{\mbb{Z}_p} F \xrightarrow{\opn{dlog}_{\tau} \otimes 1} \omega_{\mathcal{A}^D, \tau} \otimes_{\mathcal{O}_F} F = \omega_{A^D, \tau} .
    \]
    The morphisms $\opn{HT}_{\tau}$ are surjective and $F$-linear, and isomorphisms when $\tau \neq \tau_0$ by the signature condition. Let $\{ f_{1, \tau}, \dots, f_{2n, \tau} \}$ be a $F$-basis of $V_{\ide{p}_{\tau}}(A) \otimes_{\mbb{Q}_p} F$ respecting the filtration given by $\opn{HT}_{\tau}$ (i.e. we require that $\opn{HT}_{\tau_0}(f_{2, \tau_0}), \dots, \opn{HT}_{\tau_0}(f_{2n, \tau_0})$ generate $\omega_{A^D, \tau_0}$). Any other choice of basis differs by an element of $P_G(F)$ (ignoring the similitude factor). Then $\pi_{\opn{HT}, G, K_p}(x) \in \mathtt{FL}^G(F)/K_p = P_G(F) \backslash G(F) / K_p$ is represented by the element $g = (1; g_{\tau}) \in G(F) = F^{\times} \times \prod_{\tau \in \Psi} \opn{GL}_{2n}(F)$ satisfying
    \[
    (e_{1, \tau}, \dots, e_{2n, \tau}) = ( f_{1, \tau}, \dots, f_{2n, \tau} ) \cdot g_{\tau} 
    \]
    for all $\tau \in \Psi$.

    We now prove parts (1) and (3). Firstly, if $x \in \mathcal{S}_{G, w_n}(p^{\beta})$ then we can choose the bases $\{ f_{i, \tau} \}$ above such that $f_{i, \tau_0} = e_{i-1, \tau_0}$ for $i=2, \dots, n+1$, $f_{1, \tau_0} = e_{n+1, \tau_0}$ and $f_{i, \tau_0} = e_{i, \tau_0}$ for $i=n+2, \dots, 2n$ (because $\mathcal{C}_{n, \tau_0}$ is \'{e}tale and $\mathcal{C}_{n+1, \tau_0}/\mathcal{C}_{n, \tau_0}$ is multiplicative). Hence, we see that $\pi_{\opn{HT}, G, K_p}(x)$ is represented by $w_n$. On the other hand, suppose $\pi_{\opn{HT}, G, K_p}(x)$ is represented by $w_n$; so we can find bases above such that $g = w_n$. This means that $f_{i, \tau} = e_{i, \tau}$ for all $(i, \tau)$ with $\tau \neq \tau_0$, and $f_{i, \tau_0} = e_{i-1, \tau_0}$ for $i=2, \dots, n+1$, $f_{1, \tau_0} = e_{n+1, \tau_0}$ and $f_{i, \tau_0} = e_{i, \tau_0}$ for $i=n+2, \dots, 2n$. Then we see that 
    \[
    \{ \opn{dlog}_{\tau_0}(e_{1, \tau_0}), \dots, \opn{dlog}_{\tau_0}(e_{n, \tau_0}) \}
    \]
    give elements of $\omega_{\mathcal{C}_{n, \tau_0}^D}$ and their image in $\omega_{\mathcal{A}^D, \tau_0}/p^{\beta}$ span a free $\mathcal{O}_F/p^{\beta}$-subspace of rank $n$. This can only happen if $\opn{deg}(\mathcal{C}_{n, \tau_0}^D) = n \beta$, which coincides with its rank. This implies $\mathcal{C}_{n, \tau_0}$ is \'{e}tale. Similarly, the fact that $\opn{HT}_{\tau_0}(e_{n+1, \tau_0}) = \opn{HT}_{\tau_0}(f_{1, \tau_0}) = 0$ implies that $\mathcal{C}_{n+1, \tau_0}/\mathcal{C}_{n, \tau_0}$ is multiplicative. Putting this altogether, we see that
    \[
    \mathcal{S}_{G, w_n}(p^{\beta})^{\opn{rk} 1} = \pi_{\opn{HT}, G, K_p}^{-1}\left( \mathcal{P}_G \backslash \mathcal{P}_G \cdot w_n \cdot K_p \right)^{\opn{rk} 1} = \bigcap_{k \geq 1} \left(\mathcal{U}^G_{K_p, k}\right)^{\opn{rk} 1} .
    \]
    The claim in part (1) easily follows from this.

    We now prove part (3). Suppose that $x \in \mathcal{Z}_{G, > n+1}(p^{\beta})^{\opn{rk} 1}$. Then, since the $p$-rank of $\mathcal{A}[\ide{p}_{\tau_0}]^D$ must be $\geq n-1$, we can find a $p$-divisible group $\mathcal{D}_{n+1, \tau_0} \subset \mathcal{A}[\ide{p}_{\tau_0}^{\infty}]$ with $\mathcal{D}_{n+1, \tau_0}[p^{\beta'}] = \mathcal{C}_{n+1, \tau_0}$ and $\mathcal{E} \defeq \mathcal{A}[\ide{p}_{\tau_0}^{\infty}]/\mathcal{D}_{n+1, \tau_0}$ \'{e}tale. This implies that we have a quotient $\omega_{A^D, \tau_0} \twoheadrightarrow \omega_{\mathcal{E}^D} \otimes_{\mathcal{O}_F} F$ of rank $n-1$, and hence we may take $f_{i, \tau_0} = e_{i, \tau_0}$ for $i=n+2, \dots, 2n$. Furthermore, we may take $f_{i, \tau_0}$ ($i=1, \dots, n+1$) to be in the $\mathcal{O}_F$-span of $\{ e_{1, \tau_0}, \dots, e_{n+1, \tau_0} \}$. This implies that $g \in \bigcap_{m \geq 1}\mathcal{Q}_m(F)$. On the other hand, suppose that matrix $g$ representing $\pi_{\opn{HT}, G, K_p}(x)$ is in $\bigcap_{m \geq 1}\mathcal{Q}_m(F)$. Then we see that $f_{i, \tau}$ ($i=1, \dots, n+1$) is in the $\mathcal{O}_F$-span of $\{ e_{1, \tau_0}, \dots, e_{n+1, \tau_0} \}$. This implies that $\{ \opn{dlog}_{\tau_0}(e_{1, \tau_0}), \dots, \opn{dlog}_{\tau_0}(e_{n+1, \tau_0}) \}$ generate a free submodule of $\omega_{\mathcal{A}^D, \tau_0}/p^{\beta}$ of rank $n$. This forces $\omega_{\mathcal{C}_{n+1, \tau_0}^D}$ to have rank $n$ and hence 
    \[
    \opn{deg}\left(\left(\mathcal{A}[\ide{p}_{\tau_0}^{\beta}]/\mathcal{C}_{n+1, \tau_0}\right)^D \right) = (2n-1)\beta - \opn{deg}(\mathcal{C}^D_{n+1, \tau_0}) = (n-1)\beta .
    \]
    This implies that $\mathcal{A}[\ide{p}_{\tau_0}^{\beta}]/\mathcal{C}_{n+1, \tau_0}$ is \'{e}tale and part (3) follows.

    For the proof of part (2), let $K_p = K^H_{\diamondsuit}(p^{\beta})$ and take any rank one point $x \colon \opn{Spa}(F, \mathcal{O}_F) \to \mathcal{S}_{H, \diamondsuit}(p^{\beta})$ (with $F$ separably closed). Then, by Proposition \ref{HlociisHasseProp}(2), the condition $x \in \mathcal{S}_{H, \opn{id}}(p^{\beta})$ is equivalent to $\hat{\iota}(x) \in \mathcal{S}_{G, w_n}(p^{\beta})$, which as shown above, is equivalent to $\pi_{\opn{HT}, H, K_p}(x) \hat{\gamma} \in P_G(F) \cdot w_n \cdot K^G_{\opn{Iw}}(p^{\beta})$. But, using the fact that $\hat{\gamma} \in P_G(F) w_n$ and $K_p = \hat{\gamma} K^G_{\opn{Iw}}(p^{\beta}) \hat{\gamma}^{-1} \cap H(\mbb{Q}_p)$, we have 
    \[
    P_G(F) \backslash (P_G(F) \cdot w_n \cdot K^G_{\opn{Iw}}(p^{\beta}))/K^G_{\opn{Iw}}(p^{\beta}) = P_H(F) \backslash (P_H(F) \cdot K_p) / K_p ,
    \]
    so $x \in \mathcal{S}_{H, \opn{id}}(p^{\beta})$ is equivalent to the condition that $\pi_{\opn{HT}, H, K_p}(x)$ is equal to $1 \in \mathtt{FL}^H(F)$. Part (2) now follows.
\end{proof}

\subsection{Topological Hecke correspondences} \label{TopologicalHeckeCorrsSection}

We continue with the notation introduced in the previous section.

\begin{definition}
    Let $T^{G,-} \subset T(\mbb{Q}_p)$ (resp. $T^{H, -} \subset T(\mbb{Q}_p)$) denote the submonoid of elements $t \in T(\mbb{Q}_p)$ which satisfy $v_p(\alpha(t)) \leq 0$ for any positive root $\alpha$ of $G$ (resp. $H$). For any $t \in T^{G,-}$ (resp. $t \in T^{H, -}$) we let 
    \begin{align*} 
    \opn{min}_G(t) &= \opn{min}\{ -v_p(\alpha(t)) : \alpha \in \Phi_G^+ \} \\
    \opn{max}_G(t) &= \opn{max}\{ -v_p(\alpha(t)) : \alpha \in \Phi_G^+ \} \\
    \text{ ( resp. } \opn{min}_H(t) &= \opn{min}\{ -v_p(\alpha(t)) : \alpha \in \Phi_H^+ \} \\
    \opn{max}_H(t) &= \opn{max}\{ -v_p(\alpha(t)) : \alpha \in \Phi_H^+ \} \text{ ) }
    \end{align*}
    where $\Phi_G^+$ and $\Phi_H^+$ denote the positive roots of $G$ and $H$ respectively.
\end{definition}

For brevity and to ease notation, we place ourselves in one of the following two cases: 
\begin{equation} \label{GgKpZEqn}
(\mathscr{G}, g, K_p, \mathcal{Z}) \in \left\{ (G, w_n, K^G_{\opn{Iw}}(p^{\beta}), \mathcal{Z}_{G, >n+1}(p^{\beta})), (H, \opn{id}, K^H_{\diamondsuit}(p^{\beta}), \mathcal{Z}_{H, \opn{id}}(p^{\beta})) \right\} 
\end{equation}
where, by abuse of notation, we write $\mathcal{Z}_{H, \opn{id}}(p^{\beta})$ for its intersection with $\mathcal{S}_{H, \diamondsuit}(p^{\beta})$. For an integer $m \geq 1$, let $\mathcal{Z}_m$ denote the closure of $\pi_{\opn{HT}, G, K_p}^{-1}(\mathcal{P}_G \backslash \mathcal{P}_G \mathcal{Q}_m K_p )$ (resp. $\mathcal{U}^H_{K_p, m}$) in $\mathcal{S}_{\mathscr{G}, K_p}$ in the case $\mathscr{G} = G$ (resp. $\mathscr{G} = H$).

Let $t \in T^{\mathscr{G}, -}$. Then we can consider the following correspondence
\begin{equation} \label{GlobalSGKCorrepondence}
\begin{tikzcd}
                  & \mathcal{S}_{\mathscr{G},K_p'} \arrow[ld, "p_1"'] \arrow[rd, "p_2"] &                   \\
\mathcal{S}_{\mathscr{G}, K_p} &                                                         & \mathcal{S}_{\mathscr{G}, K_p}
\end{tikzcd}
\end{equation}
where $K_p' = t K_p t^{-1} \cap K_p$, $p_1$ is the forgetful map, and $p_2$ is induced from right multiplication by $t$. Note that $p_1$ and $p_2$ are finite \'{e}tale. The first property we need for (\ref{GlobalSGKCorrepondence}) is compatibility with support conditions. 

\begin{lemma} \label{p1p2-1ZisZLemma}
    For any integer $m \geq 1$, we have $p_1p_2^{-1}(\mathcal{Z}_m) \subset \mathcal{Z}_m$. Furthermore, we have $p_1p_2^{-1}(\mathcal{Z}) = \mathcal{Z}$.
\end{lemma}
\begin{proof}
    We first consider the case $\mathscr{G} = G$. Write $t = \tbyts{t_1}{}{}{t_2}$ with $t_1$ (resp. $t_2$) a $(n+1) \times (n+1)$ (resp. $(n-1) \times (n-1)$) diagonal matrix. Let $m \geq 1$ be an integer. Then there exists an integer $r \geq 0$ such that
    \[
    t_2^{-1} \mathcal{B}_{m+r}^{\circ} \mathcal{M}_{(n-1) \times (n+1)} t_1 \subset \mathcal{B}_{m}^{\circ} \mathcal{M}_{(n-1) \times (n+1)} .
    \]
    Furthermore, since $t \in T^{G, -}$, we see that $t_2 \mathcal{B}_{m}^{\circ} \mathcal{M}_{(n-1) \times (n+1)} t_1^{-1} \subset \mathcal{B}_{m}^{\circ} \mathcal{M}_{(n-1) \times (n+1)}$. For any $\tbyts{A}{}{C}{D} \in \mathcal{Q}_{m+r}$ (with notation as in Definition \ref{DefinitionOfCurlyQm}), we have
    \[
    \tbyt{A}{}{C}{D} = t t^{-1} \tbyt{A}{}{C}{D} t t^{-1} = t \tbyt{t_1^{-1}At_1}{}{t_2^{-1}Ct_1}{t_2^{-1}Dt_2} t^{-1} = p \tbyt{A'}{}{C'}{1} t^{-1}
    \]
    for some $p \in P_G^{\opn{an}}$, $A' \in \mathcal{GL}_{n+1}$ and $C' \in \mathcal{B}_{m}^{\circ} \mathcal{M}_{(n-1) \times (n+1)}$. Here we have used the fact that $\opn{GL}_{n+1}^{\opn{an}}$ modulo any parabolic subgroup is proper, so we may assume $A' \in \mathcal{GL}_{n+1}$. By the above inclusion and the Iwahori factorisation of $\mathcal{Q}_m$, we therefore see that
    \[
    P_G^{\opn{an}} \mathcal{Q}_{m+r} \subset P_G^{\opn{an}}\mathcal{Q}_m t^{-1}
    \]
    and so $P_G^{\opn{an}} \mathcal{Q}_{m+r} K_p \subset P_G^{\opn{an}}\mathcal{Q}_m K_p t^{-1} K_p$.

    On the other hand, we also have $\mathcal{Q}_m K_p t^{-1} \subset \mathcal{Q}_m t^{-1} K_p$ because $t \in T^{G, -}$. Indeed, by the Iwahori factorisation of $K_p$, we may replace $K_p$ with lower-triangular matrices. This, and the inclusions above, imply that
    \[
    P_G^{\opn{an}} \mathcal{Q}_m K_p t^{-1} K_p \subset P_G^{\opn{an}} \mathcal{Q}_m t^{-1} K_p \subset P_G^{\opn{an}} \mathcal{Q}_m K_p .
    \]
    By pulling back along $\pi_{\opn{HT}, G, K_p}$ and taking closures, we therefore see that $\mathcal{Z}_{m+r} \subset p_1p_2^{-1}(\mathcal{Z}_m) \subset \mathcal{Z}_m$, which gives the first claim. The second claim follows from taking the intersection over $m \geq 1$. 

    The case $\mathscr{G} = H$ is even simpler. Indeed, one can easily show that 
    \[
    ]C^H_{\opn{id}}[_{k+r, k+r} K_p \subset ]C^H_{\opn{id}}[_{k, k} (K_p t^{-1} K_p) \subset ]C^H_{\opn{id}}[_{k, k} K_p
    \]
    for any $r \geq \opn{max}_H(t)$.
\end{proof}

We now consider an overconvergent version of the above correspondence. We first note the following lemma:
\begin{lemma} \label{SupportOfCorrespContainedInCanonicalLemma}
    Let $t \in T^{\mathscr{G}, -}$ and $p_1, p_2$ the maps in the correspondence (\ref{GlobalSGKCorrepondence}). Let $d = \opn{max}_{\mathscr{G}}(t)$ and $k \geq 1$ an integer. Then
    \begin{enumerate}
        \item $p_2(\mathcal{U}^{\mathscr{G}}_{K_p', k+d}) \subset \mathcal{U}^{\mathscr{G}}_{K_p, k}$
        \item There exists an integer $m \geq k$ such that $p_1^{-1}(\mathcal{U}^{\mathscr{G}}_{K_p, k}) \cap p_2^{-1}(\mathcal{Z}_m) \subset \mathcal{U}_{K_p', k}^{\mathscr{G}}$
        \item $\mathcal{U}_{K_p', k}^{\mathscr{G}} \subset p_1^{-1}(\mathcal{U}^{\mathscr{G}}_{K_p, k})$ is an open and closed subspace.
    \end{enumerate}
    In particular, the induced map $p_1 \colon \mathcal{U}_{K_p', k}^{\mathscr{G}} \to \mathcal{U}^{\mathscr{G}}_{K_p, k}$ is finite \'{e}tale.
\end{lemma}
\begin{proof}
    For part (1), we have 
    \[
    ]C^{\mathscr{G}}_g [_{k, k} K_p' t = ]C^{\mathscr{G}}_g [_{k, k} t t^{-1} K_p' t \subset ]C^{\mathscr{G}}_g [_{k, k} t K_p, \quad \quad k \geq 1,
    \]
    so it suffices to show that $]C^{\mathscr{G}}_g [_{k+d, k+d} t \subset ]C^{\mathscr{G}}_g [_{k, k}$. But this is clear.

    For part (2), we first prove the claim for $\mathscr{G} = G$. For any integer $m \geq 1$, set $A_m \defeq \mathcal{P}_G \backslash \mathcal{P}_G \mathcal{Q}_m \subset \mathtt{FL}^G$. Then it suffices to show that for $m \gg k$, one has
    \[
    ]C^G_{w_n}[_{k, k} K_p \cap A_m K_p t^{-1} \subset ]C^G_{w_n}[_{k, k} K_p'.
    \]
    Note that $A_m K_p \subset ]C^G_{w_n}[_{m, \bar{0}} K_p \bigcup \cup_{w < w_n} ]X^G_{w}[$ (with notation as in \cite[Definition 3.1.1]{UFJ}) and we have $]C^G_{w_n}[_{k, k} K_p \cap ]X^G_{w}[t^{-1} \subset ]C^G_{w_n}[_{k, k} K_p \cap ]X^G_{w}[ = \varnothing$ for any $w < w_n$ (because $t \in T^{G, -}$). In addition to this, suppose that we have 
    \[
    [x_0 : \cdots : x_{2n-1}] \in ]C^G_{w_n}[_{k, k} K_p \cap ]C^G_{w_n}[_{m, \bar{0}} K_p t^{-1} .
    \]
    Let $\opn{diag}(t_1, \dots, t_{2n})$ denote the $\tau_0$-component of $t$. Then there exist $a_1, \dots, a_{2n}, a_1', \dots, a_{2n}' \in \mbb{Z}_p$ with the following properties:
    \begin{itemize}
        \item $a_{n+1}, a_{n+1}' \in \mbb{Z}_p^{\times}$ and $a_{n+2}, \dots, a_{2n}, a_{n+2}', \dots, a_{2n}' \in p^{\beta}\mbb{Z}_p$
        \item For $i=1, \dots, n$, we have $x_{i-1} \in a_i + \mathcal{B}_k$
        \item $x_n \in (a_{n+1} + \mathcal{B}_k^{\circ}) \cap (a_{n+1}' + \mathcal{B}_m^{\circ})$
        \item For $i=n+2, \dots, 2n$, we have $x_{i-1} \in (a_i + \mathcal{B}_k^{\circ}) \cap (t_i t_{n+1}^{-1} a_i' + \mathcal{B}_m^{\circ})$.
    \end{itemize}
    This implies that 
    \[
    (x_0, \cdots, x_{2n-1}) \in (a_1, \dots, a_n, a'_{n+1}, t_{n+2}t_{n+1}^{-1} a_{n+2}', \dots, t_{2n}t_{n+1}^{-1} a_{2n}') + (\overline{\mathcal{B}}_0, \dots, \overline{\mathcal{B}}_0, \mathcal{B}^{\circ}_m, \dots, \mathcal{B}^{\circ}_m) 
    \]
    with $n$ lots of $\overline{\mathcal{B}}_0$, and hence $]C^G_{w_n}[_{k, k} K_p \cap ]C^G_{w_n}[_{m, \bar{0}} K_p t^{-1} \subset ]C^G_{w_n}[_{m, \bar{0}} K_p'$. So to prove (2) in the case $\mathscr{G} = G$, it suffices to show that
    \[
    ]C^G_{w_n}[_{k, k} K_p \cap ]C^G_{w_n}[_{m, \bar{0}} \subset ]C^G_{w_n}[_{k, k} K_p' .
    \]
    But, as long as we take $m \geq k + \opn{max}_G(t)$, then this follows from a similar explicit calculation above (take $a_{i}' = 0$ for $i \neq n+1$). This completes the proof of (2) in the case $\mathscr{G} = G$. The proof of part (2) for $\mathscr{G} = H$ is simpler. Indeed, one can easily show that $p_2^{-1}(\mathcal{U}^H_{K_p, m}) \subset \mathcal{U}^H_{K_p', k}$ for any $m \gg k$ (depending on $\opn{max}_H(t)$).

    We now prove part (3). For $\mathscr{G} = G$, we simply note that
    \[
    ] C^G_{w_n} [_{k, k} K_p = \bigcup_{\underline{a}} \left( \underline{a} + (\mathcal{B}_k, \dots, \mathcal{B}_k, 0, \mathcal{B}_k^{\circ}, \dots, \mathcal{B}_k^{\circ}) \right) 
    \]
    where the $0$ is in the $(n+1)$-th place and the union is over tuples $(a_1, \dots, a_n, 1, a_{n+2}, \dots, a_{2n})$, with $a_i \in \mbb{Z}/p^k\mbb{Z}$ (resp. $a_i \in p^{\beta}\mbb{Z}/p^{\opn{max}(\beta, k)}\mbb{Z}$) for $i \in \{1, \dots, n \}$ (resp. $i \in \{ n+2, \dots, 2n \}$). Clearly there are only finitely many such tuples, and this is a disjoint union of open and closed subspaces. Furthermore, one easily sees that $]C_{w_n}^G[_{k, k} K_p'$ is a union over the cosets for which 
    \[
    a_i \in p^{\beta'_i}\mbb{Z}/p^{\opn{max}(\beta'_i, k)}\mbb{Z} \subset p^{\beta}\mbb{Z}/p^{\opn{max}(\beta, k)}\mbb{Z}, \quad \quad \beta'_i \defeq \beta + v_p(t_it_{n+1}^{-1})
    \]
    for all $i \in \{n+2, \dots, 2n\}$. This proves the claim -- the proof of (3) in the case $\mathscr{G} = H$ follows an identical argument. Finally, since $p_1 \colon p_1^{-1}(\mathcal{U}^{\mathscr{G}}_{K_p, k}) \to \mathcal{U}^{\mathscr{G}}_{K_p, k}$ is finite \'{e}tale, the induced morphism $p_1 \colon \mathcal{U}^{\mathscr{G}}_{K_p', k} \to \mathcal{U}^{\mathscr{G}}_{K_p, k}$ is also finite \'{e}tale (by part (3)).
\end{proof}

We now consider the overconvergent version of the Hecke correpondence. Let $t \in T^{\mathscr{G}, -}$ and $d = \opn{max}_{\mathscr{G}}(t)$. Then for any $k \geq 1$, Lemma \ref{SupportOfCorrespContainedInCanonicalLemma} implies that we have a correspondence:
\begin{equation} \label{OCSGKCorrepondence}
\begin{tikzcd}
                  & {\mathcal{U}^{\mathscr{G}}_{K_p', k+d}} \arrow[ld, "p_1"'] \arrow[rd, "p_2"] &                   \\
{\mathcal{U}^{\mathscr{G}}_{K_p, k+d}} &                                                         & {\mathcal{U}^{\mathscr{G}}_{K_p, k} .}
\end{tikzcd}
\end{equation}
This is compatible with the correspondence in (\ref{GlobalSGKCorrepondence}) in  the sense that we have a commutative diagram 
\begin{equation} \label{Eqn:ObviousCommute}
\begin{tikzcd}
{\mathcal{U}^{\mathscr{G}}_{K_p, k+d}} \arrow[d, hook] & {\mathcal{U}^{\mathscr{G}}_{K_p', k+d}} \arrow[d, hook] \arrow[l, "p_1"'] \arrow[r, "p_2"] & {\mathcal{U}^{\mathscr{G}}_{K_p, k}} \arrow[d, hook] \\
{\mathcal{S}_{\mathscr{G}, K_p}}                       & {\mathcal{S}_{\mathscr{G}, K_p'}} \arrow[l, "p_1"'] \arrow[r, "p_2"]                       & {\mathcal{S}_{\mathscr{G}, K_p}}                    
\end{tikzcd} 
\end{equation}
We note that the squares in \eqref{Eqn:ObviousCommute} are often not Cartesian, because the horizontal maps often do not have the same degree.

Finally, we will consider a version of this correspondence over the ordinary locus. Recall from \S \ref{QuotientsOfIgVarSubSec} that we have group schemes $J^G_{\opn{Iw}}(p^{\beta})$ and $J^H_{\diamondsuit}(p^{\beta})$, both of which we denote by $J_p$ in cases $\mathscr{G} = G$ and $\mathscr{G} = H$ respectively. For any $t \in T^{\mathscr{G}, -}$ we can naturally view $t' = g t g^{-1} \in T(\mbb{Q}_p) \subset J_{\mathscr{G}, \opn{ord}}$, and we consider the group scheme $J_p' \defeq t' J_p (t')^{-1} \cap J_p$. Then we have a natural correspondence between quotients of Igusa varieties
\begin{equation} \label{IgusaSGKCorrepondence}
\begin{tikzcd}
                  & {\mathcal{IG}_{\mathscr{G}}/J_p'} \arrow[ld, "q_1"'] \arrow[rd, "q_2"] &                   \\
{\mathcal{IG}_{\mathscr{G}}/J_p} &                                                         & {\mathcal{IG}_{\mathscr{G}}/J_p}
\end{tikzcd}
\end{equation}
where $q_1$ is the natural forgetful map and $q_2$ is induced from the right-action of $t'$. Here $\mathcal{IG}_{\mathscr{G}}$ is the adic generic fibre of $\ide{IG}_{\mathscr{G}}$. Both of the morphisms $q_1$ and $q_2$ are finite \'{e}tale. Furthermore, there exists a finite extension $L' / L$ (depending on $t$) and a commutative diagram:
\begin{equation} \label{DiagramBetweenOrdHeckeAndOCHeckeEqn}
\begin{tikzcd}
\left( \mathcal{IG}_{\mathscr{G}}/J_p \right)_{L'} \arrow[d, "g", hook] & \left( \mathcal{IG}_{\mathscr{G}}/J_p' \right)_{L'} \arrow[d, "g", hook] \arrow[l, "q_1"'] \arrow[r, "q_2"] & \left( \mathcal{IG}_{\mathscr{G}}/J_p \right)_{L'} \arrow[d, "g", hook] \\
{(\mathcal{U}^{\mathscr{G}}_{K_p, k+d})_{L'}}                           & {(\mathcal{U}^{\mathscr{G}}_{K_p', k+d})_{L'}} \arrow[l, "p_1"'] \arrow[r, "p_2"]                           & {(\mathcal{U}^{\mathscr{G}}_{K_p, k})_{L'}}                            
\end{tikzcd}
\end{equation}
where the vertical maps are open immersions induced from right-translation by $g$.

\begin{remark}
    The extension $L'$ is obtained by adjoining $\mu_{p^{\beta + d}}$ to $L$ (where $d=\opn{max}_{\mathscr{G}}(t)$). The reason for this is similar to the discussion at the end of \S \ref{QuotientsOfIgVarSubSec} and is due to the fact that we need to trivialise $\mu_{p^{\beta + d}}$ to compare level structures for quotients of the Igusa variety and adic Shimura varieties.
\end{remark}

\begin{lemma} \label{LeftHandSqaureIsCartesianLemma}
    If $k \geq \beta - 1$, then the left-hand square in (\ref{DiagramBetweenOrdHeckeAndOCHeckeEqn}) is Cartesian.
\end{lemma}
\begin{proof}
    It suffices to check the claim on rank one points. Note that $\left( \mathcal{IG}_{\mathscr{G}}/J_p \right)_{L'}^{\opn{rk} 1}$ (resp. $\left( \mathcal{IG}_{\mathscr{G}}/J_p' \right)_{L'}^{\opn{rk} 1}$) is identified with $\pi_{\opn{HT}, K_p}^{-1}\left( \mathcal{P}_{\mathscr{G}}\backslash \mathcal{P}_{\mathscr{G}} g K_p \right)^{\opn{rk} 1}$ (resp. $\pi_{\opn{HT}, K_p'}^{-1}\left( \mathcal{P}_{\mathscr{G}}\backslash \mathcal{P}_{\mathscr{G}} g K_p' \right)^{\opn{rk} 1}$). Therefore, it is enough to show that
    \[
    \mathcal{P}_{\mathscr{G}}\backslash \mathcal{P}_{\mathscr{G}} g K_p \; \cap \; ]C^{\mathscr{G}}_g[_{k+d, k+d} K_p' = \mathcal{P}_{\mathscr{G}}\backslash \mathcal{P}_{\mathscr{G}} g K_p'
    \]
    with the intersection taking place in $\mathtt{FL}^{\mathscr{G}}$. In the case $\mathscr{G} = G$, an element 
    \[
    [x_0 : \cdots : x_{2n-1}] \in \mathcal{P}_{\mathscr{G}}\backslash \mathcal{P}_{\mathscr{G}} g K_p \; \cap \; ]C^{\mathscr{G}}_g[_{k+d, k+d}
    \]
    can be described (uniquely) as a point on the flag variety satisfying:
    \begin{itemize}
        \item For $i=0, \dots n-1$, we have $x_i \in \mbb{Z}_p \cap \mathcal{B}_{k+d} = p^{k+d}\mbb{Z}_p$
        \item $x_n = 1$
        \item For $i=n+1, \dots, 2n-1$, we have $x_i \in p^{\beta}\mbb{Z}_p \cap \mathcal{B}^{\circ}_{k+d} = p^{\opn{max}(\beta, k+d+1)} \mbb{Z}_p \subset p^{\beta + d}\mbb{Z}_p$ .
    \end{itemize}
    We then see that $\mathcal{P}_{\mathscr{G}}\backslash \mathcal{P}_{\mathscr{G}} g K_p \; \cap \; ]C^{\mathscr{G}}_g[_{k+d, k+d} \; \subset \mathcal{P}_{\mathscr{G}}\backslash \mathcal{P}_{\mathscr{G}} g K_p'$ which is sufficient for establishing the claim. The case $\mathscr{G} = H$ is similar and left to the reader.
\end{proof}

\subsection{Cohomological correspondences} \label{HeckeCohCorrespSubSec}

We now discuss the maps of sheaves associated with each correspondence in the previous section. We then explain how one obtains the action of Hecke operators on cohomology (with and without partial compact support). We continue with the general notation in the previous section (i.e. we deal with the cases $\mathscr{G} = G$ and $\mathscr{G} = H$ simultaneously).  Let $\langle - \rangle \colon \mbb{Q}_p^{\times} \to \mbb{Z}_p^{\times}$ denote the natural projection arising from the identification $\mbb{Q}_p^{\times} = p^{\mbb{Z}} \times \mbb{Z}_p^{\times}$. We extend this naturally to a map $\langle - \rangle \colon T(\mbb{Q}_p) \to T(\mbb{Z}_p)$.

We begin with the cohomological correspondence for classical weights. Recall that we have an $M_{\mathscr{G}}^{\opn{an}}$-torsor $M^{\opn{an}}_{\mathscr{G}, \opn{dR}, K_p} \to \mathcal{S}_{\mathscr{G}, K_p}$ (where we now include the level subgroup in the notation). Let $t \in T^{\mathscr{G}, -}$ and consider the correspondence in (\ref{GlobalSGKCorrepondence}). 

\begin{definition} \label{phitclassicalcohcorrespDefinition}
    \begin{enumerate}
        \item Let $\phi_t \colon p_1^{-1} M^{\opn{an}}_{\mathscr{G}, \opn{dR}, K_p} \to p_2^{-1} M^{\opn{an}}_{\mathscr{G}, \opn{dR}, K_p}$ denote the composition of the map induced by the $G(\mbb{Q}_p)$-equivariant structure on the torsors $M^{\opn{an}}_{\mathscr{G}, \opn{dR},-}$ and the action of $gtg^{-1} \in M_{\mathscr{G}}(\mbb{Q}_p)$ through the torsor structure (in either order).
        \item Let $\kappa \in X^*(T)$ be a $M_{\mathscr{G}}$-dominant weight, and let $V_{\kappa}^*$ denote the linear dual of the algebraic representation of $M_{\mathscr{G}}$ of highest weight $\kappa$. We let $\phi_{t, \kappa^*} \colon p_2^* \mathscr{M}_{\mathscr{G}, \kappa^*} \to p_1^* \mathscr{M}_{\mathscr{G}, \kappa^*}$ denote the morphism of $\mathcal{O}_{\mathcal{S}_{\mathscr{G}, K_p'}}$-modules obtained as the $M_{\mathscr{G}}^{\opn{an}}$ invariants of:
        \[
        p_2^*\mathscr{M}_{\mathscr{G}} \otimes V^*_{\kappa} \xrightarrow{\phi_t^* \otimes (\kappa^*(gt^{-1}\langle t \rangle g^{-1}) gtg^{-1} \cdot -)} p_1^*\mathscr{M}_{\mathscr{G}} \otimes V_{\kappa}^* .
        \]
    \end{enumerate}
\end{definition}

This induces the following operators on cohomology.

\begin{definition}
    Let $U_t \colon R\Gamma\left( \mathcal{S}_{\mathscr{G}, K_p}, \mathscr{M}_{\mathscr{G}, \kappa^*} \right) \to R\Gamma\left( \mathcal{S}_{\mathscr{G}, K_p}, \mathscr{M}_{\mathscr{G}, \kappa^*} \right)$ denote the Hecke correspondence defined as the composition:
    \begin{align*}
        R\Gamma\left( \mathcal{S}_{\mathscr{G}, K_p}, \mathscr{M}_{\mathscr{G}, \kappa^*} \right) &\xrightarrow{p_2^*} R\Gamma\left( \mathcal{S}_{\mathscr{G}, K_p'}, p_2^*\mathscr{M}_{\mathscr{G}, \kappa^*} \right) \\
        &\xrightarrow{\phi_{t, \kappa^*}} R\Gamma\left( \mathcal{S}_{\mathscr{G}, K_p'}, p_1^*\mathscr{M}_{\mathscr{G}, \kappa^*} \right) \\
        &\xrightarrow{(-g^{-1}\rho + \rho)(t\langle t \rangle^{-1})\opn{Tr}_{p_1}} R\Gamma\left( \mathcal{S}_{\mathscr{G}, K_p}, \mathscr{M}_{\mathscr{G}, \kappa^*} \right)
    \end{align*}
    where $\rho$ is the half-sum of the positive roots in $\mathscr{G}$ and $\opn{Tr}_{p_1}$ is the trace map from \cite[Lemma 2.1.2]{BoxerPilloni}. Since $p_1p_2^{-1}(\mathcal{Z}) = \mathcal{Z}$ (Lemma \ref{p1p2-1ZisZLemma}) we also obtain a Hecke correspondence $U_t$ on $R\Gamma_{\mathcal{Z}}\left( \mathcal{S}_{\mathscr{G}, K_p}, \mathscr{M}_{\mathscr{G}, \kappa^*} \right)$ in exactly the same way, and these two operators are compatible with each other under the natural corestriction map
    \[
    R\Gamma_{\mathcal{Z}}\left( \mathcal{S}_{\mathscr{G}, K_p}, \mathscr{M}_{\mathscr{G}, \kappa^*} \right) \to R\Gamma\left( \mathcal{S}_{\mathscr{G}, K_p}, \mathscr{M}_{\mathscr{G}, \kappa^*} \right) .
    \]
\end{definition}

\begin{remark}
    The Hecke operator $U_t$ is optimally normalised for those weights $\kappa$ with $C(\kappa^*)^{-} = \{ g \}$ (see \cite[\S 5.9]{BoxerPilloni}). 
\end{remark}

We now discuss the overconvergent version of this. Let $k \geq 1$ be an integer and recall the definition of $\mathscr{G}^1_{k, k}$ from \cite[\S 3.3.10]{BoxerPilloni} (i.e. the subgroup of elements in $\mathscr{G}$ which land in the lower-triangular Borel unipotent modulo $p^{k}$, and in the upper-triangular Borel unipotent modulo $p^{k+\varepsilon}$ for some $\varepsilon > 0$). Let $K_p'' = t^{-1}K_p't$. Then, for $U_p = K_p, K_p', K_p''$ we have torsors 
\begin{align*} 
\mathtt{M}_{\mathscr{G}, \opn{HT}, U_p, k} = g^{-1} g \mathscr{G}^{1}_{k, k} U_p g^{-1} / \left( \opn{Unip}(\mathcal{P}_{\mathscr{G}}) \cap g \mathscr{G}^{1}_{k, k} U_p g^{-1} \right) &\to ]C^{\mathscr{G}}_{g}[_{k, k} U_p \\
x &\mapsto x^{-1}
\end{align*} 
under the group $U_{p, g, k}^{\circ} \subset \mathcal{M}_{\mathscr{G}}$ obtained as the image of $g \mathscr{G}^{1}_{k, k} U_p g^{-1} \cap \mathcal{P}_{\mathscr{G}}$ under projection to the Levi. Let $U_{p, g, k} \subset \mathcal{M}_{\mathscr{G}}$ denote the affinoid subgroup obtained as the image of $g \mathscr{G}^{1}_{k} U_p g^{-1} \cap \mathcal{P}_{\mathscr{G}}$ under projection to the Levi, where $\mathscr{G}^1_k \subset \mathscr{G}$ denotes the open affinoid subgroup of elements which reduce to the identity modulo $p^k$. Note that $K_{p, g, k} = \mathcal{M}^{\square}_{G, k}$ (resp. $K_{p, g, k} = \mathcal{M}^{\diamondsuit}_{H, k}$) in the case $\mathscr{G} = G$ (resp. $\mathscr{G} = H$).

By pulling back under the Hodge--Tate period morphism, descending to finite level, twisting along the Hodge cocharacter $\mu$ (or its restriction to $1+p^{\beta}\mbb{Z}_p$ in the case $\mathscr{G} = H$), and pushing out along the inclusion $U_{p, g, k}^{\circ} \subset U_{p, g, k}$ we obtain \'{e}tale $U_{p, g, k}$-torsors ${^\mu\mathcal{M}_{\mathscr{G}, \opn{HT}, U_p, k}} \to \mathcal{U}^{\mathscr{G}}_{U_p, k}$. Suppose we are in the setting of (\ref{GlobalSGKCorrepondence}), and let $t \colon \mathcal{U}^{\mathscr{G}}_{K_p',k+d} \to \mathcal{U}^{\mathscr{G}}_{K_p'', k}$ denote the map induced from right-translation by $t$. Then we have a diagram:
\[
\begin{tikzcd}
{{^{\mu} \mathcal{M}_{\mathscr{G}, \opn{HT}, K_p', k+d}}} \arrow[d] \arrow[r, "\phi_t"] & {t^* \left({^{\mu} \mathcal{M}_{\mathscr{G}, \opn{HT}, K_p'', k}}\right)} \arrow[d]        \\
{\left. p_1^{-1}\left( {^{\mu} \mathcal{M}_{\mathscr{G}, \opn{HT}, K_p, k+d}}\right) \right|_{\mathcal{U}^{\mathscr{G}}_{K_p', k+d}}}          & {\left. p_2^{-1} \left( {^{\mu} \mathcal{M}_{\mathscr{G}, \opn{HT}, K_p, k}}\right) \right|_{\mathcal{U}^{\mathscr{G}}_{K_p', k+d}}}
\end{tikzcd}
\]
where the vertical maps are reductions of structure and $\phi_t$ is induced from the morphism 
\begin{align*} 
\mathtt{M}_{\mathscr{G}, \opn{HT}, K_p', k+d} &\to \mathtt{M}_{\mathscr{G}, \opn{HT}, K_p'', k} \\ [x] &\mapsto [t^{-1} x (g t g^{-1})] .
\end{align*}
As the notation suggests, the map $\phi_t$ commutes with the one in Definition \ref{phitclassicalcohcorrespDefinition}(1). We now construct the overconvergent cohomological correspondences.

\begin{definition}
    Let $\mathscr{O}_{{^{\mu} \mathcal{M}_{\mathscr{G}, \opn{HT}, U_p, k}}}$ denote the pushforward of the structure sheaf of ${^{\mu} \mathcal{M}_{\mathscr{G}, \opn{HT}, U_p, k}}$ to $\mathcal{U}^{\mathscr{G}}_{U_p, k}$. Let $(R, R^+)$ be a complete Tate affinoid algebra over $(L, \mathcal{O}_L)$. Let $T^{M_G, -} \subset T(\mbb{Q}_p)$ be the submonoid of elements $t \in T(\mbb{Q}_p)$ satisfying $v_p(\alpha(t)) \leq 0$ for all positive roots $\alpha \in \Phi_{M_G}^+$.
    \begin{enumerate}
        \item Suppose we are in case $\mathscr{G} = G$. Let $\kappa \colon T(\mbb{Z}_p) \to (R^+)^{\times}$ be an $s$-analytic character and let $k \geq s+1$. Recall that $D_{G, \kappa^*}^{s\opn{-an}}$ denotes the continuous $R$-dual of $V_{G, \kappa}^{\circ, s\opn{-an}}$, which comes equipped with an action of the submonoid of $G(\mbb{Q}_p)$ generated by $K_{p,w_n,k}$ and $T^{M_G, -}$ (see \cite[\S 6.2.20]{BoxerPilloni}\footnote{One can (and must) uniquely modify the action in \cite[\S 6.2.20]{BoxerPilloni} slightly so that the $K_{p, w_n, k}$ and $T^{M_G, -}$ actions are compatible, and the action of any $t \in T^{M_G, -}$ with $\langle t \rangle = 1$ is the same as in \emph{loc.cit.}.}). We set $\mathscr{M}_{G, \kappa^*}^{s\opn{-an}} \defeq \left( \mathscr{O}_{{^{\mu} \mathcal{M}_{G, \opn{HT}, K_p, k}}} \hatot D_{G, \kappa^*}^{s\opn{-an}} \right)^{K_{p, w_n,k}}$. Then we obtain a cohomological correspondence
        \[
        \phi_{t, \kappa^*}^{s\opn{-an}} \colon p_2^*(\mathscr{M}_{G, \kappa^*}^{s\opn{-an}}) \to p_1^*(\mathscr{M}_{G, \kappa^*}^{s\opn{-an}})
        \]
        defined on $\mathcal{U}^G_{K_p', k+d}$, which is induced from the morphism
        \[
        \left( t^* \mathscr{O}_{{^{\mu} \mathcal{M}_{G, \opn{HT}, K_p'', k}}} \hatot D_{G, \kappa^*}^{s\opn{-an}} \right)^{K''_{p, w_n, k}} \to \left( \mathscr{O}_{{^{\mu} \mathcal{M}_{G, \opn{HT}, K_p', k+d}}} \hatot D_{G, \kappa^*}^{s\opn{-an}} \right)^{K'_{p, w_n, k+d}}
        \]
        given as the tensor product $\phi_t^* \otimes ((w_n t w_n^{-1}) \cdot -)$, noting that $w_n t w_n^{-1} \in T^{M_G, -}$.
        \item Suppose that we are in case $\mathscr{G} = H$ and let $\sigma \colon M^H_{\diamondsuit}(p^{\beta}) \to (R^+)^{\times}$ be an $s$-analytic character  which extends to a character $T(\mbb{Z}_p) \to (R^+)^{\times}$ (also denoted $\sigma$). Let $k \geq s$. We set $\mathscr{M}^{\opn{an}}_{H, \sigma} \defeq \left( \mathscr{O}_{{^{\mu} \mathcal{M}_{H, \opn{HT}, K_p, k}}} \hatot \sigma \right)^{K_{p, \opn{id}, k}}$. Then we obtain a cohomological correspondence
        \[
        \phi_{t, \sigma}^{\opn{an}} \colon p_2^*(\mathscr{M}_{H, \sigma}^{\opn{an}}) \to p_1^*(\mathscr{M}_{H, \sigma}^{\opn{an}})
        \]
        defined over $\mathcal{U}^H_{K_p', k+d}$, which is induced from the morphism
        \[
        \left( t^* \mathscr{O}_{{^{\mu} \mathcal{M}_{H, \opn{HT}, K_p'', k}}} \hatot \sigma \right)^{K''_{p, \opn{id}, k}} \to \left( \mathscr{O}_{{^{\mu} \mathcal{M}_{H, \opn{HT}, K_p', k+d}}} \hatot \sigma \right)^{K'_{p, \opn{id}, k+d}}
        \]
        given as the tensor product $\phi_t^* \otimes  \sigma(\langle t \rangle) $. This is well-defined because $\sigma(x)=\sigma(t x t^{-1})$ for any $x \in K_{p, \opn{id}, k}''$.
    \end{enumerate}
\end{definition}

\begin{remark} \label{OCphitequalsrestrictedphitRemark}
    If $\kappa \in X^*(T)$ is $M_G$-dominant, then the morphism $\phi_{t, \kappa^*}^{s\opn{-an}}$ is compatible with the restriction of $\phi_{t, \kappa^*}$ to $\mathcal{U}^G_{K_p', k+d}$ via the natural map $\mathscr{M}_{G, \kappa^*}^{s\opn{-an}} \to \mathscr{M}_{G, \kappa^*}$ induced from the map of representations $D_{G, \kappa^*}^{s\opn{-an}} \to V_{\kappa}^*$. Similarly, if $\sigma$ is an algebraic character of $M_H$, then the morphism $\phi_{t, \sigma}^{\opn{an}}$ coincides with the restriction of $\phi_{t, \sigma}$ to $\mathcal{U}^H_{K_p', k+d}$.
\end{remark}

By Lemmas \ref{p1p2-1ZisZLemma} and \ref{SupportOfCorrespContainedInCanonicalLemma} the collections $\{ \mathcal{U}^{\mathscr{G}}_{K_p, k} \}_{k \geq 1}$ and $\{ \mathcal{Z}_m \}_{m \geq 1}$ form a system of support conditions for the correspondence (\ref{GlobalSGKCorrepondence}), in the sense of Definition \ref{SystemOfSupportConditionsDef}. Furthermore, note that:
\begin{itemize}
    \item $R\Gamma(\mathcal{U}^G_{K_p, \bullet}, \mathcal{Z}_{\bullet}; \mathscr{M}_{G, \kappa^*}) = R\Gamma^G_{w_n}(\kappa^*; \beta)^{(-, \dagger)}$
    \item $R\Gamma(\mathcal{U}^G_{K_p, \bullet}, \mathcal{Z}_{\bullet}; \mathscr{M}^{s\opn{-an}}_{G, \kappa^*}) = R\Gamma^G_{w_n, s\opn{-an}}(\kappa^*; \beta)^{(-, \dagger)}$
    \item $R\Gamma(\mathcal{U}^H_{K_p, \bullet}, \mathcal{Z}_{\bullet}; \mathscr{M}_{H, \sigma}) = R\Gamma_{\mathcal{Z}_{H, \opn{id}}(p^{\beta})}\left( \mathcal{S}_{H, \diamondsuit}(p^{\beta}), \mathscr{M}_{H, \sigma} \right)$
    \item $R\Gamma(\mathcal{U}^H_{K_p, \bullet}, \mathcal{Z}_{\bullet}; \mathscr{M}^{\opn{an}}_{H, \sigma}) = R\Gamma^H_{\opn{id}, \opn{an}}(\mathcal{S}_{H, \diamondsuit}(p^{\beta}), \sigma)^{(-, \dagger)}$
\end{itemize}
because the support conditions defined by inequalities involving $\hat{\delta}_{G, n+1}^+$, $\delta_{G, >n+1}^+$ and $\hat{\delta}_{H, n+1}^+$ intertwine with the support conditions $(\mathcal{U}^{\mathscr{G}}_{K_p, \bullet}, \mathcal{Z}_{\bullet})$. By applying the general construction in \S \ref{CohomologyAndCorrespondencesSection} (and normalising the trace map by  $(-g^{-1}\rho + \rho)(t\langle t \rangle^{-1})$), we therefore obtain Hecke operators on each of the above cohomologies, all of which we denote by $U_t$.

\begin{remark}
    By Remark \ref{OCphitequalsrestrictedphitRemark}, the natural restriction map 
    \[
    R\Gamma_{\mathcal{Z}_{G, >n+1}(p^{\beta})}\left( \mathcal{S}_{G, \opn{Iw}}(p^{\beta}), \mathscr{M}_{G, \kappa^*} \right) \to R\Gamma^G_{w_n}(\kappa^*; \beta)^{(-, \dagger)}
    \]
    is $U_t$-equivariant. Similarly, the Hecke operator $U_t$ on $R\Gamma^G_{w_n}(\kappa^*; \beta)^{(-, \dagger)}$ can be seen as the colimit of the following operators:
    \begin{align*} 
    R\Gamma_{\mathcal{U}^{G}_{K_p, k} \cap \mathcal{Z}}\left( \mathcal{U}^{G}_{K_p, k}, \mathscr{M}_{G, \kappa^*} \right) &\xrightarrow{p_2^*} R\Gamma_{\mathcal{U}^{G}_{K_p', k+d} \cap p_2^{-1}\mathcal{Z}}\left( \mathcal{U}^{G}_{K_p', k+d}, p_2^*\mathscr{M}_{G, \kappa^*} \right) \\
    &\xrightarrow{\phi_{t, \kappa^*}} R\Gamma_{\mathcal{U}^{G}_{K_p', k+d} \cap p_2^{-1}\mathcal{Z}}\left( \mathcal{U}^{G}_{K_p', k+d}, p_1^*\mathscr{M}_{G, \kappa^*} \right) \\
    &\xrightarrow{(-g^{-1}\rho+ \rho)(t\langle t \rangle^{-1}) \tilde{\opn{Tr}}_{p_1}} R\Gamma_{\mathcal{U}^{G}_{K_p, k+d} \cap \mathcal{Z}}\left( \mathcal{U}^{G}_{K_p, k+d}, \mathscr{M}_{G, \kappa^*} \right)
    \end{align*}
    where $\tilde{\opn{Tr}}_{p_1}$ denotes the trace map associated with the finite flat morphism $p_1 \colon \mathcal{U}^{G}_{K_p', k+d} \to \mathcal{U}^{G}_{K_p, k+d}$ and we have used the fact that $p_1p_2^{-1}(\mathcal{Z}) = \mathcal{Z}$ (see \cite[Lemma 2.1.2]{BoxerPilloni}). This alternative description holds because $\mathcal{U}^{G}_{K_p', k+d}$ is open and closed in $p_1^{-1}(\mathcal{U}^{G}_{K_p, k+d})$.
\end{remark}

\begin{remark}
    As explained in \cite[\S 4.2]{BoxerPilloni}, the operators $\{ U_t : t \in T^{\mathscr{G}, -} \}$ commute with each other after passing to the cohomology groups of the complexes above.
\end{remark}

\subsection{Frobenius} \label{DiscussionOfFrobeniusSubSec}

We now discuss the Frobenius operator acting on the cohomology of the Shimura--Deligne varieties. We will continue to use the same notation as in \S \ref{TopologicalHeckeCorrsSection}, i.e. $(\mathscr{G}, g, K_p, \mathcal{Z})$ will denote one of the tuples in (\ref{GgKpZEqn}). However in this section we will focus on a single correspondence which is not necessarily associated with an element of $T^{\mathscr{G}, -}$.

Let $\xi \in T(\mbb{Q}_p)$ denote the element which is trivial outside the $\tau_0$-component, and in the $\tau_0$-component is equal to 
\[
\xi = \opn{diag}(p^{-1}, 1, \dots, 1) .
\]
Note that $\xi \in T^{M_{\mathscr{G}}, -}$. We let $t = g^{-1} \xi g$ and set $K_p' = tK_pt^{-1} \cap K_p$ and $K_p'' = t^{-1}K_p' t$. We consider the following correspondence
\begin{equation} \label{FrobeniusTopCorrEqn}
\begin{tikzcd}
                                 & {\mathcal{S}_{\mathscr{G}, K_p'}} \arrow[ld, "p_1"'] \arrow[rd, "p_2"] &                                  \\
{\mathcal{S}_{\mathscr{G}, K_p}} &                                                                        & {\mathcal{S}_{\mathscr{G}, K_p}}
\end{tikzcd}
\end{equation}
where $p_1$ is the forgetful map, and $p_2$ is induced from right-translation by $t$. We have the following properties:

\begin{lemma} \label{FrobSystemSupportLemma}
    Let $k,m \geq 1$ be integers. Then
    \begin{enumerate}
        \item $p_1p_2^{-1}(\mathcal{Z}_m) \subset \mathcal{Z}_m$ and $p_1 p_2^{-1}(\mathcal{Z}) = \mathcal{Z}$.
        \item $p_2(\mathcal{U}^{\mathscr{G}}_{K_p', k+1}) \subset \mathcal{U}^{\mathscr{G}}_{K_p, k}$ and the map $p_2 \colon \mathcal{U}^{\mathscr{G}}_{K_p', k+1} \to \mathcal{U}^{\mathscr{G}}_{K_p, k}$ is finite \'{e}tale of degree $[K_p : K_p'']$.
        \item The map $p_1 \colon \mathcal{U}^{\mathscr{G}}_{K_p', k} \to \mathcal{U}^{\mathscr{G}}_{K_p, k}$ is an isomorphism.
        \item For $m \gg k$, we have $p_1^{-1}(\mathcal{U}^{\mathscr{G}}_{K_p, k}) \cap p_2^{-1}(\mathcal{Z}_m) \subset \mathcal{U}^{\mathscr{G}}_{K_p', k}$.
    \end{enumerate}
    In particular, the collections $\{ \mathcal{U}^{\mathscr{G}}_{K_p, k} \}_{k \in \mbb{N}}$ and $\{ \mathcal{Z}_m \}_{m \in \mbb{N}}$ form a system of support conditions for the correspondence (\ref{FrobeniusTopCorrEqn}) as in Definition \ref{SystemOfSupportConditionsDef}. 
\end{lemma}
\begin{proof}
    The proof is very similar to Lemmas \ref{p1p2-1ZisZLemma} and \ref{SupportOfCorrespContainedInCanonicalLemma} even though $t \not\in T^{\mathscr{G}, -}$. Indeed, for part (1) in the case $G = \mathscr{G}$, it is enough to show that
    \[
    P_G^{\opn{an}} \mathcal{Q}_{m+1} K_p \subset P_G^{\opn{an}} \mathcal{Q}_m K_p t^{-1} K_p \subset P_G^{\opn{an}} \mathcal{Q}_m K_p 
    \]
    and we just follow the same argument in Lemma \ref{p1p2-1ZisZLemma} (noting that we only need $v_p(\alpha(t)) \leq 0$ for the positive roots $\alpha$ not contained in the $(n+1, n-1)$-parabolic in the $\tau_0$-component). The case $\mathscr{G} = H$ is similar.

    For part (2) it is enough to show $p_2^{-1}(\mathcal{U}^{\mathscr{G}}_{K_p, k}) = \mathcal{U}^{\mathscr{G}}_{K_p', k+1}$, but this follows from the easy calculation that
    \[
    ]C^{\mathscr{G}}_g[_{k, k}K_p t^{-1} K_p' = ]C^{\mathscr{G}}_g[_{k+1, k+1} K_p' .
    \]
    For part (3), one can show that $]C^{\mathscr{G}}_g[_{k, k}K_p$ is the disjoint union of $[K_p : K_p']$ open and closed subspaces, and one of these subspaces is equal to $]C^{\mathscr{G}}_g[_{k, k} K_p'$, hence the map $p_1 \colon \mathcal{U}^{\mathscr{G}}_{K_p', k} \to \mathcal{U}^{\mathscr{G}}_{K_p, k}$ must be an isomorphism.

    Finally, for part (4) the case $\mathscr{G} = H$ follows from $p_2^{-1}(\mathcal{U}^{\mathscr{G}}_{K_p, k}) = \mathcal{U}^{\mathscr{G}}_{K_p', k+1}$, so we give the proof in the case $\mathscr{G} = G$. In this case, one can follow the same proof as in Lemma \ref{SupportOfCorrespContainedInCanonicalLemma} with some minor modifications. Firstly, one can reduce to showing 
    \[
    ]C^G_{w_n}[_{k, k} K_p \cap ]C^G_{w_n}[_{m, \bar{0}} K_pt^{-1} \subset ]C^G_{w_n}[_{k, k}K_p'
    \]
    because $A_mK_pt^{-1} \subset A_mK_p$. 
     
    By a similar explicit calculation as in the proof of Lemma \ref{SupportOfCorrespContainedInCanonicalLemma}, one can show that: if $[x_0 : \cdots : x_{2n-1}] \in ]C^G_{w_n}[_{k, k} K_p \cap ]C^G_{w_n}[_{m, \bar{0}} K_pt^{-1}$, then there exist 
    \[
    a_1, \dots, a_n, a_1', \dots, a_n' \in \mbb{Z}_p \; \text{ and } \; a_{n+2}, \dots, a_{2n}, a_{n+2}', \dots, a_{2n}' \in p^{\beta}\mbb{Z}_p
    \]
    such that $x_{i-1} \in (a_i + \mathcal{B}_k) \cap p(a_i' + \overline{\mathcal{B}}_0)$ for $1 \leq i \leq n$, $x_{n} \in (1+\mathcal{B}^{\circ}_k) \cap (1+\mathcal{B}_m^{\circ})$, and $x_{i-1} \in (a_i + \mathcal{B}_k^{\circ}) \cap p(a_i + \mathcal{B}_m^{\circ})$ for $n+2 \leq i \leq 2n$. This implies that 
    \[
    ]C^G_{w_n}[_{k, k} K_p \cap ]C^G_{w_n}[_{m, \bar{0}} K_pt^{-1} \subset ]C^G_{w_n}[_{m, \bar{1}} K_p' .
    \]
    To complete the proof of part (4), it therefore suffices to show that $]C^G_{w_n}[_{k, k} K_p \cap ]C^G_{w_n}[_{m, \bar{1}} \subset ]C^G_{w_n}[_{k, k}K_p'$. But this latter inclusion follows again from a similar explicit calculation as in the proof of Lemma \ref{SupportOfCorrespContainedInCanonicalLemma}, namely: if $[x_0 : \cdots : x_{2n-1}] \in ]C^G_{w_n}[_{k, k} K_p \cap ]C^G_{w_n}[_{m, \bar{1}}$, then there exist 
    \[
    a_1, \dots, a_n \in \mbb{Z}_p \; \text{ and } \; a_{n+2}, \dots, a_{2n} \in p^{\beta}\mbb{Z}_p
    \]
    such that $x_{i-1} \in (a_i + \mathcal{B}_k) \cap \overline{\mathcal{B}}_1$ for $1 \leq i \leq n$, $x_n \in (1+\mathcal{B}_k^{\circ}) \cap (1+\mathcal{B}_m^{\circ})$, and $x_{i-1} \in (a_i + \mathcal{B}_k^{\circ}) \cap \mathcal{B}_m^{\circ}$ for $n+2 \leq i \leq 2n$. Since $\mathcal{B}_k \subset \overline{\mathcal{B}}_1$ and $\mbb{Z}_p \cap \overline{\mathcal{B}}_1 = p\mbb{Z}_p$, we find that $a_i \in p\mbb{Z}_p$, and hence $[x_0 : \cdots : x_{2n-1}] \in ]C^G_{w_n}[_{k, k}K_p'$ (for $m \gg k$) as required.
\end{proof}

We now introduce the cohomological correspondences. We are only interested in such morphisms for automorphic sheaves of classical weight. Let $\kappa \in X^*(T)$ be a $M_{\mathscr{G}}$-dominant weight, and let $V_{\kappa}$ denote the algebraic representation of $M_{\mathscr{G}}$ with highest weight $\kappa$. Let $\psi \colon p_1^{-1}M_{\mathscr{G}, \opn{dR}, K_p}^{\opn{an}} \to p_2^{-1}M_{\mathscr{G}, \opn{dR}, K_p}^{\opn{an}}$ denote the morphism obtained as the composition of the equivariant structure on the torsors $M_{\mathscr{G}, \opn{dR}, -}^{\opn{an}}$ and the action of $\xi \in M_{\mathscr{G}}(\mbb{Q}_p)$ through the torsor structure. Let $\pi_i \colon p_i^{-1}M_{\mathscr{G}, \opn{dR}, K_p}^{\opn{an}} \to \mathcal{S}_{\mathscr{G}, K_p'}$ denote the structural map. Then we let
\[
\psi_{\mathscr{G}, \kappa^*} \colon p_2^*\mathscr{M}_{\mathscr{G}, \kappa^*} = \left( (\pi_2)_* \mathcal{O}_{p_2^{-1}M_{\mathscr{G}, \opn{dR}, K_p}^{\opn{an}}} \otimes V_{\kappa}^* \right)^{M_{\mathscr{G}}^{\opn{an}}} \to \left( (\pi_1)_* \mathcal{O}_{p_1^{-1}M_{\mathscr{G}, \opn{dR}, K_p}^{\opn{an}}} \otimes V_{\kappa}^* \right)^{M_{\mathscr{G}}^{\opn{an}}} = p_1^*\mathscr{M}_{\mathscr{G}, \kappa^*}
\]
denote the morphism induced from $\psi^* \otimes ( \kappa^*(\xi^{-1}) \xi \cdot -)$. Here $\kappa^* = -w_{M_{\mathscr{G}}}^{\opn{max}}\kappa$. Recall that 
\[
R\Gamma(\mathcal{U}^G_{K_p, \bullet}, \mathcal{Z}_{\bullet}; \mathscr{M}_{G, \kappa^*}) = R\Gamma^G_{w_n}(\kappa^*; \beta)^{(-, \dagger)}  \text{ and } R\Gamma(\mathcal{U}^H_{K_p, \bullet}, \mathcal{Z}_{\bullet}; \mathscr{M}_{H, \kappa^*}) = R\Gamma_{\mathcal{Z}_{H, \opn{id}}(p^{\beta})}\left( \mathcal{S}_{H, \diamondsuit}(p^{\beta}), \mathscr{M}_{H, \kappa^*} \right).
\]
Then by the general procedure in \S \ref{CohomologyAndCorrespondencesSection}, we obtain Hecke operators (both of which we denote by $\varphi_{\mathscr{G}}$) on $R\Gamma^G_{w_n}(\kappa^*; \beta)^{(-, \dagger)}$ and $R\Gamma_{\mathcal{Z}_{H, \opn{id}}(p^{\beta})}\left( \mathcal{S}_{H, \diamondsuit}(p^{\beta}), \mathscr{M}_{H, \kappa^*} \right)$ associated with the correspondence (\ref{FrobeniusTopCorrEqn}) and cohomological correspondence $\psi_{\mathscr{G}, \kappa^*}$. These Hecke operators only depend on the restriction of $\psi_{\mathscr{G}, \kappa^*}$ to $\mathcal{U}^{\mathscr{G}}_{K_p', k}$ for any integer $k \geq 1$, therefore it will be useful to consider the following morphism
\[
\varphi \colon \mathcal{U}_{K_p, k+1}^{\mathscr{G}} \xrightarrow{p_1^{-1}} \mathcal{U}^{\mathscr{G}}_{K_p', k+1} \xrightarrow{p_2} \mathcal{U}^{\mathscr{G}}_{K_p, k} .
\]
We can view the restriction of $\psi_{\mathscr{G}, \kappa^*}$ as a morphism $\varphi^* \mathscr{M}_{\mathscr{G}, \kappa^*}|_{\mathcal{U}^{\mathscr{G}}_{K_p, k}} \to \mathscr{M}_{\mathscr{G}, \kappa^*}|_{\mathcal{U}^{\mathscr{G}}_{K_p, k+1}}$, or by adjunction, as a morphism
\begin{equation} \label{AdjunctVersionOfFrobEqn}
\psi_{\mathscr{G}, \kappa^*} \colon \mathscr{M}_{\mathscr{G},\kappa^*}|_{\mathcal{U}^{\mathscr{G}}_{K_p, k}} \to \varphi_*\mathscr{M}_{\mathscr{G},\kappa^*}|_{\mathcal{U}^{\mathscr{G}}_{K_p, k+1}} .
\end{equation}
By Lemma \ref{FrobSystemSupportLemma} we have $\varphi^{-1}(\mathcal{Z} \cap \mathcal{U}^{\mathscr{G}}_{K_p, k} ) = \mathcal{Z} \cap \mathcal{U}^{\mathscr{G}}_{K_p, k+1}$, and the operator is simply induced from the colimit of the morphism (\ref{AdjunctVersionOfFrobEqn}) on cohomology with support in $\mathcal{Z}$, i.e. $\varphi_{\mathscr{G}}$ is the colimit of the maps
\[
R\Gamma_{\mathcal{Z} \cap \mathcal{U}^{\mathscr{G}}_{K_p, k}}\left( \mathcal{U}^{\mathscr{G}}_{K_p, k}, \mathscr{M}_{\mathscr{G}, \kappa^*} \right) \xrightarrow{\psi_{\mathscr{G}, \kappa^*}} R\Gamma_{\mathcal{Z} \cap \mathcal{U}^{\mathscr{G}}_{K_p, k}}\left( \mathcal{U}^{\mathscr{G}}_{K_p, k}, \varphi_*\mathscr{M}_{\mathscr{G}, \kappa^*} \right) = R\Gamma_{\mathcal{Z} \cap \mathcal{U}^{\mathscr{G}}_{K_p, k+1}}\left( \mathcal{U}^{\mathscr{G}}_{K_p, k+1}, \mathscr{M}_{\mathscr{G}, \kappa^*} \right)
\]
where the last equality uses the fact that $\varphi$ is finite.

\subsubsection{A version over the ordinary locus} \label{AVersionOfFrobOverOrdSSec}

We now consider the Frobenius operator acting on Igusa towers. Suppose that $L/\mbb{Q}_p$ is a sufficiently large extension containing $\mu_{p^{\beta+1}}$. Let $J_p$ be as at the end of \S \ref{TopologicalHeckeCorrsSection} and set $J_p' = \xi J_p \xi^{-1} \cap J_p$ and $J_p'' = \xi^{-1} J_p' \xi$. Then we have a correspondence
\[
\begin{tikzcd}
                               & \mathcal{IG}_{\mathscr{G}}/J_p' \arrow[ld, "q_1"'] \arrow[rd, "q_2"] &                                \\
\mathcal{IG}_{\mathscr{G}}/J_p &                                                                      & \mathcal{IG}_{\mathscr{G}}/J_p
\end{tikzcd}
\]
where $q_1$ is the natural forgetful map and $q_2$ is induced from right-translation by $\xi$. One can easily check that $q_1$ is an isomorphism and $q_2$ is finite \'{e}tale of degree $[J_p : J_p''] = [K_p : K_p'']$. We let 
\[
\varphi \colon \mathcal{IG}_{\mathscr{G}}/J_p \xrightarrow{q_1^{-1}} \mathcal{IG}_{\mathscr{G}}/J_p' \xrightarrow{q_2} \mathcal{IG}_{\mathscr{G}}/J_p 
\]
which is finite \'{e}tale of degree $[K_p : K_p'']$. We have a Cartesian diagram
\[
\begin{tikzcd}
\mathcal{IG}_{\mathscr{G}}/J_p \arrow[d, hook] \arrow[r, "\varphi"] & \mathcal{IG}_{\mathscr{G}}/J_p \arrow[d, hook] \\
{\mathcal{U}^{\mathscr{G}}_{K_p, k+1}} \arrow[r, "\varphi"]         & {\mathcal{U}^{\mathscr{G}}_{K_p, k}}          
\end{tikzcd}
\]
and the pullback $\mathscr{M}_{\mathscr{G}, \kappa^*}|_{\mathcal{IG}_{\mathscr{G}}/J_p} \to \varphi_* \mathscr{M}_{\mathscr{G}, \kappa^*}|_{\mathcal{IG}_{\mathscr{G}}/J_p}$ of $\psi_{\mathscr{G}, \kappa^*}$ in (\ref{AdjunctVersionOfFrobEqn}) can be described as follows. Let $U(J_p) \subset J_p$ denote the unipotent part, and $M(J_p)$ the Levi part. Let $\pi \colon \mathcal{IG}_{\mathscr{G}}/U(J_p) \to \mathcal{IG}_{\mathscr{G}}/J_p$ denote the corresponding pro\'{e}tale $M(J_p)$-torsor. Then for an open $V \subset \mathcal{IG}_{\mathscr{G}}/J_p$, we have
\[
\mathscr{M}_{\mathscr{G}, \kappa^*}(V) = \{ f \colon \pi^{-1}(V) \to V_{\kappa}^* : f(x \cdot m) = m^{-1} \cdot f(x) \text{ all } m \in M(J_p) \} .
\]
The morphism $\psi_{\mathscr{G}, \kappa^*} \colon \mathscr{M}_{\mathscr{G}, \kappa^*}(V) \to \mathscr{M}_{\mathscr{G}, \kappa^*}(\varphi^{-1}(V))$ is then simply described as
\[
(\psi_{\mathscr{G}, \kappa^*}f)(x) = \kappa^*(\xi^{-1}) \xi \cdot f(x \cdot \xi) \quad \quad \quad x \in \pi^{-1}(\varphi^{-1}(V)) .
\]

\subsection{Properties of evaluation maps} \label{PropsOfEvMapsInterpolationSubSec}

In this section, we prove an ``interpolation formula'' for the evaluation maps $\opn{Ev}_{\kappa, j, \chi, \beta}^{\dagger, \circ}$ when the conductor of $\chi$ is divisible by $p$. We expect one can also establish a similar formula for unramified characters, however this seems harder than the ramified case. The proofs in this section will involve studying the relations between Hecke operators, Frobenius, and the action of $C^{\opn{la}}(U_{G, \beta}, L)$, and will continually make use of the following strategy. Firstly, one establishes the relation over the ordinary locus. Then one constructs an appropriate $\mathscr{M}_{G, \kappa^*}$-acyclic (or $\mathscr{M}_{H, \sigma_{\kappa}^{[j]}}$-acyclic) cover such that the restriction map from sections over an element of this cover to sections over the ordinary locus is injective. This injectivity property allows one to establish the desired overconvergent version of the relation on the level of \v{C}ech complexes -- the final result is then obtained by passing to cohomology. Although the idea is not too difficult, this section is unfortunately rather technical.

We now state the main result. Fix $(\kappa, j) \in \mathcal{E}$ satisfying Assumption \ref{AssumpOnKJforACchar} (i.e. $\kappa_0 = 0$, $\kappa_{1, \tau_0} + \kappa_{n+1, \tau_0} = n-1$, and $w = \kappa_{2, \tau_0} + \kappa_{2n, \tau_0} = -1$). Let $\chi \in \Sigma_{\kappa, j}(\ide{N}_{\beta})$ be an anticyclotomic character. Let $\chi_{p, \bar{\tau}_0} \colon \mbb{Z}_p^{\times} \to \mbb{C}^{\times}$ denote the finite-order character obtained as the restriction of $\chi$ to $\mathcal{O}_{F_{\overline{\ide{p}}_{\tau_0}}}^{\times} \cong \mbb{Z}_p^{\times}$. We suppose that the conductor of $\chi_{p, \bar{\tau}_0}$ is $p^{\beta'}$ for some $1 \leq \beta' \leq \beta$. Finally, we fix $\varepsilon = (\varepsilon_h)_{h \geq 0}$ a compatible system of $p$-th roots of unity in $\overline{\mbb{Q}}$, i.e. $\varepsilon_h$ is a primitive $p^h$-th root of unity, and $\varepsilon_{h+1}^p = \varepsilon_h$ for all $h \geq 0$. The Gauss sum associated with $\chi_{p, \bar{\tau}_0}$ is
\[
\mathscr{G}(\chi_{p, \bar{\tau}_0}) \defeq \frac{1}{p^{h-\beta'}} \sum_{a \in \left(\mbb{Z}/p^{h}\mbb{Z}\right)^{\times}} \chi_{p, \bar{\tau}_0}(a) \varepsilon_{\beta'}^a
\]
for any $h \geq \beta'$ (which is non-zero because $\chi_{p, \bar{\tau}_0}$ has conductor $p^{\beta'}$). For $i=0, 1$, let $t_i \in T^{G, -}$ be the element which is the identity outside the $\tau_0$-component, and in the $\tau_0$-component is equal to $\opn{diag}(1, \dots, 1, p, \dots, p)$ where there are $n-i$ lots of $p$.

\begin{theorem} \label{MainThmOfPropEvMapSection}
    Let $\eta \in \opn{H}^{n-1}_{w_n}(\kappa^*; \beta)^{(-, \dagger)}$ and suppose that $U_{t_i} \cdot \eta = \alpha_i \eta$ for $i=0, 1$ and some $\alpha_i \in L^{\times}$. Then
    \[
    \opn{Ev}^{\dagger, \circ}_{\kappa, j, \chi, \beta}( \eta ) = \left( \frac{\alpha_0}{\alpha_1} \right)^{\beta'} p^{\beta'\kappa_{n+1, \tau_0}}(1-p^{-1}) \chi( \varpi_{\ide{p}_{\bar{\tau}_0}} )^{-\beta'} \chi_{p, \bar{\tau}_0}(-1) \mathscr{G}(\chi_{p, \bar{\tau}_0}) \opn{Ev}^{\dagger}_{\kappa, j, \chi, \beta}(\eta)
    \]
    where $\varpi_{\ide{p}_{\bar{\tau}_0}} \in \mbb{A}_F^{\times}$ denotes the image of $p$ under the natural embedding $\mbb{Q}_p^{\times} \cong F_{\bar{\ide{p}}_{\tau_0}}^{\times} \hookrightarrow \mbb{A}_F^{\times}$.
\end{theorem}

We will prove this theorem in several steps. First, we recall and introduce some notation. Recall the definition of the weighted indicator function $1_{U_{G, \beta}^{\circ}, \chi_p} \in C^{\opn{la}}(U_{G, \beta}, L)$, namely
\[
1_{U_{G, \beta}^{\circ}, \chi_p}(a_2, \dots, a_{2n}) = \left\{ \begin{array}{cc} \chi_{p, \bar{\tau}_0}(a_{n+1}) & \text{ if } (a_2, \dots, a_{2n}) \in U_{G, \beta}^{\circ} \\ 0 & \text{ otherwise } \end{array} \right. .
\]
We let $f_{\chi_p} \in C^{\opn{la}}(U_{G, \beta}, L)$ denote the function given by
\[
f_{\chi_p}(a_2, \dots, a_{2n}) = \left\{ \begin{array}{cc} \chi_{p, \bar{\tau}_0}(p^{\beta'}a_{n+1}) & \text{ if } a_{n+1} \in p^{-\beta'}\mbb{Z}_p^{\times} \\ 0 & \text{ otherwise } \end{array} \right.
\]
and let $1_{U_{H, \beta'}} \in C^{\opn{la}}(U_{H, \beta}, L)$ denote the indicator function of the subset $U_{H, \beta'} = \left( p^{-\beta'} \mbb{Z}_p \right)^{\oplus n-1} \subset U_{H, \beta}$. We also let $1_{U_{H, \beta'}} \in C^{\opn{la}}(U_{G, \beta}, L)$ denote the function given by $(a_2, \dots, a_{2n}) \mapsto 1_{U_{H, \beta'}}(a_2, \dots, a_n)$. Finally, we set 
\[
1_{U_{G, \beta}^{\circ}, \chi_p}^{\spadesuit} \defeq f_{\chi_p} \cdot 1_{U_{H, \beta'}} \in C^{\opn{la}}(U_{G, \beta}, L) .
\]

\subsubsection{Step 1}

In this step, we describe the eigenvalues of the cohomology class $[\chi]$ under the action of certain $U_p$-Hecke operators. We first describe the transpose of the Hecke operators defined in \S \ref{HeckeCohCorrespSubSec} with respect to Serre duality. Let $t \in T^{H, -}$ and consider the correspondence in (\ref{GlobalSGKCorrepondence}) (in the case $\mathscr{G} = H$). Consider the morphism
\[
 (2\rho_{H, \opn{nc}})(t\langle t \rangle^{-1})\phi_{t, \sigma_{\kappa}^{[j], \vee}}^{-1} \colon p_1^* \mathscr{M}_{H, \sigma_{\kappa}^{[j], \vee}} \to p_2^* \mathscr{M}_{H, \sigma_{\kappa}^{[j], \vee}}
\]
which is well-defined because the morphism in Definition \ref{phitclassicalcohcorrespDefinition}(2) is an isomorphism. Because $p_1p_2^{-1}(\mathcal{U}^H_{K_p, k}) \subset \mathcal{U}^H_{K_p, k}$, we obtain a Hecke correspondence 
\[
U_{t^{-1}} \colon \opn{H}^0_{\opn{id}}(\mathcal{S}_{H, \diamondsuit}(p^{\beta}), \sigma_{\kappa}^{[j], \vee} )^{(+, \dagger)} \to \opn{H}^0_{\opn{id}}(\mathcal{S}_{H, \diamondsuit}(p^{\beta}), \sigma_{\kappa}^{[j], \vee} )^{(+, \dagger)}
\]
as the colimit over $k$ of following composition:
\begin{align*} 
\opn{H}^0\left( \mathcal{U}^H_{K_p, k}, \mathscr{M}_{H, \sigma_{\kappa}^{[j], \vee}} \right) \xrightarrow{\opn{res}} \opn{H}^0\left( p_1p_2^{-1}(\mathcal{U}^H_{K_p, k}), \mathscr{M}_{H, \sigma_{\kappa}^{[j], \vee}} \right) \xrightarrow{p_1^*} \opn{H}^0\left( p_2^{-1}(\mathcal{U}^H_{K_p, k}), p_1^*\mathscr{M}_{H, \sigma_{\kappa}^{[j], \vee}} \right) \\ \xrightarrow{(2\rho_{H, \opn{nc}})(t\langle t \rangle^{-1})\phi_{t, \sigma_{\kappa}^{[j], \vee}}^{-1}} \opn{H}^0\left( p_2^{-1}(\mathcal{U}^H_{K_p, k}), p_2^*\mathscr{M}_{H, \sigma_{\kappa}^{[j], \vee}} \right) \xrightarrow{\opn{Tr}_{p_2}} \opn{H}^0\left( \mathcal{U}^H_{K_p, k}, \mathscr{M}_{H, \sigma_{\kappa}^{[j], \vee}} \right) .
\end{align*} 
Recall that we have a Serre duality pairing 
\[
\langle \cdot, \cdot \rangle \colon \opn{H}^{n-1}_{\opn{id}}(\mathcal{S}_{H, \diamondsuit}(p^{\beta}), \sigma_{\kappa}^{[j]} )^{(-, \dagger)} \times \opn{H}^0_{\opn{id}}(\mathcal{S}_{H, \diamondsuit}(p^{\beta}), \sigma_{\kappa}^{[j], \vee} )^{(+, \dagger)} \to L .
\]
We have the following relation between $U_t$ and $U_{t^{-1}}$.

\begin{lemma} \label{Lem:UtUt-1adjoint}
    For any $t \in T^{H, -}$, one has $\langle U_t \cdot -, - \rangle = \langle -, U_{t^{-1}} \cdot - \rangle$.
\end{lemma}
\begin{proof}
    This follows from the duality of pullback and trace maps (note that the normalisation of $\opn{Tr}_{p_1}$ in the definition of $U_t$ is trivial), and the fact that $\sigma_{\kappa}^{[j]}(t\langle t \rangle^{-1})\phi_{t, \sigma^{[j]}_{\kappa}}$ corresponds to $\sigma_{\kappa}^{[j], \vee}(t^{-1}\langle t \rangle)\phi_{t, \sigma_{\kappa}^{[j], \vee}}^{-1}$ under Verdier duality (see \cite[Proposition 4.2.9]{BoxerPilloni}).
\end{proof}

We have the following property for the cohomology classes associated with anticyclotomic characters.

\begin{lemma} \label{UtinverseEigenvalueLemma}
    Let $t \in T^{H, -}$ be any element which is trivial outside the $\tau_0$-component, and in the $\tau_0$-component is given by $\opn{diag}(y, 1, \dots, 1)$. We have
    \[
     U_{t^{-1}} \cdot \opn{res}[\chi] = \sigma_{\kappa}^{[j], \vee}(\langle t \rangle^{-1}) \hat{\chi}(z_{t^{-1}}) \cdot \opn{res}[\chi]  
    \]
    where $z_{t^{-1}}$ denotes the idele which is $y$ in the factor corresponding to the prime $\overline{\ide{p}}_{\tau_0}$ (and the identity outside this prime). 
\end{lemma}
\begin{proof}
    By analysing the construction of $[\chi]$ in \cite[\S 7]{UFJ}, one can see that the action of 
    \[
     \sigma_{\kappa}^{[j], \vee}(\langle t \rangle) (2\rho_{H, \opn{nc}})(t\langle t \rangle^{-1})^{-1} U_{t^{-1}}
    \]
    is described as follows. We will freely use notation from the proof of Lemma \ref{TheUniversalChiCharacterLemma}. Without loss of generality, we may also assume that $L = \mbb{C}_p$.

    Let $C = \opn{det}(U^p K_p)$ and $C' = \opn{det}(U^p K_p')$, which are compact open subgroups of $\mbf{R}(\mbb{A}_f)$. One can easily see that $C = C'$ because $K_p'$ contains all elements in $T(\mbb{Z}_p)$ which lie in $T^{\diamondsuit}$ modulo $p^{\beta}$. Then we have a commutative diagram of torsors:
    \[
\begin{tikzcd}
{p_2^{-1}M_{H, \opn{dR}, K_p}^{\opn{an}}} \arrow[d, "\phi_{t}^{-1}"'] \arrow[r] & {R_{\opn{dR}, C_p}^{\opn{an}}} \arrow[d, "f"] \\
{p_1^{-1}M_{H, \opn{dR}, K_p}^{\opn{an}}} \arrow[r]                             & {R_{\opn{dR}, C_p}^{\opn{an}}}               
\end{tikzcd}
    \]
    where the horizontal maps are induced from the determinant morphism, and the morphism $f$ is described as 
    \begin{align*}
        \left[ \mbf{R}(\mbb{Q}) \backslash \left( \mbf{R}(\mbb{A}_f^p)/C^p \times \mbf{R}(\mbb{Q}_p) \right) \times R^{\opn{an}} \right] / C_p &\to \left[ \mbf{R}(\mbb{Q}) \backslash \left( \mbf{R}(\mbb{A}_f^p)/C^p \times \mbf{R}(\mbb{Q}_p) \right) \times R^{\opn{an}} \right] / C_p \\
        [x, x', x''] &\mapsto [x, \opn{det}(t^{-1}) x', x''] .
    \end{align*}
    Let $\widehat{\chi}'$ be the unique $p$-adic Hecke character on $\opn{Res}_{F^+/\mbb{Q}}\opn{U}(1)$ such that $\widehat{\chi} = \widehat{\chi}' \circ \mathcal{N}$, and let $\lambda \colon \mbf{R}(\mbb{Q}) \backslash \mbf{R}(\mbb{A}_f) \to L^{\times}$ denote the continuous character given by $\lambda(z_1, z_2) = \widehat{\chi}'(z_2/z_1)$. Then, since $\opn{res}[\chi]$ is constructed as the pullback of the function $\xi \colon R^{\opn{an}}_{\opn{dR}, C_p} \to \mbb{A}^{1, \opn{an}}$, $\xi([x, x', x'']) \defeq \lambda(x)\lambda(x')\mu(x'')$ (where $\mu$ is the unique algebraic character satisfying $\sigma_{\kappa}^{[j], \vee} = \mu^{-1} \circ \opn{det}$), it is enough to understand $\xi \circ f$. But clearly one has
    \[
    \xi \circ f = \lambda(\opn{det}(t^{-1})) \xi = \hat{\chi}(z_{t^{-1}}) \xi .
    \]
     Recall the morphism $p_2 \colon \mathcal{U}^H_{K_p', k+\opn{max}_H(t)} \to \mathcal{U}^H_{K_p, k}$ (given by right-translation by $t$). The above then implies that 
    \[
    \sigma_{\kappa}^{[j], \vee}(\langle t \rangle)(2\rho_{H, \opn{nc}})(t\langle t \rangle^{-1})^{-1} U_{t^{-1}} \cdot \opn{res}[\chi] = \opn{deg}(p_2) \hat{\chi}(z_{t^{-1}}) \cdot \opn{res}[\chi]  
    \]
    by the relation $\opn{Tr}_{p_2} \circ \; p_2^* = \opn{deg}(p_2)$.
    
    The result now follows from the fact that the degree of the morphism $p_2 \colon \mathcal{U}^H_{K_p', k+\opn{max}_H(t)} \to \mathcal{U}^H_{K_p, k}$ is the degree of $q_2 \colon \mathcal{IG}_H/J_p' \to \mathcal{IG}/{J_p}$ (given by right-translation by $t$), where $J_p = J^H_{\diamondsuit}(p^{\beta})$ and $J_p' = t J_p t^{-1} \cap J_p$. But the degree of $q_2$ is given by
    \[
    [J^H_{\diamondsuit}(p^{\beta}) : t^{-1}J_p' t] = [(p^{\beta}T_p(\mu_{p^{\infty}}))^{\oplus n-1} : y^{-1} \cdot (p^{\beta}T_p(\mu_{p^{\infty}}))^{\oplus n-1}] = p^{v_p(y^{-1})(n-1)} = (2\rho_{H, \opn{nc}})( \langle t \rangle t^{-1} ). \qedhere
    \]
\end{proof}

\subsubsection{Step 2}

We now establish an intertwining property between Frobenius and the evaluation maps. We introduce the following matrices. For any $c \in \mbb{Z}_p^{\times}$, let $\xi_c$ denote the diagonal matrix which is the identity outside the $\tau_0$-component, and at the $\tau_0$-component is given by $(c + p^{\beta'}, 1, \dots, 1)$. We let $t_c = w_n^{-1} \xi_c w_n$. For notational brevity, let $L_p = K^H_{\diamondsuit}(p^{\beta})$ and $K_p = K^G_{\opn{Iw}}(p^{\beta})$. For any integer $k \geq 1$, we consider the following map
\[
B_c \colon \mathcal{U}^H_{L_p, k+\beta'} \xrightarrow{\varphi^{\beta'}} \mathcal{U}^H_{L_p, k} \xrightarrow{ \cdot \xi_c} \mathcal{U}^H_{L_p, k} \xrightarrow{\hat{\iota}} \mathcal{U}^G_{K_p, k} \xrightarrow{ \cdot t_c^{-1}} \mathcal{U}^G_{K_p, k} .
\]
We also consider
\[
A \colon \mathcal{U}^H_{L_p, k+\beta'} \xrightarrow{\hat{\iota}} \mathcal{U}^G_{K_p, k+\beta'} \xrightarrow{ \varphi^{\beta'} } \mathcal{U}^G_{K_p, k} .
\]
Note that $B_c$ only depends on $c$ modulo $p^{\beta}$. We will need the following lemma.

\begin{lemma} \label{ABcisTheSameMorphismLemma}
Let $c \in \mbb{Z}_p^{\times}$.
\begin{enumerate}
	\item One has $A^{-1}(\mathcal{Z}_{G, >n+1}(p^{\beta})) = B_c^{-1}(\mathcal{Z}_{G, >n+1}(p^{\beta})) = \mathcal{Z}_{H, \opn{id}}(p^{\beta})$ (where by abuse of notation, $\mathcal{Z}_{H, \opn{id}}(p^{\beta})$ also denotes its intersection with $\mathcal{S}_{H, \diamondsuit}(p^{\beta})$).
	\item Suppose that $U \subset \mathcal{X}_{G, w_n}(p^{\beta})$ is a quasi-compact open subset arising as the adic generic fibre of an open subspace in $\ide{X}_{G, w_n}(p^{\beta})$. Then $A^{-1}(U) = B_c^{-1}(U)$.
\end{enumerate}
\end{lemma}
\begin{proof}
    The first part follows from Lemma \ref{p1p2-1ZisZLemma} and Lemma \ref{FrobSystemSupportLemma}. For the second part we may assume, without loss of generality, that $L$ contains $\mbb{Q}_p^{\opn{cycl}}$. Then, over $\mathcal{X}_{G, w_n}(p^{\beta})$, the morphisms $A$ and $B_c$ are induced from the morphisms 
    \[
    \mathcal{IG}_{H} \to \mathcal{IG}_G
    \]
    given by right-translation by $\gamma \xi^{\beta'}$ and $\xi^{\beta'} \xi_c \gamma \xi_c^{-1}$ respectively, where $\xi$ is the element of the torus in the definition of Frobenius (see \S \ref{DiscussionOfFrobeniusSubSec}) and $\gamma$ is the element introduced at the end of \S \ref{QuotientsOfIgVarSubSec}. For any $c \in \mbb{Z}_p^{\times}$, let 
    \[
    \omega_c = (\omega_{c, i})_{1 \leq i \leq 2n-1} \in \widetilde{\mu_{p^{\infty}}}^{\oplus 2n-1}(\mbb{Z}_p^{\opn{cycl}}) 
    \]
    denote the element satisfying $\omega_{c, i} = 1$ if $i \neq n$, and $\omega_{c, n} = \varepsilon^c$. Let $u_c \in J_{G,\opn{ord}}^+(\mbb{Z}_p^{\opn{cycl}})$ denote the element which is the identity outside the $\tau_0$-component, and at the $\tau_0$-component is given by the block matrix
    \[
    \tbyt{1}{\omega_c}{}{1} .
    \]
    Then an explicit calculation shows that $\xi^{\beta'} \xi_c \gamma \xi_c^{-1} = \gamma \xi^{\beta'} u_c$. The claim now follows from the fact that right-translation by $u_c$ maps any open subscheme of $\mathfrak{IG}_{G}$ into itself.
\end{proof}

We now introduce some morphisms which encode the intertwining property over the ordinary loci.

\begin{definition}
Let $U \subset \mathcal{X}_{G, w_n}(p^{\beta})$ be a quasi-compact open subset arising as the adic generic fibre of an open subspace of $\ide{X}_{G, w_n}(p^{\beta})$. Set $U' = A^{-1}(U) = B_c^{-1}(U)$ (for any $c \in \mbb{Z}_p^{\times}$).
\begin{enumerate}
    \item Let $\rho_{A, U}^{\opn{ord}} \colon \mathscr{M}_{G, \kappa^*}(U) \to \mathscr{M}_{H, \sigma_{\kappa}^{[j]}}(U')$ denote the $L$-linear morphism given by the composition
    \[
    \mathscr{M}_{G, \kappa^*}(U) \xrightarrow{\psi_{G, \kappa^*}^{\beta'}} \mathscr{M}_{G, \kappa^*}((\varphi^{\beta'})^{-1}U) \xrightarrow{\vartheta^{\dagger, \circ}_{\kappa, j+\chi_p, \beta}} \mathscr{M}_{H, \sigma_{\kappa}^{[j]}}(U') .
    \]
    \item Let $\rho_{B_c, U}^{\opn{ord}} \colon \mathscr{M}_{G, \kappa^*}(U) \to \mathscr{M}_{H, \sigma_{\kappa}^{[j]}}(U')$ denote the $L$-linear morphism obtained as the following composition:
    \begin{align*}
        \mathscr{M}_{G, \kappa^*}(U) &\xrightarrow{\phi_{t_c^{-1}, \kappa^*}} \mathscr{M}_{G, \kappa^*}((t_c^{-1})^{-1}U) \\
         &\xrightarrow{\vartheta^{\dagger}_{\kappa, j, \beta}} \mathscr{M}_{H, \sigma_{\kappa}^{[j]}}((t_c^{-1} \circ \hat{\iota})^{-1}U) \\
         &\xrightarrow{\phi_{\xi_c, \sigma_{\kappa}^{[j]}}} \mathscr{M}_{H, \sigma_{\kappa}^{[j]}}((t_c^{-1} \circ \hat{\iota} \circ \xi_c)^{-1}U) \\
         &\xrightarrow{1_{U_{H, \beta'}} \star -} \mathscr{M}_{H, \sigma_{\kappa}^{[j]}}((t_c^{-1} \circ \hat{\iota} \circ \xi_c)^{-1}U) \\
         &\xrightarrow{\psi_{H, \sigma_{\kappa}^{[j]}}^{\beta'}} \mathscr{M}_{H, \sigma_{\kappa}^{[j]}}(U') .
    \end{align*}
    We set 
    \[
     \rho_{B, U}^{\opn{ord}} \defeq \frac{p^{\beta' j_{\tau_0}}}{p^{\beta-\beta'}\mathscr{G}(\chi_{p, \bar{\tau}_0}^{-1})} \sum_{c \in \left(\mbb{Z}/p^{\beta}\mbb{Z}\right)^{\times}} \chi_{p, \bar{\tau}_0}(c)^{-1}  \rho_{B_c, U}^{\opn{ord}} 
    \]
    for a fixed choice of representatives of $\left(\mbb{Z}/p^{\beta}\mbb{Z}\right)^{\times}$.
\end{enumerate}
\end{definition}

The following lemma shows that these morphisms are in fact equal.

\begin{lemma} \label{LemmaRhoAURhoBUordEquals}
We have $\rho_{A, U}^{\opn{ord}} = \rho_{B, U}^{\opn{ord}}$.
\end{lemma}
\begin{proof}
    It suffices to prove the statement after base-changing to an extension $L$ containing $\mbb{Q}_p^{\opn{cycl}}$. Let $\pi_G \colon \mathcal{IG}_{G, w_n}(p^{\beta}) \to \mathcal{X}_{G, w_n}(p^{\beta})$ and $\pi_H \colon \mathcal{IG}_{H, \opn{id}}(p^{\beta}) \to \mathcal{S}_{H, \opn{id}}(p^{\beta})$ denote the structural maps. Let $F \in \mathscr{M}_{G, \kappa^*}(U)$, which we can view as a morphism $F \colon \pi_G^{-1}(U) \to V_{\kappa}^*$ satisfying $F(- \cdot m) = m^{-1} \cdot F(-)$ for all $m \in M^G_{\opn{Iw}}(p^{\beta})$. Recall we have a differential operator $\delta_{\kappa, j} \in V_{\kappa} \otimes C^{\opn{pol}}(\mbb{Q}_p^{\oplus 2n-1}, L)$ which is an eigenvector for the action of $M_H(\mbb{Q}_p)$ with eigencharacter $\sigma_{\kappa}^{[j], -1}$. Fix a basis $\{ v_i \}_{i \in I}$ of $V_{\kappa}$. Then we can write $\delta_{\kappa, j} = \sum_{i \in I} v_i \otimes \delta_i$, for some $\delta_i \in C^{\opn{pol}}(\mbb{Q}_p^{\oplus 2n-1}, L)$. Recall
    \[
    \delta^{\dagger, \circ}_{\kappa, j+\chi_p, \beta} = 1_{U_{G, \beta}^{\circ}, \chi_p} \cdot \delta^{\dagger}_{\kappa, j, \beta} = \sum_{i \in I} u^{-1} \cdot v_i \otimes 1_{U_{G, \beta}^{\circ}, \chi_p} \cdot (u^{-1} \delta_i)|_{U_{G, \beta}} .
    \]
    Consider the action of $\xi^{-1}$ on $C^{\opn{la}}(U_{G, \beta}, L)$ given by $(\xi^{-1} \cdot \phi)(a_2, \dots, a_{2n}) = \phi(pa_2, \dots, pa_{2n})$. Then one can easily verify that
    \[
    \xi^{\beta'} \cdot \left( (\xi^{-\beta'} \cdot \phi) \star h \right) = \phi \star (\xi^{\beta'} \cdot h)
    \]
    for any $h \in \mathcal{O}(\pi_G^{-1}(U))$ and $\phi \in C^{\opn{la}}(U_{G, \beta}, L)$. 

    Since $\xi$ commutes with $u$, we have $(\xi^{-\beta'}, \xi^{-\beta'}) \cdot \delta^{\dagger}_{\kappa, j, \beta} = \sigma_{\kappa}^{[j]}(\xi^{\beta'}) \delta^{\dagger}_{\kappa, j, \beta}$, and hence
    \begin{equation} \label{XiTwistedDaggerDeltaEqn}
    \delta^{\dagger}_{\kappa, j, \beta} = \sigma_{\kappa}^{[j]}(\xi^{\beta'})^{-1} \sum_{i \in I} \left( \xi^{-\beta'} u^{-1} \cdot v_i \otimes  (\xi^{-\beta'} u^{-1} \delta_i)|_{U_{G, \beta}} \right) .
    \end{equation}
    We now compute $\rho^{\opn{ord}}_{A, U}(F)$. Firstly, consider the function $H \defeq \psi_{G, \kappa^*}^{\beta'}(F)$. By the explicit description of the Frobenius morphism over the ordinary locus (see \S \ref{AVersionOfFrobOverOrdSSec}), we have
    \[
    H(-) = \kappa^*(\xi^{-\beta'}) \xi^{\beta'} \cdot F(- \cdot \xi^{\beta'} ) .
    \]
    For any $x \in \pi_H^{-1}(U')$, we therefore have:
    \begin{align*}
    \rho^{\opn{ord}}_{A, U}(F)(x) &= \vartheta^{\dagger, \circ}_{\kappa, j+\chi_p, \beta}(H)(x) \\  &= \sum_{i \in I} \left( \left[ 1_{U_{G, \beta}^{\circ}, \chi_p} \cdot (u^{-1} \delta_i)|_{U_{G, \beta}} \right] \star \langle H(-), u^{-1}v_i \rangle \right)(\hat{\iota}(x)) \\ &= \kappa^*(\xi^{-\beta'}) \sum_{i \in I} \left( \left[ 1_{U_{G, \beta}^{\circ}, \chi_p} \cdot (u^{-1} \delta_i)|_{U_{G, \beta}} \right] \star \xi^{\beta'} \cdot \langle F(-), \xi^{-\beta'}u^{-1}v_i \rangle \right)(\hat{\iota}(x)) \\ &= \kappa^*(\xi^{-\beta'}) \sum_{i \in I} \left( \left[ 1^{\spadesuit}_{U_{G, \beta}^{\circ}, \chi_p} \cdot (\xi^{-\beta'}u^{-1} \delta_i)|_{U_{G, \beta}} \right] \star \langle F(-), \xi^{-\beta'}u^{-1}v_i \rangle \right)(\hat{\iota}(x) \xi^{\beta'}) \\ 
     &= p^{\beta' j_{\tau_0}} \sum_{i \in I} \left( \left[ 1^{\spadesuit}_{U_{G, \beta}^{\circ}, \chi_p} \cdot (u^{-1} \delta_i)|_{U_{G, \beta}} \right] \star \langle F(-), u^{-1}v_i \rangle \right)(\hat{\iota}(x) \xi^{\beta'})
    \end{align*}
    where for the fourth equality we have used the fact that $\xi^{-\beta'} \cdot 1_{U_{G, \beta}^{\circ}, \chi_p} = 1^{\spadesuit}_{U_{G, \beta}^{\circ}, \chi_p}$, and for the fifth equality we have used $\kappa^*(\xi^{-\beta'})\sigma_{\kappa}^{[j]}(\xi^{\beta'}) = p^{\beta' j_{\tau_0}}$ and (\ref{XiTwistedDaggerDeltaEqn}).

    Recall the definitions of $\omega_c$ and $u_c$ from Lemma \ref{ABcisTheSameMorphismLemma}. For any tuple $\underline{d} = (d_2, \dots, d_n) \in (p^{\beta'} \mbb{Z})^{\oplus n-1}$, let 
    \[
    \omega_{\underline{d}} = (\omega_{\underline{d}, i})_{2 \leq i \leq 2n} \in \widetilde{\mu_{p^{\infty}}}^{\oplus 2n-1}(\mbb{Z}_p^{\opn{cycl}})
    \]
    be the element which satisfies $\omega_{\underline{d}, i} = \varepsilon^{d_i}$ if $2 \leq i \leq n$, and $\omega_{\underline{d}, i} = 1$ if $i \geq n+1$. Let $v_{\underline{d}} \in J_{H, \opn{ord}}^+(\mbb{Z}_p^{\opn{cycl}}) \subset J_{G, \opn{ord}}^+(\mbb{Z}_p^{\opn{cycl}})$ denote the element which is the identity outside the $\tau_0$-component, and in the $\tau_0$ is given by the block matrix 
    \[
    \tbyt{1}{\omega_{\underline{d}}}{}{1} .
    \]
    Then, with notation as in \S \ref{DiffOpsCcontsubsec}, we have the following relations: 
    \begin{itemize}
        \item The function $f_{\chi_p}$ satisfies 
        \begin{equation} \label{fchiInTermsOfchisEqn}
        f_{\chi_p} = \frac{1}{p^{\beta-\beta'}\mathscr{G}(\chi_{p, \bar{\tau}_0}^{-1})} \sum_{c \in \left(\mbb{Z}/p^{\beta}\mbb{Z}\right)^{\times}} \chi_{p, \bar{\tau}_0}(c)^{-1} \chi_{\omega_c} .
        \end{equation}
        Indeed, evaluating the right-hand side of (\ref{fchiInTermsOfchisEqn}) at an element $(a_2, \dots, a_{2n}) \in U_{G, \beta}$, we have
    \[
    \frac{1}{p^{\beta-\beta'}\mathscr{G}(\chi_{p, \bar{\tau}_0}^{-1})} \sum_{c \in \left(\mbb{Z}/p^{\beta}\mbb{Z}\right)^{\times}} \chi_{p, \bar{\tau}_0}(c)^{-1} \varepsilon_{\beta}^{c p^{\beta}a_{n+1}} = \left\{ \begin{array}{cc} 0 & \text{ if } a_{n+1} \not\in p^{-\beta'}\mbb{Z}_p^{\times} \\ \chi_{p, \bar{\tau}_0}(p^{\beta'}a_{n+1}) & \text{ if } a_{n+1} \in p^{-\beta'}\mbb{Z}_p^{\times} \end{array} \right.
    \]
    \item The function $1_{U_{H, \beta'}}$ satisfies
    \[
    1_{U_{H, \beta'}} = \frac{1}{p^{(n-1)(\beta - \beta')}}\sum_{\underline{d} \in \left(p^{\beta'}\mbb{Z}/p^{\beta}\mbb{Z}\right)^{\oplus n-1}} \chi_{\omega_{\underline{d}}} .
    \]
    \end{itemize}
    Using these expressions, we see that $\rho^{\opn{ord}}_{A, U}(F)(x)$ is equal to:
    \begin{align*} 
    \frac{p^{\beta' j_{\tau_0}}}{p^{\beta-\beta'}\mathscr{G}(\chi_{p, \bar{\tau}_0}^{-1})} \sum_{c \in \left(\mbb{Z}/p^{\beta}\mbb{Z}\right)^{\times}} \chi_{p, \bar{\tau}_0}(c)^{-1} &\cdot \\    
      &\hspace{-3cm} \left( \frac{1}{p^{(n-1)(\beta - \beta')}}\sum_{\underline{d} \in \left(p^{\beta'}\mbb{Z}/p^{\beta}\mbb{Z}\right)^{\oplus n-1}} \sum_{i \in I} \left(  (u^{-1} \delta_i)|_{U_{G, \beta}} \star \langle F(-), u^{-1}v_i \rangle \right)(\hat{\iota}(x) \xi^{\beta'} u_c v_{\underline{d}}) \right) .
    \end{align*} 
    Now because $\hat{\iota}$ is induced from right-translation by $\gamma$ and commutes with elements in $\opn{Unip}(J_{H, \opn{ord}})$, we see that $\hat{\iota}(x) \xi^{\beta'} u_c v_{\underline{d}} = \hat{\iota}(x \xi^{\beta'} v_{\underline{d}} \xi_c) \xi_c^{-1}$. Set $F' \defeq \phi_{t_c^{-1}, \kappa^*} (F)$ -- explicitly, one has  $F'(-) = \xi_c^{-1} \cdot F(- \cdot \xi_c^{-1})$. Since $(\xi_c^{-1}, \xi_c^{-1}) \cdot \delta^{\dagger}_{\kappa, j, \beta} = \sigma_{\kappa}^{[j]}(\xi_c) \delta^{\dagger}_{\kappa, j, \beta}$, we see that
    \begin{align*}
    \sum_{i \in I} \left(  (u^{-1} \delta_i)|_{U_{G, \beta}} \star \langle F(-), u^{-1}v_i \rangle \right)(\hat{\iota}(x \xi^{\beta'} v_{\underline{d}} \xi_c) \xi_c^{-1}) &= \sigma_{\kappa}^{[j]}(\xi_c) \cdot \vartheta^{\dagger}_{\kappa, j, \beta}(F')( x \xi^{\beta'} v_{\underline{d}} \xi_c ) \\ &= \sigma_{\kappa}^{[j]}(\xi_c) \cdot \vartheta^{\dagger}_{\kappa, j, \beta}(F')( x \xi^{\beta'} v_{\underline{d}} \xi_c ) .
    \end{align*} 
    Therefore, we see that
    \[
     \rho^{\opn{ord}}_{A, U}(F) = \frac{p^{\beta' j_{\tau_0}}}{p^{\beta-\beta'}\mathscr{G}(\chi_{p, \bar{\tau}_0}^{-1})} \sum_{c \in \left(\mbb{Z}/p^{\beta}\mbb{Z}\right)^{\times}} \chi_{p, \bar{\tau}_0}(c)^{-1} \left[ \psi_{H, \sigma_{\kappa}^{[j]}}^{\beta'} \circ (1_{U_{H, \beta'}} \star - ) \circ \phi_{\xi_c, \sigma_{\kappa}^{[j]}} \right]( \vartheta^{\dagger}_{\kappa, j, \beta}(F') ) . 
    \]
    This coincides with $\rho^{\opn{ord}}_{B, U}$ as required.
\end{proof}

We now consider an overconvergent version of this result. We will continually use the following fact: if $f \colon \mathcal{X} \to \mathcal{Y}$ is a finite morphism of adic spaces which is of finite presentation, and $\mathscr{F}$ is a sheaf on $\mathcal{Y}$ which is locally free of finite rank, then $f^{-1}$ sends $\mathscr{F}$-acyclic covers of $\mathcal{Y}$ to $f^*\mathscr{F}$-acyclic covers of $\mathcal{X}$. This follows from a simple application of the projection formula. We will apply this to the setting where $\mathcal{X}$ and $\mathcal{Y}$ are both open subspaces of a third space $\mathcal{C}$, $\mathscr{F}$ is a sheaf on $\mathcal{C}$ locally free of finite rank, and $f$ corresponds to a Hecke operator (for example Frobenius) satisfying the property $f^* \mathscr{F}|_{\mathcal{Y}} \cong \mathscr{F}|_{\mathcal{X}}$. One then sees that $f^{-1}$ sends $\mathscr{F}$-acyclic covers of $\mathcal{Y}$ to $\mathscr{F}$-acyclic covers of $\mathcal{X}$.

Since, in this section, we are working with Shimura varieties rather than the moduli spaces of unitary abelian varieties, we set $\mathcal{S}_{G, w_n}(p^{\beta})_k \defeq \mathcal{X}_{G, w_n}(p^{\beta})_k$ and $\mathcal{S}_{H, \opn{id}}(p^{\beta})_k = \mathcal{X}_{H, \opn{id}}(p^{\beta})_k \cap \mathcal{S}_{H, \diamondsuit}(p^{\beta})$ for any integer $k \geq 1$. We will also frequently use the following notation: if $\{ U_i \}_{i \in I}$ is an open cover of an adic space $\mathcal{X}$, then for any subset $\mbf{I} \subset I$, we set $U_{\mbf{I}} \defeq \cap_{i \in \mbf{I}} U_i$.

We begin with the following proposition.

\begin{proposition} \label{ExistenceOfInjectiveAcyclicOCcoverProp}
    Let $k \geq 1$ be an integer such that $\mathcal{S}_{G, w_n}(p^{\beta})_k \subset \mathcal{U}^G_{K_p, \beta}$. Let $\ide{U} = \{ U_i \}_{i \in I}$ be a finite $\mathscr{M}_{G, \kappa^*}$-acyclic cover of $\mathcal{S}_{G, w_n}(p^{\beta})_k$ as in Lemma \ref{AcyclicityCoverLemma} (i.e. $U_i \in \mathcal{C}_{G,H}$ and $\ide{V} = \{ V_i \}_{i \in I} = \{ U_i - \mathcal{Z}_{G, >n+1}(p^{\beta}) \}_{i \in I}$ is an acyclic cover). Then there exists a sufficiently large integer $k' \geq k$ and a finite collection $\ide{U}' = \{ U_j' \}_{j \in J}$ of quasi-compact open subspaces of $\mathcal{S}_{H, \diamondsuit}(p^{\beta})$ such that:
    \begin{enumerate}
        \item There exists a surjective map $q \colon J \twoheadrightarrow I$ such that for all $j \in J$:
        \[
        U_j' \subset A^{-1}(U_{q(j)}) \cap \bigcap_{c \in \mbb{Z}_p^{\times}} B_c^{-1}(U_{q(j)}) .
        \]
        \item $\mathcal{S}_{H, \opn{id}}(p^{\beta})_{k'} \subset \bigcup_{j \in J} U_j'$ and the covers $\{ U_j' \cap \mathcal{S}_{H, \opn{id}}(p^{\beta})_{k'} \}_{j \in J}$ and $\{ V_j' \cap \mathcal{S}_{H, \opn{id}}(p^{\beta})_{k'} \}_{j \in J}$ are $\mathscr{M}_{H, \sigma_{\kappa}^{[j]}}$-acyclic. Here $V_j' \defeq U_j' - \mathcal{Z}_{H, \opn{id}}(p^{\beta})$.
        \item For any subset $\mbf{J} \subset J$, the restriction map
        \[
        \mathscr{M}_{H, \sigma_{\kappa}^{[j]}}( U'_{\mbf{J}} ) \to \mathscr{M}_{H, \sigma_{\kappa}^{[j]}} ( U'_{\mbf{J}} \cap \mathcal{S}_{H, \opn{id}}(p^{\beta}))
        \]
        is injective.
        \item For any subset $\mbf{J} \subset J$ and $c \in \mbb{Z}_p^{\times}$:
        \begin{itemize}
            \item The map $\vartheta^{\dagger, \circ}_{\kappa, j+\chi_p, \beta}$ of ind-systems induces well-defined maps
            \begin{align*}
                \vartheta^{\dagger, \circ}_{\kappa, j+\chi_p, \beta} \colon \mathscr{M}_{G, \kappa^*}\left( (\varphi^{\beta'})^{-1}(U_{q(\mbf{J})}) \right) &\to \mathscr{M}_{H, \sigma_{\kappa}^{[j]}}(U'_{\mbf{J}}) \\
                \vartheta^{\dagger, \circ}_{\kappa, j+\chi_p, \beta} \colon \mathscr{M}_{G, \kappa^*}\left( (\varphi^{\beta'})^{-1}(V_{q(\mbf{J})}) \right) &\to \mathscr{M}_{H, \sigma_{\kappa}^{[j]}}(V'_{\mbf{J}})
            \end{align*}
            which are functorial in $\mbf{J}$ (via the various restriction maps).
            \item The map $1_{U_{H, \beta'}} \star -$ of ind-systems induces well-defined maps:
            \begin{align*}
                \mathscr{M}_{H, \sigma_{\kappa}^{[j]}}\left( (t_c^{-1} \circ \hat{\iota} \circ \xi_c)^{-1}(U_{q(\mbf{J})}) \right) &\xrightarrow{1_{U_{H, \beta'}} \star -} \mathscr{M}_{H, \sigma_{\kappa}^{[j]}}\left( (t_c^{-1} \circ \hat{\iota} \circ \xi_c)^{-1}(U_{q(\mbf{J})}) \cap \mathcal{S}_{H, \opn{id}}(p^{\beta})_r \right) \\
                \mathscr{M}_{H, \sigma_{\kappa}^{[j]}}\left( (t_c^{-1} \circ \hat{\iota} \circ \xi_c)^{-1}(V_{q(\mbf{J})}) \right) &\xrightarrow{1_{U_{H, \beta'}} \star -} \mathscr{M}_{H, \sigma_{\kappa}^{[j]}}\left( (t_c^{-1} \circ \hat{\iota} \circ \xi_c)^{-1}(V_{q(\mbf{J})}) \cap \mathcal{S}_{H, \opn{id}}(p^{\beta})_r \right)
            \end{align*}
            which are functorial in $\mbf{J}$, for some $r \gg 0$ (not depending on $c$ nor $\mbf{J}$) satisfying 
            \[
            U'_{\mbf{J}} \subset (\varphi^{\beta'})^{-1} \left[ (t_c^{-1} \circ \hat{\iota} \circ \xi_c)^{-1}(U_{q(\mbf{J})}) \cap \mathcal{S}_{H, \opn{id}}(p^{\beta})_r \right] .
            \]
        \end{itemize}
    \end{enumerate}
\end{proposition}
\begin{proof}
    Since $t_c^{-1} \circ \hat{\iota} \circ \xi_c$ only depends on $c$ modulo $p^{\beta}$, we can certainly find an integer $r \gg 0$ such that for any subset $\mbf{I} \subset I$ and $c \in \mbb{Z}_p^{\times}$, the maps $\vartheta^{\dagger, \circ}_{\kappa, j+\chi_p, \beta}$ and $1_{U_{H, \beta'}} \star -$ induce morphisms:
    \begin{align*}
        \mathscr{M}_{G, \kappa^*}\left( (\varphi^{\beta'})^{-1}(U_{\mbf{I}}) \right) &\xrightarrow{\vartheta^{\dagger, \circ}_{\kappa, j+\chi_p, \beta}} \mathscr{M}_{H, \sigma_{\kappa}^{[j]}}\left( (\varphi^{\beta'} \circ \hat{\iota})^{-1}(U_{\mbf{I}}) \cap \mathcal{S}_{H, \opn{id}}(p^{\beta})_r \right) \\
                \mathscr{M}_{G, \kappa^*}\left( (\varphi^{\beta'})^{-1}(V_{\mbf{I}}) \right) &\xrightarrow{\vartheta^{\dagger, \circ}_{\kappa, j+\chi_p, \beta}} \mathscr{M}_{H, \sigma_{\kappa}^{[j]}}\left( (\varphi^{\beta'} \circ \hat{\iota})^{-1}(V_{\mbf{I}}) \cap \mathcal{S}_{H, \opn{id}}(p^{\beta})_r \right) \\
                \mathscr{M}_{H, \sigma_{\kappa}^{[j]}}\left( (t_c^{-1} \circ \hat{\iota} \circ \xi_c)^{-1}(U_{\mbf{I}}) \right) &\xrightarrow{1_{U_{H, \beta'}} \star -} \mathscr{M}_{H, \sigma_{\kappa}^{[j]}}\left( (t_c^{-1} \circ \hat{\iota} \circ \xi_c)^{-1}(U_{\mbf{I}}) \cap \mathcal{S}_{H, \opn{id}}(p^{\beta})_r \right) \\
                \mathscr{M}_{H, \sigma_{\kappa}^{[j]}}\left( (t_c^{-1} \circ \hat{\iota} \circ \xi_c)^{-1}(V_{\mbf{I}}) \right) &\xrightarrow{1_{U_{H, \beta'}} \star -} \mathscr{M}_{H, \sigma_{\kappa}^{[j]}}\left( (t_c^{-1} \circ \hat{\iota} \circ \xi_c)^{-1}(V_{\mbf{I}}) \cap \mathcal{S}_{H, \opn{id}}(p^{\beta})_r \right)
    \end{align*}
    Note that $(\varphi^{\beta'})^{-1}(U_{\mbf{I}}) \in \mathcal{C}_G$ and $(t_c^{-1} \circ \hat{\iota} \circ \xi_c)^{-1}(U_{\mbf{I}}) \in \mathcal{C}_H$ because the morphisms under which we are pulling back the opens extend to integral models over the ordinary loci.

    Set $U_i^{\opn{ord}} \defeq U_i \cap \mathcal{S}_{G, w_n}(p^{\beta})$. This is a quasi-compact open subspace which is the adic generic fibre of an open in $\ide{X}_{G, w_n}(p^{\beta})$, hence we have $W_i \defeq A^{-1}(U_i^{\opn{ord}}) = B_c^{-1}(U_i^{\opn{ord}})$ for any $c \in \mbb{Z}_p^{\times}$ (see Lemma \ref{ABcisTheSameMorphismLemma}). Choose a refinement $\ide{U}'' = \{ U_j'' \}_{j \in J}$ of the cover
    \[
    \left\{ A^{-1}(U_i) \cap \bigcap_{c \in \mbb{Z}_p^{\times}} B_c^{-1}(U_i) \cap \mathcal{S}_{H, \opn{id}}(p^{\beta})_r \cap (\varphi^{\beta'})^{-1}(\mathcal{S}_{H, \opn{id}}(p^{\beta})_r) \right\}_{i \in I} =: \{ Y_i \}_{i \in I}
    \]
    such that $\mathscr{M}_{H, \sigma_{\kappa}^{[j]}}|_{U_j''} = \mathscr{M}_{H, \sigma_{\kappa}^{[j]}}(U_j'') \otimes \mathcal{O}_{U_j''}$ (for all $j \in J$). We may assume that $\ide{U}''$ is finite and consists of quasi-compact open subspaces. Let $q \colon J \twoheadrightarrow I$ be a surjective map such that
    \[
    U_j'' \cap \mathcal{S}_{H, \opn{id}}(p^{\beta}) = U_j'' \cap W_{q(j)} \subset U_j'' \subset Y_{q(j)} .
    \]
    Let $U_j'$ denote the union of all connected components of $U_j''$ which intersect non-trivially with $\mathcal{S}_{H, \opn{id}}(p^{\beta})$. Then, for any subset $\mbf{J} \subset J$, the map $\pi_0(U_{\mbf{J}}' \cap \mathcal{S}_{H, \opn{id}}(p^{\beta})) \twoheadrightarrow \pi_0(U_{\mbf{J}}')$ is surjective. Since $U_{\mbf{J}}'$ is an open subspace of a smooth qcqs adic space, this implies that the natural restriction map $\mathcal{O}(U_{\mbf{J}}') \to \mathcal{O}(U_{\mbf{J}}' \cap \mathcal{S}_{H, \opn{id}}(p^{\beta}))$ is injective. This gives part (3) because $\mathscr{M}_{H, \sigma_{\kappa}^{[j]}}$ is free over $U_{\mbf{J}}'$. Furthermore, $\cup_{i \in I}Y_i = \cup_{j \in J} U_j''$ contains the closure of $\mathcal{S}_{H, \opn{id}}(p^{\beta})$, hence $\cup_{j \in J}U_j'$ must also contain the closure of $\mathcal{S}_{H, \opn{id}}(p^{\beta})$. This implies that there exists an integer $k' \geq k$ such that $\mathcal{S}_{H, \opn{id}}(p^{\beta})_{k'} \subset \cup_{j \in J}U_j'$. The rest of the proposition easily follows.
\end{proof}

We now define the overconvergent morphisms.

\begin{definition}
    Let $\ide{U}$ be a finite cover as in Proposition \ref{ExistenceOfInjectiveAcyclicOCcoverProp}, and let $\ide{U}'$ be the corresponding collection of quasi-compact open subspaces of $\mathcal{S}_{H, \diamondsuit}(p^{\beta})$.
    \begin{enumerate}
        \item For any subset $\mbf{J} \subset J$, let $\rho^{\dagger}_{A, \mbf{J}} \colon \mathscr{M}_{G, \kappa^*}(U_{q(\mbf{J})}) \to \mathscr{M}_{H, \sigma_{\kappa}^{[j]}}(U'_{\mbf{J}})$ denote the continuous $L$-linear morphism (which is functorial in $\mbf{J}$) given by the following composition:
        \[
        \mathscr{M}_{G, \kappa^*}(U_{q(\mbf{J})}) \xrightarrow{\psi_{G, \kappa^*}^{\beta'}} \mathscr{M}_{G, \kappa^*}( (\varphi^{\beta'})^{-1} U_{q(\mbf{J})} ) \xrightarrow{\vartheta^{\dagger, \circ}_{\kappa, j+\chi_p, \beta}} \mathscr{M}_{H, \sigma_{\kappa}^{[j]}}(U'_{\mbf{J}}) .
        \]
        We define $\rho^{\dagger}_{A, \mbf{J}} \colon \mathscr{M}_{G, \kappa^*}(V_{q(\mbf{J})}) \to \mathscr{M}_{H, \sigma_{\kappa}^{[j]}}(V'_{\mbf{J}})$ similarly, which is the unique continuous $L$-linear morphism extending $\rho^{\dagger}_{A, \mbf{J}}$.
        \item Let $c \in \mbb{Z}_p^{\times}$. For any subset $\mbf{J} \subset J$, let $\rho^{\dagger}_{B_c, \mbf{J}} \colon \mathscr{M}_{G, \kappa^*}(U_{q(\mbf{J})}) \to \mathscr{M}_{H, \sigma_{\kappa}^{[j]}}(U'_{\mbf{J}})$ denote the continuous $L$-linear morphism (which is functorial in $\mbf{J}$) given by the following composition:
        \begin{align*}
        \mathscr{M}_{G, \kappa^*}(U_{q(\mbf{J})}) &\xrightarrow{\phi_{t_c^{-1}, \kappa^*}} \mathscr{M}_{G, \kappa^*}((t_c^{-1})^{-1}U_{q(\mbf{J})}) \\
         &\xrightarrow{\vartheta^{\dagger}_{\kappa, j, \beta}} \mathscr{M}_{H, \sigma_{\kappa}^{[j]}}((t_c^{-1} \circ \hat{\iota})^{-1}U_{q(\mbf{J})}) \\
         &\xrightarrow{\phi_{\xi_c, \sigma_{\kappa}^{[j]}}} \mathscr{M}_{H, \sigma_{\kappa}^{[j]}}((t_c^{-1} \circ \hat{\iota} \circ \xi_c)^{-1}U_{q(\mbf{J})}) \\
         &\xrightarrow{1_{U_{H, \beta'}} \star - } \mathscr{M}_{H, \sigma_{\kappa}^{[j]}}\left( (t_c^{-1} \circ \hat{\iota} \circ \xi_c)^{-1}(U_{q(\mbf{J})}) \cap \mathcal{S}_{H, \opn{id}}(p^{\beta})_r \right) \\
         &\xrightarrow{\psi_{H, \sigma_{\kappa}^{[j]}}^{\beta'}} \mathscr{M}_{H, \sigma_{\kappa}^{[j]}}(U'_{\mbf{J}}) .
    \end{align*}
    We define $\rho^{\dagger}_{B_c, \mbf{J}} \colon \mathscr{M}_{G, \kappa^*}(V_{q(\mbf{J})}) \to \mathscr{M}_{H, \sigma_{\kappa}^{[j]}}(V'_{\mbf{J}})$ similarly, which is the unique continuous $L$-linear morphism extending $\rho^{\dagger}_{B_c, \mbf{J}}$.
    \end{enumerate}
    We set 
    \[
    \rho^{\dagger}_{B, \mbf{J}} \defeq \frac{p^{\beta' j_{\tau_0}}}{p^{\beta-\beta'}\mathscr{G}(\chi_{p, \bar{\tau}_0}^{-1})} \sum_{c \in \left(\mbb{Z}/p^{\beta}\mbb{Z}\right)^{\times}} \chi_{p, \bar{\tau}_0}(c)^{-1}  \rho^{\dagger}_{B_c, \mbf{J}} 
    \]
    for a fixed set of representatives of $\left( \mbb{Z}/p^{\beta} \mbb{Z} \right)^{\times}$.
\end{definition}

\begin{lemma} \label{LemmaRhoDaggerAequalsB}
    One has $\rho^{\dagger}_{A, \mbf{J}} = \rho^{\dagger}_{B, \mbf{J}}$.
\end{lemma}
\begin{proof}
    Since the morphisms $\rho^{\dagger}_{A, \mbf{J}}, \rho^{\dagger}_{B, \mbf{J}} \colon \mathscr{M}_{G, \kappa^*}(V_{q(\mbf{J})}) \to \mathscr{M}_{H, \sigma_{\kappa}^{[j]}}(V'_{\mbf{J}})$ uniquely extend 
    \begin{equation} \label{RhoARhoBMapsEqn}
    \rho^{\dagger}_{A, \mbf{J}}, \rho^{\dagger}_{B, \mbf{J}} \colon \mathscr{M}_{G, \kappa^*}(U_{q(\mbf{J})}) \to \mathscr{M}_{H, \sigma_{\kappa}^{[j]}}(U'_{\mbf{J}})
    \end{equation}
    it suffices to prove the morphisms in (\ref{RhoARhoBMapsEqn}) are equal. Set $U_{q(\mbf{J})}^{\opn{ord}} = U_{q(\mbf{J})} \cap \mathcal{S}_{G, w_n}(p^{\beta})$ which is a quasi-compact open subspace which is the adic generic fibre of an open in $\ide{X}_{G, w_n}(p^{\beta})$. Set $(U'_{\mbf{J}})^{\opn{ord}} = U_{\mbf{J}}' \cap \mathcal{S}_{H, \opn{id}}(p^{\beta})$ and note that $(U'_{\mbf{J}})^{\opn{ord}} \subset A^{-1}(U_{q(\mbf{J})})$. Then for $\bullet = A, B$ we have a commutative diagram:
    \[
\begin{tikzcd}
{\mathscr{M}_{G, \kappa^*}(U^{\opn{ord}}_{q(\mbf{J})})} \arrow[r, "{\opn{res} \circ \rho^{\opn{ord}}_{\bullet, U^{\opn{ord}}_{q(\mbf{J})}}}"] & {\mathscr{M}_{H, \sigma_{\kappa}^{[j]}}((U'_{\mbf{J}})^{\opn{ord}})}   \\
{\mathscr{M}_{G, \kappa^*}(U_{q(\mbf{J})})} \arrow[u] \arrow[r, "{\rho^{\dagger}_{\bullet, \mbf{J}}}"]                                        & {\mathscr{M}_{H, \sigma_{\kappa}^{[j]}}(U'_{\mbf{J}})} \arrow[u, hook]
\end{tikzcd}
    \]
    where the vertical arrows are given by restriction. The result now follows from Lemma \ref{LemmaRhoAURhoBUordEquals} because the right-hand vertical map is injective (see Proposition \ref{ExistenceOfInjectiveAcyclicOCcoverProp}).
\end{proof}

We now consider the analogous morphisms on \v{C}ech complexes. More precisely, let $\ide{U}'_{k'} = \{ U_j' \cap \mathcal{S}_{H, \opn{id}}(p^{\beta})_{k'} \}_{j \in J}$ denote the $\mathscr{M}_{H, \sigma_{\kappa}^{[j]}}$-acyclic cover as in Proposition \ref{ExistenceOfInjectiveAcyclicOCcoverProp}. Then, for $\bullet = A, B$, we obtain a commutative diagram of \v{C}ech complexes
\[
\begin{tikzcd}
{\opn{Cech}(\mathscr{M}_{G, \kappa^*}; \ide{U} )} \arrow[r] \arrow[d, "\opn{res} \circ \rho^{\dagger}_{\bullet}"'] & {\opn{Cech}(\mathscr{M}_{G, \kappa^*}; \ide{V} )} \arrow[d, "\opn{res} \circ \rho^{\dagger}_{\bullet}"] \\
{\opn{Cech}(\mathscr{M}_{H, \sigma_{\kappa}^{[j]}}; \ide{U}'_{k'} )} \arrow[r]                     & {\opn{Cech}(\mathscr{M}_{H, \sigma_{\kappa}^{[j]}}; \ide{V}'_{k'} )}                   
\end{tikzcd}
\]
where $\opn{res}$ denotes restriction from elements of $\ide{U}'$ (resp. $\ide{V}'$) to $\ide{U}'_{k'}$ (resp. $\ide{V}'_{k'}$). Here $\rho^{\dagger}_{\bullet}$ denotes the induced morphism of complexes built up from the maps $\rho^{\dagger}_{\bullet, \mbf{J}}$ for different subsets $\mbf{J} \subset J$. By passing to cohomology of the mapping fibres of the horizontal arrows, we obtain morphisms
\[
\opn{H}^{n-1}(\rho^{\dagger}_{\bullet}) \colon \opn{H}^{n-1}_{\mathcal{Z}_{G,>n+1}(p^{\beta})}\left( \mathcal{S}_{G, w_n}(p^{\beta})_k, \mathscr{M}_{G, \kappa^*} \right) \to \opn{H}^{n-1}_{\mathcal{Z}_{H, \opn{id}}(p^{\beta})}\left( \mathcal{S}_{H, \opn{id}}(p^{\beta})_{k'}, \mathscr{M}_{H, \sigma_{\kappa}^{[j]}} \right) .
\]
By Lemma \ref{LemmaRhoDaggerAequalsB}, one has $\opn{H}^{n-1}(\rho^{\dagger}_{A}) = \opn{H}^{n-1}(\rho^{\dagger}_{B})$. On the other hand, using the acyclicity property of finite morphisms discussed just before Proposition \ref{ExistenceOfInjectiveAcyclicOCcoverProp}:
\begin{itemize}
    \item One has $\opn{H}^{n-1}(\rho^{\dagger}_{A}) = \vartheta^{\dagger, \circ}_{\kappa, j+\chi_p, \beta} \circ \varphi_G^{\beta'}$.
    \item One has 
    \[
    \opn{H}^{n-1}(\rho^{\dagger}_{B}) = \frac{p^{\beta' j_{\tau_0}}}{p^{\beta-\beta'}\mathscr{G}(\chi_{p, \bar{\tau}_0}^{-1})} \sum_{c \in \left(\mbb{Z}/p^{\beta}\mbb{Z}\right)^{\times}} \chi_{p, \bar{\tau}_0}(c)^{-1}  \left( \varphi_H^{\beta'} \circ (1_{U_{H, \beta'}} \star -) \circ U_{\xi_c} \circ \vartheta_{\kappa, j, \beta}^{\dagger} \circ U_{t_c^{-1}} \right) .
    \]
\end{itemize}
We therefore obtain the following corollary.

\begin{corollary} \label{CorollaryEvCircEquals}
    For any $\eta \in \opn{H}^{n-1}_{w_n}(\kappa^*; \beta)^{(-, \dagger)}$, one has
    \[
    \opn{Ev}^{\dagger, \circ}_{\kappa, j, \chi, \beta}( \varphi_G^{\beta'} \eta ) = p^{\beta'\kappa_{n+1, \tau_0} - 1}(p-1) \chi(z_{\xi^{-\beta'}}) \chi_{p, \bar{\tau}_0}(-1) \mathscr{G}(\chi_{p, \bar{\tau}_0}) \opn{Ev}^{\dagger}_{\kappa, j, \chi, \beta}(\eta) .
    \]
\end{corollary}
\begin{proof}
     We first note that $U_{t_c^{-1}}$ acts trivially on the cohomology group $\opn{H}^{n-1}_{w_n}(\kappa^*; \beta)^{(-, \dagger)}$. With notation as in \S \ref{DiscussionOfFrobeniusSubSec}, let $\varphi_H^t$ denote the colimit (after passing to cohomology) of the morphisms $\opn{deg}(\varphi)^{-1}\opn{Tr}_{\varphi} \circ \psi_{H, \sigma_{\kappa}^{[j]}}^{-1}$ acting on $\opn{H}^{n-1}_{\opn{id}}(\mathcal{S}_{H, \diamondsuit}(p^{\beta}), \sigma_{\kappa}^{[j]})^{(-, \dagger)} = \varinjlim_k \opn{H}^{n-1}_{\mathcal{Z}}(\mathcal{U}^H_{K_p, k}, \mathscr{M}_{H, \sigma_{\kappa}^{[j]}})$. We define:
    \begin{itemize}
        \item $F_1$ to be the Hecke operator acting on $\opn{H}^0_{\opn{id}}\left( \mathcal{S}_{H, \diamondsuit}(p^{\beta}), \sigma_{\kappa}^{[j], \vee} \right)^{(+, \dagger)}$ given by $U_{\xi^{-1}}$. Recall from Lemma \ref{Lem:UtUt-1adjoint} that $U_{\xi^{-1}}$ is adjoint to the operator $U_{\xi}$ acting on $\opn{H}^{n-1}_{\opn{id}}(\mathcal{S}_{H, \diamondsuit}(p^{\beta}), \sigma_{\kappa}^{[j]})^{(-, \dagger)}$ via the Serre duality pairing. Since $\xi \in T^{H, -}$, we also have $U_{\xi} = \varphi_H$ (this can be seen by tracing through the definitions of the Hecke action coming from $T^{H, -}$ and the Frobenius action). In particular, since pullback and pushforward are adjoint under the Serre duality pairing, we see that $U_{\xi^{-1}}$ is equal to the composition:
        \begin{align*}
            \varinjlim_k \opn{H}^0(\mathcal{U}^H_{K_p, k+1}, \mathscr{M}_{H, \sigma_{\kappa}^{[j], \vee}}) \xrightarrow{\psi_{H, \sigma_{\kappa}^{[j], \vee}}^{-1}} \varinjlim_k \opn{H}^0(\mathcal{U}^H_{K_p, k+1}, \varphi^*\mathscr{M}_{H, \sigma_{\kappa}^{[j], \vee}}) \xrightarrow{\opn{Tr}_{\varphi}} \varinjlim_k \opn{H}^0(\mathcal{U}^H_{K_p, k}, \mathscr{M}_{H, \sigma_{\kappa}^{[j], \vee}}) .
        \end{align*}
        \item $F_2$ to be the operator acting on $\opn{H}^0_{\opn{id}}\left( \mathcal{S}_{H, \diamondsuit}(p^{\beta}), \sigma_{\kappa}^{[j], \vee} \right)^{(+, \dagger)}$ given by $\opn{deg}(\varphi)^{-1} \psi_{H, \sigma_{\kappa}^{[j], \vee}} \circ \varphi^*$. Again, since pullback and pushforward are adjoint under the Serre duality pairing, we see that $F_2$ is adjoint to $\varphi_H^t$. Moreover, the relation $\opn{Tr}_{\varphi} \circ \; \varphi^* = \opn{deg}(\varphi)$ (the map $\varphi$ is finite \'{e}tale by Lemma \ref{FrobSystemSupportLemma}) implies that $F_1 \circ F_2 = \opn{id}$.
    \end{itemize}
    For $0 \leq m \leq \beta$, we see that $F_1^m$ and $F_2^m$ are adjoint under Serre duality to $\varphi_H^{m}$ and $(\varphi_H^t)^{m}$ respectively, $F_1^m = U_{\xi^{-m}}$, and $F_1^m \circ F_2^m = \opn{id}$.
    
    By Lemma \ref{UtinverseEigenvalueLemma}, we therefore see that
    \[
    F_1^m(\opn{res}[\chi]) = \hat{\chi}(z_{\xi^{-m}}) \opn{res}[\chi], \quad \quad F_2^m(\opn{res}[\chi]) = \hat{\chi}(z_{\xi^{-m}})^{-1} \opn{res}[\chi] .
    \]
    Note that $\hat{\chi}(z_{\xi^{-\beta'}}) = \chi(z_{\xi^{-\beta'}}) p^{\beta'(\kappa_{n+1, \tau_0} - j_{\tau_0})}$. By Remark \ref{FrobHIndicatorRelationRemark}, we have 
    \[
    (\varphi^t_H)^{\beta - \beta'} \circ (1_{U_{H, \beta'}} \star -) = (\varphi^t_H)^{\beta - \beta'}
    \]
    hence $\langle 1_{U_{H, \beta'}} \star x, \opn{res}[\chi] \rangle = \langle x, \opn{res}[\chi] \rangle$ for any $x \in \opn{H}^{n-1}_{\opn{id}}(\mathcal{S}_{H, \diamondsuit}(p^{\beta}), \sigma_{\kappa}^{[j]} )^{(-, \dagger)}$. We also have  $U_{\xi_c^{-1}} \cdot \opn{res}[\chi] = \sigma_{\kappa}^{[j], \vee}(\langle \xi_c \rangle^{-1})\hat{\chi}(z_{\xi_c^{-1}}) \cdot \opn{res}[\chi] = \chi_{p, \tau_0}(c) \cdot \opn{res}[\chi]$. 

    Putting this all together, we find that $\opn{Ev}^{\dagger, \circ}_{\kappa, j, \chi, \beta}( \varphi_G^{\beta'} \eta )$ is equal to
    \begin{align*} 
    \frac{p^{\beta' j_{\tau_0}}}{p^{\beta-\beta'}\mathscr{G}(\chi_{p, \bar{\tau}_0}^{-1})} &\sum_{c \in \left(\mbb{Z}/p^{\beta}\mbb{Z}\right)^{\times}} \chi_{p, \bar{\tau}_0}(c)^{-1} \langle \left( \varphi_H^{\beta'} \circ (1_{U_{H, \beta'}} \star -) \circ U_{\xi_c} \circ \vartheta_{\kappa, j, \beta}^{\dagger} \circ U_{t_c^{-1}} \right)(\eta), \opn{res}[\chi] \rangle \\
    &= \frac{p^{\beta' j_{\tau_0}}}{p^{\beta-\beta'}\mathscr{G}(\chi_{p, \bar{\tau}_0}^{-1})} \sum_{c \in \left(\mbb{Z}/p^{\beta}\mbb{Z}\right)^{\times}} \chi(z_{\xi^{-\beta'}}) p^{\beta'(\kappa_{n+1, \tau_0} - j_{\tau_0})} \langle \vartheta^{\dagger}_{\kappa, j, \beta}(\eta), \opn{res}[\chi] \rangle \\
    &= \frac{p^{\beta'(1+\kappa_{n+1, \tau_0})}(p-1)}{p \; \mathscr{G}(\chi_{p, \bar{\tau}_0}^{-1})} \chi(z_{\xi^{-\beta'}}) \opn{Ev}^{\dagger}_{\kappa, j, \chi, \beta}(\eta) .
    \end{align*} 
    Now using the relation 
    \[
    \mathscr{G}(\chi_{p, \bar{\tau}_0}^{-1})^{-1} = \chi_{p, \bar{\tau}_0}(-1) p^{-\beta'} \mathscr{G}(\chi_{p, \bar{\tau}_0})
    \]
    the result follows.
\end{proof}

\subsubsection{Step 3}

To complete the proof of Theorem \ref{MainThmOfPropEvMapSection}, we need a relation between Hecke operators, Frobenius and the action of indicator functions. As above, we will establish this over the ordinary locus and then use the injectivity of certain restriction maps to deduce the overconvergent version. For $i=0, 1$, recall that $t_i \in T^{G, -}$ denotes the element which is the identity outside the $\tau_0$-component, and in the $\tau_0$-component is equal to
\[
\opn{diag}(1, \dots, 1, p, \dots, p)
\]
where there are $n-i$ lots of $p$. We will consider the Hecke operators associated with these elements as in \S \ref{TopologicalHeckeCorrsSection}. However to be able to establish a relation between them, we need to view them as a cohomological correspondence for the same (topological) correspondence. We first consider the correspondences over the ordinary locus. Let $t'_i = w_n t_i w_n^{-1} \in J_{G, \opn{ord}}$. With notation as in \S \ref{TopologicalHeckeCorrsSection}, we have a commutative diagram:
\begin{equation} \label{BigIGDiagramT0T1}
\begin{tikzcd}
                   &                                                                & {\mathcal{IG}_G/J_{p, 1}'} \arrow[ld] \arrow[lldd, bend right] \arrow[rd, "\cdot \xi^{\beta'}"'] \arrow[rrdd, "\cdot (t_1')^{\beta'}", bend left] &                                                                            &                    \\
                   & \mathcal{IG}_G/J_p \arrow[ld] \arrow[rd, "\cdot \xi^{\beta'}"] &                                                                                                                                                   & {\mathcal{IG}_G/J_{p, 0}'} \arrow[ld] \arrow[rd, "\cdot (t_0')^{\beta'}"'] &                    \\
\mathcal{IG}_G/J_p &                                                                & \mathcal{IG}_G/J_p                                                                                                                                &                                                                            & \mathcal{IG}_G/J_p
\end{tikzcd}
\end{equation} 
where $J_{p, i}' = (t_i')^{\beta'} J_p (t_i')^{-\beta'} \cap J_p$ and the unlabelled arrows are the natural maps. The inner square is not Cartesian in general, however $\mathcal{IG}_G/J'_{p, 1}$ is an open and closed subspace of the Cartesian product $\mathcal{IG}_G/J_p \times_{\mathcal{IG}_G/J_p} \mathcal{IG}_G/J_{p, 0}'$ (with respect to the maps in the bottom half of the square). We will denote the left-hand (resp. right-hand) curved arrow by $q_1$ (resp. $q_2$). 

We claim that we have an overconvergent version of the above diagram. To see this, consider the following compact open subgroups. Let $K_p = K^G_{\opn{Iw}}(p^{\beta})$, $K_{p, i}' = t_i^{\beta'} K_p t_i^{-\beta'} \cap K_p$, $K_{p, \varphi^{\beta'}}' = (w_n^{-1}\xi^{\beta'} w_n) K_p (w_n^{-1}\xi^{\beta'} w_n)^{-1} \cap K_p$, and $K_p'' = K_{p, \varphi^{\beta'}}' \cap (w_n^{-1}\xi^{\beta'} w_n) K_{p, 0}' (w_n^{-1}\xi^{\beta'} w_n)^{-1}$. Note that $K_p'' \subset K_{p, 1}'$. We have a commutative diagram:
\[
\begin{tikzcd}
                       &                                                                                    & {\mathcal{S}_{G, K_p''}} \arrow[ld] \arrow[rd, "\cdot w_n^{-1} \xi^{\beta'} w_n"'] \arrow[rrdd, "\cdot t_1^{\beta'}", bend left] &                                                                  &                        \\
                       & {\mathcal{S}_{G, K_{p, \varphi^{\beta'}}'}} \arrow[ld] \arrow[rd, "\cdot w_n^{-1} \xi^{\beta'} w_n"] &                                                                                                                & {\mathcal{S}_{G, K_{p, 0}'}} \arrow[ld] \arrow[rd, "\cdot t_0^{\beta'}"'] &                        \\
{\mathcal{S}_{G, K_p}} &                                                                                    & {\mathcal{S}_{G, K_p}}                                                                                         &                                                                  & {\mathcal{S}_{G, K_p}}
\end{tikzcd}
\]
and $\mathcal{S}_{G, K_p''}$ is an open and closed subspace in the Cartesian product $\mathcal{S}_{G, K_{p, \varphi^{\beta'}}'} \times_{\mathcal{S}_{G, K_p}} \mathcal{S}_{G, K_{p, 0}'}$ formed from the maps in the bottom half of the middle square. One can easily show that the natural map $\mathcal{S}_{G, K_p''} \to \mathcal{S}_{G, K_{p, 1}'}$ induces an isomorphism $a \colon \mathcal{U}^G_{K_p'', k} \xrightarrow{\sim} \mathcal{U}^G_{K_{p, 1}', k}$. We obtain a commutative diagram:
\begin{equation} \label{BigCompositionOnOCnhoods}
\begin{tikzcd}
                                  &                                                                                                               & {\mathcal{U}^G_{K_{p, 1}', k+2\beta'}} \arrow[ld] \arrow[rd, "(\star)"'] \arrow[rrdd, "\cdot t_1^{\beta'}", bend left] \arrow[lldd, bend right] &                                                                                    &                          \\
                                  & {\mathcal{U}^G_{K_{p, \varphi^{\beta'}}', k+2\beta'}} \arrow[ld] \arrow[rd, "\cdot w_n^{-1} \xi^{\beta'}w_n"] &                                                                                                                                                & {\mathcal{U}^G_{K_{p, 0}', k+\beta'}} \arrow[ld] \arrow[rd, "\cdot t_0^{\beta'}"'] &                          \\
{\mathcal{U}^G_{K_p, k+ 2\beta'}} &                                                                                                               & {\mathcal{U}^G_{K_p, k+\beta'}}                                                                                                                &                                                                                    & {\mathcal{U}^G_{K_p, k}}
\end{tikzcd}
\end{equation} 
where the map $(\star)$ is the composition of $a^{-1}$ with $\cdot w_n^{-1} \xi^{\beta'} w_n$. We will also denote the left-hand (resp. right-hand map) in the above diagram by $q_1$ (resp. $q_2$). We have a commutative diagram
\[
\begin{tikzcd}
\mathcal{IG}_G/J_p \arrow[d, hook] & {\mathcal{IG}_G/J_{p, 1}'} \arrow[l, "q_1"'] \arrow[r, "q_2"] \arrow[d, hook] & \mathcal{IG}_G/J_p \arrow[d, hook] \\
{\mathcal{U}_{K_p, k+2\beta'}^G}   & {\mathcal{U}^G_{K_{p, 1}', k+2\beta'}} \arrow[l, "q_1"'] \arrow[r, "q_2"]     & {\mathcal{U}^G_{K_p, k}}          
\end{tikzcd}
\]
and the left-hand square is Cartesian for $k$ sufficiently large (and the vertical maps are induced from right-translation by $w_n$). 

Consider the inner diamond of (\ref{BigCompositionOnOCnhoods}) (and identifying $\mathcal{U}^G_{K_{p, \varphi^{\beta'}}, k+2\beta'} \cong \mathcal{U}^G_{K_p, k+2\beta'}$)
\[
\begin{tikzcd}
                                                                & {\mathcal{U}^G_{K_{p, 1}', k+2\beta'}} \arrow[ld, "q_1"'] \arrow[rd, "\lambda"] &                                                         \\
{\mathcal{U}^G_{K_p, k+2\beta'}} \arrow[rd, "\varphi^{\beta'}"] &                                                                                 & {\mathcal{U}^G_{K_{p, 0}', k+\beta'}} \arrow[ld, "\mu"] \\
                                                                & {\mathcal{U}^G_{K_p, k+\beta'}}                                                 &                                                        
\end{tikzcd}
\]
Note that all the maps in this diagram are finite \'{e}tale. Let $\ide{U} = \{ U_i \}_{i \in I}$ be a finite $\mathscr{M}_{G, \kappa^*}$-acyclic cover of $\mathcal{X}_{G, w_n}(p^{\beta})_k \subset \mathcal{U}^G_{K_p, k_0 + \beta'}$ as in Lemma \ref{AcyclicityCoverLemma}, where $k_0 \geq 1$ is an auxiliary integer. Since $\varphi^{\beta'}$ is finite and integral over the ordinary locus, we see that $(\varphi^{\beta'})^{-1}(U_i) \in \mathcal{C}_{G}$ and the cover $\{ (\varphi^{\beta'})^{-1}(U_i) \}_{i \in I}$ is $(\varphi^{\beta'})^*\mathscr{M}_{G, \kappa^*} \cong \mathscr{M}_{G, \kappa^*}$-acyclic. For any subset $\mbf{I} \subset I$, we consider the following two maps:
\begin{itemize}
    \item The composition $\alpha_{0, \mbf{I}}$ given by 
    \[
    \mu^*\mathscr{M}_{G, \kappa^*}(\mu^{-1}(U_{\mbf{I}})) \xrightarrow{(-w_n^{-1}\rho + \rho)(t_0^{\beta'}) \opn{Tr}_{\mu}} \mathscr{M}_{G, \kappa^*}(U_{\mbf{I}}) \xrightarrow{\psi^{\beta'}_{G, \kappa^*}} \mathscr{M}_{G, \kappa^*}((\varphi^{\beta'})^{-1}(U_{\mbf{I}})) .
    \]
    \item The composition $\alpha_{1, \mbf{I}}$ given by
    \begin{align*} 
    \mu^*\mathscr{M}_{G, \kappa^*}(\mu^{-1}(U_{\mbf{I}})) &\xrightarrow{\lambda^*} \lambda^* \mu^*\mathscr{M}_{G, \kappa^*}(\lambda^{-1} \mu^{-1}(U_{\mbf{I}})) \\ &= q_1^* (\varphi^{\beta'})^* \mathscr{M}_{G, \kappa^*}(q_1^{-1}(\varphi^{\beta'})^{-1}(U_{\mbf{I}})) \\ &\xrightarrow{(-w_n^{-1}\rho + \rho)(t_1^{\beta'}) \opn{Tr}_{q_1}} (\varphi^{\beta'})^* \mathscr{M}_{G, \kappa^*}((\varphi^{\beta'})^{-1}(U_{\mbf{I}})) \\ &\xrightarrow{\psi^{\beta'}_{G, \kappa^*}} \mathscr{M}_{G, \kappa^*}((\varphi^{\beta'})^{-1}(U_{\mbf{I}})).
    \end{align*} 
\end{itemize}

Let $\mbb{I}_{\beta'} \in C^{\opn{la}}(U_{G, \beta}, L)$ denote the indicator function associated with the subset $(p^{\beta' - \beta}\mbb{Z}_p)^{\oplus n} \oplus \mbb{Z}_p^{\oplus n-1} \subset U_{G, \beta}$. Then, following the same strategy as in the proof of Proposition \ref{ExistenceOfInjectiveAcyclicOCcoverProp}, there exists an integer $k' \geq k$ and a finite collection $\ide{U}' = \{ U_j' \}_{j \in J}$ of quasi-compact open subspaces of $\mathcal{U}^G_{K_p, k_0 + 2\beta'}$, with the following properties:
\begin{itemize}
    \item One has a surjective map $q \colon J \twoheadrightarrow I$ such that: for any $j \in J$
    \[
    U'_j \subset (\varphi^{\beta'})^{-1}(U_{q(j)}) .
    \]
    Furthermore, we have an action map $\mbb{I}_{\beta'} \star - \colon \mathscr{M}_{G, \kappa^*}((\varphi^{\beta'})^{-1}(U_{q(\mbf{J})})) \to \mathscr{M}_{G, \kappa^*}(U'_{\mbf{J}})$ for any subset $\mbf{J} \subset J$.
    \item We have $\mathcal{X}_{G, w_n}(p^{\beta})_{k'} \subset \bigcup_{j \in J} U_j'$ and $\ide{U}'_{k'} = \{ U_j' \cap \mathcal{X}_{G, w_n}(p^{\beta})_{k'}\}_{j\in J}$ is a $\mathscr{M}_{G, \kappa^*}$-acyclic cover.
    \item For any subset $\mbf{J} \subset J$, the restriction map $\mathscr{M}_{G, \kappa^*}(U_{\mbf{J}}') \to \mathscr{M}_{G, \kappa^*}(U_{\mbf{J}}' \cap \mathcal{X}_{G, w_n}(p^{\beta}))$ is injective.
\end{itemize}

This leads to the following lemma.

\begin{lemma} \label{ResAlphaIndEqualityLemma}
    For any subset $\mbf{J} \subset J$, we have $\opn{res} \circ \alpha_{0, q(\mbf{J})} = (\mbb{I}_{\beta'} \star - ) \circ \alpha_{1, q(\mbf{J})}$ as morphisms
    \[
    \mu^*\mathscr{M}_{G, \kappa^*}(\mu^{-1}(U_{q(\mbf{J})})) \to \mathscr{M}_{G, \kappa^*}(U'_{\mbf{J}}) .
    \]
\end{lemma}
\begin{proof}
Without loss of generality, we may work over $\mbb{Q}_p^{\opn{cycl}}$. We can then consider the following diagram (which is the inner diamond of (\ref{BigIGDiagramT0T1})):
\begin{equation} \label{InnerDiamondForIGs}
\begin{tikzcd}
                                              & {\mathcal{IG}_G/J_{p, 1}'} \arrow[ld, "q_1"'] \arrow[rd, "\xi^{\beta'} =: \; \lambda"] &                                              \\
\mathcal{IG}_G/J_p \arrow[rd, "\xi^{\beta'}"] &                                                                          & {\mathcal{IG}_G/J_{p, 0}'} \arrow[ld, "\mu"] \\
                                              & \mathcal{IG}_G/J_p                                                       &                                             
\end{tikzcd}
\end{equation}
and define ordinary analogues of the morphisms above. Indeed, for any open $U \subset \mathcal{X}_{G, w_n}(p^{\beta}) = \mathcal{IG}_G/J_p$ arising as the adic generic fibre of an open in $\ide{X}_{G, w_n}(p^{\beta})$, we define:
\begin{itemize}
    \item $\alpha^{\opn{ord}}_{0, U} \defeq \psi_{G, \kappa^*}^{\beta'} \circ (-w_n^{-1}\rho + \rho)(t_0^{\beta'}) \opn{Tr}_{\mu}$,
    \item $\alpha^{\opn{ord}}_{1, U} \defeq \psi_{G, \kappa^*}^{\beta'} \circ (-w_n^{-1}\rho + \rho)(t_1^{\beta'}) \opn{Tr}_{q_1} \circ \; \lambda^*$,
\end{itemize}
both of which are morphisms $\mu^* \mathscr{M}_{G, \kappa^*}(\mu^{-1}(U)) \to \mathscr{M}_{G, \kappa^*}((\varphi^{\beta'})^{-1}(U))$. We then have commutative diagrams:

\[
\begin{tikzcd}
{\mu^*\mathscr{M}_{G, \kappa^*}(\mu^{-1}(U_{q(\mbf{J})} \cap \mathcal{X}_{G, w_n}(p^{\beta})))} \arrow[r, "{\opn{res} \circ \; \alpha_{0, q(\mbf{J})}^{\opn{ord}}}"] & {\mathscr{M}_{G, \kappa^*}(U'_{\mbf{J}} \cap \mathcal{X}_{G, w_n}(p^{\beta}))}  \\
{\mu^*\mathscr{M}_{G, \kappa^*}(\mu^{-1}(U_{q(\mbf{J})}))} \arrow[u] \arrow[r, "{\opn{res} \circ \; \alpha_{0, q(\mbf{J})}}"]                                        & {\mathscr{M}_{G, \kappa^*}(U'_{\mbf{J}})} \arrow[u, hook]                      
\end{tikzcd}
\]

\[
\begin{tikzcd}
{\mu^*\mathscr{M}_{G, \kappa^*}(\mu^{-1}(U_{q(\mbf{J})} \cap \mathcal{X}_{G, w_n}(p^{\beta})))} \arrow[r, "{(\mbb{I}_{\beta'} \star -) \circ \; \alpha_{1, q(\mbf{J})}^{\opn{ord}}}"] & {\mathscr{M}_{G, \kappa^*}(U'_{\mbf{J}} \cap \mathcal{X}_{G, w_n}(p^{\beta}))} \\
{\mu^*\mathscr{M}_{G, \kappa^*}(\mu^{-1}(U_{q(\mbf{J})}))} \arrow[u] \arrow[r, "{(\mbb{I}_{\beta'} \star -) \circ \; \alpha_{1, q(\mbf{J})}}"]                                        & {\mathscr{M}_{G, \kappa^*}(U'_{\mbf{J}})} \arrow[u, hook]                     
\end{tikzcd}
\]
where the vertical maps are restriction, and $\alpha^{\opn{ord}}_{i, q(\mbf{J})}$ denotes the morphism $\alpha^{\opn{ord}}_{i, W}$ above for $W = U_{q(\mbf{J})} \cap \mathcal{X}_{G, w_n}(p^{\beta})$. Since the right-hand vertical maps are injective, it suffices to prove that $\opn{res} \circ \alpha^{\opn{ord}}_{0, q(\mbf{J})} = (\mbb{I}_{\beta'} \star - ) \circ \alpha^{\opn{ord}}_{1, q(\mbf{J})}$. We can do this for any $U$ as above, i.e. show that $\alpha_{0, U}^{\opn{ord}} = (\mbb{I}_{\beta'} \star -) \circ \alpha_{1, U}^{\opn{ord}}$. But this follows from an explicit calculation. Indeed, let $U(J_p)$ and $M(J_p)$ denote the unipotent and Levi parts of $J_p$ respectively, so that $J_p = U(J_p) \rtimes M(J_p)$. We use similar notation for $J_{p, i}'$ ($i=0, 1$). Let $\pi \colon \mathcal{IG}_G/U(J_p) \to \mathcal{IG}_G/J_p$ and $\pi_i' \colon \mathcal{IG}_G/U(J_{p, i}') \to \mathcal{IG}_G/J_{p, i}'$ denote the natural maps, and consider the following diagram lifting the morphisms in (\ref{InnerDiamondForIGs})
\[
\begin{tikzcd}
                                              & {\mathcal{IG}_G/U(J_{p, 1}')} \arrow[ld, "\tilde{q}_1"'] \arrow[rd, "\tilde{\lambda}"] &                                              \\
\mathcal{IG}_G/U(J_p) \arrow[rd, "\varphi^{\beta'}"] &                                                                          & {\mathcal{IG}_G/U(J_{p, 0}')} \arrow[ld, "\tilde{\mu}"] \\
                                              & \mathcal{IG}_G/U(J_p)                                                       &                                             
\end{tikzcd}
\]
(i.e., $\pi'_0 \circ \tilde{\lambda} = \lambda \circ \pi'_1$ etc.). Let $F \in \mu^* \mathscr{M}_{G, \kappa^*}(\mu^{-1}(U))$, which we view as a function
\[
F \colon (\pi_0')^{-1}(\mu^{-1}U) \to V_{\kappa}^*
\]
satisfying a certain transformation property under the action of $M(J'_{p, 0})$. Then for any $x \in (\varphi^{\beta'} \circ \pi)^{-1}(U)$, we have
\begin{align*}
\alpha_{0, U}^{\opn{ord}}(F)(x) &= (-w_n^{-1}\rho + \rho)(t_0^{\beta'}) \kappa^*(\xi^{-\beta'}) \sum_{m \in M(J_p)/M(J_{p, 0}')} \sum_{\substack{y \in (\pi \circ \tilde{\mu})^{-1}(U) \\ \tilde{\mu}(y) = x \xi^{\beta'} \cdot m }} \xi^{\beta'} m \cdot F(y) \\
\alpha_{1, U}^{\opn{ord}}(F)(x) &= (-w_n^{-1}\rho + \rho)(t_1^{\beta'}) \kappa^*(\xi^{-\beta'}) \sum_{m \in M(J_p)/M(J_{p, 1}')} \sum_{\substack{z \in (\pi \circ \varphi^{\beta'} \circ \tilde{q}_1)^{-1}(U) \\ \tilde{q}_1(z) = x \cdot m }} \xi^{\beta'}m \cdot F(\tilde{\lambda}(z)) .
\end{align*} 
Note that $M(J_p)/M(J_{p, 0}') = M(J_p)/M(J_{p, 1}')$. For $m \in M(J_p)/M(J_{p, 0}')$, consider the following sets:
\begin{align*}
    X_{0,m} &\defeq \left\{ y \in (\pi \circ \tilde{\mu})^{-1}(U) : \tilde{\mu}(y) = x \xi^{\beta'} \cdot m \right\} \\
    X_{1,m} &\defeq \left\{ z \in (\pi \circ \varphi^{\beta'} \circ \tilde{q}_1)^{-1}(U) : \tilde{q}_1(z) = x \cdot m \right\} .
\end{align*}
Then since $\xi$ commutes with $M(J_p)$, we clearly have $\tilde{\lambda}(X_{1, m}) \subset X_{0, m}$. Note that $\tilde{q}_1$ is an isomorphism and $\tilde{\mu}$ has degree $p^{n \beta'}$. Write $X_{1, m} = \{ x_{1, m} \}$. Then we have
\[
X_{0, m} = \{ \tilde{\lambda}(x_{1, m}) \cdot u : u \in U(J_p)/U(J_{p, 0}') \} = \{ \tilde{\lambda}(x_{1, m} \cdot v) : v \in \mathscr{V} \} , 
\]
where $\mathscr{V} = \xi^{\beta'} U(J_p) \xi^{-\beta'}/(\xi^{\beta'} U(J_{p, 0}') \xi^{-\beta'})$. Hence
\begin{align*}
    \alpha_{0, U}^{\opn{ord}}(F)(x) &= (-w_n^{-1}\rho + \rho)(t_0^{\beta'}) \kappa^*(\xi^{-\beta'}) \sum_{m \in M(J_p)/M(J_{p, 0}')} \sum_{y \in X_{0, m}} \xi^{\beta'} m \cdot F(y) \\
    &= (-w_n^{-1}\rho + \rho)(t_0^{\beta'}) \kappa^*(\xi^{-\beta'}) \sum_{m \in M(J_p)/M(J_{p, 1}')} \sum_{v \in \mathscr{V}} \xi^{\beta'} m \cdot F(\tilde{\lambda}(x_{1, m}\cdot v)) \\ 
    &= (-w_n^{-1}\rho + \rho)(t_0^{\beta'}) \kappa^*(\xi^{-\beta'}) \sum_{m \in M(J_p)/M(J_{p, 1}')} \sum_{v \in \mathscr{V}} \sum_{\substack{z \in (\pi \circ \varphi^{\beta'} \circ \tilde{q}_1)^{-1}(U) \\ \tilde{q}_1(z) = x \cdot (mv m^{-1}) m }} \xi^{\beta'} m \cdot F(\tilde{\lambda}(z)) \\
    &= \sum_{v \in \mathscr{V}} \frac{(-w_n^{-1}\rho + \rho)(t_0^{\beta'})}{(-w_n^{-1}\rho + \rho)(t_1^{\beta'})} \alpha^{\opn{ord}}_{1, U}(F)(x \cdot v) \\
    &= \frac{1}{p^{n \beta'}}\sum_{v \in \mathscr{V}} \alpha^{\opn{ord}}_{1, U}(F)(x \cdot v) \\
    & = (\mbb{I}_{\beta'} \star \alpha_{1, U}^{\opn{ord}}(F))(x)
\end{align*}
as required. Here in the fourth equality we have replaced $m v m^{-1}$ by $v$, since conjugation by $m$ permutes $\mathscr{V}$.
\end{proof}

We now consider the following morphisms $\tilde{\alpha}_{0, \mbf{J}} \defeq \opn{res} \circ \alpha_{0, q(\mbf{J})}$ and $\tilde{\alpha}_{1, \mbf{J}} = \opn{res} \circ (\mbb{I}_{\beta'} \star -) \circ  \alpha_{1, q(\mbf{J})}$ which induce maps of \v{C}ech complexes
\[
\tilde{\alpha}_0, \tilde{\alpha}_1 \colon \opn{Cech}( \mu^* \mathscr{M}_{G, \kappa^*} ; \{ \mu^{-1}(U_i) \}_{i \in I}) \to \opn{Cech}( \mathscr{M}_{G, \kappa^*}; \{ U'_j \cap \mathcal{X}_{G, w_n}(p^{\beta})_{k'} \}_{j \in J}) .
\]
By Lemma \ref{ResAlphaIndEqualityLemma}, these maps of \v{C}ech complexes coincide. After possibly increasing $k'$, these morphisms extend uniquely to 
\[
\tilde{\alpha}_0, \tilde{\alpha}_1 \colon \opn{Cech}( \mu^* \mathscr{M}_{G, \kappa^*} ; \{ \mu^{-1}(V_i) \}_{i \in I}) \to \opn{Cech}( \mathscr{M}_{G, \kappa^*}; \{ V'_j \cap \mathcal{X}_{G, w_n}(p^{\beta})_{k'} \}_{j \in J})
\]
where $V_i = U_i \cap \left( \mathcal{X}_{G, \opn{Iw}}(p^{\beta}) - \mathcal{Z}_{G, > n+1}(p^{\beta}) \right)$ and $V'_j = U'_j \cap \left( \mathcal{X}_{G, \opn{Iw}}(p^{\beta}) - \mathcal{Z}_{G, > n+1}(p^{\beta}) \right)$. For ease of notation, set $\mathcal{Z} = \mathcal{Z}_{G, > n+1}(p^{\beta})$ and $\mathcal{X}_k = \mathcal{X}_{G, w_n}(p^{\beta})_k$. We therefore see that 
\begin{itemize}
    \item $\tilde{\alpha}_{0}$ induces a map on cohomology
    \[
    \tilde{\alpha}_{0} \colon \opn{H}^{\bullet}_{\mu^{-1}(\mathcal{X}_k \cap \mathcal{Z})}\left( \mu^{-1}(\mathcal{X}_k), \mu^* \mathscr{M}_{G, \kappa^*} \right) \to \opn{H}^{\bullet}_{\mathcal{X}_{k'} \cap \mathcal{Z}}\left( \mathcal{X}_{k'}, \mathscr{M}_{G, \kappa^*} \right)
    \]
    which is just $(-w_n^{-1}\rho + \rho)(t_0^{\beta'}) \varphi_G^{\beta'} \circ \opn{Tr}_{\mu}$.
    \item $\tilde{\alpha}_{1}$ induces a map on cohomology 
    \[
    \tilde{\alpha}_{1} \colon \opn{H}^{\bullet}_{\mu^{-1}(\mathcal{X}_k \cap \mathcal{Z})}\left( \mu^{-1}(\mathcal{X}_k), \mu^* \mathscr{M}_{G, \kappa^*} \right) \to \opn{H}^{\bullet}_{\mathcal{X}_{k'} \cap \mathcal{Z}}\left( \mathcal{X}_{k'}, \mathscr{M}_{G, \kappa^*} \right)
    \]
    which is just $(\mbb{I}_{\beta'} \star -) \circ (-w_n\rho + \rho)(t_1^{\beta'}) \opn{Tr}_{q_1} \circ \; q_1^*(\psi^{\beta'}_{G, \kappa^*}) \circ \lambda^*$.
\end{itemize}
These two morphisms are equal. 

\begin{corollary}
    We have $\varphi_G^{\beta'} \circ U_{t_0}^{\beta'} = (\mbb{I}_{\beta'} \star -) \circ U_{t_1}^{\beta'}$ as endomorphisms of $\opn{H}^{n-1}_{w_n}(\kappa^*; \beta)^{(-, \dagger)}$.
\end{corollary}
\begin{proof}
    Consider the right-half of the diagram in (\ref{BigCompositionOnOCnhoods}), i.e.
    \[
    \begin{tikzcd}
                                   {\mathcal{U}^G_{K_{p, 1}', k_0+2\beta'}}  \arrow[rd, "\lambda"'] \arrow[rrdd, "q_2", bend left]  &                                                                                    &                          \\
                                                                                         & {\mathcal{U}^G_{K_{p, 0}', k_0+\beta'}} \arrow[ld, "\mu"] \arrow[rd, "\sigma"'] &                          \\
 {\mathcal{U}^G_{K_p, k_0+\beta'}}                                                                                                                &                                                                                    & {\mathcal{U}^G_{K_p, k_0}}
\end{tikzcd}
    \]
    and consider the map 
    \begin{align*} 
    \gamma_0 \colon \opn{H}^{\bullet}_{\mathcal{U}^G_{K_p, k_0} \cap \mathcal{Z}}\left( \mathcal{U}^G_{K_p, k_0}, \mathscr{M}_{G, \kappa^*} \right) &\xrightarrow{\sigma^*} \opn{H}^{\bullet}_{\mathcal{U}^G_{K_{p, 0}', k_0+\beta'} \cap \sigma^{-1}\mathcal{Z}}\left( \mathcal{U}^G_{K_{p, 0}', k_0+\beta'}, \sigma^*\mathscr{M}_{G, \kappa^*} \right) \\ &\xrightarrow{\phi_{t_0^{\beta'}, \kappa^*}} \opn{H}^{\bullet}_{\mathcal{U}^G_{K_{p, 0}', k_0+\beta'} \cap \sigma^{-1}\mathcal{Z}}\left( \mathcal{U}^G_{K_{p, 0}', k_0+\beta'}, \mu^*\mathscr{M}_{G, \kappa^*} \right) \\ &\xrightarrow{\opn{cores}} \opn{H}^{\bullet}_{\mathcal{U}^G_{K_{p, 0}', k_0+\beta'} \cap \mu^{-1}\mathcal{Z}}\left( \mathcal{U}^G_{K_{p, 0}', k_0+\beta'}, \mu^*\mathscr{M}_{G, \kappa^*} \right) \\ &\xrightarrow{\opn{res}} \opn{H}^{\bullet}_{\mu^{-1}(\mathcal{X}_k \cap \mathcal{Z})}\left( \mu^{-1}(\mathcal{X}_k), \mu^* \mathscr{M}_{G, \kappa^*} \right)
    \end{align*}
    where the corestriction is well-defined because $\sigma^{-1}\mathcal{Z} \subset \mu^{-1}\mathcal{Z}$ (see Lemma \ref{p1p2-1ZisZLemma}). Also consider the map:
    \begin{align*}
        \gamma_1 \colon \opn{H}^{\bullet}_{\mathcal{U}^G_{K_p, k_0} \cap \mathcal{Z}}\left( \mathcal{U}^G_{K_p, k_0}, \mathscr{M}_{G, \kappa^*} \right) &\xrightarrow{q_2^*} \opn{H}^{\bullet}_{\mathcal{U}^G_{K_{p,1}', k_0+2\beta'} \cap q_2^{-1}\mathcal{Z}}\left( \mathcal{U}^G_{K_{p,1}', k_0+2\beta'}, q_2^*\mathscr{M}_{G, \kappa^*} \right) \\
        &\xrightarrow{\phi_{t_1^{\beta'}, \kappa^*}} \opn{H}^{\bullet}_{\mathcal{U}^G_{K_{p,1}', k_0+2\beta'} \cap q_2^{-1}\mathcal{Z}}\left( \mathcal{U}^G_{K_{p,1}', k_0+2\beta'}, q_1^*\mathscr{M}_{G, \kappa^*} \right) \\
        &\xrightarrow{\opn{cores}} \opn{H}^{\bullet}_{\mathcal{U}^G_{K_{p,1}', k_0+2\beta'} \cap q_1^{-1}\mathcal{Z}}\left( \mathcal{U}^G_{K_{p,1}', k_0+2\beta'}, q_1^*\mathscr{M}_{G, \kappa^*} \right) \\
        &\xrightarrow{\opn{res}} \opn{H}^{\bullet}_{\lambda^{-1}\mu^{-1}(\mathcal{X}_k \cap \mathcal{Z})}\left( \lambda^{-1}\mu^{-1}(\mathcal{X}_k), q_1^*\mathscr{M}_{G, \kappa^*} \right) .
    \end{align*}
    One can easily verify that $\gamma_1 = q_1^*(\psi^{\beta'}_{G, \kappa^*}) \circ \lambda^* \circ \gamma_0$. We then see that 
    \[
    (\mbb{I}_{\beta'} \star -) \circ (-w_n\rho + \rho)(t_1^{\beta'}) \opn{Tr}_{q_1} \circ \gamma_1 = \tilde{\alpha}_1 \circ \gamma_0 = \tilde{\alpha}_0 \circ \gamma_0 .
    \]
    The left-hand side induces the operator $(\mbb{I}_{\beta'} \star -) \circ U_{t_1}^{\beta'}$ and the right-hand side induces the operator $\varphi_G^{\beta'} \circ U_{t_0}^{\beta'}$ (both after passing to the colimit over $k_0$ and $k'$). They must therefore be equal.    
\end{proof}

\begin{remark}
    One can also show that $U_{t_0}^{\beta'} \circ \varphi_G^{\beta'} = U_{t_1}^{\beta'}$ as endomorphisms of $\opn{H}^{n-1}_{w_n}(\kappa^*; \beta)^{(-, \dagger)}$, however we will not need this. We leave the details to the reader.
\end{remark}

\begin{remark} \label{FrobHIndicatorRelationRemark}
    By exactly the same method (which is even simpler in this case), one can show that, for $0 \leq m \leq \beta$, $(\varphi_H^t)^{m} \circ \varphi_H^{m} = \opn{id}$ and $(\varphi_H)^{m} \circ (\varphi_H^t)^{m} = 1_{U_{H, \beta-m}} \star -$ as endomorphisms of $\opn{H}^{n-1}_{\opn{id}}(\mathcal{S}_{H, \diamondsuit}(p^{\beta}), \sigma_{\kappa}^{[j]})^{(-, \dagger)}$.
\end{remark}

\subsubsection{Step 4}

We now finish the proof of Theorem \ref{MainThmOfPropEvMapSection}. Let $\eta \in \opn{H}^{n-1}_{w_n}(\kappa^*; \beta)^{(-, \dagger)}$ and suppose that $U_{t_i} \cdot \eta = \alpha_i \eta$ for $i=0, 1$ and some $\alpha_i \in L^{\times}$. Then 
\[
\varphi_G^{\beta'} \eta = \alpha_0^{-\beta'} (\varphi_G^{\beta'} \circ U_{t_0}^{\beta'}) \eta = \alpha_0^{-\beta'} \mbb{I}_{\beta'} \star (U_{t_1}^{\beta'} \eta) = \left( \frac{\alpha_1}{\alpha_0} \right)^{\beta'} \mbb{I}_{\beta'} \star \eta .
\]
Furthermore, one can easily see that $\opn{Ev}^{\dagger, \circ}_{\kappa, j, \chi, \beta} (\mbb{I}_{\beta'} \star -) = \opn{Ev}^{\dagger, \circ}_{\kappa, j, \chi, \beta} (-)$ because $\mbb{I}_{\beta'} \cdot 1_{U_{G, \beta}^{\circ}, \chi_p} = 1_{U_{G, \beta}^{\circ}, \chi_p}$. Theorem \ref{MainThmOfPropEvMapSection} now follows from Corollary \ref{CorollaryEvCircEquals}.

\subsection{Results from higher Coleman theory}

In this section we recall the main results from \cite{BoxerPilloni} which we will use in this article. We begin by recalling the definition of ``small slope''. Recall that $T^{G, -} \subset T(\mbb{Q}_p)$ denotes the submonoid of elements $t \in T(\mbb{Q}_p)$ which satisfy $v_p(\alpha(t)) \leq 0$ for any positive root $\alpha$ of $G$. Let $T^{G, --} \subset T^{G, -}$ denote the semigroup of elements $t \in T(\mbb{Q}_p)$ which satisfy $v_p(\alpha(t)) < 0$ for any positive root $\alpha$ of $G$.

\begin{definition}
    Let $M$ be an $L$-module (or a complex of $L$-modules in the derived category) equipped with an action of $T^{G, -}$, and suppose that there exists an element $x \in T^{G, --}$ such that for all $h \in \mbb{Q}_{\geq 0}$, the module $M$ has a slope $\leq h$ decomposition with respect to the action of $x$ (see \cite[\S 5.1]{BoxerPilloni}). Let $\kappa \in X^*(T)$ be a $M_G$-dominant character satisfying $C(\kappa^*)^- = \{ w_n \}$ (see \cite[\S 5.10.1]{BoxerPilloni} for the definition of $C(\kappa^*)^-$; here $\kappa^* = -w_{M_G}^{\opn{max}}\kappa$).
    \begin{enumerate}
        \item We say a monoid homomorphism $\theta \colon T^{G, -} \to L^{\times}$ satisfies the small slope condition $(-, \opn{ss}^M(\kappa^*))$ if, for every $w \in {^MW_G} - \{ w_n \}$, there exists an element $y \in T^{G, -}$ such that
        \[
        v_p(\theta(y)) < v_p( (w^{-1} \star \kappa^*)(y) ) - v_p( (w_n^{-1} \star \kappa^*)(y) )
        \]
        where $w^{-1} \star \kappa^* = w^{-1}\cdot(\kappa^* + \rho_G) - \rho_G$ and $\rho_G$ denotes the half-sum of positive roots of $G$.
        \item We say a monoid homomorphism $\theta \colon T^{G, -} \to L^{\times}$ satisfies the small slope condition $(-, \opn{ss}_{M, w_n}(\kappa^*))$ if, for every $w \in W_{M_G} - \{ 1 \}$, there exists an element $y \in T^{G, -}$ such that
        \[
        v_p(\theta(y)) < v_p( (w_n^{-1} w \star \kappa^*)(y) ) - v_p( (w_n^{-1} \star \kappa^*)(y) ).
        \]
        \item Let $\lambda \defeq -w_G^{\opn{max}} ( w_n^{-1} \star \kappa^* )$. We say a monoid homomorphism $\theta \colon T^{G, -} \to L^{\times}$ satisfies the small slope condition $(-, \opn{ss}(\lambda))$ if, for every $w \in W_{G}$ with $w \star \lambda \neq \lambda$, there exists an element $y \in T^{G, -}$ such that
        \[
        v_p(\theta(y)) < v_p( - w \star (w_G^{\opn{max}}\lambda)(y) ) - v_p( - w_G^{\opn{max}}\lambda(y) ) .
        \]
        \item For $\opn{ss} \in \{ \opn{ss}^M(\kappa^*), \opn{ss}_{M_{w_n}}(\kappa^*), \opn{ss}(\lambda)\}$, we let $M^{(-,\opn{ss})}$ denote the sum of generalised eigenspaces of $M^{\leq h}$ with eigencharacter given by a $(-, \opn{ss})$ small slope homomorphism, for $h$ sufficiently large. This subspace is independent of $h \gg 0$ (and also of the choice of $x$ above). 
    \end{enumerate}
\end{definition}

\begin{remark}
    Note that these small slope conditions are normalised differently from those in \cite[\S 5.11]{BoxerPilloni} because we will always view $\theta$ as an eigensystem for the optimally normalised $U_p$ Hecke operators. Furthermore, as explained in \cite[Proposition 5.11.10]{BoxerPilloni}, the small slope condition $(-,\opn{ss}(\lambda))$ is satisfied if and only if both $(-,\opn{ss}^M(\kappa^*))$ and $(-,\opn{ss}_{M, w_n}(\kappa^*))$ are satisfied. The condition $(-,\opn{ss}(\lambda))$ is the usual one appearing in the theory of $p$-adic families of automorphic forms via singular cohomology, and is the condition that we will impose in \S \ref{AutoRepsSection}.
\end{remark}

\begin{example}
    Let $\kappa \in X^*(T)$ be a $M_G$-dominant character such that $C(\kappa^*)^{-} = \{ w_n \}$. Let
    \[
    M \in \left\{ R\Gamma\left( \mathcal{S}_{G, \opn{Iw}}(p^{\beta}), \mathscr{M}_{G, \kappa^*} \right), R\Gamma_{\mathcal{Z}_{G,>n+1}(p^{\beta})}\left( \mathcal{S}_{G, \opn{Iw}}(p^{\beta}), \mathscr{M}_{G, \kappa^*} \right), R\Gamma^G_{w_n}(\kappa^*; \beta)^{(-, \dagger)} \right\} .
    \]
    The $M$ carries an action of $T^{G, -}$ via the action of the Hecke operators introduced in \S \ref{HeckeCohCorrespSubSec} (i.e. the action of $t \in T^{G, -}$ is through the action of the Hecke operator $U_t$). In this setting, we fix once and for all an element $t_{\opn{aux}} \in T^{G, --}$ and consider all slope decompositions on $M$ with respect to the action of $U_{t_{\opn{aux}}}$. Any small slope part we consider will be with respect to this action.  
\end{example}

We would like to compare these cohomology groups for varying $\beta$. To do this we must introduce the relevant trace maps on these cohomology complexes.

\begin{lemma} \label{TraceMapsOnMbetaLemma}
    Let $\beta \geq 1$ and suppose $L \supset \mu_{p^{\beta+1}}$. Let $\kappa \in X^*(T)$ be a $M_G$-dominant character. Let 
    \[
    M_{\beta} \in \left\{ R\Gamma\left( \mathcal{S}_{G, \opn{Iw}}(p^{\beta}), \mathscr{M}_{G, \kappa^*} \right), R\Gamma_{\mathcal{Z}_{G,>n+1}(p^{\beta})}\left( \mathcal{S}_{G, \opn{Iw}}(p^{\beta}), \mathscr{M}_{G, \kappa^*} \right), R\Gamma^G_{w_n}(\kappa^*; \beta)^{(-, \dagger)} \right\} .
    \]
    \begin{enumerate}
        \item One has trace maps
        \[
        \opn{Tr} \colon M_{\beta+1} \to M_{\beta} 
        \]
        which are equivariant for $T^{G, -}$ (and the action of Hecke operators away from $p$).
        \item One has a commutative diagram:
        \[
\begin{tikzcd}
{R\Gamma^G_{w_n}(\kappa^*; \beta+1)^{(-, \dagger)}} \arrow[d, "\opn{Tr}"'] & {R\Gamma_{\mathcal{Z}_{G,>n+1}(p^{\beta+1})}\left( \mathcal{S}_{G, \opn{Iw}}(p^{\beta+1}), \mathscr{M}_{G, \kappa^*} \right)} \arrow[d, "\opn{Tr}"] \arrow[l, "\opn{res}"] \arrow[r, "\opn{cores}"] & {R\Gamma\left( \mathcal{S}_{G, \opn{Iw}}(p^{\beta+1}), \mathscr{M}_{G, \kappa^*} \right)} \arrow[d, "\opn{Tr}"] \\
{R\Gamma^G_{w_n}(\kappa^*; \beta)^{(-, \dagger)}}                          & {R\Gamma_{\mathcal{Z}_{G,>n+1}(p^{\beta})}\left( \mathcal{S}_{G, \opn{Iw}}(p^{\beta}), \mathscr{M}_{G, \kappa^*} \right)} \arrow[l, "\opn{res}"] \arrow[r, "\opn{cores}"]                         & {R\Gamma\left( \mathcal{S}_{G, \opn{Iw}}(p^{\beta}), \mathscr{M}_{G, \kappa^*} \right)}                        
\end{tikzcd}
        \]
        \item One has a factorisation:
        \[
\begin{tikzcd}
M_{\beta+1} \arrow[d, "\opn{Tr}"'] \arrow[r, "U_x"] & M_{\beta+1} \arrow[d, "\opn{Tr}"] \\
M_{\beta} \arrow[r, "U_x"] \arrow[ru, dashed]       & M_{\beta}                        
\end{tikzcd}
        \]
        for any $x \in T^{G, --}$.
    \end{enumerate}
\end{lemma}
\begin{proof}
    The trace maps are induced from the trace map associated with the forgetful map 
    \[
    f \colon \mathcal{S}_{G, \opn{Iw}}(p^{\beta+1}) \to \mathcal{S}_{G, \opn{Iw}}(p^{\beta})
    \]
    as in \cite[Lemma 2.1.2]{BoxerPilloni}, using the identification $f^*\mathscr{M}_{G, \kappa^*} = \mathscr{M}_{G, \kappa^*}$ (note that $f(\mathcal{Z}_{G, >n+1}(p^{\beta+1})) \subset \mathcal{Z}_{G, >n+1}(p^{\beta})$). The equivariance follows from \cite[Lemma 4.2.14]{BoxerPilloni}. Part (2) is clear from construction. Part (3) is also very similar to \cite[Lemma 4.2.14]{BoxerPilloni}. For example, let
    \[
    \mathcal{S}_{G,K_p} \xleftarrow{p_1} \mathcal{S}_{G,K_p'} \xrightarrow{p_2} \mathcal{S}_{G,K_p}
    \]
    denote the correspondence associated with $x$ as in \S \ref{TopologicalHeckeCorrsSection} (with $K_p = K^G_{\opn{Iw}}(p^{\beta})$). Then $K_p' \subset K^G_{\opn{Iw}}(p^{\beta+1})$ and we have a commutative diagram:
    \[
\begin{tikzcd}
{p_2^{-1}(\mathcal{Z}_{G, >n+1}(p^{\beta}))} \arrow[d] \arrow[r] & {p_1^{-1}(\mathcal{Z}_{G, >n+1}(p^{\beta}))} \arrow[d] \arrow[r] & \mathcal{S}_{G,K_p'} \arrow[d] \arrow[dd, "p_1", bend left=60] \\
{\mathcal{Z}_{G,>n+1}(p^{\beta+1})} \arrow[r]                    & {f^{-1}(\mathcal{Z}_{G,>n+1}(p^{\beta}))} \arrow[r] \arrow[d]    & {\mathcal{S}_{G, \opn{Iw}}(p^{\beta+1})} \arrow[d, "f"]      \\
                                                                 & {\mathcal{Z}_{G,>n+1}(p^{\beta})} \arrow[r]                      & \mathcal{S}_{G,K_p}                                           
\end{tikzcd}
    \]
    (because for any $m \geq 1$, we have $\mathcal{Q}_m K_p x^{-1} \subset \mathcal{Q}_m x^{-1} K^G_{\opn{Iw}}(p^{\beta+1})$ -- c.f. the proof of Lemma \ref{p1p2-1ZisZLemma})  which implies that the trace map $R\Gamma_{p_2^{-1}(\mathcal{Z}_{G, >n+1}(p^{\beta}))}(\mathcal{S}_{G,K_p'}, p_1^*\mathscr{M}_{G, \kappa^*}) \to R\Gamma_{\mathcal{Z}_{G,>n+1}(p^{\beta})}(\mathcal{S}_{G,K_p}, \mathscr{M}_{G, \kappa^*})$ factorises as
    \begin{align*} 
    R\Gamma_{p_2^{-1}(\mathcal{Z}_{G, >n+1}(p^{\beta}))}(\mathcal{S}_{G,K_p'}, p_1^*\mathscr{M}_{G, \kappa^*}) &\to R\Gamma_{\mathcal{Z}_{G,>n+1}(p^{\beta+1})}(\mathcal{S}_{G, \opn{Iw}}(p^{\beta+1}), f^*\mathscr{M}_{G, \kappa^*}) \\ &\xrightarrow{\opn{Tr}} R\Gamma_{\mathcal{Z}_{G,>n+1}(p^{\beta})}(\mathcal{S}_{G,K_p}, \mathscr{M}_{G, \kappa^*}) .
    \end{align*} 
\end{proof}

We have the following theorem.

\begin{theorem}[Boxer--Pilloni] \label{MainThmBoxerPilloni1}
    Let $\kappa \in X^*(T)$ be $M_G$-dominant with $C(\kappa^*)^{-} = \{w_n \}$. Then for any $h \in \mbb{Q}_{\geq 0}$, $\beta \geq 1$, and $L \supset \mu_{p^{\beta}}$:
    \begin{enumerate}
        \item The cohomology complexes 
        \[
        R\Gamma\left( \mathcal{S}_{G, \opn{Iw}}(p^{\beta}), \mathscr{M}_{G, \kappa^*} \right), R\Gamma_{\mathcal{Z}_{G,>n+1}(p^{\beta})}\left( \mathcal{S}_{G, \opn{Iw}}(p^{\beta}), \mathscr{M}_{G, \kappa^*} \right), R\Gamma^G_{w_n}(\kappa^*; \beta)^{(-, \dagger)}
        \]
        admit slope $\leq h$ decompositions.
        \item The following morphisms 
        \begin{equation} \label{HCTclassicalityMorphismsSS}
        R\Gamma^G_{w_n}(\kappa^*; \beta)^{(-, \dagger)} \xleftarrow{\opn{res}} R\Gamma_{\mathcal{Z}_{G,>n+1}(p^{\beta})}\left( \mathcal{S}_{G, \opn{Iw}}(p^{\beta}), \mathscr{M}_{G, \kappa^*} \right) \xrightarrow{\opn{cores}} R\Gamma\left( \mathcal{S}_{G, \opn{Iw}}(p^{\beta}), \mathscr{M}_{G, \kappa^*} \right)
        \end{equation}
        are quasi-isomorphisms on $(-, \opn{ss}^M(\kappa^*))$ small slope parts, and $R\Gamma\left( \mathcal{S}_{G, \opn{Iw}}(p^{\beta}), \mathscr{M}_{G, \kappa^*} \right)^{(-,\opn{ss}^M(\kappa^*))}$ is concentrated in degree $n-1$.
        \item Suppose $L \supset \mu_{p^{\beta+1}}$. Then the trace maps in Lemma \ref{TraceMapsOnMbetaLemma}(1) are quasi-isomorphisms on slope $\leq h$ parts.
    \end{enumerate}
\end{theorem}
\begin{proof}
    We first explain the proofs of (1) and (2) for $\beta = 1$. Consider the correspondence associated with $t_{\opn{aux}}$:
    \[
    \mathcal{S}_{G,K_p} \xleftarrow{p_1} \mathcal{S}_{G,K_p'} \xrightarrow{p_2} \mathcal{S}_{G,K_p} 
    \]
    as in \S \ref{TopologicalHeckeCorrsSection}, with $K_p = K^G_{\opn{Iw}}(p)$. Set $T(-) = p_2 p_1^{-1}(-)$ and $T^t(-) = p_1 p_2^{-1}(-)$, and consider the morphism $\phi \colon p_2^*\mathscr{M}_{G, \kappa^*} \to p_1^*\mathscr{M}_{G, \kappa^*}$ given by $\phi \defeq (-w_n^{-1}\rho + \rho)(t_{\opn{aux}} \langle t_{\opn{aux}} \rangle^{-1}) \phi_{t_{\opn{aux}}, \kappa^*}$. With notation as in \cite[\S 3]{UFJ}, set 
    \[
    \mathcal{U}_1 \defeq \pi_{\opn{HT}, K_p}^{-1}\left( ]Y_{w_n}^G[ \right), \quad \quad \mathcal{Z}_1 \defeq \pi_{\opn{HT}, K_p}^{-1}\left( \overline{]X_{w_n}^G[} \right) .
    \]
    Then, as explained in \cite[\S 3.5]{BoxerPilloni}, the pair $(\mathcal{U}_1, \mathcal{Z}_1)$ forms an open/closed support condition for the correspondence above in the sense of \cite[Definition 6.1.3]{HHTBoxerPilloni}. For $k,m \geq 1$, set $\mathcal{U}_k \defeq T^{k-1}(\mathcal{U}_1)$ and $\mathcal{Z}_m \defeq (T^t)^{m-1}(\mathcal{Z}_1)$. Then $\{ \mathcal{U}_k \}_{k \in \mbb{N}}$ and $\{ \mathcal{Z}_m \}_{m \in \mbb{N}}$ form a system of support conditions as in Definition \ref{SystemOfSupportConditionsDef} (see Example \ref{HHTsupportEqualsThisExample}).

    Following \S \ref{CohomologyAndCorrespondencesSection}, set $R\Gamma(\mathcal{U}_k, \mathcal{Z}_m) \defeq R\Gamma_{\mathcal{U}_k \cap \mathcal{Z}_m}(\mathcal{U}_k, \mathscr{M}_{G, \kappa^*})$ and $R\Gamma(\mathcal{U}_{\bullet}, \mathcal{Z}_{\bullet}) = \varinjlim_k \varprojlim_m R\Gamma(\mathcal{U}_k, \mathcal{Z}_m)$. Then we obtain Hecke operators $T = U_{t_{\opn{aux}}}$ on $R\Gamma(\mathcal{U}_k, \mathcal{Z}_m)$ (associated with the above correspondence and $\phi$) which are all compatible with varying $k$ and $m$. By \cite[Lemma 2.5.25]{BoxerPilloni}, the operator $T$ on $R\Gamma(\mathcal{U}_k, \mathcal{Z}_m)$ is potent compact, hence one has slope decompositions for any $h \in \mbb{Q}_{\geq 0}$. Furthermore $R\Gamma(\mathcal{U}_k, \mathcal{Z}_m)^{\leq h}$ is identified with $R\Gamma(\mathcal{U}_1, \mathcal{Z}_1)^{\leq h}$ via the various restriction/corestriction maps (see \cite[\S 5.3]{BoxerPilloni} for more details). Similarly, $R\Gamma(\mathcal{S}_{G, K_p}, \mathcal{Z}_1) \defeq R\Gamma_{\mathcal{Z}_1}(\mathcal{S}_{G, K_p}, \mathscr{M}_{G, \kappa^*})$ and $R\Gamma(\mathcal{S}_{G, K_p}) \defeq R\Gamma(\mathcal{S}_{G, K_p}, \mathscr{M}_{G, \kappa^*})$ admit slope $\leq h$ decompositions for any $h \in \mbb{Q}_{\geq 0}$.

    As explained in the proof of \cite[Theorem 4.6.6]{UFJ}, the morphisms 
    \[
    R\Gamma(\mathcal{U}_1, \mathcal{Z}_1) \xleftarrow{\opn{res}} R\Gamma(\mathcal{S}_{G, K_p}, \mathcal{Z}_1) \xrightarrow{\opn{cores}} R\Gamma(\mathcal{S}_{G, K_p})
    \]
    are quasi-isomorphisms on small slope parts\footnote{We have reversed the order of restriction and corestriction here (compared to the statement in \cite[Theorem 4.6.6]{UFJ}). However the vanishing results for the spectral sequence in \cite[\S 5]{BoxerPilloni} imply this claim. More generally, there is a commutative square combining the maps here and in \cite[Theorem 4.6.6]{UFJ}, with all maps becoming quasi-isomorphisms on small slope parts.}, which implies that the following morphisms 
    \begin{equation} \label{RGUlimZQISOeqn}
        R\Gamma(\mathcal{U}_{\bullet}, \mathcal{Z}_{\bullet}) \xleftarrow{\opn{res}} \varprojlim_m R\Gamma(\mathcal{S}_{G, K_p}, \mathcal{Z}_m) \xrightarrow{\opn{cores}} R\Gamma(\mathcal{S}_{G, K_p})
    \end{equation}
    are also quasi-isomorphisms on small slope parts. Additionally, one can show that $R\Gamma(\mathcal{S}_{G, K_p})^{(-,\opn{ss}^M(\kappa^*))}$ is concentrated in degree $n-1$. Now the proofs of (1) and (2) for $\beta=1$ follow from the fact that $\cap_{m\in \mbb{N}} \mathcal{Z}_m = \mathcal{Z}_{G, >n+1}(p)$ and the systems of support conditions 
    \[
    \{ \mathcal{U}_k \}, \{ \mathcal{Z}_m \} \text{ and } \{ \mathcal{U}^G_{K_p, k} \}, \{ \overline{\pi_{\opn{HT}, K_p^{-1}}(\mathcal{P}_G\backslash \mathcal{P}_G \mathcal{Q}_m K_p )} \}
    \]
    are intertwined in the sense of Definition \ref{IntertwinedSupportConditionsDef} (which can be checked on the level of flag varieties). Hence the diagram (\ref{RGUlimZQISOeqn}) is identified with (\ref{HCTclassicalityMorphismsSS}).

    The general case now follows from Lemma \ref{TraceMapsOnMbetaLemma} (crucially using the factorisation in part (3)). 
\end{proof}

We now discuss the result for general $p$-adic weights. Let $(R, R^+)$ be a Tate affinoid pair over $(L, \mathcal{O}_L)$ and let $\kappa_R \colon T(\mbb{Z}_p) \to R^{\times}$ be an $s_0$-analytic character, for some auxiliary integer $s_0 \geq 1$. For any morphism $(R, R^+) \to (S, S^+)$ of Tate affinoid pairs over $(L, \mathcal{O}_L)$, we let $\kappa_S \colon T(\mbb{Z}_p) \to S^{\times}$ denote the induced $s_0$-analytic character. If $L'/L$ is a finite extension and $z \colon (R, R^+) \to (L', \mathcal{O}_{L'})$ is a morphism over $(L, \mathcal{O}_L)$, then we will also use the notation $\kappa_z \defeq z \circ \kappa_R$.

For any $s \geq s_0$, we have the cohomology complex $R\Gamma^G_{w_n, s\opn{-an}}(\kappa_R^*; \beta)^{(-, \dagger)}$, and one can define a trace map 
\begin{equation} \label{locallyAnalyticTraceEqn}
\opn{Tr} \colon R\Gamma^G_{w_n, s\opn{-an}}(\kappa_R^*; \beta+1)^{(-, \dagger)} \to R\Gamma^G_{w_n, s\opn{-an}}(\kappa_R^*; \beta)^{(-, \dagger)}
\end{equation}
which is equivariant for $T^{G, -}$ in a similar way to Lemma \ref{TraceMapsOnMbetaLemma}.

\begin{theorem}[Boxer--Pilloni] \label{MainThmBoxerPilloni2}
    Fix $L/\mbb{Q}_p$ a finite extension containing $\mu_p$, and let $(R, R^+)$ be a Tate affinoid pair over $(L, \mathcal{O}_L)$ with $s_0$-analytic character $\kappa_R$ as above. Then for any $z_0 \colon (R, R^+) \to (L, \mathcal{O}_{L})$ and $h \in \mbb{Q}_{\geq 0}$, there exists an open neighbourhood $\opn{Spa}(S, S^+) \subset \opn{Spa}(R, R^+)$ containing $z_0$ such that: for all $\beta \geq 1$, for all finite extensions $L'/L$ containing $\mu_{p^{\beta}}$, and for all $s \geq s_0$
    \begin{enumerate}
        \item The cohomology complex $R\Gamma^G_{w_n, s\opn{-an}}(\kappa_{S'}^*; \beta)^{(-, \dagger)}$ admits a slope $\leq h$ decomposition, where $\opn{Spa}(S', (S')^+) = \opn{Spa}(S, S^+) \times_{\opn{Spa}(L, \mathcal{O}_L)} \opn{Spa}(L', \mathcal{O}_{L'})$.
        \item The natural map
        \[
        R\Gamma^G_{w_n, (s+1)\opn{-an}}(\kappa_{S'}^*; \beta)^{(-, \dagger)} \to R\Gamma^G_{w_n, s\opn{-an}}(\kappa_{S'}^*; \beta)^{(-, \dagger)}
        \]
        is a quasi-isomorphism on slope $\leq h$ parts.
        \item For any $z \in \opn{Spa}(S', (S')^+)(L')$ such that $\kappa_z \in X^*(T)$ is $M_G$-dominant and satisfies $C(\kappa_z^*)^{-} = \{ w_n \}$, the natural map (induced by the map from distributions to the algebraic representation)
        \[
        R\Gamma^G_{w_n, s\opn{-an}}(\kappa_{z}^*; \beta)^{(-, \dagger)} \to R\Gamma^G_{w_n}(\kappa_{z}^*; \beta)^{(-, \dagger)}
        \]
        is a quasi-isomorphism on $(-, \opn{ss}_{M, w_n}(\kappa_z^*))$ small slope parts. 
        \item Suppose $L'$ contains $\mu_{p^{\beta +1}}$. Then the trace map 
        \[
        \opn{Tr} \colon R\Gamma^G_{w_n, s\opn{-an}}(\kappa_{S'}^*; \beta+1)^{(-, \dagger)} \to R\Gamma^G_{w_n, s\opn{-an}}(\kappa_{S'}^*; \beta)^{(-, \dagger)}
        \]
        is a quasi-isomorphism on slope $\leq h$ parts.
    \end{enumerate}
\end{theorem}
\begin{proof}
    As in Lemma \ref{TraceMapsOnMbetaLemma}, one can establish an analogous factorisation diagram for $U_{t_{\opn{aux}}}$ and the trace map in (\ref{locallyAnalyticTraceEqn}). Hence it suffices to prove the theorem when $\beta = 1$. But part (1) just follows the same strategy as in the proof of Theorem \ref{MainThmBoxerPilloni1}, by considering the system of support conditions $\{ \mathcal{U}_k \}_{k \geq k_0}$, $\{ \mathcal{Z}_m \}_{m \geq m_0}$ for $k_0, m_0$ sufficiently large (because the locally projective Banach sheaf associated with $D_{G, \kappa_{S'}^*}^{s\opn{-an}}$ can be defined on an open neighbourhood of $p_1^{-1}(\mathcal{U}_{k_0}) \cap p_2^{-1}(\mathcal{Z}_{m_0})$). Part (2) follows from the fact that the action of $t_{\opn{aux}}$ on $D_{G, \kappa_{S'}^*}^{s\opn{-an}}$ factors through $D_{G, \kappa_{S'}^*}^{(s+1)\opn{-an}} \to D_{G, \kappa_{S'}^*}^{s\opn{-an}}$ (c.f. \cite[\S 6.4.4]{BoxerPilloni}). Part (3) follows from analysing the locally analytic BGG resolution (see \cite[Corollary 6.8.4]{BoxerPilloni} -- the expected slope bounds hold because the Shimura variety is compact). 
\end{proof}

\subsection{Trace-compatibility of overconvergent evaluation maps}

The final ingredient we need for constructing the $p$-adic $L$-function is a trace-compatibility relation between $\opn{Ev}^{\dagger, \circ}_{\kappa, j, \chi, \beta}$ for varying $\beta$.

\begin{lemma} \label{OCunramifiedEvmapTraceLemma}
    Let $(\kappa, j) \in \mathcal{E}$ satisfying Assumption \ref{AssumpOnKJforACchar} and let $\beta \geq 1$ be an integer. Let $\chi \in \Sigma_{\kappa, j}(\ide{N}_{\beta})$ such that the conductor of $\chi$ is prime to $p$ (i.e. divides $\ide{N}$). Let $L/\mbb{Q}_p$ be a finite extension containing $F^{\opn{cl}}(\chi)$ and $\mbb{Q}_p(\mu_{p^{\beta+1}})$. Then
    \[
    \opn{Ev}^{\dagger, \circ}_{\kappa, j, \chi, \beta} \circ \opn{Tr} = \opn{Ev}^{\dagger, \circ}_{\kappa, j, \chi, \beta+1}
    \]
    where $\opn{Tr} \colon \opn{H}^{n-1}_{w_n}(\kappa^*; \beta+1)^{(-, \dagger)} \to \opn{H}^{n-1}_{w_n}(\kappa^*; \beta)^{(-, \dagger)}$ is the trace map as in Lemma \ref{TraceMapsOnMbetaLemma}.
\end{lemma}
\begin{proof}
    Note that $\opn{Ev}^{\dagger, \circ}_{\kappa, j, \chi, \beta} = \opn{Ev}^{\dagger}_{\kappa, j, \chi, \beta} \circ (1_{U_{G, \beta}^{\circ}} \star -)$ and $\opn{Ev}^{\dagger}_{\kappa, j, \chi, \beta} \circ \opn{Tr} = \opn{Ev}^{\dagger}_{\kappa, j, \chi, \beta+1}$ for the same reasons as in Proposition \ref{TraceCompatibilityOfThetaProp}. Therefore, it suffices to show that 
    \[
    (1_{U_{G, \beta}^{\circ}} \star -) \circ \opn{Tr} = \opn{Tr} \circ (1_{U_{G, \beta+1}^{\circ}} \star - ) .
    \]
    As in \S \ref{PropsOfEvMapsInterpolationSubSec}, we can reduce this to a statement over the ordinary locus, and the claim follows from the fact that $\opn{Tr}$ is expressed as the sum over the action of representatives of $M^G_{\opn{Iw}}(p^{\beta})/M^G_{\opn{Iw}}(p^{\beta+1})$ and $1_{U_{G, \beta}^{\circ}}$ is fixed under the action of $M^G_{\opn{Iw}}(p^{\beta})$.
\end{proof}

\subsection{Test data for the \texorpdfstring{$p$}{p}-adic \texorpdfstring{$L$}{L}-function} \label{TestDataForPAdicLSSec}

We now describe the general recipe for constructing the $p$-adic $L$-function, which we will be used in \S \ref{AutoRepsSection}.  As a guide for the reader, we explain in Remark \ref{Rem:WhatTheTestDataMeans} what role the following objects will play in \S \ref{AutoRepsSection}.

Fix $L/\mbb{Q}_p$ a finite extension containing $\mu_p$ and $\iota_p(F_{\ide{N}})$. Suppose that $(R, R^+)$ is a Tate affinoid pair over $(L, \mathcal{O}_L)$, and we have an $s_0$-analytic character $\kappa_R \colon T(\mbb{Z}_p) \to R^{\times}$ satisfying Assumption \ref{AssumpOnKJforACchar}. Let $z_0 \in \opn{Spa}(R, R^+)(L)$ such that $\kappa_{z_0} \in X^*(T)$ satisfies the conditions in Definition \ref{DefOfAlgWeightsE} and $C(\kappa_{z_0}^*)^{-} = \{ w_n \}$. Let $\lambda_R = -w_G^{\opn{max}}( w_n^{-1} \star \kappa_R^* )$. Fix a rational number $h \in \mbb{Q}_{\geq 0}$. 

Let $\Omega = \opn{Spa}(S, S^+)$ be the neighbourhood of $z_0$ as in Theorem \ref{MainThmBoxerPilloni2}. We let $\Upsilon \subset \Omega(\Qpb)$ denote a subset of points $z \in \Omega(\Qpb)$ such that:
\begin{itemize}
    \item $z_0 \in \Upsilon$;
    \item for any $z \in \Upsilon$, the character $\kappa_z \in X^*(T)$ satisfies the conditions in Definition \ref{DefOfAlgWeightsE} and $C(\kappa_z^*)^{-} = \{ w_n \}$.
\end{itemize}
Recall the definition of $\Omega_s = \Omega \times_{\mbb{Q}_p} \mathcal{W}(\ide{N}p^{\infty})_s$ from \S \ref{OverconvergentEVMapsPAdicWtSSec}. Let $\Sigma'_{\beta} \subset \bigcup_{s \geq 1} \Omega_{s}(\Qpb)$ denote the maximal subset of points $x \in \bigcup_{s \geq 1} \Omega_{s}(\Qpb)$ such that its projection to $\Omega$ lies in $\Upsilon$, and which satisfy the conditions in Notation \ref{NotationForSigmaSigmaPrime}. Let $\Sigma \subset \Sigma_{1}'$ denote the largest subset such that $\chi_x$ has conductor dividing $\ide{N}$ for any $x \in \Sigma$. Note that
\[
\Sigma \subset \Sigma_1' \subset \Sigma_2' \subset \cdots .
\]
We now introduce the test data. 

\begin{definition} \label{TestDataAbstractDef}
    With set-up as above, let $\Upsilon^{\opn{int}} \subset \Upsilon$ be a subset containing $z_0$ and let
\[
\eta \in \opn{H}^{n-1}_{w_n, s_0\opn{-an}}( \kappa_S^*; 1 )^{(-, \dagger, \leq h)}
\]
denote a cohomology class such that:
\begin{itemize}
    \item for any $z \in \Upsilon^{\opn{int}}$, the specialisation $\eta_{z, 1, s_0\opn{-an}} \in \opn{H}^{n-1}_{w_n, s_0\opn{-an}}(\kappa_{z}^*; 1)^{(-, \dagger, \leq h)}$ of $\eta$ at the point $z$ is an eigenvector for the action of $T^{G, -}$ with eigencharacter $\theta_z$ satisfying the $(-, \opn{ss}(\lambda_z))$ small slope condition.
\end{itemize}
\end{definition}

\begin{definition}
    Let $\eta$ be as in Definition \ref{TestDataAbstractDef}, and for $s \geq s_0$, let $\eta_{1, s\opn{-an}} \in \opn{H}^{n-1}_{w_n, s\opn{-an}}( \kappa_S^*; 1 )^{(-, \dagger, \leq h)}$ denote the unique lift of $\eta$ under the isomorphism in Theorem \ref{MainThmBoxerPilloni2}(2). Suppose that $\Sigma$ is Zariski dense in $\Omega_1$. We define:
    \[
    \Xi(\eta) \defeq \opn{Ev}^{\dagger, \opn{la}}_{\kappa_S, 1}( (\eta_{1, s-\opn{an}})_{s \geq s_0} ) \in \mathscr{D}^{\opn{la}}\left( \Gal(F_{\ide{N}p^{\infty}}/F), \mathscr{O}_{\Omega} \right) 
    \]
    (equivalently, we can view $\Xi(\eta)$ as a global section on $\Omega \times_{\mbb{Q}_p} \mathcal{W}(\ide{N}p^{\infty})$).
\end{definition}

We have the following interpolation property.

\begin{theorem} \label{AbstractPadicLMachineTheorem}
    With set-up as above, let $\eta$ and $\Upsilon^{\opn{int}}$ be as in Definition \ref{TestDataAbstractDef}. Suppose that $\Sigma$ is Zariski dense in $\Omega_1$, and let
    \[
    \Sigma^{\opn{int}} = \left\{ x \in \bigcup_{\beta \geq 1} \Sigma_{\beta}' : \begin{array}{c} \ide{p}_{\tau_0} \text{ divides the conductor of } \chi_x \text{ and } \\ x \in \Upsilon^{\opn{int}} \times \mathcal{W}(\ide{N}p^{\infty}) \end{array} \right\}.
    \]
    Then, for any $x \in \Sigma^{\opn{int}}$, one has
    \[
    \Xi(\eta)(x) = (1-p^{-1}) A(x) \cdot \opn{Ev}_x(\eta_z^{\opn{cl}})
    \]
    where
    \begin{itemize}
        \item $z \in \Upsilon^{\opn{int}}$ denotes the projection of $x$ to $\Omega$, $p^{\beta'}$ denotes the conductor of $\chi_{x, p, \bar{\tau}_0}$, and
        \[
        A(x) \defeq \left( \frac{\theta_z(t_0)}{\theta_z(t_1)} \right)^{\beta'} p^{\beta'\kappa_{x,n+1, \tau_0}} \chi_x( \varpi_{\ide{p}_{\bar{\tau}_0}} )^{-\beta'} \chi_{x,p, \bar{\tau}_0}(-1) \mathscr{G}(\chi_{x,p, \bar{\tau}_0})
        \]
        with $\varpi_{\ide{p}_{\bar{\tau}_0}} \in \mbb{A}_F^{\times}$ denoting the idele obtained as the image of $p$ under the natural embedding $\mbb{Q}_p^{\times} \cong F_{\bar{\ide{p}}_{\tau_0}}^{\times} \hookrightarrow \mbb{A}_F^{\times}$.
        \item $\opn{Ev}_x(\eta_z^{\opn{cl}}) \defeq \opn{Ev}_{\kappa_x, j_{x} - \chi_{x, p}, \chi_x, \beta}(\eta_{z, \beta}^{\opn{cl}})$ where $\beta \geq 1$ is any integer such that $x \in \Sigma_{\beta}'$, and $\eta_{z, \beta}^{\opn{cl}} \in \opn{H}^{n-1}\left( \mathcal{S}_{G, \opn{Iw}}(p^{\beta}), \mathscr{M}_{G, \kappa_z^*} \right)^{(-,\opn{ss}(\lambda_z))}$ denotes the unique class whose trace down to $\mathcal{S}_{G, \opn{Iw}}(p)$ uniquely corresponds to $\eta_{z, 1, s_0\opn{-an}}$ via the isomorphisms in Theorem \ref{MainThmBoxerPilloni2}(3) and Theorem \ref{MainThmBoxerPilloni1}(2). 
    \end{itemize}
\end{theorem}
\begin{proof}
    Note that $\opn{Ev}_x(\eta_z^{\opn{cl}})$ is indeed independent of $\beta$ by Proposition \ref{EVBeta1EqualsEVBetaProp}. We will define an analogue of $\Xi(\eta)$ for general $\beta \geq 1$. Let $L_{\beta} = L(\mu_{p^{\beta}})$, and set $\opn{Spa}(S_{\beta}, S_{\beta}^+) = \Omega_{L_{\beta}}$. Let $\eta_{\beta, s\opn{-an}} \in \opn{H}^{n-1}_{w_n, s\opn{-an}}( \kappa_{S_{\beta}}^*; 1 )^{(-, \dagger, \leq h)}$ denote the unique lift of $\eta$ under the morphisms in Theorem \ref{MainThmBoxerPilloni2}. Then we define $\Xi_{\beta}(\eta)$ as
    \[
    \Xi_{\beta}(\eta) \defeq \opn{Ev}^{\dagger, \opn{la}}_{\kappa_{S_{\beta}}, \beta}\left( (\eta_{\beta, s\opn{-an}})_{s \geq \opn{max}(s_0, \beta)} \right) \in \mathscr{D}^{\opn{la}}\left( \Gal(F_{\ide{N}p^{\infty}}/F), \mathscr{O}_{\Omega_{L_{\beta}}} \right)
    \]
    We see from \S \ref{CompatibilitiesOfEvMapSSEC} and Theorem \ref{MainThmOfPropEvMapSection} that, for any $x \in \Sigma^{\opn{int}} \cap \Sigma_{\beta}'$, one has the interpolation property:
    \[
    \Xi_{\beta}(\eta)(x) = (1-p^{-1}) A(x) \cdot \opn{Ev}_x(\eta_z^{\opn{cl}}) .
    \]
    Note that, by tracing through the definitions and using Proposition \ref{DaggerAndClassicalUnderBPisosProp}, $\opn{Ev}_{\kappa_x, j_x - \chi_{x, p}, \chi_x, \beta}^{\dagger}$ does indeed correspond to $\opn{Ev}_{\kappa_x, j_x - \chi_{x, p}, \chi_x, \beta}$ via the classicality isomorphisms in Theorem \ref{MainThmBoxerPilloni1}(2). 
    
    In addition to this, for any $x \in \Sigma$ we have, by Lemma \ref{OCunramifiedEvmapTraceLemma}, the compatibility
    \[
    \Xi_{\beta}(\eta)(x) = \opn{Ev}^{\dagger, \circ}_{\kappa_x, j_x, \chi_x, \beta}(\opn{sp}_x(\eta_{\beta,\opn{la}})) = \opn{Ev}^{\dagger, \circ}_{\kappa_x, j_x, \chi_x, 1}(\opn{sp}_x(\eta_{1,\opn{la}})) = \Xi(\eta)(x)
    \]
    where $\opn{sp}_x(\eta_{\beta,\opn{la}}) \in \opn{H}^{n-1}_{w_n}(\kappa_x^*; \beta)^{(-, \dagger)}$ denotes the specialisation of $\eta_{\beta, s\opn{-an}}$ (for any $s \geq \opn{max}(s_0, \beta)$) at the point $x \in \Sigma$ (note that $\opn{sp}_x(\eta_{\beta,\opn{la}})$ are trace-compatible for varying $\beta$). Since $\Sigma$ is Zariski dense, this implies that the image of $\Xi(\eta)$ in $\mathscr{D}^{\opn{la}}\left( \Gal(F_{\ide{N}p^{\infty}}/F), \mathscr{O}_{\Omega_{L_{\beta}}} \right)$ must coincide with $\Xi_{\beta}(\eta)$. The result follows.
\end{proof}

\begin{remark} \label{Rem:WhatTheTestDataMeans}
    Let us explain the context of this abstract test data and how it will be applied in \S \ref{AutoRepsSection}. 
    \begin{itemize}
        \item $z_0$ will correspond to a classical point in the $n[F^+:\mbb{Q}]$-dimensional weight space $\mathcal{W}_0$ which parameterises self-dual continuous characters of $T(\mbb{Z}_p)$. It will correspond to the weight of the fixed automorphic representation $\pi$ of $\mbf{G}(\mbb{A})$ that we start with in the construction of the $p$-adic $L$-functions.
        \item $\Omega$ will be a sufficiently small open affinoid neighbourhood of $z_0$ in $\mathcal{W}_0$.
        \item $\Upsilon$ essentially corresponds to the set of all classical weights $z$ in $\Omega$ such that $\kappa_z$ satisfies the conditions needed for constructing the evaluation maps (on overconvergent cohomology of weight $\kappa^*_z$).
    \end{itemize}
    The $p$-adic $L$-functions will be (specialisations of) global sections on the product $\Omega \times_{\mbb{Q}_p} \mathcal{W}(\ide{N}p^{\infty})$, where the first variable concerns the variation of the weight of the automorphic representations, and the second variable concerns the variation of the anticyclotomic characters. For technical reasons though, one must construct (as above) compatible global sections on an increasing open affinoid cover $\{ \Omega_{s}\}_{s \geq 1}$ of $\Omega \times_{\mbb{Q}_p} \mathcal{W}(\ide{N}p^{\infty})$.

    The class $\eta$ gives rise to the inputs into the $p$-adic evaluation maps. One should view this as a family of overconvergent cohomology classes over the space $\Omega$ which specialises to a small slope Hecke eigenvector at $z_0$ (which in practice corresponds to $\pi$). We let $\Upsilon^{\opn{int}} \subset \Upsilon$ denote a subset of weights (containing $z_0$) such that the specialisation of $\eta$ at any point in $\Upsilon^{\opn{int}}$ is also a small slope Hecke eigenvector. If one thinks of $\eta$ as a $p$-adic family of overconvergent forms for $\mbf{G}$, then the specialisation of $\eta$ at points in $\Upsilon^{\opn{int}}$ will correspond to the ``classical specialisations'' of the family. We note however, that in this set-up, we do not necessarily want to assume $\eta$ is a Hecke eigenvector over $\Omega$ nor that $\Upsilon^{\opn{int}}$ is Zariski-dense in $\Omega$ (although this will be the case in \S \ref{VariationInColemanFamiliesSSec}).

    The $p$-adic $L$-functions will of course satisfy an interpolation property.
    \begin{itemize}
        \item The subset $\Upsilon^{\opn{int}}$ corresponds to the specialisations in the \emph{first} variable of $\Omega \times_{\mbb{Q}_p} \mathcal{W}(\ide{N}p^{\infty})$ that form part of the interpolation formulae.
        \item We let $\Sigma'_{\beta} \subset \Omega \times_{\mbb{Q}_p} \mathcal{W}(\ide{N}p^{\infty})$ denote the subset such that the \emph{second} variable corresponds to an algebraic anticyclotomic Hecke character, satisfying the conditions needed to participate in the evaluation maps, and with conductor at $p$ bounded by $p^{\beta}$. Let us call $\bigcup_{\beta \geq 1} \Sigma'_{\beta}$ the set of ``classical anticyclotomic characters''.
        \item We let $\Sigma^{\opn{int}} \subset \Omega \times_{\mbb{Q}_p} \mathcal{W}(\ide{N}p^{\infty})$ denote the subset of points such that the first variable lies in $\Upsilon^{\opn{int}}$ and the second variable is a ``classical anticyclotomic character'' with conductor divisible by $\ide{p}_{\tau_0}$. This will be the interpolation region for the $p$-adic $L$-function.
    \end{itemize}
    We note that it should be possible with more work to remove the condition that $\ide{p}_{\tau_0}$ divides the conductor of the ``classical anticyclotomic character''; this assumption is an artifact of the methods used in the proof of Theorem \ref{MainThmOfPropEvMapSection}.

    Finally, we note that the main choice that influences the $p$-adic $L$-function is the family $\eta$ and the set of ``classical specialisations'' $\Upsilon^{\opn{int}}$. If we make another choice of $(\eta, \Upsilon^{\opn{int}})$, then we will obtain a different $p$-adic $L$-function. The construction in Theorem \ref{FirstMainTheoremFinalSection} will therefore (a priori) depend on this choice. The construction in Theorem \ref{SecondMainTheoremFinalSSec} will be independent of the choice however, since the interpolation region $\Sigma^{\opn{int}}$ will be Zariski-dense in $\Omega \times_{\mbb{Q}_p} \mathcal{W}(\ide{N}p^{\infty})$.     
\end{remark}

%----------------------------------------

\section{Automorphic representations} \label{AutoRepsSection}

In this section we introduce the relevant automorphic representations of $\mbf{G}(\mbb{A})$ and construct the $p$-adic $L$-functions in Theorem \ref{FirstTheoremIntro} and Theorem \ref{SecondTheoremIntro}. 

\subsection{Assumptions} \label{FinalSectionAssumptionsSSec}

Let $\pi$ be a cuspidal automorphic representation of $\mbf{G}(\mbb{A})$. Write $\pi = \otimes_{v}' \pi_v$ for its restricted tensor product decomposition over the places of $\mbb{Q}$. By abuse of notation, we will also use the notation $\pi_{\infty}$ to denote the underlying $(\ide{g}, K_{\infty})$-module, where $K_{\infty} \subset \mbf{G}(\mbb{R})$ is the maximal compact-mod-centre subgroup whose complexification equals $M_{\mbf{G}}(\mbb{C})$. We impose the following hypotheses:
\begin{itemize}
    \item (``generic at $\infty$'') We suppose that $\pi_{\infty}$ lies in the discrete series with Harish-Chandra parameter of the form $w_n \cdot (\lambda + \rho_G)$ for some dominant $\lambda \in X^*(T)^+$. In particular, we set $\kappa \defeq -w_{M_G}^{\opn{max}} \cdot ( w_n \star (-w_G^{\opn{max}} \lambda ))$, which satisfies $C(\kappa^*)^{-} = \{ w_n \}$. 
    \item (``self-duality'') The character $\lambda$ is self-dual, i.e., $\lambda = -w_G^{\opn{max}}\lambda$.
    \item Let $S$ be a finite set of places of $\mbb{Q}$ containing $\infty$ and all primes where $\pi$ is ramified.\footnote{We say $\pi$ ramifies at $\ell$ if there does not exist a maximal special subgroup of $\mbf{G}(\mbb{Q}_{\ell})$ under which $\pi_{\ell}$ has non-trivial invariants.} We let $K \subset \mbf{G}(\mbb{A}_f)$ be a neat compact open subgroup such that $\pi_f^K \neq 0$. We suppose that $K$ factorises as 
    \[
    K = K_S K^S, \quad \quad K_S \subset \mbf{G}(\mbb{A}_{f,S}), \; K^S \subset \mbf{G}(\mbb{A}^S), 
    \]
    where $K^S = \prod_{\ell \not\in S} K_{\ell}$ with $K_{\ell} \subset \mbf{G}(\mbb{Q}_{\ell})$ a good maximal special compact open subgroup (as in \cite[\S 2.1]{Minguez}). We assume that $p \not\in S$ and that $K_p$ is hyperspecial (and identifies with $G(\mbb{Z}_p)$ with respect to a fixed reductive integral model of $\mbf{G}_{\mbb{Q}_p}$).
\end{itemize}

We let $\ide{N}$ denote the smallest ideal of $\mathcal{O}_F$ such that 
\[
\nu(K^p \cap \mbf{H}(\mbb{A}_f^p)) \subset \mathcal{N}( ( \widehat{\mathcal{O}}_{F^+}^{(p)} + \ide{N} \widehat{\mathcal{O}}_{F}^{(p)} )^{\times} )
\]
(see Definition \ref{DefinitionOfMathfrakNbeta}).

\begin{remark}
    Write $\lambda = (\lambda_0; \lambda_{1, \tau}, \dots, \lambda_{2n, \tau})$ and $\kappa = (\kappa_0; \kappa_{1, \tau}, \dots, \kappa_{2n, \tau})$. Since $\lambda$ is self-dual, we have $\lambda_0 = 0$ and $\lambda_{i, \tau} = -\lambda_{2n+1-i, \tau}$ for any $i \in \{1, \dots, 2n\}$ and $\tau \in \Psi$. An explicit calculation then shows that $\kappa_0 = 0$ and 
    \begin{align*}
        (\kappa_{1, \tau_0}, \dots, \kappa_{2n, \tau}) &= (n-\lambda_{n+1, \tau_0}, -\lambda_{2n, \tau_0}, \dots, -\lambda_{n+2, \tau_0}, -(1+\lambda_{n, \tau_0}), \dots, -(1+\lambda_{1, \tau_0}) ) \\ 
        (\kappa_{1, \tau}, \dots, \kappa_{2n, \tau}) &= (\lambda_{1, \tau}, \dots, \lambda_{2n, \tau}) \quad \quad (\tau \neq \tau_0) .
    \end{align*}
    We therefore see that $\kappa$ satisfies the assumptions in Definition \ref{DefOfAlgWeightsE} and also Assumption \ref{AssumpOnKJforACchar}.
\end{remark}

\subsubsection{Test data}

 We now describe the test data that will specify the specialisation of family $\eta$ at $z_0$ as described in Remark \ref{Rem:WhatTheTestDataMeans} (more accurately, it will determine the dual test data in \S \ref{Subsub:DualTestData} and hence the class $\eta^{\opn{cl}}_{\pi, 1}$, which corresponds to the specialisation of $\eta$ at $z_0$ via higher Coleman theory).
\begin{itemize}
    \item Let $\phi_{\infty} \in \pi_{\infty}$ be a non-zero element in the lowest $K_{\infty}$-type which is an eigenvector for the action of $K_{\infty} \cap \mbf{H}(\mbb{R})$ with eigencharacter $\sigma_{\kappa}^{[0], \vee}$ (note the lowest $K_{\infty}$-type has highest weight $\kappa - 2\rho_{H, \opn{nc}}$, so there does indeed exist such an eigenvector).
    \item We fix a non-zero element $\phi_S \in \pi_{f,S}^{K_S}$
    \item For $\ell \notin S' \defeq S \cup \{ p \}$, we fix a non-zero element $\phi_{\ell} \in \pi_{\ell}^{K_{\ell}}$. We let
    \[
    \theta^{S'}_{\pi} \colon C^{\infty}( K^{S'} \backslash \mbf{G}(\mbb{A}^{S'}) / K^{S'}; \mbb{C}) \to \mbb{C}
    \]
    denote the corresponding Hecke eigensystem with respect to a fixed Haar measure on $\mbf{G}(\mbb{A}^{S'})$, i.e., for any $f \in C^{\infty}( K^{S'} \backslash \mbf{G}(\mbb{A}^{S'}) / K^{S'}; \mbb{C})$, we have 
    \[
    f \cdot \phi^{S'} = \theta_{\pi}^{S'}(f) \phi^{S'}, \quad \quad \quad  \phi^{S'} \defeq  \bigotimes_{\ell \not\in S'} \phi_{\ell} \in \bigotimes'_{\ell \not\in S'} \pi_{\ell}.
    \]
    \item We suppose that there exists a monoid homomorphism $\theta_{\pi, p} \colon T^{G,-} \to \Qpb^{\times}$ satisfying $(-, \opn{ss}(\lambda))$ and a non-zero vector $\phi_p \in \pi_p^{K^G_{\opn{Iw}}(p)}$ such that
    \[
    \lambda(t'\langle t' \rangle^{-1}) [K^G_{\opn{Iw}}(p) \cdot t' \cdot K^G_{\opn{Iw}}(p)] \cdot \phi_p = \theta_{\pi, p}(t) \phi_p, \quad \quad t' = (w_G^{\opn{max}})^{-1} t w_G^{\opn{max}} 
    \]
    for all $t \in T^{G, -}$. Here $[K^G_{\opn{Iw}}(p) \cdot t' \cdot K^G_{\opn{Iw}}(p)]$ denotes the (unnormalised) Hecke operator associated with the element $t'$, and we are using the identification $\iota_p \colon \mbb{C} \cong \Qpb$.
\end{itemize}

We let $\phi \defeq \phi_{\infty} \otimes \phi_S \otimes \phi^{S'} \otimes \phi_{p}$. Set 
\[
\mbb{T}^{-} = C^{\infty}( K^{S'} \backslash \mbf{G}(\mbb{A}^{S'}) / K^{S'}; \mbb{Q}) \otimes_{\mbb{Q}} \mbb{Q}[T^{G, -}]
\]
and let $\theta_{\pi} \colon \mbb{T}^{-}_{\mbb{C}} \to \mbb{C}$ denote the homomorphism given by $\theta_{\pi} \defeq \theta^{S'}_{\pi} \otimes \theta_{\pi, p}$. We note that $\theta_{\pi}$ is defined over a number field (see \cite[Proposition 2.15]{ShinTemplier}), so there exists a finite extension $L/\mbb{Q}_p$ such that $\theta_{\pi}$ is defined over $L$ (again using the fixed identification $\iota_p \colon \mbb{C} \cong \Qpb$). We assume that $L$ contains $\mu_p$ and $\iota_p(F_{\ide{N}})$. 

\begin{definition} \label{PlusVersionOfTestDataDef}
    Let $u_{\opn{sph}} \in \mbf{G}(\mbb{Q}_p) = \mbb{Q}_p^{\times} \times \prod_{\tau \in \Psi} \opn{GL}_{2n}(\mbb{Q}_p)$ denote the element $u_{\opn{sph}} = 1 \times \prod_{\tau} u_{\opn{sph}, \tau}$ where $u_{\opn{sph}, \tau}$ is the block matrix (with block size $n \times n$) given by
    \[
    u_{\opn{sph}, \tau} = \tbyt{1}{-w_{\opn{GL}_n}^{\opn{max}}}{}{1}
    \]
    where $w_{\opn{GL}_n}^{\opn{max}}$ denotes the antidiagonal matrix with $1$s along the antidiagonal. For tuples $j = (j_{\tau}) \in \prod_{\tau \in \Psi} \mbb{Z}_{\geq 0}$ and $e = (e_{\tau}) \in \prod_{\tau \in \Psi} \mbb{Z}_{\geq 1}$ such that $(\kappa, j) \in \mathcal{E}$, we define
    \[
    \phi_{e}^{[j]} \defeq \left( \Delta_{\kappa}^{[j]} \cdot \phi_{\infty} \right) \otimes \phi_S \otimes \phi^{S'} \otimes \left( u_{\opn{sph}} t_p^e \cdot \phi_p \right) \; \; \in \; \; \pi 
    \]
    where $\Delta_{\kappa}^{[j]}$ is as in Definition \ref{AppendixDefOfDELTAkappaJDiffOp} and $t_p^e \in T(\mbb{Q}_p)$ denotes the element 
    \[
    t_p^e = 1 \times \prod_{\tau} \opn{diag}(p^{e_{\tau}(2n-1)}, p^{e_{\tau}(2n-2)}, \dots, p^{e_{\tau}}, 1).
    \]
\end{definition}

\begin{remark} \label{AdditionalUnimposedAssumptionRem}
The automorphic forms in Definition \ref{PlusVersionOfTestDataDef} will be the input into the unitary Friedberg--Jacquet periods, and their definition is very closely related to the input data for the Friedberg--Jacquet periods for general linear groups (see \cite{5author, ClassicalLocus}). In fact, for the automorphic periods we consider to be non-trivial we should impose the following additional assumptions:
\begin{itemize}
    \item (``symplectic type'') The weak base-change of $\pi$ to an automorphic representation of $\opn{GL}_1(\mbb{A}_E) \times \opn{GL}_{2n}(\mbb{A}_F)$ is of symplectic type with trivial similitude character (see \cite[Conjecture 7.4(2)]{ChenGan}). Furthermore, $\opn{Hom}_{\mbf{H}(\mbb{A}_f)}(\pi_f, \mbb{C}) \neq 0$. 
    \item $\pi_p$ is generic (i.e. isomorphic to an unramified principal series representation) with distinct Satake parameters in each $\tau$-component, and the choice of $p$-stabilisation $\phi_p$ above is ``spin'' (see \cite[\S 6]{5author} or \cite[\S 3]{ClassicalLocus}) .
\end{itemize}
In the constructions below, we prefer not to make these additional assumptions because this affords us with extra flexibility when deforming $\pi$ in a Coleman family (and it is not necessary to assume \cite[Conjecture 7.4]{ChenGan}).
\end{remark}

For the test data for the $p$-adic $L$-function however, we need to consider ``dual versions'' of these automorphic forms, which we introduce in the following section.

\subsubsection{Dual test data} \label{Subsub:DualTestData}

Let $\beta \geq 1$ be an integer, and set 
\[
s_p = 1 \times \prod_{\tau \in \Psi} \opn{diag}(1, p, p^2, \dots, p^{2n-1}) \; \in \; T(\mbb{Q}_p) .
\]
View $w_G^{\opn{max}} \in G(\mbb{Q}_p)$ as the matrix which in the $\tau$-component is given by the the antidiagonal matrix with $1$s along the antidiagonal (for any $\tau \in \Psi$). We define:
\[
\psi_{p, \beta} \defeq (w_G^{\opn{max}}\lambda)(s_p)^{\beta}\theta_{\pi, p}(s_p)^{-\beta} [K^G_{\opn{Iw}}(p^{\beta}) s_p^{\beta} w_G^{\opn{max}} K^G_{\opn{Iw}}(p) ] \cdot \phi_p \; \in \; \pi_p^{K^G_{\opn{Iw}}(p^{\beta})}  .
\]

\begin{lemma} \label{PropsOfPsiBetaClassLemma}
    The class $\psi_{p, \beta}$ satisfies
    \begin{equation} \label{TransformationForDualEVEqn}
    (w_G^{\opn{max}}\lambda)(t \langle t \rangle^{-1}) [ K^G_{\opn{Iw}}(p^{\beta}) t K^G_{\opn{Iw}}(p^{\beta}) ] \cdot \psi_{p, \beta} = \theta_{\pi, p}(t) \psi_{p, \beta}
    \end{equation}
    for any $t \in T^{G, -}$. Moreover $\opn{tr}(\psi_{p, \beta+1}) = \psi_{p, \beta}$, where $\opn{tr}$ denotes the (unnormalised) trace from level $K^G_{\opn{Iw}}(p^{\beta+1})$ to level $K^G_{\opn{Iw}}(p^{\beta})$.
\end{lemma}
\begin{proof}
    Let $t \in T^{G, -}$. Then an explicit calculation shows that
    \[
    [K^G_{\opn{Iw}}(p^{\beta}) t K^G_{\opn{Iw}}(p^{\beta})] \circ [K^G_{\opn{Iw}}(p^{\beta}) s_p^{\beta} w_G^{\opn{max}} K^G_{\opn{Iw}}(p) ] = [K^G_{\opn{Iw}}(p^{\beta}) s_p^{\beta} w_G^{\opn{max}} K^G_{\opn{Iw}}(p) ] \circ [K^G_{\opn{Iw}}(p) t' K^G_{\opn{Iw}}(p)]
    \]
    where $t' = (w_G^{\opn{max}})^{-1} t w_G^{\opn{max}}$. Indeed, set $y = s_p^{\beta}w_G^{\opn{max}}$. Then both
    \[
    ( t^{-1} K^G_{\opn{Iw}}(p^{\beta}) t \cap K^G_{\opn{Iw}}(p^{\beta}) ) \backslash K^G_{\opn{Iw}}(p^{\beta}) / (y K^G_{\opn{Iw}}(p) y^{-1} \cap K^G_{\opn{Iw}}(p^{\beta}) )
    \]
    and
    \[
    ( y^{-1} K^G_{\opn{Iw}}(p^{\beta}) y \cap K^G_{\opn{Iw}}(p) ) \backslash K^G_{\opn{Iw}}(p) / (t' K^G_{\opn{Iw}}(p) (t')^{-1} \cap K^G_{\opn{Iw}}(p) )
    \]
    are singletons (which can easily be checked using Iwahori decompositions). This implies the transformation property in (\ref{TransformationForDualEVEqn}). Furthermore, a similar calculation shows that
    \[
    \opn{tr} \circ [K^G_{\opn{Iw}}(p^{\beta+1}) s_p^{\beta+1} w_G^{\opn{max}} K^G_{\opn{Iw}}(p) ] = [K^G_{\opn{Iw}}(p^{\beta}) s_p K^G_{\opn{Iw}}(p^{\beta})] \circ [K^G_{\opn{Iw}}(p^{\beta}) s_p^{\beta} w_G^{\opn{max}} K^G_{\opn{Iw}}(p) ]
    \]
    which implies the trace compatibility. 
\end{proof}

For an integer $\beta \geq 1$ and a tuple $j = (j_{\tau}) \in \prod_{\tau \in \Psi} \mbb{Z}_{\geq 0}$ such that $(\kappa, j) \in \mathcal{E}$, we set 
\[
\psi_{\beta}^{[j]} \defeq \left( \Delta_{\kappa}^{[j]} \cdot \phi_{\infty} \right) \otimes \phi_S \otimes \phi^{S'} \otimes \hat{\gamma} \cdot \psi_{p, \beta} \; \in \; \pi . 
\]
We now consider the corresponding cohomology classes. Let $V_{\kappa - 2\rho_{H, \opn{nc}}}$ denote the algebraic representation of $K_{\infty}$ with highest weight $\kappa - 2\rho_{H, \opn{nc}}$. Then there exists a unique $K_{\infty}$-equivariant homomorphism $F_{\infty} \in \opn{Hom}_{K_{\infty}}(V_{\kappa - 2\rho_{H, \opn{nc}}}, \pi_{\infty})$ such that $F_{\infty}(\alpha_{-} \odot v_{\kappa}^{[0]}) = \phi_{\infty}$, where $\alpha_{-} \in V_{-2\rho_{H, \opn{nc}}} = \bigwedge^{n-1}\overline{\ide{u}}_G$ denotes a generator of the line $\bigwedge^{n-1}\overline{\ide{u}}_H$ and $\odot$ denotes the Cartan product. This extends to a $K_{\infty}$-equivariant homomorphism $F_{\infty} \in \opn{Hom}_{K_{\infty}}( \bigwedge^{n-1}\overline{\ide{u}}_G, \pi_{\infty} \otimes V_{\kappa}^* )$ in the obvious way. By Arthur's multiplicity formula for unitary groups \cite{CZ21}, the representation $\pi$ appears with multiplicity one in the $K_{\circ}$-finite vectors $C^{\infty}([\mbf{G}])^{K_{\circ}\opn{-fin}}$ of the smooth functions on $[\mbf{G}] = \mbf{G}(\mbb{Q})A_{\mbf{G}}(\mbb{R})^{\circ} \backslash \mbf{G}(\mbb{A})$ (see \S \ref{RelationWithUFJperiodsSubSubSec}). We fix such an embedding and consider the induced homomorphism
\[
\opn{Hom}_{K_{\infty}}( \bigwedge^{n-1}\overline{\ide{u}}_G, \pi_{\infty} \otimes V_{\kappa}^* ) \otimes \pi_f^{K^p K^G_{\opn{Iw}}(p^{\beta})} \hookrightarrow \opn{Hom}_{K_{\infty}}( \bigwedge^{n-1}\overline{\ide{u}}_G, C^{\infty}([\mbf{G}]/K^p K^G_{\opn{Iw}}(p^{\beta}))^{K_{\circ}\opn{-fin}} \otimes V_{\kappa}^* ) .
\]
We let $\eta^{\opn{cl}}_{\pi, \beta}$ denote the image of $F_{\infty} \otimes (\phi_S \otimes \phi^{S'} \otimes \psi_{p, \beta})$ under this embedding. Since the $(\overline{\ide{p}}_G, K_{\infty})$-cohomology of $\pi_{\infty} \otimes V_{\kappa}^*$ is one-dimensional and concentrated in degree $n-1$ (see \cite[Theorem 3.2.1]{BHR94}), this class $\eta^{\opn{cl}}_{\pi, \beta}$ must represent a cohomology class in $\opn{H}^{n-1}\left( S_{\mbf{G}, \opn{Iw}}(p^{\beta})(\mbb{C}), \mathscr{M}_{G, \kappa^*} \right)$ which we also denote by $\eta^{\opn{cl}}_{\pi, \beta}$. By Lemma \ref{PropsOfPsiBetaClassLemma}, the classes $\eta^{\opn{cl}}_{\pi, \beta}$ are trace-compatible as $\beta$ varies, and are Hecke eigenvectors -- for any $f \in C^{\infty}(K^{S'} \backslash \mbf{G}(\mbb{A}^{S'}) / K^{S'})$ and $t \in T^{G, -}$ we have
\[
f \cdot \eta^{\opn{cl}}_{\pi, \beta} = \theta^{S'}_{\pi}(f) \eta^{\opn{cl}}_{\pi, \beta}, \quad \quad U_t \cdot \eta^{\opn{cl}}_{\pi, \beta} = \theta_{\pi, p}(t) \eta^{\opn{cl}}_{\pi, \beta} .
\]
We can (and do) rescale $\phi = \phi_{\infty} \otimes \phi_S \otimes \phi^{S'} \otimes \phi_{p}$ so that $\eta^{\opn{cl}}_{\pi, 1}$ is defined over a number field, and hence by rigid GAGA, gives rise to a cohomology class in $\opn{H}^{n-1}\left( \mathcal{S}_{G, \opn{Iw}}(p^{\beta}), \mathscr{M}_{G, \kappa^*} \right)$, where we view $\mathcal{S}_{G, \opn{Iw}}(p^{\beta})$ as an adic space over $\opn{Spa}(L, \mathcal{O}_L)$. This cohomology class lies in the $(-, \opn{ss}(\lambda))$ part by assumption.

\subsection{Variation in the anticyclotomic direction}

We now construct the $p$-adic $L$-function appearing in Theorem \ref{FirstTheoremIntro}. We first introduce the interpolation set. 

\begin{definition}
    Let $\Sigma_{\pi}$ denote the set of anticyclotomic algebraic Hecke characters 
    \[
    \chi \colon F^{\times} \mbb{A}_{F^+}^{\times} \backslash \mbb{A}_F^{\times} \to \mbb{C}^{\times}
    \]
    such that:
    \begin{itemize}
        \item The $\infty$-type of $\chi$ is of the form
        \[
\chi( z ) = z_{\tau_0}^{-(\lambda_{n, \tau_0} + 1 + j_{\tau_0})} \bar{z}_{\tau_0}^{\lambda_{n, \tau_0} + 1 + j_{\tau_0}} \cdot \prod_{\tau \in \Psi - \{\tau_0\}} z_{\tau}^{-j_{\tau}} \bar{z}_{\tau}^{j_{\tau}} 
\]
for all $z = (z_{\tau}) \in (\mbb{R} \otimes_{\mbb{Q}} F)^{\times} = \prod_{\tau \in \Psi} \mbb{C}^{\times}$, for some tuple of integers $j = (j_{\tau}) \in \prod_{\tau \in \Psi} \mbb{Z}_{\geq 0}$ such that $j_{\tau_0} \leq \lambda_{n-1, \tau_0} - \lambda_{n, \tau_0}$ and $j_{\tau} \leq \lambda_{n, \tau}$ for $\tau \neq \tau_0$.
\item The conductor of $\chi$ divides $\ide{N}p^{\infty}$. In particular, we let $c = c(\chi) = (c_{\tau}) \in \prod_{\tau \in \Psi} \mbb{Z}_{\geq 0}$ denote the tuple of integers such that the $p$-part of the conductor of $\chi$ is of the form $\prod_{\tau \in \Psi} (\ide{p}_{\tau}\cdot \ide{p}_{\bar{\tau}})^{c_{\tau}}$. 
\item One has $c_{\tau_0} \geq 1$.
    \end{itemize}
    For any such character satisfying these assumptions, we let $j = j(\chi)$ denote the tuple in the first bullet point, and we let $e = e(\chi) = (e_{\tau}) \in \prod_{\tau \in \Psi} \mbb{Z}_{\geq 1}$ denote the tuple satisfying $e_{\tau} = \opn{max}(c_{\tau}, 1)$.
\end{definition}

For any $i=1, \dots, 2n$ and $\tau \in \Psi$, let $t_{p, i, \tau} \in \mbf{G}(\mbb{Q}_p)$ denote the element which is the identity outside the $\tau$-component, and in the $\tau$-component is given by $\opn{diag}(p, \dots, p, 1, \dots, 1)$ where there are $i$ lots of $p$. Let $\alpha_{i, \tau} \in \Qpb^{\times}$ and $\alpha_p^e$ denote the elements which satisfy
\[
[K^G_{\opn{Iw}}(p) \cdot t_{p, i, \tau} \cdot K^G_{\opn{Iw}}(p)] \cdot \phi_p = \alpha_{i, \tau} \phi_p \quad \text{ and } \quad [K^G_{\opn{Iw}}(p) \cdot t_{p}^e \cdot K^G_{\opn{Iw}}(p)] \cdot \phi_p = \alpha_p^e \phi_p .
\]
For any $\chi \in \Sigma_{\pi}$, we let $\mathscr{E}_p(\pi, \chi)$ denote the following $p$-adic multiplier
\[
\mathscr{E}_p(\pi, \chi) = p^{-e_{\tau_0}} \left( \frac{\alpha_{n, \tau_0}}{\alpha_{n-1,\tau_0}}\right)^{e_{\tau_0}} \chi_{\ide{p}_{\bar{\tau}_0}}(-1) \chi_{\ide{p}_{\bar{\tau}_0}}(p)^{-e_{\tau_0}} \mathscr{G}(\chi_{\ide{p}_{\bar{\tau}_0}}) \left( \prod_{\tau \in \Psi} \chi_{\ide{p}_{\bar{\tau}}}(-1)^n \right) (\alpha_p^e \delta_B(t_p^e))^{-1}
\]

where $\delta_B$ denotes the modulus function associated with the upper-triangular Borel subgroup of $\mbf{G}(\mbb{Q}_p)$, and $\chi_{\ide{p}_{\bar{\tau}}}$ denotes the restriction of $\chi$ to $F_{\ide{p}_{\bar{\tau}}}^{\times} \cong \mbb{Q}_p^{\times}$.

\begin{theorem} \label{FirstMainTheoremFinalSection}
    Let $\pi$ be a cuspidal automorphic representation satisfying the assumptions in \S \ref{FinalSectionAssumptionsSSec} (except Remark \ref{AdditionalUnimposedAssumptionRem}). Given test data $\phi$ as in \S \ref{FinalSectionAssumptionsSSec}, there exists a locally analytic distribution $\mathscr{L}_{p, \phi}(\pi, -) \in \mathscr{D}^{\opn{la}}(\Gal(F_{\ide{N}p^{\infty}}/F), L)$ such that
    \[
    \mathscr{L}_{p, \phi}(\pi, \hat{\chi}) = (\star) \cdot (2 \pi i)^{-(n-1)} \cdot \mathscr{E}_p(\pi, \chi) \cdot \int_{[\mbf{H}]} \phi^{[j]}_e(h) \chi'\left(\frac{\opn{det}h_2}{\opn{det}h_1} \right) dh
    \]
    for any $\chi \in \Sigma_{\pi}$, where $(\star)$ is a non-zero rational number independent of $\pi$ and $\chi$.
\end{theorem}
\begin{proof}
    We wish to apply the general construction in \S \ref{TestDataForPAdicLSSec}.  We remind the reader that a description of the roles the objects play is given in Remark \ref{Rem:WhatTheTestDataMeans}. 
    
    Let $\mathcal{W}_0$ denote the ($n[F^+:\mbb{Q}]$-dimensional) weight space over $\opn{Spa}(L, \mathcal{O}_L)$ parameterising continuous characters $\xi$ on $T(\mbb{Z}_p)$ which satisfy $\xi = -w_G^{\opn{max}}\xi$. The weight $\lambda$ therefore corresponds to a point $z_0 \in \mathcal{W}_0(L)$. Let $\opn{Spa}(R, R^+) \subset \mathcal{W}_0$ be an open affinoid neighbourhood of $z_0$ with $s_0$-analytic universal character $\lambda_R \colon T(\mbb{Z}_p) \to R^{\times}$. Let $\kappa_R = -w_{M_G}^{\opn{max}} \cdot ( w_n \star \lambda_R)$. This satisfies Assumption \ref{AssumpOnKJforACchar}. Furthermore, the specialisation $\kappa_z$ of $\kappa_R$ at any dominant classical weight $z \in \opn{Spa}(R, R^+)$ satisfies the assumptions in Definition \ref{DefOfAlgWeightsE} and $C(\kappa_z^*)^{-} = \{w_n \}$. Fix a rational number $h \in \mbb{Q}_{\geq 0}$ which is larger than $v_p(\theta_{\pi, p}(t_{\opn{aux}}))$. Note that $\kappa_{z_0} = \kappa$ and $\lambda_{z_0} = \lambda$.

    Let $\Omega = \opn{Spa}(S, S^+) \subset \opn{Spa}(R, R^+)$ denote open neighbourhood of $z_0$ as in Theorem \ref{MainThmBoxerPilloni2}, and let $\Upsilon \subset \Omega(\Qpb)$ denote the subset of classical dominant weights. Let $\eta_{z_0, 1, s_0\opn{-an}} \in \opn{H}^{n-1}_{w_n, s_0\opn{-an}}(\kappa_{z_0}^*; 1)^{(-, \dagger, \opn{ss}(\lambda_{z_0}))}$ denote the unique class corresponding to $\eta^{\opn{cl}}_{\pi, 1}$ via the classicality isomorphisms in Theorem \ref{MainThmBoxerPilloni1} and Theorem \ref{MainThmBoxerPilloni2}. It is an eigenvector for the action of $\mbb{T}^{-}_L$ with eigencharacter $\theta_{\pi}$. We claim that, after possibly shrinking $\Omega$, we can lift $\eta_{z_0, 1, s_0\opn{-an}}$ to a class
    \[
    \eta \in \opn{H}^{n-1}_{w_n, s_0\opn{-an}}(\kappa_S^*; 1)^{(-, \dagger, \leq h)} .
    \]
    Indeed, let $I_{\pi}$ denote the kernel of $\mbb{T}^{-}_L \xrightarrow{\theta_{\pi}} L$ and let $I_{S}$ denote the kernel of the composition $\mbb{T}^{-}_S \to \mbb{T}^{-}_L \xrightarrow{\theta_{\pi}} L$, where the first map is induced from specialisation at $z_0$. Let $S_0$ denote the (rigid) localisation of $S$ at the maximal ideal corresponding to $z_0$. The Tor-spectral sequence:
    \[
    E_2^{i, j} \colon \opn{Tor}^{S_0}_{-i}\left( \opn{H}^{j}_{w_n, s_0\opn{-an}}(\kappa_S^*; 1)^{(-, \dagger, \leq h)}_{I_{S}}, L \right) \Rightarrow \opn{H}^{i+j}_{w_n, s_0\opn{-an}}(\kappa_{z_0}^*; 1)^{(-, \dagger, \leq h)}_{I_{\pi}}
    \]
    and the vanishing/classicality results in Theorem \ref{MainThmBoxerPilloni1} and Theorem \ref{MainThmBoxerPilloni2} imply that 
    \[
    \opn{H}^{j}_{w_n, s_0\opn{-an}}(\kappa_S^*; 1)^{(-, \dagger, \leq h)}_{I_{S}} = 0 \quad \text{ for } j \neq n-1
    \]
    and $\opn{H}^{n-1}_{w_n, s_0\opn{-an}}(\kappa_S^*; 1)^{(-, \dagger, \leq h)}_{I_{S}}$ is free of finite-rank over $S_0$ (here we are using the assumption that $\theta_{\pi, p}$ is small slope, and the local criterion for flatness). Furthermore, we have
    \[
    \opn{H}^{n-1}_{w_n, s_0\opn{-an}}(\kappa_S^*; 1)^{(-, \dagger, \leq h)}_{I_{S}} \otimes_{S_0} L = \opn{H}^{n-1}_{w_n, s_0\opn{-an}}(\kappa_{z_0}^*; 1)^{(-, \dagger, \leq h)}_{I_{\pi}} 
    \] 
    and $\opn{H}^{n-1}_{w_n, s_0\opn{-an}}(\kappa_S^*; 1)^{(-, \dagger, \leq h)}_{I_{S}}$ is a direct factor of $\opn{H}^{n-1}_{w_n, s_0\opn{-an}}(\kappa_S^*; 1)^{(-, \dagger, \leq h)} \otimes_S S_0$. This implies that we can shrink the neighbourhood $\Omega$ around $z_0$ and lift $\eta_{z_0, 1, s_0\opn{-an}}$ to a class $\eta \in \opn{H}^{n-1}_{w_n, s_0\opn{-an}}(\kappa_S^*; 1)^{(-, \dagger, \leq h)}$. 

    Set $\Upsilon^{\opn{int}} = \{ z_0 \}$ and note that $\Sigma^{\opn{int}}$ in Theorem \ref{AbstractPadicLMachineTheorem} is equal to $\{ z_0 \} \times \Sigma_{\pi}$. Furthermore, the set $\Sigma$ is Zariski dense in $\Omega_1$ (because $\lambda_{z,n-1, \tau_0} - \lambda_{z,n, \tau_0}$ and $\lambda_{z, n, \tau}$ become arbitrarily large as $z$ runs over $\Upsilon$). We can therefore apply Theorem \ref{AbstractPadicLMachineTheorem} and define $\mathscr{L}_{p, \phi}(\pi, -) = \opn{sp}_{z_0}\Xi(\eta)$, where $\opn{sp}_{z_0}$ denotes the specialisation at $z_0 \in \Omega$.  Let $\nu_{n-1} \defeq (-1)^{(n-1)}$. Then, $\mathscr{L}_{p, \phi}(\pi, -)$ has the following interpolation property:
    \begin{align*} 
    \mathscr{L}_{p,\phi}(\pi, \hat{\chi}) &= (1-p^{-1}) \left( \frac{\theta_{\pi,p}(t_0)}{\theta_{\pi,p}(t_1)} \right)^{e_{\tau_0}} p^{e_{\tau_0}\kappa_{n+1, \tau_0}} \chi_{\ide{p}_{\bar{\tau}_0}}(p)^{-e_{\tau_0}} \chi_{\ide{p}_{\bar{\tau}_0}}(-1) \mathscr{G}(\chi_{\ide{p}_{\bar{\tau}_0}}) \opn{Ev}_{\kappa, j, \chi, \beta} (\eta^{\opn{cl}}_{\pi, \beta}) \\
    &= (1-p^{-1}) \cdot \mathscr{E}_p(\pi, \chi) \cdot \chi_{\ide{p}_{\bar{\tau}_0}}(\nu_{n-1})\left( \prod_{\tau \in \Psi} \chi_{\ide{p}_{\bar{\tau}}}(-1)^n \right) \alpha_p^e \delta_B(t_p^e) \cdot \opn{Ev}_{\kappa, j, \chi, \beta} (\eta^{\opn{cl}}_{\pi, \beta}) 
    \end{align*} 
    where $\beta \geq \opn{max}\{ e_{\tau} : \tau \in \Psi \}$ and for the second equality we have used $\kappa_{n+1, \tau_0} = -1-\lambda_{n, \tau_0}$. Set 
    \[
    \mathscr{P}_{\pi, \chi}(\psi) = \int_{[\mbf{H}]} \psi(h) \chi'\left(\frac{\opn{det}h_2}{\opn{det}h_1} \right) dh
    \]
    for any $\psi \in \pi$. Note that this is defines an equivariant linear map $\mathscr{P}_{\pi, \chi} \in \opn{Hom}_{\mbf{H}(\mbb{A})}(\pi, (\chi')^{-1} \circ \nu)$. We now apply Proposition \ref{ClassicalComplexUFJrelationProp} and see that
    \[
    \opn{Ev}_{\kappa, j, \chi, \beta} (\eta^{\opn{cl}}_{\pi, \beta}) = (2 \pi i)^{-(n-1)} \opn{Vol}(K_{H, \beta}; dh){^{-1}} \mathscr{P}_{\pi, \chi}(\psi_{\beta}^{[j]}) .
    \]

 Let $\psi_{p, e}$ denote the trace of $\psi_{\beta}^{[j]}$ down to depth $p^{e_{\tau}}$ Iwahori level in the $\tau$-component, and define $\psi_e^{[j]}$ in exactly the same way as $\psi_{\beta}^{[j]}$, replacing $\psi_{p, \beta}$ with $\psi_{p, e}$. Let $\hat{\gamma}_{\clubsuit} \in \mbf{G}(\mbb{Q}_p)$ denote the element which is trivial in the similitude component, and equal to the element \eqref{SimpleExampleofHatGammaAppendix} in each $\tau$-component. With notation as in Lemma \ref{OpenOrbitForgammauvLemma}, we have the relation 
 \[
 \hat{\gamma} \in \zeta \cdot \hat{\gamma}_{\clubsuit} \cdot B_G(\mbb{Z}_p)
 \]
 where $\zeta$ is the element which is trivial outside the $\tau_0$-component and equal to $\left( \begin{smallmatrix} X & \\ & 1\end{smallmatrix} \right)$ in the $\tau_0$-component. Using the transformation properties in Appendix \ref{EquivariantLinearFuncsAppendix}, we have

    \begin{align*}
        \mathscr{P}_{\pi, \chi}(\psi_{\beta}^{[j]}) &= \prod_{\tau \in \Psi} p^{(e_{\tau}-\beta)n(2n-1)} \mathscr{P}_{\pi, \chi}(\psi_{e}^{[j]}) \\
        &= \chi_{\ide{p}_{\bar{\tau}_0}}(\opn{det}X) \prod_{\tau \in \Psi} p^{(e_{\tau}-\beta)n(2n-1)} \mathscr{P}_{\pi, \chi}(\zeta^{-1} \cdot \psi_{e}^{[j]}) \\
        &= \chi_{\ide{p}_{\bar{\tau}_0}}(\opn{det}X) (\alpha_p^e \delta_B(t_p^e))^{-1} \prod_{\tau \in \Psi} p^{(e_{\tau}-\beta)n(2n-1)} \chi_{\ide{p}_{\bar{\tau}}}(-1)^n p^{-e_{\tau}n(2n-1)} \cdot \mathscr{P}_{\pi, \chi}(\phi_{e}^{[j]}) \\
        &= \chi_{\ide{p}_{\bar{\tau}_0}}(\nu_{n-1}) (\alpha_p^e \delta_B(t_p^e))^{-1} p^{-\beta [F^+:\mbb{Q}]n(2n-1)} \prod_{\tau \in \Psi} \chi_{\ide{p}_{\bar{\tau}}}(-1)^n \cdot \mathscr{P}_{\pi, \chi}(\phi_{e}^{[j]})
    \end{align*}

where the first equality is Lemma \ref{Lem:1stAppendixB}; the second equality follows from the $\mbf{H}(\mbb{Q}_p)$-equivariance of $\mathscr{P}_{\pi, \chi}$; in the third equality we have used Lemma \ref{Lem:2ndAppendixB} and the fact that $\zeta^{-1} \cdot \psi_{e}^{[j]}$ is the dual eigenvector associated with $\phi_p$ as in Appendix \ref{EquivariantLinearFuncsAppendix} multiplied by $\alpha_p^{-e}$ (as well as the fact that $p^{-n(2n-1)}[K_{G, 1} : K_{G, e_{\tau}}]^{-1} = p^{-e_{\tau}n(2n-1)}$); the fourth equality uses the fact that $\opn{det}X = \nu_{n-1}$.

We now define
    \[
    (\star) = (1-p^{-1}) \opn{Vol}(K_{H, \beta})^{-1} p^{-\beta[F^+:\mbb{Q}]n(2n-1)}
    \]
    which is a non-zero rational number which is independent of $\beta$ (it only depends on $n$ and the volume of $K^p \cap \mbf{H}(\mbb{A}_f^p)$). The result follows.
\end{proof}

\subsection{Variation in Coleman families} \label{VariationInColemanFamiliesSSec}

We continue with the notation introduced in the previous sections -- in particular, let $\pi$ be a cuspidal automorphic representation of $\mbf{G}(\mbb{A})$ satisfying the assumptions in \S \ref{FinalSectionAssumptionsSSec} (except Remark \ref{AdditionalUnimposedAssumptionRem}). To construct $p$-adic families through $\pi$ and a $p$-adic distribution associated with these families, we impose the following additional hypotheses:
\begin{itemize}
    \item We assume that the finite primes in $S$ split in $E/\mbb{Q}$.
    \item ($\phi$ is new away from $p$) We assume that there exists compact open subgroup $K_S^{\opn{new}} \subset \mbf{G}(\mbb{A}_{f, S})$ such that $\phi_S \in \pi_{f, S}^{K_S^{\opn{new}}}$ and $\opn{dim}_{\mbb{C}}\pi_{f, S}^{K_S^{\opn{new}}} = 1$ (note that, under the above assumption on $S$, such a compact open always exists by the local newform theory for general linear groups -- see \cite[Remark 6.1.2]{UFJ}). We assume that $K_S$ is a normal subgroup of $K_S^{\opn{new}}$. If we set $K^{\opn{new}} \defeq K_S^{\opn{new}}K^S$, then we have $\opn{dim}_{\mbb{C}}\pi_f^{K^{\opn{new}}} = 1$.
    \item ($\phi_p$ is a $p$-regular $p$-stabilisation) We assume that the generalised eigenspace of $\pi_p^{K^G_{\opn{Iw}}(p)}$ associated with the character $\theta_{\pi, p}$ is one-dimensional. 
\end{itemize}
Under these assumptions, we say that $\phi$ is a $p$-regular $p$-stabilisation which is new away from $p$. In what follows, the Shimura varieties $S_{\mbf{G}, \opn{Iw}}(p^{\beta})$ will have prime-to-$p$ level given by $K^p$. Note that the finite group $K_S^{\opn{new}}/K_S$ acts on these Shimura varieties.

\begin{proposition} \label{MultOnePropFinalSec}
    Let $I_{\pi}$ denote the kernel of the map $\theta_{\pi} \colon \mbb{T}_L^{-} \to L$. Then 
    \[
    \opn{H}^{n-1}\left( \mathcal{S}_{G, \opn{Iw}}(p), \mathscr{M}_{G, \kappa^*} \right)_{I_{\pi}}^{K_S^{\opn{new}}/K_S}
    \]
    is one-dimensional over $L$.
\end{proposition}
\begin{proof}
    Note that it is enough to prove the analogous statement over $\mbb{C}$, i.e., that the $\mbb{C}$-vector space
    \[
    \opn{H}^{n-1}\left( S(\mbb{C}), \mathscr{M}_{G, \kappa^*} \right)_{I_{\pi}}^{K_S^{\opn{new}}/K_S}, \quad \quad S \defeq S_{\mbf{G}, \opn{Iw}}(p),
    \]
    is one dimensional (where by abuse of notation, $I_{\pi}$ now denotes the kernel of the map $\theta_{\pi} \colon \mbb{T}_{\mbb{C}}^{-} \to \mbb{C}$). Suppose that $\sigma$ is a cuspidal automorphic representation of $\mbf{G}(\mbb{A})$ such that $\sigma_{\infty}$ is cohomological. Let $S' \supset S$ be a finite set of primes including all primes where $\mbf{G}$ and $\sigma$ are ramified. Then there exists an automorphic representation $\Sigma$ of $\opn{GL}_1(\mbb{A}_E) \times \opn{GL}_{2n}(\mbb{A}_F)$ such that $\Sigma^{S'} \cong \opn{BC}^{S'}(\sigma^{S'})$, where $\opn{BC}^{S'}$ denotes the unramified base-change outside the set of places $S'$ (see \cite{ShinAppendix}). If $\Pi$ denotes the base-change of $\pi$ to an automorphic representation of $\opn{GL}_1(\mbb{A}_E) \times \opn{GL}_{2n}(\mbb{A}_F)$, then by strong multiplicity one for automorphic representations of general linear groups (see \cite[Theorem 4.4]{JacquetShalikaMultOne}), if $\pi^{S} \cong \sigma^{S}$ then we must have $\Pi \cong \Sigma$. Furthermore, since every prime in $S$ splits in $E/\mbb{Q}$, this implies that $\pi_{f, S} \cong \sigma_{f, S}$ (see \cite[Theorem A.1(2)]{ShinAppendix}). Hence the condition $\pi^S \cong \sigma^{S}$ implies that $\pi_f \cong \sigma_f$. The proposition now follows from the same proof as in \cite[Corollary 6.2.3]{UFJ} (note that we do not need to assume that $\Pi$ is cuspidal).
\end{proof}

We now prove the main theorem on the existence of $p$-adic $L$-functions as $\pi$ varies in a Coleman family.

\begin{theorem} \label{SecondMainTheoremFinalSSec}
    Let $\pi$ be as above and let $\mathcal{W}_0$ denote the $n[F^+:\mbb{Q}]$-dimensional weight space over $\opn{Spa}(L, \mathcal{O}_L)$ as in the proof of Theorem \ref{FirstMainTheoremFinalSection}. Then (after possibly increasing $L$ by a finite extension) there exists a open affinoid neighbourhood $\Omega = \opn{Spa}(\mathscr{O}_{\Omega}) \subset \mathcal{W}_0$ containing $\lambda_{\pi} \defeq \lambda$ such that:
    \begin{enumerate}
        \item There exists a Zariski dense subset $\Upsilon^{\opn{int}} \subset \Omega(L)$ of classical\footnote{I.e. the weights which are induced from a self-dual dominant algebraic character of the torus $T$.} weights and a morphism $\theta_{\underline{\pi}} \colon \mbb{T}_{\mathscr{O}_{\Omega}}^{-} \to \mathscr{O}_{\Omega}$ such that:
        \begin{itemize}
            \item For any $z \in \Upsilon^{\opn{int}}$, there exists a unique (up to isomorphism) cuspidal automorphic representation $\underline{\pi}_z$ satisfying the assumptions in \S \ref{FinalSectionAssumptionsSSec} and at the start of \S \ref{VariationInColemanFamiliesSSec} (with respect to the same compact open subgroups $K$ and $K^{\opn{new}}$) such that the specialisation of $\theta_{\underline{\pi}}$ at $z$ is equal to $\theta_{\underline{\pi}_z}$ (and the Harish-Chandra parameter of $(\underline{\pi}_z)_{\infty}$ is of the form $w_n \cdot (\lambda_z + \rho_G)$, where $\lambda_z$ is the self-dual dominant character corresponding to $z$).
            \item The specialisation of $\theta_{\underline{\pi}}$ at $\lambda_{\pi}$ is equal to $\theta_{\pi}$.
            \item The set $\Sigma_{\underline{\pi}} \defeq \bigcup_{z \in \Upsilon^{\opn{int}}} \{ z \} \times \Sigma_{\underline{\pi}_z}$ is Zariski dense in $\Omega \times \mathcal{W}(\ide{N}p^{\infty})$.
        \end{itemize}
        \item For any $z \in \Upsilon^{\opn{int}}$, let $\mathscr{L}_p(\underline{\pi}_{z}, -) \in \mathscr{D}^{\opn{la}}(\Gal(F_{\ide{N}p^{\infty}}/F), L)$ denote the locally analytic distribution in Theorem \ref{FirstMainTheoremFinalSection} associated with a fixed choice of $p$-regular $p$-stabilisation $\phi_z$ for $\underline{\pi}_z$ which is new away from $p$. Then there exist constants $\{ c_z \in L^{\times} : z \in \Upsilon^{\opn{int}} \}$ and a locally analytic distribution $\mathscr{L}_p(\underline{\pi}, -) \in \mathscr{D}^{\opn{la}}(\Gal(F_{\ide{N}p^{\infty}}/F), \mathscr{O}_{\Omega})$ such that
        \[
        \opn{sp}_z \mathscr{L}_p(\underline{\pi}, -) = c_z \cdot \mathscr{L}_p(\underline{\pi}_z, -)
        \]
        for any $z \in \Upsilon^{\opn{int}}$, where $\opn{sp}_z$ denotes the specialisation at $z$. 
    \end{enumerate}
\end{theorem}
\begin{proof}
    Following the proof of Theorem \ref{FirstMainTheoremFinalSection}, there exist an open affinoid neighbourhood $\Omega = \opn{Spa}(\mathscr{O}_{\Omega}) = \opn{Spa}(S, S^+)$ of $\lambda_{\pi}$ and a cohomology class
    \[
    \eta \in \opn{H}^{n-1}_{w_n, s_0\opn{-an}}(\kappa_S^*; 1)^{(-, \dagger, \leq h)}
    \]
    lifting $\eta_{z_0, 1, s_0\opn{-an}}$ (where $z_0$ denotes the point corresponding to $\lambda_{\pi}$). By Proposition \ref{MultOnePropFinalSec}, we see that $I_{\pi}$ defines a point on the eigenvariety (over $\mathcal{W}_0$), and the projection to $\mathcal{W}_0$ is \'{e}tale at this point. This implies that (after possibly shrinking $\Omega$), there exists a morphism $\theta_{\underline{\pi}} \colon \mbb{T}_{S}^{-} \to S$ with kernel $I_{\underline{\pi}}$ such that
    \[
    \eta \in \opn{H}^{n-1}_{w_n, s_0\opn{-an}}(\kappa_S^*; 1)^{(-, \dagger, \leq h), K_S^{\opn{new}}/K_S}_{I_{\underline{\pi}}}
    \]
    is a basis of this one-dimensional free $S$-module. It is also an eigenvector for the action of $\mbb{T}_{S}^{-}$ with eigencharacter $\theta_{\underline{\pi}}$. Let $\Upsilon \subset \Omega(L)$ be the subset of classical weights (which is Zariski dense). Then, by the same argument as in \cite[Theorem 6.2.5]{UFJ} and after possibly shrinking $\Omega$, the specialisation of $\theta_{\underline{\pi}}$ at $z \in \Upsilon$ determines a unique (up to isomorphism) cuspidal automorphic representation $\underline{\pi}_z$ of $\mbf{G}(\mbb{A})$ such that:
    \begin{itemize}
        \item $(\underline{\pi}_z)_{\infty} \otimes V_{\kappa_z}^*$ has non-vanishing $(\overline{\ide{p}}_{\circ}, K^{\circ})$-cohomology in degree $n-1$;
        \item $\underline{\pi}_z$ satisfies the assumptions in \S \ref{FinalSectionAssumptionsSSec} and \S \ref{VariationInColemanFamiliesSSec}, except $(\underline{\pi}_z)_{\infty}$ may not lie in the discrete series.
    \end{itemize}
    We let $\Upsilon^{\opn{int}} \subset \Upsilon$ denote the subset of classical weights $z$ where $(\underline{\pi}_z)_{\infty}$ lies in the discrete series with Harish-Chandra parameter $w_n \cdot (\lambda_z + \rho_G)$. By \cite[Theorem 3.5 \& Lemma 3.6.1]{HarrisAA}, we see that one can force this condition by assuming that $\lambda_z$ is sufficiently regular, therefore $\Upsilon^{\opn{int}}$ is still Zariski dense in $\Omega$.

    The rest of the theorem now follows by applying the main construction in Theorem \ref{AbstractPadicLMachineTheorem} to the class $\eta$, i.e., we define $\mathscr{L}_p(\underline{\pi}, -) = \Xi(\eta)$. Note that $\eta$ does indeed specialise to small slope eigenvectors at points in $\Upsilon^{\opn{int}}$. The constants $c_z$ measure the difference between the choices of $p$-regular $p$-stabilisations for $\underline{\pi}_z$ which are new away from $p$ and the specialisations of the cohomology class $\eta$.
\end{proof}

\subsection{The version for non-similitude groups}

We finish by explaining how one can deduce the versions of Theorem \ref{FirstMainTheoremFinalSection} and Theorem \ref{SecondMainTheoremFinalSSec} for non-similitude unitary groups given in the introduction. Firstly, we note that for any cuspidal automorphic representation of $\mbf{G}_0(\mbb{A})$, we can lift this to a cuspidal automorphic representation of $\mbf{G}(\mbb{A})$ such that $A_{\mbf{G}}(\mbb{A})$ acts trivially (see \cite[Proposition 3.11]{ACES} and the references therein). Here $A_{\mbf{G}}$ denotes the maximal $\mbb{Q}$-split torus in the centre of $\mbf{G}$. Conversely, the irreducible constituents of a cuspidal automorphic representation of $\mbf{G}(\mbb{A})$ (such that $A_{\mbf{G}}(\mbb{A})$ acts trivially) restricted to $\mbf{G}_0(\mbb{A})$ all lie in the same $L$-packet. We also have extension and restriction results for the local components of these automorphic representations. This means that:
\begin{itemize}
    \item We can lift an automorphic representation of $\mbf{G}_0(\mbb{A})$ and the test data satisfying the assumptions in \S \ref{StatementOfMainResultsSSecIntro} (resp. \S \ref{StatementOfMainResultsSSecIntro} and \S \ref{ResultsInFamiliesSSecIntro}) to an automorphic representation of $\mbf{G}(\mbb{A})$ and test data satisfying the assumptions in \S \ref{FinalSectionAssumptionsSSec} (resp. \S \ref{FinalSectionAssumptionsSSec} and \S \ref{VariationInColemanFamiliesSSec}), with the additional property that $A_{\mbf{G}}(\mbb{A})$ acts trivially.
    \item For any automorphic representation of $\mbf{G}(\mbb{A})$ such that $A_{\mbf{G}}(\mbb{A})$ acts trivially and test data satisfying the assumptions in \S \ref{FinalSectionAssumptionsSSec} (resp. \S \ref{FinalSectionAssumptionsSSec} and \S \ref{VariationInColemanFamiliesSSec}), we can find an irreducible constituent of its restriction to $\mbf{G}_0(\mbb{A})$ and test data satisfying the assumptions in \S \ref{StatementOfMainResultsSSecIntro} (resp. \S \ref{StatementOfMainResultsSSecIntro} and \S \ref{ResultsInFamiliesSSecIntro}). 
\end{itemize}
We can now deduce Theorem \ref{FirstTheoremIntro} and Theorem \ref{SecondTheoremIntro} from Theorem \ref{FirstMainTheoremFinalSection} and Theorem \ref{SecondMainTheoremFinalSSec} respectively. Indeed, we can force the additional condition that $A_{\mbf{G}}(\mbb{A})$ acts trivially for the family $\underline{\pi}$ by the same argument as in \cite[Corollary 8.2.4]{UFJ}, and we can relate the automorphic periods by using the fact that $[\mbf{H}]/A_{\mbf{G}}(\mbb{A})$ is a disjoint union of finitely many translates of $[\mbf{H}_0]$ (see \cite[Remark 8.2.8]{UFJ}).

%--------------------------------------------

\appendix

\section{Some representation theory} \label{SomeRepTheoryAppendix}

In this appendix, we record some results from the representation theory of general linear groups which are used in \S \ref{RelationWithUFJperiodsSubSubSec}. The notation in the first part of this appendix (up until \S \ref{ExplicitBranchingOpsAppendixSec}) differs from the rest of the article. 

Let $1 \leq a \leq b$ be integers, and set $G = \opn{GL}_{a + b}$. We let $H \subset G$ denote the subgroup $H = \opn{GL}_a \times \opn{GL}_b$ embedded block diagonally. We consider $G$ and $H$ as algebraic groups over $\mbb{Q}$ (or any field of characteristic zero). We let $T \subset H \subset G$ denote the standard diagonal torus.

\begin{notation} \label{YXtAppendixNotation}
    Let $R = \mbb{Q}[T_1, \dots, T_a]/(T_1^2, \dots, T_a^2)$ and let $\ide{m}_R \subset R$ denote the maximal ideal. Let $t = \opn{diag}(t_1, \dots, t_a) \in M_{a\times a}(R)$ denote a diagonal matrix with $t_i \in \ide{m}_R$ for all $i=1, \dots, a$, where $M_{a \times a}(R)$ denotes the space of $(a \times a)$-matrices with coefficients in $R$. For any permutation $\sigma \in S_a$ (which we view as a permutation of the set $\{1, \dots, a \}$), we set
    \begin{align*} 
    Y^{(a)}_{\sigma}(t) \defeq Y^{(a)}_{\sigma}(t_1, \dots, t_a) &\defeq \left( \delta_{\sigma(i), j} t_i \right)_{1 \leq i,j \leq a} \in M_{a \times a}(R) \\
    X^{(a)}_{\sigma}(t) \defeq X^{(a)}_{\sigma}(t_1, \dots, t_a) &\defeq \opn{I}_{a \times a} + Y^{(a)}_{\sigma}(t) \in \opn{GL}_a(R)
    \end{align*} 
    where $\opn{I}_{a \times a}$ is the $(a \times a)$ identity matrix, and $\delta_{-,-}$ denotes the Kronecker delta function.
\end{notation}

We have the following useful lemma:

\begin{lemma} \label{AppendixIWahoriFactorisationLemma}
    Let $t = \opn{diag}(t_1, \dots, t_a)$ and $\sigma$ be as in Notation \ref{YXtAppendixNotation}. Let 
    \[
    \sigma = \sigma_1 \circ \cdots \circ \sigma_r
    \]
    denote the decomposition of $\sigma$ into cycles, and let $|\sigma_i| \defeq \{ k \in \{ 1, \dots, a \} : \sigma_i(k) \neq k \}$. Set $m_i \defeq \opn{min}(\sigma_i) \defeq \opn{min}|\sigma_i|$. Then:
    \begin{enumerate}
        \item $X^{(a)}_{\sigma}(t)$ has an Iwahori decomposition
        \[
        X^{(a)}_{\sigma}(t) = X^{(a)}_{\sigma}(t)^+ \cdot X^{(a)}_{\sigma}(t)^{-}
        \]
        where $X^{(a)}_{\sigma}(t)^+$ (resp. $X^{(a)}_{\sigma}(t)^{-}$) lies in the standard upper-triangular Borel subgroup (resp. lower-triangular unipotent subgroup) of $\opn{GL}_a(R)$.
        \item The projection of $X^{(a)}_{\sigma}(t)^+$ to the standard diagonal torus is of the form $\opn{diag}(x_1, \dots, x_a)$ with
        \[
        x_j = \left\{ \begin{array}{cc} 1 & \text{ if } j \neq m_i \text{ for any } i=1, \dots, r \\ \left( 1 + \opn{sgn}(\sigma_i) \prod_{k \in |\sigma_i|} t_k \right) & \text{ if } j=m_i \end{array} \right.
        \]
        where $\opn{sgn}(-)$ denotes the sign of a permutation.
    \end{enumerate}
\end{lemma}
\begin{proof}
    Part (1) is immediate because $t$ has coefficients in $\ide{m}_R$ and $R$ is an Artinian local ring. For part (2), we will prove this by induction on $a$. Clearly the claim is true when $a = 1$, so we suppose $a > 1$.

    Suppose that $\sigma(a) = a$. Then
    \[
    X^{(a)}_{\sigma}(t_1, \dots, t_a) = \tbyt{X_{\sigma}^{(a-1)}(t_1, \dots, t_{a-1})}{}{}{1 + t_a}
    \]
    where the right-hand side is a block matrix, with top left block of size $(a-1) \times (a-1)$. The claim now follows from the induction hypothesis.

    Finally, suppose that $\sigma(a) \neq a$ (so in particular, $\sigma(a) < a$). Define a new permutation $\tau \in S_{a-1}$ by removing $a$ from its (non-trivial) cycle, i.e. we define
    \[
    \tau(j) \defeq \left\{ \begin{array}{cc} \sigma(a) & \text{ if } j=\sigma^{-1}(a) \\ \sigma(j) & \text{ if } j \neq \sigma^{-1}(a) \end{array} \right.
    \]
    for $j=1, \dots, a-1$. Suppose, without loss of generality, that $a \in |\sigma_1|$. Then the decomposition of $\tau$ into cycles is of the form $\tau = \tau_1 \circ \sigma_2 \circ \cdots \circ \sigma_r$, where $\tau_1$ is the cycle obtained by omitting $a$. Note that $\opn{sgn}(\tau_1) = -\opn{sgn}(\sigma_1)$. For $1 < i \leq r$, set $\tau_i = \sigma_i$.

    Let $Z$ denote the $(a \times a)$ matrix which is zero everywhere, except in the $(a, \sigma(a))$-entry where it is equal to $-t_a$. Then
    \[
    X_{\sigma}^{(a)}(t_1, \dots, t_a) \cdot (\opn{I}_{a \times a} + Z ) = \tbyt{U}{V}{}{1}
    \]
    where the right-hand side is a block matrix with $U \in \opn{GL}_{a-1}(R)$, and the matrix $U$ is given by
    \[
    U = X_{\tau}^{(a-1)}(t_1, \dots, t_{\sigma^{-1}(a)-1}, - t_{\sigma^{-1}(a)} t_a, t_{\sigma^{-1}(a)+1}, \dots, t_{a-1}) .
    \]
    Using the induction hypothesis, we therefore see that the projection of $X^{(a)}_{\sigma}(t)^+$ to the standard diagonal torus is of the form $\opn{diag}(x_1, \dots, x_{a-1}, 1)$ with
    \[
    x_j = \left\{ \begin{array}{cc} 1 & \text{ if } j \neq \opn{min}(\tau_i) \; (= \opn{min}(\sigma_i)) \text{ for any } i=1, \dots, r \\
    \left( 1 + \opn{sgn}(\sigma_i)\prod_{k \in |\sigma_i|} t_k \right) & \text{ if } j = \opn{min}(\sigma_i) \text{ with } i > 1 \\
    \left( 1 + \opn{sgn}(\tau_1) \cdot (- t_{\sigma^{-1}(a)}t_a ) \cdot \prod_{\substack{k \in |\tau_1| \\ k \neq \sigma^{-1}(a)}} t_k \right) & \text{ if } j = \opn{min}(\sigma_1) \end{array} \right.
    \]
    The claim now follows from $\opn{sgn}(\sigma_1) = -\opn{sgn}(\tau_1)$ and $|\tau_1| \cup \{a\} = |\sigma_1|$.
\end{proof}

The above lemma will be used in the proof of the following proposition. Let $\mbb{Q}[G]$ denote the ring of algebraic functions on $G$. This comes equipped with an action of $G$ given by $(g \cdot f)(-) = f(g^{-1} - )$ (for $g \in G$ and $f \in \mbb{Q}[G]$), and we can differentiate this to obtain an action of $\ide{g} = \opn{Lie}(G)$. We will denote this action of $\ide{g}$ by $\star$. We denote by
\[
u \defeq \tbyt{1}{u'}{}{1} \in G(\mbb{Q})
\]
the block matrix (with top left block of size $(a \times a)$, and bottom right of size $(b \times b)$), where $u'$ is the $(a \times b)$-matrix given by
\[
(u')_{i,j} = \delta_{b+1-i,j} , \quad \quad i \in \{1, \dots, a\}, j \in \{1, \dots, b\} .
\]
We let $w_a \in \opn{GL}_a(\mbb{Q})$ denote the antidiagonal matrix with each non-zero entry equal to $1$.

\begin{proposition} \label{NonVanishingOnUPropAppendix}
    Let $f \in \mbb{Q}[G]$ be an algebraic function such that
    \begin{itemize}
        \item $f(h^{-1} - ) = (\opn{det}h_1)^{-\nu_1} (\opn{det}h_2)^{-\nu_2} f(-)$ for all $h = (h_1, h_2) \in \opn{GL}_a \times \opn{GL}_b = H$
        \item $f( - b) = \kappa(b^{-1}) f(-)$ for all $b$ in the standard lower-triangular Borel subgroup of $G$
    \end{itemize}
    for some integers $\nu_1, \nu_2 \in \mbb{Z}$ and algebraic character $\kappa = (\kappa_1, \dots, \kappa_{a+b}) \in X^*(T)$ (inflated to the lower-triangular Borel subgroup). Let $\sigma \colon \{1, \dots, a \} \hookrightarrow \{1, \dots, b \}$ be an injective map, and set
    \[
    \mu_{\sigma} \defeq \prod_{i=1}^a E_{i, a+b+1-\sigma(i)} \in \mathcal{U}(\ide{g})
    \]
    where $E_{i, j}$ denotes the elementary matrix with non-zero entry in the $(i, j)$-th place (and the product takes place in $\mathcal{U}(\ide{g})$). 
    \begin{enumerate}
        \item If $\opn{im}(\sigma) \cap \{a+1, \dots, b \} \neq \varnothing$ then
        \[
        (\mu_{\sigma} \star f)(u) = 0 .
        \]
        \item Suppose $\opn{im}(\sigma) \cap \{a+1, \dots, b\} = \varnothing$ (so $\sigma$ is a permutation of $\{1, \dots, a \}$). Set $\opn{max}(\sigma_i) \defeq \opn{max}|\sigma_i|$, where $\sigma = \sigma_1 \circ \cdots \circ \sigma_r$ is the decomposition of $\sigma$ into cycles. Then 
        \[
        (\mu_{\sigma} \star f)(u) = (-1)^a \opn{sgn}(\sigma) \prod_{i=1}^r (\nu_1 + \kappa_{\opn{max}(\sigma_i)}) \cdot f(u) .
        \]
    \end{enumerate}
\end{proposition}
\begin{proof}
    Unless specified otherwise, in this proof any $(a+b) \times (a+b)$ matrix written as $\left( \begin{smallmatrix} * & * \\ * & * \end{smallmatrix} \right)$ means that matrix is written as a block matrix with top left block of size $(a \times a)$ (and hence the bottom right block is of size $(b \times b)$). We continue to denote $R = \mbb{Q}[T_1, \dots, T_a]/(T_1^2, \dots, T_a^2)$.

    Let $C \in M_{a \times b}(R)$, and let $C' \in M_{a \times a}(R)$ denote the matrix $(C')_{i, j} = C_{i, b-a+j}$ (i.e. the right-hand $(a \times a)$ block of $C$). Suppose that $C' \in \opn{GL}_a(R)$. Then there exists a matrix $Z \in \opn{GL}_b(R)$ in the standard lower-triangular unipotent such that $C Z \in M_{a \times b}(R)$ satisfies 
    \[
    (C Z)_{i, j} = 0 \text{ if } j \in \{1, \dots, b-a \}, \text{ and } (CZ)_{i, j} = (C')_{i, j-(b-a)} \text{ if } j \in \{b-a+1, \dots, b\}.
    \]
    We now prove part (1). Let $l = \sigma(n) \in \{a+1, \dots, b\}$. Let $U_{\sigma} \in M_{a \times b}(R)$ denote the matrix
    \[
    (U_{\sigma})_{i, j} = \delta_{b+1-\sigma(i),j} T_i .
    \]
    Then 
    \begin{equation} \label{MultipleDiffAppendix}
    (\mu_{\sigma} \star f)(u) = (-1)^a \left. \frac{\partial}{\partial T_1} \cdots \frac{\partial}{\partial T_a} f\left( \tbyt{1}{U_{\sigma}}{}{1} \cdot u \right) \right|_{T_1 = 0, \dots, T_a = 0}
    \end{equation}
    where one is also permitted to permute the order of differentiation freely. The power of $(-1)$ arises because $\left( \begin{smallmatrix} 1 & U_{\sigma} \\ 0 & 1 \end{smallmatrix} \right)^{-1} = \left( \begin{smallmatrix} 1 & -U_{\sigma} \\ 0 & 1 \end{smallmatrix} \right)$. We claim that $f\left( \left( \begin{smallmatrix} 1 & U_{\sigma} \\ 0 & 1 \end{smallmatrix} \right) \cdot u \right)$ doesn't depend on $T_n$ (where $n$ is the element such that $\sigma(n) = l$), which will give the claim because it will be killed by $\frac{\partial}{\partial T_n}$. We have
    \[
    \tbyt{1}{U_{\sigma}}{}{1} \cdot u = \tbyt{1}{U_{\sigma} + u'}{}{1} .
    \]
    Let $C = U_{\sigma} + u'$. Then the right-hand $(a \times a)$ block is invertible, so by the above argument we can find a lower triangular unipotent matrix $Z \in \opn{GL}_b(R)$ such that $(C Z)_{i, j} = 0$ for $j \in \{1, \dots, b-a\}$, and $(CZ)_{i, j}$ doesn't depend on $T_n$ if $j \in \{b-a+1, \dots, b\}$. By the transformation properties of $f$, we see that
    \[
    f\left( \tbyt{1}{U_{\sigma} + u'}{}{1} \right) = f\left(\tbyt{1}{}{}{Z^{-1}} \tbyt{1}{U_{\sigma} + u'}{}{1} \tbyt{1}{}{}{Z} \right) = f\left( \tbyt{1}{CZ}{}{1} \right)
    \]
    which doesn't depend on $T_n$. 

    We now prove part (2). Set $\sigma'(i) \defeq a+1-\sigma(a+1 - i)$, and $\sigma''(i) \defeq a+1 - \sigma(i)$. Let $U_{\sigma} \in M_{a \times b}(R)$ denote the matrix 
    \[
    (U_{\sigma})_{i, j} = \delta_{b+1-\sigma(i),j} T_{a+1 - i} = \delta_{b-a + \sigma''(i), j} T_{a+1 - i} .
    \]
    Then, as above, we can calculate $(\mu_{\sigma} \star f)(u)$ be the same expression in (\ref{MultipleDiffAppendix}). Let $C = U_{\sigma} + u'$ and $C' \in \opn{GL}_a(R)$ the right-hand $(a \times a)$ block. Then $C_{i, j} = 0$ if $j \in \{1, \dots, b-a\}$, and $C' = w_a X_{\sigma'}^{(a)}(t)$ where $t = \opn{diag}(T_1, \dots, T_a)$. Let 
    \begin{itemize}
        \item $A = w_a X_{\sigma'}^{(a)}(t)^+ w_a^{-1} \in \opn{GL}_a(R)$ which is lower-triangular. One has 
        \[
        \opn{det}A = \prod_{\ide{c}} \left( 1 + \opn{sgn}(\ide{c}) \prod_{k \in |\ide{c}|} T_k \right)
        \]
        by Lemma \ref{AppendixIWahoriFactorisationLemma}, where the product is over all cycles $\ide{c}$ which appear in the decomposition of $\sigma'$.
        \item $X = A^{-1}$ which is lower-triangular. By Lemma \ref{AppendixIWahoriFactorisationLemma}, its projection to the diagonal torus is of the form $(x_1, \dots, x_a)$ where
        \[
        x_j = \left\{ \begin{array}{cc} \left( 1 + \opn{sgn}(\ide{c}) \prod_{k \in |\ide{c}|} T_k \right)^{-1} & \text{ if } j = a+1-\opn{min}(\ide{c}) \text{ for some cycle } \ide{c} \\ 1 \text{ otherwise } \end{array} \right. .
        \]
        \item $Z \in \opn{GL}_b(R)$ is the matrix given by $Z_{i, j} = \delta_{i, j}$ if either $i$ or $j$ is not contained in $\{b-a+1, \dots, b\}$, and $Z_{i, j} = (X^{(a)}_{\sigma'}(t)^{-})_{i-(b-a), j-(b-a)}$ if $i, j \in \{b-a+1, \dots, b\}$. This lies in the lower triangular unipotent of $\opn{GL}_b(R)$. 
        \item $B = Z^{-1}$, which satisfies $\opn{det}B = 1$.
    \end{itemize}
    Using the transformation properties for $f$, we see that:
    \begin{align*}
        f\left( \tbyt{1}{U_{\sigma} + u'}{}{1} \right) &= f\left( \tbyt{A}{}{}{B} \tbyt{1}{u'}{}{1} \tbyt{X}{}{}{Z} \right) \\
         &= (\opn{det}A)^{\nu_1} \prod_{\ide{c}} \left( 1 + \opn{sgn}(\ide{c}) \prod_{k \in |\ide{c}|} T_k \right)^{\kappa_{a+1-\opn{min}(\ide{c})}} f(u) \\
         &= \prod_{\ide{c}} \left( 1 + \opn{sgn}(\ide{c}) \prod_{k \in |\ide{c}|} T_k \right)^{\nu_1 + \kappa_{a+1-\opn{min}(\ide{c})}} f(u).
    \end{align*}
    Differentiating and setting $T_1 = \cdots = T_a = 0$, we see that
    \[
    (\mu_{\sigma} \star f)(u) = (-1)^a \prod_{\ide{c}} \opn{sgn}(\ide{c}) (\nu_1 + \kappa_{a+1-\opn{min}(\ide{c})}) .
    \]
    To conclude, we note that $\prod_{\ide{c}} \opn{sgn}(\ide{c}) = \opn{sgn}(\sigma') = \opn{sgn}(\sigma)$, and the cycles appearing in the decomposition of $\sigma$ are $w_a \circ \ide{c} \circ w_a$ (and we have $\opn{max}(w_a \circ \ide{c} \circ w_a) = a+1- \opn{min}(\ide{c})$). Here $w_a$ is the permutation of $\{1, \dots, a\}$ given by $w_a(i) = a+1 - i$.
\end{proof}

\subsection{Explicit branching operators} \label{ExplicitBranchingOpsAppendixSec}

We now return to the setting of \S \ref{BranchingLawsPreliminarySection}. It turns out that (up to a non-zero rational number depending on $(\kappa, j) \in \mathcal{E}$), one can express $v_{\kappa}^{[j]}$ in terms of $v_{\kappa}^{[0]}$ via certain explicit elements of the universal enveloping algebras of $\ide{g} = \opn{Lie}(G)$ and $\ide{m}_G = \opn{Lie}M_G$. More precisely, let $J \subset \{ n+1, \dots, 2n\}$ be a multiset of size $j_{\tau_0}$. Fix a basis $\{ e_J \}$ of $S_{-j}$ (as $J$ runs over all such multisets) such that a general torus element $(x; y_{1, \tau}, \dots, y_{2n, \tau}) \in T$ acts on $e_J$ as
\[
(x; y_{1, \tau}, \dots, y_{2n, \tau}) \cdot e_J = y_{1, \tau_0}^{j_{\tau_0}} \left( \prod_{i \in J} y_{i, \tau_0}^{-1} \right) e_J .
\]
If $j_{\tau_0} = 0$, then we identify $S_{-j}$ with the trivial representation and the basis is a singleton $\{ e_{\varnothing} \}$.

For any $(i, j) \in \{1, \dots, 2n\}$, let $E_{i, j, \tau} \in \ide{g}$ denote the element which is zero outside the $\tau$-component, and in the $\tau$-component is equal to the elementary $(2n \times 2n)$-matrix with non-zero entry in the $(i, j)$-th place. For any $\tau \neq \tau_0$, let $\opn{det}_{\tau} \in \mathcal{U}(\ide{gl}_{2n})$ denote the element obtained as the determinant of the $\mathcal{U}(\ide{gl}_{2n})$-valued $(n \times n)$-matrix $(E_{i, j+n, \tau})_{i, j}$ (where $i, j \in \{1, \dots, n\}$). This is well-defined and independent of any ordering because all the elementary matrices considered commute with each other. We view $\opn{det}_{\tau}$ as elements of $\mathcal{U}(\ide{m}_G)$ by identifying $\ide{gl}_{2n}$ with the $\tau$-component of $\ide{m}_G$.

Similarly, for any $k \in \{ n+1, \dots, 2n \}$, we define $(-1)^{k-(n+1)}\opn{det}_{k, \tau_0} \in \mathcal{U}(\ide{gl}_1 \oplus \ide{gl}_{2n-1})$ as the determinant of the $\mathcal{U}(\ide{gl}_1 \oplus \ide{gl}_{2n-1})$-valued $((n-1) \times (n-1))$-matrix whose $(i, j)$-th component is
\[
\left\{ \begin{array}{cc} E_{i+1, j+n} & \text{ if } j < k - n \\ E_{i+1, j+n+1} & \text{ if } j \geq k-n \end{array} \right. 
\]
as $i, j$ run through the set $\{1, \dots, n-1\}$. As above, this is well-defined and independent of ordering because the elementary matrices considered commute with each other. We view $\opn{det}_{k, \tau_0}$ as an element of $\mathcal{U}(\ide{m}_G)$ by identifying the $\tau_0$-component of $\ide{m}_G$ with $\ide{gl}_1 \oplus \ide{gl}_{2n-1}$. Finally, for any multiset $J \subset \{n+1, \dots, 2n \}$ of size $j_{\tau_0}$, we define
\[
\opn{det}^{[j]}_J \defeq \left( \prod_{k \in J} \opn{det}_{k, \tau_0} \right) \prod_{\tau \neq \tau_0} \opn{det}_{\tau}^{j_{\tau}} \in \mathcal{U}(\ide{m}_G) .
\]

\begin{proposition} \label{ExistenceOfCkappaJAppendixProp}
    Let $(\kappa, j) \in \mathcal{E}$. Then
    \[
    \sum_J \opn{det}^{[j]}_J \cdot v_{\kappa}^{[0]} \otimes e_J \in V_{\kappa} \otimes S_{-j}
    \]
    is a non-zero multiple of $v_{\kappa}^{[j]}$, where the sum runs over all multisets $J \subset \{n+1, \dots, 2n\}$ of size $j_{\tau_0}$.
\end{proposition}
\begin{proof}
    For $\tau \in \Psi$, let $M_{G, \tau}$ denote the $\tau$-component of $M_G$ (so $M_{G, \tau_0} = \opn{GL}_1 \times \opn{GL}_{2n-1}$ and $M_{G, \tau} = \opn{GL}_{2n}$ for $\tau \neq \tau_0$). We use similar notation for $M_H$. Write
    \[
    V_{\kappa} = (-)^{\kappa_0} \otimes \bigotimes_{\tau \in \Psi} V_{\kappa_{\tau}}
    \]
    where $V_{\kappa_{\tau}}$ is (up to isomorphism) the irreducible representation of $M_{G, \tau}$ of highest weight $\kappa_{\tau} = (\kappa_{1, \tau}, \dots, \kappa_{2n, \tau})$, and $(-)^{\kappa_0}$ denotes the line on which $\opn{GL}_1$ acts through the character $x \mapsto x^{\kappa_0}$. Similarly, we let $S_{-j, \tau_0}$ denote the irreducible representation of $M_{H, \tau_0}$ of highest weight $(j_{\tau_0}, \dots, -j_{\tau_0})$.
    
    In this proof only, we will use a slightly different identification of $V_{\kappa_{\tau}}$ with algebraic functions than in \S \ref{BranchingLawsPreliminarySection}. More precisely, we identify $V_{\kappa_{\tau}}$ with the space of algebraic functions
    \[
    f \colon M_{G, \tau} \to \mbb{A}^1
    \]
    such that $f(- b) = \kappa_{\tau}(b^{-1})f(-)$ for any $b$ in the standard \emph{lower-triangular} Borel subgroup of $M_{G, \tau}$. Then $v_{\kappa}^{[0]}$ is identified with 
    \[
    1 \otimes \bigotimes_{\tau \in \Psi} f_{\tau} \in (-)^{\kappa_0} \otimes \bigotimes_{\tau \in \Psi} V_{\kappa_{\tau}}
    \]
    where $f_{\tau} \colon M_{G, \tau} \to \mbb{A}^1$ are functions as above, with $f_{\tau}(m^{-1} -) = \sigma_{\kappa}^{[0], -1}(m) f_{\tau}(-)$ for any $m \in M_{H, \tau}$. If we let $u_{\tau}^t$ denote the transpose in $M_{G, \tau}$ of the $\tau$-component of $u$ (see Definition \ref{NewDefOfGamma}), then $u_{\tau}^t$ is a representative for the open orbit of $M_{H, \tau}$ on the quotient of $M_{G, \tau}$ by the lower-triangular Borel subgroup. Therefore, we have $f_{\tau}(u_{\tau}^t) \neq 0$ for all $\tau \in \Psi$. 

    It suffices to prove the claim for each individual $\tau \in \Psi$. We first deal with the case $\tau \neq \tau_0$. In this setting, we have an algebraic function
    \[
    f_{\tau} \colon \opn{GL}_{2n} \to \mbb{A}^1
    \]
    such that $f(m^{-1}-) = f(-)$ for all $m \in M_{H, \tau}$, and $f_{\tau}(- b) = \kappa_{\tau}(b^{-1}) f(-)$ for all $b$ in the lower-triangular Borel subgroup of $\opn{GL}_{2n}$. Note that $\opn{Ad}(m)\opn{det}_{\tau} = \opn{det}m_1 \opn{det}m_2^{-1} \opn{det}_{\tau}$ for $m = (m_1, m_2) \in \opn{GL}_n \times \opn{GL}_n = M_{H, \tau}$. By iterating Proposition \ref{NonVanishingOnUPropAppendix}, we see that
    \[
    (\opn{det}_{\tau}^{j_{\tau}} \cdot f_{\tau})(u_{\tau}^t) = \pm C \cdot f_{\tau}(u_{\tau}^t)
    \]
    where $C$ is product of terms of the form
    \begin{equation} \label{SnSumKappaNu}
    \sum_{\sigma \in S_n} \prod_{i} ( \kappa_{m_i, \tau} - \nu )
    \end{equation}
    where $\nu$ is an integer satisfying $0 \leq \nu \leq j_{\tau}-1$, and $1 \leq m_i \leq n$ are integers depending on the cycle decomposition of $\sigma$. Here, in the notation of Proposition \ref{NonVanishingOnUPropAppendix}, we have used that $\opn{det}_{\tau} = \pm \left( \sum_{\sigma \in S_n} \opn{sgn}(\sigma) \mu_{\sigma} \right)$. Since $j_{\tau} \leq \kappa_{n, \tau} \leq \kappa_{n-1, \tau} \leq \cdots \leq \kappa_{1, \tau}$, we see that (\ref{SnSumKappaNu}) is a sum of (non-zero) positive integers, hence must be non-zero itself. This implies $C \neq 0$ and hence $\opn{det}_{\tau}^{j_{\tau}} \cdot f_{\tau}$ is non-zero. By multiplicity one, it must therefore be a non-zero multiple of the $\tau$-component of $v_{\kappa}^{[j]}$.

    We now consider the case $\tau = \tau_0$. Let $W_{-j, \tau_0}$ denote the algebraic representation of $M_{G, \tau_0}$ of highest weight $(j_{\tau_0}, 0, \dots, 0, -j_{\tau_0})$. We can naturally view $S_{-j, \tau_0} \subset W_{-j, \tau_0}$, and note that the subspace $S_{-j, \tau_0}$ is killed under the action of any elementary matrix $E_{i, j, \tau_0} \in \ide{gl}_1 \oplus \ide{gl}_{2n-1}$ with $i \in \{2, \dots, n \}$ and $j \in \{n+1, \dots, 2n\}$. Recall from the proof of Theorem \ref{TheoremForClassicalBranching} that $V_{\kappa_{\tau_0}'}$ appears in $V_{\kappa_{\tau_0}} \otimes S_{-j, \tau_0} \subset V_{\kappa_{\tau_0}} \otimes W_{-j, \tau_0}$ with multiplicity one, where $\kappa_{\tau_0}'$ is the weight 
    \[
    (\kappa_{1,\tau_0}+j_{\tau_0}, \kappa_{2, \tau_0}, \dots, \kappa_{n, \tau_0}, \kappa_{n+1, \tau_0} - j_{\tau_0}, \kappa_{n+2, \tau_0}, \dots, \kappa_{2n, \tau_0}) .
    \]
    Suppose $0 \leq j_{\tau_0} < \kappa_{n+1, \tau_0} - \kappa_{n+2, \tau_0}$. Suppose that $\sum_{J} \prod_{k \in J}\opn{det}_{k, \tau_0} \cdot f_{\tau_0} \otimes e_{J} \in V_{\kappa_{\tau_0}} \otimes S_{-j, \tau_0} \subset V_{\kappa_{\tau_0}} \otimes W_{-j, \tau_0}$ is a non-zero multiple of the $\tau_0$-component of $v_{\kappa}^{[j]}$. Then, by multiplicity one, it must correspond to a function $F \in V_{\kappa_{\tau_0}'}$.

    Let $H \defeq \sum_{k=n+1}^{2n} \opn{det}_{k, \tau_0} \cdot F \otimes e_{\{k\}} \in V_{\kappa_{\tau_0}'} \otimes S_{-1, \tau_0}$. Then the image of $H$ under the map
    \begin{equation} \label{VkprimeToVkTwist}
    V_{\kappa_{\tau_0}'} \otimes S_{-1, \tau_0} \to V_{\kappa_{\tau_0}} \otimes S_{-j, \tau_0} \otimes S_{-1, \tau_0} \to V_{\kappa_{\tau_0}} \otimes S_{-(j+1), \tau_0}
    \end{equation}
    is equal to 
    \[
    \sum_{l=n+1}^{2n} \sum_{J} \prod_{k \in J} \opn{det}_{l, \tau_0}\opn{det}_{k, \tau_0} \cdot f_{\tau_0} \otimes e_{J \cup \{l\}} = \sum_{A} \prod_{k \in A}\opn{det}_{k, \tau_0} \cdot f_{\tau_0} \otimes e_{A}
    \]
    where $A$ runs of all multisets in $\{ n+1, \dots, 2n \}$ of size $j_{\tau_0}+1$, and the last map in (\ref{VkprimeToVkTwist}) is the natural one using the fact that $S_{-j, \tau_0} = \opn{Sym}^{j_{\tau_0}} S_{-1, \tau_0}$. Note that Lemma \ref{PierisRuleLemma} implies the map (\ref{VkprimeToVkTwist}) is injective on the $M_{H, \tau_0}$-eigenspaces with eigencharacter:
    \begin{align} 
    M_{H, \tau_0} = \opn{GL}_1 \times \opn{GL}_{n-1} \times \opn{GL}_n &\to \mbb{G}_m \nonumber \\
    (m_1, m_2, m_3) &\mapsto m_1^{\kappa_{1, \tau_0} + j_{\tau_0} +1} \opn{det}m_2^{w + j_{\tau_0}+1-\kappa_{n+1, \tau_0}} \opn{det}m_3^{\kappa_{n+1, \tau_0} - (j_{\tau_0}+1)} . \label{JplusOneEigenchar}
    \end{align} 
    We will show that $H$ is non-zero and an eigenvector under $M_{H, \tau_0}$ with eigencharacter (\ref{JplusOneEigenchar}), and then the general claim follows from induction on $j_{\tau_0}$. 

    The fact that $H$ transforms through the character (\ref{JplusOneEigenchar}) under the action of $M_{H, \tau_0}$ follows from a direct calculation. To show $H$ is non-zero, we will again apply Proposition \ref{NonVanishingOnUPropAppendix}. It suffices to show that $(\opn{det}_{n+1, \tau_0} \cdot F)(u_{\tau_0}^t)$ is non-zero.\footnote{In fact, one can also use Proposition \ref{NonVanishingOnUPropAppendix} to show that $(\opn{det}_{k, \tau_0} \cdot F)(u_{\tau_0}^t) =0$ for any $k > n+1$.} But in the notation of Proposition \ref{NonVanishingOnUPropAppendix} (with $a=n-1$, $b=n$), we see that $\opn{det}_{n+1, \tau_0} = \pm \sum_{\sigma \in S_{n-1}} \opn{sgn}(\sigma)\mu_{\sigma}$, where we naturally view $\mu_{\sigma}$ in $\mathcal{U}(\ide{m}_G)$ via the inclusion $\mathcal{U}(\ide{gl}_{2n-1}) \subset \mathcal{U}(\ide{m}_G)$. But now we are in the setting of Proposition \ref{NonVanishingOnUPropAppendix}(2), and hence $(\opn{det}_{n+1, \tau_0} \cdot F)(u_{\tau_0}^t) = \pm C \cdot F(u_{\tau_0}^t)$, where $C$ is of the form
    \[
    C = \sum_{\sigma \in S_{n-1}} \prod_i ( \kappa_{n+1, \tau_0} - j_{\tau_0} - w + \kappa_{m_i, \tau_0} ) = \sum_{\sigma \in S_{n-1}} \prod_i ( \kappa_{n+1, \tau_0} - \kappa_{2n+2-m_i, \tau_0} - j_{\tau_0} ) 
    \]
    where $2 \leq m_i \leq n$ are certain integers. This is non-zero because $j_{\tau_0} < \kappa_{n+1, \tau_0} - \kappa_{n+2, \tau_0}$, and $F(u_{\tau_0}^t)$ is non-zero by the induction hypothesis. This finishes the proof of the proposition. 
\end{proof}

Let $\opn{det}_{\tau_0} \in \mathcal{U}(\ide{gl}_{2n})$ denote the determinant of the $\mathcal{U}(\ide{gl}_{2n})$-valued matrix $(E_{i, j+n, \tau_0})_{i, j}$ with $i, j \in \{1, \dots, n \}$. We view this as an element of $\mathcal{U}(\ide{g})$ by identifying the $\tau_0$-component of $\ide{g}$ with $\ide{gl}_{2n}$. We make the following definition:

\begin{definition} \label{AppendixDefOfDELTAkappaJDiffOp}
    Let $(\kappa, j) \in \mathcal{E}$ and recall the definition of $\delta_{\kappa, j}$ from Definition \ref{DefinitionOfDeltakappaj} (which makes sense over any characteristic zero field). With notation as above:
    \begin{enumerate}
        \item Let $C_{\kappa, j} \in \mbb{Q}^{\times}$ denote the unique non-zero rational number satisfying
        \[
         \delta_{\kappa, j} = C_{\kappa, j} \cdot \left( \sum_J \opn{det}^{[j]}_J \cdot \delta_{\kappa, 0} \otimes x_J \right)
        \]
        where the sum runs over all multisets $J \subset \{n+1, \dots, 2n\}$ of size $j_{\tau_0}$, and $x_J \in C^{\opn{pol}}(\mbb{G}_a^{\oplus 2n-1}, \mbb{G}_a)$ is the algebraic function satisfying $x_J(a_2, \dots, a_{2n}) = \prod_{k \in J}a_k$. Note that $C_{\kappa, j}$ exists by Proposition \ref{ExistenceOfCkappaJAppendixProp}.
        \item We define:
        \[
        \Delta_{\kappa}^{[j]} \defeq C_{\kappa, j} \cdot \prod_{\tau \in \Psi} \opn{det}_{\tau}^{j_{\tau}} \; \in \; \mathcal{U}(\ide{g}).
        \]
        Note that $\Delta_{\kappa}^{[j]} = \sum_J \left( \prod_{k \in J} E_{1, k, \tau_0} \right) \opn{det}^{[j]}_J$, where the sum runs over all multisets $J \subset \{n+1, \dots, 2n\}$ of size $j_{\tau_0}$.
    \end{enumerate}
\end{definition}

\section{Equivariant linear functionals} \label{EquivariantLinearFuncsAppendix}

Let $G = \opn{GL}_{2n}$ and $H = \opn{GL}_n \times \opn{GL}_n$, and consider the block diagonal embedding $H \subset G$. In this section, we record some useful transformation properties for $H(\mbb{Q}_p)$-equivariant linear functionals on smooth representations of $G(\mbb{Q}_p)$.

\subsection{Trace-compatibility}

Let $\beta \geq 1$ be an integer and let $K_{G, \beta} \subset G(\mbb{Z}_p)$ denote the depth $p^{\beta}$ upper-triangular Iwahori subgroup, i.e., all elements which land in the standard upper triangular Borel subgroup modulo $p^{\beta}$. Let $\hat{\gamma} \in G(\mbb{Z}_p)$ be any element such that $H(\mbb{Z}_p) \cdot \hat{\gamma} \cdot B(\mbb{Z}_p)$ is Zariski dense in $G(\mbb{Z}_p)$ (such an element exists because the pair $(G, H)$ is spherical). Here $B \subset G$ denotes the upper-triangular Borel subgroup. 

Let $K_{H, \beta} \defeq \hat{\gamma} K_{G, \beta} \hat{\gamma}^{-1} \cap H(\mbb{Z}_p)$. Let $\nu \colon H \to \mbb{G}_m$ denote the morphism given by $\nu(h_1, h_2) = \opn{det}h_2/\opn{det}h_1$. Then $K_{H, \beta}$ is contained in $\nu^{-1}(1 + p^{\beta}\mbb{Z}_p) \subset H(\mbb{Z}_p)$. Indeed, it suffices to check this for a single choice of $\hat{\gamma}$ -- for example the block matrix
\begin{equation} \label{SimpleExampleofHatGammaAppendix}
\hat{\gamma} = \tbyt{1}{}{w_{\opn{GL}_{n}}^{\opn{max}}}{1}
\end{equation} 
with block sizes $n \times n$, where $w_{\opn{GL}_{n}}^{\opn{max}}$ is the antidiagonal matrix with $1$s along the antidiagonal -- and this is a simple computation. Furthermore, one can verify that $\hat{\gamma}^{-1} K_{H, \beta} \hat{\gamma} \backslash K_{G, \beta} / K_{G, \beta+1}$ is a singleton and $\hat{\gamma}^{-1} K_{H, \beta} \hat{\gamma} \cap K_{G, \beta+1} = \hat{\gamma}^{-1} K_{H, \beta+1} \hat{\gamma}$ (again, it suffices to check this for one example of $\hat{\gamma}$). 

\begin{lemma} \label{Lem:1stAppendixB}
    Let $\beta \geq 1$ and $1 \leq e \leq \beta$. Suppose that $\chi \colon \mbb{Q}_p^{\times} \to \mbb{C}^{\times}$ is a smooth character which is trivial on $1 + p^e\mbb{Z}_p$. Let $\pi$ be a smooth representation of $G(\mbb{Q}_p)$ and let $\mathfrak{Z} \in \opn{Hom}_{H(\mbb{Q}_p)}(\pi, \chi^{-1} \circ \nu)$. Let $\phi \in \pi^{K_{G, \beta}}$ and let
    \[
    \opn{tr}(\phi) \defeq \sum_{k \in K_{G, e}/K_{G, \beta}} k \cdot \phi \; \in \; \pi^{K_{G, e}}.
    \]
    Then $\mathfrak{Z}(\hat{\gamma} \cdot \opn{tr}(\phi)) = p^{(\beta - e)(2n-1)n} \mathfrak{Z}(\hat{\gamma} \cdot \phi)$.
\end{lemma}
\begin{proof}
    By above, we have
    \[
    \hat{\gamma} \cdot \opn{tr}(\phi) = \sum_{l \in K_{H, e}/K_{H, \beta}} l \cdot \hat{\gamma} \cdot \phi .
    \]
    The result now follows from the $H$-equivariance of $\mathfrak{Z}$, the fact that $\chi^{-1}(\nu(l)) = 1$ for any $l \in K_{H, e}$, and $[K_{H, e} : K_{H, \beta}] = [K_{G, e} : K_{G, \beta}] = p^{(\beta - e)(2n-1)n}$.
\end{proof}

\subsection{Dual eigenvectors}

We continue with the notation in the previous section, but we now fix $\hat{\gamma}$ to be the element in (\ref{SimpleExampleofHatGammaAppendix}). Let $\phi \in \pi^{K_{G, 1}}$ and let $t_p, s_p \in G(\mbb{Q}_p)$ denote the diagonal matrices
\[
t_p = \opn{diag}(p^{2n-1}, p^{2n-2}, \dots, p, 1), \quad \quad s_p = w_G^{\opn{max}} t_p w_G^{\opn{max}}
\]
where $w_G^{\opn{max}} \in G(\mbb{Z}_p)$ denotes the antidiagonal matrix with $1$s along the antidiagonal. For any $\beta \geq 1$, set
\[
\psi_{\beta} \defeq [K_{G, \beta} s_p^{\beta} w_G^{\opn{max}} K_{G, 1}] \cdot \phi = \sum_{k \in K_{G, \beta}/(s_p^{\beta} w_G^{\opn{max}}K_{G,1}w_G^{\opn{max}} s_p^{-\beta} \cap K_{G, \beta})} k s_p^{\beta} w_G^{\opn{max}} \cdot \phi . 
\]
Let $\delta_B \colon B(\mbb{Q}_p) \to \mbb{C}^{\times}$ denote the standard modulus character which satisfies:
\[
\delta_B(b) = |t_1|^{2n-1} \cdot |t_2|^{2n-3} \cdots |t_{2n-1}|^{3-2n} \cdot |t_{2n}|^{1-2n}, \quad b \in B(\mbb{Q}_p),
\]
where $\opn{diag}(t_1, \dots, t_{2n})$ denotes the projection of $b$ to the diagonal torus, and $|\cdot|$ is the $p$-adic absolute value normalised so that $|p| = p^{-1}$.

\begin{lemma} \label{Lem:2ndAppendixB}
    Let $\chi \colon \mbb{Q}_p^{\times} \to \mbb{C}^{\times}$ be a smooth character which is trivial on $1+p^{\beta}\mbb{Z}$, and let  $\mathfrak{Z} \in \opn{Hom}_{H(\mbb{Q}_p)}(\pi, \chi^{-1} \circ \nu)$. Then
    \[
    \mathfrak{Z}(\hat{\gamma} \cdot \psi_{\beta}) = \delta_B(t_p)^{-\beta} p^{-n(2n-1)} [K_{G, 1} : K_{G, \beta}]^{-1} \chi(-1)^n \cdot \mathfrak{Z}( {^t\hat{\gamma}^{-1}} t_p^{\beta} \cdot \phi ) .
    \]
\end{lemma}
\begin{proof}
    An explicit calculation shows that 
    \[
    \hat{\gamma}^{-1}K_{H, \beta} \hat{\gamma} \backslash K_{G, \beta}/(s_p^{\beta} w_G^{\opn{max}}K_{G,1}w_G^{\opn{max}} s_p^{-\beta} \cap K_{G, \beta})
    \]
    is a singleton, hence using the transformation properties of $\mathfrak{Z}$ (and the fact that $\chi$ is trivial on $1+p^{\beta}\mbb{Z}_p$)
    \begin{align*} 
    \mathfrak{Z}(\hat{\gamma} \cdot \psi_{\beta}) &= [K_{G, \beta} : (s_p^{\beta} w_G^{\opn{max}}K_{G,1}w_G^{\opn{max}} s_p^{-\beta} \cap K_{G, \beta})] \mathfrak{Z}( \hat{\gamma} s_p^{\beta} w_G^{\opn{max}} \cdot \phi ) \\
     &= \delta_B(t_p)^{-\beta} p^{-n(2n-1)} [K_{G, 1} : K_{G, \beta}]^{-1} \mathfrak{Z}( \hat{\gamma} s_p^{\beta} w_G^{\opn{max}} \cdot \phi ) .
    \end{align*} 
    Now we have
    \begin{align*} 
    \hat{\gamma} s_p^{\beta} w_G^{\opn{max}} &= \tbyt{1}{}{w_{\opn{GL}_n}^{\opn{max}}}{1} w_G^{\opn{max}} t_p^{\beta} \\ 
     &= \tbyt{-1}{}{}{1} \tbyt{1}{-w_{\opn{GL}_n}^{\opn{max}}}{}{1} \tbyt{1}{}{w_{\opn{GL}_n}^{\opn{max}}}{1} t_p^{\beta} \\
     &\in \tbyt{-1}{}{}{1} {^t\hat{\gamma}^{-1}} t_p^{\beta} K_{G, 1} .
    \end{align*} 
    The claim now follows because $\phi$ is fixed by $K_{G, 1}$.
\end{proof}

%----------------------------------------

\newcommand{\etalchar}[1]{$^{#1}$}
\renewcommand{\MR}[1]{}
\providecommand{\bysame}{\leavevmode\hbox to3em{\hrulefill}\thinspace}
\providecommand{\MR}{\relax\ifhmode\unskip\space\fi MR }
\providecommand{\MRhref}[2]{%
  \href{http://www.ams.org/mathscinet-getitem?mr=#1}{#2}
}
\providecommand{\href}[2]{#2}

\Addresses

\end{document}